\pdfoutput=1
\documentclass[10pt]{amsart}
\usepackage[utf8]{inputenc}
\usepackage[T1]{fontenc}
\usepackage[english]{babel}
\usepackage[babel=true]{microtype}
\usepackage{amsmath,amssymb,amsthm,csquotes,bm,setspace,changepage,float,tikz,tkz-base,subcaption,paralist,tikz-cd,mathrsfs}

\usepackage[backend=biber,style=alphabetic,giveninits,maxbibnames=10]{biblatex}
\DeclareNameAlias{default}{family-given}
\AtBeginBibliography{\small}

\usepackage[all,cmtip]{xy}
\usepackage[margin=1in]{geometry}
\usepackage{hyperref}
\hypersetup{colorlinks=true,linkcolor=purple,anchorcolor=green,citecolor=blue,filecolor=black,menucolor=black,urlcolor=brown}
\numberwithin{equation}{section}

\theoremstyle{plain}
\newtheorem{theorem}{Theorem}[section]
\newtheorem{lemma}[theorem]{Lemma}
\newtheorem{corollary}[theorem]{Corollary}
\newtheorem{proposition}[theorem]{Proposition}
\newtheorem{conjecture}[theorem]{Conjecture}

\theoremstyle{definition}
\newtheorem{definition}[theorem]{Definition}
\newtheorem{example}[theorem]{Example}

\theoremstyle{remark}
\newtheorem{remark}[theorem]{Remark}
\newtheorem{notation}[theorem]{Notation}
\newtheorem*{conventions}{Conventions}
\newtheorem*{acknowledgements}{Acknowledgements}

\DeclareMathOperator{\Bl}{Bl}
\DeclareMathOperator{\Cl}{Cl}
\DeclareMathOperator{\Def}{Def}
\DeclareMathOperator{\Gal}{Gal}
\DeclareMathOperator{\Gr}{Gr}
\DeclareMathOperator{\gen}{gen}

\DeclareMathOperator{\K3}{K3}
\DeclareMathOperator{\im}{im}
\DeclareMathOperator{\MM}{MM}
\DeclareMathOperator{\Newt}{Newt}
\DeclareMathOperator{\NS}{NS}
\DeclareMathOperator{\Pic}{Pic}
\DeclareMathOperator{\res}{res}
\DeclareMathOperator{\rk}{rank}

\title{Modularity of Landau--Ginzburg models}
\author[Doran, Harder, Katzarkov, Ovcharenko, and Przyjalkowski]{Charles~Doran, Andrew~Harder, Ludmil~Katzarkov, Mikhail~Ovcharenko, and Victor~Przyjalkowski}
\date{}

\address{Charles Doran \newline
   \textnormal{Department of Mathematical and Statistical Sciences, 632 CAB, University of Alberta, Edmonton, AB T6G 2G1, Canada. \newline
    Bard College, Annandale-on-Hudson, NY 12571, USA. \newline
    Center of Mathematical Sciences and Applications, Harvard University, 20 Garden Street, Cambridge, MA 02138, USA.} \newline
    \textnormal{\texttt{charles.doran@ualberta.ca}}}

\address{Andrew Harder \newline
   \textnormal{Department of Mathematics, Lehigh University, Chandler--Ullmann Hall, 17 Memorial Dr. E.,, Bethlehem, PA 18015, USA.} \newline
   \textnormal{\texttt{anh318@lehigh.edu}}}

\address{Ludmil Katzarkov \newline
   \textnormal{University of Miami, Coral Gables, 33146 FL, USA. \newline
   National Research University Higher School of Economics, Laboratory of Mirror Symmetry, NRU HSE, 6 Usacheva str., Moscow, Russia, 119048. \newline
   Institute of Mathematics and Informatics, Bulgarian Academy of Sciences, Acad. G. Bonchev Str. bl. 8, 1113, Sofia, Bulgaria.} \newline
   \textnormal{\texttt{lkatzarkov@gmail.com}}}

\address{Mikhail Ovcharenko \newline
   \textnormal{Steklov Mathematical Institute of Russian Academy of Sciences, 8 Gubkina St., Moscow 119991, Russia. \newline
   National Research University Higher School of Economics, Laboratory of Mirror Symmetry, NRU HSE, 6 Usacheva str., Moscow 119048, Russia.} \newline
   \textnormal{\texttt{michael.a.ovcharenko@gmail.com, ovcharenko@mi-ras.ru}}}

\address{Victor Przyjalkowski \newline
   \textnormal{Steklov Mathematical Institute of Russian Academy of Sciences, 8 Gubkina St., Moscow 119991, Russia. \newline
   National Research University Higher School of Economics, Laboratory of Mirror Symmetry, NRU HSE, 6 Usacheva str., Moscow 119048, Russia.} \newline
   \textnormal{\texttt{victorprz@mi-ras.ru, victorprz@gmail.com}}}

\begin{document}

\begin{abstract}
  For each smooth Fano threefold, we construct a family of Landau--Ginzburg models which satisfy many expectations coming from different aspects of mirror symmetry; they are log Calabi--Yau varieties with proper superpotential maps; they admit open algebraic torus charts on which the superpotential function \(\mathsf{w}\) restricts to a Laurent polynomial satisfying a deformation of the Minkowski ansatz~\cite{akhtar2012minkowski}; the general fibres of \(\mathsf{w}\) are Dolgachev--Nikulin dual to the anticanonical hypersurfaces in the initial Fano threefold. To do this, we study the deformation theory of Landau--Ginzburg models in arbitrary dimension, following~\cite{katzarkov2017bogomolov}, specializing to the case of Landau--Ginzburg models obtained from Laurent polynomials. Our proof of Dolgachev--Nikulin mirror symmetry is by detailed case-by-case analysis, refining work of Cheltsov and the fifth-named author~\cite{cheltsov2018katzarkov}.
\end{abstract}

\maketitle

\section{Introduction}

\subsection{Relations between mirror symmetry predictions for Landau--Ginzburg models}

Mirror symmetry studies relations between symplectic geometry and complex algebraic geometry. In particular, the form of mirror symmetry that we are interested in here is the relationship between Fano manifolds and their mirror Landau--Ginzburg models. For the moment, a Landau--Ginzburg model will denote simply a smooth quasi-projective variety \(Y\) equipped with a regular function \(\mathsf{w}\). In the past decades, many different inter-related forms of mirror symmetry have been proposed. One of the goals of this paper is to understand how predictions coming from different forms of mirror symmetry relate to one another. In particular, if \(X\) is a Fano variety, and \((Y,\mathsf{w})\) is its Landau--Ginzburg mirror, we have the following predictions, which we state in the case where \(\dim(X) = 3\). Our discussion of mirror symmetry in this section is quite coarse, and we often suppress specific choices of symplectic and complex structures for simplicity.

\begin{itemize}[\quad -]

\item \textbf{Homological mirror symmetry} (e.g.,~\cite{katzarkov2008hodge,katzarkov2017bogomolov}). Homological mirror symmetry, initiated by Kontsevich for Calabi--Yau varieties in~\cite{kontsevich1994homological} and extended to the case of certain Fano varieties (see~\cite{seidel2001vanishing} for exposition), predicts that for a Fano manifold with smooth anticanonical divisor \(V\), the log Calabi--Yau pair \((X,V)\) has a mirror log Calabi--Yau pair \((Z,D)\). The act of compactifying \(X - V\) to \(X\) corresponds under mirror symmetry to equipping \(Y = Z-D\) with a proper function \(\mathsf{w}\) (see~\cite{auroux2007mirror}). Therefore, if \(X\) and \((Y,\mathsf{w})\) form a mirror Fano/Landau--Ginzburg model pair, one expects that \(Y\) admits a log Calabi--Yau compactification \(Z\) to which \(\mathsf{w}\) extends to a morphism \(\mathsf{f} \colon Z\rightarrow \mathbb{P}^1\). The fibre over \(\infty\) is snc and anticanonical.

\item \textbf{Hodge-theoretic mirror symmetry} (e.g.,~\cite{iritani2009integral,iritani2016quantum}). Hodge-theoretic mirror symmetry predicts the identification of the regularized quantum cohomology D-module of a Fano variety \(X\) with the Gauss--Manin connection on the fibres of \(\mathsf{w}: Y\rightarrow \mathbb{C}\). This formulation of Hodge-theoretic mirror symmetry goes back to Givental~\cite{givental1998mirror}, with the regularization introduced by Golyshev~\cite{golyshev2004modularity}. Iritani~\cite{iritani2009integral} and concurrent work of Katzarkov--Kontsevich--Pantev~\cite{katzarkov2008hodge} equipped the (regularized) quantum D-module with an integral structure  which should match the natural integral structure underlying the B-model variation of Hodge structure. Furthermore, the ambient quantum D-module of an anticanonical Calabi--Yau hypersurface in \(X\) is identified with a sub-local system of the solution sheaf of the quantum D-module~\cite{iritani2011quantum,iritani2016quantum}.

  In dimension 3, if the Picard lattice of a very general anticanonical hypersurface of \(X\) is denoted \(\Pic(X)\), then one expects that the transcendental lattice of a very general fibre of \(\mathsf{w}\) is \(H \oplus \Pic(X)\). Here \(H\) indicates the unique rank 2 even unimodular lattice of signature \((1,1)\). This relation is called \emph{Dolgachev--Nikulin mirror symmetry}~\cite{dolgachev1996mirror}. For a discussion of this aspect of mirror symmetry, see~\cite{ueda2020mirror}. This can also be extracted from homological mirror symmetry.

\item \textbf{Fanosearch programme} (e.g.,~\cite{coates2014mirror,coates2016quantum,coates2021maximally}). In an ongoing series of papers, Coates, Corti, Kasprzyk, and a number of collaborators have pursued a program aimed towards understanding the classification of Fano varieties through mirror symmetry and Laurent polynomials. The basic observation, which goes back in some form at least to the work of Batyrev--Ciocan-Fontanine--Kim--van Straten~\cite{batyrev2000mirror}, is that if \(X\) is a Fano variety, and \(X\) admits a degeneration to a Gorenstein Fano toric variety \(T\), then the mirror of \(X\) admits a torus chart \((\mathbb{C}^*)^n\) to which \(\mathsf{w}\) restricts to give a Laurent polynomial \(\mathsf{p}\) whose Newton polytope is the anticanonical polytope of \(T\). These Laurent polynomials do not have general coefficients, and the choice of coefficients seems to have a deep connection to the structure of the degeneration of \(X\) to \(T\).

  One of the main ideas appearing in the work of Coates, Corti, Kasprzyk, and collaborators is that this process can be inverted; by characterizing the Laurent polynomials that appear in this way, one should be able to \emph{create} a family of Fano manifolds starting from the data of a Laurent polynomial of appropriate type by applying deformation-theoretic techniques (see, for instance,~\cite{coates2019inversion,akhtar2016mirror}). The main challenge has been to characterize which polynomials correspond to Fano varieties. The current expectation is that there is a bijection between mutation classes of \emph{rigid maximally mutable Laurent polynomial} (see Subsection~\ref{subsection:mink}) and TG Fano varieties (see~\cite{coates2021maximally} for precise details).
\end{itemize}

These three strands of mirror symmetry are deeply interwoven; the Fanosearch programme is influenced by Hodge-theoretic mirror symmetry, and one expects that homological mirror symmetry at least partially implies Hodge-theoretic mirror symmetry. The goal of this article is to show that the predictions made about the class of objects mirror to Fano varieties by these three aspects of mirror symmetry are in harmony with one another. We provide a more detailed outline of our results below.

\subsection{Picard lattices of Landau--Ginzburg threefolds}

The theory of toric Landau--Ginzburg models (see Subsection~\ref{subsection:LG-models} for definition or~\cite{przyjalkowski2018review} for a comprehensive overview) gives an effective approach of constructing mirror Landau--Ginzburg models of Fano varieties. According to this theory, a Fano variety \(X\) corresponds to a Laurent polynomial \(\mathsf{p}\), interpreted as a regular function on \((\mathbb{C}^*)^d\), so that the periods of the fibres of \(\mathsf{p}\) correspond to the Gromov--Witten invariants of \(X\). Motivated by mirror symmetry, one expects~\cite{przyjalkowski2017compactification} that \((\mathbb{C}^*)^d\) also admits a compactification \(Z\) so that \(\mathsf{p}\) extends to a morphism \(\mathsf{f} \colon Z\rightarrow \mathbb{P}^1\), and so that \(\mathsf{f}^{-1}(\infty)\) is a simple normal crossings anticanonical divisor. Finally, overwhelming amounts of computations coming from the Fanosearch programme lead us to believe that there  should exist a degeneration of \(X\) to a toric variety \(T\) whose fan polytope is the Newton polytope of the Laurent polynomial (see Subsection~\ref{subsection:LG-models}). A Laurent polynomial satisfying all of these expectations is said to be a \emph{toric Landau--Ginzburg model} of \(X\).

It was proved in~\cite{przyjalkowski2008LG,przyjalkowski2013weak,akhtar2012minkowski,coates2016quantum} that smooth Fano threefolds have toric Landau--Ginzburg mirrors. Note that the corresponding Laurent polynomial is not uniquely determined by a Fano variety. Nevertheless, if the anticanonical class \(-K_X\) is very ample, then the toric Landau--Ginzburg model of a smooth Fano threefold \(X\) is provided by a \emph{Minkowski polynomial} (see Subsection~\ref{subsection:minkowski}). It was proved in~\cite{akhtar2012minkowski} that they are all related by birational transformations called \emph{mutations}. Thus one can choose any Laurent polynomial from~\cite[Appendix~B]{akhtar2012minkowski} among mirror partners for \(X\) as well.

Dolgachev--Nikulin duality is a form of mirror symmetry between families of lattice-polarized K3 surfaces: for any lattice \(L\) (under some natural restrictions on the lattice) and any complete family of \(L\)-polarized K3 surfaces there is a corresponding complete family of \(L^{\vee}\)-polarized K3 surfaces. Here \(L^{\vee}\) is the \emph{Dolgachev--Nikulin dual lattice} (see Subsection~\ref{subsection:DN-duality} for a brief review, and~\cite{dolgachev1996mirror} for further details). Given a Fano threefold \(X\), there is a lattice polarization on each smooth anticanonical divisor \(S\) of \(X\), obtained by taking the sublattice \(\im (H^2(X,\mathbb{Z})\rightarrow H^2(S,\mathbb{Z}))\) and equipping this sublattice with the induced bilinear form. Let us denote this lattice by \(\Pic(X)\). Beauville~\cite{beauvillefano} has shown that for a very general choice of \(S\), there is an isomorphism between \(\Pic(X)\) and \(\Pic(S)\), and that the deformations of pairs \((X, S)\) form a complete family of \(\Pic(X)\)-polarized K3 surfaces of dimension \(20 - \rho(X)\). Consequently, Dolgachev--Nikulin duality implies that a complete family of \(\Pic(X)^{\vee}\)-polarized K3 surfaces has dimension \(\rho(X)\), and the Picard lattice of its very general fibre is isomorphic to \(\Pic(X)^{\vee}\) (see~\cite[MS2\('\)]{dolgachev1996mirror}).

In~\cite{ueda2020mirror}, Ueda applies work of Iritani~\cite{iritani2009integral} to prove, modulo some lattice-theoretic details, that if \(T\) is a smooth toric Fano threefold, then the fibres of the Landau--Ginzburg mirror of \(T\) are also Dolgachev--Nikulin dual to anticanonical hypersurfaces in \(T\). In this case, the Landau--Ginzburg mirror of \(T\) is \((\mathbb{C}^*)^3\) equipped with a general Laurent polynomial \(\mathsf{p}\) with Newton polytope equal to the fan polytope of \(T\). Since Iritani's results on mirror symmetry for variations of integral Hodge structures are not known for general Fano threefolds, Ueda's arguments do not apply more generally. However, we are led to the following conjecture.

\begin{conjecture}\label{conjecture:DN-duality}
  Let \(X\) be a smooth Fano threefold, and \(F\) be a general fibre of its log Calabi--Yau compactified toric Landau--Ginzburg model \(\mathsf{f} \colon Z \rightarrow \mathbb{P}^1\). Then there is an isomorphism of lattices:
  \[
    \Pic(X)^{\vee} \cong \im \left ( H^2(Z, \mathbb{Z}) \xrightarrow{\res} H^2(F, \mathbb{Z}) \right ).
  \]
\end{conjecture}

\begin{proposition}[see~{\cite[Theorem~5.25]{przyjalkowski2018review}}]\label{proposition:prrev}
  Conjecture~\ref{conjecture:DN-duality} holds for Fano threefolds with \(\rho(X) = 1\).
\end{proposition}

\begin{definition}\label{definition:LG-standard}
  Let \(X\) be a smooth Fano threefold with \(\rho(X) > 1\). We refer to the toric Landau--Ginzburg model of the Fano threefold \(X\) used in~\cite{cheltsov2018katzarkov} as a \emph{standard} Landau--Ginzburg mirror of \(X\).
\end{definition}

\begin{theorem}\label{theorem:DN-duality}
  Conjecture~\ref{conjecture:DN-duality} holds for standard toric Landau--Ginzburg models.
\end{theorem}

\subsection{Deformations of standard Landau--Ginzburg models}

For each Fano threefold, we have only a single standard Landau--Ginzburg model, therefore, in general, the fibres of \(\mathsf{w} \colon Y\rightarrow \mathbb{A}^1\) do not form a complete family of \(\Pic(X)^\vee\)-polarized family of K3 surfaces. Even worse, the Picard lattice of a (very) general fibre of~\(\mathsf{w}\) can be \emph{strictly larger} than \(\Pic(X)^\vee\). Following Cheltsov--Przyjalkowski (see~\cite{cheltsov2018katzarkov}), the Dolgachev--Nikulin dual lattice (i.e., the lattice of monodromy invariants) arises as the \(\Gal(\mathbb{C}(t))\)-invariant part of some larger sublattice in the Picard lattice of a general fibre. It was noticed in~\cite{cheltsov2018katzarkov} that the \(\Gal(\mathbb{C}(t))\)-action is non-trivial for Families~2.12,~4.2, and~6.1--10.1.

By work of the fifth-named author~\cite{przyjalkowski2017compactification}, each standard Landau--Ginzburg model associated to a Fano threefold with very ample anticanonical bundle admits a tame compactification (see Definition~\ref{definition:tcom} below), and thus results of Katzarkov--Kontsevich--Pantev~\cite{katzarkov2017bogomolov} show that the deformation theory of \((Z, D, \mathsf{f})\) is unobstructed. Results of Cheltsov--Przyjalkowski~\cite{cheltsov2018katzarkov}, along with results in Subsection~\ref{subsection:def} below, show that the space of deformations of standard tame compactified Landau--Ginzburg models \((Z, D, \mathsf{f})\) has dimension \(h^{1,1}(X) + 2\), and the deformation space of pairs \((Z,F)\), where \(F\) is a fibre of \(\mathsf{f}\), has dimension \(\rk(\Pic(X))\). Since the moduli space of \(L_Z\)-polarized K3 surfaces has dimension \(20 - \rk(L_Z) = \rk(\Pic(X))\) (where the last equality follows from Theorem~\ref{theorem:DN-duality}), it reasonable to expect that the forgetful map \(\Def(Z,F)\rightarrow \Def(F)\) is surjective onto the space of deformations which preserve \(L_Z\)-polarization. This is indeed the case, as we prove in Proposition~\ref{proposition:main}. More generally:

\begin{theorem}
  Suppose \((Z,D,\mathsf{f})\) is a tame compactified Landau--Ginzburg model of dimension 3 so that \((Z,D)\) is a log Calabi--Yau pair satisfying certain minor topological conditions. Let \(F\) be a smooth fibre of \(\mathsf{f}\).
  \begin{enumerate}
  \item The forgetful map \(\Def(Z,D,\mathsf{f})\rightarrow \Def(Z)\) is surjective, hence deformations of \(Z\) are unobstructed\footnote{This result is essentially claimed in~\cite[Corollary 2.18]{katzarkov2017bogomolov}, however, the proof there is incomplete. We fill in the gaps in the proof of {op. cit.} in Subsection~\ref{subsection:def}.}.
  \item The forgetful map \(\Def(Z,F) \rightarrow \Def(F)\) is a submersion of relative dimension \(h^{2,1}(Z)\) onto the subspace of \(\Def(F)\) preserving \(L_Z\)-polarization.
  \end{enumerate}
  Consequently, for a very general deformation of \((Z,D,\mathsf{f})\) and a very general fibre \(F\) of \(\mathsf{f}\), the Picard lattice of \(F\) is isomorphic to \(L_Z\).
\end{theorem}

This result is quite abstract. It tells us very little about the nature of deformations of \((Z,D,\mathsf{f})\). Our next result describes this deformation space algebraically. As mentioned above, the Landau--Ginzburg model \((Z,D,\mathsf{f})\) can be constructed from a specific type of Laurent polynomial called a Minkowski polynomial. The construction was first given in~\cite{przyjalkowski2017compactification} and recalled below in Proposition~\ref{proposition:prz}. We recall the definition of a Minkowski polynomial in Subsection~\ref{subsection:minkowski}, but there are two important points: (a) a Minkowski decomposition of each face of \(\Newt(\mathsf{p})\) along with a corresponding factorization of the face polynomials of \(\mathsf{p}\), and (b) a normalization condition: for each vertex \(\nu\) of \(\Newt(\mathsf{p})\) the corresponding coefficient is equal to \(1\). Dropping the normalization requirement, we obtain a family of Laurent polynomials. We call such a polynomial a \emph{parametrized Minkowski polynomial}. Przyjalkowski's method allows us to construct a tame compactified Landau--Ginzburg model for any parametrized Minkowski polynomial \(\mathsf{p}\), and thus an algebraic family of tame compactified Landau--Ginzburg models. We prove the following result in Subsection~\ref{subsection:LG-deformations}.

\begin{theorem}
  Suppose \(\mathsf{p}\) is a parametrized Minkowski polynomial of dimension at least \(3\). Any small deformation of \((Z_\mathsf{p},D_\mathsf{p},\mathsf{f})\) is obtained by varying the coefficients of \(\mathsf{p}\).
\end{theorem}
\begin{remark}
    The variation of coefficients of \(\mathsf{p}\) corresponding to deformations of \((Z_\mathsf{p},D_\mathsf{p}, \mathsf{f})\) is not arbitrary. See Theorem~\ref{theorem:defmink} below for more details.
\end{remark}

To every Laurent polynomial we attach what we call \emph{Minkowski data}, corresponding to a factorization of the face polynomials of \(\mathsf{p}\) of its Newton polytope \(P\). Minkowski data are partially ordered by refinement. It is clear that if \(\mathsf{p}\) is a Laurent polynomial, the Minkowski data of \(\mathsf{p}\) is a refinement of any deformation of \(\mathsf{p}\). A subtle point in the discussion below is that this need not be an equality. The following statement is a direct consequence of Theorem~\ref{theorem:defmink}.

\begin{proposition}\label{proposition:invariance}
  Let \(X\) be a smooth Fano threefold with very ample anticanonical class, and let \((Z,D,\mathsf{f})\) be its standard Landau--Ginzburg model, obtained by partially compactifying the corresponding Laurent polynomial \(\mathsf{p}\) of \(X\). There is a \(\rk (\Pic(X))\)-dimensional family of Landau--Ginzburg models deformation equivalent to \((Z,D,\mathsf{f})\) so that the following statements hold.
  \begin{enumerate}
  \item Any small deformation of \(Z\) is obtained by deforming \(\mathsf{p}\).
  \item The deformation space of pairs \((Z,F)\) form a complete family of \(\Pic(X)^\vee\)-polarized K3 surfaces.
  \end{enumerate}
\end{proposition}

\begin{remark}[Modularity of Landau--Ginzburg mirrors of Fano threefolds]
  In~\cite{golyshev2004modularity,golyshev2007classification} Golyshev studied quantum cohomology of Fano threefolds of Picard rank 1 and showed that their regularized quantum differential equations come from symmetric squares of uniformizing differential equations of classical modular curves. In~\cite{ilten2013toric}, Ilten, Lewis, and the fifth-named author show that the regularized quantum differential operator can be obtained as the Picard--Fuchs operator attached to a Laurent polynomial whose fibres are K3 surfaces of Picard rank 19.

  The original motivation for this paper was to understand how this statement generalizes to higher Picard rank. Together, Propositions~\ref{proposition:prrev} and~\ref{proposition:invariance} tell us that for any Fano threefold \(X\) with very ample anticanonical bundle there is a \((\rk(\Pic(X)) -1)\)-dimensional family of Landau--Ginzburg models \((Z_t,D_t,\mathsf{f}_t)\). Here \(t\) denotes a parameter for this family of Landau--Ginzburg models. Furthermore, we see that this deformation is obtained by deforming the Laurent polynomial \(\mathsf{p}\) defining the standard Landau--Ginzburg model of \(X\). Finally, we see that as \(t\) varies, the fibres of \(\mathsf{f}_t\) form a family of \(\Pic(X)^\vee\)-polarized K3 surfaces, and that this family of K3 surfaces have a finite, dominant period map onto the moduli space of \(\Pic(X)^\vee\)-polarized K3 surfaces. This should be viewed as a generalization of Golyshev's original results to higher Picard rank. We still lack a complete understanding of what exactly this map looks like, but a few examples are computed in Section~\ref{section:ex}.

  Moreover, Proposition~\ref{proposition:surj} and Theorem~\ref{theorem:DN-duality} combined imply that for any smooth Fano threefold with very ample anticanonical bundle the corresponding moduli space of \(\Pic(X)^{\vee}\)-polarized K3 surfaces is uniruled (see Corollary~\ref{corollary:ssur}). In particular, we obtain that the moduli space of \(\mathbb{E}_8(-1)^2 \oplus \langle 2n \rangle\)-polarized K3 surfaces is uniruled for \(n = 6,7,8,10,11,13,14,16,17\), which generalizes a known result of Gritsenko--Hulek for \(n = 21\) in~\cite{gritsenko1998minimal}. See Subsubsection~\ref{subsubsection:siegel-hilbert} for more details and further discussion of related topics.
\end{remark}

\subsection{Parametrized Landau--Ginzburg models}

We expect that the families of Landau--Ginzburg models obtained by deforming \(\mathsf{p}\), as in Proposition~\ref{proposition:invariance}, form mirrors of Fano varieties equipped with general complexified divisor classes.

\begin{definition}
  Let \(X\) be a smooth Fano variety, and \(\Pic(X) \simeq \mathbb{Z}^{\rho(X)}\) be any choice of basis. We refer to a Laurent polynomial \(\mathsf{p} \in \mathbb{Z}[a_1^{\pm 1}, \ldots, a_{\rho(X)}^{\pm 1}][x_1^{\pm 1}, \ldots, x_{\dim(X)}^{\pm 1}]\) as a \emph{parametrized Landau--Ginzburg model} if for any choice of the complexified divisor class \(K = (\alpha_1, \ldots, \alpha_{\rho(X)}) \in \Pic(X) \otimes \mathbb{C} \) the induced parameter specialization \(\mathsf{p}_K = \mathsf{p}(\{a_i = \exp(-\alpha_i)\})\) is a toric Landau--Ginzburg model for the pair \((X, K)\).
\end{definition}

\begin{definition}
  Let \(X\) be a smooth Fano threefold, and \(\mathsf{f}\) be its parametrized Landau--Ginzburg model. We refer to \(\mathsf{f}\) as \emph{standard} if its specialization \(\mathsf{f}_0\) is a standard toric Landau--Ginzburg model for \(X\).
\end{definition}

As an illustration, we have explicitly constructed the isomorphism from Conjecture~\ref{conjecture:DN-duality} for parametrized Landau--Ginzburg models for smooth Fano threefolds of the form \(S \times \mathbb{P}^1\), where \(S\) is a smooth del Pezzo surface with very ample anticanonical class. In this case we are able to verify the period condition.

\begin{proposition}[see Appendix~\ref{appendix:parametrized}]\label{proposition:parametrized}
  There exist standard parametrized toric Landau--Ginzburg models \(\mathsf{f} \colon Z_{(X, K)} \rightarrow \mathbb{P}^1\) for the families of Fano varieties \textnumero2.11,~3.27,~3.28,~4.10,~5.3,~6.1,~7.1,~8.1. In these cases for any complexified divisor class \(K\) there is a natural isomorphism
  \[
    \Pic(X)^{\vee} = \im \left ( H^2(Z_{(X, K)}, \mathbb{Z}) \xrightarrow{\res} H^2(F_{(X, K)}, \mathbb{Z}) \right ).
  \]
\end{proposition}

\begin{remark}
    While the results in Proposition~\ref{proposition:parametrized} are stated abstractly, we note that the parametrized Landau--Ginzburg models in question are in fact constructed in Appendix~\ref{appendix:parametrized}.
\end{remark}

\subsection{Outline}

The paper is structured as follows. In Section~\ref{section:modlg} we study the deformation theory of Landau--Ginzburg models and related objects. We extend~\cite[Corollary 2.18]{katzarkov2017bogomolov} and show that if \((Z,D,\mathsf{f})\) is a Landau--Ginzburg model of dimension at least \(3\) satisfying certain mild topological conditions, then the natural maps \(T_s \Def(Z,D,\mathsf{f})\rightarrow T_s \Def(Z,D)\) and \(T_s \Def(Z,D) \rightarrow T_s \Def(Z)\) are surjective, and, therefore, that \((Z,D)\) and \(Z\) have unobstructed deformations, and that any small deformation  \(Z'\) of \(Z\) admits a function \(\mathsf{f}':Z'\rightarrow \mathbb{P}^1\) so that \((Z',D'= (\mathsf{f}')^{-1}(\infty), \mathsf{f}')\) is a deformation of \((Z,D,\mathsf{f})\).
As a consequence, if \(F\) is a smooth fibre of \(\mathsf{f}\), then the deformations of \((Z,F)\) are unobstructed.
In Subsection~\ref{subsection:iacono-beauville} we compute the image of the map \(T_s \Def(Z,F) \rightarrow T_s \Def(F)\), following Beauville~\cite{beauvillefano}.

In Section~\ref{section:wreg}  we focus on Landau--Ginzburg models which are obtained from Laurent polynomials. We define a class of Laurent polynomials called \emph{weakly non-degenerate \(\mathsf{M}\)-polynomials}, which generalize the Minkowski polynomials of~\cite{akhtar2012minkowski}. We show that weakly non-degenerate \(\mathsf{M}\)-polynomials give rise to well-behaved Landau--Ginzburg models, and we show that if \((Z,D,\mathsf{f})\) is such a Landau--Ginzburg model coming from a Laurent polynomial \(\mathsf{p}\), then any deformation of \((Z,D,\mathsf{f})\) comes from a deformation of \(\mathsf{p}\).

Section~\ref{section:k3} specializes our results to the case where \(\dim(Z) = 3\). In Subsection~\ref{subsection:DN-duality} we recall basic notions regarding lattice polarizations of K3 surfaces and Dolgachev--Nikulin duality. Then we show that if \((Z,D,\mathsf{f})\) is a Landau--Ginzburg threefold satisfying certain mild topological conditions, and its fibres are K3 surfaces, then the fibres of \(Z\) and its deformations form a complete family of lattice-polarized K3 surfaces. In Section~\ref{section:ex} we compute several examples which are related to well-known modular families of K3 surfaces.

In Section~\ref{section:DN-duality} we describe our proof of Dolgachev--Nikulin duality for anticanonical hypersurfaces of Fano threefolds and fibres of their standard toric Landau--Ginzburg mirrors. We describe computational tools used to compute the ambient Picard lattice of the fibres of standard Landau--Ginzburg mirrors of Fano threefolds and lattice-theoretic tools used to compare abstract lattices. We then describe the computational and technical tools used to compute Picard lattices for anticanonical divisors in Fano threefolds. The outcome of these computations can be found in Appendix~\ref{appendix:picard-lattices-fano}.

Finally, in Appendices~\ref{appendix:rank-02},~\ref{appendix:rank-03},~\ref{appendix:rank-04},~\ref{appendix:higher-ranks} we prove Theorem~\ref{theorem:DN-duality} for smooth Fano threefolds of ranks 2, 3, 4, and 5--10, respectively. All these sections are split by subsections whose numbers match the numbers of families of smooth Fano threefolds given in~\cite{coates2016quantum}. In Appendix~\ref{appendix:parametrized} we prove Proposition~\ref{proposition:parametrized} and explicitly construct the corresponding parametrized Landau--Ginzburg models.

\begin{conventions}
  Unless otherwise stated, all varieties are taken to be over \(\mathbb{C}\). We follow the labelling of families of Fano threefolds of Picard rank greater than 1 given by Mori--Mukai in~\cite{mori1981classification}. We refer to the family of Fano threefolds of rank \(4\) missing in the original classificaiton as Family \textnumero 4.13.
\end{conventions}

\begin{acknowledgements}
  The authors would like to thank Jacob Lewis for his contributions to this work in its early stages. The authors would also like to acknowledge the following support:
  \begin{enumerate}[(1)]
  \item The work of C.F.D. was supported by the McCalla Professorship of Science at the University of Alberta, the Visiting Campobassi Professorship of Physics at the University of Maryland, the Visiting Distinguished Professorship of Mathematics and Physics at Bard College,  the Natural Sciences and Engineering Research Council of Canada, and from Harvard University's Center for Mathematical Sciences and Applications.
  \item A.H. has received research support from the Natural Sciences and Engineering Research Council of Canada, and from Harvard University's Center for Mathematical Sciences and Applications, and from the Simons Foundation Collaboration in Homological Mirror Symmetry. He has received Simons Travel Support for Mathematicians.
  \item L.K. was supported by Simons Foundation Collaboration in Homological Mirror Symmetry, Simons Investigators award and by the Basic Research Program of the National Research University Higher School of Economics. He was funded by the National Science Fund of Bulgaria, National Scientific Program ``VIHREN'', Project no. KP-06-DV-7.
  \item M.O. was supported by the project ``International academic cooperation'' National Research University Higher School of Economics. V.P. was supported by the Basic Research Program of the National Research University Higher School of Economics.
  \end{enumerate}
\end{acknowledgements}

\section{Modularity of tame compactified Landau--Ginzburg models and their fibres}\label{section:modlg}

In this section we prove that the fibres of tame compactified Landau--Ginzburg models satisfying some mild topological conditions (see conditions (a), (b), and (c) in Theorem~\ref{theorem:main1}) form complete families of Calabi--Yau varieties. The arguments in this section owe much to Katzarkov--Kontsevich--Pantev~\cite{katzarkov2017bogomolov} and Beauville~\cite{beauvillefano}.

\subsection{Deformations of tame compactified Landau--Ginzburg models}\label{subsection:def}

The main objects that we will study in this paper are Landau--Ginzburg models. In a very broad sense, a Landau--Ginzburg model is simply a quasi-projective variety equipped with a function. We will restrict ourselves to those Landau--Ginzburg models that we expect to contain the class of Landau--Ginzburg models which are mirror to Fano varieties.

\begin{definition}\label{definition:tcom}
  A \emph{proper, tame compactified Landau--Ginzburg model} is a triple \((Z,D,\mathsf{f})\) consisting of a smooth projective variety \(Z\), a simple normal crossings (snc) divisor \(D\), and a morphism \(\mathsf{f} \colon Z \rightarrow \mathbb{P}^1\) so that \(\mathsf{f}^*(\infty) = D\). We say that a tame compactified Landau--Ginzburg model satisfies the \emph{Calabi--Yau condition} if \(D\) is an anticanonical divisor of \(Z\).

  We use the notation \(Y\) to denote \(Z\setminus D\), and \(\mathsf{w}\) denotes the restriction of \(\mathsf{f}\) to \(Y\).
\end{definition}

In~\cite{katzarkov2017bogomolov}, the third-named author, along with Kontsevich and Pantev, study the deformation theory and the Hodge theory of Landau--Ginzburg models. A deformation of Landau--Ginzburg models is a quadruple \((\mathscr{Z},\mathscr{D},\bm{f}, \varpi)\), where \(\mathscr{Z}\) is a smooth manifold, \(\mathscr{D}\) is a snc divisor in \(\mathscr{Z}\), \({\bm f} \colon \mathscr{Z}\rightarrow \mathbb{P}^1\) is a morphism so that \(\bm{f}^{-1}(\infty) = \mathscr{D}\), and where \(\varpi \colon \mathscr{Z}\rightarrow B\) is a smooth projective morphism so that for each \(b \in B\), the preimage \(Z_b = \varpi^{-1}(b)\)  along with \(D_b = Z_b \cap \mathscr{D}_b\) and \(\mathsf{f}_b = \bm{f}|_{Z_b}\) form a tame compactified Landau--Ginzburg model.

Let \(T_Z(-\log D)\) denote the sheaf of tangent vectors to \(Z\) which vanish logarithmically along \(D\). Then the deformations of the pair \((Z,D)\) are controlled by the sheaf \(T_Z(-\log D)\), and the deformations of the triple \((Z,D,\mathsf{f})\) are controlled by the complex
\[
  \mathfrak{g}^\bullet = \left[  T_Z(-\log D) \xrightarrow{\mathrm{d}\mathsf{f}} \mathsf{f}^*T_{\mathbb{P}^1}(-\log \infty)\right].
\]
Observe that the global sections of \(T_{\mathbb{P}^1}(-\log \infty)\cong \mathcal{O}_{\mathbb{P}^1}(1)\) can be written as \((\lambda z + \mu)\partial_z\) on the affine chart with coordinate \(z\). Therefore, global sections of \(\mathsf{f}^{-1}T_{\mathbb{P}^1}(-\log \infty)\) are of the form \((\lambda \mathsf{f}  + \mu)\partial_\mathsf{f}\), and sections of \(\mathsf{f}^*T_{\mathbb{P}^1}(-\log \infty)\) are expressed similarly.

Following Horikawa~\cite{horikawa1976deformations}, we may give a simple presentation of the Kodaira--Spencer map. We represent the hypercohomology of \(\mathfrak{g}^\bullet\) by choosing a Stein covering of \(Z\), which we denote \(\mathfrak{U}\). We then obtain a \v{C}ech resolution of \(\mathfrak{g}^\bullet\) which represents \(\mathbb{H}^1(Z,\mathfrak{g}^\bullet)\) as \(H^1\) of the complex
\begin{align*}
  0 \rightarrow C^0(\mathfrak{U},T_Z(-\log D)) \rightarrow C^0(\mathfrak{U},\mathsf{f}^*T_{\mathbb{P}^1}(-\log \infty))& \oplus C^1(\mathfrak{U},T_Z(-\log D))\\ & \rightarrow C^1(\mathfrak{U},\mathsf{f}^*T_{\mathbb{P}^1}(- \log \infty)) \oplus C^2(\mathfrak{U},T_Z(-\log D)) \rightarrow \cdots
\end{align*}
equipped with the induced differentials. Given a deformation of \((Z,D,\mathsf{f})\) over a base \(B\), which we denote \((\mathscr{Z},\mathscr{D},\bm{f})\), we associate the cohomology class in \(H^1(Z,T_Z(-\log D))\) in the following way. Assume that \(\widetilde{\mathfrak{U}} =\{\widetilde{U}_\alpha\}\) is a Stein covering of \(\mathscr{Z}\) which restricts to a Stein covering \(\mathfrak{U} = \{U_\alpha\}\) of \(Z\). Let \(g_{\alpha,\beta}\) be the transition functions for this atlas which satisfy the 1-cocycle condition. These transition functions satisfy the obvious conditions:
\begin{itemize}
\item \(g_{\alpha,\beta}^*\bm{f}|_{\widetilde{U}_{\alpha}} = \bm{f}|_{\widetilde{U}_\beta}\), and \(\bm{f}|_{s=0} = \mathsf{f}\);
\item if \(h_{\alpha}\) is a local function defining \(\mathscr{D}\cap \widetilde{U}_\alpha\), then \(g_{\alpha,\beta}^* h_\alpha = u_{\alpha,\beta}h_\beta\), where \(u_{\alpha,\beta}\) is a unit in \(\Gamma(\widetilde{U}_{\alpha,\beta},\mathcal{O}_{U_{\alpha,\beta}})\).
\end{itemize}
For a coordinate function \(t\) on \(B\) we define
\[
  \tau_\alpha = \partial_t \bm{f}_\alpha|_{t=0},\qquad \rho_{\alpha,\beta} = \partial_t g_{\alpha,\beta}|_{t=0}.
\]
It is elementary to check that \(\rho_{\alpha,\beta} - \rho_{\alpha,\gamma} + \rho_{\beta,\gamma} = 0\) on \(U_{\alpha,\beta,\gamma}\) for each triple of indices \(\alpha,\beta,\gamma\), and that \(\tau_\alpha -\tau_\beta = \mathrm{d} \mathsf{f}(\rho_{\alpha,\beta})\). Calculations in local coordinates can be used to show that the zeroes of \(\tau_\alpha\) and \(\rho_{\alpha,\beta}\) are logarithmic. Therefore, \((\tau_\alpha,\rho_{\alpha,\beta})\) form a \v{C}ech cycle in \(C^0(\mathfrak{U},\mathsf{f}^*T_{\mathbb{P}^1}(-\log \infty)) \oplus C^1(\mathfrak{U},T_Z(-\log D))\). Suppose \(b \in B\) is a point so that \((Z,D,\mathsf{f}) = (Z_b,D_b,\mathsf{f}_b)\), the assignment
\[
  \kappa \colon T_{B,b} \rightarrow \mathbb{H}^1(Z,\mathfrak{g}^\bullet)
\]
is called the \emph{Kodaira--Spencer} map of the deformation \((\mathscr{Z},\mathscr{D},\bm{f})\) of \((Z,D,\mathsf{f})\).

\begin{definition}
  Let \((Z,D,\mathsf{f})\) be a tame compactified Landau--Ginzburg model. Deformations of \((Z,D,\mathsf{f})\) are \emph{unobstructed} if there is a family \(\varpi: (\mathscr{Z},\mathscr{D})\rightarrow B\) over smooth base \(B\) so that the Kodaira--Spencer map is surjective.
\end{definition}

We have the following theorem which provides a starting point for our study.

\begin{theorem}[Katzarkov--Kontsevich--Pantev,~\cite{katzarkov2017bogomolov}]\label{theorem:kkp}
  Suppose \((Z,D,\mathsf{f})\) is a Landau--Ginzburg model satisfying the Calabi--Yau condition and for which \(H^1(Z,\mathbb{Q}) = 0\). Then deformations of \((Z,D,\mathsf{f})\) are unobstructed.
\end{theorem}

Any Landau--Ginzburg model \((Z,D,\mathsf{f})\) admits a deformation over the unit polydisc \(\Delta\) in \(\mathbb{A}^2\) with coordinate \((\lambda,\mu)\) which is obtained by deforming the potential function:
\begin{equation}\label{equation:trivdef}
  (\mathscr{Z},\mathscr{D},\bm{f}) = (Z\times \Delta, D\times \Delta, (1-\lambda)\mathsf{f} + \mu),\qquad \varpi \colon \mathscr{Z} = Z\times \Delta \rightarrow \Delta.
\end{equation}
In~\eqref{equation:trivdef} we interpret \((1 - \lambda)\mathsf{f} + \mu\) as scaling the potential function \(\mathsf{f}\) by \((1 - \lambda)\) and translating by \(\mu \in \mathbb{C} \subseteq \mathbb{P}^1\).

\begin{proposition}
  Let \((Z,D,\mathsf{f})\) be a proper tame compactified Landau--Ginzburg model.
  \begin{enumerate}
  \item \(T_s \Def(Z,D,\mathsf{f}) \cong H^1(Z,\mathfrak{g}^\bullet)\) sits in an exact sequence,
    \begin{equation}\label{equation:kkp}
      \cdots \rightarrow \mathbb{C}^2 \xrightarrow{\psi} T_s \Def(Z,D,\mathsf{f}) \cong \mathbb{H}^1(Z,\mathfrak{g}^\bullet) \rightarrow T_s \Def(Z,D)\cong H^1(Z,T_Z(-\log D))\rightarrow  \cdots
    \end{equation}
  \item The image of the \(\kappa \colon T_{\Delta,0} \rightarrow \mathbb{H}^1(Z,\mathfrak{g}^\bullet)\) for the family in~\eqref{equation:trivdef} is identified with the image of \(\psi\) in~\eqref{equation:kkp}.
  \end{enumerate}
\end{proposition}

\begin{proof}
  To prove the first claim, note that there is a short exact sequence of complexes of sheaves
  \[
    0 \rightarrow \mathsf{f}^*T_{\mathbb{P}^1} (-\log \infty) [1] \rightarrow \mathfrak{g}^\bullet \rightarrow   T_Z(-\log D) \rightarrow 0.
  \]
  Therefore, under the assumptions in the statement of the proposition, we have an exact sequence computing the hypercohomology of \(\mathfrak{g}^\bullet\),
  \[
    \cdots \rightarrow H^0(Z,\mathsf{f}^*T_{\mathbb{P}^1}(-\log \infty)) \rightarrow \mathbb{H}^1(Z,\mathfrak{g}^\bullet) \rightarrow H^1(Z,T_Z(-\log D))   \rightarrow H^1(Z,\mathsf{f}^*T_{\mathbb{P}^1}(-\log \infty)) \rightarrow \cdots
  \]
  Then we apply the fact that \(T_{\mathbb{P}^1}(-\log \infty) \cong \mathcal{O}_{\mathbb{P}^1}(1)\) and hence \(\mathsf{f}^*\mathcal{O}_{\mathbb{P}^1}(1) = \mathcal{O}_Z(D)\). Since \(D\) is the fibre of a morphism to \(\mathbb{P}^1\), we have \(h^0(Z,\mathcal{O}_Z(D)) = 2\). This proves the result.

  The second claim may be proved via straightforward and elementary analysis of the Kodaira--Spencer map as presented above. In this case, the transition functions \(g_{\alpha,\beta}\) are trivial, therefore, \(\rho_{\alpha,\beta} = 0\) for all \(\alpha,\beta\). We also calculate easily that given in any chart \(\widetilde{U}_\alpha = U_\alpha \times \Delta\), we have \(\partial_t((1-\lambda t)\mathsf{f} +(\mu t))|_{t=0} = (-\lambda \mathsf{f} + \mu)\partial_\mathsf{f}\). Therefore,
  \[
    \kappa(Z\times \Delta, D\times \Delta, (1-\lambda) \mathsf{f} + \mu) = (\{((1-\lambda) \mathsf{f} + \mu)|_{U_\alpha}\partial_\mathsf{f}\}, 0) \in C^0(\mathfrak{U},\mathsf{f}^*T_{\mathbb{P}^1}(-\log \infty)) \oplus C^1(\mathfrak{U},T_Z(-\log D)).
  \]
  Since the vector fields \((\lambda \mathsf{f} + \mu)|_{U_\alpha} \partial_\mathsf{f}\) are all restrictions of a global vector field, they form a cocycle in \(C^0(\mathfrak{U},\mathsf{f}^*T_{\mathbb{P}^1}(-\log \infty))\). The induced map \(H^0(Z,\mathsf{f}^*T_{\mathbb{P}^1}(-\log \infty))\rightarrow \mathbb{H}^1(Z,\mathfrak{g}^\bullet)\) sends global vector fields of the form \((\lambda \mathsf{f} + \mu)\partial_\mathsf{f}\) to the corresponding \v{C}ech cocycle. This concludes the proof.
\end{proof}

\begin{lemma}\label{lemma:surjmap}
  If \((Z,D)\) has unobstructed deformations and the map \(H^1(Z,T_Z(-\log D))\rightarrow H^1(Z,T_Z)\) is surjective, then \(Z\) has unobstructed deformations. Similarly, if \((Z,D,\mathsf{f})\) has unobstructed deformations and the forgetful map \(\mathbb{H}^1(Z,\mathfrak{g}^\bullet) \rightarrow H^1(Z,T_Z(-\log D))\) is surjective, then deformations of \((Z,D)\) are also unobstructed.
\end{lemma}

\begin{proof}
  For both \((Z,D)\) and \(Z\) there exist semi-universal deformation spaces and forgetful morphisms between them. The tangent map of this forgetful morphism is the map \(H^1(Z,T_Z(-\log D)) \rightarrow H^1(Z,T_Z)\). If \((\mathscr{Z},\mathscr{D})\) is a deformation of \((X,D)\) over a smooth base with surjective Kodaira--Spencer map, the induced deformation \(\mathscr{Z}\), therefore, also has surjective Kodaira--Spencer map.
\end{proof}

\begin{example}
  Suppose \(C_1\) and \(C_2\) are distinct cubic curves in \(\mathbb{P}^2\) meeting in 9 distinct points. Let \(Z\) be the blow up of \(\mathbb{P}^2\) at the intersection of \(C_1\) and \(C_2\). There is an induced map \(\mathsf{f} \colon Z \rightarrow \mathbb{P}^1\) so that \(\mathsf{f}^{-1}(\infty) = C_2\). According to a result of Horikawa~\cite{horikawa1976deformations}, deformations of \(Z\) come from deformations of the nine points blown up to obtain \(Z\). A simple dimension count shows that the collection of 9-tuples of points in \(\mathbb{P}^2\) which are the base locus of a pencil of cubic curves is of codimension \(2\).  Therefore, in this case, the map \(T_s \Def(Z,D,\mathsf{f})\rightarrow T_s \Def(Z)\) is not surjective, and, in particular, the unobstructedness of deformations of \((Z,D,\mathsf{f})\) does not imply unobstructed deformations of \(Z\) (cf.~\cite[Corollary 2.18]{katzarkov2017bogomolov}). In other words, there are deformations of rational elliptic surfaces which are not rational elliptic. In the next section, we will show that this phenomenon does not persist in higher dimensions.
\end{example}

\subsection{Calabi--Yau Landau--Ginzburg models}

In this section, we will wish to compute the map
\[
  T_s \Def(Z,D,\mathsf{f})\rightarrow T_s \Def(Z,D)
\]
by Hodge-theoretic means. To do this, we need the following construction from~\cite{katzarkov2017bogomolov}. Suppose that \((Z,D)\) satisfies the Calabi--Yau condition, and that \(\omega\) is a nonvanishing section of the line bundle \(\Omega^d_Z(\log D)\). Here \(d\) is the dimension of \(Z\). Then there is a map \(i_\omega: T^i_Z\rightarrow \Omega^{d-i}_Z(\log D)\). Define a subsheaf
\[
  \Omega_Z^i(\log D,\mathsf{f}) = \{ \eta \in \Omega_Z^i(\log D) \mid \mathrm{d}\mathsf{f} \wedge \eta \in \Omega_Z^{i+1}(\log D)\}.
\]
This is called the \emph{sheaf of \(\mathsf{f}\)-adapted logarithmic forms} on \(Z\). We let \(T^{d-i}_Z(-\log D, \mathsf{f})\) be the preimage of \(\Omega_Z^i(\log D,\mathsf{f})\) under \(i_\omega\). The morphism \(\wedge \mathrm{d}\mathsf{f}\) makes \(\Omega_{Z}^i(\log D,\mathsf{f})\) a complex of sheaves. We may define \(\mathfrak{G}^i = T^{1-i}_Z(-\log  D, \mathsf{f})\). Under the isomorphism \(\Omega_Z^\bullet(\log D, \mathsf{f})\rightarrow \mathfrak{G}^{1-d+\bullet}\) the map \(\wedge \mathrm{d}\mathsf{f}\) is identified with the contraction \(\iota_{\mathrm{d}\mathsf{f}}\). It is shown in~\cite{katzarkov2017bogomolov} that there is a quasi-isomorphism of complexes induced by the natural injection on sheaves
\[
  \sigma_{\geq 0} \mathfrak{G}^\bullet \hookrightarrow \mathfrak{g}^\bullet
\]
and isomorphisms
\[
  \left[ T_Z(-\log D, \mathsf{f} ) \xrightarrow{\iota_{\mathrm{d}\mathsf{f}}} \mathcal{O}_Z\right] \xrightarrow{ \iota_{\omega} } \left[ \Omega^{d-1}_Z(\log D,\mathsf{f}) \xrightarrow{\wedge \mathrm{d}\mathsf{f} }\Omega^d_Z(\log D)\right].
\]
Here we have used the fact that \(\Omega_Z^d(\log D,\mathsf{f}) = \Omega_Z^d(\log D)\) which follows directly from the definition.

\begin{proposition}\label{proposition:surjmap}
  Suppose \((Z,D,\mathsf{f})\) is a proper Landau--Ginzburg model satisfying the Calabi--Yau condition. If the map \(H^{1}(Z,\Omega^{d-1}_Z(\log D,\mathsf{f})) \rightarrow H^1(Z,\Omega^{d-1}_Z(\log D))\) is surjective, then so is \(T_s \Def(Z,D,\mathsf{f}) \rightarrow T_s \Def(Z,D)\).
\end{proposition}

\begin{proof}
  There is a commutative diagram
  \[
    \begin{tikzcd}
      \sigma_{\geq 0}\mathfrak{G}^\bullet \ar[d,"\mathrm{q.is}"] \ar[r] & T_Z(-\log D ,\mathsf{f} ) \ar[r,"\cong"] \ar[d,hookrightarrow] &  \Omega^{d-1}_Z(\log D,\mathsf{f}) \ar[d,hookrightarrow] \\
    \mathfrak{g}^\bullet \ar[r] & T_Z(-\log D) \ar[r,"\cong"] & \Omega_Z^{d-1}(\log D)
    \end{tikzcd}
  \]
  where the horizontal maps are the induced morphisms of complexes coming from truncation. We know that the map \(\mathbb{H}^1(Z,\mathfrak{g}^\bullet) \rightarrow H^1(Z,T_Z(-\log D))\) is identified with \(T_s \Def(Z,D,\mathsf{f}) \rightarrow T_s \Def(Z,D)\). The result then follows by a simple diagram chase after applying the cohomology functor.
\end{proof}

\begin{remark}
  In~\cite{shamoto2018hodge}, Shamoto shows that there is a cohomological mixed Hodge complex underlying the complex \((\Omega_Z^\bullet(\log D), d_\mathrm{dR})\). The injective map of complexes \((\Omega_Z^\bullet(\log D,\mathsf{f}),d_\mathrm{dR})\rightarrow (\Omega_Z^\bullet(\log D), d_\mathrm{dR})\) underlies a morphism of cohomological mixed Hodge complexes, therefore, the map \(H^1(Z,\Omega^{d-1}_Z(\log D,\mathsf{f}))\rightarrow H^1(Z,\Omega^{d-1}_Z(\log D))\) is equal to the induced map
  \[
    \Gr^F_1H^d(Z,(\Omega_Z^\bullet(\log D,\mathsf{f}),d_\mathrm{dR})) \rightarrow \Gr^F_1 H^d(Z,(\Omega_Z^\bullet(\log D), d_\mathrm{dR}))
  \]
  by strictness. This allows us to use tools from Hodge theory to determine the surjectivity of \(T_s \Def(Z,D,\mathsf{f}) \rightarrow T_s \Def(Z,D)\). However, as we will see below, under topological assumptions which are frequently satisfied, this is not necessary.
\end{remark}

\begin{proposition}\label{proposition:unoblog}
  Let \((Z,D,\mathsf{f})\) be a proper tame compactified Landau--Ginzburg model satisfying the Calabi--Yau condition with the property that \(H^1(Z,\mathbb{Q}) = 0\). Suppose \(d \geq 2\), and assume that the fibres of \(\mathsf{f}\) are generically smooth Calabi--Yau varieties. Then deformations of \((Z,D)\) are unobstructed.
\end{proposition}

\begin{proof}
  According to~\cite{shamoto2018hodge}, there is a short exact sequence of sheaves
  \[
    0 \rightarrow \Omega_Z^{d-1}(\log D,\mathsf{f}) \rightarrow \Omega_Z^{d-1}(\log D) \rightarrow \Omega_{Z_\infty/\Delta}^{d-1}(\log D)\rightarrow 0,
  \]
  where \(Z_\infty\) is a preimage under \(\mathsf{f}\) of a small disc \(\Delta\) around \(\infty\). According to~\cite{steenbrink1976limits}, \(H^1(Z_\infty,\Omega^{d-1}_{Z_\infty/\Delta})\) is isomorphic to \(H^1(F,\Omega^{d-1}_F)\), where \(F\) is a smooth fibre of \(\mathsf{f}\). By assumption, this is \(0\). Therefore, the induced map \(H^1(Z,\Omega_Z^{d-1}(\log D,\mathsf{f}))\rightarrow H^1(Z,\Omega_Z^{d-1}(\log D))\) is surjective. By Proposition~\ref{proposition:surjmap} we see that \(T_s \Def(Z,D,\mathsf{f}) \rightarrow T_s \Def(Z,D)\) is surjective. Applying Theorem~\ref{theorem:kkp} and Lemma~\ref{lemma:surjmap}, the result follows.
\end{proof}

\begin{remark}
  Observe that \(\dim T_s \Def(Z,D,\mathsf{f}) = \dim H^1(Z,\Omega_Z^{d-1}(\log D,\mathsf{f})) + 1\). Therefore, a direct analogue of the local Torelli theorem for Landau--Ginzburg models does not hold.
\end{remark}

\subsection{Relation between deformations of \texorpdfstring{\((Z,D)\)}{(Z,D)} and deformations of \texorpdfstring{\(Z\)}{Z}}

Next, we will look at the map \(H^1(Z,T_Z(-\log D)\rightarrow H^1(Z,T_Z)\). This map relates deformations of the pair \((Z,D)\) to deformations of \(Z\).

\begin{proposition}\label{proposition:lastprop}
  Suppose \((Z,D,\mathsf{f})\) is a tame compactified Landau--Ginzburg model satisfying the Calabi--Yau condition. Assume that all components of \(D\) satisfy \(H^0(D_i,\omega_{D_i}) = H^1(D_i,\omega_{D_i}) = 0\). Then the map \(H^1(Z,T_Z(-\log D))\rightarrow H^1(Z,T_Z)\) is an isomorphism. Consequently, by Lemma~\ref{lemma:surjmap}, if \((Z,D)\) admit unobstructed deformations, then so does \(Z\).
\end{proposition}

\begin{proof}
  There is a short exact sequence of sheaves
  \begin{equation}\label{equation:smooth}
    0 \rightarrow T_Z(-\log D) \rightarrow T_Z \rightarrow N_{D} \rightarrow 0,
  \end{equation}
  where \(N_{D}\) is the {equisingular normal sheaf} of \(D\) in \(Z\).  Kawamata~\cite[pp. 250]{kawamata1977deformations} observes that in the case where \(D\) is normal crossings with components \(D_1, \ldots, D_n\), we have an isomorphism \(N_{D}\cong \bigoplus_{i=1}^n k_{i*}N_{D_i}\) where \(k_i \colon D_i \rightarrow D\) is the natural inclusion map. By adjunction and the fact that \(\omega_Z = \mathcal{O}_Z(-F)\) is trivial when restricted to \(D_i\), we can compute that
  \[
    \omega_{D_i} \cong k_{i}^*(\omega_Z \otimes \mathcal{O}_Z(D_i)) \cong k_{i}^*\mathcal{O}_Z(D_i) \cong N_{D_i}.
  \]
  Here, as in the proof of Proposition~\ref{proposition:unoblog}, we have used \(F\) to denote a smooth fibre of \(\mathsf{f}\). Under the conditions in the proposition, and applying the long exact sequence in cohomology coming from~\eqref{equation:smooth}, we see that the claim holds.
\end{proof}

\begin{remark}
  Observe that the conditions of Proposition~\ref{proposition:lastprop} fail in dimension 2, since \(H^1(D_i,\omega_{D_i}) \cong \mathbb{C}\) for any smooth projective curve. This condition can also fail in higher dimension. For instance, if \(D\) is a type II degeneration of K3 surfaces, then \(D\) is a union of surfaces \(D_0, \ldots, D_n\) so that \(D_0\) and \(D_n\) are rational, but \(D_i\), \(i \neq 0,n\), are ruled over an elliptic curve. Note that \(H^0(D_i,\omega_{D_i}) \cong \mathbb{C}\) for \(i \neq 0,1\), so the conditions above fail if \(i \neq 1\). On the other hand, if all irreducible components of \(D\) are rational and \(\dim(Z) \geq 3\), then the conditions of Proposition~\ref{proposition:lastprop} are satisfied. The case that we are most interested in this paper is when \(F\) is a K3 surface, and \(D\) is a type III degenerate K3 surface (see, e.g.,~\cite{kulikov1977degenerations}), then this condition is satisfied.
\end{remark}

Combining Proposition~\ref{proposition:lastprop} and Proposition~\ref{proposition:unoblog}, we obtain the following result.
\begin{theorem}[cf.~{\cite[Corollary 2.18]{katzarkov2017bogomolov}}]\label{theorem:main1}
  Suppose \((Z,D,\mathsf{f})\) is a tame compactified Landau--Ginzburg model of dimension \(d\) satisfying the following conditions:
  \begin{enumerate}[\quad (a)]
  \item \(D \sim -K_Z\);
  \item \(H^1(Z,\mathbb{Q})\cong 0\);
  \item \(H^0(D_i,\omega_{D_i}), H^1(D_i,\omega_{D_i}) \cong 0\) for all irreducible components \(D_i\) of \(D\).
  \end{enumerate}
  Then the tangent map \(T_s \Def(Z,D,\mathsf{f}) \rightarrow T_s \Def(Z)\) is surjective, and its kernel is spanned by deformations of the potential function of \((Z,D,\mathsf{f})\). Consequently, by Lemma~\ref{lemma:surjmap}, deformations of \(Z\) are unobstructed.
\end{theorem}

\begin{remark}
  As a consequence of the unobstructedness of deformations of the tuples \((Z,D,\mathsf{f}), (Z,D)\), and \(Z\), we obtain smooth moduli stacks \(\mathcal{M}_{(Z,D,\mathsf{f})}, \mathcal{M}_{(Z,D)}\), and \(\mathcal{M}_Z\) along with dominant forgetful morphisms
  \[
    \mathcal{M}_{(Z,D,\mathsf{f})} \rightarrow \mathcal{M}_{(Z,D)} \rightarrow \mathcal{M}_Z.
  \]
  The morphism \(\mathcal{M}_{(Z,D,\mathsf{f})}\rightarrow \mathcal{M}_{(Z,D)}\) has relative dimension \(2\). In fact, this maps is a \(\mathbb{C}^* \times \mathbb{C}\) fibre bundle with fibres given by deformations of the potential function \(\mathsf{f}\). The morphism \(\mathcal{M}_{(Z,D)}\rightarrow\mathcal{M}_Z\) is finite.
\end{remark}

\subsection{Deformations of the pair \texorpdfstring{\((Z,F)\)}{(Z,F)} and \texorpdfstring{\(F\)}{F}}\label{subsection:iacono-beauville}

For a tame compactified Landau--Ginzburg model with general fibre \(F\), let \(j_F \colon F \hookrightarrow Z\) indicate the natural injection map. Since \(F\) is a smooth Calabi--Yau variety of dimension \(d-1\), there is a canonical perfect pairing for any \(a,b\),
\begin{equation}\label{equation:pairing}
  (-,-): H^{a}(F,\Omega_F^b) \otimes_\mathbb{C} H^{d-1-a}(F,\Omega_F^{d-1-b}) \rightarrow H^{d-1}(F,\Omega_F^{d-1}) \cong \mathbb{C}
\end{equation}
coming from Serre duality. We define:
\[
  H^1_p(F,\Omega^{d-2}_F) = \im\left(j_{F}^* \colon H^{d-2}(Z,\Omega^{1}_Z)\rightarrow H^{d-2}(F,\Omega^{1}_F)\right)^\perp.
\]
Here the orthogonal complement is taken with respect to the pairing in~\eqref{equation:pairing} with \(a=1,b=d-2\).

\begin{proposition}\label{proposition:main}
  Let \(Z\) be a smooth projective variety, and let \(\mathsf{f} \colon Z\rightarrow \mathbb{P}^1\) be a morphism whose fibres are Calabi--Yau. Put \(D = \mathsf{f}^{-1}(\infty)\) and assume \((Z,D,\mathsf{f})\) satisfies conditions (a), (b) and (c) of Theorem~\ref{theorem:main1}. Let \(F\) be a smooth fibre of \(\mathsf{f}\). Deformations of \((Z,F)\) are unobstructed. The image of \(\kappa \colon T_s \Def(Z,F)\rightarrow T_s \Def(F) \cong H^1(F,\Omega^{d-2}_F)\) is \(H^1_p(F,\Omega_F^{d-2})\), and the kernel of \(\kappa\) has dimension \(h^1(Z,\Omega_Z^{d-1})\).
\end{proposition}

\begin{proof}
  Let \((X,E)\) be any smooth pair, where \(E \in |-mK_X|\) for a positive integer \(m\). According to Iacono~\cite[Corollary 4.8]{iacono2015deformations}, deformations of \((X,E)\) are unobstructed. This proves the first claim.  The second claim is straightforward and follows the same argument as is used by Beauville in~\cite[\S 3]{beauvillefano}. The induced morphism is provided by the standard short exact sequence of sheaves
  \[
    0 \rightarrow T_Z\otimes \mathcal{O}_Z(-F) \rightarrow T_Z(-\log F) \rightarrow i_{F*}T_F \rightarrow 0.
  \]
  Since \(F\) is anticanonical, there is a contraction isomorphism
  \[
    T_Z \otimes \mathcal{O}_Z(-F)\cong T_Z\otimes \omega_Z \rightarrow \Omega_Z^{d-1}.
  \]
  Therefore, by our topological assumptions on \(F\), we have an exact sequence
  \[
    0 \rightarrow H^1(Z,\Omega_Z^{d-1}) \rightarrow H^1(Z,T_Z(-\log F )) \rightarrow H^1(F,T_F) \xrightarrow{\gamma} H^2(Z,\Omega_Z^{d-1})\rightarrow  \cdots.
  \]
  Following the arguments in~\cite[Section 3]{beauvillefano} exactly, one identifies the kernel of \(\gamma\) with the orthogonal complement of the image of the induced map \(H^{d-2}(Z,\Omega_Z^{1})\rightarrow H^{d-2}(F,\Omega^{1}_F)\) under the natural perfect pairing \(H^1(F,T_F)\otimes_\mathbb{C} H^{d-2}(F,\Omega_F^1)\rightarrow H^{d-2}(F,\mathcal{O}_F)\). There is a commutative diagram of sheaves
  \[
    \begin{tikzcd}
      T_F \otimes \Omega^{1}_F \ar[r] \ar[d,"i_\omega \otimes \mathrm{id}"] & \mathcal{O}_F \ar[d,"i_\omega"]  \\
      \Omega^{d-2}_F \otimes \Omega_F^1 \ar[r,"(-)\wedge(-)"] & \omega_F
    \end{tikzcd}
  \]
  where \(i_\omega\) indicates contraction with a holomorphic \((d-1)\)-form on \(F\). Applying this identification to the induced maps in cohomology completes the proof.
\end{proof}

\begin{remark}
  Since deformations of the pair \((Z,F)\) are unobstructed, we obtain a smooth moduli stack \(\mathcal{M}_{(Z,F)}\). If we assume that \(H^1(F,\mathbb{Q}) = 0\) as well, then the moduli stack of Calabi--Yau manifolds deformation equivalent to \(\mathcal{M}_F\) is also smooth. The computation given above tells us that the forgetful map
  \[
    \mathcal{M}_{(Z,F)}\rightarrow \mathcal{M}_F
  \]
  has fibre of dimension \(h^{1}(Z,\Omega^{d-1}_Z)\), and image is of dimension \(h^1_p(F,\Omega_F^{d-2})\).
\end{remark}

\section{Weakly nondegenerate Laurent polynomials and their deformations}\label{section:wreg}

It has been proven in~\cite{przyjalkowski2017compactification}, following~\cite{coates2016quantum}, that all Fano threefolds with very ample anticanonical bundle have mirrors which are constructed from a certain class of Laurent polynomials, called Minkowski polynomials. In this section, we use this construction to produce what we believe are well-behaved versal families of Landau--Ginzburg mirrors of Fano threefolds with very ample anticanonical bundle.

\begin{remark}[Conventions regarding toric varieties]
  We use the following conventions. The toric variety attached to a reflexive polytope \(P\) is the toric variety whose underlying fan is the spanning fan of the polytope \(P\). We denote this variety \(\mathrm{TV}(P)\). If \(X\) is a toric variety coming from a fan \(\Sigma\) in a lattice \(M\) of dimension \(d\), then each cone in \(\Sigma\) of dimension \(k\) corresponds to a torus of dimension \(d-k\) in \(X\). If \(c\) is a cone in \(\Sigma\), we denote the corresponding torus \(\mathbb{T}_c\). We use the phrase \emph{toric boundary} of a toric variety \(X\) to denote the union of all torus invariant divisors of \(X\) or, in the notation introduced above, \(X - \mathbb{T}_{\bm 0}\) where \({\bm 0}\) denotes the cone in \(\Sigma\) consisting of just the origin in \(M\). For general background on toric varieties, standard references include Fulton's classic text~\cite{fulton1993introduction} and the more recent book by Cox--Little--Schenck~\cite{cox2011toric}.
\end{remark}

\subsection{\texorpdfstring{\(\mathsf{M}\)}{M}-polynomials}\label{subsection:mink}

For a polytope \(\delta\) in a lattice \(M\), a \emph{lattice Minkowski decomposition} \(\mathsf{M}(\delta)\) of \(\delta\) is an unordered list of pairs \(\{(\delta_1,n_1),\ldots, (\delta_k,n_k)\}\) of where \(\delta_i\) is an integral polytope and \(n_i\) is a positive integer so that
\[
  \delta = \underbrace{\delta_1 + \cdots + \delta_1}_{n_1} + \cdots + \underbrace{\delta_k + \cdots + \delta_k}_{n_k}
\]
Furthermore, for each \(\delta_i\), let \(\delta_{i,\mathbb{Z}} = \delta_i \cap M\). We also require that
\[
  \delta \cap M =: \delta_\mathbb{Z} = \underbrace{\delta_{1,\mathbb{Z}} + \cdots + \delta_{1,\mathbb{Z}}}_{n_1} + \cdots + \underbrace{\delta_{k,\mathbb{Z}} + \cdots + \delta_{k,\mathbb{Z}}}_{n_k}
\]
Let \(\mathsf{M}_1(\delta) = \{(\sigma_1,m_1),\dots, (\sigma_\ell, m_\ell)\}\) and \(\mathsf{M}_2(\delta) = \{(\delta_1,n_1),\dots, (\delta_k,n_k)\}\) be lattice Minkowski decompositions of \(\delta\). We say that \(\mathsf{M}_1(\delta)\) refines \(\mathsf{M}_2(\delta)\) if there is a partition \(I_1\cup \dots \cup I_k=\{1,\dots, \ell\}\) so that \(n_j \mid m_i\) for all \(m_i \in I_j\) and if for all \(j\),
\[
\sum_{i \in I_j} \left(\dfrac{m_i}{n_j}\right)\sigma_i = \delta_j.
\]
For instance, any lattice Minkowski decomposition of \(\delta\) refines the trivial decomposition \(\{(\delta,1)\}\) and the lattice Minkowski decomposition \(\{(\delta,k)\}\) refines \(\{(\delta,1),\dots, (\delta,1)\}\). For a polytope \(\delta\), Minkowski decompositions of \(\delta\) are partially ordered by refinement. If \(\mathsf{M}_1(\delta)\) is a refinement of \(\mathsf{M}_2(\delta)\) we say that \(\mathsf{M}_2(\delta) \preceq \mathsf{M}_1(\delta)\). Observe that if \(\delta'\) is a face of \(\delta\) then any lattice Minkowski decomposition of \(\delta\) induces a lattice Minkowski decomposition of \(\delta'\) which we denote \(\mathsf{M}(\delta)|_{\delta'}\).

\begin{definition}
  Given a reflexive polytope \(P\), a \emph{Minkowski datum} for \(P\) is the data of a lattice Minkowski decomposition \(\mathsf{M}(\delta)\) of each face \(\delta\) of \(P\) so that for any faces \(\delta' \subseteq \delta\) the induced Minkowski decomposition of \(\delta'\) satisfies
  \[
    \mathsf{M}(\delta)|_{\delta'} \preceq \mathsf{M}(\delta').
  \]
\end{definition}

Pairs \((P,\mathsf{M})\) are also partially ordered by refinement, that is, if \((P,\mathsf{M})\) and \((P,\mathsf{M}')\) are Minkowski decompositions of \(P\), we say that \((P,\mathsf{M}) \preceq (P,\mathsf{M}')\) if \(\mathsf{M}(\delta) \preceq \mathsf{M}'(\delta)\) for all \(\delta\). We let \(\mathsf{M}_\mathrm{triv}\) denote the trivial Minkowski decomposition.

If \(\mathsf{p}= \sum_{\rho \in P_\mathbb{Z}}c_\rho x^\rho\) is a Laurent polynomial with Newton polytope \(P\) then for each face \(\delta\) of \(P\) there is an associated face polynomial
\[
  \mathsf{p}_\delta = \sum_{\rho \in \delta \cap P_\mathbb{Z}} c_\rho x^\rho.
\]

\begin{definition}\label{definition:Mpol}
  Let \(\mathsf{M}\) be a Minkowski datum for a reflexive polytope \(P\). A Laurent polynomial \(\mathsf{p}\) with Newton polytope \(P\) is an \emph{\(\mathsf{M}\)-polynomial} if for each face \(\delta\) of \(P\) we have \(\mathsf{M}(\delta) = \{(\delta_1,n_1), \ldots, (\delta_k,n_k)\}\), and there are polynomials \(\mathsf{h}_{\delta_1},\dots, \mathsf{h}_{\delta_k}\) so that the Newton polytope of \(\mathsf{h}_{\delta_i}\) is \(\delta_i\) and
  \[
    \mathsf{p}_\delta = (\mathsf{h}_{\delta_1}^{n_1} \cdots \mathsf{h}_{\delta_k}^{n_k})x^\nu
  \]
  for some monomial \(x^\nu\). The set of all \(\mathsf{M}\)-polynomials is denoted \(\mathcal{L}_{(P,\mathsf{M})}\). If \((P,\mathsf{M}) \preceq (P,\mathsf{M}')\), then we have \(\mathcal{L}_{(P,\mathsf{M}')} \subseteq \mathcal{L}_{(P,\mathsf{M})}\).
\end{definition}

Suppose we are given an \(\mathsf{M}\)-polynomial \(\mathsf{p}\) with  reflexive Newton polytope \(P\). Let \(P^*\) denote the polar dual polytope and let \({X}_{P^*}\) denote a crepant resolution of the toric variety \(\mathrm{TV}(P^*)\), if such a resolution exists. Such a resolution always exists if \(\dim P \leq 3\). In this case, the vanishing locus of \(\mathsf{p}\) in \(X_{P^*}\) provides a possibly singular hypersurface, which we can denote \(F_\mathsf{p}\), determined by the vanishing of a section \(\sigma_\mathsf{p}\) of the anticanonical bundle of \(X_{P^*}\). Each cone \(c\) in the fan \(\Sigma\) underlying \(X_{P^*}\) corresponds to a toric stratum in \(X_{P^*}\), which we denote \(\mathbb{T}_c\). Let \(\rho_1, \ldots, \rho_k\) be ray generators of \(c\). The intersection \(\mathbb{T}_c \cap F_\mathsf{p}\) is the vanishing locus of the Laurent polynomial \(\mathsf{p}_{c^*}\), where
\[
  c^* = \{ m \in P \mid \langle m, \rho_i \rangle = -1 \, \text{ for all } i = 1, \ldots, k\}.
\]

\begin{definition}\label{definition:wnd}
  We say that an \(\mathsf{M}\)-polynomial \(\mathsf{p}\) is \emph{weakly non-degenerate} if for each cone \(c\), the intersection \(F_\mathsf{p}\cap \mathbb{T}_c\) is a divisor whose irreducible components are smooth, not necessarily reduced, and so that the intersection of any collection of irreducible components of \(F_\mathsf{p} \cap \mathbb{T}_c\) is smooth.
\end{definition}

\begin{remark}
  The notion of weak non-degeneracy here is a weakening of the classical notion of non-degeneracy for a Laurent polynomial~\cite{danilov1987newton}. By the discussion preceding Definition~\ref{definition:wnd}, weak nondegeneracy is equivalent to requiring that \(V(\mathsf{p}_\delta)\) have smooth irreducible components whose intersections are all smooth.
\end{remark}

Suppose we are given a weakly non-degenerate \(\mathsf{M}\)-polynomial \(\mathsf{p}\) supported on a reflexive Newton polytope \(P\). We obtain a pencil of anticanonical hypersurfaces
\begin{equation}\label{equation:pencil}
\mathscr{P}_\mathsf{p}: s\sigma_\mathsf{p}  + t \prod_{\rho \in P\cap M} x_\rho.
\end{equation}
Let \(D_P\) denote the snc union of all torus invariant divisors in \(X_{P^*}\). The base locus of this pencil is the intersection of \(\sigma_\mathsf{p}\) with \(D_P\). Given a toric stratum \(\mathbb{T}_c\) attached to a cone \(c\) of \(\Sigma\), the vanishing locus of \(\sigma_\mathsf{p}\) in \(\mathbb{T}_c\) can be computed using \(P^*\) as in the discussion preceding Definition~\ref{definition:wnd}.

\begin{proposition}[Przyjalkowski~\cite{przyjalkowski2017compactification}, cf.~Duistermaat--van der Kallen~{\cite[Theorem 4]{duistermaat1998constant}}]\label{proposition:prz}
  Let \(\mathsf{p}\) be a weakly non-degenerate polynomial in \(n\) variables with reflexive Newton polytope \(P\). Assume that the toric Fano variety attached to \(P^*\) admits a crepant resolution \(X_{P^*}\). There exists a tame compactified Landau--Ginzburg model \((Z_\mathsf{p},D_\mathsf{p},\mathsf{f})\), obtained by systematically resolving the base locus of \(\mathscr{P}_\mathsf{p}\), with the following properties.
  \begin{enumerate}
  \item The fibres of \(\mathsf{f}\) are compactifications of the fibres of \(\mathsf{p}\).
  \item The Landau--Ginzburg model \((Z_\mathsf{p},D_\mathsf{p},\mathsf{f})\) satisfies conditions (a), (b), and (c) of Theorem~\ref{theorem:main1}.
  \end{enumerate}
\end{proposition}

\begin{proof}
  This is proved for 3-dimensional Minkowski polynomials in~\cite{przyjalkowski2017compactification}. The general case follows the same arguments, and the conclusion is ensured by the strong assumptions that we have made about the structure of the base locus of the pencil in~\eqref{equation:pencil}. The proof proceeds along the following lines.

  Suppose \(B \subseteq D_i\) is a smooth, irreducible component of the base locus of \(\mathscr{P}_\mathsf{p}\). Blowing up along \(B\) reduces the multiplicity of \(\mathscr{P}_\mathsf{p}\) along \(B\) by \(1\). The pencil of hypersurfaces in~\eqref{equation:pencil} remains an anticanonical pencil after blow up along a component of the base locus, however, if \(B\) intersects another irreducible component \(D_j\) of \(D_P\), then another smooth irreducible component might be introduced into the base locus, a \(\mathbb{P}^1\)-bundle over \(B \cap D_j\), which is smooth by assumption. The multiplicity of this new component is less than \(\mathrm{mult}_\mathsf{p}(B)\).

  Now that this is taken care of, the algorithm is to blow up, one-by-one, the irreducible components of the base locus of \(\mathscr{P}_\mathsf{p}\) in order of multiplicity. Each round of blow-ups introduces finitely many new components but reduces maximal multiplicity by at least \(1\). Therefore, this procedure eventually resolves the base locus of \(\mathscr{P}_\mathsf{p}\).

  Finally, we check the topological conditions (a), (b), and (c) in Theorem~\ref{theorem:main1} hold. Since each blow up is at a smooth centre contained in the base-locus of a pencil of anticanonical hypersurfaces, the proper transform of \(\partial X_{P^*}\) remains anticanonical after each blow-up, verifying (a). The variety \(X_{P^*}\) satisfies \(H^1(X_{P^*},\mathbb{Q}) = 0\) and smooth blow up does not affect \(H^1\). This verifies (b). The irreducible components of \(\partial X_{P^*}\) are rational. Since we do not blow up any stratum in \(\partial X_{P^*}\), all irreducible components of \(D_\mathsf{p}\) are birational to an irreducible component of \(\partial X_{P^*}\). Thus they too are rational and satisfy (c).
\end{proof}

\begin{remark}
  The assumption that the toric crepant resolution \(X_{P^*}\) exists ensures that the smooth model of \((Z_\mathsf{p},D_\mathsf{p},\mathsf{f})\) exists. On the other hand, one may always find a partial crepant orbifold resolution of \(\mathrm{TV}(P^*)\). Carrying out the same procedure with this partial crepant orbifold resolution produces an orbifold Landau--Ginzburg model. We speculate that the results of~\cite{katzarkov2017bogomolov} can be extended to the orbifold context along with all of the results in this  paper.
\end{remark}

\begin{corollary}\label{corollary:ade}
  The vanishing locus of a weakly non-degenerate polynomial supported on a Newton polytope is a Calabi--Yau variety which admits a crepant resolution of singularities.
\end{corollary}

\begin{proof}
  Resolving the base locus in the pencil determined by \(\mathsf{p}\), we obtain a tame compactified Landau--Ginzburg model \((Z_\mathsf{p},D_\mathsf{p},\mathsf{f})\) satisfying the Calabi--Yau condition by Proposition~\ref{proposition:prz}. The fibres of \(\mathsf{f}\) are smooth Calabi--Yau varieties birational to the fibres of \(\mathsf{p}\).
\end{proof}

\begin{notation}
  Given the data of a reflexive polytope \(P\) of dimension \(3\) along with lattice Minkowski decomposition of all faces of \(P\), as in Definition~\ref{definition:Mpol}, let \(\mathcal{L}_{(P,\mathsf{M})}\) be the collection of all \(\mathsf{M}\)-polynomials supported on \(P\). Let \(\mathcal{L}^\text{wn-d}_{(P,\mathsf{M})}\) denote the subset of all weakly non-degenerate \(\mathsf{M}\)-polynomials.
\end{notation}

\begin{remark}
  We observe that moduli spaces of weakly nondegenerate Laurent polynomials form substrata of the discriminant locus of moduli space of Laurent polynomials with Newton polytope \(P\). This space has been well studied and admits a compactification, called the \emph{secondary stack} of \(P\). This space has been studied in relation to the moduli space of Landau--Ginzburg models by Diemer--Katzarkov--Kerr~\cite{diemer2016symplectomorphism,diemer2013compactifications}. It would be very interesting to understand the stratification of \(\mathcal{L}^{\text{wn-d}}_P\) obtained by varying Minkowski data.
\end{remark}

The construction in Proposition~\ref{proposition:prz} is consistent in families. Therefore, given the universal family \({\bm p} \colon (\mathbb{C}^*)^d \times \mathcal{L}_{(P,{\mathsf M})}^\text{wn-d}\rightarrow \mathbb{C}\), one obtains a family of tame compactified Landau--Ginzburg models satisfying the Calabi--Yau condition:
\begin{equation}\label{equation:verfam}
\varpi \colon (\mathscr{Z}_{\bm{p}},\mathscr{D}_{\bm{p}},\bm{f})\rightarrow \mathcal{L}^\text{wn-d}_{(P,\mathsf{M})}.
\end{equation}
One might hope that this family is versal, however this expectation fails immediately in dimension \(2\). The question of versality of the family in~\eqref{equation:verfam} is subtle, and it involves both the combinatorics and algebraic geometry of \(\mathsf{M}\).

\begin{example}\label{example:2.28}
  Let us compute Minkowski decomposed polynomials associated to the polytope with vertices labeled as polytope 68 in the list of Kreuzer--Skarke (see~\cite{kreuzer2004PALP}):
  \[
    \left(
      \begin{array}{rrrrrrr}
        1 & 0 &1 &0 &-1& 0& 1 \\
        0 & 1 & 1 & -1&  0& 0& 1 \\
        0 & 0 & 2 & -1& -1 &1 & 1
      \end{array}
    \right).
  \]
  The corresponding Laurent polynomials are generically of the form
  \[
    \mathsf{p} = a_0x + a_1 y+ a_2xyz^2 + \dfrac{a_3}{yz} + \dfrac{a_4}{xz} + a_5z + a_6xyz + a_7.
  \]
  We let \(a_i \in \mathbb{C}^*\). Observe that we could allow \(a_7 = 0\) without changing the Newton polytope of \(\mathsf{p}\) however we avoid this for simplicity. All faces of the polynomial \(P\) are simplices except for three, which are unit squares. The corresponding face polynomials are
  \[
    \mathsf{p}_{\delta_1} = a_1y + a_2xyz^2 + \dfrac{a_4}{xz} + a_5z,\quad \mathsf{p}_{\delta_2} = a_0x + a_2xyz^2 +\dfrac{a_3}{yz} + a_5z, \quad \mathsf{p}_{\delta_3} = a_0x + a_1y + \dfrac{a_3}{yz} + a_5z.
  \]
  To a square, there is only one admissible Minkowski sum decomposition, into a pair of transverse line segments of length \(1\). Minkowski decomposition of the first face occurs precisely when
  \[
    a_5a_1 - a_2a_4 = 0,\qquad a_0a_5 - a_2a_3=0,\qquad a_1a_5 - a_0a_3 = 0.
  \]
  We observe that this is a 5-dimensional torus in \((\mathbb{C}^*)^8\). There is a torus action by scaling each coordinate on the space of all such Laurent polynomials which reduces the space of all Minkowski polynomials of this form to a 1-dimensional torus. We observe that this is consistent with the fact that the corresponding family of Fano varieties, \textnumero 2.28, has Picard rank \(2\). After parametrization, this family of Laurent polynomials can be written as
  \[
    \mathsf{p}_{a,b} = \dfrac{(xyz + a) (bxyz^2 + xyz + x + y)}{xyz} - 1.
  \]
  In Subsection~\ref{subsection:mm228} we will show that this family of Laurent polynomials is versal.
\end{example}

\begin{example}
  Let us take the family of Laurent polynomials
  \begin{equation}\label{equation:dp6}
    \mathsf{p} =  \dfrac{b_0}{z} + b_1z  + \dfrac{ \sum_{0 \leq i+j\leq 3} a_{i,j} x^i y^j}{xy}
  \end{equation}
  Each of the six facets of this Laurent polynomial is an \(A_3\) triangle (Definition~\ref{definition:Mink} below), so the unique Minkowski polynomial attached to this data is given by
  \[
    \mathsf{p} = \dfrac{1}{z} + z + \dfrac{(x+y + 1)^3}{xy}.
  \]
  The corresponding family of Landau--Ginzburg models is {not} versal.
\end{example}

\subsection{Relations between \texorpdfstring{\(\mathsf{M}\)}{M}-polynomials and other types of polynomials}\label{subsection:minkowski}

The definition of an \(\mathsf{M}\)-polynomial is inspired by the definition of Minkowski polynomials in~\cite{coates2014mirror}. The goal of this subsection is to explain the relationship between \(\mathsf{M}\)-polynomials and Minkowski polynomials, and to indicate a possible relationship between \(\mathsf{M}\)-polynomials and a more general class of polynomials, called \emph{maximally mutable Laurent polynomials}. This discussion of this section will not be used in the remainder of the paper.

We recall the definition of a Minkowski polynomial.

\begin{definition}\label{definition:Mink}
  Suppose \(\mathsf{p}\) is a Laurent polynomial  with reflexive Newton polytope of dimension \(3\). We say that \(\mathsf{p}\) is a \emph{Minkowski polynomial} if it is an \(\mathsf{M}\)-polynomial so that or each face \(\delta\) we have \(\mathsf{M}(\delta) = \sum_{i=1}^n n_i \delta_i\), where \(\delta_i\) is either affinely equivalent to an \(A_n\)-triangle with vertices \((0,0),(1,0)\) and \((0,n)\) or an interval of length \(1\). We require that if \(\delta_i\) is an \(A_n\) triangle, then, after possible change of variables,
  \[
    \mathsf{p}_{\delta_i} = (1+y)^n + x.
  \]
  If \(\delta_i\) is an interval of length 1, we require that, after a possible change of variables, \(\mathsf{p}_{\delta_i} = (1 + x)\).
\end{definition}

\noindent In other words, Minkowski polynomials are \(\mathsf{M}\)-polynomials whose Minkowski summands are tightly restricted, and so that the coefficients of \(\mathsf{p}\) are specialized. Case-by-case analysis shows that all Minkowski polynomials are weakly non-degenerate~\cite{przyjalkowski2017compactification}. All Fano threefolds with very ample anticanonical bundle have a mirror which comes from a Minkowski polynomial, however, there are Minkowski polynomials which are not mirror to Fano manifolds, even in dimension \(3\). To deal with this, and to model Landau--Ginzburg mirrors of higher dimensional Fano varieties, the authors of~\cite{coates2021maximally} introduce the notion of a \emph{maximally mutable Laurent polynomial}. This notation is inspired by the notion of mutation in cluster theory, going back at least to work of Fomin--Zelevinsky~\cite{fomin2002cluster}. The connection to Landau--Ginzburg models was first made in work of Galkin--Usnich~\cite{galkin2010mutations}.

\begin{definition}[see~\cite{akhtar2012minkowski}]\label{definition:mut}
  Let \(\mathsf{p}\) be a Laurent polynomial whose Newton polytope \(P\) is contained in a lattice \(M\). Choose a primitive element \({\mathbf{m}} \in M\) and let \(x_{\mathbf{m}}\) be the corresponding function on \(\mathbb{C}^{*d}\). Choosing a complementary set of coordinate functions, \(x_1,\dots, x_{d-1}\) we may write
  \[
    \mathsf{p} = \sum_{i\in \mathbb{Z}} \mathsf{q}_i x_{\mathbf{m}}^i,\qquad \mathsf{q}_i = 0 \text{ for all but finitely many } i
  \]
  where \(\mathsf{q}_i \in \mathbb{C}[x_1^{\pm 1},\dots, x_{d-1}^{\pm 1}]\). For any \(\mathsf{g} \in \mathbb{C}[x_1^{\pm 1},\dotsm x_{d-1}^{\pm 1}]\) we define the \emph{mutation} of \(\mathsf{p}\) along \(\mathsf{g}\) in direction \({\mathbf{m}}\) to be the polynomial
  \[
    \mu_{{\mathbf{m}},\mathsf{g}}(\mathsf{p}) = \sum_{i \in \mathbb{Z}} \mathsf{q}_i\mathsf{g}^i x_{\mathbf{m}}^i
  \]
  \emph{if} \(\mu_{{\mathbf{m}},\mathsf{g}}(\mathsf{p}) \in \mathbb{C}[M]\). Given data \(({\mathbf{m}},\mathsf{g})\), we say that \(\mathsf{p}\) is \emph{mutable} along \(({\mathbf{m}},\mathsf{g})\) if \(\mu_{{\mathbf{m}},\mathsf{g}}(\mathsf{p})\) exists, or, equivalently, if \(\mathsf{g}^{-i}\) divides \(\mathsf{q}_i\) for all \(i < 0\).
\end{definition}

\noindent  For \({\mathbf{m}}\) as in Definition~\ref{definition:mut}, put
\[
  \delta_{\mathbf{m}} = \left\{ x \in P \,\, \left| \,\, {\mathbf{m}}(x) = \min_{y \in P}\{{\mathbf{m}}(y)\}\right.\right\}.
\]
Assuming that \(0\) is in the interior of \(P\), it is a direct observation that if \(\mathsf{p}\) is mutable along \(({\mathbf{m}},\mathsf{g})\) then \(\mathsf{p}_{\delta_{\mathbf{m}}} = \mathsf{g} h_{({\mathbf{m}},\mathsf{g})}\) for some polynomial \(h_{({\mathbf{m}},\mathsf{g})}\). If \(P\) is reflexive, and \(\delta_{\mathbf{m}}\) is a facet of \(P\), then \(\min_{y \in P}\{{\mathbf{m}}(y)\}=-1\) so being mutable along \(({\mathbf{m}},\mathsf{g})\) is equivalent to \(\Newt(\mathsf{g})\) being a Minkowski summand of \(\delta_{\mathbf{m}}\), however, if \(\delta_{\mathbf{m}}\) is a face of \(P\) and \(\min_{y \in P}\{{\mathbf{m}}(y)\}\neq (-1)\), then mutability is stronger than the existence of a Minkowski decomposition. For a Laurent polynomial \(\mathsf{p}\) we let
\[
  S_\mathsf{p} = \{ ({\mathbf{m}},\mathsf{g}) \mid \mathsf{p} \text{ is mutable along } ({\mathbf{m}},\mathsf{g})\}.
\]
We say that \(\mathsf{p}\) is \emph{rigid maximally mutable} if the only normalized polynomial with Newton polytope \(P\) and mutation set \(S_\mathsf{p}\) is \(\mathsf{p}\) itself\footnote{This is \emph{not} the definition of a rigid, maximally mutable polynomial given in~\cite{coates2021maximally}, however, this definition is more compact and it is conjectured in in {op. cit.} that the definition presented here is equivalent to the proper definition of a rigid maximally mutable Laurent polynomial.}. A Laurent polynomial is normalized if all of the coefficients associated to vertices of \(\Newt(\mathsf{p})\) are equal to \(1\). It would be interesting to know whether all maximally mutable Laurent polynomials are weakly regular.

\subsection{Deformations of LG models and deformations of weakly non-degenerate Laurent polynomials}\label{subsection:LG-deformations}

Suppose \(\mathsf{p}\) is a weakly non-degenerate Laurent polynomial with Newton polytope \(P\). The goal of this section is to show that there is Minkowski data \((P,\mathsf{M}_{\gen})\) and an induced map
\[
  \mathcal{L}^{\text{wn-d}}_{(P,\mathsf{M})}\rightarrow \mathcal{M}_{(Z_\mathsf{p},D_\mathsf{p})}
\]
which is dominant.

Before addressing this question, we discuss a more general question of how blow-ups of pairs behave under deformation. The following result essentially goes back to Horikawa~\cite[Theorem 9.1]{horikawa1976deformations}.

\begin{proposition}\label{proposition:defblow}
  Let \((X,D)\) be a pair consisting of a smooth variety \(X\) and a snc divisor \(D\). Let \(C\) be a smooth codimension 1 subvariety of an irreducible component \(D_1\) of \(D\), and let \(b : \widetilde{X}\rightarrow X\) be the blow up along \(C\) and let \(\widetilde{D}\) be the proper transform of \(D\) under \(b\). Suppose \((\widetilde{\mathscr{X}},\widetilde{\mathscr{D}})\) is a deformation of \((\widetilde{X},\widetilde{D})\). Then there is a deformation \((\mathscr{X},\mathscr{D})\) of \((X,D)\) and a contraction morphism
  \[
    \pi \colon (\widetilde{\mathscr{X}},\widetilde{\mathscr{D}}) \rightarrow (\mathscr{X},\mathscr{D})
  \]
  with exceptional divisor \(\mathscr{E}\) so the image of \(\mathscr{E}\) is a deformation of \(C\) in \(D_1\).
\end{proposition}

\begin{proof}
 According to Horikawa~\cite[Theorem 9.1]{horikawa1976deformations}, any deformation \(\widetilde{\mathscr{X}}\rightarrow B\) of \(\widetilde{X} = \Bl_CX\) admits a contraction map \(b \colon \widetilde{\mathscr{X}} \rightarrow \mathscr{X}\), where the centre of \(b\) is a deformation of \(C\). Therefore, any deformation of the pair \((\widetilde{X},\widetilde{D})\), where \(\widetilde{D}\) represents the proper transform of \(D\), admits a contraction map. The only thing that we need to show is that in deforming \((\widetilde{X},\widetilde{D})\), the centre of \(b\) remains in \(D_1\). This is true if and only if the exceptional divisor \(E\) of the blow up intersects \(D_1\). So it is sufficient to see that any deformation of the pair \((\widetilde{X},\widetilde{D})\) comes from a deformation of the pair \((\widetilde{X},\widetilde{D}+E)\). This reduces to the following calculation in cohomology.

  We have a short exact sequence of sheaves, (e.g,~\cite[Proposition 1]{kawamata1977deformations})
  \begin{equation}\label{equation:seq}
    0 \rightarrow T_{\Bl_CX}(-\log   b^{-1} D) \rightarrow T_{\Bl_CX}(-\log \widetilde{D} ) \rightarrow N_E  \rightarrow 0.
  \end{equation}
  Since \(E\) is the exceptional divisor of a blow up, we have that \(N_E \cong \mathcal{O}_E(-1)\). By, e.g.,~\cite[III Exercise 8.4]{hartshorne2013algebraic}, we have that \(R^ib_*\mathcal{O}_E(-1) \cong 0\) for all \(i\). Therefore, by the Leray spectral sequence, we have \(H^j(E,N_E) \cong 0\) for all \(j\). Consequently, the map in~\eqref{equation:seq} gives isomorphisms
  \[
    H^j(\Bl_CX,T_{\Bl_{C}X} (-\log  b^{-1}D )) \xrightarrow{\sim} H^j(\Bl_CX,T_{\Bl_CX}(-\log  \widetilde{D}))
  \]
  for all \(j\). Thus both the tangent space and obstruction spaces of the deformation space of \((\widetilde{X},\widetilde{D})\) and \((\widetilde{X},b^{-1}D)\) are identified. In other words, any deformation of the pair \((\widetilde{X},\widetilde{D})\) comes from a deformation of the pair \((\widetilde{X},\widetilde{D} + E)\). In particular, in any small deformation of \((\widetilde{X},\widetilde{D})\) the exceptional divisor \(E\) intersects \(D_1\) in a deformation of \(C\), and the intersections of \(C\) with \(D_i\) deform smoothly as well.
\end{proof}

Let \(\mathsf{p}\) be a weakly non-degenerate Laurent polynomial, which we identify with a rational function on \(X_{P^*}\). A deformation of \(\mathsf{p}\) over connected base \(S\) with a marked point \(s_0\) is a Laurent polynomial with coefficients in \(\Gamma(S,\mathcal{O}_S)\) (or, equivalently, a rational map \({\bm p} \colon X_{P^*}\times S\dashrightarrow \mathbb{C}\)) so that \({\bm p}|_{X_{P^*} \times s_0} = \mathsf{p}\) and so that each \(s\in S\), \(\mathsf{p}|_{X_{P^*}\times s}\) is a weakly non-degenerate Laurent polynomial. Any small deformation of a weakly non-degenerate \(\mathsf{M}\)-polynomial gives rise to a deformation of Landau--Ginzburg models by applying Proposition~\ref{proposition:prz} over \(\mathcal{O}_S\).

\begin{theorem}\label{theorem:defmink}
  Let \(\mathsf{p}\) be a weakly non-degenerate \(\mathsf{M}\)-polynomial with Newton polytope \(P\). Suppose a projective crepant resolution \(X_{P^*}\) exists. Let \((Z_\mathsf{p},D_\mathsf{p})\) be the corresponding Landau--Ginzburg model. There is Minkowski data \((P,\mathsf{M}_{\gen})\) so that \((P,\mathsf{M}_{\gen})\preceq (P,\mathsf{M})\) along with an open subset \(U\) of \(\mathcal{L}_{(P,\mathsf{M}_{\gen})}^\text{wn-d}\) so that the induced map
  \[
    \mathcal{L}^{\text{wn-d}}_{(P,\mathsf{M}_{\gen})}\supseteq U \rightarrow \mathcal{M}_{(Z_\mathsf{p},D_\mathsf{p})}.
  \]
  is dominant.
\end{theorem}

\begin{proof}
  By Proposition~\ref{proposition:defblow}, a small deformation \((\mathscr{Z}_\mathsf{p},\mathscr{D}_\mathsf{p})\)  of the pair \((Z_\mathsf{p},D_\mathsf{p})\) over a base \(B\) admits a contraction map \(\pi \colon (\mathscr{Z}_\mathsf{p},\mathscr{D}_\mathsf{p}) \rightarrow (\mathscr{X}_P,\mathscr{D}_P)\), where \((\mathscr{X}_P,\mathscr{D}_P)\) is a deformation of the toric variety \(X_P\) and its boundary divisor \(D_P\). Anticanonical toric pairs are rigid, since \(H^1(X_P,T_X(-\log D_P )) \cong H^1(X_P,\Omega^{n-1}_{X_P}(\log D_P)) \cong 0\). It follows that \(\mathscr{X}_P  \cong X_P \times B\). From Proposition~\ref{proposition:main}, we know that for any small deformation of \((Z_\mathsf{p},D_\mathsf{p})\), there is a small deformation of the data \((Z_\mathsf{p},D_\mathsf{p},\mathsf{f})\), unique up to scaling and translating \(\mathsf{f}\). Let \((\mathscr{Z}_\mathsf{p},\mathscr{D}_\mathsf{p},{\bm f})\) denote such a deformation over a base \(B\), and let \(\pi\) denote the contraction to \((\mathscr{X}_P,\mathscr{D}_P) = (X_P, D_P)\times B\). By construction, \({\bm f}\) induces a rational function \({\bm p}\) on \((X_P,D_P)\times B\) with polar locus \(D_P \times B\), or in other words, a family of Laurent polynomials.

  The last thing that we must check is that the locus of indeterminacy of \({\bm p}\) is a deformation of the locus of indeterminacy of \(\mathsf{p}\). However, this is ensured by Proposition~\ref{proposition:defblow}, since the locus of indeterminacy is precisely the union of the centres of \(\pi\), and multiplicities are equal to the number of exceptional divisors contracting to each component in the centre.
\end{proof}

\begin{remark}\label{remark:spec}
  We call \((P,\mathsf{M}_{\gen})\) the \emph{general} Minkowski data of \((P,\mathsf{M})\). An expectation of the Fanosearch programme is that the Minkowski data of a Laurent polynomial should describe data involved in \(\mathbb{Q}\)-Gorenstein deformation from \(\mathrm{TV}(P)\) to a terminal Gorenstein Fano variety. The total space of this degeneration is expected to be projective, and polarized by a power of the relative anticanonical sheaf. This forces the central fibre to be Fano and toric. However, by changing polarization, one can sometimes partially desingularize the central fibre, at the cost of changing polarization. If \((Z_\mathsf{p},D_\mathsf{p},\mathsf{f})\) is mirror to a Fano variety \(X\), then the parameters on \(\mathcal{M}_{(Z_\mathsf{p},D_\mathsf{p})}\) should correspond to complexified K\"ahler parameters on \(X\). General Minkowski data should describe how much the central fibre of a \(\mathbb{Q}\)-Gorenstein degeneration can be resolved by changing polarization. For instance all Gorenstein toric degenerations of del Pezzo surfaces may be viewed as smooth deformations. Correspondingly, for any reflexive \(P\), and Minkowski data \(\mathsf{M}\), one can see in Appendix~\ref{appendix:parametrized} that the general Minkowski data of \(\mathsf{M}\) is trivial.
\end{remark}

\section{The case of Fano threefolds}\label{section:k3}

\subsection{Moduli spaces of lattice-polarized K3 surfaces and Dolgachev--Nikulin duality}\label{subsection:DN-duality}

Let us recall some basic notions regarding lattice polarizations of K3 surfaces.

We refer to the lattice \(L_{\K3} = H^{\oplus 3} \oplus \mathbb{E}_8(-1)^{\oplus 2}\) as the \emph{K3 lattice}.

\begin{proposition}[see~{\cite[Proposition~1.3.5]{huybrechts2016K3}}]
  For any complex K3 surface \(S\) we have \(H^2(S, \mathbb{Z}) \cong L_{\K3}\).
\end{proposition}

\begin{definition}
  Let \(L\) be an even lattice of signature \((1,r)\). We say that a K3 surface \(S\) is \emph{\(L\)-polarized} if there is a primitive embedding \(\iota \colon L \hookrightarrow \Pic(S)\) whose image contains an ample class. A family \(\varpi: \mathcal{S}\rightarrow B\) of K3 surfaces is \emph{\(L\)-polarized} if there is a trivial sub-local system \(\mathbb{L} \subseteq R^2\varpi_*\underline{\mathbb{Z}}_{\mathcal{S}}\) which induces an \(L\)-polarization on each fibre of \(\varpi\).
\end{definition}

There is a well-defined coarse moduli space of \(L\)-polarized K3 surfaces (see, e.g.,~\cite[Theorem 3.1]{dolgachev1996mirror} or~\cite{alexeev/k3} for details). We start with the period domain
\[
  \mathcal{P}_{L} = \{ z \in L^\perp \otimes_\mathbb{Z}\mathbb{C} \mid (z,z) = 0,(z,\overline{z}) > 0\}.
\]
Define the group
\[
  O^+(L) = \{ \gamma \in O(L_{\K3}) \mid \gamma|_L = \mathrm{id}_L\}.
\]
Let \(\Delta(L^\perp)\) denote the set of all \(\delta \in L^\perp\) so that \((\delta,\delta) = (-2)\), and define \(H_\delta = \{ z \in \mathcal{P}_{L} \mid (z,\delta) =0 \}\) for any \(\delta \in \Delta(L^\perp)\). Then the coarse moduli space of \(L\)-polarized K3 surfaces is the arithmetic quotient
\[
  \mathcal{M}_L = O^+(L^\perp)\backslash \left(\mathcal{P}_{L} - \bigcup_{\delta \in \Delta(L^\perp)}H_\delta\right).
\]
To any family of \(L\)-polarized K3 surfaces, \(\varpi \colon \mathcal{S}\rightarrow B\), there is a period morphism \(\Pi \colon B \rightarrow \mathcal{M}_{L}\). We say that \(\mathcal{S}\) is a \emph{complete} family of \(L\)-polarized K3 surfaces if \(\Pi\) is dominant.

\begin{theorem}[see~{\cite{dolgachev1996mirror}}]
  Let \(L\) be a lattice such that there exists a unique (up to isometry) primitive embedding into the K3 lattice \(L_{\K3}\). Then \(\mathcal{M}_L\) is a coarse moduli space of \(L\)-polarized K3 surfaces.
\end{theorem}

\begin{remark}\label{remark:embedding}
  There exist sufficient conditions on a lattice \(L\) to have a primitive embedding into the K3 lattice (see~\cite[Theorem~1.4.6]{dolgachev1982lattices}), and to ensure that this embedding is unique (see~\cite[Theorem~1.4.8]{dolgachev1982lattices}). More generally, the coarse moduli space of \(L\)-polarized K3 surfaces has a finite number of irreducible components corresponding to equivalence classes of embeddings of \(L\) into \(L_{\K3}\).
\end{remark}

There is a natural projective compactification of \(\mathcal{M}_L\), called the \emph{Baily--Borel compactification}. We will use the notation \(\overline{\mathcal{M}}_{L}\) to denote the Baily--Borel compactification.  It is obtained by adding points called type III boundary components, and curves called type II boundary components to \(\mathcal{M}_L\). The  type III (resp., type II) boundary components are in set-theoretic bijection between \(O^+(L^\perp)\) equivalence classes of rank 1 (resp., rank 2) totally isotropic sublattices of \(L^\perp\). See~\cite{scattone1987compactification,dolgachev1996mirror} for details.

\begin{definition}
  Let \(L \subset L_{\K3}\) be a primitive sublattice. Assume that the orthogonal complement admits the decomposition \(L^{\perp} = H \oplus L^{\vee}\) for some lattice \(L^{\vee}\). We refer to \(L^{\vee}\) as the \emph{Dolgachev--Nikulin dual} to \(L\).
\end{definition}

\begin{remark}\label{remark:orthogonal-complement}
  Let \(L\) and \(M\) be lattices primitively embedded into the K3 lattice. One can explicitly check whether \(L\) and \(M\) are actually Dolgachev--Nikulin dual to each other using~\cite[Lemma~1.4.5]{dolgachev1982lattices}.
\end{remark}

Dolgachev--Nikulin duality is a mirror symmetry correspondence between families of lattice-polarized K3 surfaces: for a primitively embedded lattice \(L \subset L_{\K3}\) one may construct complete families of \(L\) and \(L^\vee\)-polarized K3 surfaces which are related by various forms of mirror symmetry.

\begin{remark}
  Observe that the existence of \(L^\vee\) corresponds to a specific choice of embedding \(H\) into \(L^\perp\) which, in turn, corresponds to a particular choice of type III boundary component. Dolgachev--Nikulin duality relates \(L\)-polarized K3 surfaces whose complex structure is near this type III boundary point to the symplectic structure of \(L^\vee\)-polarized K3 surfaces.
\end{remark}

\subsection{The canonical lattice of Landau--Ginzburg threefold}

Suppose \((Z,D,\mathsf{f})\) is a tame compactified Landau--Ginzburg model which satisfies the Calabi--Yau condition. By adjunction, smooth fibres of \(\mathsf{f}\) are Calabi--Yau, and, therefore, if \(\dim Z = 3\) they are either K3 surfaces or abelian surfaces. Furthermore, there is a natural polarizing lattice on the fibres of \(\mathsf{f}\). Suppose \(F\) is a smooth fibre of \(\mathsf{f}\), then the restriction morphism \(j_F: H^2(Z,\mathbb{Z})\rightarrow H^2(F,\mathbb{Z})\) has primitive image which is contained in \(\Pic(F)\). Furthermore, the image of \(j_F\) is monodromy invariant and thus generates a trivial local subsystem of \(R^2\mathsf{f}^\circ_*\underline{\mathbb{Z}}_{U}\), where \(U\) is union of all smooth fibres of \(\mathsf{f}\) and \(\mathsf{f}^\circ = \mathsf{f}|_U\). Let \(L_Z\) denote the image of \(j_F\). Let \(V\) denote the regular values of \(\mathsf{f}\) in \(\mathbb{P}^1\).

\begin{lemma}
  The family \(\mathsf{f}^\circ: U\rightarrow V\) is an \(L_Z\)-polarized family of K3 or abelian surfaces.
\end{lemma}

\begin{proof}
  If \(\alpha\) is ample on \(Z\), then its restriction is in \(L_Z\) and is ample on \(F\). Now we prove that \(L_Z\) is a primitive sublattice. Any element \(\beta\) of \(\Pic(F) \cap L_Z\otimes_\mathbb{Z}\mathbb{Q}\) in \(\Pic(F)\otimes_\mathbb{Z}\mathbb{Q}\) is also monodromy-invariant. By the global invariant cycles theorem it follows that \(\beta\) comes from a cohomology class in \(H^2(Z,\mathbb{Z})\). Thus \(L_Z\) is primitive.
\end{proof}

In Proposition~\ref{proposition:prz} we constructed Landau--Ginzburg models from sufficiently regular Laurent polynomials in three variables supported on reflexive polytopes. In the process of this construction, we chose to blow up the toric threefold \(X_{P^*}\) in a sequence of codimension 2 subvarieties. The order in which we perform these blow ups determines the resulting variety \((Z_\mathsf{p},D_\mathsf{p})\), up to birational transformation. The following result implies in particular that the lattice \(L_{Z_\mathsf{p}}\) does not depend on this choice.

\begin{proposition}
  Suppose two K3 fibred varieties \(f_1,f_2:X_1,X_2\rightarrow C\) and that there is a birational map \(g \colon X_1\rightarrow X_2\) making the diagram
  \[
    \begin{tikzcd}
      X_1 \ar[rd,"f_1"] \ar[dashed, rr,"g"] & & X_2 \ar[ld,"f_2",swap] \\
      & C &
    \end{tikzcd}
  \]
  commute. Then \(L_{X_1} = L_{X_2}\).
\end{proposition}

\begin{proof}
  By definition, this map induces a birational map between each individual fibre. Since K3 surfaces are minimal, this means that each smooth fibre of \(f_1\) is isomorphic to the corresponding smooth fibre of \(f_2\). Therefore, there is a nonempty open subset \(U\) of \(C\) over which \(g\) restricts to an isomorphism \(g \colon f_1^{-1}(U) \rightarrow f_2^{-1}(U)\). By the global invariant cycles theorem (see, e.g.,~\cite[\S 4.3.3]{voisin2003hodge}) we may identify the image of \(H^2(X_i,\mathbb{Z})\rightarrow H^2(F,\mathbb{Z})\) with the image of \(H^2(f^{-1}_i(U),\mathbb{Z}) \rightarrow H^2(F,\mathbb{Z})\). Therefore, the result holds.
\end{proof}

An important operation in the study of Minkowski polynomials is called \emph{mutation}, which is a type of birational map making the following diagram commute
\[
  \begin{tikzcd}
    (\mathbb{C}^*)^3 \ar[rd,"\mathsf{p}_1"] \ar[dashed, rr,"h"] & & (\mathbb{C}^*)^3 \ar[ld,"\mathsf{p}_2",swap] \\
    & \mathbb{C} &
  \end{tikzcd}
\]
Since both \(((\mathbb{C}^*)^3,\mathsf{p}_1)\) and \(((\mathbb{C}^*)^3,\mathsf{p}_2)\) admit relative compactifications, we obtain the following result.

\begin{corollary}
  If two \(\mathsf{M}\)-polynomials \(\mathsf{p}_1,\mathsf{p}_2\) are related by mutation, then \(L_{Z_{\mathsf{p}_1}} = L_{Z_{\mathsf{p}_2}}\).
\end{corollary}

\begin{remark}
  Throughout this paper we will freely change birational models for K3 surfaces in order to compute Picard lattices. The fact that birational maps between smooth K3 surfaces are isomorphisms justifies this tactic.
\end{remark}

\subsection{Characterizing the internal period maps of Landau--Ginzburg threefolds}

In this section we restrict to the case where \(\dim(Z) = 3\) and, furthermore, that \(H^{i,0}(Z) = 0\) for all \(i \neq 0\).  In addition to conditions (a), (b), and (c) of Theorem~\ref{theorem:main1}, assume that \((Z,D,\mathsf{f})\) satisfies:
\begin{itemize}[\quad (d)]
\item Smooth fibres of \(\mathsf{f}\) are K3 surfaces, and the fibre over \(\infty\) is a type III degenerate K3 surface.
\end{itemize}
Precisely, this means that the dual intersection complex of the fibre over \(\infty\) is a normal crossings union of rational surfaces whose dual intersection complex is a triangulation of the 2-sphere (see~\cite{kulikov1977degenerations} for more details). For instance, if \(\mathsf{p}\) is a sufficiently regular Laurent polynomial in three variables, \((Z_\mathsf{p},D_\mathsf{p},\mathsf{f})\) satisfies this criterion.

Each 3-dimensional Landau--Ginzburg model \((Z,D,\mathsf{f})\) admits an internal algebraic period map to the moduli space of \(L_Z\)-polarized K3 surfaces corresponding to the family \(\mathsf{f} \colon Z \rightarrow \mathbb{P}^1\). We denote it by \(\Pi_Z \colon \mathbb{P}^1 \rightarrow \overline{\mathcal{M}}_{L_Z}\). Such period maps send type III degenerate K3 surfaces to type III boundary points of \(\overline{\mathcal{M}}_{L_Z}\).

\begin{proposition}\label{proposition:surj}
  Let \((Z,D,\mathsf{f})\) be a tame compactified Landau--Ginzburg model of dimension 3 satisfying conditions (a), (b), and (c) in Theorem~\ref{theorem:main1}, and condition (d) above.
Then the moduli space \(\overline{\mathcal{M}}_{L_Z}\) is uniruled by a pencil of curves passing through a type III boundary point.
\end{proposition}

\begin{proof}
  There is a finite forgetful map \({\bm \varphi} \colon \mathcal{M}_{L_Z}\rightarrow \mathcal{M}_{\K3}\), where \(\mathcal{M}_{\K3}\) is the 20-dimensional coarse moduli space of (complex) K3 surfaces. There is a factorization
  \[
    \begin{tikzcd}
\mathcal{M}_{(Z,F)} \ar[rd] \ar[d] & \\
\mathcal{M}_{L_Z} \ar[r] &  \mathcal{M}_F = \mathcal{M}_{\K3}.
    \end{tikzcd}
  \]
  The image of \({\bm \varphi}\) in \(\mathcal{M}_{\K3}\) has dimension
  \[
    h^1_p(F,\Omega^1_F) = 20 - \dim \left(\im\left(H^1(Z,\Omega_Z^1)\rightarrow H^1(F,\Omega_F^1)\right)\right) = \dim \mathcal{M}_{L_Z}.
  \]
  Therefore, the map \(\mathcal{M}(Z,F)\rightarrow \mathcal{M}_{L_Z}\) is dominant. Since the map \(\mathcal{M}(Z,D,\mathsf{f})\rightarrow \mathcal{M}(Z)\) is dominant, any small \(L_Z\)-polarized deformation of \(F\) is a fibre of a small deformation of \((Z',D',\mathsf{f}')\).

  Choosing a versal deformation \(\varpi :(\mathscr{Z},\mathscr{D},{\bm f})\rightarrow B\) we obtain a family of morphisms \(\varpi \times {\bm f} \colon \mathscr{Z}\rightarrow B\times \mathbb{P}^1\) whose general fibre is a smooth K3 surface. Thus there is a holomorphic map \(\Pi \colon B \times \mathbb{P}^1 \rightarrow \overline{\mathcal{M}}_{L_Z}\). Since each fibre of \(\varpi:\mathscr{D}\rightarrow B\) is a deformation of \(B\), each \(\varpi^{-1}(b)\) is a type III degenerate K3 surface. Consequently, \(\Pi(B \times \infty)\) maps to a type III boundary point \(p_\infty\) of \(\overline{\mathcal{M}}_{L_Z}\), and \(\varpi(b\times \mathbb{P}^1)\) cut out a family of rational curves passing through \(p_\infty\) as \(b\) varies which locally cover \(\mathcal{M}_{L_Z}\). Since \(\overline{\mathcal{M}}_{L_Z}\) is projective this family of curves can be extended globally.
\end{proof}

\begin{remark}
  The existence of a type III boundary point in \(\mathcal{M}_{L_Z}\) implies that the lattice \(L_Z^\perp\) admits a totally isotropic sublattice of rank \(1\). In fact, Hodge-theoretic mirror symmetry suggests more, namely, that there is an embedding of the lattice \(H\) into \(L_Z^\perp\).
\end{remark}

\begin{remark}
    As the name suggests, Kulikov~\cite{kulikov1977degenerations} classified other types of degenerations of K3 surfaces, of type I and II respectively. Type I semistable degenerations of K3 surfaces are simply smooth K3 surfaces, and semistable type II degenerations correspond to chains of surfaces \(S_0 \cup \cdots \cup S_n\) so that \(S_0, S_n\) are rational, and \(S_i, i \neq 0,n\) are ruled surfaces over an elliptic curve. Under the period mapping, type II degenerations of K3 surfaces map to type II boundary points. The conclusions of Proposition~\ref{proposition:surj} continue to hold if \(\mathsf{f}^{-1}(\infty)\) is a type II degenerate K3 surface with \(n=1\), however the uniruling is by curves passing through a type II boundary component, not a type III boundary point. If \(\mathsf{f}^{-1}(\infty)\) is type I (i.e. smooth) or type II with \(n > 1\) then condition (c) of Theorem~\ref{theorem:main1} fails.
\end{remark}

This places strong restrictions on the lattice polarizations that can appear on anticanonical hypersurfaces of Fano threefolds. In the case where \(\rk(L_Z) = 19\), however, this is simply the observation that \(\overline{\mathcal{M}}_{L_Z} \cong \mathbb{P}^1\). We note that in the case where \(\rk(L_Z) = 19\), the condition that \(\overline{\mathcal{M}}_{L_Z}  \cong \mathbb{P}^1\) is not enough to characterize the lattices \(L_Z\), even under the condition that \(H\) is a sublattice of \(L_Z^\perp\). See~\cite[Theorem 2.12]{doran2020calabi}.

\begin{definition}
  Suppose \(\mathbb{V}\) is a \(\mathbb{Q}\)-local system over a nonempty Zariski-open subset of \(\mathbb{P}^1\). Let \(i \colon U\rightarrow \mathbb{P}^1\) denote the canonical embedding. We say that \(\mathbb{V}\) is \emph{extremal} if \(H^1(\mathbb{P}^1,i_*\mathbb{V}) = 0\). Given a map \(\phi \colon \mathbb{P}^1 \rightarrow \overline{\mathcal{M}}_{L_Z}\), there are finitely many variations of Hodge structure over \(\phi^{-1}(\mathcal{M}_{L_Z})\) for which the period map is \(\phi\). We say that the map \(\phi\) is \emph{extremal} if one of these variations of Hodge structure has extremal underlying local system.
\end{definition}

It is shown in~\cite{przyjalkowski2017compactification} (see also~\cite{cheltsov2018katzarkov}) that for all Fano threefolds, there is a Landau--Ginzburg mirror \((Z,D,\mathsf{f})\) so that \(H^3(Z,\mathbb{Q})\cong 0\). The same is true for all deformations of \(Z\). By the Leray spectral sequence, one can argue that this implies that if \(V\) is the set of regular values of \(\mathsf{f}\), and \(\mathsf{f}^\circ\) denotes the restriction of \(\mathsf{f}\) to \(\mathsf{f}^{-1}(V)\), then local system \(R^2\mathsf{f}_*^\circ\underline{\mathbb{Q}}_{\mathsf{f}^{-1}(V)}\) is extremal (see, e.g.,~\cite[Proof of Lemma 3.2]{doran2016calabi}). Thus the period maps of any deformation of \((Z,D,w)\) are also extremal. This property is predicted by mirror symmetry (see~\cite{harder2016geometry,shamoto2018hodge}). Consequently, we have the following result when combined with Theorem~\ref{theorem:DN-duality}.

\begin{corollary}\label{corollary:ssur}
  Suppose \(X\) is a Fano threefold with very ample anticanonical bundle. Then the moduli space of \(\Pic(X)^\vee\)-polarized K3 surfaces admits a type III boundary point \(p_\infty\) and a ruling by extremal curves passing through \(p_\infty\).
\end{corollary}

\begin{remark}\label{remark:smifano}
  Corollary~\ref{corollary:ssur} is predicted by mirror symmetry and Proposition~\ref{proposition:surj}. Suppose \(X\) is Fano and it has a tame compactified mirror Landau--Ginzburg model \((Z,D,\mathsf{f})\)\footnote{It is not true that all Fano varieties have tame compactified mirrors. For more details, see~\cite{przyjalkowski2022singular}, but essentially, the existence of tame compactified mirrors seems related to having \emph{very} ample anticanonical bundle.}. Let \(Y = Z \setminus \mathsf{f}^{-1}(\infty)\). One expects an identification between \(H^3(Y,F)\) and \(\oplus_{i=0}^3H^{i,i}(X)\), and cup product with \(c_1(T_X)\) on \(\oplus_{i=0}^3 H^{i,i}(X)\) should correspond to the action of \(\log(T_\infty)\), where \(T_\infty\) is the monodromy operator
  \[
    T_\infty \colon H^3(Y,F) \rightarrow H^3(Y,F)
  \]
  obtained by letting \(F\) vary in a small loop around infinity (see, e.g.,~\cite[\S 3.1]{katzarkov2017bogomolov} for a more detailed explanation of these explanations). Shamoto~\cite{shamoto2018hodge} equips \(H^3(Y,F)\) with a mixed Hodge structure whose weight filtration is induced by \(\log(T_\infty)\) and whose Hodge filtration has graded dimensions equal to those of the canonical Hodge filtration on \(H^3(Y,F)\).

  Let \(H^3_{\mathrm{sh}}(Y,F)\) denote this mixed Hodge structure. The Hodge filtration on \(H^3(Y,F)\) should agree, under mirror symmetry, with the degree filtration on \(\bigoplus_{i=0}^3 H^{i,i}(X)\). Consequently, if \(c_1(T_X)\) is ample, then one expects the weight and Hodge filtrations of \(H_{\mathrm{sh}}^3(Y,F)\) to be opposed, and hence \(H^3_\mathrm{sh}(Y,F)\) will be mixed Tate. One can deduce from this (\cite{harder2021hodge}) that both \(H^3(Z,\mathbb{Q}) = 0\) and that \(H^2_\infty(F,\mathbb{Q})\) (the limit mixed Hodge structure at \(\infty\) of \(F\)) is mixed Tate, thus \(D\) is a type III degenerate K3 surface. Combining this with Proposition~\ref{proposition:surj}, we predict that Corollary~\ref{corollary:ssur} should hold.

  We remark that this prediction requires two things: (a) the existence of a tame compactified mirror \((Z,D,\mathsf{f})\) and (b) that \(\cup c_1(T_X)\) behaves like the cup product with an ample divisor. The existence of a tame compactified mirror \((Z,D,\mathsf{f})\) seems to be a feature of many weak-Fano varieties (e.g., weak Fano toric threefolds or, more generally, weak Fano toric complete intersections) and the second identifies \(X\) as a \emph{semi-}Fano threefold in the terminology of Corti et al.~\cite{corti2013asymptotically, corti2015g_2}, since this condition implies that the anticanonical map is semi-small. The classification of semi-Fano threefolds is related to the classification of \(G_2\)-manifolds (op. cit.). If we take mirror symmetry for granted, the classification of semi-Fano threefolds should therefore be related to the classification of lattices for which \(\mathcal{M}_L\) has the properties in Corollary~\ref{corollary:ssur}.
\end{remark}

\subsection{A Noether--Lefschetz type result}

If \(S\) is a surface in a \(\mathbb{P}^3\) which is very general and of degree at least \(4\), then it is a classical theorem, first proved by Lefschetz~\cite{lefschetz1924analysis} by topological means, later proved by Griffiths--Harris~\cite{griffiths1985noether} algebraically, that the Picard rank of \(S\) is \(1\). More general results of this nature have been proved for hypersurfaces in toric varieties, for instance if \(X_\Sigma\) is a simplicial toric variety, and \(S\) is an ample hypersurface in \(X_\Sigma\) of large enough degree, then Bruzzo--Grassi~\cite{bruzzo2012picard} show that the same result holds. In this section, we formulate and prove a more general version of this for nef anticanonical hypersurfaces in smooth toric threefolds, and we extend this to describing the image of the period map.

\begin{proposition}\label{proposition:nl}
  Let \(\mathsf{p}\) be a weakly non-degenerate Laurent polynomial supported on a reflexive polytope \(P\) of dimension 3 with Minkowski data \(\mathsf{M}\). Suppose that \((P,\mathsf{M}) = (P,\mathsf{M}_{\gen})\). Let \(\mathsf{p}'\in \mathcal{L}^{\text{wn-d}}_{(P,\mathsf{M})}\) be a very general Laurent polynomial, and let \(F\) denote a minimal resolution of any compactification of \(\mathsf{p}'\). Then \(\Pic(F)\) is spanned by divisors of the following type.
  \begin{itemize}
  \item Proper transforms of irreducible components of \(D_P\).
  \item Exceptional divisors of the resolution map \(b :F\rightarrow \overline{F}\).
  \end{itemize}
\end{proposition}

More precisely; the phrase \emph{very general} indicates that \(\mathsf{p}'\) corresponds to a point in the complement of a countable collection of codimension 1 subvarieties of \(\mathcal{L}^{\text{wn-d}}_{(P,\mathsf{M})}\).

\begin{proof}
  For notational simplicity we let \(\mathsf{p} = \mathsf{p}'\). We may construct a smooth resolution of the vanishing locus of \(\mathsf{p}\) by constructing the Landau--Ginzburg model \((Z_{\mathsf{p}},D_{\mathsf{p}}, \mathsf{f}_\mathsf{p})\), following the recipe in Proposition~\ref{proposition:prz}. Then \(S = \mathsf{f}_{\mathsf{p}}^{-1}(0)\), and
  \[
    \mathcal{L}^\text{wn-d}_{(P,\mathsf{M})}\rightarrow  \mathcal{M}_{L_{Z_\mathsf{p}}}
  \]
  is, locally, a composition of dominant maps by Theorem~\ref{theorem:defmink} and Proposition~\ref{proposition:surj}. The period map is algebraic, so this map is dominant. There is a countable, dense, collection of divisors in
  \(\mathcal{M}_{L_{Z_\mathsf{p}}}\), called Noether--Lefschetz divisors, whose complement parametrizes K3 surface whose Picard lattice is \(L_{Z_\mathsf{p}}\), the image of the natural map \(\Pic(Z_\mathsf{p})\rightarrow \Pic(F)\). Given a Minkowski decomposition \((P,{\mathsf{M}})\), the blow up process in Proposition~\ref{proposition:prz} is an iterated blow up along smooth centres. Therefore, \(H^2(Z_\mathsf{p},\mathbb{Z})\) is spanned by the proper transforms of toric boundary divisors of \(X_{P^*}\) and exceptional divisors.

  We need to compute the restriction of \(\Pic(Z_{\mathsf{p}})\) to a fibre \(F\) of \(\mathsf{f}\). Let \(b : Z_\mathsf{p} \rightarrow X_{P^*}\) denote the blow up map, and we let \(b|_F : F\rightarrow \overline{F}\) be the induced resolution map. If \(D\) is the proper transform of an irreducible component of \(D_P\), then its intersection with \(F\) is empty, and thus \([D]|_F = 0\). If \(E\) is an exceptional divisor of the blow up map \(b: Z_{\mathsf{p}}\rightarrow X_{P^*}\), its image is either a point in \(V(\mathsf{p})\), in which case \(E \cap F\) is contained in the exceptional locus of the map \(b : F\rightarrow \overline{F}\), or its image is a curve in \(D_P \cap \overline{F}\). Since we have assumed that \(\mathsf{p}\) is very general, we know that all algebraic cycles in \(\Pic(F)\) are in the image of the restriction map \(i^* \colon \Pic(Z_\mathsf{p})\rightarrow \Pic(F)\), so are, in particular, all exceptional divisors of \(b \colon F\rightarrow \overline{F}\) and all proper transforms of irreducible components of \(D_P\cap \overline{F}\). This finishes the proof.
\end{proof}

\begin{remark}
  Given an anticanonical hypersurface \(S\) in a smooth toric weak Fano threefold \(X_{P^*}\), the pullback map \(i^*:\Pic(X_{P^*})\rightarrow \Pic(S)\) is neither surjective nor injective in general. Depending on the geometry of \(P^*\) and its dual, there are irreducible torus invariant divisors of \(X_{P^*}\) which do not intersect \(S\), hence are in the kernel of \(i^*\), and similarly, there can be torus invariant divisors \(D\) of \(X_{P^*}\) so that \(D\cap S\) is a union of irreducible curves for general \(S\). For a general choice of \(S\), the irreducible components of \(S \cap D_P\) (recall that \(D_P\) denotes the toric boundary of \(X_P\) as in Section~\ref{section:wreg}) are a simple normal crossings union of curves. This normal crossings union of curves only depends on \(P\) and \(P^*\). We reinterpret Proposition~\ref{proposition:nl} in this case as the statement that if \(\mathsf{p}\) is a general Laurent polynomial whose Newton polytope is a 3-dimensional reflexive polytope, and \(S\) is the vanishing locus of \(\mathsf{p}\) in \(X_{P^*}\), then \(\Pic(S)\) is spanned by the irreducible components of \(S \cap D_P\). In~\cite{whitcher2015reflexive}, this lattice was called \(\Pic_\mathrm{cor}(S)\) and the image of \(\Pic(X_{P^*})\rightarrow \Pic(S)\) is called \(\Pic_\mathrm{tor}(S)\). These lattices were computed by Rohsiepe in~\cite{rohsiepe2004lattice}\footnote{Unfortunately, the URL listing the lattices computed in~\cite{rohsiepe2004lattice} is no longer valid.}.
\end{remark}

\section{Examples}\label{section:ex}

In this section we give several explicit examples of the phenomena described in the previous sections. The first example that we look at is the mirror to Family \textnumero 2.1. This example does not have very ample anticanonical bundle, therefore we do not automatically obtain a complete algebraic family of Landau--Ginzburg mirrors directly from Theorem~\ref{theorem:defmink} and~\cite{przyjalkowski2017compactification}. We construct a parametrized family of Laurent polynomials in this section, obtained by taking a 1-parameter deformation of the standard Landau--Ginzburg model for Family \textnumero 2.1. Applying the techniques used in~\cite[Subsection~2.1]{cheltsov2018katzarkov} we obtain a 1-parameter family of tame compactified Landau--Ginzburg models. In Subsection~\ref{subsection:MM21} we show that this family of Landau--Ginzburg models is complete.

The second pair of examples that we look at are mirrors to Family \textnumero 2.28 and Family \textnumero 2.33. In these cases, Proposition~\ref{proposition:surj} and Corollary~\ref{corollary:ssur} can be applied directly. The interesting aspect of these examples is that the anticanonical Picard lattices, denoted \(L_{{2.28}}\) and \(L_{{2.33}}\), are isomorphic (see Appendix~\ref{appendix:picard-lattices-fano} below), therefore the fibres of the mirror Landau--Ginzburg models are members of the same moduli space of lattice-polarized K3 surfaces, and by Corollary~\ref{corollary:ssur} we obtain two extremal rulings passing through two, possibly different, type III boundary points. We show that, in this case, the type III boundary points coincide, but the induced rulings are very different.

\subsection{Mori--Mukai 2.1}\label{subsection:MM21}

For the Landau--Ginzburg mirror of Family \textnumero 2.1, there is a natural compactification which does not satisfy Definition~\ref{definition:tcom} (see~\cite{przyjalkowski2022singular}). This Landau--Ginzburg model is constructed, according to~\cite{ilten2013toric}, from the Laurent polynomial
\begin{equation*}
  \mathsf{p} = \dfrac{(x + y + 1)^6(z + 1)}{xy^2z} + z.
\end{equation*}
The polytope dual to the Newton polytope of \(\mathsf{p}\) is not integral, therefore, the construction of \(Z_{\mathsf{p}}\) described in Proposition~\ref{proposition:prz} cannot be carried out verbatim.

It is proved in Appendix~\ref{subsection:02-01} that the fibres of \(\mathsf{p}\) compactify to a family of K3 surfaces whose ambient Picard lattice is
\[
  L_{2.1} = H \oplus \mathbb{E}_8(-1)^2.
\]
There is Minkowski data attached to \(\mathsf{p}\), which we denote \(\mathsf{M}\). Up to isomorphism, any element of \(\mathcal{L}_{(P,\mathsf{M})}\) can be described in the following form:
\begin{equation}\label{equation:2.1}
  \mathsf{p}_c = \dfrac{(x + y + 1)^6(z + c)}{xy^2z} + z.
\end{equation}
One may apply the techniques in~Appendix~\ref{appendix:rank-02} to find a fibrewise compactification \(\mathcal{S}\) of the fibres of \(\mathsf{p}_c\) for which the lattice \(L_\mathcal{S}\) is isomorphic to \(L_{2.1}\). We expect that this family of Landau--Ginzburg models mirror the K\"ahler moduli space of Fano varieties in Family \textnumero 2.1. In this section, we verify that the conclusions of Proposition~\ref{proposition:main} hold for this family of Landau--Ginzburg models.

K3 surfaces with lattice polarization by \(L_{2.1}\) have been studied by Clingher and the first-named author~\cite{clingher2007modular}. In that paper, the lattice \(L_{2.1}\) is simply called \(M\). We we use the notation \(L_{2.1}\) to avoid confusion. K3 surfaces admitting \(L_{2.1}\)-polarization can be written as the vanishing locus of quartic polynomials of the form
\begin{equation}\label{equation:cd}
  y^2zw - 4x^3z + 3\alpha xzw^2 + \beta zw^3 - \dfrac{1}{2}(\delta z^2w^2 + w^4).
\end{equation}
We may view the coordinates \((\alpha, \beta, \delta)\) as weighted projective coordinates, with weights \((2,3,6)\) respectively, on the Baily--Borel compactification of \(\mathcal{M}_{L_{2.1}}\). In this notation, the locus \(\delta =0\) is a type II boundary curve, and the point \(\alpha^3= \beta^2, \delta =0\) is a type III boundary point.

\begin{proposition}\label{proposition:2.1}
  For each polynomial \(\mathsf{p}_c\), the induced period map \(\Pi_c \colon \mathbb{P}^1 \rightarrow \overline{\mathcal{M}}_{L_{2.1}}\) is given by
  \[
    \Pi_c = \left[\lambda^{2/3} : 1728c\mu^2 - \lambda^2 + 864\lambda \mu  : 2^{12}3^6{(c\mu + \lambda)\mu^3c}\right].
  \]
For general values of \(c\), the map \(\Pi_c\) sends \([\lambda:0]\) to the type III boundary point in \(\overline{\mathcal{M}}_{L_{2.1}}\) and sends \([-cs:s]\) to a point in the type II boundary component. The union of the images of all maps \(\Pi_c\) is Zariski-open in \(\mathfrak{\mathcal{M}}_{L_{2.1}}\).
\end{proposition}

\begin{proof}
  Take a hypersurface \(S_{c,\lambda}\) presented as \(\mathsf{p}_c = \lambda\), where \(\mathsf{p}_c\) is as in~\eqref{equation:2.1}. The map \(z \colon S_{c,\lambda}\rightarrow \mathbb{C}^*\) has a general fibre which is a (singular affine) curve of geometric genus \(1\). Let \(u= x+y+q\). The hypersurface \(S_{c,\lambda}\) may be compactified to a family of sextics curves in \(\mathbb{P}^2_{u,y,q}\) over \(\mathbb{P}^1_{z,w}\) of the form
  \[
    X^6(z+ w)q + q^3(u-y-q)z(cz-\lambda w) = 0.
  \]
  In the affine chart \(y=w=1\) we make the change of variables \(q=\widetilde{q}u^2\). The proper transform of \(S_{c,\lambda}\) is then
  \[
    (z+ 1)u^2\widetilde{q} + \widetilde{q}^3(u-1-\widetilde{q}u^2)y^2z(cz-\lambda) = 0.
  \]
  This is now quadratic in \(u^2\). The resulting double cover of \(\mathbb{A}^2_{\widetilde{q},z}\)  has ramification locus
  \[
    \widetilde{q}^6z^2(cz-\lambda)^2 + 4((z+ 1)\widetilde{q} - \widetilde{q}^4z(cz-\lambda))(\widetilde{q}^3z(cz-\lambda)).
  \]
  After normalization, we see that \(S_{c,\lambda}\) is birational to a hypersurface in \(\mathbb{A}^2_{\widetilde{q},z,Y}\) expressed in terms of the equations
  \[
    Y^2 = (\lambda-cz)(-4z^2X^3c + 4zX^3t + z^2X^2c - zX^2t + 4 + 4z)
  \]
  after renaming \(X = \widetilde{q}\). Viewing this as a family of cubics in \((X,Y)\), we may put it into Weierstrass form.
  \[
    Y^2 = X^3 - \dfrac{z^4}{3}X - \dfrac{2 z^5 ((864\lambda c + 864c^2)z^2 + (-\lambda^2 + 864\lambda + 1728c)z + 864)}{27\lambda^2}.
  \]
  These K3 surfaces have two fibres of Kodaira type \(\mathrm{II}^*\) and four singular fibres of type \(\mathrm{I}_1\). On the other hand, quartic K3 surfaces of the form~\eqref{equation:cd} also admit a fibration of the same type:
  \[
    Y^2 = X^3 -192\alpha z^4X + 512z^5 (\delta z^2 -2\beta z + 1).
  \]
  After appropriate coordinate scaling, one identifies parameters and obtains the desired result.

  Finally, to check that the image is open in \(\overline{\mathcal{M}}_{L_{2.1}}\) it is enough to see that the tangent map has full rank at some point. This is straightforward, and we omit the details.
\end{proof}

\subsection{Mori--Mukai 2.33}

The unique Fano threefold making up Family \textnumero 2.33 is a blow up of \(\mathbb{P}^3\) in a line. Therefore, Givental's algorithm tells us that the mirror Laurent polynomial is of the form
\begin{equation}\label{equation:2.33}
  a_0x + a_1y+ a_2z + \dfrac{a_3x}{z} + \dfrac{a_4}{xy} + a_5,\qquad a_i \in \mathbb{C}^*.
\end{equation}
Let \(P\) denote the Newton polytope of the Laurent polynomials in~\eqref{equation:2.33}, and let \(\mathcal{L}_P\) denote the space of all such Laurent polynomials. Combining Proposition~\ref{proposition:nl} and the computations in Appendix~\ref{appendix:rank-02}, we know that the Picard lattice of the vanishing locus of some \(\mathsf{p} \in \mathcal{L}_P\) has the form
\[
  L_{2.33} = \mathbb{E}_8(-1)^{\oplus 2}\oplus \left(\begin{matrix} 0 & 3 \\ 3 & 4 \end{matrix} \right),
\]
and, in fact, we know a priori that the period map \(\mathcal{L}_P\rightarrow \mathcal{M}_{L_{2.33}}\) is dominant by Proposition~\ref{proposition:nl}. Letting \((\mathbb{C}^*)^3\) act on the coordinates of members of \(\mathcal{L}_P\), we see that all elements of \(\mathcal{L}_P\) are equivalent to polynomials of the form
\[
  \mathsf{p}_{a,b} = x + y+ z + \dfrac{x}{z} + \dfrac{a}{xy} + b,
\]
where \(a = a_0^3a_1a_4/(a_2^2 a_3^2)\) and \(b = a_0a_5/(a_2a_3)\).

There is a universal family of \(H \oplus \mathbb{E}_8(-1)\oplus \mathbb{E}_7(-1)\)-polarized K3 surfaces described by Clingher--Doran~\cite{clingher2012lattice} and by Kumar~\cite{kumar2008k3}. We write Kumar's model as
\begin{equation}\label{equation:kumar}
  y^2  = x^3 - t^3 \left(\dfrac{I_4}{12}t + 1 \right)x + t^{5}\left( \dfrac{I_{10}}{4}t^2 + \dfrac{(I_2I_4 - 3I_6)}{108}t + \dfrac{I_2}{24}\right).
\end{equation}
These K3 surfaces are related by the Shioda--Inose construction~\cite{morrison1984k3} to principally polarized abelian surfaces, and the moduli space of \(H \oplus \mathbb{E}_8(-1)\oplus \mathbb{E}_7(-1)\)-polarized K3 surfaces can be identified with the moduli space of principally polarized abelian surfaces. The parameters \(I_2,I_4, I_6, I_{10}\) in~\eqref{equation:kumar} are Igusa's modular invariants for genus 2 curves~\cite{igusa1960arithmetic}, and thus can be written in terms of Siegel modular forms and vice versa. According to Clingher--Doran~\cite{clingher2012lattice}, the Baily--Borel compactification of the moduli space of \(H \oplus \mathbb{E}_8(-1)\oplus \mathbb{E}_7(-1)\)-polarized K3 surfaces is given by the map
\[
  (\alpha,\beta,\gamma,\delta) = \left( \dfrac{1}{36} I_4, \dfrac{-1}{216} I_2I_4 + \dfrac{1}{72} I_6, \dfrac{1}{4} I_{10}, \dfrac{1}{96} I_2I_{10}\right)
\]
and \(\alpha,\beta,\gamma,\delta\) are variables on \(\mathbb{P}(2,3,5,6)\). The family of K3 surfaces in~\eqref{equation:cd} can be obtained by setting \(\gamma = 0\). In terms of these variables, there is a type II boundary component at \(\gamma=\delta=0\), a type III boundary component at \(\gamma=\delta=\beta^2-\alpha^3=0\).

\begin{proposition}\label{proposition:2.33}
  Any K3 surface which admits lattice polarization by
  \[
    L_{2.33} = \mathbb{E}_8(-1)^2 \oplus \left( \begin{array}{rr} 0 & 3  \\ 3 & 4 \end{array} \right)
  \]
  also admits lattice polarization by \(H \oplus \mathbb{E}_8(-1) \oplus \mathbb{E}_7(-1)\). Therefore, there is a period map from \(\mathcal{L}_P\) to the moduli space \(\mathcal{M}_{L_{2.33}}\). This morphism is given by the map
  \[
    \begin{array}{l}
      \alpha = \frac{1}{9} (144 a_{1} a_{2} a_{3} a_{4} + 24 a_{0} a_{1} a_{4} a_{5} + a_{5}^{4}),  \\  \beta = -\frac{1}{27} (216 a_{0}^{2} a_{1}^{2} a_{4}^{2} - 648 a_{1} a_{2} a_{3} a_{4} a_{5}^{2} + 36 a_{0} a_{1} a_{4} a_{5}^{3} + a_{5}^{6}), \\
      \gamma  = 1024  a_{3}  a_{2}  a_{0}^{2}  a_{4}^{3}  a_{1}^{3}, \\  \delta= \frac{1024}{3} a_{3} a_{2}  a_{4}^{3}  a_{1}^{3}  (12 a_{2}^{2} a_{3}^{2} - 12 a_{0} a_{2} a_{3} a_{5} + a_{0}^{2} a_{5}^{2}).
    \end{array}
  \]
  Consequently, the period map \(\Pi_{a,b} \colon \mathbb{P}^1\rightarrow \mathcal{M}_N\) is given by
  \[
    \begin{array}{l}
      \alpha(\mu,\lambda) = \frac{1}{9} (144 a\lambda^4 + 24 ab\lambda^3\mu + b^{4}\mu^4),  \\  \beta(\mu,\lambda) = -\frac{1}{27}  (216 a^2\lambda^6 - 648 ab^2\lambda^4\mu^2 + 36 ab^3\lambda^3\mu^3 + b^6\mu^6), \\
      \gamma(\mu,\lambda)  = 1024  a^3\lambda^{10}, \\  \delta(\mu,\lambda) = \frac{1024}{3} a^3\lambda^{10}  (12 \lambda^2 - 12 b\lambda\mu + b^{2}\mu^2).
    \end{array}
  \]
  For each \(a,b\), the map \(\Pi_{a,b}\) sends \([0:\lambda]\) to the type III boundary point but that no other point maps to closure of the type II boundary component of \(\mathcal{M}_{L_{2.33}}\).
\end{proposition}

\begin{proof}
  The first statement follows from the fact that the lattice \(\mathbb{E}_8(-1)\oplus \mathbb{E}_7(-1)\oplus H\) embeds primitively into the lattice \(L_{2.33}\), see, e.g.,~\cite{kumar2015hilbert}. To express this period map in terms of Clingher--Doran's modular invariants, we use a similar tactic to those described in the proof of Proposition~\ref{proposition:2.1}. We omit details.

  After an appropriate toric change of variables we may express~\eqref{equation:2.33} so that projecting onto one of the variables, \(x,y\), or \(z\), provides an elliptic fibration which also admits an \(\mathrm{II}^*\) and \(\mathrm{III}^*\) type singular fibres. Once we have such an equation, it can be put into Weierstrass form, \(y^2 = x^3 + g_2(a,b,\lambda,\mu)x + g_3(a,b,\lambda,\mu)\). After this is completed, we scale variables and coefficients until our Weierstrass equation looks like~\eqref{equation:kumar}. At this point, we then solve for \(I_2,I_4,I_6\), and \(I_{12}\). Changing variables to the \(\alpha,\beta,\gamma\), and \(\delta\) of Clingher and Doran, we obtain the desired result.
\end{proof}

\subsection{Mori--Mukai 2.28}\label{subsection:mm228}

We observe from Appendix~\ref{appendix:picard-lattices-fano} that the fibres of the Landau--Ginzburg mirrors of Family \textnumero 2.28 admit the same lattice polarization as the fibres of the mirror Landau--Ginzburg models to Family \textnumero 2.33. The standard mirror for Family \textnumero2.28 can be obtained by taking the family of Laurent polynomials in Example~\ref{example:2.28}
\begin{equation}\label{equation:pcd}
  \mathsf{p}_{c,d} = \dfrac{(xyz + 1) (cxyz^2 + dxyz + x + y)}{xyz} - 1,
\end{equation}
and specializing \(c=d=1\). We have the following result.

\begin{proposition}\label{proposition:228}
  The fibre of \(\mathsf{p}_{c,d}\) over \(\lambda\) is birational to the vanishing locus of
  \[
    x + y+ z + \dfrac{x}{z} + \dfrac{c\lambda^2}{xy} + 1 - d\lambda.
  \]
  Therefore, by Proposition~\ref{proposition:2.33}, period maps \(\Pi_{c,d}\) are of the form
  \[
    \begin{array}{l}
      \alpha(\mu,\lambda) = \left(\frac{1}{9}\right)  (d^{4} \lambda^{4} - 4 d^{3} \lambda^{3} \mu - 24 c d \lambda^{3} \mu + 6 d^{2} \lambda^{2} \mu^{2} + 168 c \lambda^{2} \mu^{2} - 4 d \lambda \mu^{3} + \mu^{4}), \\
      \beta(\lambda,\mu) = \left(-\frac{1}{27}\right)  (d^{6} \lambda^{6} - 6 d^{5} \lambda^{5} \mu - 36 c d^{3} \lambda^{5} \mu + 15 d^{4} \lambda^{4} \mu^{2} - 540 c d^{2} \lambda^{4} \mu^{2} - 20 d^{3} \lambda^{3} \mu^{3} + 216 c^{2} \lambda^{4} \mu^{2},
      \\ \hspace{7cm}+ 1188 c d \lambda^{3} \mu^{3} + 15 d^{2} \lambda^{2} \mu^{4} - 612 c \lambda^{2} \mu^{4} - 6 d \lambda \mu^{5} + \mu^{6}), \\
      \gamma(\mu,\lambda)  = \left(1024\right)  c^{3} \mu^{4} \lambda^{6}\\
      \delta(\mu,\lambda)= \left(\frac{1024}{3}\right)  c^{3}  \mu^{4}  \lambda^{6} (d^{2} \lambda^{2} + 10 d \lambda \mu + \mu^{2}).
    \end{array}
  \]
\end{proposition}

\begin{proof}
  The fibre over \([\lambda:\mu]\in \mathbb{P}^1 \) of \(\mathsf{p}_{c,d}\) is written as the vanishing locus of the Laurent polynomial
  \[
    \lambda \dfrac{(xyz + 1) (cxyz^2 + dxyz + x + y)}{xyz} - \mu.
  \]
  After scaling coordinates \((x,y,z)\mapsto (x/\lambda,y/\lambda,\lambda ^2z)\), this becomes
  \[
    \dfrac{(xyz + 1) (c\lambda^2xyz^2 + d\lambda xyz + x + y)}{xyz} - \mu.
  \]
  Using appropriate toric change of variables, we can express fibres of \(\mathsf{p}_{c,d}\) and \(\mathsf{p}_{a,b}\) as Jacobian elliptic surfaces with a \(\mathrm{I}_4^*\) type fibre at 0, and an \(\mathrm{I}_8\) fibre at \(\infty\). Using this representation we express the \(a_0, \ldots, a_5\) coefficients of~\eqref{equation:2.33} in terms of \(c,d,\lambda,\mu\). Applying Proposition~\ref{proposition:2.33}, we obtain the equations in the statement of the current proposition.
\end{proof}

\begin{corollary}
  The family of Laurent polynomials in~\eqref{equation:pcd} is versal. In other words, the family of varieties \(Z_{\mathsf{p}_{c,d}}\) obtained by applying Proposition~\ref{proposition:prz} is a versal deformation of \(Z_{\mathsf{p}_{1,1}}\).
\end{corollary}

\begin{proof}
  If \(Z_{\mathsf{p}_{c,d}}\) were not versal, the induced deformation of pairs \((Z_{\mathsf{p}_{c,d}},F)\), obtained by deforming fibres in \(Z_{\mathsf{p}_{c,d}}\), would not induce a complete family of \(L_{2.28}\)-polarized K3 surfaces. From Proposition~\ref{proposition:228} one checks directly that this is not the case.
\end{proof}

\begin{remark}
  Notice that for each value of \(c,d\), the map \(\Pi_{c,d}\) maps both \([\mu:0]\) to the type III boundary point in \(\overline{\mathcal{M}}_{L_{2.33}}\) and maps \([0:\lambda]\) to a point in the type II boundary component of \(\overline{\mathcal{M}}_{L_{2.33}}\). It is interesting to compare the two maps. Observe that, since the images of \(\Pi_{a,b}\) do not intersect the type II boundary component, the images of the two maps are distinct, but that they both describe families of pointed curves in \(\overline{\mathcal{M}}_{L_{2.33}}\) passing through the same type III boundary point.
\end{remark}

\section{Dolgachev--Nikulin duality for toric Landau--Ginzburg models of Fano threefolds}\label{section:DN-duality}

In this section we will explain the proof of Theorem~\ref{theorem:DN-duality}, as stated in the introduction to this paper. Let us recall that Dolgachev--Nikulin duality is a mirror symmetry correspondence between families of lattice-polarized K3 surfaces (see Subsection~\ref{subsection:DN-duality} for details). After a brief summary of the theory of toric Landau--Ginzburg models, we show how to compute the intersection pairing on
\[
  \im(H^2(Z,\mathbb{Z}) \rightarrow H^2(F,\mathbb{Z})),
\]
where \((Z, \mathsf{f})\) is the standard toric Landau--Ginzburg model for a Fano threefold (see Definition~\ref{definition:LG-standard}), and \(F\) is its general fibre. Then we explain how to compute the intersection pairing on
\[
  \im(H^2(X,\mathbb{Z})\rightarrow H^2(S,\mathbb{Z})),
\]
where \(S\) is a smooth anticanonical divisor in a smooth Fano threefold \(X\). We then explain how one uses these results to verify Dolgachev--Nikulin mirror symmetry for families containing \(F\) and \(S\).

\subsection{Toric Landau--Ginzburg models}\label{subsection:LG-models}

Let \(\varphi[\mathsf{p}]\) be the constant term of a Laurent polynomial \(\mathsf{p}\). Define \emph{the main period} for \(\mathsf{p}\) as the following formal series: \(I_{\mathsf{p}}(t) = \sum \varphi[\mathsf{p}^j] t^j\). The following theorem (see~\cite[Proposition~2.3]{przyjalkowski2013weak} for the proof) justifies this definition.

\begin{theorem}
  Let \(\mathsf{p}\) be a Laurent polynomial in \(n\) variables. Let \(\mathbf {D}\) be a Picard--Fuchs differential operator for a pencil of hypersurfaces in a torus provided by \(\mathsf{p}\). Then we have \(\mathbf{D}[I_{\mathsf{p}}(t)] = 0\).
\end{theorem}

Let us recall the definition of a toric Landau--Ginzburg model (see~\cite[2.1]{przyjalkowski2018review} for the details).

\begin{definition}
  A \emph{toric Landau--Ginzburg model} for a pair of a smooth Fano variety \(X\) of dimension \(n\) and divisor \(D\) on it is a Laurent polynomial \(\mathsf{p} \in \mathbb{C}[x_1^{\pm 1}, \ldots, x_n^{\pm 1}]\) which satisfies the following conditions.
  \begin{itemize}
  \item[\textbf{Period condition:}] One has \(I_{\mathsf{p}} = \widetilde{I}_0^{(X, D)}\), where \(\widetilde{I}_0^{(X, D)}\) is the so called \emph{restriction of the constant term of regularized I-series to the anticanonical direction}.
  \item[\textbf{Calabi--Yau condition:}] There exists a relative compactification of a family \(\mathsf{p} \colon (\mathbb{C}^*)^n \rightarrow \mathbb{C}\) whose total space is a (non-compact) smooth Calabi--Yau variety \(Y\). We refer to such a compactification as a \emph{Calabi--Yau (partial) compactification}.
  \item[\textbf{Toric condition:}] There is a flat degeneration \(X \leadsto T\) to a toric variety \(T\) so that \(F(T) = N(\mathsf{p})\), where \(F(T)\) is a fan polytope of \(T\), and \(N(\mathsf{p})\) is the Newton polytope for \(\mathsf{p}\).
  \end{itemize}
\end{definition}

It was proved in~\cite{przyjalkowski2017compactification} that standard toric Landau--Ginzburg models of smooth Fano threefolds satisfy the stronger compactification condition: they have a log Calabi--Yau compactification.

\begin{definition}\label{definition:log-CY}
  A compactification of the family \(\mathsf{p} \colon (\mathbb{C}^*)^n \rightarrow \mathbb{C}\) to a family \(\mathsf{f} \colon Z \rightarrow \mathbb{P}^1\), where \(Z\) is smooth, and \(-K_Z \sim \mathsf{f}^{-1}(\infty)\), is called a \emph{log Calabi--Yau compactification}.
\end{definition}

\begin{remark}
  The notion of log Calabi--Yau compactification
differs from the notion of tame compactification by not requiring snc condition for the fiber over infinity. However, in most of known cases of log Calabi--Yau compactifications, in particular, for standard toric Landau--Ginzburg models, the fibers over infinity are snc. These notions may become different if we allow singularities over infinity (see~\cite{przyjalkowski2022singular}); however, even in these cases statements from~\cite{katzarkov2017bogomolov} and other papers hold.
\end{remark}

\subsection{Pencil of quartic surfaces}

For every smooth Fano threefold \(X\) with \(\rho(X) > 1\) and very ample anticanonical divisor \(-K_X\) we can always choose a toric change of variables for the corresponding Minkowski polynomial \(\mathsf{p}\) in~\cite{akhtar2012minkowski} such that there is a pencil \(\mathcal{S}\) of quartic surfaces on \(\mathbb{P}^3\) (arising from the Minkowski polynomial in a natural way, see below) that expands the compactification diagram (see Definition~\ref{definition:log-CY})
\begin{equation}\label{equation:diagram}
  \xymatrix
  {
    (\mathbb C^*)^3\ar@{^{(}->}[rr]\ar@{->}[d]_{\mathsf{p}}&&
    Y\ar@{->}[d]^{\mathsf{w}}\ar@{^{(}->}[rr]&&
    Z\ar@{->}[d]^{\mathsf{f}} \\
    \mathbb{C}\ar@{=}[rr] && \mathbb{C}\ar@{^{(}->}[rr] && \mathbb{P}^1
  }
\end{equation}
to the following commutative diagram:
\begin{equation}\label{equation:extented-diagram}
  \xymatrix
  {
    (\mathbb{C}^*)^3\ar@{^{(}->}[rr]\ar@{->}[d]_{\mathsf{p}} &&
    Y\ar@{->}[d]^{\mathsf{w}}\ar@{^{(}->}[rr] &&
    Z\ar@{->}[d]^{\mathsf{f}} &&
    V\ar@{-->}[ll]_{\chi}\ar@{->}[d]^{\mathsf{g}}\ar@{->}[rr]^{\pi} &&
    \mathbb{P}^3\ar@{-->}[lld]^{\phi} \\
    \mathbb{C}\ar@{=}[rr] &&
    \mathbb{C}\ar@{^{(}->}[rr] &&
    \mathbb{P}^1\ar@{=}[rr] &&
    \mathbb{P}^1 &&
  }
\end{equation}
Here the pencil \(\mathcal{S}\) is just a fibrewise homogenisation of the Minkowski polynomial under the embedding \(\mathbb{C}^3 \subset \mathbb{P}^3\), \([x, y, z] \mapsto [x : y : z : 1]\), the map \(\phi\) is a rational map given by the pencil \(\mathcal{S}\), the map \(\pi\) is a birational morphism explicitly constructed in~\cite{cheltsov2018katzarkov}, the threefold \(V\) is smooth, and \(\chi\) is a composition of flops. We refer the reader to~\cite[Section~1]{cheltsov2018katzarkov} for all important details and constructions related to the pencil \(\mathcal{S}\).

If the anticanonical class \(-K_X\) of a smooth Fano threefold is not very ample, it is possible to find a Laurent polynomial with non-reflexive Newton polytope (and the corresponding pencil) that gives the commutative diagram (\ref{equation:diagram}) and the analogue of the commutative diagram (\ref{equation:extented-diagram}) (see~\cite[\nopp 2.1,~2.2,~2.3,~9.1,~10.1]{cheltsov2018katzarkov}).

\subsection{Idea of the proof}

The proof of Theorem~\ref{theorem:DN-duality} is heavily based on the paper~\cite{cheltsov2018katzarkov}.

It was proved in~\cite[Lemma~1.5.3]{cheltsov2018katzarkov} that a general element \(\mathcal{S}_{\lambda}\) of the pencil \(\mathcal{S}\) is a du Val surface, and the minimal resolution \(\widetilde{\mathcal{S}}_{\lambda} \rightarrow \mathcal{S}_{\lambda}\) can be identified with a general fibre of the Landau--Ginzburg model. It is useful to treat with a generic fibre of the Landau--Ginzburg model as well. Namely, we can consider the du Val surface \(\mathcal{S}_{\Bbbk}\) over \(\Bbbk = \mathbb{C}(\lambda)\) associated with the pencil \(\mathcal{S}\). Then we can identify a generic fibre of the Landau--Ginzburg model with the minimal resolution \(\widetilde{\mathcal{S}}_{\Bbbk} \rightarrow \mathcal{S}_{\Bbbk}\).

Moreover, the pencil \(\mathcal{S}\) allows us to compute the lattice of invariant cycles from Theorem~\ref{theorem:DN-duality}.

\begin{notation}
  Let \(\mathbb{P}_{\Bbbk}\) be the ambient variety of the generic member \(\mathcal{S}_{\Bbbk}\) of the pencil \(\mathcal{S}\). We denote by \(A_{\mathcal{S}}\) the subgroup of \(\Cl(\mathcal{S}_{\Bbbk})\) generated by the linear equivalence classes \([\mathcal{C}_i]\) of the irreducible curves composing the base locus of the pencil and the restriction \(i^* \Pic(\mathbb{P}_{\Bbbk})\), where we denote the inclusion by \(i \colon \mathcal{S}_{\Bbbk} \hookrightarrow \mathbb{P}_{\Bbbk}\).
\end{notation}

\begin{remark}
  We have \(\mathbb{P}_{\Bbbk} = \mathbb{P}_{\Bbbk}^3\) in all cases except for the families \(2.1\) and \(10.1\), where we have \(\mathbb{P}_{\Bbbk} = \mathbb{P}_{\Bbbk}^1 \times \mathbb{P}_{\Bbbk}^2\).
\end{remark}

\begin{notation}
  Let \(\widetilde{\mathcal{S}_{\lambda}}\) and \(\widetilde{\mathcal{S}_{\Bbbk}}\) be the minimal resolution of a general member \(\mathcal{S}_{\lambda}\) and the generic member \(\mathcal{S}_{\Bbbk}\) of the pencil, respectively. Denote by \(L_{\lambda} \subset \Pic(\widetilde{\mathcal{S}_{\lambda}})\) and \(L_{\mathcal{S}} \subset \Pic(\widetilde{\mathcal{S}_{\Bbbk}})\) the subgroups generated by linear equivalence classes of exceptional divisors of the resolution and of strict transforms of curves from \(A_{\mathcal{S}}\).
\end{notation}

\begin{remark}
  By construction the subgroup \(L_{\lambda}\) is equipped with the \(\Gal(\Bbbk)\)-action, and the subgroup \(L_{\mathcal{S}}\) can be identified with the subgroup \(L_{\lambda}^{\Gal(\Bbbk)} \subset L_{\lambda}\). Note that the sublattice \(L_{\lambda}^{\Gal(\Bbbk)} \subset L_{\lambda}\) is primitive.
\end{remark}

The following statement is implicitly contained in the proof of~\cite[Main~Theorem]{cheltsov2018katzarkov}.

\begin{proposition}[see~{\cite[Subsections~1.4,~1.9,~1.13]{cheltsov2018katzarkov}}]
  Let \(X\) be a smooth Fano threefold, and let \(F\) be a general fibre of its log Calabi--Yau compactified standard toric Landau--Ginzburg model \(\mathsf{f} \colon Z \rightarrow \mathbb{P}^1\). We have
  \[
    \im(H^2(Z, \mathbb{Z}) \xrightarrow{\res} H^2(F, \mathbb{Z})) = L_{\mathcal{S}}.
  \]
  Moreover, it is a lattice of rank
  \(\rk(\Pic(\widetilde{\mathcal{S}_{\Bbbk}})) - \rk(\Pic(\mathcal{S}_{\Bbbk})) + \rk(A_{\mathcal{S}}) = 20 - \rk(\Pic(X))\).
\end{proposition}

\begin{example}
  The paper~\cite{cheltsov2018katzarkov} contains a number of examples of explicit application of this approach: we refer the reader to~\cite[Examples~1.7.1,~1.8.6,~1.10.11,~1.12.3,~1.13.2]{cheltsov2018katzarkov}.
\end{example}

\begin{remark}
  Note that the sublattice \(L_{\mathcal{S}} \subset \Pic(F)\) is primitive by~\cite[Theorem~4.24]{voisin2003hodge}.
\end{remark}

The proof of Theorem~\ref{theorem:DN-duality} can be summarized as follows:
\begin{enumerate}
\item Explicitly compute the sublattices \(L_{\lambda} \subset \Pic(\widetilde{\mathcal{S}_{\lambda}})\) and \(L_{\mathcal{S}} = L_{\lambda}^{\Gal(\Bbbk)} \subset \Pic(\widetilde{\mathcal{S}_{\Bbbk}})\).
\item Prove that \(L_{\mathcal{S}}^{\perp} \simeq H \oplus \Pic(X)\subset L_{\K3}\)  (using Remarks~\ref{remark:embedding} and~\ref{remark:orthogonal-complement}).
\end{enumerate}

\begin{remark}
  Note that Picard lattices \(\Pic(X)\) of smooth Fano threefolds are explicitly presented in Appendix~\ref{appendix:picard-lattices-fano}. Moreover, Remark~\ref{remark:orthogonal-complement} provides an explicit method to prove that \(L_{\mathcal{S}}^{\perp} \simeq H \oplus \Pic(X)\) in terms of discriminant forms on these integral lattices (see~\cite[\nopp 3.3]{ebeling2013lattices} for a review of the theory).
\end{remark}

\subsection{Base locus of the pencil \texorpdfstring{\(\mathcal{S}\)}{S}}

Let \(C_i\) be the irreducible curves composing the base locus of \(\mathcal{S}\). We can find these curves from the intersection of a surface \(\mathcal{S}_{\lambda}\) for a general \(\lambda \in \mathbb{C}\) with the irreducible components of the member \(\mathcal{S}_{\infty}\) over infinity. We want to find a minimal set of generators of \(A_{\mathcal{S}}\) over \(\mathbb{Z}\). We can think of \(A_{\mathcal{S}}\) as an abelian group that is given by the elements \(\mathcal{C}_i\), a general hyperplane section \(H_{\mathcal{S}}\), and the relations of linear equivalence. It is usually enough to consider only linear equivalence of the hyperplane sections of~\(\mathcal{S}_{\lambda}\). Then by means of integral linear algebra we can pick a minimal subset from the set of generators of \(A_{\mathcal{S}}\). Note that in general the obtained subset would not be an integral basis.

\subsection{Computation of the lattice \texorpdfstring{\(L_{\lambda}\)}{L\_l}}\label{subsection:algorithm}

It is well-known that the minimal resolution of a du Val surface can be obtained as a series of blow-ups at singular points. Denote by \(E_i^j\) the \(j\)-th exceptional curve of the resolution at the singular point \(p_i\). Recall that the exceptional divisor of a blow up at a point can be identified with the projectivization of the tangent cone at this point. We can use this to compute the intersection numbers of the form \(\widetilde{C_l} \cdot E_j^k\), where \(\widetilde{C}_l\) is the strict transform of \(l\)-th curve generating \(A_{\mathcal{S}}\).

Namely, we can pass to a local chart \(U\) containing the given singular point \(p \in \mathcal{S}_{\lambda}\) and compute the initial term of \(\left. \mathcal{S}_{\lambda} \right|_U\) at this point. After a homogenezation, we obtain \emph{the quadratic term} of the singularity --- a homogeneous quadratic polynomial that defines the projectivization of the tangent cone \(E_p = \mathbb{P}(T_p \mathcal{S}_{\lambda})\).

The following cases may occur:
\begin{enumerate}
\item \(E_p\) is irreducible, then \(E_p \simeq \mathbb{P}^1\), and the point has the type \(\mathbb{A}_1\);
\item \(E_p\) consists of two irreducible components, and the point has the type \(\mathbb{A}_n\) for \(n > 1\). If a curve \(\mathcal{C}\) is tangent to the only one (respectively, to the both) of the irreducible components of \(E_p\), then its strict transform \(\widetilde{\mathcal{C}}\) intersects one of the ``outer'' (respectively, one of the ``inner'') exceptional curves at the point \(p \in \mathcal{S}_{\lambda}\) with respect to the corresponding Dynkin diagram.
\item \(E_p\) is a rational curve \(\mathcal{E} \simeq \mathbb{P}^1\) counted with multiplicity two. This case occurs when the point has the type \(\mathbb{D}_n\) for \(n > 4\) or \(\mathbb{E}_6\), \(\mathbb{E}_7\), \(\mathbb{E}_8\).
\end{enumerate}

We can compute an integral basis of the lattice \(L_{\lambda}\) as follows:
\begin{enumerate}
\item If we manage to obtain a basis of \(A_{\mathcal{S}}\) as a subset in the set \(\{[C_l]\}\), then an integral basis of \(L_{\lambda}\) consists of the exceptional curves \(E_j^k\) of the resolution and strict transforms \(\widetilde{C_i}\) of curves from \(A_{\mathcal{S}}\).
\item In general, we have to compute the intersection matrix on the exceptional curves \(E_j^k\) of the resolution and strict transforms \(\widetilde{C}_l\) of curves generating the group \(A_{\mathcal{S}}\). The intersection matrix would be degenerate, and the lattice \(L_{\lambda}\) is isomorphic to the free abelian group \(\mathbb{Z}[E_j^k, \widetilde{C}_l]\) modulo the identities from the kernel of the obtained matrix.
\end{enumerate}

\subsection{Conventions and shortcuts}\label{subsection:conventions}

Actual computation of the lattice of invariant cycles \(L_{\mathcal{S}}\) and checking out the Dolgachev--Nikulin duality \(L_{\mathcal{S}} \simeq \Pic(X)^{\vee}\) is very tedious and time-consuming. To this end, we have implemented the algorithm of Subsection~\ref{subsection:algorithm} in Sage, which can be found at \url{https://github.com/MikhailOvcharenko/DN-duality}.

For the sake of brevity we invent the following handy convention. By construction our intersection matrix \(L_{\lambda}\) is a symmetric block matrix of the form
\[
  \begin{pmatrix}
    A & B^T \\
    B & C
  \end{pmatrix},
\]
where \(A\) is the intersection matrix of exceptional curves, \(C\) is the intersection matrix of strict transforms of the curves forming the base locus of the pencil \(\mathcal{S}\), and \(B\) is the remaining ``mixed'' intersection matrix.

Moreover, a general element \(\mathcal{S}_{\lambda}\) of the pencil has du Val singularities, hence \(A\) is a block diagonal matrix of the form \(A = \oplus (-A_i)\), where \(i\) runs over all singular points of the surface \(\mathcal{S}_{\lambda}\), and \(A_i\) is the Cartan matrix of the corresponding Dynkin diagram. In other words, the matrix \(A\) can be easily reconstructed in each case.

To exclude any misunderstanding, let us fix the notation for Cartan matrices which is used in Sage:
\[
  \mathbb{D}_n =
  \left (
  \begin{array}{cccc|cccc}
    \ddots & \ddots & \ddots & \cdots & \cdots & \cdots & \cdots & \cdots \\
    \ddots & 2 & -1 & 0 & 0 & 0 & 0 & 0 \\
    \ddots & -1 & 2 & -1 & 0 & 0 & 0 & 0 \\
    \vdots & 0 & -1 & 2 & -1 & 0 & 0 & 0 \\
    \hline
    \vdots & 0 & 0 & -1 & 2 & -1 & 0 & 0 \\
    \vdots & 0 & 0 & 0 & -1 & 2 & -1 & -1 \\
    \vdots & 0 & 0 & 0 & 0 & -1 & 2 & 0 \\
    \vdots & 0 & 0 & 0 & 0 & -1 & 0 & 2
  \end{array}
  \right ), \quad
  \mathbb{E}_6 =
  \begin{pmatrix}
    2 & 0 & -1 & 0 & 0 & 0 \\
    0 & 2 & 0 & -1 & 0 & 0 \\
    -1 & 0 & 2 & -1 & 0 & 0 \\
    0 & -1 & -1 & 2 & -1 & 0 \\
    0 & 0 & 0 & -1 & 2 & -1 \\
    0 & 0 & 0 & 0 & -1 & 2
  \end{pmatrix}.
\]

Therefore, we only have to provide the submatrix \((B \vert C)\) and the ordered list of du Val singularities of \(\mathcal{S}_{\lambda}\). In this setting we say that the intersection matrix \(L_{\lambda}\) is \emph{represented} by the submatrix \((B \vert C)\).

We also use the following notations.

\begin{notation}
  We correspond to the lattice \(L_{\mathcal{S}}\) the following data:
  \begin{itemize}
  \item \(G'\) is the matrix of selected generators of the discriminant group of the lattice;
  \item \(B'\) is the matrix of \(\mathbb{Q} / \mathbb{Z}\)-valued bilinear discriminant form in this basis;
  \item \(Q'\) is the vector of values of the \(\mathbb{Q} / 2 \mathbb{Z}\)-valued quadratic discriminant form.
  \end{itemize}
  We use the similar notation \(G''\), \(B''\), and \(Q''\) for the conjectural orthogonal complement \(H \oplus \Pic(X)\).
\end{notation}

\begin{notation}
  We denote by \(C_{(f, g)}\) a curve in \(\mathbb{P}^3\) or \(\mathbb{P}^1 \times \mathbb{P}^2\) defined by (bi-)homogeneous polynomials \(f, g\). We also use the same notation for a surface \(S_{(f)}\) in \(\mathbb{P}^3\) or \(\mathbb{P}^1 \times \mathbb{P}^2\).
\end{notation}

\subsection{Computation of \texorpdfstring{\(\Pic(X)\)}{Pic(X)}}\label{subsection:L_X}

If \(S\) is an anticanonical hypersurface in a smooth Fano threefold \(X\), then by the Lefschetz hyperplane theorem we have the inclusion \(i^* \colon \Pic(X) \hookrightarrow \Pic(S)\). Conversely, Beauville proves in~\cite{beauvillefano} that for a very general choice of a deformation of \(X\) and an anticanonical hypersurface \(S\), the map \(i^*\) is an isomorphism. Thus the Picard rank of a very general such \(S\) is equal to \(b_2(X)\), and \(H^2(X,\mathbb{Z})\) can be given a symmetric bilinear form \(\langle \cdot, \cdot \rangle_X\) defined to be
\[
  \langle \alpha, \beta \rangle_X : = \int_X c_1(X)\cup \alpha\cup \beta
\]
for \(\alpha,\beta \in H^2(X,\mathbb{Z})\). Let \(\Pic(X)\) denote \(H^2(X,\mathbb{Z})\) equipped with this pairing. There is a natural primitive embedding
\[
  \Pic(X) \hookrightarrow \Pic(S)
\]
for each anticanonical hypersurface of \(X\). Beginning from explicit descriptions of Fano threefolds given by~\cite{mori1981classification}, the lattices \(\Pic(X)\) can be computed without much difficulty.  In this section we present tools for performing these computations. In Appendix~\ref{appendix:picard-lattices-fano}, we present our computations. According to~\cite{mori1981classification}, Fano threefolds are presented in one of the following forms.
\begin{enumerate}
\item Complete intersections inside of products of projective spaces.
\item Double covers of well-understood Fano threefolds ramified along a smooth divisor.
\item Blow-ups of Fano threefolds of lower rank.
\item Projective bundles over \(\mathbb{P}^1\) or del Pezzo surfaces.
\end{enumerate}
We explain our techniques for computing Picard lattices in each of these cases individually.

\subsubsection{Complete intersections in products of projective spaces}
If \(X\) is a complete intersection in a product of projective spaces, calculation of intersection theory on \(X\) is elementary, and adjunction may be used to compute the Picard lattice of \(X\).

\subsubsection{Double covers of other Fano threefolds}
Computation of the Picard lattice of a double cover \(\phi:X \rightarrow Z\) ramified along a smooth divisor is again elementary. It is well-known that
\[
  K_X = \phi^*K_Z + R
\]
where \(R\) is the ramification divisor of \(\phi\). We apply this formula along with the fact that for three divisors \(D_1,D_2\), and \(D_3\) on \(Z\), we have
\[
  \phi^*D_1 \cdot \phi^* D_2 \cdot \phi^* D_3 = (\deg \phi)( D_1 \cdot D_2 \cdot D_3).
\]

\begin{example}
  Let \(X\) be a double cover of \(\mathbb{P}^1 \times \mathbb{P}^2\)  ramified along a divisor of bidegree \((2,2)\). Thus \(-K_X\) is given by \(\phi^* \mathcal{O}_{\mathbb{P}^1 \times \mathbb{P}^2}(1,2)\). Then \(\Pic(X)\) is generated by the classes \(\phi^*[p \times \mathbb{P}^2]\) and \(\phi^*[\mathbb{P}^1 \times H]\), where \(H\) is a hyperplane section in \(\mathbb{P}^2\). Then we computes that
  \[
    -K_X\cdot \phi^*[p\times \mathbb{P}^2] \cdot \phi^*[\mathbb{P}^1 \times H]  = (\deg \phi)([p\times \mathbb{P}^2] + 2[\mathbb{P}^1 \times H])\cdot [p\times \mathbb{P}^2] \cdot [\mathbb{P}^1 \times H] = 4.
  \]
  Similarly, \(-K_X \cdot [p \times \mathbb{P}^2]^2 = 0\) and \(-K_X \cdot [\mathbb{P}^1 \times H]^2 = 2\).
\end{example}

\subsubsection{Blow-ups of Fano threefolds of lower Picard rank}

One may compute explicitly how the lattice \(\langle \cdot, \cdot \rangle_X\) changes under blow-up. We have the following computation.

\begin{proposition}\label{proposition:blupform}
  Let \(X\) be a smooth Fano threefold and let \(\pi:\widetilde{X}\rightarrow {X}\) be the blow up of \(X\) along a smooth curve \(Y\). Let \(\widetilde{Y}\) be the exceptional divisor of \(\pi\). Then \(\Pic(\widetilde{X})\) is generated by \(\widetilde{Y}\) and the pullback of divisors in \(X\). Furthermore, the following assertions hold.
  \begin{enumerate}
  \item For \(D_1,D_2\) divisors on \(X\), \(\langle \pi^*D_1, \pi^*D_2 \rangle_{\widetilde{X}} = \langle D_1,D_2 \rangle_X\).
  \item For \(D\) a divisor on \(X\), \(\langle D,\widetilde{Y} \rangle_{\widetilde{X}} = D \cdot Y\).
  \item \(\langle \widetilde{Y},\widetilde{Y} \rangle_{\widetilde{X}} = \chi(Y)\).
  \end{enumerate}
\end{proposition}

\begin{proof}
  The condition that \(X\) and \(\widetilde{X}\) are both smooth Fano threefolds ensures that \(|-K_X|\) and \(|-K_{\widetilde{X}}|\) both contain smooth members. The rest of the proof is simply an application of~\cite[Lemma 2.11]{iskovskikh1977fano}. The claim that \(\Pic(X)\) is generated by \(\widetilde{Y}\) and the pullback of \(\Pic(X)\) is part of the aforementioned lemma. Let \(\pi^*Y\) be the tautological class on the exceptional divisor \(\widetilde{Y}\) of \(\widetilde{X}\). We notice that  \(-K_{\widetilde{X} }= -\pi^* K_X  - \widetilde{Y}\). For any divisors \(D_1,D_2\) on \(X\) we have
  \begin{align*}
    -K_{\widetilde{X}} \cdot \pi^* D_1 \cdot \pi^*D_2 &= -\pi^*K_X\cdot \pi^* D_1\cdot \pi^*D_2 - \widetilde{Y}\cdot \pi^* D_1 \cdot \pi^* D_2 \\
                                          & =  -K_X\cdot D_1\cdot D_2  = \langle D_1,D_2 \rangle_X.
  \end{align*}
  Let \(D\) be a divisor on \(X\). We find that if \(F\) denotes the class of a fibre of the contraction map \(\widetilde{Y}\rightarrow Y\),
  \begin{align*}
    -K_{\widetilde{X}} \cdot \pi^*D \cdot \widetilde{Y} & = -\pi^* K_X \cdot \pi^*D \cdot \widetilde{Y} - \pi^*D \cdot (\widetilde{Y})^2\\
                                              &= \pi^* D \cdot (\deg N_{Y/X} F - \pi^*Y) \\
                                              & = D \cdot Y.
  \end{align*}
  Finally, we have
  \begin{align*}
    (-\pi^*K_X - \widetilde{Y}) \cdot \widetilde{Y}^2 &= -\pi^*K_X \cdot \widetilde{Y}^2 - \widetilde{Y}^3 \\
                                                & = -\pi^* K_X \cdot (\pi^*Y - \deg N_{Y/X}F) + K_X \cdot Y + \chi(Y) \\
                                                & =  -\pi^*K_X\cdot \pi^* Y + K_X \cdot Y + \chi(Y) = \chi(Y).
  \end{align*}
  This completes the calculation.
\end{proof}

This proposition allows us to compute the Picard lattices of Fano threefolds which are obtained by smooth blow ups along smooth curves on Fano threefolds of lower Picard rank. Similarly, one may compute the Picard lattice of a Fano threefold \(X\) blown up at a point.

\begin{remark}
  Suppose \(X\) is a Fano threefold of Picard rank 1 whose Fano index is  divisible by 2 and \((-K_X)^3 = 16n\). Suppose that \(S_1,S_2 \in |-(\tfrac{1}{2})K_X|\) are smooth sections meeting transversally in an elliptic curve \(E\). It is a direct application of Proposition~\ref{proposition:blupform} to see that \(\Bl_EX\) has Picard lattice isomorphic to \(H(n)\), the lattice with Gram matrix
  \[
    \begin{pmatrix}
      0 & n \\
      n & 0
    \end{pmatrix}.
  \]
 This can be observed directly by looking at Appendix~\ref{appendix:picard-lattices-fano}. This phenomenon occurs for Families \textnumero 2.1, 2.3, 2.5, 2.10, and 2.14.
\end{remark}

\subsubsection{Projective bundles}

Finally, if \(X\) is a projective bundle over a variety \(Y\) associated to a vector bundle \(\mathcal{E}\), then we may use the following result.

\begin{theorem}[see, e.g.,~{\cite[pp. 606]{griffiths2014principles}}]\label{theorem:GH}
  Let \(Y\) be any smooth complex projective variety, and \(\mathcal{E}\) a complex vector bundle of rank \(r\). Then the ring \(H^*(\mathbb{P}(\mathcal{E}),\mathbb{Z})\) is generated as an \(H^*(Y,\mathbb{Z})\)-module by \(\zeta = c_1(\mathcal{O}_{\mathbb{P}(\mathcal{E})}(1))\), with the single relation
  \[
    \zeta^r + \pi^* c_1(\mathcal{E})\zeta^{r-1} + \cdots + \pi^* c_r(\mathcal{E}) = 0.
  \]
  Here \(\mathcal{O}_{\mathbb{P}(\mathcal{E})}(1)\) is the tautological line bundle on the projective bundle \(X\).
\end{theorem}

Using the same notation as in Theorem~\ref{theorem:GH} and letting \(p:\mathbb{P}(\mathcal{E}) \rightarrow Y\) be the natural fibration morphism, we have the following formula given by Reid in~\cite[pp. 349]{reid46young},
\begin{align}\label{equation:Reid}
  K_{\mathbb{P}(\mathcal{E})} &= p^*(K_Y + \det \mathcal{E} ) \otimes \mathcal{O}_{\mathbb{P}(\mathcal{E})}(-r) \\
                              & = -r \zeta + p^* K_Y - p^*(\det \mathcal{E})
\end{align}
to compute the Picard lattice of a projective bundle. Here, \(r\) is the rank of \(\mathcal{E}\) and \(\zeta\) is the relative hyperplane class. These two facts allow us to effectively compute the Picard lattice of a Fano projective bundle.

\begin{example}
  Let \(Y=\mathbb{P}^2\), and let \(\mathcal{E}\) be the vector bundle \(\mathcal{O}_{\mathbb{P}^2} \oplus \mathcal{O}_{\mathbb{P}^2}(2)\) with projection map \(p : \mathbb{P}(\mathcal{E}) \rightarrow \mathbb{P}^2\) and let \(X\) be the projective bundle \(\mathbb{P}(\mathcal{E})\). Thus \(X\) is a unique member of Family 2.36 in the list of Mori--Mukai. We let \([H]\) be the class of a hyperplane on \(\mathbb{P}^2\) and remark that \(p^*[H ]\cdot p^*[H] \)  is the class of a fibre \(f\) of \(p\), and that the intersection of the class of \(f\) and \(\zeta\) is the class of a single point. We compute that \(c_1(\mathcal{E}) = [2H] = \det \mathcal{E}\). Thus Theorem~\ref{theorem:GH} gives
  \[
    \zeta^2 = -p^*[2H]\cdot  \zeta.
  \]
  Thus
  \[
    \zeta^3 = (p^* [2H])^2 \cdot \zeta = 4,
  \]
  and Equation~\eqref{equation:Reid} gives
  \[
    [K_X] = -2 \zeta + p^*[H].
  \]
  Furthermore, \(H^2(X,\mathbb{Z})\) is generated by \(p^*[H]\) and \(\zeta\). Thus we find
  \begin{align*}
    -[K_X]\cdot \zeta\cdot\zeta & = 2 \zeta^3 + p^*[H]\cdot p^*[2H] \cdot \zeta = 8 +2 = 10, \\
     -[K_X]\cdot p^* [H]\cdot \zeta  & =  2 p^*[2H] \cdot p^*[H] \cdot \zeta + (p^*[H])^2 \cdot \zeta = 5, \\
    -[K_X]\cdot p^* [H]\cdot p^*[H ] & = 2 p^*[H] \cdot p^*[H] \cdot \zeta = 2.
  \end{align*}
  Thus the Picard lattice of \(X\) has Gram matrix given by
  \[
    \left( \begin{matrix} 2 & 5 \\ 5 & 10 \end{matrix} \right).
  \]
\end{example}

\begin{remark}
  The authors of~\cite{coates2016quantum} have produced toric complete intersection models for many Fano threefolds. In these situations, it is possible to compute the Picard lattice of \(X\) using toric geometry directly.
\end{remark}

\subsubsection{Varieties \(\mathbb{P}^1 \times S_d\)}

It remains to calculate the Picard lattices of anticanonical sections of \(Y_d = \mathbb{P}^1 \times S_d\), where \(S_d\) is a del Pezzo surface of degree \(d\). We see that the Picard group of \(S_{9-d}\) is spanned by the class \(H\) of the pullback of hyperplane section in \(\mathbb{P}^2\) and the exception classes \(E_i\) of blown up points. Therefore, we have that the Picard group of \(S_{9-d} \times \mathbb{P}^1\) is spanned by the classes of divisors
\[
  R_i = \mathbb{P}^1 \times E_i, G = \mathbb{P}^1 \times H, S= p \times S_{9-d},
\]
where \(p\) is a point in \(\mathbb{P}^1\). The anti-canonical divisor on \(S_{9-d}\) is given by \(3H - \sum_{i = 1}^d E_i\), and the anticanonical divisor on \(\mathbb{P}^1 \times S_{9-d}\) is given by the tensor product of the pullback of anticanonical classes on \(S_{9-d}\) and \(\mathbb{P}^1\). Therefore,
\[
  -K_{Y_d} = \left[ 3G - \sum_{i=1}^d R_i\right] + 2 \left[S\right].
\]
We calculate the Picard lattice of very general anticanonical hypersurfaces of \(Y_d\) by taking the intersection form on \(\Pic(Y_d)\) written as
\[
  \left< D_1, D_2 \right>_{Y_d} = D_1 \cdot D_2 \cdot (-K_{Y_d}),
\]
as given by~\cite{beauvillefano}. We may actually complete this calculation easily. We see that \(\Pic(X)\) has rank \(d + 2\), and
\[
  \left<R_i, R_j\right>_{Y_d} = -2\delta_{ij}, \left<R_i,G\right>_{Y_d} = 0, \left< R_i, S \right>_{Y_d} = -1, \left<G,S\right>_{Y_d} = 3, \left<G,G\right>_{Y_d} = 2  \text{ and } \left< S, S \right>_{Y_d} = 0.
\]
Since \(R_i\), \(S\), and \(G\) are generators of the Picard lattice of \(Y_d\), we have proved the following.

\begin{proposition}
  Let \(Y_d = \mathbb{P}^1 \times S_d\), then the Picard lattice of a general hypersurface in \(Y_d\) is of rank \(d+2\), and may be represented by the matrix
  \[
    N_d= \left(
      \begin{matrix}
        -2 & 0 & 0 & \cdots & 0 & 0 & -1 \\
        0 & -2 & 0 & \cdots & 0 & 0 & -1 \\
        0 & 0 & -2 & \ddots & 0 & 0 & -1 \\
        \vdots & \vdots & \ddots & \ddots  & \ddots & \vdots & \vdots \\
        0 & 0 & 0 & \ddots & -2 & 0 & -1 \\
        0 & 0 & 0 & \cdots & 0 & 2 & 3 \\
        -1 & -1 &-1 & \cdots & -1 & 3 & 0
      \end{matrix}
    \right).
  \]
\end{proposition}

It is well-known that for \(S_d\) with \(d \neq 9,8,7\) and \(S_d \neq \mathbb{P}^1 \times \mathbb{P}^1\), there are canonically associated root lattices \(R_d=\mathbb{E}_8,\mathbb{E}_7,\mathbb{E}_6,\mathbb{D}_5,\mathbb{A}_4,\mathbb{A}_2 \times \mathbb{A}_1\) for \(d = 1, \ldots, 6\) (see~\cite[Chapter~8]{dolgachev2012classical} for details). These lattices are obtained as the orthogonal complements of \(K_{S_d}\) in \(\NS(S_d)\). We find the following relation to the Picard lattice of the anticanonical K3 surface in \(\mathbb{P}^1\times S_d\).

\begin{proposition}\label{proposition:dppic}
  Let \(S_d\) be a del Pezzo surface of degree \(d\), and let \(d = 1, \ldots, 6\). The lattice \(N_{\mathbb{P}^1\times S_d}\) is isomorphic to \(H \oplus R_d(2)\), where \(R_d\) is the root lattice described above and \(R_d(2)\) indicates the lattice on the same group as \(R_d\) but with intersection form multiplied by \(2\).
\end{proposition}

\begin{proof}
  Let \(D_1\) and \(D_2\) be divisors on \(S_d\) which are orthogonal to \(S_d\). Then we have
  \[
    -(\mathbb{P}^1 \times D_1)\cdot(\mathbb{P}^1\times D_2)\cdot K_{\mathbb{P}^1\times S_d} = -D_1\cdot D_2\cdot(\mathbb{P}^1 \times K_{S_d}  - 2p \times S_d) = 2D_1\cdot D_2
  \]
  for \(p\) a point in \(\mathbb{P}^1\). Hence we have an embedding of \(R_d(2)\) into \(N_{\mathbb{P}^1\times S_d}\). This lattice is spanned by classes \(3R_i - G\) for \(i=1, \ldots, d\). It remains to prove that the orthogonal complement of \(R_d(2)\) is isomorphic to \(H\). We see that the orthogonal complement contains \(\mathbb{P}^1 \times K_{S_d}\). We see that
  \[
    \langle 3G-S , 3R_i - G \rangle = 0.
  \]
  The lattice spanned by \(3G-S\) and \(\mathbb{P}^1 \times K_{S_d}\) is isomorphic to \(H\). Thus we have that \(H \oplus R_d(2)\) is an overlattice of \(N_{\mathbb{P}^1\times S_d}\). It can be checked explicitly that the two lattices have the same discriminant and thus are isomorphic.
\end{proof}

\subsection{Discussion} Here we will make a number of comments regarding the lattices that we compute in Appendices~\ref{appendix:picard-lattices-fano}--\ref{appendix:parametrized}.

\subsubsection{Isomorphisms between lattices appearing in Appendix~\ref{appendix:picard-lattices-fano}} We have already observed that the Picard lattices of Family \textnumero\({2.28}\) and Family \textnumero\({2.33}\) are the same. One can identify many other pairs of Picard lattices by comparing their discriminant lattices using Sage.
\begin{itemize}[\quad -]
\item Picard rank 2 (these isomorphisms have already been observed by Mase~\cite{mase2014families}):
  \[
    \begin{array}{cccc}
      \{\MM_{2.6}, \MM_{2.32}\}, & \{\MM_{2.4}, \MM_{2.28}, \MM_{2.33}\}, & \{\MM_{2.15}, \MM_{2.30}\}, & \{\MM_{2.8}, \MM_{2.35}\}, \\
      \{\MM_{2.9}, \MM_{2.19}, \MM_{2.27}\}, & \{\MM_{2.11}, \MM_{2.31}\}, & \{\MM_{2.7}, \MM_{2.23}, \MM_{2.29}\}.
    \end{array}
  \]
\item Picard rank 3:
  \[
    \begin{array}{cccc}
      \{\MM_{3.1}, \MM_{3.27}\}, & \{\MM_{3.2}, \MM_{3.28}\}, & \{\MM_{3.3}, \MM_{3.5}, \MM_{3.11}, \MM_{3.17}\}, & \{\MM_{3.8}, \MM_{3.15}\}, \\
      \{\MM_{3.9}, \MM_{3.29}\}, & \{\MM_{3.14}, \MM_{3.22}, \MM_{3.26}\}, & \{\MM_{3.21}, \MM_{3.24}\}.
    \end{array}
  \]
\item Picard rank 4:
  \[
    \begin{array}{ccc}
      \{\MM_{4.2}, \MM_{4.8}\} & \{\MM_{4.5}, \MM_{4.7}\} & \{\MM_{4.6}, \MM_{4.13}\}.
    \end{array}
  \]
\end{itemize}

In certain cases, these isomorphisms can be seen easily, however, some are more cryptic.

\subsubsection{Relationship to Siegel modular threefolds and Hilbert modular surfaces}\label{subsubsection:siegel-hilbert} In Golyshev's work~\cite{golyshev2004modularity,golyshev2007classification}, an important role is played by modular curves \(X_0(n)^+\) which are classically known to be moduli spaces of elliptic curves with level \(n\) structure. It was noticed by Dolgachev~\cite{dolgachev1996mirror} that these curves are moduli spaces of \(\mathbb{E}_8(-1)^2 \oplus H \oplus \langle 2n \rangle\)-polarized K3 surfaces, which are the Shioda--Inose partners (see~\cite{morrison1984k3}) of abelian surfaces of type \(E_{\tau_1} \times E_{\tau_2}\), where \(E_{\tau_1}\) and \(E_{\tau_2}\) are \(n\)-isogenous elliptic curves. More generally, a K3 surface with Picard lattice \(M\) of rank 18 admits a Shioda--Inose structure if and only if its transcendental lattice is of the form \(H \oplus L\) for a lattice \(L\) of rank 2 and signature \((1,1)\). In this case, the moduli space of \(M\)-polarized K3 surfaces is closely related to a Hilbert modular surface~\cite{harder2011moduli}. In particular, the moduli space of \(\Pic(X)^\vee\)-polarized K3 surfaces is precisely a classical Humbert surface if and only if \(\Pic(X)^\vee\) is isomorphic to a lattice of the following form (see, e.g.,~\cite{elkies2014k3}):
\begin{equation}\label{equation:kelkies}
  \left(\begin{matrix}-2 & 0 \\ 0 & D/2 \end{matrix} \right),\qquad \left(\begin{matrix}-2 & 1 \\ 1 & -(D+1)/2 \end{matrix} \right)
\end{equation}
for a positive integer \(D\) congruent to \(0\) or \(1 \bmod 4\), respectively.

A K3 surface of Picard lattice \(M\) of rank 17 admits Shioda--Inose structure if and only if its transcendental lattice is of the form \(H^2 \oplus \langle -2n \rangle\) for a positive integer \(n\). In this case the moduli space of \(M\)-polarized K3 surfaces is related to the moduli space of \((1,n)\)-polarized abelian surfaces~\cite{gritsenko1998minimal}. No K3 surface of Picard rank less than 17 admits Shioda--Inose structure. With this in mind we have the following result.
\begin{itemize}[\quad -]
\item The fibres of all Landau--Ginzburg mirrors of Fano threefolds of Picard rank 2 are parametrized by a quotient of a Hilbert modular surface. In the cases listed in Figure~\ref{figure:modular2}, the moduli space of \(\Pic(X)^\vee\)-polarized K3 surfaces is identified with a classical Humbert surface.

\begin{table}[H]
  \begin{minipage}[t]{0.45\textwidth}
    \begin{tabular}{|c|c|}
      \hline
      Family of Fano threefolds & Discriminant \\
      \hline
      \hline
      2.1 & \(1\) \\
      \hline
      2.2, 2.10 & \(4\) \\
      \hline
      2.36 & \(5\) \\
      \hline
      2.8, 2.35 & \(8\) \\
      \hline
      2.4, 2.28, 2.33 & \(9\) \\
      \hline
      2.16, 2.30  & \(12\) \\
      \hline
    \end{tabular}
  \end{minipage}%
  \begin{minipage}[t]{0.45\textwidth}
    \begin{tabular}{|c|c|}
      \hline
      Family of Fano threefolds & Discriminant \\
      \hline
      \hline
      \(2.12, 2.31\) & \(13\) \\
      \hline
      \(2.7, 2.24, 2.29\) & 16 \\
      \hline
      \(2.17\) & 20 \\
      \hline
      \(2.23\) & 24 \\
      \hline
      \(2.22\) & 28 \\
      \hline
      \(2.21\) & 29 \\
      \hline
    \end{tabular}
  \end{minipage}
  \caption{Fano threefolds whose Picard lattices are of the form described in~\eqref{equation:kelkies}.}\label{figure:modular2}
\end{table}

\item In the cases listed in Figure~\ref{figure:modular} there is an isomorphism between \(\Pic(X)\) and \(H \oplus \langle -2n \rangle\). Therefore, there is an isomorphism between \(\Pic(X)^\vee\) and \(\mathbb{E}_8(-1)^2 \oplus \langle 2n \rangle\).

\begin{table}[H]
  \begin{minipage}[t]{0.45\textwidth}
    \begin{tabular}{|c|c|}
      \hline
      Family of Fano threefolds & Picard lattice \\
      \hline
      \hline
      3.31 & \(H \oplus \langle -12 \rangle\) \\
      \hline
      3.30 & \(H \oplus \langle -14 \rangle\) \\
      \hline
      3.2, 3.28  & \(H \oplus \langle -16 \rangle\) \\
      \hline
      3.25 & \(H \oplus \langle -20 \rangle\) \\
      \hline
      3.21, 3.24 & \(H \oplus \langle -22 \rangle\) \\
      \hline
    \end{tabular}
  \end{minipage}%
  \begin{minipage}[t]{0.45\textwidth}
    \begin{tabular}{|c|c|}
      \hline
      Family of Fano threefolds & Picard lattice \\
      \hline
      \hline
      3.18 & \(H \oplus \langle -26 \rangle\) \\
      \hline
      3.3, 3.5, 3.11, 3.17  & \(H \oplus \langle -28 \rangle\) \\
      \hline
      3.6 & \(H \oplus \langle -32 \rangle\) \\
      \hline
      3.8, 3.15 & \(H \oplus \langle -34 \rangle\) \\
      \hline
    \end{tabular}
  \end{minipage}
  \caption{Fano threefolds whose Picard lattices are of the form \(H \oplus \langle -2n \rangle\) for \(n \in \mathbb{Z}_{> 0}\).}\label{figure:modular}
  \end{table}
\end{itemize}

The moduli spaces of \(\mathbb{E}_8(-1)^2 \oplus \langle 2n \rangle\)-polarized K3 surfaces is studied by Gritsenko--Hulek in~\cite{gritsenko1998minimal}, and it is referred to as \(\mathcal{A}^*_{n}\) in those works. The moduli space \(\mathcal{A}_n^*\) is a quotient of the moduli space of \((1,n)\)-polarized abelian surfaces, which is denoted \(\mathcal{A}_n\). It is known~\cite{gritsenko1994irrationality} that for \(n \geq 13\) and \(n \neq 14,  15, 16, 18, 20, 24, 30, 36\) the moduli space \(\mathcal{A}_n\) is not unirational. We have the following consequence of Corollary~\ref{corollary:ssur} and Theorem~\ref{theorem:DN-duality}.

\begin{proposition}
  For \(n = 6, 7, 8, 10, 11, 13, 14, 16, 17\), the moduli space \(\mathcal{A}_n^*\) is uniruled.
\end{proposition}

Observe that Gritsenko's result (see~\cite{gritsenko1994irrationality}) shows that when \(n=13,17\) the moduli space \(\mathcal{A}_n\) is not unirational. In~\cite{gritsenko2014uniruledness}, Gritsenko and Hulek use techniques coming from the study of modular forms on orthogonal modular varieties to prove that \(\mathcal{A}_{21}^*\) is uniruled. We remark that Gross and Popescu prove rationality and unirationality results for many moduli spaces \(\mathcal{A}_n\) and their covers in~\cite{gross2001calabi,gross2011calabi}. Gross and Popescu's results are related to the existence of abelian surface fibred Calabi--Yau threefolds, which is similar in flavour to our results.

\begin{remark}
  Our results also show that the Humbert surfaces \(\mathcal{H}_d\) with \(d = 1, 4, 5, 8, 9, 12, 13 , 16, 20, 24, 28\), and \(19\) are uniruled, however, this result is well-known (see, e.g.,~\cite{elkies2014k3}).
\end{remark}

\subsubsection{Relative Fourier--Mukai partners, or absence thereof} It is well-known that Calabi--Yau varieties may admit more than one mirror Calabi--Yau variety. This phenomenon is closely related to the existence of Fourier--Mukai partners, and in the case of K3 surfaces, this culminates in Orlov's derived Torelli theorem for K3 surfaces~\cite{orlov1997equivalences}. More precisely, two K3 surfaces \(S_1\) and \(S_2\) are derived equivalent if and only if their transcendental lattices are Hodge isometric.

We may ask whether Fano threefolds also admit multiple mirrors. More precisely, we can ask whether to any of the Landau--Ginzburg models appearing in this paper there are distinct, non-isomorphic families of K3 surfaces whose transcendental variation of Hodge structure is isomorphic, or equivalently, which are fibre-wise Fourier--Mukai partners to the fibres of \((Y,\mathsf{w})\). The answer to this question is \emph{no}; according to, e.g.,~\cite[Corollary 3.8]{huybrechts2016K3}, any K3 surface with Picard rank at least 12 or which admits an elliptic fibration with section has no non-trivial Fourier--Mukai partners. In all cases except Family \textnumero 9.1 and Family \textnumero 10.1 the Picard rank of the fibres of the Landau--Ginzburg mirror has Picard rank at least \(12\). In the case of Family \textnumero 9.1 and Family \textnumero 10.1, each fibre admits an elliptic fibration with section. Therefore, we conclude that any family of K3 surfaces whose transcendental variation of Hodge structure is isomorphic to that of \((Y,\mathsf{w})\) is indeed fibre-wise isomorphic to \((Y,\mathsf{w})\).

\subsection{Further directions}

\begin{enumerate}[(1)]
\item The results in this paper place strong constraints on the K3 surface fibres of Landau--Ginzburg mirrors of Fano threefolds, or more precisely, if \(\Pic(X)\) is the Picard lattice of the very general anticanonical hypersurface in \(X\), then Corollary~\ref{corollary:ssur} says that there is a ruling on \(\mathcal{M}_{\Pic(X)^\vee}\) of a very particular type. We explain in Remark~\ref{remark:smifano} that the same thing should hold for mirrors of semi-Fano varieties as well. It would be interesting to classify moduli spaces of K3 surfaces satisfying the criteria of Corollary~\ref{corollary:ssur} and to relate this classification to the classification of semi-Fano varieties. Similarly, our results provide tight constraints on the moduli spaces of Calabi--Yau varieties which can be mirror to anticanonical hypersurfaces in semi-Fano varieties in higher dimensions.

\item Given a family of Landau--Ginzburg models, for convenience, denoted by \((Z_t,D_t,\mathsf{f}_t)\), there is an associated isomonodromic deformation attached to the local system \(j_t^*R^n\mathsf{f}_{t*}\underline{\mathbb{Z}}_{Z_t}\). Here \(j_t :U_t\rightarrow \mathbb{P}^1\) indicates inclusion of the collection of smooth values of \(\mathsf{f}_t\) into \(\mathbb{P}^1\). It would be interesting to investigate the relationship between this isomonodromy problem and the deformation theory of the triple \((Z,D,\mathsf{f})\). To deformations of \((Z,D,\mathsf{f})\) and to the isomonodromy of \(j_{t*}R^n\mathsf{f}_{t*}\underline{\mathbb{Z}}_{Z_t}\) one may associate Frobenius manifolds. It would be interesting to understand the relationship between these two Frobenius manifolds.

\item Roughly stated, Theorem~\ref{theorem:defmink} says that if \((Z,D,\mathsf{f})\) is a Landau--Ginzburg model arising from a weakly nondegenerate (see Definition~\ref{definition:wnd}) Laurent polynomial, then the same is true for any small deformation of \((Z,D,\mathsf{f})\). Under deformation, the combinatorics of the face polynomials of the corresponding Laurent polynomial can change. It would be interesting to give a precise description of how the combinatorics of this polynomial can change, and what this change means under mirror symmetry. See Remark~\ref{remark:spec} for some speculation in this direction.

\item In Section~\ref{section:k3}, we show that the moduli spaces of K3 surfaces which appear as the fibres of Landau--Ginzburg mirrors of Fano threefolds are of a very special type. The analogous result for  Fano threefolds of Picard rank 1 (see~\cite{golyshev2004modularity,golyshev2007classification}) were instrumental in classifications of K3 fibred Calabi--Yau threefolds of high Picard rank in~\cite{doran2020calabi}. According to the philosophy of~\cite{doran2017mirror}, one would expect that Calabi--Yau threefolds fibred by \(L\)-polarized K3 surfaces, where \(L\) is one of the Picard lattices appearing in Appendices~\ref{appendix:rank-02}--\ref{appendix:parametrized} should be related to the classification of Fano threefolds obtained by smoothing Tyurin degenerate Calabi--Yau threefolds coming from pairs of Fano threefolds. It would be interesting to carry out this program for Fano threefolds of higher Picard rank, and understanding the moduli spaces of these K3 surfaces will be vital in this project.
\end{enumerate}

\appendix

\section{Picard lattices of anticanonical hypersurfaces in Fano threefolds}\label{appendix:picard-lattices-fano}

This section lists the Picard lattices of all Fano threefolds except \(\mathbb{P}^1 \times S_d\), whose Picard lattices are described in Proposition~\ref{proposition:dppic}. See also~\cite{mase2014families} for the rank 2 case.

\begin{table}[H]
  \begin{minipage}[t]{0.25\textwidth}
    \centering
    \noindent
    \begin{tabular}[t]{|c|c|}
      \hline
      \textnumero & \(\Pic(X)\) \\
      \hline
      \hline
      2.1 & \(\left( \begin{matrix} 0 & 1 \\ 1 & 2 \end{matrix} \right)\) \\
      \hline
      2.2 & \(\left( \begin{matrix} 0 & 2 \\ 2 & 2 \end{matrix} \right)\) \\
      \hline
      2.3 & \(\left( \begin{matrix} 0 & 2 \\ 2 & 4 \end{matrix} \right)\) \\
      \hline
      2.4 & \(\left( \begin{matrix} 0 & 3 \\ 3 & 4 \end{matrix} \right)\) \\
      \hline
      2.5 & \(\left( \begin{matrix} 0 & 3 \\ 3 & 6 \end{matrix} \right)\) \\
      \hline
      2.6 & \(\left( \begin{matrix} 2 & 4 \\ 4 & 2 \end{matrix} \right)\) \\
      \hline
      2.7 & \(\left( \begin{matrix} 8 & 8 \\ 8 & 6 \end{matrix} \right)\) \\
      \hline
      2.8 & \(\left( \begin{matrix} -2 & 0 \\ 0 & 4 \end{matrix} \right)\) \\
      \hline
      2.9 & \(\left( \begin{matrix} 8 & 7 \\ 7 & 4 \end{matrix} \right)\) \\
      \hline
    \end{tabular}
  \end{minipage}%
  \begin{minipage}[t]{0.25\textwidth}
    \centering
    \begin{tabular}[t]{|c|c|}
      \hline
      \textnumero & \(\Pic(X)\) \\
      \hline
      \hline
      2.10 & \(\left( \begin{matrix} 0 & 4 \\ 4 & 8 \end{matrix} \right)\) \\
      \hline
      2.11 & \(\left( \begin{matrix} -2 & 1 \\ 1 & 6 \end{matrix} \right)\) \\
      \hline
      2.12 & \(\left( \begin{matrix} 4 & 6 \\ 6 & 4 \end{matrix} \right)\) \\
      \hline
      2.13 & \(\left( \begin{matrix} 2 & 6 \\ 6 & 6 \end{matrix} \right)\) \\
      \hline
      2.14 & \(\left( \begin{matrix} 0 & 5 \\ 5 & 10 \end{matrix} \right)\) \\
      \hline
      2.15 & \(\left( \begin{matrix} 6 & 6 \\ 6 & 4 \end{matrix} \right)\) \\
      \hline
      2.16 & \(\left( \begin{matrix} -2 & 2 \\ 2 & 8 \end{matrix} \right)\) \\
      \hline
      2.17 & \(\left( \begin{matrix} 0 & 5 \\ 5 & 6 \end{matrix} \right)\) \\
      \hline
      2.18 & \(\left( \begin{matrix} 0 & 4 \\ 4 & 2 \end{matrix} \right)\) \\
      \hline
    \end{tabular}
  \end{minipage}%
  \begin{minipage}[t]{0.25\textwidth}
    \centering
    \begin{tabular}[t]{|c|c|}
      \hline
      \textnumero & \(\Pic(X)\) \\
      \hline
      \hline
      2.19 & \(\left( \begin{matrix} -2 & 1 \\ 1 & 8 \end{matrix} \right)\) \\
      \hline
      2.20 & \(\left( \begin{matrix} -2 & 3 \\ 3 & 10 \end{matrix} \right)\) \\
      \hline
      2.21 & \(\left( \begin{matrix} -2 & 4 \\ 4 & 6 \end{matrix} \right)\) \\
      \hline
      2.22 & \(\left( \begin{matrix} -2 & 2 \\ 2 & 10 \end{matrix} \right)\) \\
      \hline
      2.23 & \(\left( \begin{matrix} 0 & 4 \\ 4 & 6 \end{matrix} \right)\) \\
      \hline
      2.24 & \(\left( \begin{matrix} 2 & 5 \\ 5 & 2 \end{matrix} \right)\) \\
      \hline
      2.25 & \(\left( \begin{matrix} 0 & 4 \\ 4 & 4 \end{matrix} \right)\) \\
      \hline
      2.26 & \(\left( \begin{matrix} -2 & 1 \\ 1 & 10 \end{matrix} \right)\) \\
      \hline
      2.27 & \(\left( \begin{matrix} -2 & 3 \\ 3 & 4 \end{matrix} \right)\)  \\
      \hline
    \end{tabular}
  \end{minipage}%
  \begin{minipage}[t]{0.25\textwidth}
    \centering
    \begin{tabular}[t]{|c|c|}
      \hline
      \textnumero & \(\Pic(X)\) \\
      \hline
      \hline
      2.28 & \(\left( \begin{matrix} 0 & 3 \\ 3 & 4 \end{matrix} \right)\) \\
      \hline
      2.29 & \(\left( \begin{matrix} -2 & 2 \\ 2 & 6 \end{matrix} \right)\) \\
      \hline
      2.30 & \(\left( \begin{matrix} -2 & 2 \\ 2 & 4 \end{matrix} \right)\) \\
      \hline
      2.31 & \(\left( \begin{matrix} -2 & 1 \\ 1 & 6 \end{matrix} \right)\) \\
      \hline
      2.32 & \(\left( \begin{matrix} 2 & 4 \\ 4 & 2 \end{matrix} \right)\) \\
      \hline
      2.33 & \(\left( \begin{matrix} 0 & 3 \\ 3 & 4 \end{matrix} \right)\) \\
      \hline
      2.34 & \(\left( \begin{matrix} 0 & 3 \\ 3 & 2 \end{matrix} \right)\) \\
      \hline
      2.35 & \(\left( \begin{matrix} -2 & 0 \\ 0 & 4 \end{matrix} \right)\) \\
      \hline
      2.36 & \(\left( \begin{matrix} 2 & 5 \\ 5 & 10 \end{matrix} \right)\) \\
      \hline
    \end{tabular}
  \end{minipage}%
  \vspace{0.2 in}
  \begin{minipage}[t]{0.24\textwidth}
    \centering
    \noindent
    \begin{tabular}[t]{|c|c|}
      \hline
      \textnumero & \(\Pic(X)\) \\
      \hline
      \hline
      3.1 & \(\left(\begin{matrix} 0 & 2 & 2 \\ 2 & 0 & 2 \\ 2 & 2 & 0 \end{matrix} \right)\) \\
      \hline
      3.2 & \(\left(\begin{matrix} 0 & 2 & 1 \\ 2 & 0 & 2 \\ 1 & 2 & -2 \end{matrix} \right)\) \\
      \hline
      3.3 & \(\left(\begin{matrix} 0 & 2 & 3 \\ 2 & 0 & 3 \\ 3 & 3 & 2  \end{matrix} \right)\) \\
      \hline
      3.4 & \(\left(\begin{matrix} -2 & 2 & 0 \\ 2 & 0 & 4 \\ 0 & 4 & 2  \end{matrix} \right)\) \\
      \hline
      3.5 & \(\left(\begin{matrix} -2 & 2 & 5 \\ 2 & 0 & 3 \\ 5 & 3 & 2 \end{matrix} \right)\) \\
      \hline
      3.6 & \(\left(\begin{matrix} -2 & 0 & 1 \\ 0 & 0 & 4 \\ 1 & 4 & 4 \end{matrix} \right)\) \\
      \hline
      3.7 & \(\left(\begin{matrix} 0 & 3 & 3 \\ 3 & 2 & 4 \\ 3 & 4 & 2 \end{matrix} \right)\) \\
      \hline
      3.8 & \(\left(\begin{matrix} 2 & 0 & 5 \\ 0 & -2 & 2 \\ 5 & 2 & 2 \end{matrix} \right)\) \\
      \hline
    \end{tabular}
  \end{minipage}%
  \begin{minipage}[t]{0.24\textwidth}
    \centering
    \begin{tabular}[t]{|c|c|}
      \hline
      \textnumero & \(\Pic(X)\) \\
      \hline
      \hline
      3.9 & \(\left(\begin{matrix} 4 & 8 & 4 \\ 8 & 10 & 5 \\ 4 & 5 & 2 \end{matrix} \right)\) \\
      \hline
      3.10 & \(\left(\begin{matrix} -2 & 0 & 2 \\ 0 & -2 & 2 \\ 2 & 2 & 6 \end{matrix} \right)\) \\
      \hline
      3.11 & \(\left(\begin{matrix} 0 & 1 & 4 \\ 1 & -2 & 0 \\ 4 & 0 & 4 \end{matrix} \right)\) \\
      \hline
      3.12 & \(\left(\begin{matrix} -2 & 0 & 1 \\ 0 & -2 & 3 \\ 1 & 3 & 4 \end{matrix} \right)\) \\
      \hline
      3.13 & \(\left(\begin{matrix} -2 & 2 & 2 \\ 2 & 2 & 4 \\ 2 & 4 & 2 \end{matrix} \right)\) \\
      \hline
      3.14 & \(\left(\begin{matrix} -2 & 0 & 0 \\ 0 & 0 & 3 \\ 0 & 3 & 4 \end{matrix} \right)\) \\
      \hline
      3.15 & \(\left(\begin{matrix} -2 & 0 & 1 \\ 0 & -2 & 2 \\1 & 2 & 6 \end{matrix} \right)\) \\
      \hline
      3.16 & \(\left(\begin{matrix} -2 & 1 & 3 \\ 1 & -2 & 0 \\ 3 & 0 & 4 \end{matrix} \right)\) \\
      \hline
    \end{tabular}
  \end{minipage}%
  \begin{minipage}[t]{0.24\textwidth}
    \centering
    \begin{tabular}[t]{|c|c|}
      \hline
      \textnumero & \(\Pic(X)\) \\
      \hline
      \hline
      3.17 & \(\left(\begin{matrix} 0 & 2 & 3 \\ 2 & 0 & 3 \\ 3 & 3 & 2 \end{matrix} \right)\) \\
      \hline
      3.18 & \(\left(\begin{matrix} -2 & 0 & 1 \\0 & -2 & 2 \\ 1 & 2 & 4 \end{matrix} \right)\)  \\
      \hline
      3.19 & \(\left(\begin{matrix} -2 & 0 & 0 \\ 0 & -2 & 0 \\0 & 0 & 6 \end{matrix} \right)\)  \\
      \hline
      3.20 & \(\left(\begin{matrix} -2 & 0 & 1 \\ 0 & -2 & 1 \\ 1 & 1 & 6 \end{matrix} \right)\)  \\
      \hline
      3.21 & \(\left(\begin{matrix} -2 & 2 & 1 \\ 2 & 0 & 3 \\ 1 & 3 & 2 \end{matrix} \right)\) \\
      \hline
      3.22 & \(\left(\begin{matrix} -2 & 0 & 2 \\ 0 & 0 & 2 \\ 2 & 3 & 2 \end{matrix} \right)\) \\
      \hline
      3.23 & \(\left(\begin{matrix} -2 & 1 & 2 \\ 1 & -2 & 0 \\ 2 & 0 & 4 \end{matrix} \right)\) \\
      \hline
      3.24 & \(\left(\begin{matrix} -2 & 0 & 1 \\ 0 & 2 & 4 \\ 1 & 4 & 2 \end{matrix} \right)\) \\
      \hline
    \end{tabular}
  \end{minipage}%
  \begin{minipage}[t]{0.24\textwidth}
    \centering
    \begin{tabular}[t]{|c|c|}
      \hline
      \textnumero & \(\Pic(X)\) \\
      \hline
      \hline
      3.25 & \(\left(\begin{matrix} -2 & 0 & 1 \\ 0 & -2 & 1 \\ 1 & 1 & 4 \end{matrix} \right)\) \\
      \hline
      3.26 & \(\left(\begin{matrix}  -2 & 0 & 0 \\ 0 & -2 & 1 \\ 0 & 1 & 4 \end{matrix} \right)\) \\
      \hline
      3.27 & \(\left(\begin{matrix}  0 & 2 & 2 \\ 2 & 0 & 2 \\ 2 & 2 & 0 \end{matrix} \right)\) \\
      \hline
      3.28 & \(\left(\begin{matrix} -2 & 1 & 0 \\ 1 & 0 & 3 \\ 0 & 3 & 4 \end{matrix} \right)\) \\
      \hline
      3.29 & \(\left(\begin{matrix} -2 & -1 & 0 \\ -1 & -2 & 0 \\ 0 & 0 & 4 \end{matrix} \right)\) \\
      \hline
      3.30 & \(\left(\begin{matrix} -2 & 1 & 1 \\ 1 & -2 & 0 \\ 1 & 0 & 4 \end{matrix} \right)\) \\
      \hline
      3.31 & \(\left(\begin{matrix} 0 & 2 & 3 \\ 2 & 0 & 3 \\ 3 & 3 & 6 \end{matrix} \right)\) \\
      \hline
    \end{tabular}
  \end{minipage}
  \caption{Picard lattices of Fano threefolds of rank \(2,3\).}
  \label{table:picard-lattice-2-3}
\end{table}

\begin{table}[H]
  \begin{minipage}[t]{0.31\textwidth}
    \centering
    \begin{tabular}[t]{|c|c|}
      \hline
      \textnumero & \(\Pic(X)\) \\
      \hline
      \hline
      4.1 & \(\left(\begin{matrix} 0 & 2 & 2 & 2 \\ 2 & 0 & 2 & 2 \\ 2 & 2 & 0 & 2 \\ 2 & 2 & 2 & 0 \end{matrix} \right)\) \\
      \hline
      4.2 & \(\left(\begin{matrix} 0 & 2 & 2 & 4 \\ 2 & 0 & 2 & 3 \\ 2 & 2 & 0 & 3 \\ 4 & 3 & 3 & 6 \end{matrix} \right)\)  \\
      \hline
      4.3 & \(\left( \begin{matrix} -2 & 1 & 1 & 2 \\ 1 & 0 & 2 & 2 \\ 1 & 2 & 0 & 2 \\ 2 & 2 & 2 & 0 \end{matrix} \right)\) \\
      \hline
      4.4 & \(\left(\begin{matrix} -2 & 1 & 1 & 2 \\ 1 & -2 & 0 & 0 \\ 1 & 0 & -2 & 0 \\ 2 & 0 & 0 & 6 \end{matrix} \right)\) \\
      \hline
      4.5 & \(\left(\begin{matrix} -2 & 0 & 1 & 2 \\ 0 & -2 & 0 & 1 \\ 1 & 0 & 2 & 3 \\ 2 & 1 & 3 & 0 \end{matrix} \right)\)  \\
      \hline
    \end{tabular}
  \end{minipage}%
  \begin{minipage}[t]{0.31\textwidth}
    \centering
    \begin{tabular}[t]{|c|c|}
      \hline
      \textnumero & \(\Pic(X)\) \\
      \hline
      \hline
      4.6 & \(\left(\begin{matrix} -2 & 0 & 0 & 1 \\ 0 & -2 & 0 & 1 \\ 0 & 0 & -2 & 1 \\ 1 & 1 & 1 & 4 \end{matrix} \right)\) \\
      \hline
      4.7 & \(\left(\begin{matrix} -2 & 0 & 1 & 0 \\ 0 & -2 & 0 & 1 \\ 1 & 0 & 2 & 4 \\ 0 & 1 & 4 & 2 \end{matrix} \right)\) \\
      \hline
      4.8 & \(\left(\begin{matrix} -2 & 0 & 1 & 1 \\ 0 & 0 & 2 & 2 \\ 1 & 2 & 0 & 2 \\ 1 & 2 & 2 & 0 \end{matrix} \right)\) \\
      \hline
      4.9 & \(\left( \begin{matrix} -2  & 0 & -1 & 0 \\ 0 & -2  & 0 & 1 \\  -1 & 0 & -2 & 1 \\ 0  & 1 & 1 & 4 \end{matrix} \right)\) \\
      \hline
      4.10 & \(\left( \begin{matrix} -2 & 0 & -1 & 0 \\ 0 & -2  & -1 & 0 \\ -1 & -1 & 0 & 3 \\ 0  & 0 & 3 & 2 \end{matrix} \right)\) \\
      \hline
    \end{tabular}
  \end{minipage}%
  \begin{minipage}[t]{0.35\textwidth}
    \centering
    \begin{tabular}[t]{|c|c|}
      \hline
      \textnumero & \(\Pic(X)\) \\
      \hline
      \hline
      4.11 & \(\left( \begin{matrix} -2 & -1 & 0 & 0 \\ -1 & -2  & -1 & 0 \\ 0 & -1 & 0 & 3 \\ 0  & 0 & 3 & 2 \end{matrix} \right)\)  \\
      \hline
      4.12 & \(\left( \begin{matrix} -2 & 0 & -1 & 0 \\ 0 & -2  & -1 & 0 \\ -1 & - 1 & -2 & 1 \\ 0  & 0 & 1 & 4 \end{matrix} \right)\) \\
      \hline
      4.13 & \(\left( \begin{matrix} -2  & 1 & 1 & 3 \\ 1 & 0  & 2 & 2 \\  1 & -2 & 0 & 2 \\ 3 & 2 & 2 & 0 \end{matrix} \right)\) \\
      \hline
      5.1 & \(\left( \begin{matrix} -2 & 0 & 0 & -1 & 0 \\ 0 & -2 & 0 & -1 & 0 \\ 0 & 0 & -2 & -1 & 0 \\ -1 & -1 & -1 & -2 & 2 \\ 0 & 0 & 0 & 2 & 6 \end{matrix} \right)\) \\
      \hline
      5.2 & \(\left( \begin{matrix} -2 & 0 & 0 & -1 & 0 \\ 0 & -2 & 0 & -1 & 0 \\ 0& 0 & -2 & 0 & 1 \\ -1 & -1 & 0 & -2 & 1 \\ 0 & 0 & 1 & 1 & 4 \end{matrix} \right)\) \\
      \hline
    \end{tabular}
  \end{minipage}
  \caption{Picard lattices of Fano threefolds of rank \(4,5\).}
  \label{table:picard-lattice-4-5}
\end{table}

\section{Dolgachev--Nikulin duality for Fano threefolds: rank 2}\label{appendix:rank-02}
\subsection{Family \textnumero2.1}\label{subsection:02-01}

The pencil \(\mathcal{S}\) is defined by the equation
\[
  X (X + Y) C^3 = Y (Y + \lambda (X + Y)) (A^3 - B C (A - B)). 
\]
Members \(\mathcal{S}_{\lambda}\) of the pencil are irreducible for any \(\lambda \in \mathbb{P}^1\) except
\[
  \mathcal{S}_{\infty} = S_{(Y)} + S_{(X + Y)} + S_{(A^3 - B C(A - B))}, \;
  \mathcal{S}_{-1} = S_{(X)} + S_{(Y (A^3 - B C (A - B)) + C^3 (X + Y))}.
\]
The base locus of the pencil \(\mathcal{S}\) consists of the following curves:
\[
  C_1 = C_{(Y, C)}, \;
  C_2 = C_{(A, C)}, \;
  C_3 = C_{(X, A^3 - B C (A - B))}, \;
  C_4 = C_{(X + Y, A^3 - B C (A - B))}.
\]
Their linear equivalence classes on the generic member \(\mathcal{S}_{\Bbbk}\) of the pencil satisfy the following relations:
\[
  [C_{3}] = [C_{4}] = [H_{\mathcal{S}}^{(1)}] = 3 [C_{1}].
\]

For a general choice of \(\lambda \in \mathbb{C}\) the surface \(\mathcal{S}_{\lambda}\) has the following singularities:
\begin{itemize}\setlength{\itemindent}{2cm}
\item[\(P_1 = P_{(Y, A, C)}\):] type \(\mathbb{A}_8\) with the quadratic term \(\lambda Y \cdot C\);
\item[\(P_2 = P_{(A, C, Y + \lambda (X + Y))}\):] type \(\mathbb{A}_8\) with the quadratic term \(\lambda (\lambda + 1) Y \cdot C\).
\end{itemize}

Galois action on the lattice \(L_{\lambda}\) is trivial. The intersection matrix on \(L_{\lambda} = L_{\mathcal{S}}\) is represented by
\begin{table}[H]
  \begin{tabular}{|c||cccccccc|cccccccc|ccc|}
    \hline
    \(\bullet\) & \(E_1^1\) & \(E_1^2\) & \(E_1^3\) & \(E_1^4\) & \(E_1^5\) & \(E_1^6\) & \(E_1^7\) & \(E_1^8\) & \(E_2^1\) & \(E_2^2\) & \(E_2^3\) & \(E_2^4\) & \(E_2^5\) & \(E_2^6\) & \(E_2^7\) & \(E_2^8\) & \(\widetilde{C_1}\) & \(\widetilde{C_2}\) & \(\widetilde{H_{\mathcal{S}}^{(2)}}\) \\
    \hline
    \hline
    \(\widetilde{C_1}\) & \(0\) & \(0\) & \(1\) & \(0\) & \(0\) & \(0\) & \(0\) & \(0\) & \(0\) & \(0\) & \(0\) & \(0\) & \(0\) & \(0\) & \(0\) & \(0\) & \(-2\) & \(0\) & \(1\) \\
    \(\widetilde{C_2}\) & \(0\) & \(0\) & \(0\) & \(0\) & \(0\) & \(0\) & \(0\) & \(1\) & \(1\) & \(0\) & \(0\) & \(0\) & \(0\) & \(0\) & \(0\) & \(0\) & \(0\) & \(-2\) & \(0\) \\
    \(\widetilde{H_{\mathcal{S}}^{(2)}}\) & \(0\) & \(0\) & \(0\) & \(0\) & \(0\) & \(0\) & \(0\) & \(0\) & \(0\) & \(0\) & \(0\) & \(0\) & \(0\) & \(0\) & \(0\) & \(0\) & \(1\) & \(0\) & \(2\) \\
    \hline
  \end{tabular}.
\end{table}

Note that the intersection matrix is degenerate. We choose the following integral basis of the lattice \(L_{\lambda}\):
\begin{align*}
  \begin{pmatrix}
    [E_2^8]
  \end{pmatrix} =
  \begin{pmatrix}
    -5 & -10 & -15 & -14 & -13 & -12 & -11 & -10 & -8 \\ -7 & -6 & -5 & -4 & -3 & -2 & -6 & -9 & 3
  \end{pmatrix} \cdot \\
  \begin{pmatrix}
    [E_1^1] & [E_1^2] & [E_1^3] & [E_1^4] & [E_1^5] & [E_1^6] & [E_1^7] & [E_1^8] & [E_2^1] & \\
    [E_2^2] & [E_2^3] & [E_2^4] & [E_2^5] & [E_2^6] & [E_2^7] & [\widetilde{C_1}] & [\widetilde{C_2}] & [\widetilde{H_{\mathcal{S}}^{(2)}}]
  \end{pmatrix}^T.
\end{align*}

Discriminant groups of the lattices \(L_{\mathcal{S}}\) and \(H \oplus \Pic(X)\) are both trivial.


\subsection{Family \textnumero2.2}\label{subsection:02-02}

The pencil \(\mathcal{S}\) is defined by the equation
\begin{gather*}
  (X Y + T^2  + Z (Y - T)) (Z^2 - \lambda X Y) = X Z^3.
\end{gather*}
Members \(\mathcal{S}_{\lambda}\) of the pencil are irreducible for any \(\lambda \in \mathbb{P}^1\) except
\begin{gather*}
  \mathcal{S}_{\infty} = S_{(X)} + S_{(Y)} + S_{(Y (X + Z) - T (Z - T))}, \;
  \mathcal{S}_{0} = 2 S_{(Z)} + S_{(Y (X + Z) - T (Z - T) - X Z)}.
\end{gather*}
The base locus of the pencil \(\mathcal{S}\) consists of the following curves:
\begin{gather*}
  C_{1} = C_{(X, Z)}, \;
  C_{2} = C_{(Y, Z)}, \;
  C_{3} = C_{(X, Z (Y - T) + T^2)}, \;
  C_{4} = C_{(Y, Z (X + T) - T^2)}, \;
  C_{5} = C_{(Z, X Y + T^2)}.
\end{gather*}
Their linear equivalence classes on the generic member \(\mathcal{S}_{\Bbbk}\) of the pencil satisfy the following relations:
\begin{gather*}
  \begin{pmatrix}
    [C_{3}] \\ [C_{4}] \\ [C_{5}]
  \end{pmatrix} = 
  \begin{pmatrix}
    -2 & 0 & 1 \\
    0 & -2 & 1 \\
    -1 & -1 & 1
  \end{pmatrix} \cdot
  \begin{pmatrix}
    [C_{1}] \\ [C_{2}] \\ [H_{\mathcal{S}}]
  \end{pmatrix}.
\end{gather*}

For a general choice of \(\lambda \in \mathbb{C}\) the surface \(\mathcal{S}_{\lambda}\) has the following singularities:
\begin{itemize}\setlength{\itemindent}{2cm}
\item[\(P_{1} = P_{(X, Y, Z)}\):] type \(\mathbb{A}_1\) with the quadratic term \(\lambda X Y - Z^2\);
\item[\(P_{2} = P_{(X, Z, T)}\):] type \(\mathbb{A}_9\) with the quadratic term \(\lambda X \cdot (X + Z)\);
\item[\(P_{3} = P_{(Y, Z, T)}\):] type \(\mathbb{E}_6\) with the quadratic term \(\lambda Y^2\).
\end{itemize}

Galois action on the lattice \(L_{\lambda}\) is trivial. The intersection matrix on \(L_{\lambda} = L_{\mathcal{S}}\) is represented by
\begin{table}[H]
  \begin{tabular}{|c||c|ccccccccc|cccccc|ccc|}
    \hline
    \(\bullet\) & \(E_1^1\) & \(E_2^1\) & \(E_2^2\) & \(E_2^3\) & \(E_2^4\) & \(E_2^5\) & \(E_2^6\) & \(E_2^7\) & \(E_2^8\) & \(E_2^9\) & \(E_3^1\) & \(E_3^2\) & \(E_3^3\) & \(E_3^4\) & \(E_3^5\) & \(E_3^6\) & \(\widetilde{C_{1}}\) & \(\widetilde{C_{2}}\) & \(\widetilde{H_{\mathcal{S}}}\) \\
    \hline
    \hline
    \(\widetilde{C_{1}}\) & \(1\) & \(0\) & \(1\) & \(0\) & \(0\) & \(0\) & \(0\) & \(0\) & \(0\) & \(0\) & \(0\) & \(0\) & \(0\) & \(0\) & \(0\) & \(0\) & \(-2\) & \(0\) & \(1\) \\
    \(\widetilde{C_{2}}\) & \(1\) & \(0\) & \(0\) & \(0\) & \(0\) & \(0\) & \(0\) & \(0\) & \(0\) & \(0\) & \(1\) & \(0\) & \(0\) & \(0\) & \(0\) & \(0\) & \(0\) & \(-2\) & \(1\) \\
    \(\widetilde{H_{\mathcal{S}}}\) & \(0\) & \(0\) & \(0\) & \(0\) & \(0\) & \(0\) & \(0\) & \(0\) & \(0\) & \(0\) & \(0\) & \(0\) & \(0\) & \(0\) & \(0\) & \(0\) & \(1\) & \(1\) & \(4\) \\
    \hline
  \end{tabular}.
\end{table}
Note that the intersection matrix is degenerate. We choose the following integral basis of the lattice \(L_{\lambda}\):
\begin{align*}
  \begin{pmatrix}
    [E_2^9]
  \end{pmatrix} =
  \begin{pmatrix}
-4 & -4 & -8 & -7 & -6 & -5 & -4 & -3 & -2 & -4 & -3 & -5 & -6 & -4 & -2 & -5 & -3 & 2
  \end{pmatrix} \cdot \\
  \begin{pmatrix}
    [E_1^1] & [E_2^1] & [E_2^2] & [E_2^3] & [E_2^4] & [E_2^5] & [E_2^6] & [E_2^7] & [E_2^8] & \\
    [E_3^1] & [E_3^2] & [E_3^3] & [E_3^4] & [E_3^5] & [E_3^6] & [\widetilde{C_{1}}] & [\widetilde{C_{2}}] & [\widetilde{H_{\mathcal{S}}}]
  \end{pmatrix}^T.
\end{align*}

Discriminant groups and discriminant forms of the lattices \(L_{\mathcal{S}}\) and \(H \oplus \Pic(X)\) are given by
\begin{gather*}
  G' = 
  \begin{pmatrix}
    \frac{1}{2} & \frac{1}{2} & 0 & 0 & 0 & 0 & 0 & 0 & 0 & 0 & \frac{1}{2} & \frac{1}{2} & 0 & 0 & 0 & \frac{1}{2} & \frac{1}{2} & \frac{1}{2} \\
    0 & \frac{1}{2} & 0 & 0 & 0 & 0 & 0 & 0 & 0 & 0 & \frac{1}{2} & \frac{1}{2} & 0 & 0 & 0 & \frac{1}{2} & \frac{1}{2} & 0
  \end{pmatrix}, \\
  G'' = 
  \begin{pmatrix}
    0 & 0 & \frac{1}{2} & 0 \\
    0 & 0 & 0 & \frac{1}{2}
  \end{pmatrix}; \;
  B' = 
  \begin{pmatrix}
    0 & \frac{1}{2} \\
    \frac{1}{2} & \frac{1}{2}
  \end{pmatrix}, \;
  B'' = 
  \begin{pmatrix}
    0 & \frac{1}{2} \\
    \frac{1}{2} & \frac{1}{2}
  \end{pmatrix}; \;
  \begin{pmatrix}
    Q' \\ Q''
  \end{pmatrix}
  =
  \begin{pmatrix}
    0 & \frac{3}{2} \\
    0 & \frac{1}{2}    
  \end{pmatrix}.
\end{gather*}


\subsection{Family \textnumero2.3}\label{subsection:02-03}

The pencil \(\mathcal{S}\) is defined by the equation
\begin{gather*}
  X^3 Y + (Y + Z) (X (Z + T) - T^2) (Y + \lambda Z) = 0.
\end{gather*}
Members \(\mathcal{S}_{\lambda}\) of the pencil are irreducible for any \(\lambda \in \mathbb{P}^1\) except
\begin{gather*}
  \mathcal{S}_{\infty} = S_{(Z)} + S_{(Y + Z)} + S_{(X (Z + T) - T^2)}, \;
  \mathcal{S}_{0} = S_{(Y)} + S_{(X^3 - (Y + Z) (X (Z + T) - T^2))}.
\end{gather*}
The base locus of the pencil \(\mathcal{S}\) consists of the following curves:
\begin{gather*}
  C_{1} = C_{(X, T)}, \;
  C_{2} = C_{(Y, Z)}, \;
  C_{3} = C_{(X, Y + Z)}, \;
  C_{4} = C_{(Y, X (Z + T) - T^2)}, \;
  C_{5} = C_{(Z, X^3 + Y T (X - T))}.
\end{gather*}
Their linear equivalence classes on the generic member \(\mathcal{S}_{\Bbbk}\) of the pencil satisfy the following relations:
\begin{gather*}
  \begin{pmatrix}
    [C_{2}] \\ [C_{4}] \\ [C_{5}]
  \end{pmatrix} = 
  \begin{pmatrix}
    0 & -3 & 1 \\
    0 & 6 & -1 \\
    0 & 3 & 0
  \end{pmatrix} \cdot
  \begin{pmatrix}
    [C_{1}] \\ [C_{3}] \\ [H_{\mathcal{S}}]
  \end{pmatrix}.
\end{gather*}

For a general choice of \(\lambda \in \mathbb{C}\) the surface \(\mathcal{S}_{\lambda}\) has the following singularities:
\begin{itemize}\setlength{\itemindent}{2cm}
\item[\(P_{1} = P_{(X, Y, Z)}\):] type \(\mathbb{A}_5\) with the quadratic term \((Y + Z) \cdot (Y + \lambda Z)\);
\item[\(P_{2} = P_{(X, Z, T)}\):] type \(\mathbb{A}_1\) with the quadratic term \(X (Z + T) - T^2\);
\item[\(P_{3} = P_{(X, T, Y + Z)}\):] type \(\mathbb{A}_5\) with the quadratic term \((\lambda - 1) X \cdot (Y + Z)\);
\item[\(P_{4} = P_{(X, T, Y + \lambda Z)}\):] type \(\mathbb{A}_5\) with the quadratic term \((\lambda - 1) X \cdot (Y + \lambda Z)\).
\end{itemize}

Galois action on the lattice \(L_{\lambda}\) is trivial. The intersection matrix on \(L_{\lambda} = L_{\mathcal{S}}\) is represented by
\begin{table}[H]
  \begin{tabular}{|c||ccccc|c|ccccc|ccccc|ccc|}
    \hline
    \(\bullet\) & \(E_1^1\) & \(E_1^2\) & \(E_1^3\) & \(E_1^4\) & \(E_1^5\) & \(E_2^1\) & \(E_3^1\) & \(E_3^2\) & \(E_3^3\) & \(E_3^4\) & \(E_3^5\) & \(E_4^1\) & \(E_4^2\) & \(E_4^3\) & \(E_4^4\) & \(E_4^5\) & \(\widetilde{C_{1}}\) & \(\widetilde{C_{3}}\) & \(\widetilde{H_{\mathcal{S}}}\) \\
    \hline
    \hline
    \(\widetilde{C_{1}}\) & \(0\) & \(0\) & \(0\) & \(0\) & \(0\) & \(1\) & \(1\) & \(0\) & \(0\) & \(0\) & \(0\) & \(1\) & \(0\) & \(0\) & \(0\) & \(0\) & \(-2\) & \(0\) & \(1\) \\
    \(\widetilde{C_{3}}\) & \(1\) & \(0\) & \(0\) & \(0\) & \(0\) & \(0\) & \(0\) & \(0\) & \(0\) & \(1\) & \(0\) & \(0\) & \(0\) & \(0\) & \(0\) & \(0\) & \(0\) & \(-2\) & \(1\) \\
    \(\widetilde{H_{\mathcal{S}}}\) & \(0\) & \(0\) & \(0\) & \(0\) & \(0\) & \(0\) & \(0\) & \(0\) & \(0\) & \(0\) & \(0\) & \(0\) & \(0\) & \(0\) & \(0\) & \(0\) & \(1\) & \(1\) & \(4\) \\
    \hline
  \end{tabular}.
\end{table}
Note that the intersection matrix is degenerate. We choose the following integral basis of the lattice \(L_{\lambda}\):
\begin{align*}
  \begin{pmatrix}
    [E_4^5]
  \end{pmatrix} =
  \begin{pmatrix}
    -5 & -4 & -3 & -2 & -1 & -3 & -7 & -8 & -9 & -10 & -5 & -5 & -4 & -3 & -2 & -6 & -6 & 3
  \end{pmatrix} \cdot \\
  \begin{pmatrix}
    [E_1^1] & [E_1^2] & [E_1^3] & [E_1^4] & [E_1^5] & [E_2^1] & [E_3^1] & [E_3^2] & [E_3^3] & \\
    [E_3^4] & [E_3^5] & [E_4^1] & [E_4^2] & [E_4^3] & [E_4^4] & [\widetilde{C_{1}}] & [\widetilde{C_{3}}] & [\widetilde{H_{\mathcal{S}}}]
  \end{pmatrix}^T.
\end{align*}

Discriminant groups and discriminant forms of the lattices \(L_{\mathcal{S}}\) and \(H \oplus \Pic(X)\) are given by
\begin{gather*}
  G' = 
  \begin{pmatrix}
    0 & 0 & 0 & 0 & 0 & \frac{1}{2} & \frac{1}{2} & 0 & \frac{1}{2} & 0 & \frac{1}{2} & 0 & 0 & 0 & 0 & 0 & 0 & 0 \\
    \frac{1}{2} & 0 & \frac{1}{2} & 0 & \frac{1}{2} & \frac{1}{2} & 0 & 0 & 0 & 0 & 0 & 0 & 0 & 0 & 0 & 0 & 0 & \frac{1}{2}
  \end{pmatrix}, \\
  G'' = 
  \begin{pmatrix}
    0 & 0 & \frac{1}{2} & 0 \\
    0 & 0 & 0 & \frac{1}{2}
  \end{pmatrix}; \;
  B' = 
  \begin{pmatrix}
    0 & \frac{1}{2} \\
    \frac{1}{2} & 0
  \end{pmatrix}, \;
  B'' = 
  \begin{pmatrix}
    0 & \frac{1}{2} \\
    \frac{1}{2} & 0
  \end{pmatrix}; \;
  \begin{pmatrix}
    Q' \\ Q''
  \end{pmatrix}
  =
  \begin{pmatrix}
    0 & 1 \\
    0 & 1 
  \end{pmatrix}.
\end{gather*}


\subsection{Family \textnumero2.4}\label{subsection:02-04}

The pencil \(\mathcal{S}\) is defined by the equation
\begin{gather*}
  X^{3} Y + X^{2} Y^{2} + X^{3} Z + 4 X^{2} Y Z + 2 X Y^{2} Z + 2 X^{2} Z^{2} + 4 X Y Z^{2} + Y^{2} Z^{2} + X Z^{3} + Y Z^{3} + 3 X^{2} Y T + \\ 2 X Y^{2} T + 2 X^{2} Z T + 2 Y^{2} Z T + 2 X Z^{2} T + 3 Y Z^{2} T + 3 X Y T^{2} + Y^{2} T^{2} + X Z T^{2} + 3 Y Z T^{2} + Y T^{3} = \lambda X Y Z T.
\end{gather*}
Members \(\mathcal{S}_{\lambda}\) of the pencil are irreducible for any \(\lambda \in \mathbb{P}^1\) except
\begin{gather*}
  \mathcal{S}_{\infty} = S_{(X)} + S_{(Y)} + S_{(Z)} + S_{(T)}, \;
  \mathcal{S}_{- 7} = S_{(X + Z + T)} + S_{(X + Y + Z + T)} + S_{(Y (X + Z + T) + X Z)}.
\end{gather*}
The base locus of the pencil \(\mathcal{S}\) consists of the following curves:
\begin{gather*}
  C_{1} = C_{(X, Y)}, \;
  C_{2} = C_{(Y, Z)}, \;
  C_{3} = C_{(X, Z + T)}, \;
  C_{4} = C_{(Z, X + T)}, \;
  C_{5} = C_{(T, X + Z)}, \\
  C_{6} = C_{(X, Y + Z + T)}, \;
  C_{7} = C_{(Y, X + Z + T)}, \;
  C_{8} = C_{(Z, X + Y + T)}, \;
  C_{9} = C_{(T, X + Y + Z)}, \;
  C_{10} = C_{(T, X Z + Y (X + Z))}.
\end{gather*}
Their linear equivalence classes on the generic member \(\mathcal{S}_{\Bbbk}\) of the pencil satisfy the following relations:
\begin{gather*}
  \begin{pmatrix}
    [C_{2}] \\ [C_{6}] \\ [C_{7}] \\ [C_{8}] \\ [C_{9}] \\ [C_{10}]
  \end{pmatrix} = 
  \begin{pmatrix}
    -1 & 2 & 2 & 2 & -1 \\
    -1 & -2 & 0 & 0 & 1 \\
    0 & -1 & -1 & -1 & 1 \\
    1 & -2 & -4 & -2 & 2 \\
    0 & 5 & 5 & 3 & -3 \\
    0 & -5 & -5 & -4 & 4
  \end{pmatrix} \cdot
  \begin{pmatrix}
    [C_{1}] \\ [C_{3}] \\ [C_{4}] \\ [C_{5}] \\ [H_{\mathcal{S}}]
  \end{pmatrix}.
\end{gather*}

For a general choice of \(\lambda \in \mathbb{C}\) the surface \(\mathcal{S}_{\lambda}\) has the following singularities:
\begin{itemize}\setlength{\itemindent}{2cm}
\item[\(P_{1} = P_{(X, Z, T)}\):] type \(\mathbb{D}_4\) with the quadratic term \((X + Z + T)^2\);
\item[\(P_{2} = P_{(X, Y, Z + T)}\):] type \(\mathbb{A}_4\) with the quadratic term \((\lambda + 7) X \cdot Y\);
\item[\(P_{3} = P_{(Y, Z, X + T)}\):] type \(\mathbb{A}_4\) with the quadratic term \((\lambda + 7) Y \cdot Z\);
\item[\(P_{4} = P_{(Y, T, X + Z)}\):] type \(\mathbb{A}_1\) with the quadratic term \((X + Z + T) (X + Y + Z + T) - (\lambda + 7) Y T\).
\end{itemize}

Galois action on the lattice \(L_{\lambda}\) is trivial. The intersection matrix on \(L_{\lambda} = L_{\mathcal{S}}\) is represented by
\begin{table}[H]
  \begin{tabular}{|c||cccc|cccc|cccc|c|ccccc|}
    \hline
    \(\bullet\) & \(E_1^1\) & \(E_1^2\) & \(E_1^3\) & \(E_1^4\) & \(E_2^1\) & \(E_2^2\) & \(E_2^3\) & \(E_2^4\) & \(E_3^1\) & \(E_3^2\) & \(E_3^3\) & \(E_3^4\) & \(E_4^1\) & \(\widetilde{C_{1}}\) & \(\widetilde{C_{3}}\) & \(\widetilde{C_{4}}\) & \(\widetilde{C_{5}}\) & \(\widetilde{H_{\mathcal{S}}}\) \\
    \hline
    \hline
    \(\widetilde{C_{1}}\) & \(0\) & \(0\) & \(0\) & \(0\) & \(0\) & \(1\) & \(0\) & \(0\) & \(0\) & \(0\) & \(0\) & \(0\) & \(0\) & \(-2\) & \(0\) & \(0\) & \(0\) & \(1\) \\
    \(\widetilde{C_{3}}\) & \(1\) & \(0\) & \(0\) & \(0\) & \(1\) & \(0\) & \(0\) & \(0\) & \(0\) & \(0\) & \(0\) & \(0\) & \(0\) & \(0\) & \(-2\) & \(0\) & \(0\) & \(1\) \\
    \(\widetilde{C_{4}}\) & \(0\) & \(0\) & \(1\) & \(0\) & \(0\) & \(0\) & \(0\) & \(0\) & \(0\) & \(0\) & \(0\) & \(1\) & \(0\) & \(0\) & \(0\) & \(-2\) & \(0\) & \(1\) \\
    \(\widetilde{C_{5}}\) & \(0\) & \(0\) & \(0\) & \(1\) & \(0\) & \(0\) & \(0\) & \(0\) & \(0\) & \(0\) & \(0\) & \(0\) & \(1\) & \(0\) & \(0\) & \(0\) & \(-2\) & \(1\) \\
    \(\widetilde{H_{\mathcal{S}}}\) & \(0\) & \(0\) & \(0\) & \(0\) & \(0\) & \(0\) & \(0\) & \(0\) & \(0\) & \(0\) & \(0\) & \(0\) & \(0\) & \(1\) & \(1\) & \(1\) & \(1\) & \(4\) \\
    \hline
  \end{tabular}.
\end{table}
Note that the intersection matrix is non-degenerate.

Discriminant groups and discriminant forms of the lattices \(L_{\mathcal{S}}\) and \(H \oplus \Pic(X)\) are given by
\begin{gather*}
  G' = 
  \begin{pmatrix}
    \frac{5}{9} & \frac{1}{3} & \frac{8}{9} & \frac{2}{9} & \frac{5}{9} & \frac{1}{3} & \frac{2}{9} & \frac{1}{9} & \frac{8}{9} & \frac{7}{9} & \frac{2}{3} & \frac{5}{9} & \frac{5}{9} & \frac{8}{9} & \frac{7}{9} & \frac{4}{9} & \frac{1}{9} & \frac{4}{9}
  \end{pmatrix}, \\
  G'' = 
  \begin{pmatrix}
    0 & 0 & \frac{8}{9} & \frac{1}{3}
  \end{pmatrix}; \;
  B' = 
  \begin{pmatrix}
    \frac{7}{9}
  \end{pmatrix}, \;
  B'' =
  \begin{pmatrix}
    \frac{2}{9}
  \end{pmatrix}; \;
  Q' =
  \begin{pmatrix}
    \frac{16}{9}
  \end{pmatrix}, \;
  Q'' =
  \begin{pmatrix}
    \frac{2}{9}
  \end{pmatrix}.
\end{gather*}


\subsection{Family \textnumero2.5}\label{subsection:02-05}

The pencil \(\mathcal{S}\) is defined by the equation
\begin{gather*}
  X^{2} Y Z + X Y^{2} Z + X Y Z^{2} + X^{3} T + 3 X^{2} Y T + 3 X Y^{2} T + Y^{3} T + 3 X^{2} Z T + 3 Y^{2} Z T + 3 X Z^{2} T + \\ 3 Y Z^{2} T + Z^{3} T + X^{2} T^{2} + 2 X Y T^{2} + Y^{2} T^{2} + 2 X Z T^{2} + 2 Y Z T^{2} + Z^{2} T^{2} = \lambda X Y Z T.
\end{gather*}
Members \(\mathcal{S}_{\lambda}\) of the pencil are irreducible for any \(\lambda \in \mathbb{P}^1\) except
\begin{gather*}
  \mathcal{S}_{\infty} = S_{(X)} + S_{(Y)} + S_{(Z)} + S_{(T)}, \\
  \mathcal{S}_{- 7} = S_{(X + Y + Z + T)} + S_{(T (X + Y + Z)^2 + X Y Z)}, \;
  \mathcal{S}_{- 6} = S_{(X + Y + Z)} + S_{(X Y Z + T (X + Y + Z) (X + Y + Z + T))}.
\end{gather*}
The base locus of the pencil \(\mathcal{S}\) consists of the following curves:
\begin{gather*}
  C_{1} = C_{(X, T)}, \;
  C_{2} = C_{(Y, T)}, \;
  C_{3} = C_{(Z, T)}, \;
  C_{4} = C_{(X, Y + Z)}, \;
  C_{5} = C_{(Y, X + Z)}, \\
  C_{6} = C_{(Z, X + Y)}, \;
  C_{7} = C_{(X, Y + Z + T)}, \;
  C_{8} = C_{(Y, X + Z + T)}, \;
  C_{9} = C_{(Z, X + Y + T)}, \;
  C_{10} = C_{(T, X + Y + Z)}.
\end{gather*}
Their linear equivalence classes on the generic member \(\mathcal{S}_{\Bbbk}\) of the pencil satisfy the following relations:
\begin{gather*}
  \begin{pmatrix}
    [C_{3}] \\ [C_{6}] \\ [C_{7}] \\ [C_{8}] \\ [C_{9}] \\ [H_{\mathcal{S}}]
  \end{pmatrix} = 
  \begin{pmatrix}
    -1 & -1 & 0 & 0 & 3 \\
    0 & 0 & -1 & -1 & 3 \\
    -1 & 0 & -2 & 0 & 4 \\
    0 & -1 & 0 & -2 & 4 \\
    1 & 1 & 2 & 2 & -5 \\
    0 & 0 & 0 & 0 & 4
  \end{pmatrix} \cdot
  \begin{pmatrix}
    [C_{1}] \\ [C_{2}] \\ [C_{4}] \\ [C_{5}] \\ [C_{10}]
  \end{pmatrix}.
\end{gather*}

For a general choice of \(\lambda \in \mathbb{C}\) the surface \(\mathcal{S}_{\lambda}\) has the following singularities:
\begin{itemize}\setlength{\itemindent}{2cm}
\item[\(P_{1} = P_{(X, Y, Z)}\):] type \(\mathbb{D}_4\) with the quadratic term \((X + Y + Z)^2\);
\item[\(P_{2} = P_{(X, T, Y + Z)}\):] type \(\mathbb{A}_3\) with the quadratic term \(X \cdot (X + Y + Z - (\lambda + 6) T)\);
\item[\(P_{3} = P_{(Y, T, X + Z)}\):] type \(\mathbb{A}_3\) with the quadratic term \(Y \cdot (X + Y + Z - (\lambda + 6) T)\);
\item[\(P_{4} = P_{(Z, T, X + Y)}\):] type \(\mathbb{A}_3\) with the quadratic term \(Z \cdot (X + Y + Z - (\lambda + 6) T)\).
\end{itemize}

Galois action on the lattice \(L_{\lambda}\) is trivial. The intersection matrix on \(L_{\lambda} = L_{\mathcal{S}}\) is represented by
\begin{table}[H]
  \begin{tabular}{|c||cccc|ccc|ccc|ccc|ccccc|}
    \hline
    \(\bullet\) & \(E_1^1\) & \(E_1^2\) & \(E_1^3\) & \(E_1^4\) & \(E_2^1\) & \(E_2^2\) & \(E_2^3\) & \(E_3^1\) & \(E_3^2\) & \(E_3^3\) & \(E_4^1\) & \(E_4^2\) & \(E_4^3\) & \(\widetilde{C_{1}}\) & \(\widetilde{C_{2}}\) & \(\widetilde{C_{4}}\) & \(\widetilde{C_{5}}\) & \(\widetilde{C_{10}}\) \\
    \hline
    \hline
    \(\widetilde{C_{1}}\) & \(0\) & \(0\) & \(0\) & \(0\) & \(1\) & \(0\) & \(0\) & \(0\) & \(0\) & \(0\) & \(0\) & \(0\) & \(0\) & \(-2\) & \(1\) & \(0\) & \(0\) & \(0\) \\
    \(\widetilde{C_{2}}\) & \(0\) & \(0\) & \(0\) & \(0\) & \(0\) & \(0\) & \(0\) & \(1\) & \(0\) & \(0\) & \(0\) & \(0\) & \(0\) & \(1\) & \(-2\) & \(0\) & \(0\) & \(0\) \\
    \(\widetilde{C_{4}}\) & \(0\) & \(0\) & \(1\) & \(0\) & \(1\) & \(0\) & \(0\) & \(0\) & \(0\) & \(0\) & \(0\) & \(0\) & \(0\) & \(0\) & \(0\) & \(-2\) & \(0\) & \(0\) \\
    \(\widetilde{C_{5}}\) & \(0\) & \(0\) & \(0\) & \(1\) & \(0\) & \(0\) & \(0\) & \(1\) & \(0\) & \(0\) & \(0\) & \(0\) & \(0\) & \(0\) & \(0\) & \(0\) & \(-2\) & \(0\) \\
    \(\widetilde{C_{10}}\) & \(0\) & \(0\) & \(0\) & \(0\) & \(0\) & \(0\) & \(1\) & \(0\) & \(0\) & \(1\) & \(0\) & \(0\) & \(1\) & \(0\) & \(0\) & \(0\) & \(0\) & \(-2\) \\
    \hline
  \end{tabular}.
\end{table}
Note that the intersection matrix is non-degenerate.

Discriminant groups and discriminant forms of the lattices \(L_{\mathcal{S}}\) and \(H \oplus \Pic(X)\) are given by
\begin{gather*}
  G' = 
  \begin{pmatrix}
    0 & 0 & 0 & 0 & 0 & \frac{2}{3} & \frac{1}{3} & 0 & \frac{1}{3} & \frac{2}{3} & 0 & 0 & 0 & \frac{1}{3} & \frac{2}{3} & 0 & 0 & 0 \\
    0 & 0 & \frac{2}{3} & \frac{1}{3} & 0 & \frac{2}{3} & \frac{1}{3} & 0 & \frac{1}{3} & \frac{2}{3} & 0 & 0 & 0 & 0 & 0 & \frac{1}{3} & \frac{2}{3} & 0
  \end{pmatrix}, \\
  G'' = 
  \begin{pmatrix}
    0 & 0 & \frac{1}{3} & 0 \\
    0 & 0 & 0 & \frac{1}{3}
  \end{pmatrix}; \;
  B' = 
  \begin{pmatrix}
    0 & \frac{2}{3} \\
    \frac{2}{3} & \frac{1}{3}
  \end{pmatrix}, \;
  B'' = 
  \begin{pmatrix}
    0 & \frac{1}{3} \\
    \frac{1}{3} & \frac{2}{3}
  \end{pmatrix}; \;
  \begin{pmatrix}
    Q' \\ Q''
  \end{pmatrix}
  =
  \begin{pmatrix}
    0 & \frac{4}{3} \\
    0 & \frac{2}{3}    
  \end{pmatrix}.
\end{gather*}


\subsection{Family \textnumero2.6}\label{subsection:02-06}

The pencil \(\mathcal{S}\) is defined by the equation
\begin{gather*}
  X^{2} Y Z + X Y^{2} Z + X^{2} Z^{2} + 2 X Y Z^{2} + Y^{2} Z^{2} + X Z^{3} + Y Z^{3} + X^{2} Y T + X Y^{2} T + 2 X^{2} Z T + 2 Y^{2} Z T + \\ 3 X Z^{2} T + 3 Y Z^{2} T + X^{2} T^{2} + 2 X Y T^{2} + Y^{2} T^{2} + 3 X Z T^{2} + 3 Y Z T^{2} + X T^{3} + Y T^{3} = \lambda X Y Z T.
\end{gather*}
Members \(\mathcal{S}_{\lambda}\) of the pencil are irreducible for any \(\lambda \in \mathbb{P}^1\) except
\begin{gather*}
  \mathcal{S}_{\infty} = S_{(X)} + S_{(Y)} + S_{(Z)} + S_{(T)}, \;
  \mathcal{S}_{- 4} = S_{(X + Y)} + S_{(Z + T)} + S_{(X + Z + T)} + S_{(Y + Z + T)}.
\end{gather*}
The base locus of the pencil \(\mathcal{S}\) consists of the following curves:
\begin{gather*}
  C_1 = C_{(X, Y)}, \;
  C_2 = C_{(Z, T)}, \;
  C_3 = C_{(X, Z + T)}, \;
  C_4 = C_{(Y, Z + T)}, \;
  C_5 = C_{(Z, X + Y)}, \;
  C_6 = C_{(Z, X + T)}, \\
  C_7 = C_{(Z, Y + T)}, \;
  C_8 = C_{(T, X + Y)}, \;
  C_9 = C_{(T, X + Z)}, \;
  C_{10} = C_{(T, Y + Z)}, \;
  C_{11} = C_{(X, Y + Z + T)}, \;
  C_{12} = C_{(Y, X + Z + T)}.
\end{gather*}
Their linear equivalence classes on the generic member \(\mathcal{S}_{\Bbbk}\) of the pencil satisfy the following relations:
\begin{gather*}
  \begin{pmatrix}
    [C_{4}] \\ [C_{5}] \\ [C_{7}] \\ [C_{8}] \\ [C_{9}] \\ [C_{11}] \\ [C_{12}]
  \end{pmatrix} = 
  \begin{pmatrix}
    0 & -2 & -1 & 0 & 0 & 1 \\
    -1 & -3 & -3 & -1 & 1 & 2 \\
    1 & 2 & 3 & 0 & -1 & -1 \\
    -1 & 3 & 3 & 1 & -1 & -1 \\
    1 & -4 & -3 & -1 & 0 & 2 \\
    -1 & 0 & -2 & 0 & 0 & 1 \\
    -1 & 4 & 2 & 0 & 0 & -1
  \end{pmatrix} \cdot
  \begin{pmatrix}
    [C_{1}] \\ [C_{2}] \\ [C_{3}] \\ [C_{6}] \\ [C_{10}] \\ [H_{\mathcal{S}}]
  \end{pmatrix}.
\end{gather*}

For a general choice of \(\lambda \in \mathbb{C}\) the surface \(\mathcal{S}_{\lambda}\) has the following singularities:
\begin{itemize}\setlength{\itemindent}{2cm}
\item[\(P_{1} = P_{(X, Z, T)}\):] type \(\mathbb{A}_3\) with the quadratic term \((Z + T) \cdot (X + Z + T)\);
\item[\(P_{2} = P_{(Y, Z, T)}\):] type \(\mathbb{A}_3\) with the quadratic term \((Z + T) \cdot (Y + Z + T)\);
\item[\(P_{3} = P_{(X, Y, Z + T)}\):] type \(\mathbb{A}_5\) with the quadratic term \((\lambda + 4) X \cdot Y\);
\item[\(P_{4} = P_{(Z, T, X + Y)}\):] type \(\mathbb{A}_1\) with the quadratic term \((X + Y) (Z + T) - (\lambda + 4) Z T\).
\end{itemize}

Galois action on the lattice \(L_{\lambda}\) is trivial. The intersection matrix on \(L_{\lambda} = L_{\mathcal{S}}\) is represented by
\begin{table}[H]
  \begin{tabular}{|c||ccc|ccc|ccccc|c|cccccc|}
    \hline
    \(\bullet\) & \(E_1^1\) & \(E_1^2\) & \(E_1^3\) & \(E_2^1\) & \(E_2^2\) & \(E_2^3\) & \(E_3^1\) & \(E_3^2\) & \(E_3^3\) & \(E_3^4\) & \(E_3^5\) & \(E_4^1\) & \(\widetilde{C_{1}}\) & \(\widetilde{C_{2}}\) & \(\widetilde{C_{3}}\) & \(\widetilde{C_{6}}\) & \(\widetilde{C_{10}}\) & \(\widetilde{H_{\mathcal{S}}}\) \\
    \hline
    \hline
    \(\widetilde{C_{1}}\) & \(0\) & \(0\) & \(0\) & \(0\) & \(0\) & \(0\) & \(0\) & \(0\) & \(1\) & \(0\) & \(0\) & \(0\) & \(-2\) & \(0\) & \(0\) & \(0\) & \(0\) & \(1\) \\
    \(\widetilde{C_{2}}\) & \(1\) & \(0\) & \(0\) & \(1\) & \(0\) & \(0\) & \(0\) & \(0\) & \(0\) & \(0\) & \(0\) & \(1\) & \(0\) & \(-2\) & \(0\) & \(0\) & \(0\) & \(1\) \\
    \(\widetilde{C_{3}}\) & \(0\) & \(1\) & \(0\) & \(0\) & \(0\) & \(0\) & \(1\) & \(0\) & \(0\) & \(0\) & \(0\) & \(0\) & \(0\) & \(0\) & \(-2\) & \(0\) & \(0\) & \(1\) \\
    \(\widetilde{C_{6}}\) & \(0\) & \(0\) & \(1\) & \(0\) & \(0\) & \(0\) & \(0\) & \(0\) & \(0\) & \(0\) & \(0\) & \(0\) & \(0\) & \(0\) & \(0\) & \(-2\) & \(0\) & \(1\) \\
    \(\widetilde{C_{10}}\) & \(0\) & \(0\) & \(0\) & \(0\) & \(0\) & \(1\) & \(0\) & \(0\) & \(0\) & \(0\) & \(0\) & \(0\) & \(0\) & \(0\) & \(0\) & \(0\) & \(-2\) & \(1\) \\
    \(\widetilde{H_{\mathcal{S}}}\) & \(0\) & \(0\) & \(0\) & \(0\) & \(0\) & \(0\) & \(0\) & \(0\) & \(0\) & \(0\) & \(0\) & \(0\) & \(1\) & \(1\) & \(1\) & \(1\) & \(1\) & \(4\) \\
    \hline
  \end{tabular}.
\end{table}
Note that the intersection matrix is non-degenerate.

Discriminant groups and discriminant forms of the lattices \(L_{\mathcal{S}}\) and \(H \oplus \Pic(X)\) are given by
\begin{gather*}
  G' = 
  \begin{pmatrix}
    \frac{1}{2} & 0 & 0 & 0 & 0 & 0 & 0 & \frac{1}{2} & 0 & 0 & 0 & \frac{1}{2} & \frac{1}{2} & 0 & \frac{1}{2} & 0 & 0 & 0 \\
    \frac{1}{3} & \frac{1}{3} & \frac{5}{6} & \frac{5}{6} & \frac{1}{3} & \frac{5}{6} & \frac{5}{6} & \frac{1}{6} & \frac{1}{2} & \frac{2}{3} & \frac{5}{6} & \frac{2}{3} & \frac{1}{6} & \frac{1}{3} & \frac{1}{2} & \frac{1}{3} & \frac{1}{3} & \frac{5}{6}
  \end{pmatrix}, \\
  G'' = 
  \begin{pmatrix}
    0 & 0 & \frac{1}{2} & 0 \\
    0 & 0 & \frac{2}{3} & \frac{1}{6}
  \end{pmatrix}; \;
  B' = 
  \begin{pmatrix}
    \frac{1}{2} & 0 \\
    0 & \frac{1}{6}
  \end{pmatrix}, \;
  B'' = 
  \begin{pmatrix}
    \frac{1}{2} & 0 \\
    0 & \frac{5}{6}
  \end{pmatrix}; \;
  \begin{pmatrix}
    Q' \\ Q''
  \end{pmatrix}
  =
  \begin{pmatrix}
    \frac{3}{2} & \frac{1}{6} \\
    \frac{1}{2} & \frac{11}{6}
  \end{pmatrix}.
\end{gather*}


\subsection{Family \textnumero2.7}\label{subsection:02-07}

The pencil \(\mathcal{S}\) is defined by the equation
\begin{gather*}
  X^{2} Y^{2} + X^{2} Y Z + 2 X Y^{2} Z + X Y Z^{2} + Y^{2} Z^{2} + 2 X^{2} Y T + 2 X Y^{2} T + X^{2} Z T + 2 Y^{2} Z T + \\ X Z^{2} T + 2 Y Z^{2} T + X^{2} T^{2} + 2 X Y T^{2} + Y^{2} T^{2} + 2 X Z T^{2} + 2 Y Z T^{2} + Z^{2} T^{2} = \lambda X Y Z T.
\end{gather*}
Members \(\mathcal{S}_{\lambda}\) of the pencil are irreducible for any \(\lambda \in \mathbb{P}^1\) except
\begin{gather*}
  \mathcal{S}_{\infty} = S_{(X)} + S_{(Y)} + S_{(Z)} + S_{(T)}, \;
  \mathcal{S}_{- 5} = S_{((X + Z) (Y + T) + Y T)} + S_{((X + T) (Y + Z) + X T + Y Z)}.
\end{gather*}
The base locus of the pencil \(\mathcal{S}\) consists of the following curves:
\begin{gather*}
  C_{1} = C_{(Y, T)}, \;
  C_{2} = C_{(Y, X + Z)}, \;
  C_{3} = C_{(T, X + Z)}, \;
  C_{4} = C_{(X, Y T + Z (Y + T))}, \\
  C_{5} = C_{(Y, X Z + T (X + Z))}, \;
  C_{6} = C_{(Z, Y T + X (Y + T))}, \;
  C_{7} = C_{(T, X Z + Y (X + Z))}.
\end{gather*}
Their linear equivalence classes on the generic member \(\mathcal{S}_{\Bbbk}\) of the pencil satisfy the following relations:
\[
  \begin{pmatrix}
    [C_5] \\ 2 [C_6] \\ [C_7] \\ [H_{\mathcal{S}}]
  \end{pmatrix} =
  \begin{pmatrix}
    -1 & -1 & 0 & 2 \\
    0 & 0 & 0 & 2 \\
    -1 & 0 & -1 & 2 \\
    0 & 0 & 0 & 2
  \end{pmatrix} \cdot
  \begin{pmatrix}
    [C_1] \\ [C_2] \\ [C_3] \\ [C_4]
  \end{pmatrix}.
\]

For a general choice of \(\lambda \in \mathbb{C}\) the surface \(\mathcal{S}_{\lambda}\) has the following singularities:
\begin{itemize}\setlength{\itemindent}{2cm}
\item[\(P_{1} = P_{(X, Y, Z)}\):] type \(\mathbb{D}_4\) with the quadratic term \((X + Y + Z)^2\);
\item[\(P_{2} = P_{(X, Y, T)}\):] type \(\mathbb{A}_3\) with the quadratic term \((Y + T) \cdot (X + Y + T)\);
\item[\(P_{3} = P_{(X, Z, T)}\):] type \(\mathbb{D}_4\) with the quadratic term \((X + Z + T)^2\);
\item[\(P_{4} = P_{(Y, Z, T)}\):] type \(\mathbb{A}_3\) with the quadratic term \((Y + T) \cdot (Y + Z + T)\);
\item[\(P_{5} = P_{(Y, T, X + Z)}\):] type \(\mathbb{A}_1\) with the quadratic term \((X + Z) (Y + T) - (\lambda + 4) Y T\).
\end{itemize}

Galois action on the lattice \(L_{\lambda}\) is trivial. The intersection matrix on \(L_{\lambda} = L_{\mathcal{S}}\) is represented by
\begin{table}[H]
  \setlength{\tabcolsep}{4pt}
  \begin{tabular}{|c||cccc|ccc|cccc|ccc|c|ccccc|}
    \hline
    \(\bullet\) & \(E_1^1\) & \(E_1^2\) & \(E_1^3\) & \(E_1^4\) & \(E_2^1\) & \(E_2^2\) & \(E_2^3\) & \(E_3^1\) & \(E_3^2\) & \(E_3^3\) & \(E_3^4\) & \(E_4^1\) & \(E_4^2\) & \(E_4^3\) & \(E_5^1\) & \(\widetilde{C_{1}}\) & \(\widetilde{C_{2}}\) & \(\widetilde{C_{3}}\) & \(\widetilde{C_{4}}\) & \(\widetilde{C_{6}}\) \\
    \hline
    \hline
    \(\widetilde{C_{1}}\) & \(0\) & \(0\) & \(0\) & \(0\) & \(1\) & \(0\) & \(0\) & \(0\) & \(0\) & \(0\) & \(0\) & \(1\) & \(0\) & \(0\) & \(1\) & \(-2\) & \(0\) & \(0\) & \(0\) & \(0\) \\
    \(\widetilde{C_{2}}\) & \(1\) & \(0\) & \(0\) & \(0\) & \(0\) & \(0\) & \(0\) & \(0\) & \(0\) & \(0\) & \(0\) & \(0\) & \(0\) & \(0\) & \(1\) & \(0\) & \(-2\) & \(0\) & \(0\) & \(0\) \\
    \(\widetilde{C_{3}}\) & \(0\) & \(0\) & \(0\) & \(0\) & \(0\) & \(0\) & \(0\) & \(1\) & \(0\) & \(0\) & \(0\) & \(0\) & \(0\) & \(0\) & \(1\) & \(0\) & \(0\) & \(-2\) & \(0\) & \(0\) \\
    \(\widetilde{C_{4}}\) & \(0\) & \(0\) & \(1\) & \(0\) & \(0\) & \(1\) & \(0\) & \(0\) & \(0\) & \(1\) & \(0\) & \(0\) & \(0\) & \(0\) & \(0\) & \(0\) & \(0\) & \(0\) & \(-2\) & \(0\) \\
    \(\widetilde{C_{6}}\) & \(0\) & \(0\) & \(0\) & \(1\) & \(0\) & \(0\) & \(0\) & \(0\) & \(0\) & \(0\) & \(1\) & \(0\) & \(1\) & \(0\) & \(0\) & \(0\) & \(0\) & \(0\) & \(0\) & \(-2\) \\
    \hline
  \end{tabular}.
\end{table}
Note that the intersection matrix is degenerate. We choose the following integral basis of the lattice \(L_{\lambda}\):
\begin{align*}
  \begin{pmatrix}
    [E_4^3] \\ [\widetilde{C_{6}}]
  \end{pmatrix} =
  \begin{pmatrix}
    0 & 2 & 3 & 1 & -1 & 2 & 1 & 0 & 2 & 3 & 1 & -3 & -2 & -4 & -4 & -2 & -2 & 4 \\
    0 & -1 & -1 & -1 & 1 & 0 & 0 & 0 & -1 & -1 & -1 & 1 & 0 & 2 & 2 & 1 & 1 & -1
  \end{pmatrix} \cdot \\
  \begin{pmatrix}
    [E_1^1] & [E_1^2] & [E_1^3] & [E_1^4] & [E_2^1] & [E_2^2] & [E_2^3] & [E_3^1] & [E_3^2] & \\
    [E_3^3] & [E_3^4] & [E_4^1] & [E_4^2] & [E_5^1] & [\widetilde{C_{1}}] & [\widetilde{C_{2}}] & [\widetilde{C_{3}}] & [\widetilde{C_{4}}]
  \end{pmatrix}^T.
\end{align*}

Discriminant groups and discriminant forms of the lattices \(L_{\mathcal{S}}\) and \(H \oplus \Pic(X)\) are given by
\begin{gather*}
  G' = 
  \begin{pmatrix}
    \frac{1}{2} & 0 & 0 & \frac{1}{2} & \frac{1}{2} & 0 & \frac{1}{2} & \frac{1}{2} & 0 & 0 & \frac{1}{2} & 0 & 0 & \frac{1}{2} & 0 & 0 & 0 & 0 \\
    0 & \frac{1}{4} & \frac{3}{8} & \frac{1}{8} & \frac{1}{2} & 0 & 0 & 0 & \frac{3}{4} & \frac{5}{8} & \frac{7}{8} & 0 & 0 & \frac{1}{2} & 0 & \frac{3}{4} & \frac{1}{4} & \frac{1}{2}
  \end{pmatrix}, \\
  G'' = 
  \begin{pmatrix}
    0 & 0 & \frac{1}{2} & \frac{1}{2} \\
    0 & 0 & \frac{7}{8} & 0
  \end{pmatrix}; \;
  B' = 
  \begin{pmatrix}
    \frac{1}{2} & 0 \\
    0 & \frac{7}{8}
  \end{pmatrix}, \;
  B'' = 
  \begin{pmatrix}
    \frac{1}{2} & 0 \\
    0 & \frac{1}{8}
  \end{pmatrix}; \;
  \begin{pmatrix}
    Q' \\ Q''
  \end{pmatrix}
  =
  \begin{pmatrix}
    \frac{1}{2} & \frac{15}{8} \\
    \frac{3}{2} & \frac{1}{8}    
  \end{pmatrix}.
\end{gather*}


\subsection{Family \textnumero2.8}\label{subsection:02-08}

The pencil \(\mathcal{S}\) is defined by the equation
\begin{gather*}
  X^{2} Y^{2} + 2 X^{2} Y Z + 2 X Y^{2} Z + X^{2} Z^{2} + 2 X Y Z^{2} + Y^{2} Z^{2} + 2 X^{2} Y T + \\ 2 X^{2} Z T + X^{2} T^{2} + 2 X Y T^{2} + 2 X Z T^{2} + 2 Y Z T^{2} + 2 X T^{3} + T^{4} = \lambda X Y Z T.
\end{gather*}
Members \(\mathcal{S}_{\lambda}\) of the pencil are irreducible for any \(\lambda \in \mathbb{P}^1\) except
\begin{gather*}
  \mathcal{S}_{\infty} = S_{(X)} + S_{(Y)} + S_{(Z)} + S_{(T)}, \;
  \mathcal{S}_{- 2} = 2 S_{(X (Y + Z + T) + Y Z + T^2)}.
\end{gather*}
The base locus of the pencil \(\mathcal{S}\) consists of the following curves:
\begin{gather*}
  C_1 = C_{(X, Y Z + T^2)}, \;
  C_2 = C_{(Y, X (Z + T) + T^2)}, \;
  C_3 = C_{(Z, X (Y + T) + T^2)}, \;
  C_4 = C_{(T, X (Y + Z) + Y Z)}.
\end{gather*}
Their linear equivalence classes on the generic member \(\mathcal{S}_{\Bbbk}\) of the pencil satisfy the following relations:
\[
  2 [C_1] = 2 [C_2] = 2 [C_3] = 2 [C_4] = [H_{\mathcal{S}}].
\]

For a general choice of \(\lambda \in \mathbb{C}\) the surface \(\mathcal{S}_{\lambda}\) has the following singularities:
\begin{itemize}\setlength{\itemindent}{2cm}
\item[\(P_{1} = P_{(X, Y, T)}\):] type \(\mathbb{D}_6\) with the quadratic term \((X + Y)^2\);
\item[\(P_{2} = P_{(X, Z, T)}\):] type \(\mathbb{D}_6\) with the quadratic term \((X + Z)^2\);
\item[\(P_{3} = P_{(Y, Z, T)}\):] type \(\mathbb{D}_4\) with the quadratic term \((Y + Z + T)^2\);
\item[\(P_{4} = P_{(Y, Z, X + T)}\):] type \(\mathbb{A}_1\) with the quadratic term \((X - Y - Z + T)^2 + (\lambda + 2) Y Z\).
\end{itemize}

Galois action on the lattice \(L_{\lambda}\) is trivial. The intersection matrix on \(L_{\lambda} = L_{\mathcal{S}}\) is represented by
\begin{table}[H]
  \setlength{\tabcolsep}{3pt}
  \begin{tabular}{|c||cccccc|cccccc|cccc|c|ccccc|}
    \hline
    \(\bullet\) & \(E_1^1\) & \(E_1^2\) & \(E_1^3\) & \(E_1^4\) & \(E_1^5\) & \(E_1^6\) & \(E_2^1\) & \(E_2^2\) & \(E_2^3\) & \(E_2^4\) & \(E_2^5\) & \(E_2^6\) & \(E_3^1\) & \(E_3^2\) & \(E_3^3\) & \(E_3^4\) & \(E_4^1\) & \(\widetilde{C_{1}}\) & \(\widetilde{C_{2}}\) & \(\widetilde{C_{3}}\) & \(\widetilde{C_{4}}\) & \(\widetilde{H_{\mathcal{S}}}\) \\
    \hline
    \hline
    \(\widetilde{C_{1}}\) & \(0\) & \(0\) & \(0\) & \(0\) & \(1\) & \(0\) & \(0\) & \(0\) & \(0\) & \(0\) & \(1\) & \(0\) & \(0\) & \(0\) & \(0\) & \(0\) & \(0\) & \(-2\) & \(0\) & \(0\) & \(0\) & \(2\) \\
    \(\widetilde{C_{2}}\) & \(0\) & \(0\) & \(0\) & \(0\) & \(0\) & \(1\) & \(0\) & \(0\) & \(0\) & \(0\) & \(0\) & \(0\) & \(1\) & \(0\) & \(0\) & \(0\) & \(1\) & \(0\) & \(-2\) & \(0\) & \(0\) & \(2\) \\
    \(\widetilde{C_{3}}\) & \(0\) & \(0\) & \(0\) & \(0\) & \(0\) & \(0\) & \(0\) & \(0\) & \(0\) & \(0\) & \(0\) & \(1\) & \(0\) & \(0\) & \(1\) & \(0\) & \(1\) & \(0\) & \(0\) & \(-2\) & \(0\) & \(2\) \\
    \(\widetilde{C_{4}}\) & \(1\) & \(0\) & \(0\) & \(0\) & \(0\) & \(0\) & \(1\) & \(0\) & \(0\) & \(0\) & \(0\) & \(0\) & \(0\) & \(0\) & \(0\) & \(1\) & \(0\) & \(0\) & \(0\) & \(0\) & \(-2\) & \(2\) \\
    \(\widetilde{H_{\mathcal{S}}}\) & \(0\) & \(0\) & \(0\) & \(0\) & \(0\) & \(0\) & \(0\) & \(0\) & \(0\) & \(0\) & \(0\) & \(0\) & \(0\) & \(0\) & \(0\) & \(0\) & \(0\) & \(2\) & \(2\) & \(2\) & \(2\) & \(4\) \\
    \hline
  \end{tabular}.
\end{table}
Note that the intersection matrix is degenerate. We choose the following integral basis of the lattice \(L_{\lambda}\):
\begin{align*}
  \begin{pmatrix}
    [E_3^3] \\ [E_4^1] \\ [\widetilde{C_{4}}] \\ [\widetilde{H_{\mathcal{S}}}]
  \end{pmatrix} =
  \begin{pmatrix}
    1 & 2 & 3 & 4 & 2 & 3 & -1 & -2 & -3 & -4 & -2 & -3 & 1 & 0 & 0 & 0 & 2 & -2 \\
    -1 & -2 & -3 & -4 & -1 & -4 & 2 & 4 & 6 & 8 & 5 & 5 & -3 & -2 & -1 & 2 & -4 & 2 \\
    -1 & -1 & -1 & -1 & 0 & -1 & 0 & 1 & 2 & 3 & 2 & 2 & -1 & -1 & -1 & 1 & -1 & 1 \\
    1 & 2 & 3 & 4 & 3 & 2 & 1 & 2 & 3 & 4 & 3 & 2 & 0 & 0 & 0 & 2 & 0 & 0
  \end{pmatrix} \cdot \\
  \begin{pmatrix}
    [E_1^1] & [E_1^2] & [E_1^3] & [E_1^4] & [E_1^5] & [E_1^6] & [E_2^1] & [E_2^2] & [E_2^3] & \\
    [E_2^4] & [E_2^5] & [E_2^6] & [E_3^1] & [E_3^2] & [E_3^4] & [\widetilde{C_{1}}] & [\widetilde{C_{2}}] & [\widetilde{C_{3}}]
  \end{pmatrix}^T.
\end{align*}

Discriminant groups and discriminant forms of the lattices \(L_{\mathcal{S}}\) and \(H \oplus \Pic(X)\) are given by
\begin{gather*}
  G' = 
  \begin{pmatrix}
    0 & 0 & 0 & 0 & \frac{1}{2} & \frac{1}{2} & \frac{1}{2} & 0 & \frac{1}{2} & 0 & \frac{1}{2} & 0 & \frac{1}{2} & 0 & \frac{1}{2} & 0 & 0 & 0 \\
    \frac{3}{4} & \frac{1}{2} & \frac{1}{4} & 0 & \frac{3}{4} & 0 & \frac{1}{4} & \frac{1}{2} & \frac{3}{4} & 0 & \frac{1}{4} & 0 & 0 & 0 & 0 & \frac{1}{2} & 0 & 0
  \end{pmatrix}, \\
  G'' = 
  \begin{pmatrix}
    0 & 0 & \frac{1}{2} & 0 \\
    0 & 0 & 0 & \frac{1}{4}
  \end{pmatrix}; \;
  B' = 
  \begin{pmatrix}
    \frac{1}{2} & 0 \\
    0 & \frac{3}{4}
  \end{pmatrix}, \;
  B'' = 
  \begin{pmatrix}
    \frac{1}{2} & 0 \\
    0 & \frac{1}{4}
  \end{pmatrix}; \;
  \begin{pmatrix}
    Q' \\ Q''
  \end{pmatrix}
  =
  \begin{pmatrix}
    \frac{1}{2} & \frac{7}{4} \\
    \frac{3}{2} & \frac{1}{4}
  \end{pmatrix}.
\end{gather*}


\subsection{Family \textnumero2.9}\label{subsection:02-09}

The pencil \(\mathcal{S}\) is defined by the equation
\begin{gather*}
  X^{2} Y Z + X Y^{2} Z + X Y Z^{2} + X^{2} Y T + X Y^{2} T + X^{2} Z T + Y^{2} Z T + 2 X Z^{2} T + \\ 2 Y Z^{2} T + Z^{3} T + X^{2} T^{2} + 2 X Y T^{2} + Y^{2} T^{2} + 2 X Z T^{2} + 2 Y Z T^{2} + Z^{2} T^{2} = \lambda X Y Z T.
\end{gather*}
Members \(\mathcal{S}_{\lambda}\) of the pencil are irreducible for any \(\lambda \in \mathbb{P}^1\) except
\begin{gather*}
  \mathcal{S}_{\infty} = S_{(X)} + S_{(Y)} + S_{(Z)} + S_{(T)}, \;
  \mathcal{S}_{- 3} = S_{(Z + T)} + S_{(X + Y + Z)} + S_{(T (X + Y + Z) + X Y)}.
\end{gather*}
The base locus of the pencil \(\mathcal{S}\) consists of the following curves:
\begin{gather*}
  C_{1} = C_{(X, T)}, \;
  C_{2} = C_{(Y, T)}, \;
  C_{3} = C_{(Z, T)}, \;
  C_{4} = C_{(X, Y + Z)}, \;
  C_{5} = C_{(X, Z + T)}, \;
  C_{6} = C_{(Y, X + Z)}, \\
  C_{7} = C_{(Y, Z + T)}, \;
  C_{8} = C_{(Z, X + Y)}, \;
  C_{9} = C_{(T, X + Y + Z)}, \;
  C_{10} = C_{(Z, X Y + T (X + Y))}.
\end{gather*}
Their linear equivalence classes on the generic member \(\mathcal{S}_{\Bbbk}\) of the pencil satisfy the following relations:
\begin{gather*}
  \begin{pmatrix}
    [C_{2}] \\ [C_{5}] \\ [C_{7}] \\ [C_{8}] \\ [C_{9}] \\ [C_{10}]
  \end{pmatrix} = 
  \begin{pmatrix}
    -1 & 2 & -2 & -2 & 1 \\
    -1 & 0 & -2 & 0 & 1 \\
    1 & -2 & 2 & 0 & 0 \\
    0 & 3 & -3 & -3 & 1 \\
    0 & -3 & 2 & 2 & 0 \\
    0 & -4 & 3 & 3 & 0
  \end{pmatrix} \cdot
  \begin{pmatrix}
    [C_{1}] \\ [C_{3}] \\ [C_{4}] \\ [C_{6}] \\ [H_{\mathcal{S}}]
  \end{pmatrix}.
\end{gather*}

For a general choice of \(\lambda \in \mathbb{C}\) the surface \(\mathcal{S}_{\lambda}\) has the following singularities:
\begin{itemize}\setlength{\itemindent}{2cm}
\item[\(P_{1} = P_{(X, Y, Z)}\):] type \(\mathbb{D}_4\) with the quadratic term \((X + Y + Z)^2\);
\item[\(P_{2} = P_{(X, Z, T)}\):] type \(\mathbb{A}_2\) with the quadratic term \((X + T) \cdot (Z + T)\);
\item[\(P_{3} = P_{(Y, Z, T)}\):] type \(\mathbb{A}_2\) with the quadratic term \((Y + T) \cdot (Z + T)\);
\item[\(P_{4} = P_{(X, T, Y + Z)}\):] type \(\mathbb{A}_2\) with the quadratic term \(X \cdot (X + Y + Z - (\lambda + 3) T)\);
\item[\(P_{5} = P_{(Y, T, X + Z)}\):] type \(\mathbb{A}_2\) with the quadratic term \(Y \cdot (X + Y + Z - (\lambda + 3) T)\);
\item[\(P_{6} = P_{(Z, T, X + Y)}\):] type \(\mathbb{A}_1\) with the quadratic term \((Z + T) (X + Y + Z) - (\lambda + 3) Z T\).
\end{itemize}

Galois action on the lattice \(L_{\lambda}\) is trivial. The intersection matrix on \(L_{\lambda} = L_{\mathcal{S}}\) is represented by
\begin{table}[H]
  \begin{tabular}{|c||cccc|cc|cc|cc|cc|c|ccccc|}
    \hline
    \(\bullet\) & \(E_1^1\) & \(E_1^2\) & \(E_1^3\) & \(E_1^4\) & \(E_2^1\) & \(E_2^2\) & \(E_3^1\) & \(E_3^2\) & \(E_4^1\) & \(E_4^2\) & \(E_5^1\) & \(E_5^2\) & \(E_6^1\) & \(\widetilde{C_{1}}\) & \(\widetilde{C_{3}}\) & \(\widetilde{C_{4}}\) & \(\widetilde{C_{6}}\) & \(\widetilde{H_{\mathcal{S}}}\) \\
    \hline
    \hline
    \(\widetilde{C_{1}}\) & \(0\) & \(0\) & \(0\) & \(0\) & \(1\) & \(0\) & \(0\) & \(0\) & \(1\) & \(0\) & \(0\) & \(0\) & \(0\) & \(-2\) & \(0\) & \(0\) & \(0\) & \(1\) \\
    \(\widetilde{C_{3}}\) & \(0\) & \(0\) & \(0\) & \(0\) & \(0\) & \(1\) & \(0\) & \(1\) & \(0\) & \(0\) & \(0\) & \(0\) & \(1\) & \(0\) & \(-2\) & \(0\) & \(0\) & \(1\) \\
    \(\widetilde{C_{4}}\) & \(1\) & \(0\) & \(0\) & \(0\) & \(0\) & \(0\) & \(0\) & \(0\) & \(1\) & \(0\) & \(0\) & \(0\) & \(0\) & \(0\) & \(0\) & \(-2\) & \(0\) & \(1\) \\
    \(\widetilde{C_{6}}\) & \(0\) & \(0\) & \(1\) & \(0\) & \(0\) & \(0\) & \(0\) & \(0\) & \(0\) & \(0\) & \(1\) & \(0\) & \(0\) & \(0\) & \(0\) & \(0\) & \(-2\) & \(1\) \\
    \(\widetilde{H_{\mathcal{S}}}\) & \(0\) & \(0\) & \(0\) & \(0\) & \(0\) & \(0\) & \(0\) & \(0\) & \(0\) & \(0\) & \(0\) & \(0\) & \(0\) & \(1\) & \(1\) & \(1\) & \(1\) & \(4\) \\
    \hline
  \end{tabular}.
\end{table}
Note that the intersection matrix is non-degenerate.

Discriminant groups and discriminant forms of the lattices \(L_{\mathcal{S}}\) and \(H \oplus \Pic(X)\) are given by
\begin{gather*}
  G' = 
  \begin{pmatrix}
    0 & \frac{6}{17} & \frac{9}{17} & \frac{3}{17} & \frac{1}{17} & \frac{16}{17} & \frac{16}{17} & \frac{15}{17} & \frac{15}{17} & \frac{16}{17} & \frac{8}{17} & \frac{4}{17} & \frac{7}{17} & \frac{3}{17} & \frac{14}{17} & \frac{11}{17} & \frac{12}{17} & \frac{7}{17}
  \end{pmatrix}, \\
  G'' = 
  \begin{pmatrix}
    0 & 0 & \frac{14}{17} & \frac{1}{17}
  \end{pmatrix}; \;
  B' = 
  \begin{pmatrix}
    \frac{15}{17}
  \end{pmatrix}, \;
  B'' = 
  \begin{pmatrix}
    \frac{2}{17}
  \end{pmatrix}; \;
  Q' =
  \begin{pmatrix}
    \frac{32}{17}
  \end{pmatrix}, \;
  Q'' =
  \begin{pmatrix}
    \frac{2}{17}
  \end{pmatrix}.
\end{gather*}


\subsection{Family \textnumero2.10}\label{subsection:02-10}

The pencil \(\mathcal{S}\) is defined by the equation
\begin{gather*}
  X^{2} Y Z + X Y^{2} Z + 2 X Y Z^{2} + Y^{2} Z^{2} + Y Z^{3} + X^{2} Y T + X^{2} Z T + Y^{2} Z T + X Z^{2} T + \\ 3 Y Z^{2} T + X^{2} T^{2} + 2 X Y T^{2} + 2 X Z T^{2} + 3 Y Z T^{2} + X T^{3} + Y T^{3} = \lambda X Y Z T.
\end{gather*}
Members \(\mathcal{S}_{\lambda}\) of the pencil are irreducible for any \(\lambda \in \mathbb{P}^1\) except
\begin{gather*}
  \mathcal{S}_{\infty} = S_{(X)} + S_{(Y)} + S_{(Z)} + S_{(T)}, \\
  \mathcal{S}_{- 4} = S_{(X + Z + T)} + S_{(Y^2 Z + (Z + T) (Y (X + Z + T) + X T))}, \;
  \mathcal{S}_{- 5} = S_{(Y (X + Z + T) + X T)} + S_{(X T + Z (X + Y) + (Z + T)^2)}.
\end{gather*}
The base locus of the pencil \(\mathcal{S}\) consists of the following curves:
\begin{gather*}
  C_{1} = C_{(X, Y)}, \;
  C_{2} = C_{(Y, T)}, \;
  C_{3} = C_{(Z, T)}, \;
  C_{4} = C_{(X, Z + T)}, \;
  C_{5} = C_{(Y, Z + T)}, \;
  C_{6} = C_{(Z, X + T)}, \\
  C_{7} = C_{(T, X + Z)}, \;
  C_{8} = C_{(T, X + Y + Z)}, \;
  C_{9} = C_{(Y, X + Z + T)}, \;
  C_{10} = C_{(X, Y Z + (Z + T)^2)}, \;
  C_{11} = C_{(Z, X T + Y (X + T))}.
\end{gather*}
Their linear equivalence classes on the generic member \(\mathcal{S}_{\Bbbk}\) of the pencil satisfy the following relations:
\begin{gather*}
  \begin{pmatrix}
    [C_{7}] \\ [C_{8}] \\ [C_{9}] \\ [C_{10}] \\ [C_{11}]
  \end{pmatrix} = 
  \begin{pmatrix}
    1 & 1 & 0 & -1 & 1 & -1 & 0 \\
    -1 & -2 & -1 & 1 & -1 & 1 & 1 \\
    -1 & -1 & 0 & 0 & -1 & 0 & 1 \\
    -1 & 0 & 0 & -1 & 0 & 0 & 1 \\
    0 & 0 & -1 & 0 & 0 & -1 & 1
  \end{pmatrix} \cdot
  \begin{pmatrix}
    [C_{1}] & [C_{2}] & [C_{3}] & [C_{4}] & [C_{5}] & [C_{6}] & [H_{\mathcal{S}}]
  \end{pmatrix}^T.
\end{gather*}

For a general choice of \(\lambda \in \mathbb{C}\) the surface \(\mathcal{S}_{\lambda}\) has the following singularities:
\begin{itemize}\setlength{\itemindent}{2cm}
\item[\(P_{1} = P_{(X, Z, T)}\):] type \(\mathbb{A}_4\) with the quadratic term \(Z \cdot (X + Z + T)\);
\item[\(P_{2} = P_{(Y, Z, T)}\):] type \(\mathbb{A}_2\) with the quadratic term \((Y + T) \cdot (Z + T)\);
\item[\(P_{3} = P_{(X, Y, Z + T)}\):] type \(\mathbb{A}_4\) with the quadratic term \((\lambda + 4) X \cdot Y \);
\item[\(P_{4} = P_{(Y, T, X + Z)}\):] type \(\mathbb{A}_2\) with the quadratic term \(T \cdot (X - (\lambda + 4) Y + Z + T)\).
\end{itemize}

Galois action on the lattice \(L_{\lambda}\) is trivial. The intersection matrix on \(L_{\lambda} = L_{\mathcal{S}}\) is represented by
\begin{table}[H]
  \begin{tabular}{|c||cccc|cc|cccc|cc|ccccccc|}
    \hline
    \(\bullet\) & \(E_1^1\) & \(E_1^2\) & \(E_1^3\) & \(E_1^4\) & \(E_2^1\) & \(E_2^2\) & \(E_3^1\) & \(E_3^2\) & \(E_3^3\) & \(E_3^4\) & \(E_4^1\) & \(E_4^2\) & \(\widetilde{C_{1}}\) & \(\widetilde{C_{2}}\) & \(\widetilde{C_{3}}\) & \(\widetilde{C_{4}}\) & \(\widetilde{C_{5}}\) & \(\widetilde{C_{6}}\) & \(\widetilde{H_{\mathcal{S}}}\) \\
    \hline
    \hline
    \(\widetilde{C_{1}}\) & \(0\) & \(0\) & \(0\) & \(0\) & \(0\) & \(0\) & \(0\) & \(1\) & \(0\) & \(0\) & \(0\) & \(0\) & \(-2\) & \(1\) & \(0\) & \(0\) & \(0\) & \(0\) & \(1\) \\
    \(\widetilde{C_{2}}\) & \(0\) & \(0\) & \(0\) & \(0\) & \(1\) & \(0\) & \(0\) & \(0\) & \(0\) & \(0\) & \(1\) & \(0\) & \(1\) & \(-2\) & \(0\) & \(0\) & \(0\) & \(0\) & \(1\) \\
    \(\widetilde{C_{3}}\) & \(1\) & \(0\) & \(0\) & \(0\) & \(0\) & \(1\) & \(0\) & \(0\) & \(0\) & \(0\) & \(0\) & \(0\) & \(0\) & \(0\) & \(-2\) & \(0\) & \(0\) & \(0\) & \(1\) \\
    \(\widetilde{C_{4}}\) & \(0\) & \(0\) & \(0\) & \(1\) & \(0\) & \(0\) & \(1\) & \(0\) & \(0\) & \(0\) & \(0\) & \(0\) & \(0\) & \(0\) & \(0\) & \(-2\) & \(0\) & \(0\) & \(1\) \\
    \(\widetilde{C_{5}}\) & \(0\) & \(0\) & \(0\) & \(0\) & \(0\) & \(1\) & \(0\) & \(0\) & \(0\) & \(1\) & \(0\) & \(0\) & \(0\) & \(0\) & \(0\) & \(0\) & \(-2\) & \(0\) & \(1\) \\
    \(\widetilde{C_{6}}\) & \(0\) & \(1\) & \(0\) & \(0\) & \(0\) & \(0\) & \(0\) & \(0\) & \(0\) & \(0\) & \(0\) & \(0\) & \(0\) & \(0\) & \(0\) & \(0\) & \(0\) & \(-2\) & \(1\) \\
    \(\widetilde{H_{\mathcal{S}}}\) & \(0\) & \(0\) & \(0\) & \(0\) & \(0\) & \(0\) & \(0\) & \(0\) & \(0\) & \(0\) & \(0\) & \(0\) & \(1\) & \(1\) & \(1\) & \(1\) & \(1\) & \(1\) & \(4\) \\
    \hline
  \end{tabular}.
\end{table}
Note that the intersection matrix is degenerate. We choose the following integral basis of the lattice \(L_{\lambda}\):
\begin{align*}
  \begin{pmatrix}
    [\widetilde{C_{5}}]
  \end{pmatrix} =
  \begin{pmatrix}
    2 & 3 & 2 & 1 & -2 & -1 & -2 & -4 & -3 & -2 & -2 & -1 & -3 & -3 & 1 & 0 & 2 & 1
  \end{pmatrix} \cdot \\
  \begin{pmatrix}
    [E_1^1] & [E_1^2] & [E_1^3] & [E_1^4] & [E_2^1] & [E_2^2] & [E_3^1] & [E_3^2] & [E_3^3] & \\
    [E_3^4] & [E_4^1] & [E_4^2] & [\widetilde{C_{1}}] & [\widetilde{C_{2}}] & [\widetilde{C_{3}}] & [\widetilde{C_{4}}] & [\widetilde{C_{6}}] & [\widetilde{H_{\mathcal{S}}}]
  \end{pmatrix}^T.
\end{align*}

Discriminant groups and discriminant forms of the lattices \(L_{\mathcal{S}}\) and \(H \oplus \Pic(X)\) are given by
\begin{gather*}
  G' = 
  \begin{pmatrix}
    \frac{1}{2} & \frac{3}{4} & \frac{1}{2} & \frac{1}{4} & \frac{1}{4} & \frac{3}{4} & \frac{1}{2} & 0 & 0 & 0 & \frac{1}{2} & \frac{1}{4} & \frac{1}{2} & \frac{3}{4} & \frac{1}{4} & 0 & \frac{1}{2} & \frac{1}{4} \\
    \frac{3}{4} & \frac{1}{4} & \frac{1}{4} & \frac{1}{4} & \frac{3}{4} & 0 & \frac{1}{2} & \frac{3}{4} & \frac{1}{2} & \frac{1}{4} & 0 & \frac{1}{2} & \frac{1}{2} & \frac{1}{2} & \frac{1}{4} & \frac{1}{4} & \frac{1}{2} & \frac{3}{4}
  \end{pmatrix}, \\
  G'' = 
  \begin{pmatrix}
    0 & 0 & \frac{1}{4} & 0 \\
    0 & 0 & 0 & \frac{1}{4}
  \end{pmatrix}; \;
  B' = 
  \begin{pmatrix}
    0 & \frac{3}{4} \\
    \frac{3}{4} & \frac{1}{2}
  \end{pmatrix}, \;
  B'' = 
  \begin{pmatrix}
    0 & \frac{1}{4} \\
    \frac{1}{4} & \frac{1}{2}
  \end{pmatrix}; \;
  \begin{pmatrix}
    Q' \\ Q''
  \end{pmatrix}
  =
  \begin{pmatrix}
    0 & \frac{3}{2} \\
    0 & \frac{1}{2}
  \end{pmatrix}.
\end{gather*}


\subsection{Family \textnumero2.11}\label{subsection:02-11}

The pencil \(\mathcal{S}\) is defined by the equation
\begin{gather*}
  X Y^{2} Z + X Y Z^{2} + Y^{2} Z^{2} + 2 Y Z^{3} + Z^{4} + X^{2} Y T + X Y^{2} T + \\
  2 Y^{2} Z T + 2 X Z^{2} T + 2 Y Z^{2} T + X^{2} T^{2} + 2 X Y T^{2} + Y^{2} T^{2} = \lambda X Y Z T.
\end{gather*}
Members \(\mathcal{S}_{\lambda}\) of the pencil are irreducible for any \(\lambda \in \mathbb{P}^1\) except
\begin{gather*}
  \mathcal{S}_{\infty} = S_{(X)} + S_{(Y)} + S_{(Z)} + S_{(T)}, \;
  \mathcal{S}_{- 2} = S_{(Z (Y + Z) + T (X + Y))} + S_{(Z (Y + Z) + T (X + Y) + X Y)}.
\end{gather*}
The base locus of the pencil \(\mathcal{S}\) consists of the following curves:
\begin{gather*}
  C_{1} = C_{(Z, T)}, \;
  C_{2} = C_{(Z, X + Y)}, \;
  C_{3} = C_{(T, Y + Z)}, \;
  C_{4} = C_{(X, Y (Z + T) + Z^2)}, \\
  C_{5} = C_{(Y, X T + Z^2)}, \;
  C_{6} = C_{(Z, X T + Y (X + T))}, \;
  C_{7} = C_{(T, Y (X + Z) + Z^2)}.
\end{gather*}
Their linear equivalence classes on the generic member \(\mathcal{S}_{\Bbbk}\) of the pencil satisfy the following relations:
\[
  \begin{pmatrix}
    2 [C_5] \\ [C_6] \\ [C_7] \\ [H_{\mathcal{S}}]
  \end{pmatrix} =
  \begin{pmatrix}
    0 & 0 & 0 & 2 \\
    -1 & -1 & 0 & 2 \\
    -1 & 0 & -1 & 2 \\
    0 & 0 & 0 & 2
  \end{pmatrix} \cdot
  \begin{pmatrix}
    [C_1] \\ [C_2] \\ [C_3] \\ [C_4]
  \end{pmatrix}.
\]

For a general choice of \(\lambda \in \mathbb{C}\) the surface \(\mathcal{S}_{\lambda}\) has the following singularities:
\begin{itemize}\setlength{\itemindent}{2cm}
\item[\(P_{1} = P_{(X, Y, Z)}\):] type \(\mathbb{D}_6\) with the quadratic term \((X + Y)^2\);
\item[\(P_{2} = P_{(X, Z, T)}\):] type \(\mathbb{A}_3\) with the quadratic term \((Z + T) \cdot (X + Z + T)\);
\item[\(P_{3} = P_{(Y, Z, T)}\):] type \(\mathbb{A}_5\) with the quadratic term \(T \cdot (Y + T)\);
\item[\(P_{4} = P_{(X, T, Y + Z)}\):] type \(\mathbb{A}_1\) with the quadratic term \((Y + Z - T) (X - Y - Z + T) - (\lambda + 2) X T\).
\end{itemize}

Galois action on the lattice \(L_{\lambda}\) is trivial. The intersection matrix on \(L_{\lambda} = L_{\mathcal{S}}\) is represented by
\begin{table}[H]
  \setlength{\tabcolsep}{4pt}
  \begin{tabular}{|c||cccccc|ccc|ccccc|c|ccccc|}
    \hline
    \(\bullet\) & \(E_1^1\) & \(E_1^2\) & \(E_1^3\) & \(E_1^4\) & \(E_1^5\) & \(E_1^6\) & \(E_2^1\) & \(E_2^2\) & \(E_2^3\) & \(E_3^1\) & \(E_3^2\) & \(E_3^3\) & \(E_3^4\) & \(E_3^5\) & \(E_4^1\) & \(\widetilde{C_{1}}\) & \(\widetilde{C_{2}}\) & \(\widetilde{C_{3}}\) & \(\widetilde{C_{4}}\) & \(\widetilde{C_{5}}\) \\
    \hline
    \hline
    \(\widetilde{C_{1}}\) & \(0\) & \(0\) & \(0\) & \(0\) & \(0\) & \(0\) & \(1\) & \(0\) & \(0\) & \(1\) & \(0\) & \(0\) & \(0\) & \(0\) & \(0\) & \(-2\) & \(1\) & \(0\) & \(0\) & \(0\) \\
    \(\widetilde{C_{2}}\) & \(1\) & \(0\) & \(0\) & \(0\) & \(0\) & \(0\) & \(0\) & \(0\) & \(0\) & \(0\) & \(0\) & \(0\) & \(0\) & \(0\) & \(0\) & \(1\) & \(-2\) & \(0\) & \(0\) & \(0\) \\
    \(\widetilde{C_{3}}\) & \(0\) & \(0\) & \(0\) & \(0\) & \(0\) & \(0\) & \(0\) & \(0\) & \(0\) & \(1\) & \(0\) & \(0\) & \(0\) & \(0\) & \(1\) & \(0\) & \(0\) & \(-2\) & \(0\) & \(0\) \\
    \(\widetilde{C_{4}}\) & \(0\) & \(0\) & \(0\) & \(0\) & \(1\) & \(0\) & \(0\) & \(1\) & \(0\) & \(0\) & \(0\) & \(0\) & \(0\) & \(0\) & \(1\) & \(0\) & \(0\) & \(0\) & \(-2\) & \(0\) \\
    \(\widetilde{C_{5}}\) & \(0\) & \(0\) & \(0\) & \(0\) & \(0\) & \(1\) & \(0\) & \(0\) & \(0\) & \(0\) & \(0\) & \(1\) & \(0\) & \(0\) & \(0\) & \(0\) & \(0\) & \(0\) & \(0\) & \(-2\) \\
    \hline
  \end{tabular}.
\end{table}
Note that the intersection matrix is degenerate. We choose the following integral basis of the lattice \(L_{\lambda}\):
\begin{align*}
  \begin{pmatrix}
    [E_4^1] \\ [\widetilde{C_{5}}]
  \end{pmatrix} = 
  \begin{pmatrix}
    0 & -2 & -4 & -6 & -5 & -3 & 1 & -2 & -1 & 5 & 4 & 3 & 2 & 1 & 4 & 2 & 2 & -4 \\
    0 & -1 & -2 & -3 & -2 & -2 & 1 & 0 & 0 & 2 & 1 & 0 & 0 & 0 & 2 & 1 & 1 & -1
  \end{pmatrix} \cdot \\
  \begin{pmatrix}
    [E_1^1] & [E_1^2] & [E_1^3] & [E_1^4] & [E_1^5] & [E_1^6] & [E_2^1] & [E_2^2] & [E_2^3] & \\ [E_3^1] & [E_3^2] & [E_3^3] & [E_3^4] & [E_3^5] & [\widetilde{C_{1}}] & [\widetilde{C_{2}}] & [\widetilde{C_{3}}] & [\widetilde{C_{4}}]
  \end{pmatrix}^T.
\end{align*}

Discriminant groups and discriminant forms of the lattices \(L_{\mathcal{S}}\) and \(H \oplus \Pic(X)\) are given by
\begin{gather*}
  G' = 
  \begin{pmatrix}
    \frac{10}{13} & \frac{4}{13} & \frac{11}{13} & \frac{5}{13} & \frac{3}{13} & \frac{9}{13} & \frac{4}{13} & \frac{12}{13} & \frac{6}{13} & \frac{11}{13} & \frac{1}{13} & \frac{4}{13} & \frac{7}{13} & \frac{10}{13} & \frac{9}{13} & \frac{3}{13} & \frac{12}{13} & \frac{1}{13}
  \end{pmatrix}, \\
  G'' = 
  \begin{pmatrix}
    0 & 0 & \frac{7}{13} & \frac{1}{13}
  \end{pmatrix}; \;
  B' = 
  \begin{pmatrix}
    \frac{6}{13}
  \end{pmatrix}, \;
  B'' = 
  \begin{pmatrix}
    \frac{7}{13}
  \end{pmatrix}; \;
  Q' =
  \begin{pmatrix}
    \frac{6}{13}
  \end{pmatrix}; \;
  Q'' =
  \begin{pmatrix}
    \frac{20}{13}
  \end{pmatrix}.
\end{gather*}


\subsection{Family \textnumero2.12}\label{subsection:02-12}

The pencil \(\mathcal{S}\) is defined by the equation
\begin{gather*}
  X^{2} Y^{2} + X^{2} Y Z + X Y^{2} Z + X Y Z^{2} + 2 X^{2} Y T + 2 X Y^{2} T + \\ X^{2} T^{2} + 2 X Y T^{2} + Y^{2} T^{2} + 2 X Z T^{2} + 2 Y Z T^{2} + Z^{2} T^{2} = \lambda X Y Z T.
\end{gather*}
Members \(\mathcal{S}_{\lambda}\) of the pencil are irreducible for any \(\lambda \in \mathbb{P}^1\) except
\(\mathcal{S}_{\infty} = S_{(X)} + S_{(Y)} + S_{(Z)} + S_{(T)}\).
The base locus of the pencil \(\mathcal{S}\) consists of the following curves:
\begin{gather*}
  C_{1} = C_{(X, T)}, \;
  C_{2} = C_{(Y, T)}, \;
  C_{3} = C_{(X, Y + Z)}, \;
  C_{4} = C_{(Y, X + Z)}, \\
  C_{5} = C_{(T, X + Z)}, \;
  C_{6} = C_{(T, Y + Z)}, \;
  C_{7} = C_{(Z, X Y + T (X + Y))}.
\end{gather*}
Their linear equivalence classes on the generic member \(\mathcal{S}_{\Bbbk}\) of the pencil satisfy the following relations:
\[
  \begin{pmatrix}
    2 [C_4] \\ [C_6] \\ 2 [C_7] \\ [H_{\mathcal{S}}]
  \end{pmatrix} =
  \begin{pmatrix}
    2 & -2 & 2 & 0 \\
    1 & -1 & 2 & -1 \\
    2 & 0 & 2 & 0 \\
    2 & 0 & 2 & 0
  \end{pmatrix} \cdot
  \begin{pmatrix}
    [C_1] \\ [C_2] \\ [C_3] \\ [C_5]
  \end{pmatrix}.
\]

Put \(\mu (\mu - 1) = (\lambda + 2)^{-1}\). For a general choice of \(\lambda \in \mathbb{C}\) the surface \(\mathcal{S}_{\lambda}\) has the following singularities:
\begin{itemize}\setlength{\itemindent}{2cm}
\item[\(P_{1} = P_{(X, Y, Z)}\):] type \(\mathbb{D}_4\) with the quadratic term \((X + Y + Z)^2\);
\item[\(P_{2} = P_{(X, Y, T)}\):] type \(\mathbb{A}_1\) with the quadratic term \(X Y + T^2\);
\item[\(P_{3} = P_{(X, Z, T)}\):] type \(\mathbb{A}_1\) with the quadratic term \(X Z + (X + T)^2\);
\item[\(P_{4} = P_{(Y, Z, T)}\):] type \(\mathbb{A}_1\) with the quadratic term \(Y Z + (X + T)^2\);
\item[\(P_{5} = P_{(X, T, Y + Z)}\):] type \(\mathbb{A}_3\) with the quadratic term \(X \cdot (Y + Z - (\lambda + 2) T)\);
\item[\(P_{6} = P_{(Y, T, X + Z)}\):] type \(\mathbb{A}_3\) with the quadratic term \(Y \cdot (X + Z - (\lambda + 2) T)\);
\item[\(P_{7} = P_{(\mu X + (\mu - 1) Y, Z, (\mu - 1) Y - T)}\):] type \(\mathbb{A}_1\);
\item[\(P_{8} = P_{((\mu - 1) X + \mu Y, Z, (\mu - 1) X - T)}\):] type \(\mathbb{A}_1\).
\end{itemize}

The only non-trivial Galois orbit on the lattice \(L_{\lambda}\) is \(E_7^1 + E_8^1\).

The intersection matrix on the lattice \(L_{\lambda}\) is represented by
\begin{table}[H]
  \setlength{\tabcolsep}{4pt}
  \begin{tabular}{|c||cccc|c|c|c|ccc|ccc|c|c|cccccc|}
    \hline
    \(\bullet\) & \(E_1^1\) & \(E_1^2\) & \(E_1^3\) & \(E_1^4\) & \(E_2^1\) & \(E_3^1\) & \(E_4^1\) & \(E_5^1\) & \(E_5^2\) & \(E_5^3\) & \(E_6^1\) & \(E_6^2\) & \(E_6^3\) & \(E_7^1\) & \(E_8^1\) & \(\widetilde{C_{1}}\) & \(\widetilde{C_{2}}\) & \(\widetilde{C_{3}}\) & \(\widetilde{C_{4}}\) & \(\widetilde{C_{5}}\) & \(\widetilde{C_{7}}\)  \\
    \hline
    \hline
    \(\widetilde{C_{1}}\) & \(0\) & \(0\) & \(0\) & \(0\) & \(1\) & \(1\) & \(0\) & \(1\) & \(0\) & \(0\) & \(0\) & \(0\) & \(0\) & \(0\) & \(0\) & \(-2\) & \(0\) & \(0\) & \(0\) & \(0\) & \(0\) \\
    \(\widetilde{C_{2}}\) & \(0\) & \(0\) & \(0\) & \(0\) & \(1\) & \(0\) & \(1\) & \(0\) & \(0\) & \(0\) & \(1\) & \(0\) & \(0\) & \(0\) & \(0\) & \(0\) & \(-2\) & \(0\) & \(0\) & \(0\) & \(0\) \\
    \(\widetilde{C_{3}}\) & \(1\) & \(0\) & \(0\) & \(0\) & \(0\) & \(0\) & \(0\) & \(1\) & \(0\) & \(0\) & \(0\) & \(0\) & \(0\) & \(0\) & \(0\) & \(0\) & \(0\) & \(-2\) & \(0\) & \(0\) & \(0\) \\
    \(\widetilde{C_{4}}\) & \(0\) & \(0\) & \(1\) & \(0\) & \(0\) & \(0\) & \(0\) & \(0\) & \(0\) & \(0\) & \(1\) & \(0\) & \(0\) & \(0\) & \(0\) & \(0\) & \(0\) & \(0\) & \(-2\) & \(0\) & \(0\) \\
    \(\widetilde{C_{5}}\) & \(0\) & \(0\) & \(0\) & \(0\) & \(0\) & \(1\) & \(0\) & \(0\) & \(0\) & \(0\) & \(0\) & \(0\) & \(1\) & \(0\) & \(0\) & \(0\) & \(0\) & \(0\) & \(0\) & \(-2\) & \(0\) \\
    \(\widetilde{C_{7}}\) & \(0\) & \(0\) & \(0\) & \(1\) & \(0\) & \(1\) & \(1\) & \(0\) & \(0\) & \(0\) & \(0\) & \(0\) & \(0\) & \(1\) & \(1\) & \(0\) & \(0\) & \(0\) & \(0\) & \(0\) & \(-2\) \\
    \hline
  \end{tabular}.
\end{table}
Note that the intersection matrix is degenerate. We choose the following integral basis of the lattice \(L_{\mathcal{S}}\):
\begin{align*}
  \begin{pmatrix}
    [E_2^1] \\ [E_3^1]
  \end{pmatrix} =
  \begin{pmatrix}
    -1 & 0 & 0 & 1 & 1 & -3 & -2 & -1 & 0 & 0 & 0 & 1 & -2 & 0 & -2 & 0 & 0 & 2 \\
    -1 & 0 & 1 & 0 & 1 & -3 & -2 & -1 & 3 & 2 & 1 & 0 & -2 & 2 & -2 & 2 & 0 & 0
  \end{pmatrix} \cdot \\
  \begin{pmatrix}
    [E_1^1] & [E_1^2] & [E_1^3] & [E_1^4] & [E_4^1] & [E_5^1] & [E_5^2] & [E_5^3] & [E_6^1] \\ [E_6^2] & [E_6^3] & [E_7^1 + E_8^1] & [\widetilde{C_{1}}] & [\widetilde{C_{2}}] & [\widetilde{C_{3}}] & [\widetilde{C_{4}}] & [\widetilde{C_{5}}] & [\widetilde{C_{7}}]
  \end{pmatrix}^T.
\end{align*}

Discriminant groups and discriminant forms of the lattices \(L_{\mathcal{S}}\) and \(H \oplus \Pic(X)\) are given by
\begin{gather*}
  G' =
  \begin{pmatrix}
    0 & 0 & 0 & 0 & 0 & 0 & 0 & 0 & 0 & 0 & 0 & \frac{1}{2} & 0 & 0 & 0 & 0 & 0 & 0 \\
    \frac{2}{5} & \frac{3}{10} & \frac{4}{5} & \frac{2}{5} & \frac{3}{5} & \frac{3}{5} & \frac{2}{5} & \frac{1}{5} & \frac{4}{5} & \frac{3}{5} & \frac{2}{5} & \frac{1}{2} & \frac{3}{10} & \frac{7}{10} & \frac{1}{2} & \frac{3}{10} & \frac{1}{5} & \frac{1}{2}
  \end{pmatrix}, \\
  G'' = 
  \begin{pmatrix}
    0 & 0 & \frac{1}{2} & 0 \\
    0 & 0 & \frac{3}{5} & \frac{1}{10}
  \end{pmatrix}; \;
  B' = 
  \begin{pmatrix}
    0 & \frac{1}{2} \\
    \frac{1}{2} & \frac{4}{5}
  \end{pmatrix}, \;
  B'' = 
  \begin{pmatrix}
    0 & \frac{1}{2} \\
    \frac{1}{2} & \frac{1}{5}
  \end{pmatrix}; \;
  \begin{pmatrix}
    Q' \\ Q''
  \end{pmatrix}
  =
  \begin{pmatrix}
    1 & \frac{9}{5} \\
    1 & \frac{1}{5}
  \end{pmatrix}.
\end{gather*}


\subsection{Family \textnumero2.13}\label{subsection:02-13}

The pencil \(\mathcal{S}\) is defined by the equation
\begin{gather*}
  X^{2} Y Z + X Y^{2} Z + X Y Z^{2} + Y^{2} Z^{2} + X^{2} Z T + 2 Y^{2} Z T + \\ X Z^{2} T + X^{2} T^{2} + 2 X Y T^{2} + Y^{2} T^{2} + 2 X Z T^{2} + X T^{3} = \lambda X Y Z T.
\end{gather*}
Members \(\mathcal{S}_{\lambda}\) of the pencil are irreducible for any \(\lambda \in \mathbb{P}^1\) except
\(\mathcal{S}_{\infty} = S_{(X)} + S_{(Y)} + S_{(Z)} + S_{(T)}\).
The base locus of the pencil \(\mathcal{S}\) consists of the following curves:
\begin{gather*}
  C_{1} = C_{(X, Y)}, \;
  C_{2} = C_{(Y, T)}, \;
  C_{3} = C_{(Z, T)}, \;
  C_{4} = C_{(X, Z + T)}, \;
  C_{5} = C_{(Y, Z + T)}, \\
  C_{6} = C_{(T, X + Y)}, \;
  C_{7} = C_{(T, X + Z)}, \;
  C_{8} = C_{(Y, X + Z + T)}, \;
  C_{9} = C_{(Z, X T + (X + Y)^2)}.
\end{gather*}
Their linear equivalence classes on the generic member \(\mathcal{S}_{\Bbbk}\) of the pencil satisfy the following relations:
\begin{gather*}
  \begin{pmatrix}
    [C_{7}] \\ [C_{8}] \\ [C_{9}] \\ [H_{\mathcal{S}}]
  \end{pmatrix} = 
  \begin{pmatrix}
    2 & -1 & -1 & 2 & 0 & -1 \\
    1 & -1 & 0 & 2 & -1 & 0 \\
    2 & 0 & -2 & 2 & 0 & 0 \\
    2 & 0 & 0 & 2 & 0 & 0
  \end{pmatrix} \cdot
  \begin{pmatrix}
    [C_{1}] & [C_{2}] & [C_{3}] & [C_{4}] & [C_{5}] & [C_{6}]
  \end{pmatrix}^T.
\end{gather*}

For a general choice of \(\lambda \in \mathbb{C}\) the surface \(\mathcal{S}_{\lambda}\) has the following singularities:
\begin{itemize}\setlength{\itemindent}{2cm}
\item[\(P_{1} = P_{(X, Y, T)}\):] type \(\mathbb{A}_1\) with the quadratic term \(Y^2 + X (Y + T)\);
\item[\(P_{2} = P_{(X, Z, T)}\):] type \(\mathbb{A}_1\) with the quadratic term \(X Z + (Z + T)^2\);
\item[\(P_{3} = P_{(Y, Z, T)}\):] type \(\mathbb{A}_1\) with the quadratic term \(T^2 + Z (Y + T)\);
\item[\(P_{4} = P_{(X, Y, Z + T)}\):] type \(\mathbb{A}_5\) with the quadratic term \((\lambda + 3) X \cdot Y\);
\item[\(P_{5} = P_{(Y, T, X + Z)}\):] type \(\mathbb{A}_1\) with the quadratic term \((Y + T) (X + Z + T) - (\lambda + 1) Y T\);
\item[\(P_{6} = P_{(Z, T, X + Y)}\):] type \(\mathbb{A}_2\) with the quadratic term \(Z \cdot (X + Y - (\lambda + 3) T)\);
\item[\(P_{7} = P_{(X, Z + T, Y - (\lambda + 3) T)}\):] type \(\mathbb{A}_1\) with the quadratic term
  \[
    (\lambda + 3) ((\lambda + 3) (Z + T)^2 - X (X + Y - Z - (\lambda + 4) T)).
  \]
\end{itemize}

Galois action on the lattice \(L_{\lambda}\) is trivial. The intersection matrix on \(L_{\lambda} = L_{\mathcal{S}}\) is represented by
\begin{table}[H]
  \begin{tabular}{|c||c|c|c|ccccc|c|cc|c|cccccc|}
    \hline
    \(\bullet\) & \(E_1^1\) & \(E_2^1\) & \(E_3^1\) & \(E_4^1\) & \(E_4^2\) & \(E_4^3\) & \(E_4^4\) & \(E_4^5\) & \(E_5^1\) & \(E_6^1\) & \(E_6^2\) & \(E_7^1\) & \(\widetilde{C_{1}}\) & \(\widetilde{C_{2}}\) & \(\widetilde{C_{3}}\) & \(\widetilde{C_{4}}\) & \(\widetilde{C_{5}}\) & \(\widetilde{C_{6}}\) \\
    \hline
    \hline
    \(\widetilde{C_{1}}\) & \(1\) & \(0\) & \(0\) & \(0\) & \(1\) & \(0\) & \(0\) & \(0\) & \(0\) & \(0\) & \(0\) & \(0\) & \(-2\) & \(0\) & \(0\) & \(0\) & \(0\) & \(0\) \\
    \(\widetilde{C_{2}}\) & \(1\) & \(0\) & \(1\) & \(0\) & \(0\) & \(0\) & \(0\) & \(0\) & \(1\) & \(0\) & \(0\) & \(0\) & \(0\) & \(-2\) & \(0\) & \(0\) & \(0\) & \(0\) \\
    \(\widetilde{C_{3}}\) & \(0\) & \(1\) & \(1\) & \(0\) & \(0\) & \(0\) & \(0\) & \(0\) & \(0\) & \(1\) & \(0\) & \(0\) & \(0\) & \(0\) & \(-2\) & \(0\) & \(0\) & \(0\) \\
    \(\widetilde{C_{4}}\) & \(0\) & \(1\) & \(0\) & \(1\) & \(0\) & \(0\) & \(0\) & \(0\) & \(0\) & \(0\) & \(0\) & \(1\) & \(0\) & \(0\) & \(0\) & \(-2\) & \(0\) & \(0\) \\
    \(\widetilde{C_{5}}\) & \(0\) & \(0\) & \(1\) & \(0\) & \(0\) & \(0\) & \(0\) & \(1\) & \(0\) & \(0\) & \(0\) & \(0\) & \(0\) & \(0\) & \(0\) & \(0\) & \(-2\) & \(0\) \\
    \(\widetilde{C_{6}}\) & \(1\) & \(0\) & \(0\) & \(0\) & \(0\) & \(0\) & \(0\) & \(0\) & \(0\) & \(0\) & \(1\) & \(0\) & \(0\) & \(0\) & \(0\) & \(0\) & \(0\) & \(-2\) \\
    \hline
  \end{tabular}.
\end{table}
Note that the intersection matrix is non-degenerate.

Discriminant groups and discriminant forms of the lattices \(L_{\mathcal{S}}\) and \(H \oplus \Pic(X)\) are given by
\begin{gather*}
  G' = 
  \begin{pmatrix}
    0 & \frac{1}{2} & 0 & \frac{1}{2} & 0 & 0 & 0 & 0 & 0 & \frac{1}{2} & 0 & 0 & \frac{1}{2} & 0 & 0 & 0 & 0 & \frac{1}{2} \\
\frac{5}{6} & \frac{3}{4} & \frac{5}{12} & \frac{1}{4} & \frac{1}{2} & \frac{1}{12} & \frac{2}{3} & \frac{1}{4} & \frac{3}{4} & \frac{5}{6} & \frac{1}{6} & 0 & \frac{2}{3} & \frac{1}{2} & \frac{1}{2} & 0 & \frac{5}{6} & \frac{1}{2}
  \end{pmatrix}, \\
  G'' = 
  \begin{pmatrix}
    0 & 0 & \frac{1}{2} & 0 \\
0 & 0 & \frac{3}{4} & \frac{1}{12}
  \end{pmatrix}; \;
  B' = 
  \begin{pmatrix}
    \frac{1}{2} & 0 \\
    0 & \frac{1}{12}
  \end{pmatrix}, \;
  B'' = 
  \begin{pmatrix}
    \frac{1}{2} & 0 \\
    0 & \frac{11}{12}
  \end{pmatrix}; \;
  \begin{pmatrix}
    Q' \\ Q''
  \end{pmatrix}
  =
  \begin{pmatrix}
    \frac{3}{2} & \frac{1}{12} \\
    \frac{1}{2} & \frac{23}{12}
  \end{pmatrix}.
\end{gather*}


\subsection{Family \textnumero2.14}\label{subsection:02-14}

The pencil \(\mathcal{S}\) is defined by the equation
\begin{gather*}
  X^{2} Y^{2} + X^{2} Y Z + X Y Z^{2} + 2 X Y^{2} T + Y Z^{2} T + Z^{3} T + 2 X Y T^{2} + \\ Y^{2} T^{2} + 3 Y Z T^{2} + 3 Z^{2} T^{2} + 2 Y T^{3} + 3 Z T^{3} + T^{4} = \lambda X Y Z T.
\end{gather*}
Members \(\mathcal{S}_{\lambda}\) of the pencil are irreducible for any \(\lambda \in \mathbb{P}^1\) except
\(\mathcal{S}_{\infty} = S_{(X)} + S_{(Y)} + S_{(Z)} + S_{(T)}\).
The base locus of the pencil \(\mathcal{S}\) consists of the following curves:
\begin{gather*}
  C_{1} = C_{(X, T)}, \;
  C_{2} = C_{(Y, T)}, \;
  C_{3} = C_{(Y, Z + T)}, \;
  C_{4} = C_{(X, Y + Z + T)}, \\
  C_{5} = C_{(X, Y T + (Z + T)^2)}, \;
  C_{6} = C_{(Z, Y (X + T) + T^2)}, \;
  C_{7} = C_{(T, X (Y + Z) + Z^2)}.
\end{gather*}
Their linear equivalence classes on the generic member \(\mathcal{S}_{\Bbbk}\) of the pencil satisfy the following relations:
\begin{gather*}
  \begin{pmatrix}
    [C_{1}] \\ [C_{2}] \\ [C_{5}] \\ [H_{\mathcal{S}}]
  \end{pmatrix} = 
  \begin{pmatrix}
    3 & 0 & 0 & -1 \\
    -3 & 0 & 2 & 0 \\
    -3 & -1 & 2 & 1 \\
    0 & 0 & 2 & 0
  \end{pmatrix} \cdot
  \begin{pmatrix}
    [C_{3}] \\ [C_{4}] \\ [C_{6}] \\ [C_{7}]
  \end{pmatrix}.
\end{gather*}

For a general choice of \(\lambda \in \mathbb{C}\) the surface \(\mathcal{S}_{\lambda}\) has the following singularities:
\begin{itemize}\setlength{\itemindent}{2cm}
\item[\(P_{1} = P_{(X, Z, T)}\):] type \(\mathbb{D}_5\) with the quadratic term \((X + T)^2\);
\item[\(P_{2} = P_{(Y, Z, T)}\):] type \(\mathbb{A}_3\) with the quadratic term \(Y \cdot (Y + Z)\);
\item[\(P_{3} = P_{(X, Y, Z + T)}\):] type \(\mathbb{A}_2\) with the quadratic term \(Y \cdot ((\lambda + 3) X + Y + Z + T)\);
\item[\(P_{4} = P_{(X, Z, Y + T)}\):] type \(\mathbb{A}_1\) with the quadratic term \((X - Y - Z - T) (X - Y - 2 Z - T) + (\lambda + 3) X Z\);
\item[\(P_{5} = P_{(Y, X - (\lambda + 3) T, Z + T)}\):] type \(\mathbb{A}_2\) with the quadratic term
  \[
    Y \cdot ((\lambda + 3) X - (\lambda + 4)^2 Y - (\lambda + 4) Z - ((\lambda + 4)^2 - (\lambda + 3)) T);
  \]
\item[\(P_{6} = P_{(Z, X - (\lambda + 3) T, (\lambda + 4) Y + T)}\):] type \(\mathbb{A}_1\).
\end{itemize}

Galois action on the lattice \(L_{\lambda}\) is trivial. The intersection matrix on \(L_{\lambda} = L_{\mathcal{S}}\) is represented by
\begin{table}[H]
  \begin{tabular}{|c||ccccc|ccc|cc|c|cc|c|cccc|}
    \hline
    \(\bullet\) & \(E_1^1\) & \(E_1^2\) & \(E_1^3\) & \(E_1^4\) & \(E_1^5\) & \(E_2^1\) & \(E_2^2\) & \(E_2^3\) & \(E_3^1\) & \(E_3^2\) & \(E_4^1\) & \(E_5^1\) & \(E_5^2\) & \(E_6^1\) & \(\widetilde{C_{3}}\) & \(\widetilde{C_{4}}\) & \(\widetilde{C_{6}}\) & \(\widetilde{C_{7}}\) \\
    \hline
    \hline
    \(\widetilde{C_{3}}\) & \(0\) & \(0\) & \(0\) & \(0\) & \(0\) & \(1\) & \(0\) & \(0\) & \(1\) & \(0\) & \(0\) & \(1\) & \(0\) & \(0\) & \(-2\) & \(0\) & \(0\) & \(0\) \\
    \(\widetilde{C_{4}}\) & \(0\) & \(0\) & \(0\) & \(0\) & \(0\) & \(0\) & \(0\) & \(0\) & \(0\) & \(1\) & \(1\) & \(0\) & \(0\) & \(0\) & \(0\) & \(-2\) & \(0\) & \(0\) \\
    \(\widetilde{C_{6}}\) & \(1\) & \(0\) & \(0\) & \(0\) & \(0\) & \(0\) & \(1\) & \(0\) & \(0\) & \(0\) & \(1\) & \(0\) & \(0\) & \(1\) & \(0\) & \(0\) & \(-2\) & \(0\) \\
    \(\widetilde{C_{7}}\) & \(0\) & \(0\) & \(0\) & \(0\) & \(1\) & \(0\) & \(0\) & \(1\) & \(0\) & \(0\) & \(0\) & \(0\) & \(0\) & \(0\) & \(0\) & \(0\) & \(0\) & \(-2\) \\
    \hline
  \end{tabular}.
\end{table}
Note that the intersection matrix is non-degenerate.

Discriminant groups and discriminant forms of the lattices \(L_{\mathcal{S}}\) and \(H \oplus \Pic(X)\) are given by
\begin{gather*}
  G' = 
  \begin{pmatrix}
    \frac{1}{5} & \frac{1}{5} & \frac{1}{5} & \frac{3}{5} & \frac{3}{5} & 0 & \frac{4}{5} & \frac{2}{5} & \frac{3}{5} & 0 & \frac{4}{5} & \frac{4}{5} & \frac{2}{5} & \frac{3}{5} & \frac{1}{5} & \frac{2}{5} & \frac{1}{5} & 0 \\
    \frac{1}{5} & \frac{3}{5} & 0 & 0 & \frac{2}{5} & \frac{1}{5} & \frac{3}{5} & \frac{1}{5} & \frac{1}{5} & \frac{3}{5} & \frac{2}{5} & \frac{1}{5} & \frac{3}{5} & \frac{2}{5} & \frac{4}{5} & 0 & \frac{4}{5} & \frac{4}{5}
  \end{pmatrix}, \\
  G'' = 
  \begin{pmatrix}
    0 & 0 & \frac{1}{5} & 0 \\
    0 & 0 & 0 & \frac{1}{5}
  \end{pmatrix}; \;
  B' = 
  \begin{pmatrix}
    0 & \frac{4}{5} \\
    \frac{4}{5} & \frac{3}{5}
  \end{pmatrix}, \;
  B'' = 
  \begin{pmatrix}
    0 & \frac{1}{5} \\
    \frac{1}{5} & \frac{2}{5}
  \end{pmatrix}; \;
  \begin{pmatrix}
    Q' \\ Q''
  \end{pmatrix}
  =
  \begin{pmatrix}
    0 & \frac{8}{5} \\
    0 & \frac{2}{5}
  \end{pmatrix}.
\end{gather*}


\subsection{Family \textnumero2.15}\label{subsection:02-15}

The pencil \(\mathcal{S}\) is defined by the equation
\begin{gather*}
  X^{2} Y Z + X Y^{2} Z + X Y Z^{2} + X^{2} Y T + X Y^{2} T + X^{2} T^{2} + \\ 2 X Y T^{2} + Y^{2} T^{2} + 2 X Z T^{2} + 2 Y Z T^{2} + Z^{2} T^{2} = \lambda X Y Z T.
\end{gather*}
Members \(\mathcal{S}_{\lambda}\) of the pencil are irreducible for any \(\lambda \in \mathbb{P}^1\) except
\begin{gather*}
  \mathcal{S}_{\infty} = S_{(X)} + S_{(Y)} + S_{(Z)} + S_{(T)}, \;
  \mathcal{S}_{- 1} = S_{(X + Y + Z)} + S_{(X Y (Z + T) + T^2 (X + Y + Z))}.
\end{gather*}
The base locus of the pencil \(\mathcal{S}\) consists of the following curves:
\begin{gather*}
  C_1 = C_{(X, T)}, \;
  C_2 = C_{(Y, T)}, \;
  C_3 = C_{(Z, T)}, \;
  C_4 = C_{(X, Y + Z)}, \\
  C_5 = C_{(Y, X + Z)}, \;
  C_6 = C_{(Z, X + Y)}, \;
  C_7 = C_{(T, X + Y + Z)}, \;
  C_8 = C_{(Z, X Y + T (X + Y))}.
\end{gather*}
Their linear equivalence classes on the generic member \(\mathcal{S}_{\Bbbk}\) of the pencil satisfy the following relations:
\[
  \begin{pmatrix}
    [C_5] + [C_8] \\ [C_6] + [C_8] \\ [C_7] \\ 2 [C_8] \\ [H_{\mathcal{S}}]
  \end{pmatrix} =
  \begin{pmatrix}
    3 & -3 & -2 & 5 \\
    2 & 0 & -1 & 2 \\
    1 & -1 & -1 & 2 \\
    4 & -4 & -4 & 8 \\
    2 & 0 & 0 & 2
  \end{pmatrix} \cdot
  \begin{pmatrix}
    [C_1] \\ [C_2] \\ [C_3] \\ [C_4]
  \end{pmatrix}
\]

For a general choice of \(\lambda \in \mathbb{C}\) the surface \(\mathcal{S}_{\lambda}\) has the following singularities:
\begin{itemize}\setlength{\itemindent}{2cm}
\item[\(P_1 = P_{(X, Y, Z)}\):] type \(\mathbb{D}_4\) with the quadratic term \((X + Y + Z)^2\);
\item[\(P_2 = P_{(X, Y, T)}\):] type \(\mathbb{A}_1\) with the quadratic term \(X Y + T^2\);
\item[\(P_3 = P_{(X, Z, T)}\):] type \(\mathbb{A}_1\) with the quadratic term \(T^2 + X (Z + T)\);
\item[\(P_4 = P_{(Y, Z, T)}\):] type \(\mathbb{A}_1\) with the quadratic term \(T^2 + Y (Z + T)\);
\item[\(P_5 = P_{(X, T, Y + Z)}\):] type \(\mathbb{A}_3\) with the quadratic term \(X \cdot (X + Y + Z - (\lambda + 1) T)\);
\item[\(P_6 = P_{(Y, T, X + Z)}\):] type \(\mathbb{A}_3\) with the quadratic term \(Y \cdot (X + Y + Z - (\lambda + 1) T)\);
\item[\(P_7 = P_{(Z, T, X + Y)}\):] type \(\mathbb{A}_1\) with the quadratic term \((Z + T) (X + Y + Z) - (\lambda + 1) Z T\).
\end{itemize}

Galois action on the lattice \(L_{\lambda}\) is trivial. The intersection matrix on \(L_{\lambda} = L_{\mathcal{S}}\) is represented by
\begin{table}[H]
  \setlength{\tabcolsep}{4pt}
  \begin{tabular}{|c||cccc|c|c|c|ccc|ccc|c|ccccccc|}
    \hline
    \(\bullet\) & \(E_1^1\) & \(E_1^2\) & \(E_1^3\) & \(E_1^4\) & \(E_2^1\) & \(E_3^1\) & \(E_4^1\) & \(E_5^1\) & \(E_5^2\) & \(E_5^3\) & \(E_6^1\) & \(E_6^2\) & \(E_6^3\) & \(E_7^1\) & \(\widetilde{C_{1}}\) & \(\widetilde{C_{2}}\) & \(\widetilde{C_{3}}\) & \(\widetilde{C_{4}}\) & \(\widetilde{C_{5}}\) & \(\widetilde{C_{6}}\) & \(\widetilde{C_{8}}\) \\
    \hline
    \hline
    \(\widetilde{C_{1}}\) & \(0\) & \(0\) & \(0\) & \(0\) & \(1\) & \(1\) & \(0\) & \(1\) & \(0\) & \(0\) & \(0\) & \(0\) & \(0\) & \(0\) & \(-2\) & \(0\) & \(0\) & \(0\) & \(0\) & \(0\) & \(0\) \\
    \(\widetilde{C_{2}}\) & \(0\) & \(0\) & \(0\) & \(0\) & \(1\) & \(0\) & \(1\) & \(0\) & \(0\) & \(0\) & \(1\) & \(0\) & \(0\) & \(0\) & \(0\) & \(-2\) & \(0\) & \(0\) & \(0\) & \(0\) & \(0\) \\
    \(\widetilde{C_{3}}\) & \(0\) & \(0\) & \(0\) & \(0\) & \(0\) & \(1\) & \(1\) & \(0\) & \(0\) & \(0\) & \(0\) & \(0\) & \(0\) & \(1\) & \(0\) & \(0\) & \(-2\) & \(0\) & \(0\) & \(0\) & \(0\) \\
    \(\widetilde{C_{4}}\) & \(1\) & \(0\) & \(0\) & \(0\) & \(0\) & \(0\) & \(0\) & \(1\) & \(0\) & \(0\) & \(0\) & \(0\) & \(0\) & \(0\) & \(0\) & \(0\) & \(0\) & \(-2\) & \(0\) & \(0\) & \(0\) \\
    \(\widetilde{C_{5}}\) & \(0\) & \(0\) & \(1\) & \(0\) & \(0\) & \(0\) & \(0\) & \(0\) & \(0\) & \(0\) & \(1\) & \(0\) & \(0\) & \(0\) & \(0\) & \(0\) & \(0\) & \(0\) & \(-2\) & \(0\) & \(0\) \\
    \(\widetilde{C_{6}}\) & \(0\) & \(0\) & \(0\) & \(1\) & \(0\) & \(0\) & \(0\) & \(0\) & \(0\) & \(0\) & \(0\) & \(0\) & \(0\) & \(1\) & \(0\) & \(0\) & \(0\) & \(0\) & \(0\) & \(-2\) & \(0\) \\
    \(\widetilde{C_{8}}\) & \(0\) & \(0\) & \(0\) & \(1\) & \(0\) & \(1\) & \(1\) & \(0\) & \(0\) & \(0\) & \(0\) & \(0\) & \(0\) & \(0\) & \(0\) & \(0\) & \(0\) & \(0\) & \(0\) & \(0\) & \(-2\) \\
    \hline
  \end{tabular}.
\end{table}
Note that the intersection matrix is degenerate. We choose the following integral basis of the lattice \(L_{\lambda}\):
\begin{align*}
  \begin{pmatrix}
    [E_6^3] \\ [\widetilde{C_{5}}] \\ [\widetilde{C_{8}}]
  \end{pmatrix} =
  \begin{pmatrix}
    5 & 6 & 3 & 4 & -2 & -1 & -3 & 3 & 2 & 1 & -3 & -2 & 0 & 0 & -4 & -2 & 4 & 2 \\
    -2 & -3 & -2 & -2 & 1 & 1 & 1 & 0 & 0 & 0 & 0 & 0 & 0 & 1 & 1 & 1 & -1 & -1 \\
    1 & 0 & 0 & -1 & 1 & 0 & -1 & 3 & 2 & 1 & 0 & 0 & -1 & 2 & 0 & -1 & 2 & -1 \\
  \end{pmatrix} \cdot \\
  \begin{pmatrix}
    [E_1^1] & [E_1^2] & [E_1^3] & [E_1^4] & [E_2^1] & [E_3^1] & [E_4^1] & [E_5^1] & [E_5^2] & \\
    [E_5^3] & [E_6^1] & [E_6^2] & [E_7^1] & [\widetilde{C_{1}}] & [\widetilde{C_{2}}] & [\widetilde{C_{3}}] & [\widetilde{C_{4}}] & [\widetilde{C_{6}}]
  \end{pmatrix}^T.
\end{align*}

Discriminant groups and discriminant forms of the lattices \(L_{\mathcal{S}}\) and \(H \oplus \Pic(X)\) are given by
\begin{gather*}
  G' = 
  \begin{pmatrix}
    \frac{1}{2} & 0 & 0 & \frac{1}{2} & \frac{1}{2} & 0 & \frac{1}{2} & \frac{1}{2} & 0 & \frac{1}{2} & 0 & 0 & \frac{1}{2} & 0 & 0 & 0 & 0 & 0 \\
    \frac{2}{3} & 0 & 0 & \frac{1}{3} & \frac{5}{6} & \frac{1}{2} & \frac{1}{6} & 0 & 0 & 0 & 0 & 0 & 0 & \frac{2}{3} & 0 & \frac{1}{3} & \frac{1}{3} & \frac{2}{3}
  \end{pmatrix}, \\
  G'' = 
  \begin{pmatrix}
    0 & 0 & \frac{1}{2} & \frac{1}{2} \\
    0 & 0 & \frac{5}{6} & 0
  \end{pmatrix}; \;
  B' = 
  \begin{pmatrix}
    \frac{1}{2} & 0 \\
    0 & \frac{5}{6}
  \end{pmatrix}, \;
  B'' = 
  \begin{pmatrix}
    \frac{1}{2} & 0 \\
    0 & \frac{1}{6}
  \end{pmatrix}; \;
  \begin{pmatrix}
    Q' \\ Q''
  \end{pmatrix}
  =
  \begin{pmatrix}
    \frac{1}{2} & \frac{11}{6} \\
    \frac{3}{2} & \frac{1}{6}
  \end{pmatrix}.
\end{gather*}


\subsection{Family \textnumero2.16}\label{subsection:02-16}

The pencil \(\mathcal{S}\) is defined by the equation
\begin{gather*}
  X^{2} Y Z + X Y Z^{2} + X^{2} Y T + X Y^{2} T + X^{2} Z T + Y^{2} Z T + X Z^{2} T + \\ Y Z^{2} T + X^{2} T^{2} + X Y T^{2} + 2 X Z T^{2} + Y Z T^{2} + Z^{2} T^{2} = \lambda X Y Z T.
\end{gather*}
Members \(\mathcal{S}_{\lambda}\) of the pencil are irreducible for any \(\lambda \in \mathbb{P}^1\) except
\begin{gather*}
  \mathcal{S}_{\infty} = S_{(X)} + S_{(Y)} + S_{(Z)} + S_{(T)}, \;
  \mathcal{S}_{- 2} = S_{(X + Z)} + S_{(Y + T)} + S_{(T (X + Y + Z) + X Z)}.
\end{gather*}
The base locus of the pencil \(\mathcal{S}\) consists of the following curves:
\begin{gather*}
  C_{1} = C_{(X, Z)}, \;
  C_{2} = C_{(X, T)}, \;
  C_{3} = C_{(Y, T)}, \;
  C_{4} = C_{(Z, T)}, \;
  C_{5} = C_{(X, Y + Z)}, \;
  C_{6} = C_{(X, Y + T)}, \\
  C_{7} = C_{(Y, X + Z)}, \;
  C_{8} = C_{(Z, X + Y)}, \;
  C_{9} = C_{(Z, Y + T)}, \;
  C_{10} = C_{(T, X + Z)}, \;
  C_{11} = C_{(Y, X Z + T (X + Z))}.
\end{gather*}
Their linear equivalence classes on the generic member \(\mathcal{S}_{\Bbbk}\) of the pencil satisfy the following relations:
\begin{gather*}
  \begin{pmatrix}
    [C_{2}] \\ [C_{4}] \\ [C_{8}] \\ [C_{9}] \\ [C_{10}] \\ [C_{11}]
  \end{pmatrix} = 
  \begin{pmatrix}
    -1 & 0 & -1 & -1 & 0 & 1 \\
    3 & -1 & 1 & 1 & 1 & -1 \\
    -4 & 3 & -1 & 0 & -1 & 1 \\
    0 & -2 & 0 & -1 & 0 & 1 \\
    -2 & 0 & 0 & 0 & -1 & 1 \\
    0 & -1 & 0 & 0 & -1 & 1
  \end{pmatrix} \cdot
  \begin{pmatrix}
    [C_{1}] \\ [C_{3}] \\ [C_{5}] \\ [C_{6}] \\ [C_{7}] \\ [H_{\mathcal{S}}]
  \end{pmatrix}.
\end{gather*}

For a general choice of \(\lambda \in \mathbb{C}\) the surface \(\mathcal{S}_{\lambda}\) has the following singularities:
\begin{itemize}\setlength{\itemindent}{2cm}
\item[\(P_{1} = P_{(X, Y, Z)}\):] type \(\mathbb{A}_3\) with the quadratic term \((X + Z) \cdot (X + Y + Z)\);
\item[\(P_{2} = P_{(X, Y, T)}\):] type \(\mathbb{A}_2\) with the quadratic term \((Y + T) \cdot (X + T)\);
\item[\(P_{3} = P_{(X, Z, T)}\):] type \(\mathbb{A}_3\) with the quadratic term \(T \cdot (X + Z)\);
\item[\(P_{4} = P_{(Y, Z, T)}\):] type \(\mathbb{A}_2\) with the quadratic term \((Y + T) \cdot (Z + T)\);
\item[\(P_{5} = P_{(X, Z, Y + T)}\):] type \(\mathbb{A}_1\) with the quadratic term \((X + Z) (Y + T) - (\lambda + 2) X Z\);
\item[\(P_{6} = P_{(Y, T, X + Z)}\):] type \(\mathbb{A}_1\) with the quadratic term \((X + Z) (Y + T) - (\lambda + 2) Y T\).
\end{itemize}

Galois action on the lattice \(L_{\lambda}\) is trivial. The intersection matrix on \(L_{\lambda} = L_{\mathcal{S}}\) is represented by
\begin{table}[H]
  \begin{tabular}{|c||ccc|cc|ccc|cc|c|c|cccccc|}
    \hline
    \(\bullet\) & \(E_1^1\) & \(E_1^2\) & \(E_1^3\) & \(E_2^1\) & \(E_2^2\) & \(E_3^1\) & \(E_3^2\) & \(E_3^3\) & \(E_4^1\) & \(E_4^2\) & \(E_5^1\) & \(E_6^1\) & \(\widetilde{C_{1}}\) & \(\widetilde{C_{3}}\) & \(\widetilde{C_{5}}\) & \(\widetilde{C_{6}}\) & \(\widetilde{C_{7}}\) & \(\widetilde{H_{\mathcal{S}}}\) \\
    \hline
    \hline
    \(\widetilde{C_{1}}\) & \(1\) & \(0\) & \(0\) & \(0\) & \(0\) & \(0\) & \(0\) & \(1\) & \(0\) & \(0\) & \(1\) & \(0\) & \(-2\) & \(0\) & \(0\) & \(0\) & \(0\) & \(1\) \\
    \(\widetilde{C_{3}}\) & \(0\) & \(0\) & \(0\) & \(0\) & \(1\) & \(0\) & \(0\) & \(0\) & \(1\) & \(0\) & \(0\) & \(1\) & \(0\) & \(-2\) & \(0\) & \(0\) & \(0\) & \(1\) \\
    \(\widetilde{C_{5}}\) & \(0\) & \(0\) & \(1\) & \(0\) & \(0\) & \(0\) & \(0\) & \(0\) & \(0\) & \(0\) & \(0\) & \(0\) & \(0\) & \(0\) & \(-2\) & \(1\) & \(0\) & \(1\) \\
    \(\widetilde{C_{6}}\) & \(0\) & \(0\) & \(0\) & \(0\) & \(1\) & \(0\) & \(0\) & \(0\) & \(0\) & \(0\) & \(1\) & \(0\) & \(0\) & \(0\) & \(1\) & \(-2\) & \(0\) & \(1\) \\
    \(\widetilde{C_{7}}\) & \(0\) & \(1\) & \(0\) & \(0\) & \(0\) & \(0\) & \(0\) & \(0\) & \(0\) & \(0\) & \(0\) & \(1\) & \(0\) & \(0\) & \(0\) & \(0\) & \(-2\) & \(1\) \\
    \(\widetilde{H_{\mathcal{S}}}\) & \(0\) & \(0\) & \(0\) & \(0\) & \(0\) & \(0\) & \(0\) & \(0\) & \(0\) & \(0\) & \(0\) & \(0\) & \(1\) & \(1\) & \(1\) & \(1\) & \(1\) & \(4\) \\
    \hline
  \end{tabular}.
\end{table}
Note that the intersection matrix is non-degenerate.

Discriminant groups and discriminant forms of the lattices \(L_{\mathcal{S}}\) and \(H \oplus \Pic(X)\) are given by
\begin{gather*}
  G' = 
  \begin{pmatrix}
    \frac{1}{2} & 0 & 0 & \frac{1}{2} & 0 & \frac{1}{2} & 0 & \frac{1}{2} & 0 & \frac{1}{2} & 0 & 0 & 0 & \frac{1}{2} & 0 & 0 & \frac{1}{2} & 0 \\
    \frac{7}{10} & \frac{1}{5} & \frac{9}{10} & 0 & 0 & \frac{3}{10} & \frac{3}{5} & \frac{9}{10} & \frac{1}{5} & \frac{3}{5} & \frac{7}{10} & \frac{3}{10} & \frac{1}{5} & \frac{4}{5} & \frac{3}{5} & \frac{1}{5} & \frac{4}{5} & \frac{1}{10}
  \end{pmatrix}, \\
  G'' = 
  \begin{pmatrix}
    0 & 0 & \frac{1}{2} & 0 \\
    0 & 0 & \frac{3}{5} & \frac{1}{10}
  \end{pmatrix}; \;
  B' = 
  \begin{pmatrix}
    \frac{1}{2} & \frac{1}{2} \\
    \frac{1}{2} & \frac{2}{5}
  \end{pmatrix}, \;
  B'' = 
  \begin{pmatrix}
    \frac{1}{2} & \frac{1}{2} \\
    \frac{1}{2} & \frac{3}{5}
  \end{pmatrix}; \;
  \begin{pmatrix}
    Q' \\ Q''
  \end{pmatrix}
  =
  \begin{pmatrix}
    \frac{1}{2} & \frac{2}{5} \\
    \frac{3}{2} & \frac{8}{5}    
  \end{pmatrix}.
\end{gather*}


\subsection{Family \textnumero2.17}\label{subsection:02-17}

The pencil \(\mathcal{S}\) is defined by the equation
\begin{gather*}
  X^{2} Y Z + X Y^{2} Z + X Y Z^{2} + X^{2} Z T + Y^{2} Z T + X Z^{2} T + Y Z^{2} T + \\ X Y T^{2} + 2 X Z T^{2} + Y Z T^{2} + Z^{2} T^{2} + Y T^{3} + Z T^{3} = \lambda X Y Z T.
\end{gather*}
Members \(\mathcal{S}_{\lambda}\) of the pencil are irreducible for any \(\lambda \in \mathbb{P}^1\) except
\begin{gather*}
  \mathcal{S}_{\infty} = S_{(X)} + S_{(Y)} + S_{(Z)} + S_{(T)}, \;
  \mathcal{S}_{- 2} = S_{(X + T)} + S_{(T^2 (Z - T) + (Y + T) (Z (X + Y + Z) + T^2))}.
\end{gather*}
The base locus of the pencil \(\mathcal{S}\) consists of the following curves:
\begin{gather*}
  C_1 = C_{(X, T)}, \;
  C_2 = C_{(Y, Z)}, \;
  C_3 = C_{(Y, T)}, \;
  C_4 = C_{(Z, T)}, \;
  C_5 = C_{(X, Y + Z)}, \\
  C_6 = C_{(Y, X + T)}, \;
  C_7 = C_{(Z, X + T)}, \;
  C_8 = C_{(Y, X + Z + T)}, \;
  C_9 = C_{(T, X + Y + Z)}, \;
  C_{10} = C_{(X, Z (Y + T) + T^2)}.
\end{gather*}
Their linear equivalence classes on the generic member \(\mathcal{S}_{\Bbbk}\) of the pencil satisfy the following relations:
\begin{gather*}
  \begin{pmatrix}
    [C_{2}] \\ [C_{3}] \\ [C_{7}] \\ [C_{9}] \\ [C_{10}]
  \end{pmatrix} = 
  \begin{pmatrix}
    2 & -2 & 0 & 1 & 0 & 0 \\
    -2 & 2 & 0 & -2 & -1 & 1 \\
    -2 & 0 & 0 & -1 & 0 & 1 \\
    1 & -3 & 0 & 2 & 1 & 0 \\
    -1 & 0 & -1 & 0 & 0 & 1
  \end{pmatrix} \cdot
  \begin{pmatrix}
    [C_{1}] & [C_{4}] & [C_{5}] & [C_{6}] & [C_{8}] & [H_{\mathcal{S}}]
  \end{pmatrix}^T.
\end{gather*}

For a general choice of \(\lambda \in \mathbb{C}\) the surface \(\mathcal{S}_{\lambda}\) has the following singularities:
\begin{itemize}\setlength{\itemindent}{2cm}
\item[\(P_{1} = P_{(X, Y, T)}\):] type \(\mathbb{A}_2\) with the quadratic term \((X + T) \cdot (Y + T)\);
\item[\(P_{2} = P_{(X, Z, T)}\):] type \(\mathbb{A}_3\) with the quadratic term \(Z \cdot (X + T)\);
\item[\(P_{3} = P_{(Y, Z, T)}\):] type \(\mathbb{A}_2\) with the quadratic term \(Z \cdot (Y + T)\);
\item[\(P_{4} = P_{(X, T, Y + Z)}\):] type \(\mathbb{A}_1\) with the quadratic term \((X + T) (X + Y + Z) - (\lambda + 2) X T\);
\item[\(P_{5} = P_{(Y, Z, X + T)}\):] type \(\mathbb{A}_2\) with the quadratic term \(Y \cdot (X + (\lambda + 2) Z + T)\);
\item[\(P_{6} = P_{(Y, T, X + Z)}\):] type \(\mathbb{A}_1\) with the quadratic term \((Y + T) (X + Y + Z + T) - (\lambda + 3) Y T\);
\item[\(P_{7} = P_{(Z, T, X + Y)}\):] type \(\mathbb{A}_1\) with the quadratic term \(Z (X + Y - \lambda T) + (Z - T)^2\).
\end{itemize}

Galois action on the lattice \(L_{\lambda}\) is trivial. The intersection matrix on \(L_{\lambda} = L_{\mathcal{S}}\) is represented by
\begin{table}[H]
  \begin{tabular}{|c||cc|ccc|cc|c|cc|c|c|cccccc|}
    \hline
    \(\bullet\) & \(E_1^1\) & \(E_1^2\) & \(E_2^1\) & \(E_2^2\) & \(E_2^3\) & \(E_3^1\) & \(E_3^2\) & \(E_4^1\) & \(E_5^1\) & \(E_5^2\) & \(E_6^1\) & \(E_7^1\) & \(\widetilde{C_{1}}\) & \(\widetilde{C_{4}}\) & \(\widetilde{C_{5}}\) & \(\widetilde{C_{6}}\) & \(\widetilde{C_{8}}\) & \(\widetilde{H_{\mathcal{S}}}\) \\
    \hline
    \hline
    \(\widetilde{C_{1}}\) & \(1\) & \(0\) & \(0\) & \(0\) & \(1\) & \(0\) & \(0\) & \(1\) & \(0\) & \(0\) & \(0\) & \(0\) & \(-2\) & \(0\) & \(0\) & \(0\) & \(0\) & \(1\) \\
    \(\widetilde{C_{4}}\) & \(0\) & \(0\) & \(1\) & \(0\) & \(0\) & \(1\) & \(0\) & \(0\) & \(0\) & \(0\) & \(0\) & \(1\) & \(0\) & \(-2\) & \(0\) & \(0\) & \(0\) & \(1\) \\
    \(\widetilde{C_{5}}\) & \(0\) & \(0\) & \(0\) & \(0\) & \(0\) & \(0\) & \(0\) & \(1\) & \(0\) & \(0\) & \(0\) & \(0\) & \(0\) & \(0\) & \(-2\) & \(0\) & \(0\) & \(1\) \\
    \(\widetilde{C_{6}}\) & \(1\) & \(0\) & \(0\) & \(0\) & \(0\) & \(0\) & \(0\) & \(0\) & \(1\) & \(0\) & \(0\) & \(0\) & \(0\) & \(0\) & \(0\) & \(-2\) & \(0\) & \(1\) \\
    \(\widetilde{C_{8}}\) & \(0\) & \(0\) & \(0\) & \(0\) & \(0\) & \(0\) & \(0\) & \(0\) & \(1\) & \(0\) & \(1\) & \(0\) & \(0\) & \(0\) & \(0\) & \(0\) & \(-2\) & \(1\) \\
    \(\widetilde{H_{\mathcal{S}}}\) & \(0\) & \(0\) & \(0\) & \(0\) & \(0\) & \(0\) & \(0\) & \(0\) & \(0\) & \(0\) & \(0\) & \(0\) & \(1\) & \(1\) & \(1\) & \(1\) & \(1\) & \(4\) \\
    \hline
  \end{tabular}.
\end{table}
Note that the intersection matrix is non-degenerate.

Discriminant groups and discriminant forms of the lattices \(L_{\mathcal{S}}\) and \(H \oplus \Pic(X)\) are given by
\begin{gather*}
  G' = 
  \begin{pmatrix}
    \frac{23}{25} & \frac{24}{25} & 0 & \frac{19}{25} & \frac{13}{25} & \frac{4}{25} & \frac{2}{25} & \frac{23}{25} & \frac{2}{25} & \frac{1}{25} & \frac{19}{25} & \frac{3}{25} & \frac{7}{25} & \frac{6}{25} & \frac{14}{25} & \frac{3}{5} & \frac{13}{25} & \frac{1}{5}
  \end{pmatrix}, \\
  G'' = 
  \begin{pmatrix}
    0 & 0 & \frac{24}{25} & \frac{1}{5}
  \end{pmatrix}; \;
  B' = 
  \begin{pmatrix}
    \frac{21}{25}
  \end{pmatrix}, \;
  B'' = 
  \begin{pmatrix}
    \frac{4}{25}
  \end{pmatrix}; \;
  Q' =
  \begin{pmatrix}
    \frac{46}{25}
  \end{pmatrix}, \;
  Q'' =
  \begin{pmatrix}
    \frac{4}{25}
  \end{pmatrix}.
\end{gather*}


\subsection{Family \textnumero2.18}\label{subsection:02-18}

The pencil \(\mathcal{S}\) is defined by the equation
\begin{gather*}
  X^{2} Y Z + X Y^{2} Z + X Y Z^{2} + Y^{2} Z T + Y Z^{2} T + X^{2} T^{2} + X Y T^{2} + \\ X Z T^{2} + Y Z T^{2} + 2 X T^{3} + Y T^{3} + Z T^{3} + T^{4} = \lambda X Y Z T.
\end{gather*}
Members \(\mathcal{S}_{\lambda}\) of the pencil are irreducible for any \(\lambda \in \mathbb{P}^1\) except
\begin{gather*}
  \mathcal{S}_{\infty} = S_{(X)} + S_{(Y)} + S_{(Z)} + S_{(T)}, \;
  \mathcal{S}_{-2} = S_{(X + T)} + S_{(X + Y + Z + T)} + S_{(Y Z + T^2)}.
\end{gather*}
The base locus of the pencil \(\mathcal{S}\) consists of the following curves:
\begin{gather*}
  C_{1} = C_{(X, T)}, \;
  C_{2} = C_{(Y, T)}, \;
  C_{3} = C_{(Z, T)}, \;
  C_{4} = C_{(Y, X + T)}, \;
  C_{5} = C_{(Z, X + T)}, \\
  C_{6} = C_{(X, Y + Z + T)}, \;
  C_{7} = C_{(Y, X + Z + T)}, \;
  C_{8} = C_{(Z, X + Y + T)}, \;
  C_{9} = C_{(T, X + Y + Z)}, \;
  C_{10} = C_{(X, Y Z + T^2)}.
\end{gather*}
Their linear equivalence classes on the generic member \(\mathcal{S}_{\Bbbk}\) of the pencil satisfy the following relations:
\begin{gather*}
  \begin{pmatrix}
    [C_{4}] \\ [C_{5}] \\ [C_{6}] \\ [C_{8}] \\ [C_{9}] \\ [C_{10}]
  \end{pmatrix} = 
  \begin{pmatrix}
    0 & -2 & 0 & -1 & 1 \\
    -2 & 2 & 0 & 1 & 0 \\
    -1 & 3 & 3 & 0 & -1 \\
    2 & -2 & -2 & -1 & 1 \\
    -1 & -1 & -1 & 0 & 1 \\
    0 & -3 & -3 & 0 & 2
  \end{pmatrix} \cdot
  \begin{pmatrix}
    [C_{1}] & [C_{2}] & [C_{3}] & [C_{7}] & [H_{\mathcal{S}}]
  \end{pmatrix}^T.
\end{gather*}

For a general choice of \(\lambda \in \mathbb{C}\) the surface \(\mathcal{S}_{\lambda}\) has the following singularities:
\begin{itemize}\setlength{\itemindent}{2cm}
\item[\(P_{1} = P_{(X, Y, T)}\):] type \(\mathbb{A}_3\) with the quadratic term \(Y \cdot (X + T)\);
\item[\(P_{2} = P_{(X, Z, T)}\):] type \(\mathbb{A}_3\) with the quadratic term \(Z \cdot (X + T)\);
\item[\(P_{3} = P_{(Y, Z, T)}\):] type \(\mathbb{A}_1\) with the quadratic term \(Y Z + T^2\);
\item[\(P_{4} = P_{(X, T, Y + Z)}\):] type \(\mathbb{A}_1\) with the quadratic term \((X + T) (X + Y + Z + T) - (\lambda + 2) X T\);
\item[\(P_{5} = P_{(Y, Z, X + T)}\):] type \(\mathbb{A}_1\) with the quadratic term \((X + T) (X + Y + Z + T) + (\lambda + 2) Y Z\);
\item[\(P_{6} = P_{(Y, T, X + Z)}\):] type \(\mathbb{A}_2\) with the quadratic term \(Y \cdot (X + Y + Z - (\lambda + 1) T)\);
\item[\(P_{7} = P_{(Z, T, X + Y)}\):] type \(\mathbb{A}_2\) with the quadratic term \(Z \cdot (X + Y + Z - (\lambda + 1) T)\).
\end{itemize}

Galois action on the lattice \(L_{\lambda}\) is trivial. The intersection matrix on \(L_{\lambda} = L_{\mathcal{S}}\) is represented by
\begin{table}[H]
  \begin{tabular}{|c||ccc|ccc|c|c|c|cc|cc|ccccc|}
    \hline
    \(\bullet\) & \(E_1^1\) & \(E_1^2\) & \(E_1^3\) & \(E_2^1\) & \(E_2^2\) & \(E_2^3\) & \(E_3^1\) & \(E_4^1\) & \(E_5^1\) & \(E_6^1\) & \(E_6^2\) & \(E_7^1\) & \(E_7^2\) & \(\widetilde{C_{1}}\) & \(\widetilde{C_{2}}\) & \(\widetilde{C_{3}}\) & \(\widetilde{C_{7}}\) & \(\widetilde{H_{\mathcal{S}}}\) \\
    \hline
    \hline
    \(\widetilde{C_{1}}\) & \(0\) & \(0\) & \(1\) & \(0\) & \(0\) & \(1\) & \(0\) & \(1\) & \(0\) & \(0\) & \(0\) & \(0\) & \(0\) & \(-2\) & \(0\) & \(0\) & \(0\) & \(1\) \\
    \(\widetilde{C_{2}}\) & \(1\) & \(0\) & \(0\) & \(0\) & \(0\) & \(0\) & \(1\) & \(0\) & \(0\) & \(1\) & \(0\) & \(0\) & \(0\) & \(0\) & \(-2\) & \(0\) & \(0\) & \(1\) \\
    \(\widetilde{C_{3}}\) & \(0\) & \(0\) & \(0\) & \(1\) & \(0\) & \(0\) & \(1\) & \(0\) & \(0\) & \(0\) & \(0\) & \(1\) & \(0\) & \(0\) & \(0\) & \(-2\) & \(0\) & \(1\) \\
    \(\widetilde{C_{7}}\) & \(0\) & \(0\) & \(0\) & \(0\) & \(0\) & \(0\) & \(0\) & \(0\) & \(1\) & \(1\) & \(0\) & \(0\) & \(0\) & \(0\) & \(0\) & \(0\) & \(-2\) & \(1\) \\
    \(\widetilde{H_{\mathcal{S}}}\) & \(0\) & \(0\) & \(0\) & \(0\) & \(0\) & \(0\) & \(0\) & \(0\) & \(0\) & \(0\) & \(0\) & \(0\) & \(0\) & \(1\) & \(1\) & \(1\) & \(1\) & \(4\) \\
    \hline
  \end{tabular}.
\end{table}
Note that the intersection matrix is non-degenerate.

Discriminant groups and discriminant forms of the lattices \(L_{\mathcal{S}}\) and \(H \oplus \Pic(X)\) are given by
\begin{gather*}
  G' = 
  \begin{pmatrix}
    0 & 0 & 0 & 0 & 0 & 0 & \frac{1}{2} & \frac{1}{2} & \frac{1}{2} & 0 & 0 & 0 & 0 & 0 & 0 & 0 & 0 & \frac{1}{2} \\
    \frac{1}{8} & \frac{1}{2} & \frac{7}{8} & \frac{5}{8} & \frac{1}{2} & \frac{3}{8} & \frac{1}{4} & \frac{1}{8} & \frac{3}{8} & 0 & \frac{1}{2} & \frac{1}{2} & \frac{1}{4} & \frac{1}{4} & \frac{3}{4} & \frac{3}{4} & \frac{3}{4} & \frac{1}{8}
  \end{pmatrix}, \\
  G'' = 
  \begin{pmatrix}
    0 & 0 & 0 & \frac{1}{2} \\
    0 & 0 & \frac{7}{8} & \frac{1}{4}
  \end{pmatrix}; \;
  B' =
  \begin{pmatrix}
    \frac{1}{2} & 0 \\
    0 & \frac{1}{8}
  \end{pmatrix}, \;
  B'' = 
  \begin{pmatrix}
    \frac{1}{2} & 0 \\
    0 & \frac{7}{8}
  \end{pmatrix}; \;
  \begin{pmatrix}
    Q' \\ Q''
  \end{pmatrix}
  =
  \begin{pmatrix}
    \frac{3}{2} & \frac{1}{8} \\
    \frac{1}{2} & \frac{15}{8}
  \end{pmatrix}.
\end{gather*}


\subsection{Family \textnumero2.19}\label{subsection:02-19}

The pencil \(\mathcal{S}\) is defined by the equation
\begin{gather*}
  X^{2} Y Z + X Y^{2} Z + X Y Z^{2} + Y^{2} Z^{2} + X^{2} Y T + X Y^{2} T + \\ X^{2} Z T + Y^{2} Z T + X Z^{2} T + X^{2} T^{2} + X Z T^{2} = \lambda X Y Z T.
\end{gather*}
Members \(\mathcal{S}_{\lambda}\) of the pencil are irreducible for any \(\lambda \in \mathbb{P}^1\) except
\begin{gather*}
  \mathcal{S}_{\infty} = S_{(X)} + S_{(Y)} + S_{(Z)} + S_{(T)}, \;
  \mathcal{S}_{- 1} = S_{(X + Z)} + S_{(Z + T)} + S_{(X (Y + T) + Y^2)}.
\end{gather*}
The base locus of the pencil \(\mathcal{S}\) consists of the following curves:
\begin{gather*}
  C_{1} = C_{(X, Y)}, \;
  C_{2} = C_{(X, Z)}, \;
  C_{3} = C_{(Y, T)}, \;
  C_{4} = C_{(Z, T)}, \;
  C_{5} = C_{(X, Z + T)}, \\
  C_{6} = C_{(Y, X + Z)}, \;
  C_{7} = C_{(Y, Z + T)}, \;
  C_{8} = C_{(T, X + Y)}, \;
  C_{9} = C_{(T, X + Z)}, \;
  C_{10} = C_{(Z, X (Y + T) + Y^2)}.
\end{gather*}
Their linear equivalence classes on the generic member \(\mathcal{S}_{\Bbbk}\) of the pencil satisfy the following relations:
\begin{gather*}
  \begin{pmatrix}
    [C_{1}] \\ [C_{2}] \\ [C_{7}] \\ [C_{8}] \\ [C_{9}] \\ [C_{10}]
  \end{pmatrix} = 
  \begin{pmatrix}
    -1 & 2 & 1 & -1 & 0 \\
    2 & -4 & -3 & 2 & 1 \\
    0 & -2 & -1 & 0 & 1 \\
    3 & -9 & -6 & 5 & 2 \\
    -4 & 8 & 6 & -5 & -1 \\
    -2 & 3 & 3 & -2 & 0
  \end{pmatrix} \cdot
  \begin{pmatrix}
    [C_{3}] \\ [C_{4}] \\ [C_{5}] \\ [C_{6}] \\ [H_{\mathcal{S}}]
  \end{pmatrix}.
\end{gather*}

For a general choice of \(\lambda \in \mathbb{C}\) the surface \(\mathcal{S}_{\lambda}\) has the following singularities:
\begin{itemize}\setlength{\itemindent}{2cm}
\item[\(P_{1} = P_{(X, Y, Z)}\):] type \(\mathbb{A}_4\) with the quadratic term \(X \cdot (X + Z)\);
\item[\(P_{2} = P_{(X, Y, T)}\):] type \(\mathbb{A}_1\) with the quadratic term \(X (Y + T) + Y^2\);
\item[\(P_{3} = P_{(X, Z, T)}\):] type \(\mathbb{A}_2\) with the quadratic term \((X + Z) \cdot (Z + T)\);
\item[\(P_{4} = P_{(Y, Z, T)}\):] type \(\mathbb{A}_2\) with the quadratic term \((Y + T) \cdot (Z + T)\);
\item[\(P_{5} = P_{(X, Y, Z + T)}\):] type \(\mathbb{A}_2\) with the quadratic term \(X \cdot ((\lambda + 1) Y - Z - T)\);
\item[\(P_{6} = P_{(Y, T, X + Z)}\):] type \(\mathbb{A}_1\) with the quadratic term \((X + Z) (Y + T) - (\lambda + 1) Y T\);
\item[\(P_{7} = P_{(Z, T, X + Y)}\):] type \(\mathbb{A}_1\) with the quadratic term \((Z + T) (X + Y - T) - (\lambda + 1) Z T\).
\end{itemize}

Galois action on the lattice \(L_{\lambda}\) is trivial. The intersection matrix on \(L_{\lambda} = L_{\mathcal{S}}\) is represented by
\begin{table}[H]
  \begin{tabular}{|c||cccc|c|cc|cc|cc|c|c|ccccc|}
    \hline
    \(\bullet\) & \(E_1^1\) & \(E_1^2\) & \(E_1^3\) & \(E_1^4\) & \(E_2^1\) & \(E_3^1\) & \(E_3^2\) & \(E_4^1\) & \(E_4^2\) & \(E_5^1\) & \(E_5^2\) & \(E_6^1\) & \(E_7^1\) & \(\widetilde{C_{3}}\) & \(\widetilde{C_{4}}\) & \(\widetilde{C_{5}}\) & \(\widetilde{C_{6}}\) & \(\widetilde{H_{\mathcal{S}}}\) \\
    \hline
    \hline
    \(\widetilde{C_{3}}\) & \(0\) & \(0\) & \(0\) & \(0\) & \(1\) & \(0\) & \(0\) & \(1\) & \(0\) & \(0\) & \(0\) & \(1\) & \(0\) & \(-2\) & \(0\) & \(0\) & \(0\) & \(1\) \\
    \(\widetilde{C_{4}}\) & \(0\) & \(0\) & \(0\) & \(0\) & \(0\) & \(0\) & \(1\) & \(0\) & \(1\) & \(0\) & \(0\) & \(0\) & \(1\) & \(0\) & \(-2\) & \(0\) & \(0\) & \(1\) \\
    \(\widetilde{C_{5}}\) & \(0\) & \(0\) & \(0\) & \(0\) & \(0\) & \(0\) & \(1\) & \(0\) & \(0\) & \(1\) & \(0\) & \(0\) & \(0\) & \(0\) & \(0\) & \(-2\) & \(0\) & \(1\) \\
    \(\widetilde{C_{6}}\) & \(0\) & \(0\) & \(0\) & \(1\) & \(0\) & \(0\) & \(0\) & \(0\) & \(0\) & \(0\) & \(0\) & \(1\) & \(0\) & \(0\) & \(0\) & \(0\) & \(-2\) & \(1\) \\
    \(\widetilde{H_{\mathcal{S}}}\) & \(0\) & \(0\) & \(0\) & \(0\) & \(0\) & \(0\) & \(0\) & \(0\) & \(0\) & \(0\) & \(0\) & \(0\) & \(0\) & \(1\) & \(1\) & \(1\) & \(1\) & \(4\) \\
    \hline
  \end{tabular}.
\end{table}
Note that the intersection matrix is non-degenerate.

Discriminant groups and discriminant forms of the lattices \(L_{\mathcal{S}}\) and \(H \oplus \Pic(X)\) are given by
\begin{gather*}
  G' = 
  \begin{pmatrix}
    \frac{13}{17} & \frac{9}{17} & \frac{5}{17} & \frac{1}{17} & \frac{15}{17} & \frac{7}{17} & \frac{14}{17} & \frac{1}{17} & \frac{6}{17} & \frac{1}{17} & \frac{9}{17} & \frac{5}{17} & \frac{14}{17} & \frac{13}{17} & \frac{11}{17} & \frac{10}{17} & \frac{14}{17} & \frac{5}{17}
  \end{pmatrix}, \\
  G'' = 
  \begin{pmatrix}
    0 & 0 & \frac{9}{17} & \frac{1}{17}
  \end{pmatrix}; \;
  B' = 
  \begin{pmatrix}
    \frac{8}{17}
  \end{pmatrix}, \;
  B'' = 
  \begin{pmatrix}
    \frac{9}{17}
  \end{pmatrix}; \;
  Q' =
  \begin{pmatrix}
    \frac{8}{17}
  \end{pmatrix}, \;
  Q'' =
  \begin{pmatrix}
    \frac{26}{17}
  \end{pmatrix}.
\end{gather*}


\subsection{Family \textnumero2.20}\label{subsection:02-20}

The pencil \(\mathcal{S}\) is defined by the equation
\begin{gather*}
  X^{2} Y Z + X Y^{2} Z + X^{2} Z^{2} + X Y Z^{2} + X Y^{2} T + Y^{2} Z T + \\ X Y T^{2} + Y^{2} T^{2} + X Z T^{2} + Y Z T^{2} + Y T^{3} = \lambda X Y Z T.
\end{gather*}
Members \(\mathcal{S}_{\lambda}\) of the pencil are irreducible for any \(\lambda \in \mathbb{P}^1\) except
\(\mathcal{S}_{\infty} = S_{(X)} + S_{(Y)} + S_{(Z)} + S_{(T)}\).
The base locus of the pencil \(\mathcal{S}\) consists of the following curves:
\begin{gather*}
  C_{1} = C_{(X, Y)}, \;
  C_{2} = C_{(X, T)}, \;
  C_{3} = C_{(Y, Z)}, \;
  C_{4} = C_{(Z, T)}, \;
  C_{5} = C_{(X, Y + T)}, \;
  C_{6} = C_{(X, Z + T)}, \\
  C_{7} = C_{(Z, X + T)}, \;
  C_{8} = C_{(Z, Y + T)}, \;
  C_{9} = C_{(T, X + Y)}, \;
  C_{10} = C_{(T, Y + Z)}, \;
  C_{11} = C_{(Y, X Z + T^2)}.
\end{gather*}
Their linear equivalence classes on the generic member \(\mathcal{S}_{\Bbbk}\) of the pencil satisfy the following relations:
\begin{gather*}
  \begin{pmatrix}
    [C_{2}] \\ [C_{4}] \\ [C_{9}] \\ [C_{11}]
  \end{pmatrix} = 
  \begin{pmatrix}
    -1 & 0 & -1 & -1 & 0 & 0 & 0 & 1 \\
    0 & -1 & 0 & 0 & -1 & -1 & 0 & 1 \\
    1 & 1 & 1 & 1 & 1 & 1 & -1 & -1 \\
    -1 & -1 & 0 & 0 & 0 & 0 & 0 & 1
  \end{pmatrix} \cdot \\
  \begin{pmatrix}
    [C_{1}] & [C_{3}] & [C_{5}] & [C_{6}] & [C_{7}] & [C_{8}] & [C_{10}] & [H_{\mathcal{S}}]
  \end{pmatrix}^T.
\end{gather*}

For a general choice of \(\lambda \in \mathbb{C}\) the surface \(\mathcal{S}_{\lambda}\) has the following singularities:
\begin{itemize}\setlength{\itemindent}{2cm}
\item[\(P_{1} = P_{(X, Y, T)}\):] type \(\mathbb{A}_4\) with the quadratic term \(X \cdot (X + Y)\);
\item[\(P_{2} = P_{(X, Z, T)}\):] type \(\mathbb{A}_2\) with the quadratic term \((X + T) \cdot (Z + T)\);
\item[\(P_{3} = P_{(Y, Z, T)}\):] type \(\mathbb{A}_4\) with the quadratic term \(Z \cdot (Y + Z)\).
\end{itemize}

Galois action on the lattice \(L_{\lambda}\) is trivial. The intersection matrix on \(L_{\lambda} = L_{\mathcal{S}}\) is represented by
\begin{table}[H]
  \begin{tabular}{|c||cccc|cc|cccc|cccccccc|}
    \hline
    \(\bullet\) & \(E_1^1\) & \(E_1^2\) & \(E_1^3\) & \(E_1^4\) & \(E_2^1\) & \(E_2^2\) & \(E_3^1\) & \(E_3^2\) & \(E_3^3\) & \(E_3^4\) & \(\widetilde{C_{1}}\) & \(\widetilde{C_{3}}\) & \(\widetilde{C_{5}}\) & \(\widetilde{C_{6}}\) & \(\widetilde{C_{7}}\) & \(\widetilde{C_{8}}\) & \(\widetilde{C_{10}}\) & \(\widetilde{H_{\mathcal{S}}}\) \\
    \hline
    \hline
    \(\widetilde{C_{1}}\) & \(0\) & \(0\) & \(1\) & \(0\) & \(0\) & \(0\) & \(0\) & \(0\) & \(0\) & \(0\) & \(-2\) & \(1\) & \(0\) & \(1\) & \(0\) & \(0\) & \(0\) & \(1\) \\
    \(\widetilde{C_{3}}\) & \(0\) & \(0\) & \(0\) & \(0\) & \(0\) & \(0\) & \(0\) & \(0\) & \(1\) & \(0\) & \(1\) & \(-2\) & \(0\) & \(0\) & \(1\) & \(0\) & \(0\) & \(1\) \\
    \(\widetilde{C_{5}}\) & \(1\) & \(0\) & \(0\) & \(0\) & \(0\) & \(0\) & \(0\) & \(0\) & \(0\) & \(0\) & \(0\) & \(0\) & \(-2\) & \(1\) & \(0\) & \(1\) & \(0\) & \(1\) \\
    \(\widetilde{C_{6}}\) & \(0\) & \(0\) & \(0\) & \(0\) & \(0\) & \(1\) & \(0\) & \(0\) & \(0\) & \(0\) & \(1\) & \(0\) & \(1\) & \(-2\) & \(0\) & \(0\) & \(0\) & \(1\) \\
    \(\widetilde{C_{7}}\) & \(0\) & \(0\) & \(0\) & \(0\) & \(1\) & \(0\) & \(0\) & \(0\) & \(0\) & \(0\) & \(0\) & \(1\) & \(0\) & \(0\) & \(-2\) & \(1\) & \(0\) & \(1\) \\
    \(\widetilde{C_{8}}\) & \(0\) & \(0\) & \(0\) & \(0\) & \(0\) & \(0\) & \(1\) & \(0\) & \(0\) & \(0\) & \(0\) & \(0\) & \(1\) & \(0\) & \(1\) & \(-2\) & \(0\) & \(1\) \\
    \(\widetilde{C_{10}}\) & \(0\) & \(0\) & \(0\) & \(0\) & \(0\) & \(0\) & \(0\) & \(0\) & \(0\) & \(1\) & \(0\) & \(0\) & \(0\) & \(0\) & \(0\) & \(0\) & \(-2\) & \(1\) \\
    \(\widetilde{H_{\mathcal{S}}}\) & \(0\) & \(0\) & \(0\) & \(0\) & \(0\) & \(0\) & \(0\) & \(0\) & \(0\) & \(0\) & \(1\) & \(1\) & \(1\) & \(1\) & \(1\) & \(1\) & \(1\) & \(4\) \\
    \hline
  \end{tabular}.
\end{table}
Note that the intersection matrix is non-degenerate.

Discriminant groups and discriminant forms of the lattices \(L_{\mathcal{S}}\) and \(H \oplus \Pic(X)\) are given by
\begin{gather*}
  G' = 
  \begin{pmatrix}
    \frac{16}{29} & \frac{7}{29} & \frac{27}{29} & \frac{28}{29} & \frac{12}{29} & \frac{26}{29} & \frac{22}{29} & \frac{2}{29} & \frac{11}{29} & \frac{1}{29} & \frac{19}{29} & \frac{19}{29} & \frac{25}{29} & \frac{11}{29} & \frac{27}{29} & \frac{13}{29} & \frac{20}{29} & \frac{10}{29}
  \end{pmatrix}, \\
  G'' = 
  \begin{pmatrix}
    0 & 0 & \frac{16}{29} & \frac{1}{29}
  \end{pmatrix}; \;
  B' = 
  \begin{pmatrix}
    \frac{14}{29}
  \end{pmatrix}, \;
  B'' = 
  \begin{pmatrix}
    \frac{15}{29}
  \end{pmatrix}; \;
  Q' =
  \begin{pmatrix}
    \frac{14}{29}
  \end{pmatrix}, \;
  Q'' =
  \begin{pmatrix}
    \frac{44}{29}
  \end{pmatrix}.
\end{gather*}


\subsection{Family \textnumero2.21}\label{subsection:02-21}

The pencil \(\mathcal{S}\) is defined by the equation
\begin{gather*}
  X^{2} Y Z + X Y^{2} Z + X Y Z^{2} + X^{2} Y T + X Y^{2} T + X^{2} Z T + Y^{2} Z T + X Z^{2} T + X Y T^{2} + Y Z T^{2} = \lambda X Y Z T.
\end{gather*}
Members \(\mathcal{S}_{\lambda}\) of the pencil are irreducible for any \(\lambda \in \mathbb{P}^1\) except
\(\mathcal{S}_{\infty} = S_{(X)} + S_{(Y)} + S_{(Z)} + S_{(T)}\).
The base locus of the pencil \(\mathcal{S}\) consists of the following curves:
\begin{gather*}
  C_{1} = C_{(X, Y)}, \;
  C_{2} = C_{(X, Z)}, \;
  C_{3} = C_{(X, T)}, \;
  C_{4} = C_{(Y, Z)}, \;
  C_{5} = C_{(Y, T)}, \\
  C_{6} = C_{(Z, T)}, \;
  C_{7} = C_{(X, Y + T)}, \;
  C_{8} = C_{(Y, X + Z)}, \;
  C_{9} = C_{(Z, X + Y + T)}, \;
  C_{10} = C_{(T, X + Y + Z)}.
\end{gather*}
Their linear equivalence classes on the generic member \(\mathcal{S}_{\Bbbk}\) of the pencil satisfy the following relations:
\begin{gather*}
  \begin{pmatrix}
    [C_{2}] \\ [C_{5}] \\ [C_{7}] \\ [C_{8}]
  \end{pmatrix} = 
  \begin{pmatrix}
    0 & 0 & -1 & -1 & -1 & 0 & 1 \\
    0 & -1 & 0 & -1 & 0 & -1 & 1 \\
    -1 & -1 & 1 & 1 & 1 & 0 & 0 \\
    -1 & 1 & -1 & 1 & 0 & 1 & 0
  \end{pmatrix} \cdot
  \begin{pmatrix}
    [C_{1}] & [C_{3}] & [C_{4}] & [C_{6}] & [C_{9}] & [C_{10}] & [H_{\mathcal{S}}]
  \end{pmatrix}^T.
\end{gather*}

For a general choice of \(\lambda \in \mathbb{C}\) the surface \(\mathcal{S}_{\lambda}\) has the following singularities:
\begin{itemize}\setlength{\itemindent}{2cm}
\item[\(P_{1} = P_{(X, Y, Z)}\):] type \(\mathbb{A}_3\) with the quadratic term \(Y \cdot (X + Z)\);
\item[\(P_{2} = P_{(X, Y, T)}\):] type \(\mathbb{A}_3\) with the quadratic term \(X \cdot (Y + T)\);
\item[\(P_{3} = P_{(X, Z, T)}\):] type \(\mathbb{A}_1\) with the quadratic term \(X (Z + T) + Z T\);
\item[\(P_{4} = P_{(Y, Z, T)}\):] type \(\mathbb{A}_1\) with the quadratic term \(Y (Z + T) + Z T\);
\item[\(P_{5} = P_{(X, Z, Y + T)}\):] type \(\mathbb{A}_1\) with the quadratic term \((X + Z) (X + Y + T) - (\lambda + 2) X Z\);
\item[\(P_{6} = P_{(Y, T, X + Z)}\):] type \(\mathbb{A}_1\) with the quadratic term \((Y + T) (X + Y + Z) - (\lambda + 2) Y T\);
\item[\(P_{7} = P_{(Z, T, X + Y)}\):] type \(\mathbb{A}_1\) with the quadratic term \((Z + T) (X + Y + Z + T) - (\lambda + 4) Z T\).
\end{itemize}

Galois action on the lattice \(L_{\lambda}\) is trivial. The intersection matrix on \(L_{\lambda} = L_{\mathcal{S}}\) is represented by
\begin{table}[H]
  \begin{tabular}{|c||ccc|ccc|c|c|c|c|c|ccccccc|}
    \hline
    \(\bullet\) & \(E_1^1\) & \(E_1^2\) & \(E_1^3\) & \(E_2^1\) & \(E_2^2\) & \(E_2^3\) & \(E_3^1\) & \(E_4^1\) & \(E_5^1\) & \(E_6^1\) & \(E_7^1\) & \(\widetilde{C_{1}}\) & \(\widetilde{C_{3}}\) & \(\widetilde{C_{4}}\) & \(\widetilde{C_{6}}\) & \(\widetilde{C_{9}}\) & \(\widetilde{C_{10}}\) & \(\widetilde{H_{\mathcal{S}}}\) \\
    \hline
    \hline
    \(\widetilde{C_{1}}\) & \(1\) & \(0\) & \(0\) & \(1\) & \(0\) & \(0\) & \(0\) & \(0\) & \(0\) & \(0\) & \(0\) & \(-2\) & \(0\) & \(0\) & \(0\) & \(0\) & \(0\) & \(1\) \\
    \(\widetilde{C_{3}}\) & \(0\) & \(0\) & \(0\) & \(1\) & \(0\) & \(0\) & \(1\) & \(0\) & \(0\) & \(0\) & \(0\) & \(0\) & \(-2\) & \(0\) & \(0\) & \(0\) & \(1\) & \(1\) \\
    \(\widetilde{C_{4}}\) & \(1\) & \(0\) & \(0\) & \(0\) & \(0\) & \(0\) & \(0\) & \(1\) & \(0\) & \(0\) & \(0\) & \(0\) & \(0\) & \(-2\) & \(0\) & \(1\) & \(0\) & \(1\) \\
    \(\widetilde{C_{6}}\) & \(0\) & \(0\) & \(0\) & \(0\) & \(0\) & \(0\) & \(1\) & \(1\) & \(0\) & \(0\) & \(1\) & \(0\) & \(0\) & \(0\) & \(-2\) & \(0\) & \(0\) & \(1\) \\
    \(\widetilde{C_{9}}\) & \(0\) & \(0\) & \(0\) & \(0\) & \(0\) & \(0\) & \(0\) & \(0\) & \(1\) & \(0\) & \(1\) & \(0\) & \(0\) & \(1\) & \(0\) & \(-2\) & \(0\) & \(1\) \\
    \(\widetilde{C_{10}}\) & \(0\) & \(0\) & \(0\) & \(0\) & \(0\) & \(0\) & \(0\) & \(0\) & \(0\) & \(1\) & \(1\) & \(0\) & \(1\) & \(0\) & \(0\) & \(0\) & \(-2\) & \(1\) \\
    \(\widetilde{H_{\mathcal{S}}}\) & \(0\) & \(0\) & \(0\) & \(0\) & \(0\) & \(0\) & \(0\) & \(0\) & \(0\) & \(0\) & \(0\) & \(1\) & \(1\) & \(1\) & \(1\) & \(1\) & \(1\) & \(4\) \\
    \hline
  \end{tabular}.
\end{table}
Note that the intersection matrix is non-degenerate.

Discriminant groups and discriminant forms of the lattices \(L_{\mathcal{S}}\) and \(H \oplus \Pic(X)\) are given by
\begin{gather*}
  G' = 
  \begin{pmatrix}
    0 & 0 & 0 & \frac{1}{2} & 0 & \frac{1}{2} & 0 & \frac{1}{2} & \frac{1}{2} & \frac{1}{2} & 0 & 0 & 0 & 0 & 0 & 0 & 0 & \frac{1}{2} \\
    \frac{1}{2} & 0 & \frac{1}{2} & \frac{1}{2} & 0 & \frac{1}{2} & \frac{1}{2} & \frac{1}{2} & \frac{4}{7} & \frac{4}{7} & 0 & \frac{5}{7} & \frac{2}{7} & \frac{2}{7} & \frac{5}{7} & \frac{1}{7} & \frac{1}{7} & \frac{3}{7}
  \end{pmatrix}, \\
  G'' = 
  \begin{pmatrix}
    0 & 0 & \frac{1}{2} & 0 \\
    0 & 0 & \frac{9}{14} & \frac{1}{14}
  \end{pmatrix}; \;
  B' = 
  \begin{pmatrix}
    \frac{1}{2} & \frac{1}{2} \\
    \frac{1}{2} & \frac{3}{7}
  \end{pmatrix}, \;
  B'' =
  \begin{pmatrix}
    \frac{1}{2} & \frac{1}{2} \\
    \frac{1}{2} & \frac{4}{7}
  \end{pmatrix}; \;
  \begin{pmatrix}
    Q' \\ Q''
  \end{pmatrix}
  =
  \begin{pmatrix}
    \frac{1}{2} & \frac{3}{7} \\
    \frac{3}{2} & \frac{11}{7}
  \end{pmatrix}.
\end{gather*}


\subsection{Family \textnumero2.22}\label{subsection:02-22}

The pencil \(\mathcal{S}\) is defined by the equation
\begin{gather*}
  X^{2} Y Z + X Y^{2} Z + X^{2} Z^{2} + X Y Z^{2} + Y^{2} Z T + X Y T^{2} + X Z T^{2} + Y Z T^{2} + Y T^{3} = \lambda X Y Z T.
\end{gather*}
Members \(\mathcal{S}_{\lambda}\) of the pencil are irreducible for any \(\lambda \in \mathbb{P}^1\) except
\(\mathcal{S}_{\infty} = S_{(X)} + S_{(Y)} + S_{(Z)} + S_{(T)}\).
The base locus of the pencil \(\mathcal{S}\) consists of the following curves:
\begin{gather*}
  C_{1} = C_{(X, Y)}, \;
  C_{2} = C_{(X, T)}, \;
  C_{3} = C_{(Y, Z)}, \;
  C_{4} = C_{(Z, T)}, \;
  C_{5} = C_{(Z, X + T)}, \\
  C_{6} = C_{(T, X + Y)}, \;
  C_{7} = C_{(T, Y + Z)}, \;
  C_{8} = C_{(X, Z (Y + T) + T^2)}, \;
  C_{9} = C_{(Y, X Z + T^2)}.
\end{gather*}
Their linear equivalence classes on the generic member \(\mathcal{S}_{\Bbbk}\) of the pencil satisfy the following relations:
\begin{gather*}
  \begin{pmatrix}
    [C_{2}] \\ [C_{3}] \\ [C_{8}] \\ [C_{9}]
  \end{pmatrix} = 
  \begin{pmatrix}
    0 & -1 & 0 & -1 & -1 & 1 \\
    0 & -2 & -1 & 0 & 0 & 1 \\
    -1 & 1 & 0 & 1 & 1 & 0 \\
    -1 & 2 & 1 & 0 & 0 & 0
  \end{pmatrix} \cdot
  \begin{pmatrix}
    [C_{1}] & [C_{4}] & [C_{5}] & [C_{6}] & [C_{7}] & [H_{\mathcal{S}}]
  \end{pmatrix}^T.
\end{gather*}

For a general choice of \(\lambda \in \mathbb{C}\) the surface \(\mathcal{S}_{\lambda}\) has the following singularities:
\begin{itemize}\setlength{\itemindent}{2cm}
\item[\(P_{1} = P_{(X, Y, T)}\):] type \(\mathbb{A}_4\) with the quadratic term \(X \cdot (X + Y)\);
\item[\(P_{2} = P_{(X, Z, T)}\):] type \(\mathbb{A}_3\) with the quadratic term \(Z \cdot (X + T)\);
\item[\(P_{3} = P_{(Y, Z, T)}\):] type \(\mathbb{A}_4\) with the quadratic term \(Z \cdot (Y + Z)\);
\item[\(P_{4} = P_{(Z, T, X + Y)}\):] type \(\mathbb{A}_1\) with the quadratic term \(Z (X + Y - (\lambda + 1) T) + T^2\).
\end{itemize}

Galois action on the lattice \(L_{\lambda}\) is trivial. The intersection matrix on \(L_{\lambda} = L_{\mathcal{S}}\) is represented by
\begin{table}[H]
  \begin{tabular}{|c||cccc|ccc|cccc|c|cccccc|}
    \hline
    \(\bullet\) & \(E_1^1\) & \(E_1^2\) & \(E_1^3\) & \(E_1^4\) & \(E_2^1\) & \(E_2^2\) & \(E_2^3\) & \(E_3^1\) & \(E_3^2\) & \(E_3^3\) & \(E_3^4\) & \(E_4^1\) & \(\widetilde{C_{1}}\) & \(\widetilde{C_{4}}\) & \(\widetilde{C_{5}}\) & \(\widetilde{C_{6}}\) & \(\widetilde{C_{7}}\) & \(\widetilde{H_{\mathcal{S}}}\) \\
    \hline
    \hline
    \(\widetilde{C_{1}}\) & \(0\) & \(0\) & \(1\) & \(0\) & \(0\) & \(0\) & \(0\) & \(0\) & \(0\) & \(0\) & \(0\) & \(0\) & \(-2\) & \(0\) & \(0\) & \(0\) & \(0\) & \(1\) \\
    \(\widetilde{C_{4}}\) & \(0\) & \(0\) & \(0\) & \(0\) & \(1\) & \(0\) & \(0\) & \(1\) & \(0\) & \(0\) & \(0\) & \(1\) & \(0\) & \(-2\) & \(0\) & \(0\) & \(0\) & \(1\) \\
    \(\widetilde{C_{5}}\) & \(0\) & \(0\) & \(0\) & \(0\) & \(0\) & \(1\) & \(0\) & \(0\) & \(0\) & \(0\) & \(0\) & \(0\) & \(0\) & \(0\) & \(-2\) & \(0\) & \(0\) & \(1\) \\
    \(\widetilde{C_{6}}\) & \(0\) & \(0\) & \(0\) & \(1\) & \(0\) & \(0\) & \(0\) & \(0\) & \(0\) & \(0\) & \(0\) & \(1\) & \(0\) & \(0\) & \(0\) & \(-2\) & \(1\) & \(1\) \\
    \(\widetilde{C_{7}}\) & \(0\) & \(0\) & \(0\) & \(0\) & \(0\) & \(0\) & \(0\) & \(0\) & \(0\) & \(0\) & \(1\) & \(0\) & \(0\) & \(0\) & \(0\) & \(1\) & \(-2\) & \(1\) \\
    \(\widetilde{H_{\mathcal{S}}}\) & \(0\) & \(0\) & \(0\) & \(0\) & \(0\) & \(0\) & \(0\) & \(0\) & \(0\) & \(0\) & \(0\) & \(0\) & \(1\) & \(1\) & \(1\) & \(1\) & \(1\) & \(4\) \\
    \hline
  \end{tabular}.
\end{table}
Note that the intersection matrix is non-degenerate.

Discriminant groups and discriminant forms of the lattices \(L_{\mathcal{S}}\) and \(H \oplus \Pic(X)\) are given by
\begin{gather*}
  G' = 
  \begin{pmatrix}
    0 & 0 & 0 & \frac{1}{2} & \frac{1}{2} & 0 & \frac{1}{2} & \frac{1}{2} & 0 & \frac{1}{2} & 0 & 0 & \frac{1}{2} & 0 & 0 & 0 & \frac{1}{2} & 0 \\
    \frac{2}{3} & \frac{1}{3} & 0 & \frac{2}{3} & \frac{1}{3} & \frac{5}{6} & \frac{11}{12} & \frac{3}{4} & \frac{2}{3} & \frac{7}{12} & \frac{1}{2} & \frac{7}{12} & 0 & \frac{5}{6} & \frac{5}{12} & \frac{1}{3} & \frac{5}{12} & 0
  \end{pmatrix}, \\
  G'' = 
  \begin{pmatrix}
    0 & 0 & \frac{1}{2} & 0 \\
    0 & 0 & \frac{7}{12} & \frac{1}{12}
  \end{pmatrix}; \;
  B' = 
  \begin{pmatrix}
    \frac{1}{2} & \frac{1}{2} \\
    \frac{1}{2} & \frac{5}{12}
  \end{pmatrix}, \;
  B'' =
  \begin{pmatrix}
    \frac{1}{2} & \frac{1}{2} \\
    \frac{1}{2} & \frac{7}{12}
  \end{pmatrix}; \;
  \begin{pmatrix}
    Q' \\ Q''
  \end{pmatrix}
  =
  \begin{pmatrix}
    \frac{1}{2} & \frac{5}{12} \\
    \frac{3}{2} & \frac{19}{12}
  \end{pmatrix}.
\end{gather*}


\subsection{Family \textnumero2.23}\label{subsection:02-23}

The pencil \(\mathcal{S}\) is defined by the equation
\begin{gather*}
  X^{2} Y Z + X Y^{2} Z + X Y Z^{2} + X^{2} Y T + X Y^{2} T + X Z^{2} T + Y Z^{2} T + X Z T^{2} + Y Z T^{2} = \lambda X Y Z T.
\end{gather*}
Members \(\mathcal{S}_{\lambda}\) of the pencil are irreducible for any \(\lambda \in \mathbb{P}^1\) except
\begin{gather*}
  \mathcal{S}_{\infty} = S_{(X)} + S_{(Y)} + S_{(Z)} + S_{(T)}, \;
  \mathcal{S}_{- 1} = S_{(Z + T)} + S_{((X + Y) (X Y + Z T) + X Y Z)}.
\end{gather*}
The base locus of the pencil \(\mathcal{S}\) consists of the following curves:
\begin{gather*}
  C_{1} = C_{(X, Y)}, \;
  C_{2} = C_{(X, Z)}, \;
  C_{3} = C_{(X, T)}, \;
  C_{4} = C_{(Y, Z)}, \;
  C_{5} = C_{(Y, T)}, \\
  C_{6} = C_{(Z, T)}, \;
  C_{7} = C_{(X, Z + T)}, \;
  C_{8} = C_{(Y, Z + T)}, \;
  C_{9} = C_{(Z, X + Y)}, \;
  C_{10} = C_{(T, X + Y + Z)}.
\end{gather*}
Their linear equivalence classes on the generic member \(\mathcal{S}_{\Bbbk}\) of the pencil satisfy the following relations:
\begin{gather*}
  \begin{pmatrix}
    [C_{1}] \\ [C_{4}] \\ [C_{6}] \\ [C_{7}] \\ [C_{10}]
  \end{pmatrix} = 
  \begin{pmatrix}
    3 & 1 & -2 & -3 & 2 & 0 \\
    -3 & -1 & 1 & 2 & -2 & 1 \\
    2 & 1 & -1 & -2 & 1 & 0 \\
    -4 & -2 & 2 & 3 & -2 & 1 \\
    -2 & -2 & 0 & 2 & -1 & 1
  \end{pmatrix} \cdot
  \begin{pmatrix}
    [C_{2}] & [C_{3}] & [C_{5}] & [C_{8}] & [C_{9}] & [H_{\mathcal{S}}]
  \end{pmatrix}^T.
\end{gather*}

For a general choice of \(\lambda \in \mathbb{C}\) the surface \(\mathcal{S}_{\lambda}\) has the following singularities:
\begin{itemize}\setlength{\itemindent}{2cm}
\item[\(P_{1} = P_{(X, Y, Z)}\):] type \(\mathbb{A}_3\) with the quadratic term \(Z \cdot (X + Y)\);
\item[\(P_{2} = P_{(X, Y, T)}\):] type \(\mathbb{A}_1\) with the quadratic term \(X Y + T (X + Y)\);
\item[\(P_{3} = P_{(X, Z, T)}\):] type \(\mathbb{A}_3\) with the quadratic term \(X \cdot (Z + T)\);
\item[\(P_{4} = P_{(Y, Z, T)}\):] type \(\mathbb{A}_3\) with the quadratic term \(Y \cdot (Z + T)\);
\item[\(P_{5} = P_{(X, Y, Z + T)}\):] type \(\mathbb{A}_1\) with the quadratic term \((X + Y) (Z + T) - (\lambda + 1) X Y\);
\item[\(P_{6} = P_{(Z, T, X + Y)}\):] type \(\mathbb{A}_1\) with the quadratic term \((Z + T) (X + Y + Z) - (\lambda + 1) Z T\).
\end{itemize}

Galois action on the lattice \(L_{\lambda}\) is trivial. The intersection matrix on \(L_{\lambda} = L_{\mathcal{S}}\) is represented by
\begin{table}[H]
  \begin{tabular}{|c||ccc|c|ccc|ccc|c|c|cccccc|}
    \hline
    \(\bullet\) & \(E_1^1\) & \(E_1^2\) & \(E_1^3\) & \(E_2^1\) & \(E_3^1\) & \(E_3^2\) & \(E_3^3\) & \(E_4^1\) & \(E_4^2\) & \(E_4^3\) & \(E_5^1\) & \(E_6^1\) & \(\widetilde{C_{2}}\) & \(\widetilde{C_{3}}\) & \(\widetilde{C_{5}}\) & \(\widetilde{C_{8}}\) & \(\widetilde{C_{9}}\) & \(\widetilde{H_{\mathcal{S}}}\) \\
    \hline
    \hline
    \(\widetilde{C_{2}}\) & \(1\) & \(0\) & \(0\) & \(0\) & \(1\) & \(0\) & \(0\) & \(0\) & \(0\) & \(0\) & \(0\) & \(0\) & \(-2\) & \(0\) & \(0\) & \(0\) & \(0\) & \(1\) \\
    \(\widetilde{C_{3}}\) & \(0\) & \(0\) & \(0\) & \(1\) & \(1\) & \(0\) & \(0\) & \(0\) & \(0\) & \(0\) & \(0\) & \(0\) & \(0\) & \(-2\) & \(0\) & \(0\) & \(0\) & \(1\) \\
    \(\widetilde{C_{5}}\) & \(0\) & \(0\) & \(0\) & \(1\) & \(0\) & \(0\) & \(0\) & \(1\) & \(0\) & \(0\) & \(0\) & \(0\) & \(0\) & \(0\) & \(-2\) & \(0\) & \(0\) & \(1\) \\
    \(\widetilde{C_{8}}\) & \(0\) & \(0\) & \(0\) & \(0\) & \(0\) & \(0\) & \(0\) & \(0\) & \(1\) & \(0\) & \(1\) & \(0\) & \(0\) & \(0\) & \(0\) & \(-2\) & \(0\) & \(1\) \\
    \(\widetilde{C_{9}}\) & \(0\) & \(1\) & \(0\) & \(0\) & \(0\) & \(0\) & \(0\) & \(0\) & \(0\) & \(0\) & \(0\) & \(1\) & \(0\) & \(0\) & \(0\) & \(0\) & \(-2\) & \(1\) \\
    \(\widetilde{H_{\mathcal{S}}}\) & \(0\) & \(0\) & \(0\) & \(0\) & \(0\) & \(0\) & \(0\) & \(0\) & \(0\) & \(0\) & \(0\) & \(0\) & \(1\) & \(1\) & \(1\) & \(1\) & \(1\) & \(4\) \\
    \hline
  \end{tabular}.
\end{table}
Note that the intersection matrix is non-degenerate.

Discriminant groups and discriminant forms of the lattices \(L_{\mathcal{S}}\) and \(H \oplus \Pic(X)\) are given by
\begin{gather*}
  G' = 
  \begin{pmatrix}
    \frac{1}{2} & 0 & \frac{1}{2} & \frac{1}{2} & 0 & 0 & 0 & 0 & 0 & 0 & \frac{1}{2} & \frac{1}{2} & 0 & 0 & 0 & 0 & 0 & \frac{1}{2} \\
    0 & \frac{1}{2} & \frac{1}{4} & 0 & \frac{3}{8} & \frac{1}{4} & \frac{1}{8} & \frac{3}{8} & \frac{3}{4} & \frac{7}{8} & \frac{1}{8} & \frac{3}{8} & \frac{1}{2} & 0 & 0 & \frac{1}{4} & \frac{3}{4} & \frac{5}{8}
  \end{pmatrix}, \\
  G'' = 
  \begin{pmatrix}
    0 & 0 & 0 & \frac{1}{2} \\
    0 & 0 & \frac{7}{8} & \frac{1}{4}
  \end{pmatrix}; \;
  B' = 
  \begin{pmatrix}
    \frac{1}{2} & \frac{1}{2} \\
    \frac{1}{2} & \frac{7}{8}
  \end{pmatrix}, \;
  B'' =
  \begin{pmatrix}
    \frac{1}{2} & \frac{1}{2} \\
    \frac{1}{2} & \frac{1}{8}
  \end{pmatrix}; \;
  \begin{pmatrix}
    Q' \\ Q''
  \end{pmatrix}
  =
  \begin{pmatrix}
    \frac{1}{2} & \frac{15}{8} \\
    \frac{3}{2} & \frac{1}{8}
  \end{pmatrix}.
\end{gather*}


\subsection{Family \textnumero2.24}\label{subsection:02-24}

The pencil \(\mathcal{S}\) is defined by the equation
\begin{gather*}
  X^{2} Y^{2} + X^{2} Y Z + X Y^{2} Z + X Y Z^{2} + X^{2} Y T + Y^{2} Z T + X Z^{2} T + X Z T^{2} + Y Z T^{2} = \lambda X Y Z T.
\end{gather*}
Members \(\mathcal{S}_{\lambda}\) of the pencil are irreducible for any \(\lambda \in \mathbb{P}^1\) except
\(\mathcal{S}_{\infty} = S_{(X)} + S_{(Y)} + S_{(Z)} + S_{(T)}\).
The base locus of the pencil \(\mathcal{S}\) consists of the following curves:
\begin{gather*}
  C_{1} = C_{(X, Y)}, \;
  C_{2} = C_{(X, Z)}, \;
  C_{3} = C_{(X, T)}, \;
  C_{4} = C_{(Y, Z)}, \;
  C_{5} = C_{(Y, T)}, \\
  C_{6} = C_{(X, Y + T)}, \;
  C_{7} = C_{(Y, Z + T)}, \;
  C_{8} = C_{(Z, Y + T)}, \;
  C_{9} = C_{(T, X + Z)}, \;
  C_{10} = C_{(T, Y + Z)}.
\end{gather*}
Their linear equivalence classes on the generic member \(\mathcal{S}_{\Bbbk}\) of the pencil satisfy the following relations:
\begin{gather*}
  \begin{pmatrix}
    [C_{1}] \\ [C_{4}] \\ [C_{8}] \\ [C_{10}]
  \end{pmatrix} = 
  \begin{pmatrix}
    -1 & -1 & 0 & -1 & 0 & 0 & 1 \\
    1 & 1 & -1 & 1 & -1 & 0 & 0 \\
    -3 & -1 & 1 & -1 & 1 & 0 & 1 \\
    0 & -1 & -1 & 0 & 0 & -1 & 1
  \end{pmatrix} \cdot
  \begin{pmatrix}
    [C_{2}] & [C_{3}] & [C_{5}] & [C_{6}] & [C_{7}] & [C_{9}] & [H_{\mathcal{S}}]
  \end{pmatrix}^T.
\end{gather*}

For a general choice of \(\lambda \in \mathbb{C}\) the surface \(\mathcal{S}_{\lambda}\) has the following singularities:
\begin{itemize}\setlength{\itemindent}{2cm}
\item[\(P_{1} = P_{(X, Y, Z)}\):] type \(\mathbb{A}_2\) with the quadratic term \(Z \cdot (X + Y)\);
\item[\(P_{2} = P_{(X, Y, T)}\):] type \(\mathbb{A}_3\) with the quadratic term \(X \cdot (Y + T)\);
\item[\(P_{3} = P_{(X, Z, T)}\):] type \(\mathbb{A}_1\) with the quadratic term \(X^2 + Z (X + T)\);
\item[\(P_{4} = P_{(Y, Z, T)}\):] type \(\mathbb{A}_3\) with the quadratic term \(Y \cdot (Y + Z + T)\);
\item[\(P_{5} = P_{(X, Z, Y + T)}\):] type \(\mathbb{A}_2\) with the quadratic term \(Z \cdot ((\lambda + 2) X - Y - T)\).
\end{itemize}

Galois action on the lattice \(L_{\lambda}\) is trivial. The intersection matrix on \(L_{\lambda} = L_{\mathcal{S}}\) is represented by
\begin{table}[H]
  \begin{tabular}{|c||cc|ccc|c|ccc|cc|ccccccc|}
    \hline
    \(\bullet\) & \(E_1^1\) & \(E_1^2\) & \(E_2^1\) & \(E_2^2\) & \(E_2^3\) & \(E_3^1\) & \(E_4^1\) & \(E_4^2\) & \(E_4^3\) & \(E_5^1\) & \(E_5^2\) & \(\widetilde{C_{2}}\) & \(\widetilde{C_{3}}\) & \(\widetilde{C_{5}}\) & \(\widetilde{C_{6}}\) & \(\widetilde{C_{7}}\) & \(\widetilde{C_{9}}\) & \(\widetilde{H_{\mathcal{S}}}\) \\
    \hline
    \hline
    \(\widetilde{C_{2}}\) & \(1\) & \(0\) & \(0\) & \(0\) & \(0\) & \(1\) & \(0\) & \(0\) & \(0\) & \(1\) & \(0\) & \(-2\) & \(0\) & \(0\) & \(0\) & \(0\) & \(0\) & \(1\) \\
    \(\widetilde{C_{3}}\) & \(0\) & \(0\) & \(1\) & \(0\) & \(0\) & \(1\) & \(0\) & \(0\) & \(0\) & \(0\) & \(0\) & \(0\) & \(-2\) & \(0\) & \(0\) & \(0\) & \(0\) & \(1\) \\
    \(\widetilde{C_{5}}\) & \(0\) & \(0\) & \(0\) & \(0\) & \(1\) & \(0\) & \(1\) & \(0\) & \(0\) & \(0\) & \(0\) & \(0\) & \(0\) & \(-2\) & \(0\) & \(0\) & \(1\) & \(1\) \\
    \(\widetilde{C_{6}}\) & \(0\) & \(0\) & \(0\) & \(1\) & \(0\) & \(0\) & \(0\) & \(0\) & \(0\) & \(0\) & \(1\) & \(0\) & \(0\) & \(0\) & \(-2\) & \(0\) & \(0\) & \(1\) \\
    \(\widetilde{C_{7}}\) & \(0\) & \(0\) & \(0\) & \(0\) & \(0\) & \(0\) & \(0\) & \(1\) & \(0\) & \(0\) & \(0\) & \(0\) & \(0\) & \(0\) & \(0\) & \(-2\) & \(0\) & \(1\) \\
    \(\widetilde{C_{9}}\) & \(0\) & \(0\) & \(0\) & \(0\) & \(0\) & \(1\) & \(0\) & \(0\) & \(0\) & \(0\) & \(0\) & \(0\) & \(0\) & \(1\) & \(0\) & \(0\) & \(-2\) & \(1\) \\
    \(\widetilde{H_{\mathcal{S}}}\) & \(0\) & \(0\) & \(0\) & \(0\) & \(0\) & \(0\) & \(0\) & \(0\) & \(0\) & \(0\) & \(0\) & \(1\) & \(1\) & \(1\) & \(1\) & \(1\) & \(1\) & \(4\) \\
    \hline
  \end{tabular}.
\end{table}
Note that the intersection matrix is non-degenerate.

Discriminant groups and discriminant forms of the lattices \(L_{\mathcal{S}}\) and \(H \oplus \Pic(X)\) are given by
\begin{gather*}
  G' = 
  \begin{pmatrix}
    \frac{20}{21} & \frac{10}{21} & \frac{16}{21} & \frac{6}{7} & \frac{17}{21} & \frac{8}{21} & \frac{6}{7} & \frac{20}{21} & \frac{10}{21} & \frac{1}{3} & \frac{5}{21} & \frac{3}{7} & \frac{2}{3} & \frac{16}{21} & \frac{1}{7} & \frac{4}{7} & \frac{2}{3} & \frac{4}{21}
  \end{pmatrix}, \\
  G'' = 
  \begin{pmatrix}
    0 & 0 & \frac{8}{21} & \frac{1}{21}
  \end{pmatrix}; \;
  B' = 
  \begin{pmatrix}
    \frac{11}{21}
  \end{pmatrix}, \;
  B'' =
  \begin{pmatrix}
    \frac{10}{21}
  \end{pmatrix}; \;
  Q' =
  \begin{pmatrix}
    \frac{32}{21}
  \end{pmatrix}, \;
  Q'' =
  \begin{pmatrix}
    \frac{10}{21}
  \end{pmatrix}.
\end{gather*}


\subsection{Family \textnumero2.25}\label{subsection:02-25}

The pencil \(\mathcal{S}\) is defined by the equation
\begin{gather*}
  X^{2} Y Z + X Y^{2} Z + X Y Z^{2} + Y^{2} Z^{2} + X^{2} Y T + X Z T^{2} + Y Z T^{2} + X T^{3} = \lambda X Y Z T.
\end{gather*}
Members \(\mathcal{S}_{\lambda}\) of the pencil are irreducible for any \(\lambda \in \mathbb{P}^1\) except
\begin{gather*}
  \mathcal{S}_{\infty} = S_{(X)} + S_{(Y)} + S_{(Z)} + S_{(T)}, \;
  \mathcal{S}_{- 1} = S_{(X (Z + T) + Y Z)} + S_{(Y (X + Z) + T^2)}.
\end{gather*}
The base locus of the pencil \(\mathcal{S}\) consists of the following curves:
\begin{gather*}
  C_{1} = C_{(X, Y)}, \;
  C_{2} = C_{(X, Z)}, \;
  C_{3} = C_{(Y, T)}, \;
  C_{4} = C_{(Z, T)}, \;
  C_{5} = C_{(Y, Z + T)}, \\
  C_{6} = C_{(T, X + Y)}, \;
  C_{7} = C_{(T, X + Z)}, \;
  C_{8} = C_{(X, Y Z + T^2)}, \;
  C_{9} = C_{(Z, X Y + T^2)}.
\end{gather*}
Their linear equivalence classes on the generic member \(\mathcal{S}_{\Bbbk}\) of the pencil satisfy the following relations:
\begin{gather*}
  \begin{pmatrix}
    [C_{1}] \\ [C_{6}] \\ [C_{8}] \\ [C_{9}]
  \end{pmatrix} = 
  \begin{pmatrix}
    0 & -2 & 0 & -1 & 0 & 1 \\
    0 & -1 & -1 & 0 & -1 & 1 \\
    -1 & 2 & 0 & 1 & 0 & 0 \\
    -1 & 0 & -1 & 0 & 0 & 1
  \end{pmatrix} \cdot
  \begin{pmatrix}
    [C_{2}] & [C_{3}] & [C_{4}] & [C_{5}] & [C_{7}] & [H_{\mathcal{S}}]
  \end{pmatrix}^T.
\end{gather*}

For a general choice of \(\lambda \in \mathbb{C}\) the surface \(\mathcal{S}_{\lambda}\) has the following singularities:
\begin{itemize}\setlength{\itemindent}{2cm}
\item[\(P_{1} = P_{(X, Y, T)}\):] type \(\mathbb{A}_4\) with the quadratic term \(Y \cdot (X + Y)\);
\item[\(P_{2} = P_{(X, Z, T)}\):] type \(\mathbb{A}_5\) with the quadratic term \(Z \cdot (X + Z)\);
\item[\(P_{3} = P_{(Y, Z, T)}\):] type \(\mathbb{A}_3\) with the quadratic term \(Y \cdot (Z + T)\);
\item[\(P_{4} = P_{(Y, T, X + Z)}\):] type \(\mathbb{A}_1\) with the quadratic term \(Y (X + Z - (\lambda + 1) T) + T^2\).
\end{itemize}

Galois action on the lattice \(L_{\lambda}\) is trivial. The intersection matrix on \(L_{\lambda} = L_{\mathcal{S}}\) is represented by
\begin{table}[H]
  \begin{tabular}{|c||cccc|ccccc|ccc|c|cccccc|}
    \hline
    \(\bullet\) & \(E_1^1\) & \(E_1^2\) & \(E_1^3\) & \(E_1^4\) & \(E_2^1\) & \(E_2^2\) & \(E_2^3\) & \(E_2^4\) & \(E_2^5\) & \(E_3^1\) & \(E_3^2\) & \(E_3^3\) & \(E_4^1\) & \(\widetilde{C_{2}}\) & \(\widetilde{C_{3}}\) & \(\widetilde{C_{4}}\) & \(\widetilde{C_{5}}\) & \(\widetilde{C_{7}}\) & \(\widetilde{H_{\mathcal{S}}}\) \\
    \hline
    \hline
    \(\widetilde{C_{2}}\) & \(0\) & \(0\) & \(0\) & \(0\) & \(0\) & \(1\) & \(0\) & \(0\) & \(0\) & \(0\) & \(0\) & \(0\) & \(0\) & \(-2\) & \(0\) & \(0\) & \(0\) & \(0\) & \(1\) \\
    \(\widetilde{C_{3}}\) & \(1\) & \(0\) & \(0\) & \(0\) & \(0\) & \(0\) & \(0\) & \(0\) & \(0\) & \(1\) & \(0\) & \(0\) & \(1\) & \(0\) & \(-2\) & \(0\) & \(0\) & \(0\) & \(1\) \\
    \(\widetilde{C_{4}}\) & \(0\) & \(0\) & \(0\) & \(0\) & \(1\) & \(0\) & \(0\) & \(0\) & \(0\) & \(0\) & \(0\) & \(1\) & \(0\) & \(0\) & \(0\) & \(-2\) & \(0\) & \(0\) & \(1\) \\
    \(\widetilde{C_{5}}\) & \(0\) & \(0\) & \(0\) & \(0\) & \(0\) & \(0\) & \(0\) & \(0\) & \(0\) & \(0\) & \(1\) & \(0\) & \(0\) & \(0\) & \(0\) & \(0\) & \(-2\) & \(0\) & \(1\) \\
    \(\widetilde{C_{7}}\) & \(0\) & \(0\) & \(0\) & \(0\) & \(0\) & \(0\) & \(0\) & \(0\) & \(1\) & \(0\) & \(0\) & \(0\) & \(1\) & \(0\) & \(0\) & \(0\) & \(0\) & \(-2\) & \(1\) \\
    \(\widetilde{H_{\mathcal{S}}}\) & \(0\) & \(0\) & \(0\) & \(0\) & \(0\) & \(0\) & \(0\) & \(0\) & \(0\) & \(0\) & \(0\) & \(0\) & \(0\) & \(1\) & \(1\) & \(1\) & \(1\) & \(1\) & \(4\) \\
    \hline
  \end{tabular}.
\end{table}
Note that the intersection matrix is degenerate. We choose the following integral basis of the lattice \(L_{\lambda}\):
\begin{align*}
  \begin{pmatrix}
    [\widetilde{C_{7}}]
  \end{pmatrix} =
  \begin{pmatrix}
    -4 & -3 & -2 & -1 & 2 & 3 & 2 & 1 & 0 & -4 & -3 & -1 & -3 & 2 & -5 & 1 & -1 & 1
  \end{pmatrix} \cdot \\
  \begin{pmatrix}
    [E_1^1] & [E_1^2] & [E_1^3] & [E_1^4] & [E_2^1] & [E_2^2] & [E_2^3] & [E_2^4] & [E_2^5] & \\ [E_3^1] & [E_3^2] & [E_3^3] & [E_4^1] & [\widetilde{C_{2}}] & [\widetilde{C_{3}}] & [\widetilde{C_{4}}] & [\widetilde{C_{5}}] & [\widetilde{H_{\mathcal{S}}}]
  \end{pmatrix}^T.
\end{align*}

Discriminant groups and discriminant forms of the lattices \(L_{\mathcal{S}}\) and \(H \oplus \Pic(X)\) are given by
\begin{gather*}
  G' = 
  \begin{pmatrix}
    0 & 0 & 0 & 0 & \frac{1}{4} & 0 & \frac{1}{2} & 0 & \frac{1}{2} & \frac{1}{2} & 0 & \frac{1}{4} & 0 & \frac{1}{4} & 0 & \frac{1}{2} & \frac{1}{4} & \frac{1}{2} \\
    0 & \frac{1}{2} & 0 & \frac{1}{2} & \frac{3}{4} & 0 & \frac{3}{4} & \frac{1}{2} & \frac{1}{4} & \frac{1}{4} & 0 & \frac{1}{4} & \frac{3}{4} & \frac{1}{2} & \frac{1}{2} & \frac{1}{2} & \frac{1}{2} & 0
  \end{pmatrix}, \\
  G'' = 
  \begin{pmatrix}
    0 & 0 & \frac{1}{4} & 0 \\
    0 & 0 & 0 & \frac{1}{4}
  \end{pmatrix}; \;
  B'' =
  \begin{pmatrix}
    0 & \frac{3}{4} \\
    \frac{3}{4} & \frac{3}{4}
  \end{pmatrix}, \;
  B'' =
  \begin{pmatrix}
    0 & \frac{1}{4} \\
    \frac{1}{4} & \frac{1}{4}
  \end{pmatrix}; \;
  \begin{pmatrix}
    Q' \\ Q''
  \end{pmatrix}
  =
  \begin{pmatrix}
    0 & \frac{7}{4} \\
    0 & \frac{1}{4}
  \end{pmatrix}.
\end{gather*}


\subsection{Family \textnumero2.26}\label{subsection:02-26}

The pencil \(\mathcal{S}\) is defined by the equation
\begin{gather*}
  X^{2} Y Z + X Y^{2} Z + X Y Z^{2} + Y^{2} Z T + X^{2} T^{2} + X Y T^{2} + X Z T^{2} + Y Z T^{2} = \lambda X Y Z T.
\end{gather*}
Members \(\mathcal{S}_{\lambda}\) of the pencil are irreducible for any \(\lambda \in \mathbb{P}^1\) except
\(\mathcal{S}_{\infty} = S_{(X)} + S_{(Y)} + S_{(Z)} + S_{(T)}\).
The base locus of the pencil \(\mathcal{S}\) consists of the following curves:
\begin{gather*}
  C_{1} = C_{(X, Y)}, \;
  C_{2} = C_{(X, Z)}, \;
  C_{3} = C_{(X, T)}, \;
  C_{4} = C_{(Y, T)}, \;
  C_{5} = C_{(Z, T)}, \\
  C_{6} = C_{(X, Y + T)}, \;
  C_{7} = C_{(Y, X + Z)}, \;
  C_{8} = C_{(Z, X + Y)}, \;
  C_{9} = C_{(T, X + Y + Z)}.
\end{gather*}
Their linear equivalence classes on the generic member \(\mathcal{S}_{\Bbbk}\) of the pencil satisfy the following relations:
\begin{gather*}
  \begin{pmatrix}
    [C_{6}] \\ [C_{7}] \\ [C_{8}] \\ [C_{9}]
  \end{pmatrix} = 
  \begin{pmatrix}
    -1 & -1 & -1 & 0 & 0 & 1 \\
    -1 & 0 & 0 & -2 & 0 & 1 \\
    0 & -1 & 0 & 0 & -2 & 1 \\
    0 & 0 & -1 & -1 & -1 & 1
  \end{pmatrix} \cdot
  \begin{pmatrix}
    [C_{1}] & [C_{2}] & [C_{3}] & [C_{4}] & [C_{5}] & [H_{\mathcal{S}}]
  \end{pmatrix}^T.
\end{gather*}

For a general choice of \(\lambda \in \mathbb{C}\) the surface \(\mathcal{S}_{\lambda}\) has the following singularities:
\begin{itemize}\setlength{\itemindent}{2cm}
\item[\(P_{1} = P_{(X, Y, Z)}\):] type \(\mathbb{A}_2\) with the quadratic term \((X + Y) \cdot (X + Z)\);
\item[\(P_{2} = P_{(X, Y, T)}\):] type \(\mathbb{A}_3\) with the quadratic term \(X \cdot Y\);
\item[\(P_{3} = P_{(X, Z, T)}\):] type \(\mathbb{A}_2\) with the quadratic term \(Z \cdot (X + T)\);
\item[\(P_{4} = P_{(Y, Z, T)}\):] type \(\mathbb{A}_1\) with the quadratic term \(Y Z + T^2\);
\item[\(P_{5} = P_{(Y, T, X + Z)}\):] type \(\mathbb{A}_2\) with the quadratic term \(Y \cdot (X + Y + Z - \lambda T)\);
\item[\(P_{6} = P_{(Z, T, X + Y)}\):] type \(\mathbb{A}_2\) with the quadratic term \(Z \cdot (X + Y + Z - (\lambda + 1) T)\).
\end{itemize}

Galois action on the lattice \(L_{\lambda}\) is trivial. The intersection matrix on \(L_{\lambda} = L_{\mathcal{S}}\) is represented by
\begin{table}[H]
  \begin{tabular}{|c||cc|ccc|cc|c|cc|cc|cccccc|}
    \hline
    \(\bullet\) & \(E_1^1\) & \(E_1^2\) & \(E_2^1\) & \(E_2^2\) & \(E_2^3\) & \(E_3^1\) & \(E_3^2\) & \(E_4^1\) & \(E_5^1\) & \(E_5^2\) & \(E_6^1\) & \(E_6^2\) & \(\widetilde{C_{1}}\) & \(\widetilde{C_{2}}\) & \(\widetilde{C_{3}}\) & \(\widetilde{C_{4}}\) & \(\widetilde{C_{5}}\) & \(\widetilde{H_{\mathcal{S}}}\) \\
    \hline
    \hline
    \(\widetilde{C_{1}}\) & \(1\) & \(0\) & \(0\) & \(1\) & \(0\) & \(0\) & \(0\) & \(0\) & \(0\) & \(0\) & \(0\) & \(0\) & \(-2\) & \(0\) & \(0\) & \(0\) & \(0\) & \(1\) \\
    \(\widetilde{C_{2}}\) & \(0\) & \(1\) & \(0\) & \(0\) & \(0\) & \(1\) & \(0\) & \(0\) & \(0\) & \(0\) & \(0\) & \(0\) & \(0\) & \(-2\) & \(0\) & \(0\) & \(0\) & \(1\) \\
    \(\widetilde{C_{3}}\) & \(0\) & \(0\) & \(1\) & \(0\) & \(0\) & \(0\) & \(1\) & \(0\) & \(0\) & \(0\) & \(0\) & \(0\) & \(0\) & \(0\) & \(-2\) & \(0\) & \(0\) & \(1\) \\
    \(\widetilde{C_{4}}\) & \(0\) & \(0\) & \(0\) & \(0\) & \(1\) & \(0\) & \(0\) & \(1\) & \(1\) & \(0\) & \(0\) & \(0\) & \(0\) & \(0\) & \(0\) & \(-2\) & \(0\) & \(1\) \\
    \(\widetilde{C_{5}}\) & \(0\) & \(0\) & \(0\) & \(0\) & \(0\) & \(1\) & \(0\) & \(1\) & \(0\) & \(0\) & \(1\) & \(0\) & \(0\) & \(0\) & \(0\) & \(0\) & \(-2\) & \(1\) \\
    \(\widetilde{H_{\mathcal{S}}}\) & \(0\) & \(0\) & \(0\) & \(0\) & \(0\) & \(0\) & \(0\) & \(0\) & \(0\) & \(0\) & \(0\) & \(0\) & \(1\) & \(1\) & \(1\) & \(1\) & \(1\) & \(4\) \\
    \hline
  \end{tabular}.
\end{table}
Note that the intersection matrix is non-degenerate.

Discriminant groups and discriminant forms of the lattices \(L_{\mathcal{S}}\) and \(H \oplus \Pic(X)\) are given by
\begin{gather*}
  G' = 
  \begin{pmatrix}
    \frac{20}{21} & \frac{3}{7} & \frac{2}{21} & 0 & \frac{3}{7} & \frac{8}{21} & \frac{2}{7} & \frac{5}{7} & \frac{4}{7} & \frac{2}{7} & \frac{1}{21} & \frac{11}{21} & \frac{10}{21} & \frac{19}{21} & \frac{4}{21} & \frac{6}{7} & \frac{4}{7} & 0
  \end{pmatrix}, \\
  G'' = 
  \begin{pmatrix}
    0 & 0 & \frac{11}{21} & \frac{1}{21}
  \end{pmatrix}; \;
  B' = 
  \begin{pmatrix}
    \frac{10}{21}
  \end{pmatrix}, \;
  B'' =
  \begin{pmatrix}
    \frac{11}{21}
  \end{pmatrix}; \;
  Q' =
  \begin{pmatrix}
    \frac{10}{21}
  \end{pmatrix}, \;
  Q'' =
  \begin{pmatrix}
    \frac{32}{21}
  \end{pmatrix}.
\end{gather*}


\subsection{Family \textnumero2.27}\label{subsection:02-27}

The pencil \(\mathcal{S}\) is defined by the equation
\begin{gather*}
  X^{2} Y Z + X Y^{2} Z + X Y Z^{2} + X^{2} Y T + Y Z T^{2} + X T^{3} + Z T^{3} = \lambda X Y Z T.
\end{gather*}
Members \(\mathcal{S}_{\lambda}\) of the pencil are irreducible for any \(\lambda \in \mathbb{P}^1\) except
\(\mathcal{S}_{\infty} = S_{(X)} + S_{(Y)} + S_{(Z)} + S_{(T)}\).
The base locus of the pencil \(\mathcal{S}\) consists of the following curves:
\begin{gather*}
  C_{1} = C_{(X, Z)}, \;
  C_{2} = C_{(X, T)}, \;
  C_{3} = C_{(Y, T)}, \;
  C_{4} = C_{(Z, T)}, \\
  C_{5} = C_{(X, Y + T)}, \;
  C_{6} = C_{(Y, X + Z)}, \;
  C_{7} = C_{(T, X + Y + Z)}, \;
  C_{8} = C_{(Z, X Y + T^2)}.
\end{gather*}
Their linear equivalence classes on the generic member \(\mathcal{S}_{\Bbbk}\) of the pencil satisfy the following relations:
\begin{gather*}
  \begin{pmatrix}
    [C_{1}] \\ [C_{6}] \\ [C_{7}] \\ [C_{8}]
  \end{pmatrix} = 
  \begin{pmatrix}
    -2 & 0 & 0 & -1 & 1 \\
    0 & -3 & 0 & 0 & 1 \\
    -1 & -1 & -1 & 0 & 1 \\
    2 & 0 & -1 & 1 & 0
  \end{pmatrix} \cdot
  \begin{pmatrix}
    [C_{2}] & [C_{3}] & [C_{4}] & [C_{5}] & [H_{\mathcal{S}}]
  \end{pmatrix}^T.
\end{gather*}

For a general choice of \(\lambda \in \mathbb{C}\) the surface \(\mathcal{S}_{\lambda}\) has the following singularities:
\begin{itemize}\setlength{\itemindent}{2cm}
\item[\(P_{1} = P_{(X, Y, T)}\):] type \(\mathbb{A}_2\) with the quadratic term \(X \cdot Y\);
\item[\(P_{2} = P_{(X, Z, T)}\):] type \(\mathbb{A}_5\) with the quadratic term \(X \cdot Z\);
\item[\(P_{3} = P_{(Y, Z, T)}\):] type \(\mathbb{A}_2\) with the quadratic term \(Y \cdot (Z + T)\);
\item[\(P_{4} = P_{(X, T, Y + Z)}\):] type \(\mathbb{A}_1\) with the quadratic term \(X (X + Y + Z - \lambda T) + T^2\);
\item[\(P_{5} = P_{(Y, T, X + Z)}\):] type \(\mathbb{A}_3\) with the quadratic term \(Y \cdot (X + Y + Z - (\lambda + 1) T)\).
\end{itemize}

Galois action on the lattice \(L_{\lambda}\) is trivial. The intersection matrix on \(L_{\lambda} = L_{\mathcal{S}}\) is represented by
\begin{table}[H]
  \begin{tabular}{|c||cc|ccccc|cc|c|ccc|ccccc|}
    \hline
    \(\bullet\) & \(E_1^1\) & \(E_1^2\) & \(E_2^1\) & \(E_2^2\) & \(E_2^3\) & \(E_2^4\) & \(E_2^5\) & \(E_3^1\) & \(E_3^2\) & \(E_4^1\) & \(E_5^1\) & \(E_5^2\) & \(E_5^3\) & \(\widetilde{C_{2}}\) & \(\widetilde{C_{3}}\) & \(\widetilde{C_{4}}\) & \(\widetilde{C_{5}}\) & \(\widetilde{H_{\mathcal{S}}}\) \\
    \hline
    \hline
    \(\widetilde{C_{2}}\) & \(1\) & \(0\) & \(1\) & \(0\) & \(0\) & \(0\) & \(0\) & \(0\) & \(0\) & \(1\) & \(0\) & \(0\) & \(0\) & \(-2\) & \(0\) & \(0\) & \(0\) & \(1\) \\
    \(\widetilde{C_{3}}\) & \(0\) & \(1\) & \(0\) & \(0\) & \(0\) & \(0\) & \(0\) & \(1\) & \(0\) & \(0\) & \(1\) & \(0\) & \(0\) & \(0\) & \(-2\) & \(0\) & \(0\) & \(1\) \\
    \(\widetilde{C_{4}}\) & \(0\) & \(0\) & \(0\) & \(0\) & \(0\) & \(0\) & \(1\) & \(0\) & \(1\) & \(0\) & \(0\) & \(0\) & \(0\) & \(0\) & \(0\) & \(-2\) & \(0\) & \(1\) \\
    \(\widetilde{C_{5}}\) & \(1\) & \(0\) & \(0\) & \(0\) & \(0\) & \(0\) & \(0\) & \(0\) & \(0\) & \(0\) & \(0\) & \(0\) & \(0\) & \(0\) & \(0\) & \(0\) & \(-2\) & \(1\) \\
    \(\widetilde{H_{\mathcal{S}}}\) & \(0\) & \(0\) & \(0\) & \(0\) & \(0\) & \(0\) & \(0\) & \(0\) & \(0\) & \(0\) & \(0\) & \(0\) & \(0\) & \(1\) & \(1\) & \(1\) & \(1\) & \(4\) \\
    \hline
  \end{tabular}.
\end{table}
Note that the intersection matrix is non-degenerate.

Discriminant groups and discriminant forms of the lattices \(L_{\mathcal{S}}\) and \(H \oplus \Pic(X)\) are given by
\begin{gather*}
  G' = 
  \begin{pmatrix}
    \frac{6}{17} & \frac{4}{17} & \frac{2}{17} & \frac{11}{17} & \frac{3}{17} & \frac{12}{17} & \frac{4}{17} & 0 & \frac{15}{17} & \frac{5}{17} & \frac{10}{17} & \frac{1}{17} & \frac{9}{17} & \frac{10}{17} & \frac{2}{17} & \frac{13}{17} & \frac{15}{17} & \frac{7}{17}
  \end{pmatrix}, \\
  G'' = 
  \begin{pmatrix}
    0 & 0 & \frac{10}{17} & \frac{1}{17}
  \end{pmatrix}; \;
  B' = 
  \begin{pmatrix}
    \frac{8}{17}
  \end{pmatrix}, \;
  B'' =
  \begin{pmatrix}
    \frac{9}{17}
  \end{pmatrix}; \;
  Q' =
  \begin{pmatrix}
    \frac{8}{17}
  \end{pmatrix}, \;
  Q'' =
  \begin{pmatrix}
    \frac{26}{17}
  \end{pmatrix}.
\end{gather*}


\subsection{Family \textnumero2.28}\label{subsection:02-28}

The pencil \(\mathcal{S}\) is defined by the equation
\begin{gather*}
  X^{2} Y Z + X Y Z^{2} + X^{2} Y T + X Y^{2} T + Y^{2} Z T + X^{2} T^{2} + X Z T^{2} = \lambda X Y Z T.
\end{gather*}
Members \(\mathcal{S}_{\lambda}\) of the pencil are irreducible for any \(\lambda \in \mathbb{P}^1\) except
\begin{gather*}
  \mathcal{S}_{\infty} = S_{(X)} + S_{(Y)} + S_{(Z)} + S_{(T)}, \;
  \mathcal{S}_{- 1} = S_{(X + Z)} + S_{(X Y (Z + T) + T (X T + Y^2))}.
\end{gather*}
The base locus of the pencil \(\mathcal{S}\) consists of the following curves:
\begin{gather*}
  C_{1} = C_{(X, Y)}, \;
  C_{2} = C_{(X, Z)}, \;
  C_{3} = C_{(X, T)}, \;
  C_{4} = C_{(Y, T)}, \\
  C_{5} = C_{(Z, T)}, \;
  C_{6} = C_{(Y, X + Z)}, \;
  C_{7} = C_{(T, X + Z)}, \;
  C_{8} = C_{(Z, Y^2 + X (Y + T))}.
\end{gather*}
Their linear equivalence classes on the generic member \(\mathcal{S}_{\Bbbk}\) of the pencil satisfy the following relations:
\begin{gather*}
  \begin{pmatrix}
    [C_{2}] \\ [C_{5}] \\ [C_{6}] \\ [C_{7}] \\ [C_{8}]
  \end{pmatrix} = 
  \begin{pmatrix}
    -2 & -1 & 0 & 1 \\
    -5 & -3 & -3 & 3 \\
    -1 & 0 & -2 & 1 \\
    5 & 2 & 2 & -2 \\
    7 & 4 & 3 & -3
  \end{pmatrix} \cdot
  \begin{pmatrix}
    [C_{1}] \\ [C_{3}] \\ [C_{4}] \\ [H_{\mathcal{S}}]
  \end{pmatrix}.
\end{gather*}

For a general choice of \(\lambda \in \mathbb{C}\) the surface \(\mathcal{S}_{\lambda}\) has the following singularities:
\begin{itemize}\setlength{\itemindent}{2cm}
\item[\(P_{1} = P_{(X, Y, Z)}\):] type \(\mathbb{A}_4\) with the quadratic term \(X \cdot (X + Z)\);
\item[\(P_{2} = P_{(X, Y, T)}\):] type \(\mathbb{A}_4\) with the quadratic term \(X \cdot Y\);
\item[\(P_{3} = P_{(X, Z, T)}\):] type \(\mathbb{A}_3\) with the quadratic term \(T \cdot (X + Z)\);
\item[\(P_{4} = P_{(Y, Z, T)}\):] type \(\mathbb{A}_1\) with the quadratic term \(Y (Z + T) + T^2\);
\item[\(P_{5} = P_{(Y, T, X + Z)}\):] type \(\mathbb{A}_2\) with the quadratic term \(Y \cdot (X + Z - (\lambda + 1) T)\).
\end{itemize}

Galois action on the lattice \(L_{\lambda}\) is trivial. The intersection matrix on \(L_{\lambda} = L_{\mathcal{S}}\) is represented by
\begin{table}[H]
  \begin{tabular}{|c||cccc|cccc|ccc|c|cc|cccc|}
    \hline
    \(\bullet\) & \(E_1^1\) & \(E_1^2\) & \(E_1^3\) & \(E_1^4\) & \(E_2^1\) & \(E_2^2\) & \(E_2^3\) & \(E_2^4\) & \(E_3^1\) & \(E_3^2\) & \(E_3^3\) & \(E_4^1\) & \(E_5^1\) & \(E_5^2\) & \(\widetilde{C_{1}}\) & \(\widetilde{C_{3}}\) & \(\widetilde{C_{4}}\) & \(\widetilde{H_{\mathcal{S}}}\) \\
    \hline
    \hline
    \(\widetilde{C_{1}}\) & \(1\) & \(0\) & \(0\) & \(0\) & \(0\) & \(1\) & \(0\) & \(0\) & \(0\) & \(0\) & \(0\) & \(0\) & \(0\) & \(0\) & \(-2\) & \(0\) & \(0\) & \(1\) \\
    \(\widetilde{C_{3}}\) & \(0\) & \(0\) & \(0\) & \(0\) & \(1\) & \(0\) & \(0\) & \(0\) & \(1\) & \(0\) & \(0\) & \(0\) & \(0\) & \(0\) & \(0\) & \(-2\) & \(0\) & \(1\) \\
    \(\widetilde{C_{4}}\) & \(0\) & \(0\) & \(0\) & \(0\) & \(0\) & \(0\) & \(0\) & \(1\) & \(0\) & \(0\) & \(0\) & \(1\) & \(1\) & \(0\) & \(0\) & \(0\) & \(-2\) & \(1\) \\
    \(\widetilde{H_{\mathcal{S}}}\) & \(0\) & \(0\) & \(0\) & \(0\) & \(0\) & \(0\) & \(0\) & \(0\) & \(0\) & \(0\) & \(0\) & \(0\) & \(0\) & \(0\) & \(1\) & \(1\) & \(1\) & \(4\) \\
    \hline
  \end{tabular}.
\end{table}
Note that the intersection matrix is non-degenerate.

Discriminant groups and discriminant forms of the lattices \(L_{\mathcal{S}}\) and \(H \oplus \Pic(X)\) are given by
\begin{gather*}
  G' = 
  \begin{pmatrix}
    \frac{7}{9} & \frac{1}{3} & \frac{8}{9} & \frac{4}{9} & \frac{5}{9} & \frac{2}{3} & \frac{5}{9} & \frac{4}{9} & \frac{1}{3} & \frac{2}{9} & \frac{1}{9} & \frac{2}{3} & \frac{5}{9} & \frac{7}{9} & \frac{2}{9} & \frac{4}{9} & \frac{1}{3} & 0
  \end{pmatrix}, \\
  G'' = 
  \begin{pmatrix}
    0 & 0 & \frac{8}{9} & \frac{1}{3}
  \end{pmatrix}; \;
  B' = 
  \begin{pmatrix}
    \frac{7}{9}
  \end{pmatrix}, \;
  B'' =
  \begin{pmatrix}
    \frac{2}{9}
  \end{pmatrix}; \;
  Q' =
  \begin{pmatrix}
    \frac{16}{9}
  \end{pmatrix}, \;
  Q'' =
  \begin{pmatrix}
    \frac{2}{9}
  \end{pmatrix}.
\end{gather*}


\subsection{Family \textnumero2.29}\label{subsection:02-29}

The pencil \(\mathcal{S}\) is defined by the equation
\begin{gather*}
  X^{2} Y Z + X Y^{2} Z + X Y Z^{2} + X^{2} Y T + X Y^{2} T + X Z T^{2} + Y Z T^{2} = \lambda X Y Z T.
\end{gather*}
Members \(\mathcal{S}_{\lambda}\) of the pencil are irreducible for any \(\lambda \in \mathbb{P}^1\) except
\(\mathcal{S}_{\infty} = S_{(X)} + S_{(Y)} + S_{(Z)} + S_{(T)}\).
The base locus of the pencil \(\mathcal{S}\) consists of the following curves:
\begin{gather*}
  C_{1} = C_{(X, Y)}, \;
  C_{2} = C_{(X, Z)}, \;
  C_{3} = C_{(X, T)}, \;
  C_{4} = C_{(Y, Z)}, \\
  C_{5} = C_{(Y, T)}, \;
  C_{6} = C_{(Z, T)}, \;
  C_{7} = C_{(Z, X + Y)}, \;
  C_{8} = C_{(T, X + Y + Z)}.
\end{gather*}
Their linear equivalence classes on the generic member \(\mathcal{S}_{\Bbbk}\) of the pencil satisfy the following relations:
\begin{gather*}
  \begin{pmatrix}
    [C_{2}] \\ [C_{4}] \\ [C_{6}] \\ [C_{7}]
  \end{pmatrix} = 
  \begin{pmatrix}
    -1 & -2 & 0 & 0 & 1 \\
    -1 & 0 & -2 & 0 & 1 \\
    0 & -1 & -1 & -1 & 1 \\
    2 & 3 & 3 & 1 & -2
  \end{pmatrix} \cdot
  \begin{pmatrix}
    [C_{1}] & [C_{3}] & [C_{5}] & [C_{8}] & [H_{\mathcal{S}}]
  \end{pmatrix}^T.
\end{gather*}

For a general choice of \(\lambda \in \mathbb{C}\) the surface \(\mathcal{S}_{\lambda}\) has the following singularities:
\begin{itemize}\setlength{\itemindent}{2cm}
\item[\(P_{1} = P_{(X, Y, Z)}\):] type \(\mathbb{A}_3\) with the quadratic term \(Z \cdot (X + Y)\);
\item[\(P_{2} = P_{(X, Y, T)}\):] type \(\mathbb{A}_3\) with the quadratic term \(X \cdot Y\);
\item[\(P_{3} = P_{(X, Z, T)}\):] type \(\mathbb{A}_2\) with the quadratic term \(X \cdot (Z + T)\);
\item[\(P_{4} = P_{(Y, Z, T)}\):] type \(\mathbb{A}_2\) with the quadratic term \(Y \cdot (Z + T)\);
\item[\(P_{5} = P_{(X, T, Y + Z)}\):] type \(\mathbb{A}_1\) with the quadratic term \(X (X + Y + Z - (\lambda + 1) T) + T^2\);
\item[\(P_{6} = P_{(Y, T, X + Z)}\):] type \(\mathbb{A}_1\) with the quadratic term \(Y (X + Y + Z - (\lambda + 1) T) + T^2\);
\item[\(P_{7} = P_{(Z, T, X + Y)}\):] type \(\mathbb{A}_1\) with the quadratic term \((Z + T) (X + Y + Z) - (\lambda + 1) Z T\).
\end{itemize}

Galois action on the lattice \(L_{\lambda}\) is trivial. The intersection matrix on \(L_{\lambda} = L_{\mathcal{S}}\) is represented by
\begin{table}[H]
  \begin{tabular}{|c||ccc|ccc|cc|cc|c|c|c|ccccc|}
    \hline
    \(\bullet\) & \(E_1^1\) & \(E_1^2\) & \(E_1^3\) & \(E_2^1\) & \(E_2^2\) & \(E_2^3\) & \(E_3^1\) & \(E_3^2\) & \(E_4^1\) & \(E_4^2\) & \(E_5^1\) & \(E_6^1\) & \(E_7^1\) & \(\widetilde{C_{1}}\) & \(\widetilde{C_{3}}\) & \(\widetilde{C_{5}}\) & \(\widetilde{C_{8}}\) & \(\widetilde{H_{\mathcal{S}}}\) \\
    \hline
    \hline
    \(\widetilde{C_{1}}\) & \(1\) & \(0\) & \(0\) & \(0\) & \(1\) & \(0\) & \(0\) & \(0\) & \(0\) & \(0\) & \(0\) & \(0\) & \(0\) & \(-2\) & \(0\) & \(0\) & \(0\) & \(1\) \\
    \(\widetilde{C_{3}}\) & \(0\) & \(0\) & \(0\) & \(1\) & \(0\) & \(0\) & \(1\) & \(0\) & \(0\) & \(0\) & \(1\) & \(0\) & \(0\) & \(0\) & \(-2\) & \(0\) & \(0\) & \(1\) \\
    \(\widetilde{C_{5}}\) & \(0\) & \(0\) & \(0\) & \(0\) & \(0\) & \(1\) & \(0\) & \(0\) & \(1\) & \(0\) & \(0\) & \(1\) & \(0\) & \(0\) & \(0\) & \(-2\) & \(0\) & \(1\) \\
    \(\widetilde{C_{8}}\) & \(0\) & \(0\) & \(0\) & \(0\) & \(0\) & \(0\) & \(0\) & \(0\) & \(0\) & \(0\) & \(1\) & \(1\) & \(1\) & \(0\) & \(0\) & \(0\) & \(-2\) & \(1\) \\
    \(\widetilde{H_{\mathcal{S}}}\) & \(0\) & \(0\) & \(0\) & \(0\) & \(0\) & \(0\) & \(0\) & \(0\) & \(0\) & \(0\) & \(0\) & \(0\) & \(0\) & \(1\) & \(1\) & \(1\) & \(1\) & \(4\) \\
    \hline
  \end{tabular}.
\end{table}
Note that the intersection matrix is non-degenerate.

Discriminant groups and discriminant forms of the lattices \(L_{\mathcal{S}}\) and \(H \oplus \Pic(X)\) are given by
\begin{gather*}
  G' = 
  \begin{pmatrix}
    \frac{1}{2} & 0 & \frac{1}{2} & 0 & 0 & 0 & 0 & 0 & 0 & 0 & \frac{1}{2} & \frac{1}{2} & \frac{1}{2} & 0 & 0 & 0 & 0 & \frac{1}{2} \\
    0 & 0 & 0 & \frac{3}{8} & \frac{1}{2} & \frac{5}{8} & \frac{1}{2} & \frac{3}{4} & \frac{1}{2} & \frac{1}{4} & \frac{1}{8} & \frac{7}{8} & \frac{1}{2} & 0 & \frac{1}{4} & \frac{3}{4} & 0 & \frac{1}{2}
  \end{pmatrix}, \\
  G'' = 
  \begin{pmatrix}
    0 & 0 & \frac{1}{2} & 0 \\
    0 & 0 & \frac{5}{8} & \frac{1}{8}
  \end{pmatrix}; \;
  B' =  
  \begin{pmatrix}
    \frac{1}{2} & \frac{1}{2} \\
    \frac{1}{2} & \frac{3}{8}
  \end{pmatrix}, \;
  B'' =
  \begin{pmatrix}
    \frac{1}{2} & \frac{1}{2} \\
    \frac{1}{2} & \frac{5}{8}
  \end{pmatrix}; \;
  \begin{pmatrix}
    Q' \\ Q''
  \end{pmatrix}
  =
  \begin{pmatrix}
    \frac{1}{2} & \frac{3}{8} \\
    \frac{3}{2} & \frac{13}{8}
  \end{pmatrix}.
\end{gather*}


\subsection{Family \textnumero2.30}\label{subsection:02-30}

The pencil \(\mathcal{S}\) is defined by the equation
\begin{gather*}
  X^{2} Y Z + Y^{2} Z T + Y Z^{2} T + X^{2} T^{2} + X Y T^{2} + X Z T^{2} = \lambda X Y Z T.
\end{gather*}
Members \(\mathcal{S}_{\lambda}\) of the pencil are irreducible for any \(\lambda \in \mathbb{P}^1\) except
\(\mathcal{S}_{\infty} = S_{(X)} + S_{(Y)} + S_{(Z)} + S_{(T)}\).
The base locus of the pencil \(\mathcal{S}\) consists of the following curves:
\begin{gather*}
  C_{1} = C_{(X, Y)}, \;
  C_{2} = C_{(X, Z)}, \;
  C_{3} = C_{(X, T)}, \;
  C_{4} = C_{(Y, T)}, \\
  C_{5} = C_{(Z, T)}, \;
  C_{6} = C_{(X, Y + Z)}, \;
  C_{7} = C_{(Y, X + Z)}, \;
  C_{8} = C_{(Z, X + Y)}.
\end{gather*}
Their linear equivalence classes on the generic member \(\mathcal{S}_{\Bbbk}\) of the pencil satisfy the following relations:
\begin{gather*}
  \begin{pmatrix}
    [C_{5}] \\ [C_{6}] \\ [C_{7}] \\ [C_{8}]
  \end{pmatrix} = 
  \begin{pmatrix}
    0 & 0 & -2 & -1 & 1 \\
    -1 & -1 & -1 & 0 & 1 \\
    -1 & 0 & 0 & -2 & 1 \\
    0 & -1 & 4 & 2 & -1
  \end{pmatrix} \cdot
  \begin{pmatrix}
    [C_{1}] & [C_{2}] & [C_{3}] & [C_{4}] & [H_{\mathcal{S}}]
  \end{pmatrix}^T.
\end{gather*}

For a general choice of \(\lambda \in \mathbb{C}\) the surface \(\mathcal{S}_{\lambda}\) has the following singularities:
\begin{itemize}\setlength{\itemindent}{2cm}
\item[\(P_{1} = P_{(X, Y, Z)}\):] type \(\mathbb{A}_3\) with the quadratic term \(X \cdot (X + Y + Z)\);
\item[\(P_{2} = P_{(X, Y, T)}\):] type \(\mathbb{A}_4\) with the quadratic term \(Y \cdot T\);
\item[\(P_{3} = P_{(X, Z, T)}\):] type \(\mathbb{A}_4\) with the quadratic term \(Z \cdot T\);
\item[\(P_{4} = P_{(Y, Z, T)}\):] type \(\mathbb{A}_1\) with the quadratic term \(Y Z + T^2\);
\item[\(P_{5} = P_{(X, T, Y + Z)}\):] type \(\mathbb{A}_1\) with the quadratic term \(X^2 - T (\lambda X - Y - Z)\).
\end{itemize}

Galois action on the lattice \(L_{\lambda}\) is trivial. The intersection matrix on \(L_{\lambda} = L_{\mathcal{S}}\) is represented by
\begin{table}[H]
  \begin{tabular}{|c||ccc|cccc|cccc|c|c|ccccc|}
    \hline
    \(\bullet\) & \(E_1^1\) & \(E_1^2\) & \(E_1^3\) & \(E_2^1\) & \(E_2^2\) & \(E_2^3\) & \(E_2^4\) & \(E_3^1\) & \(E_3^2\) & \(E_3^3\) & \(E_3^4\) & \(E_4^1\) & \(E_5^1\) & \(\widetilde{C_{1}}\) & \(\widetilde{C_{2}}\) & \(\widetilde{C_{3}}\) & \(\widetilde{C_{4}}\) & \(\widetilde{H_{\mathcal{S}}}\) \\
    \hline
    \hline
    \(\widetilde{C_{1}}\) & \(1\) & \(0\) & \(0\) & \(1\) & \(0\) & \(0\) & \(0\) & \(0\) & \(0\) & \(0\) & \(0\) & \(0\) & \(0\) & \(-2\) & \(0\) & \(0\) & \(0\) & \(1\) \\
    \(\widetilde{C_{2}}\) & \(1\) & \(0\) & \(0\) & \(0\) & \(0\) & \(0\) & \(0\) & \(1\) & \(0\) & \(0\) & \(0\) & \(0\) & \(0\) & \(0\) & \(-2\) & \(0\) & \(0\) & \(1\) \\
    \(\widetilde{C_{3}}\) & \(0\) & \(0\) & \(0\) & \(0\) & \(0\) & \(0\) & \(1\) & \(0\) & \(0\) & \(0\) & \(1\) & \(0\) & \(1\) & \(0\) & \(0\) & \(-2\) & \(0\) & \(1\) \\
    \(\widetilde{C_{4}}\) & \(0\) & \(0\) & \(0\) & \(0\) & \(1\) & \(0\) & \(0\) & \(0\) & \(0\) & \(0\) & \(0\) & \(1\) & \(0\) & \(0\) & \(0\) & \(0\) & \(-2\) & \(1\) \\
    \(\widetilde{H_{\mathcal{S}}}\) & \(0\) & \(0\) & \(0\) & \(0\) & \(0\) & \(0\) & \(0\) & \(0\) & \(0\) & \(0\) & \(0\) & \(0\) & \(0\) & \(1\) & \(1\) & \(1\) & \(1\) & \(4\) \\
    \hline
  \end{tabular}.
\end{table}
Note that the intersection matrix is non-degenerate.

Discriminant groups and discriminant forms of the lattices \(L_{\mathcal{S}}\) and \(H \oplus \Pic(X)\) are given by
\begin{gather*}
  G' = 
  \begin{pmatrix}
    0 & 0 & 0 & 0 & \frac{1}{2} & 0 & \frac{1}{2} & 0 & \frac{1}{2} & 0 & \frac{1}{2} & \frac{1}{2} & 0 & \frac{1}{2} & \frac{1}{2} & 0 & 0 & 0 \\
    \frac{1}{2} & 0 & \frac{1}{2} & \frac{1}{3} & 0 & \frac{2}{3} & \frac{1}{3} & \frac{2}{3} & 0 & \frac{1}{3} & \frac{2}{3} & \frac{1}{2} & \frac{1}{2} & \frac{2}{3} & \frac{1}{3} & 0 & 0 & \frac{1}{2}
  \end{pmatrix}, \\
  G'' = 
  \begin{pmatrix}
    0 & 0 & \frac{1}{2} & 0 \\
    0 & 0 & \frac{2}{3} & \frac{1}{6}
  \end{pmatrix}; \;
  B' = 
  \begin{pmatrix}
    \frac{1}{2} & \frac{1}{2} \\
    \frac{1}{2} & \frac{1}{3}
  \end{pmatrix}, \;
  B'' =
  \begin{pmatrix}
    \frac{1}{2} & \frac{1}{2} \\
    \frac{1}{2} & \frac{2}{3}
  \end{pmatrix}; \;
  \begin{pmatrix}
    Q' \\ Q''
  \end{pmatrix}
  =
  \begin{pmatrix}
    \frac{1}{2} & \frac{1}{3} \\
    \frac{3}{2} & \frac{5}{3}    
  \end{pmatrix}.
\end{gather*}


\subsection{Family \textnumero2.31}\label{subsection:02-31}

The pencil \(\mathcal{S}\) is defined by the equation
\begin{gather*}
  X^{2} Y Z + X Y^{2} Z + X Y Z^{2} + X^{2} Z T + Y Z T^{2} + X T^{3} = \lambda X Y Z T.
\end{gather*}
Members \(\mathcal{S}_{\lambda}\) of the pencil are irreducible for any \(\lambda \in \mathbb{P}^1\) except
\(\mathcal{S}_{\infty} = S_{(X)} + S_{(Y)} + S_{(Z)} + S_{(T)}\).
The base locus of the pencil \(\mathcal{S}\) consists of the following curves:
\begin{gather*}
  C_{1} = C_{(X, Y)}, \;
  C_{2} = C_{(X, Z)}, \;
  C_{3} = C_{(X, T)}, \;
  C_{4} = C_{(Y, T)}, \;
  C_{5} = C_{(Z, T)}, \;
  C_{6} = C_{(T, X + Y + Z)}, \;
  C_{7} = C_{(Y, X Z + T^2)}.
\end{gather*}
Their linear equivalence classes on the generic member \(\mathcal{S}_{\Bbbk}\) of the pencil satisfy the following relations:
\begin{gather*}
  \begin{pmatrix}
    [C_{1}] \\ [C_{2}] \\ [C_{6}] \\ [C_{7}]
  \end{pmatrix} = 
  \begin{pmatrix}
    -2 & 0 & 3 & 0 \\
    0 & 0 & -3 & 1 \\
    -1 & -1 & -1 & 1 \\
    2 & -1 & -3 & 1
  \end{pmatrix} \cdot
  \begin{pmatrix}
    [C_{3}] \\ [C_{4}] \\ [C_{5}] \\ [H_{\mathcal{S}}]
  \end{pmatrix}.
\end{gather*}

For a general choice of \(\lambda \in \mathbb{C}\) the surface \(\mathcal{S}_{\lambda}\) has the following singularities:
\begin{itemize}\setlength{\itemindent}{2cm}
\item[\(P_{1} = P_{(X, Y, T)}\):] type \(\mathbb{A}_5\) with the quadratic term \(X \cdot Y\);
\item[\(P_{2} = P_{(X, Z, T)}\):] type \(\mathbb{A}_4\) with the quadratic term \(X \cdot Z\);
\item[\(P_{3} = P_{(Y, Z, T)}\):] type \(\mathbb{A}_2\) with the quadratic term \(Z \cdot (Y + T)\);
\item[\(P_{4} = P_{(X, T, Y + Z)}\):] type \(\mathbb{A}_1\) with the quadratic term \(X (X + Y + Z - \lambda T) + T^2\);
\item[\(P_{5} = P_{(Z, T, X + Y)}\):] type \(\mathbb{A}_2\) with the quadratic term \(Z \cdot (X + Y + Z - (\lambda + 1) T)\).
\end{itemize}

Galois action on the lattice \(L_{\lambda}\) is trivial. The intersection matrix on \(L_{\lambda} = L_{\mathcal{S}}\) is represented by
\begin{table}[H]
  \begin{tabular}{|c||ccccc|cccc|cc|c|cc|cccc|}
    \hline
    \(\bullet\) & \(E_1^1\) & \(E_1^2\) & \(E_1^3\) & \(E_1^4\) & \(E_1^5\) & \(E_2^1\) & \(E_2^2\) & \(E_2^3\) & \(E_2^4\) & \(E_3^1\) & \(E_3^2\) & \(E_4^1\) & \(E_5^1\) & \(E_5^2\) & \(\widetilde{C_{3}}\) & \(\widetilde{C_{4}}\) & \(\widetilde{C_{5}}\) & \(\widetilde{H_{\mathcal{S}}}\) \\
    \hline
    \hline
    \(\widetilde{C_{3}}\) & \(1\) & \(0\) & \(0\) & \(0\) & \(0\) & \(1\) & \(0\) & \(0\) & \(0\) & \(0\) & \(0\) & \(1\) & \(0\) & \(0\) & \(-2\) & \(0\) & \(0\) & \(1\) \\
    \(\widetilde{C_{4}}\) & \(0\) & \(0\) & \(0\) & \(0\) & \(1\) & \(0\) & \(0\) & \(0\) & \(0\) & \(0\) & \(1\) & \(0\) & \(0\) & \(0\) & \(0\) & \(-2\) & \(0\) & \(1\) \\
    \(\widetilde{C_{5}}\) & \(0\) & \(0\) & \(0\) & \(0\) & \(0\) & \(0\) & \(0\) & \(0\) & \(1\) & \(1\) & \(0\) & \(0\) & \(1\) & \(0\) & \(0\) & \(0\) & \(-2\) & \(1\) \\
    \(\widetilde{H_{\mathcal{S}}}\) & \(0\) & \(0\) & \(0\) & \(0\) & \(0\) & \(0\) & \(0\) & \(0\) & \(0\) & \(0\) & \(0\) & \(0\) & \(0\) & \(0\) & \(1\) & \(1\) & \(1\) & \(4\) \\
    \hline
  \end{tabular}.
\end{table}
Note that the intersection matrix is non-degenerate.

Discriminant groups and discriminant forms of the lattices \(L_{\mathcal{S}}\) and \(H \oplus \Pic(X)\) are given by
\begin{gather*}
  G' = 
  \begin{pmatrix}
    \frac{7}{13} & \frac{4}{13} & \frac{1}{13} & \frac{11}{13} & \frac{8}{13} & \frac{6}{13} & \frac{2}{13} & \frac{11}{13} & \frac{7}{13} & \frac{8}{13} & 0 & \frac{5}{13} & \frac{2}{13} & \frac{1}{13} & \frac{10}{13} & \frac{5}{13} & \frac{3}{13} & \frac{2}{13}
  \end{pmatrix}, \\
  G'' = 
  \begin{pmatrix}
    0 & 0 & \frac{7}{13} & \frac{1}{13}
  \end{pmatrix}; \;
  B' = 
  \begin{pmatrix}
    \frac{6}{13}
  \end{pmatrix}, \;
  B'' =
  \begin{pmatrix}
    \frac{7}{13}
  \end{pmatrix}; \;
  Q' =
  \begin{pmatrix}
    \frac{6}{13}
  \end{pmatrix}, \;
  Q'' =
  \begin{pmatrix}
    \frac{20}{13}
  \end{pmatrix}.
\end{gather*}


\subsection{Family \textnumero2.32}\label{subsection:02-32}

The pencil \(\mathcal{S}\) is defined by the equation
\begin{gather*}
  X^{2} Y Z + X Y^{2} Z + X Y Z^{2} + X Z T^{2} + Y Z T^{2} + T^{4} = \lambda X Y Z T.
\end{gather*}
Members \(\mathcal{S}_{\lambda}\) of the pencil are irreducible for any \(\lambda \in \mathbb{P}^1\) except
\(\mathcal{S}_{\infty} = S_{(X)} + S_{(Y)} + S_{(Z)} + S_{(T)}\).
The base locus of the pencil \(\mathcal{S}\) consists of the following curves:
\begin{gather*}
  C_{1} = C_{(X, T)}, \;
  C_{2} = C_{(Y, T)}, \;
  C_{3} = C_{(Z, T)}, \;
  C_{4} = C_{(T, X + Y + Z)}, \;
  C_{5} = C_{(X, Y Z + T^2)}, \;
  C_{6} = C_{(Y, X Z + T^2)}.
\end{gather*}
Their linear equivalence classes on the generic member \(\mathcal{S}_{\Bbbk}\) of the pencil satisfy the following relations:
\begin{gather*}
  \begin{pmatrix}
    [C_{4}] \\ [C_{5}] \\ [C_{6}] \\ [H_{\mathcal{S}}]
  \end{pmatrix} = 
  \begin{pmatrix}
    -1 & -1 & 3 \\
    -2 & 0 & 4 \\
    0 & -2 & 4 \\
    0 & 0 & 4
  \end{pmatrix} \cdot
  \begin{pmatrix}
    [C_{1}] \\ [C_{2}] \\ [C_{3}]
  \end{pmatrix}.
\end{gather*}

For a general choice of \(\lambda \in \mathbb{C}\) the surface \(\mathcal{S}_{\lambda}\) has the following singularities:
\begin{itemize}\setlength{\itemindent}{2cm}
\item[\(P_{1} = P_{(X, Y, T)}\):] type \(\mathbb{A}_4\) with the quadratic term \(X \cdot Y\);
\item[\(P_{2} = P_{(X, Z, T)}\):] type \(\mathbb{A}_3\) with the quadratic term \(X \cdot Z\);
\item[\(P_{3} = P_{(Y, Z, T)}\):] type \(\mathbb{A}_3\) with the quadratic term \(Y \cdot Z\);
\item[\(P_{4} = P_{(X, T, Y + Z)}\):] type \(\mathbb{A}_1\) with the quadratic term \(X (X + Y + Z - \lambda T) + T^2\);
\item[\(P_{5} = P_{(Y, T, X + Z)}\):] type \(\mathbb{A}_1\) with the quadratic term \(Y (X + Y + Z - \lambda T) + T^2\);
\item[\(P_{6} = P_{(Z, T, X + Y)}\):] type \(\mathbb{A}_3\) with the quadratic term \(Z \cdot (X + Y + Z - \lambda T)\).
\end{itemize}

Galois action on the lattice \(L_{\lambda}\) is trivial. The intersection matrix on \(L_{\lambda} = L_{\mathcal{S}}\) is represented by
\begin{table}[H]
  \begin{tabular}{|c||cccc|ccc|ccc|c|c|ccc|ccc|}
    \hline
    \(\bullet\) & \(E_1^1\) & \(E_1^2\) & \(E_1^3\) & \(E_1^4\) & \(E_2^1\) & \(E_2^2\) & \(E_2^3\) & \(E_3^1\) & \(E_3^2\) & \(E_3^3\) & \(E_4^1\) & \(E_5^1\) & \(E_6^1\) & \(E_6^2\) & \(E_6^3\) & \(\widetilde{C_{1}}\) & \(\widetilde{C_{2}}\) & \(\widetilde{C_{3}}\) \\
    \hline
    \hline
    \(\widetilde{C_{1}}\) & \(1\) & \(0\) & \(0\) & \(0\) & \(1\) & \(0\) & \(0\) & \(0\) & \(0\) & \(0\) & \(1\) & \(0\) & \(0\) & \(0\) & \(0\) & \(-2\) & \(0\) & \(0\) \\
    \(\widetilde{C_{2}}\) & \(0\) & \(0\) & \(0\) & \(1\) & \(0\) & \(0\) & \(0\) & \(1\) & \(0\) & \(0\) & \(0\) & \(1\) & \(0\) & \(0\) & \(0\) & \(0\) & \(-2\) & \(0\) \\
    \(\widetilde{C_{3}}\) & \(0\) & \(0\) & \(0\) & \(0\) & \(0\) & \(0\) & \(1\) & \(0\) & \(0\) & \(1\) & \(0\) & \(0\) & \(1\) & \(0\) & \(0\) & \(0\) & \(0\) & \(-2\) \\
    \hline
  \end{tabular}.
\end{table}
Note that the intersection matrix is non-degenerate.

Discriminant groups and discriminant forms of the lattices \(L_{\mathcal{S}}\) and \(H \oplus \Pic(X)\) are given by
\begin{gather*}
  G' = 
  \begin{pmatrix}
    0 & 0 & 0 & 0 & 0 & 0 & 0 & \frac{1}{2} & 0 & \frac{1}{2} & 0 & \frac{1}{2} & \frac{1}{2} & 0 & \frac{1}{2} & 0 & 0 & 0 \\
    0 & \frac{1}{3} & \frac{2}{3} & 0 & \frac{1}{2} & \frac{1}{3} & \frac{1}{6} & 0 & \frac{2}{3} & \frac{1}{3} & \frac{5}{6} & \frac{2}{3} & \frac{1}{2} & 0 & \frac{1}{2} & \frac{2}{3} & \frac{1}{3} & 0
  \end{pmatrix}, \\
  G'' = 
  \begin{pmatrix}
    0 & 0 & \frac{1}{2} & 0 \\
    0 & 0 & \frac{2}{3} & \frac{1}{6}
  \end{pmatrix}; \;
  B' = 
  \begin{pmatrix}
    \frac{1}{2} & 0 \\
    0 & \frac{1}{6}
  \end{pmatrix}, \;
  B'' =
  \begin{pmatrix}
    \frac{1}{2} & 0 \\
    0 & \frac{5}{6}
  \end{pmatrix}; \;
  \begin{pmatrix}
    Q' \\ Q''
  \end{pmatrix}
  =
  \begin{pmatrix}
    \frac{3}{2} & \frac{1}{6} \\
    \frac{1}{2} & \frac{11}{6}
  \end{pmatrix}.
\end{gather*}


\subsection{Family \textnumero2.33}\label{subsection:02-33}

The pencil \(\mathcal{S}\) is defined by the equation
\begin{gather*}
  X^{2} Y Z + X Y^{2} Z + X Y Z^{2} + X^{2} Y T + Z T^{3} = \lambda X Y Z T.
\end{gather*}
Members \(\mathcal{S}_{\lambda}\) of the pencil are irreducible for any \(\lambda \in \mathbb{P}^1\) except
\(\mathcal{S}_{\infty} = S_{(X)} + S_{(Y)} + S_{(Z)} + S_{(T)}\).
The base locus of the pencil \(\mathcal{S}\) consists of the following curves:
\begin{gather*}
  C_1 = C_{(X, Z)}, \;
  C_2 = C_{(X, T)}, \;
  C_3 = C_{(Y, Z)}, \;
  C_4 = C_{(Y, T)}, \;
  C_5 = C_{(Z, T)}, \;
  C_6 = C_{(T, X + Y + Z)}.
\end{gather*}
Their linear equivalence classes on the generic member \(\mathcal{S}_{\Bbbk}\) of the pencil satisfy the following relations:
\begin{gather*}
  \begin{pmatrix}
    [C_{1}] \\ [C_{3}] \\ [C_{5}] \\ [C_{6}]
  \end{pmatrix} = 
  \begin{pmatrix}
    -3 & 0 & 1 \\
    0 & -3 & 1 \\
    6 & 3 & -2 \\
    -7 & -4 & 3
  \end{pmatrix} \cdot
  \begin{pmatrix}
    [C_{2}] \\ [C_{4}] \\ [H_{\mathcal{S}}]
  \end{pmatrix}.
\end{gather*}

For a general choice of \(\lambda \in \mathbb{C}\) the surface \(\mathcal{S}_{\lambda}\) has the following singularities:
\begin{itemize}\setlength{\itemindent}{2cm}
\item[\(P_{1} = P_{(X, Y, T)}\):] type \(\mathbb{A}_2\) with the quadratic term \(X \cdot Y\);
\item[\(P_{2} = P_{(X, Z, T)}\):] type \(\mathbb{A}_6\) with the quadratic term \(X \cdot Z\);
\item[\(P_{3} = P_{(Y, Z, T)}\):] type \(\mathbb{A}_3\) with the quadratic term \(Y \cdot (Z + T)\);
\item[\(P_{4} = P_{(X, T, Y + Z)}\):] type \(\mathbb{A}_2\) with the quadratic term \(X \cdot (X + Y + Z - \lambda T)\);
\item[\(P_{5} = P_{(Y, T, X + Z)}\):] type \(\mathbb{A}_2\) with the quadratic term \(Y \cdot (X + Y + Z - (\lambda + 1) T)\).
\end{itemize}

Galois action on the lattice \(L_{\lambda}\) is trivial. The intersection matrix on \(L_{\lambda} = L_{\mathcal{S}}\) is represented by
\begin{table}[H]
  \begin{tabular}{|c||cc|cccccc|ccc|cc|cc|ccc|}
    \hline
    \(\bullet\) & \(E_1^1\) & \(E_1^2\) & \(E_2^1\) & \(E_2^2\) & \(E_2^3\) & \(E_2^4\) & \(E_2^5\) & \(E_2^6\) & \(E_3^1\) & \(E_3^2\) & \(E_3^3\) & \(E_4^1\) & \(E_4^2\) & \(E_5^1\) & \(E_5^2\) & \(\widetilde{C_{2}}\) & \(\widetilde{C_{4}}\) & \(\widetilde{H_{\mathcal{S}}}\) \\
    \hline
    \hline
    \(\widetilde{C_{2}}\) & \(1\) & \(0\) & \(1\) & \(0\) & \(0\) & \(0\) & \(0\) & \(0\) & \(0\) & \(0\) & \(0\) & \(1\) & \(0\) & \(0\) & \(0\) & \(-2\) & \(0\) & \(1\) \\
    \(\widetilde{C_{4}}\) & \(0\) & \(1\) & \(0\) & \(0\) & \(0\) & \(0\) & \(0\) & \(0\) & \(0\) & \(0\) & \(1\) & \(0\) & \(0\) & \(0\) & \(1\) & \(0\) & \(-2\) & \(1\) \\
    \(\widetilde{H_{\mathcal{S}}}\) & \(0\) & \(0\) & \(0\) & \(0\) & \(0\) & \(0\) & \(0\) & \(0\) & \(0\) & \(0\) & \(0\) & \(0\) & \(0\) & \(0\) & \(0\) & \(1\) & \(1\) & \(4\) \\
    \hline
  \end{tabular}.
\end{table}
Note that the intersection matrix is non-degenerate.

Discriminant groups and discriminant forms of the lattices \(L_{\mathcal{S}}\) and \(H \oplus \Pic(X)\) are given by
\begin{gather*}
  G' = 
  \begin{pmatrix}
    \frac{7}{9} & \frac{2}{9} & 0 & \frac{2}{3} & \frac{1}{3} & 0 & \frac{2}{3} & \frac{1}{3} & \frac{2}{3} & \frac{1}{3} & 0 & \frac{8}{9} & \frac{4}{9} & \frac{5}{9} & \frac{1}{9} & \frac{1}{3} & \frac{2}{3} & 0
  \end{pmatrix}, \\
  G'' = 
  \begin{pmatrix}
    0 & 0 & \frac{8}{9} & \frac{1}{3}
  \end{pmatrix}; \;
  B' = 
  \begin{pmatrix}
    \frac{7}{9}
  \end{pmatrix}, \;
  B'' =
  \begin{pmatrix}
    \frac{2}{9}
  \end{pmatrix}; \;
  Q' =
  \begin{pmatrix}
    \frac{16}{9}
  \end{pmatrix}, \;
  Q'' =
  \begin{pmatrix}
    \frac{2}{9}
  \end{pmatrix}.
\end{gather*}


\subsection{Family \textnumero2.34}\label{subsection:02-34}

The pencil \(\mathcal{S}\) is defined by the equation
\begin{gather*}
  X^{2} Y Z + X Y^{2} Z + X Y Z^{2} + Y Z T^{2} + X T^{3} = \lambda X Y Z T.
\end{gather*}
Members \(\mathcal{S}_{\lambda}\) of the pencil are irreducible for any \(\lambda \in \mathbb{P}^1\) except
\(\mathcal{S}_{\infty} = S_{(X)} + S_{(Y)} + S_{(Z)} + S_{(T)}\).
The base locus of the pencil \(\mathcal{S}\) consists of the following curves:
\begin{gather*}
  C_{1} = C_{(X, Y)}, \;
  C_{2} = C_{(X, Z)}, \;
  C_{3} = C_{(X, T)}, \;
  C_{4} = C_{(Y, T)}, \;
  C_{5} = C_{(Z, T)}, \;
  C_{6} = C_{(T, X + Y + Z)}.
\end{gather*}
Their linear equivalence classes on the generic member \(\mathcal{S}_{\Bbbk}\) of the pencil satisfy the following relations:
\begin{gather*}
  \begin{pmatrix}
    [C_{1}] \\ [C_{2}] \\ [C_{6}] \\ [H_{\mathcal{S}}]
  \end{pmatrix} = 
  \begin{pmatrix}
    -2 & 0 & 3 \\
    -2 & 3 & 0 \\
    -3 & 2 & 2 \\
    -2 & 3 & 3
  \end{pmatrix} \cdot
  \begin{pmatrix}
    [C_{3}] \\ [C_{4}] \\ [C_{5}]
  \end{pmatrix}.
\end{gather*}

For a general choice of \(\lambda \in \mathbb{C}\) the surface \(\mathcal{S}_{\lambda}\) has the following singularities:
\begin{itemize}\setlength{\itemindent}{2cm}
\item[\(P_{1} = P_{(X, Y, T)}\):] type \(\mathbb{A}_4\) with the quadratic term \(X \cdot Y\);
\item[\(P_{2} = P_{(X, Z, T)}\):] type \(\mathbb{A}_4\) with the quadratic term \(X \cdot Z\);
\item[\(P_{3} = P_{(Y, Z, T)}\):] type \(\mathbb{A}_2\) with the quadratic term \(Y \cdot Z\);
\item[\(P_{4} = P_{(X, T, Y + Z)}\):] type \(\mathbb{A}_1\) with the quadratic term \(X (X + Y + Z - \lambda T) + T^2\);
\item[\(P_{5} = P_{(Y, T, X + Z)}\):] type \(\mathbb{A}_2\) with the quadratic term \(Y \cdot (X + Y + Z - \lambda T)\);
\item[\(P_{6} = P_{(Z, T, X + Y)}\):] type \(\mathbb{A}_2\) with the quadratic term \(Z \cdot (X + Y + Z - \lambda T)\).
\end{itemize}

Galois action on the lattice \(L_{\lambda}\) is trivial. The intersection matrix on \(L_{\lambda} = L_{\mathcal{S}}\) is represented by
\begin{table}[H]
  \begin{tabular}{|c||cccc|cccc|cc|c|cc|cc|ccc|}
    \hline
    \(\bullet\) & \(E_1^1\) & \(E_1^2\) & \(E_1^3\) & \(E_1^4\) & \(E_2^1\) & \(E_2^2\) & \(E_2^3\) & \(E_2^4\) & \(E_3^1\) & \(E_3^2\) & \(E_4^1\) & \(E_5^1\) & \(E_5^2\) & \(E_6^1\) & \(E_6^2\) & \(\widetilde{C_{3}}\) & \(\widetilde{C_{4}}\) & \(\widetilde{C_{5}}\) \\
    \hline
    \hline
    \(\widetilde{C_{3}}\) & \(1\) & \(0\) & \(0\) & \(0\) & \(1\) & \(0\) & \(0\) & \(0\) & \(0\) & \(0\) & \(1\) & \(0\) & \(0\) & \(0\) & \(0\) & \(-2\) & \(0\) & \(0\) \\
    \(\widetilde{C_{4}}\) & \(0\) & \(0\) & \(0\) & \(1\) & \(0\) & \(0\) & \(0\) & \(0\) & \(1\) & \(0\) & \(0\) & \(1\) & \(0\) & \(0\) & \(0\) & \(0\) & \(-2\) & \(0\) \\
    \(\widetilde{C_{5}}\) & \(0\) & \(0\) & \(0\) & \(0\) & \(0\) & \(0\) & \(0\) & \(1\) & \(0\) & \(1\) & \(0\) & \(0\) & \(0\) & \(0\) & \(1\) & \(0\) & \(0\) & \(-2\) \\
    \hline
  \end{tabular}.
\end{table}
Note that the intersection matrix is non-degenerate.

Discriminant groups and discriminant forms of the lattices \(L_{\mathcal{S}}\) and \(H \oplus \Pic(X)\) are given by
\begin{gather*}
  G' = 
  \begin{pmatrix}
    \frac{1}{3} & \frac{2}{3} & 0 & \frac{1}{3} & \frac{2}{3} & \frac{1}{3} & 0 & \frac{2}{3} & \frac{8}{9} & \frac{1}{9} & 0 & \frac{1}{9} & \frac{5}{9} & \frac{4}{9} & \frac{8}{9} & 0 & \frac{2}{3} & \frac{1}{3}
  \end{pmatrix}, \\
  G'' = 
  \begin{pmatrix}
    0 & 0 & \frac{7}{9} & \frac{1}{3}
  \end{pmatrix}; \;
  B' = 
  \begin{pmatrix}
    \frac{2}{9}
  \end{pmatrix}, \;
  B'' = 
  \begin{pmatrix}
    \frac{7}{9}
  \end{pmatrix}; \;
  Q' =
  \begin{pmatrix}
    \frac{2}{9}
  \end{pmatrix}, \;
  Q'' =
  \begin{pmatrix}
    \frac{16}{9}
  \end{pmatrix}.
\end{gather*}


\subsection{Family \textnumero2.35}\label{subsection:02-35}

The pencil \(\mathcal{S}\) is defined by the equation
\begin{gather*}
  X^{2} Y Z + X Y^{2} Z + X Y Z^{2} + X^{2} T^{2} + Y Z T^{2} = \lambda X Y Z T.
\end{gather*}
Members \(\mathcal{S}_{\lambda}\) of the pencil are irreducible for any \(\lambda \in \mathbb{P}^1\) except
\(\mathcal{S}_{\infty} = S_{(X)} + S_{(Y)} + S_{(Z)} + S_{(T)}\).
The base locus of the pencil \(\mathcal{S}\) consists of the following curves:
\begin{gather*}
  C_1 = C_{(X, Y)}, \;
  C_2 = C_{(X, Z)}, \;
  C_3 = C_{(X, T)}, \;
  C_4 = C_{(Y, T)}, \;
  C_5 = C_{(Z, T)}, \;
  C_6 = C_{(T, X + Y + Z)}.
\end{gather*}
Their linear equivalence classes on the generic member \(\mathcal{S}_{\Bbbk}\) of the pencil satisfy the following relations:
\[
  \begin{pmatrix}
    [C_4] + [C_5] + [C_6] \\ 2 [C_5] \\ 2 [C_6] \\ [H_{\mathcal{S}}]
  \end{pmatrix} =
  \begin{pmatrix}
    1 & 1 & 1 \\
    1 & -1 & 2 \\
    2 & 2 & -2 \\
    1 & 1 & 2
  \end{pmatrix} \cdot
  \begin{pmatrix}
    [C_1] \\ [C_2] \\ [C_3]
  \end{pmatrix}.
\]

For a general choice of \(\lambda \in \mathbb{C}\) the surface \(\mathcal{S}_{\lambda}\) has the following singularities:
\begin{itemize}\setlength{\itemindent}{2cm}
\item[\(P_{1} = P_{(X, Y, Z)}\):] type \(\mathbb{A}_1\) with the quadratic term \(X^2 + Y Z\);
\item[\(P_{2} = P_{(X, Y, T)}\):] type \(\mathbb{A}_5\) with the quadratic term \(X \cdot Y \);
\item[\(P_{3} = P_{(X, Z, T)}\):] type \(\mathbb{A}_5\) with the quadratic term \(X \cdot Z\);
\item[\(P_{4} = P_{(Y, Z, T)}\):] type \(\mathbb{A}_1\) with the quadratic term \(Y Z + T^2\);
\item[\(P_{5} = P_{(X, T, Y + Z)}\):] type \(\mathbb{A}_1\) with the quadratic term \(X (X + Y + Z - \lambda T) + T^2\);
\item[\(P_{6} = P_{(Y, T, X + Z)}\):] type \(\mathbb{A}_1\) with the quadratic term \(Y (X + Y + Z - \lambda T) - T^2\);
\item[\(P_{7} = P_{(Z, T, X + Y)}\):] type \(\mathbb{A}_1\) with the quadratic term \(Z (X + Y + Z - \lambda T) - T^2\).
\end{itemize}

Galois action on the lattice \(L_{\lambda}\) is trivial. The intersection matrix on \(L_{\lambda} = L_{\mathcal{S}}\) is represented by
\begin{table}[H]
  \setlength{\tabcolsep}{4pt}
  \begin{tabular}{|c||ccccc|ccccc|c|c|c|c|c|cccccc|}
    \hline
    \(\bullet\) & \(E_1^1\) & \(E_2^1\) & \(E_2^2\) & \(E_2^3\) & \(E_2^4\) & \(E_2^5\) & \(E_3^1\) & \(E_3^2\) & \(E_3^3\) & \(E_3^4\) & \(E_3^5\) & \(E_4^1\) & \(E_5^1\) & \(E_6^1\) & \(E_7^1\) & \(\widetilde{C_{1}}\) & \(\widetilde{C_{2}}\) & \(\widetilde{C_{3}}\) & \(\widetilde{C_{4}}\) & \(\widetilde{C_{5}}\) & \(\widetilde{C_{6}}\) \\
    \hline
    \hline
    \(\widetilde{C_{1}}\) & \(1\) & \(0\) & \(0\) & \(0\) & \(1\) & \(0\) & \(0\) & \(0\) & \(0\) & \(0\) & \(0\) & \(0\) & \(0\) & \(0\) & \(0\) & \(-2\) & \(0\) & \(0\) & \(0\) & \(0\) & \(0\) \\
    \(\widetilde{C_{2}}\) & \(1\) & \(0\) & \(0\) & \(0\) & \(0\) & \(0\) & \(0\) & \(0\) & \(0\) & \(1\) & \(0\) & \(0\) & \(0\) & \(0\) & \(0\) & \(0\) & \(-2\) & \(0\) & \(0\) & \(0\) & \(0\) \\
    \(\widetilde{C_{3}}\) & \(0\) & \(1\) & \(0\) & \(0\) & \(0\) & \(0\) & \(1\) & \(0\) & \(0\) & \(0\) & \(0\) & \(0\) & \(1\) & \(0\) & \(0\) & \(0\) & \(0\) & \(-2\) & \(0\) & \(0\) & \(0\) \\
    \(\widetilde{C_{4}}\) & \(0\) & \(0\) & \(0\) & \(0\) & \(0\) & \(1\) & \(0\) & \(0\) & \(0\) & \(0\) & \(0\) & \(1\) & \(0\) & \(1\) & \(0\) & \(0\) & \(0\) & \(0\) & \(-2\) & \(0\) & \(0\) \\
    \(\widetilde{C_{5}}\) & \(0\) & \(0\) & \(0\) & \(0\) & \(0\) & \(0\) & \(0\) & \(0\) & \(0\) & \(0\) & \(1\) & \(1\) & \(0\) & \(0\) & \(1\) & \(0\) & \(0\) & \(0\) & \(0\) & \(-2\) & \(0\) \\
    \(\widetilde{C_{6}}\) & \(0\) & \(0\) & \(0\) & \(0\) & \(0\) & \(0\) & \(0\) & \(0\) & \(0\) & \(0\) & \(0\) & \(0\) & \(1\) & \(1\) & \(1\) & \(0\) & \(0\) & \(0\) & \(0\) & \(0\) & \(-2\) \\
    \hline
  \end{tabular}.
\end{table}
Note that the intersection matrix is degenerate. We choose the following integral basis of the lattice \(L_{\lambda}\):
\begin{align*}
  \begin{pmatrix}
    [E_7^1] \\ [\widetilde{C_{2}}] \\ [\widetilde{C_{5}}]
  \end{pmatrix} = 
  \begin{pmatrix}
    2 & -3 & 0 & 3 & 6 & 5 & -5 & -4 & -3 & -2 & -1 & 2 & -4 & 1 & 4 & -6 & 4 & -2 \\
    0 & -1 & 0 & 1 & 2 & 2 & -2 & -2 & -2 & -2 & -1 & 1 & -1 & 1 & 1 & -2 & 2 & 0 \\
    -1 & 3 & 1 & -1 & -3 & -3 & 4 & 3 & 2 & 1 & 0 & -2 & 3 & -1 & -2 & 5 & -3 & 1
  \end{pmatrix} \cdot \\
  \begin{pmatrix}
    [E_1^1] & [E_2^1] & [E_2^2] & [E_2^3] & [E_2^4] & [E_2^5] & [E_3^1] & [E_3^2] & [E_3^3] & \\ [E_3^4] & [E_3^5] & [E_4^1] & [E_5^1] & [E_6^1] & [\widetilde{C_{1}}] & [\widetilde{C_{3}}] & [\widetilde{C_{4}}] & [\widetilde{C_{6}}]
  \end{pmatrix}^T.
\end{align*}

Discriminant groups and discriminant forms of the lattices \(L_{\mathcal{S}}\) and \(H \oplus \Pic(X)\) are given by
\begin{gather*}
  G' = 
  \begin{pmatrix}
    0 & \frac{1}{2} & 0 & \frac{1}{2} & 0 & \frac{1}{2} & \frac{1}{2} & 0 & \frac{1}{2} & 0 & \frac{1}{2} & \frac{1}{2} & 0 & 0 & 0 & 0 & 0 & 0 \\
    \frac{1}{2} & \frac{1}{2} & \frac{1}{2} & \frac{1}{2} & \frac{1}{2} & \frac{1}{2} & \frac{3}{4} & 0 & \frac{1}{4} & \frac{1}{2} & \frac{3}{4} & \frac{1}{4} & \frac{3}{4} & \frac{1}{4} & 0 & \frac{1}{2} & \frac{1}{2} & 0
  \end{pmatrix}, \\
  G'' = 
  \begin{pmatrix}
    0 & 0 & \frac{1}{2} & 0 \\
    0 & 0 & 0 & \frac{1}{4}
  \end{pmatrix}; \;
  B' = 
  \begin{pmatrix}
    \frac{1}{2} & 0 \\
    0 & \frac{3}{4}
  \end{pmatrix}, \;
  B'' =
  \begin{pmatrix}
    \frac{1}{2} & 0 \\
    0 & \frac{1}{4}
  \end{pmatrix}; \;
  \begin{pmatrix}
    Q' \\ Q''
  \end{pmatrix}
  =
  \begin{pmatrix}
    \frac{1}{2} & \frac{7}{4} \\
    \frac{3}{2} & \frac{1}{4}    
  \end{pmatrix}.
\end{gather*}


\subsection{Family \textnumero2.36}\label{subsection:02-36}

The pencil \(\mathcal{S}\) is defined by the equation
\[
  X^2 Y Z + X Y^2 Z + X Y Z^2 + X^3 T + Y Z T^2 = \lambda X Y Z T.
\]
Members \(\mathcal{S}_{\lambda}\) of the pencil are irreducible for any \(\lambda \in \mathbb{P}^1\) except
\(\mathcal{S}_{\infty} = S_{(X)} + S_{(Y)} + S_{(Z)} + S_{(T)}\).
The base locus of the pencil \(\mathcal{S}\) consists of the following curves:
\[
  C_1 = C_{(X, Y)}, \;
  C_2 = C_{(X, Z)}, \;
  C_3 = C_{(X, T)}, \;
  C_4 = C_{(Y, T)}, \;
  C_5 = C_{(Z, T)}, \;
  C_6 = C_{(T, X + Y + Z)}.
\]
Their linear equivalence classes on the generic member \(\mathcal{S}_{\Bbbk}\) of the pencil satisfy the following relations:
\begin{gather*}
  \begin{pmatrix}
    [C_{2}] \\ [C_{4}] \\ [C_{5}] \\ [C_{6}]
  \end{pmatrix} = 
  \begin{pmatrix}
    -1 & -2 & 1 \\
    -3 & 0 & 1 \\
    3 & 6 & -2 \\
    0 & -7 & 2
  \end{pmatrix} \cdot
  \begin{pmatrix}
    [C_{1}] \\ [C_{3}] \\ [H_{\mathcal{S}}]
  \end{pmatrix}.
\end{gather*}

For a general choice of \(\lambda \in \mathbb{C}\) the surface \(\mathcal{S}_{\lambda}\) has the following singularities:
\begin{itemize}\setlength{\itemindent}{2cm}
\item[\(P_{1} = P_{(X, Y, Z)}\):] type \(\mathbb{A}_2\) with the quadratic term \(Y \cdot Z\);
\item[\(P_{2} = P_{(X, Y, T)}\):] type \(\mathbb{A}_6\) with the quadratic term \(X \cdot Y\);
\item[\(P_{3} = P_{(X, Z, T)}\):] type \(\mathbb{A}_6\) with the quadratic term \(X \cdot Z\);
\item[\(P_{4} = P_{(X, T, Y + Z)}\):] type \(\mathbb{A}_1\) with the quadratic term \(X (X + Y + Z - \lambda T) + T^2\).
\end{itemize}

Galois action on the lattice \(L_{\lambda}\) is trivial. The intersection matrix on \(L_{\lambda} = L_{\mathcal{S}}\) is represented by
\begin{table}[H]
  \begin{tabular}{|c||cc|cccccc|cccccc|c|ccc|}
    \hline
    \(\bullet\) & \(E_1^1\) & \(E_1^2\) & \(E_2^1\) & \(E_2^2\) & \(E_2^3\) & \(E_2^4\) & \(E_2^5\) & \(E_2^6\) & \(E_3^1\) & \(E_3^2\) & \(E_3^3\) & \(E_3^4\) & \(E_3^5\) & \(E_3^6\) & \(E_4^1\) & \(\widetilde{C_{1}}\) & \(\widetilde{C_{3}}\) & \(\widetilde{H_{\mathcal{S}}}\) \\
    \hline
    \hline
    \(\widetilde{C_{1}}\) & \(1\) & \(0\) & \(0\) & \(0\) & \(0\) & \(0\) & \(1\) & \(0\) & \(0\) & \(0\) & \(0\) & \(0\) & \(0\) & \(0\) & \(0\) & \(-2\) & \(0\) & \(1\) \\
    \(\widetilde{C_{3}}\) & \(0\) & \(0\) & \(1\) & \(0\) & \(0\) & \(0\) & \(0\) & \(0\) & \(1\) & \(0\) & \(0\) & \(0\) & \(0\) & \(0\) & \(1\) & \(0\) & \(-2\) & \(1\) \\
    \(\widetilde{H_{\mathcal{S}}}\) & \(0\) & \(0\) & \(0\) & \(0\) & \(0\) & \(0\) & \(0\) & \(0\) & \(0\) & \(0\) & \(0\) & \(0\) & \(0\) & \(0\) & \(0\) & \(1\) & \(1\) & \(4\) \\
    \hline
  \end{tabular}.
\end{table}
Note that the intersection matrix is non-degenerate.

Discriminant groups and discriminant forms of the lattices \(L_{\mathcal{S}}\) and \(H \oplus \Pic(X)\) are given by
\begin{gather*}
  G' = 
  \begin{pmatrix}
    \frac{2}{5} & \frac{1}{5} & \frac{2}{5} & \frac{1}{5} & 0 & \frac{4}{5} & \frac{3}{5} & \frac{4}{5} & \frac{4}{5} & 0 & \frac{1}{5} & \frac{2}{5} & \frac{3}{5} & \frac{4}{5} & \frac{4}{5} & \frac{3}{5} & \frac{3}{5} & \frac{1}{5}
  \end{pmatrix}, \\
  G'' = 
  \begin{pmatrix}
    0 & 0 & 0 & \frac{1}{5}
  \end{pmatrix}; \;
  B' = 
  \begin{pmatrix}
    \frac{3}{5}
  \end{pmatrix}, \;
  B'' =
  \begin{pmatrix}
    \frac{2}{5}
  \end{pmatrix}; \;
  Q' =
  \begin{pmatrix}
    \frac{8}{5}
  \end{pmatrix}, \;
  Q'' =
  \begin{pmatrix}
    \frac{2}{5}
  \end{pmatrix}.
\end{gather*}


\section{Dolgachev--Nikulin duality for Fano threefolds: rank 3}\label{appendix:rank-03}
\subsection{Family \textnumero3.1}\label{subsection:03-01}

The pencil \(\mathcal{S}\) is defined by the equation
\begin{gather*}
  X^{2} Y Z + X Y^{2} Z + X^{2} Z^{2} + 3 X Y Z^{2} + Y^{2} Z^{2} + X Z^{3} + Y Z^{3} + X^{2} Y T + X Y^{2} T + 2 X^{2} Z T + 2 Y^{2} Z T + \\ 3 X Z^{2} T + 3 Y Z^{2} T + X^{2} T^{2} + 3 X Y T^{2} + Y^{2} T^{2} + 3 X Z T^{2} + 3 Y Z T^{2} + X T^{3} + Y T^{3} = \lambda X Y Z T.
\end{gather*}
Members \(\mathcal{S}_{\lambda}\) of the pencil are irreducible for any \(\lambda \in \mathbb{P}^1\) except
\begin{gather*}
  \mathcal{S}_{\infty} = S_{(X)} + S_{(Y)} + S_{(Z)} + S_{(T)}, \;
  \mathcal{S}_{- 6} = S_{(Z + T)} + S_{(X + Y + Z + T)} + S_{((X + Y) (Z + T) + X Y)}.
\end{gather*}
The base locus of the pencil \(\mathcal{S}\) consists of the following curves:
\begin{gather*}
  C_1 = C_{(X, Y)}, \;
  C_2 = C_{(Z, T)}, \;
  C_3 = C_{(X, Z + T)}, \;
  C_4 = C_{(Y, Z + T)}, \;
  C_5 = C_{(X, Y + Z + T)}, \\
  C_6 = C_{(Y, X + Z + T)}, \;
  C_7 = C_{(Z, X + Y + T)}, \;
  C_8 = C_{(T, X + Y + Z)}, \;
  C_9 = C_{(Z, X Y + T (X + Y))}, \;
  C_{10} = C_{(T, X Y + Z (X + Y))}.
\end{gather*}
Their linear equivalence classes on the generic member \(\mathcal{S}_{\Bbbk}\) of the pencil satisfy the following relations:
\begin{gather*}
  \begin{pmatrix}
    [C_{4}] \\ [C_{5}] \\ [C_{6}] \\ [C_{8}] \\ [C_{9}] \\ [C_{10}]
  \end{pmatrix} = 
  \begin{pmatrix}
    0 & -2 & -1 & 0 & 1 \\
    -1 & 0 & -2 & 0 & 1 \\
    -1 & 4 & 2 & 0 & -1 \\
    2 & -4 & 0 & -1 & 1 \\
    0 & -1 & 0 & -1 & 1 \\
    -2 & 3 & 0 & 1 & 0
  \end{pmatrix} \cdot
  \begin{pmatrix}
    [C_{1}] \\ [C_{2}] \\ [C_{3}] \\ [C_{7}] \\ [H_{\mathcal{S}}]
  \end{pmatrix}.
\end{gather*}

For a general choice of \(\lambda \in \mathbb{C}\) the surface \(\mathcal{S}_{\lambda}\) has the following singularities:
\begin{itemize}\setlength{\itemindent}{2cm}
\item[\(P_{1} = P_{(X, Z, T)}\):] type \(\mathbb{A}_3\) with the quadratic term \((Z + T) \cdot (X + Z + T)\);
\item[\(P_{2} = P_{(Y, Z, T)}\):] type \(\mathbb{A}_3\) with the quadratic term \((Z + T) \cdot (Y + Z + T)\);
\item[\(P_{3} = P_{(X, Y, Z + T)}\):] type \(\mathbb{A}_5\) with the quadratic term \((\lambda + 6) X \cdot Y\);
\item[\(P_{4} = P_{(Z, T, X + Y)}\):] type \(\mathbb{A}_1\) with the quadratic term \((Z + T) (X + Y + Z + T) - (\lambda + 6) Z T\).
\end{itemize}

Galois action on the lattice \(L_{\lambda}\) is trivial. The intersection matrix on \(L_{\lambda} = L_{\mathcal{S}}\) is represented by
\begin{table}[H]
  \begin{tabular}{|c||ccc|ccc|ccccc|c|ccccc|}
    \hline
    \(\bullet\) & \(E_1^1\) & \(E_1^2\) & \(E_1^3\) & \(E_2^1\) & \(E_2^2\) & \(E_2^3\) & \(E_3^1\) & \(E_3^2\) & \(E_3^3\) & \(E_3^4\) & \(E_3^5\) & \(E_4^1\) & \(\widetilde{C_{1}}\) & \(\widetilde{C_{2}}\) & \(\widetilde{C_{3}}\) & \(\widetilde{C_{7}}\) & \(\widetilde{H_{\mathcal{S}}}\) \\
    \hline
    \hline
    \(\widetilde{C_{1}}\) & \(0\) & \(0\) & \(0\) & \(0\) & \(0\) & \(0\) & \(0\) & \(0\) & \(1\) & \(0\) & \(0\) & \(0\) & \(-2\) & \(0\) & \(0\) & \(0\) & \(1\) \\
    \(\widetilde{C_{2}}\) & \(1\) & \(0\) & \(0\) & \(1\) & \(0\) & \(0\) & \(0\) & \(0\) & \(0\) & \(0\) & \(0\) & \(1\) & \(0\) & \(-2\) & \(0\) & \(0\) & \(1\) \\
    \(\widetilde{C_{3}}\) & \(0\) & \(1\) & \(0\) & \(0\) & \(0\) & \(0\) & \(1\) & \(0\) & \(0\) & \(0\) & \(0\) & \(0\) & \(0\) & \(0\) & \(-2\) & \(0\) & \(1\) \\
    \(\widetilde{C_{7}}\) & \(0\) & \(0\) & \(0\) & \(0\) & \(0\) & \(0\) & \(0\) & \(0\) & \(0\) & \(0\) & \(0\) & \(1\) & \(0\) & \(0\) & \(0\) & \(-2\) & \(1\) \\
    \(\widetilde{H_{\mathcal{S}}}\) & \(0\) & \(0\) & \(0\) & \(0\) & \(0\) & \(0\) & \(0\) & \(0\) & \(0\) & \(0\) & \(0\) & \(0\) & \(1\) & \(1\) & \(1\) & \(1\) & \(4\) \\
    \hline
  \end{tabular}.
\end{table}
Note that the intersection matrix is non-degenerate.

Discriminant groups and discriminant forms of the lattices \(L_{\mathcal{S}}\) and \(H \oplus \Pic(X)\) are given by
\begin{gather*}
  G' = 
  \begin{pmatrix}
    \frac{1}{2} & 0 & \frac{1}{2} & \frac{1}{2} & 0 & \frac{1}{2} & 0 & 0 & 0 & 0 & 0 & 0 & 0 & 0 & 0 & 0 & 0 \\
    0 & 0 & \frac{1}{2} & 0 & 0 & 0 & 0 & \frac{1}{2} & 0 & 0 & 0 & 0 & \frac{1}{2} & 0 & \frac{1}{2} & 0 & 0 \\
    \frac{1}{4} & \frac{1}{2} & \frac{3}{4} & \frac{3}{4} & \frac{1}{2} & \frac{1}{4} & \frac{1}{4} & \frac{1}{2} & \frac{3}{4} & \frac{1}{2} & \frac{1}{4} & \frac{3}{4} & \frac{1}{2} & 0 & 0 & \frac{1}{2} & \frac{1}{4}
  \end{pmatrix}, \\
  G'' = 
  \begin{pmatrix}
    0 & 0 & \frac{1}{2} & 0 & 0 \\
    0 & 0 & 0 & \frac{1}{2} & 0 \\
    0 & 0 & -\frac{1}{4} & -\frac{1}{4} & \frac{1}{4}
  \end{pmatrix}; \;
  B' = 
  \begin{pmatrix}
    0 & \frac{1}{2} & 0 \\
    \frac{1}{2} & 0 & 0 \\
    0 & 0 & \frac{1}{4}
  \end{pmatrix}, \;
  B'' = 
  \begin{pmatrix}
    0 & \frac{1}{2} & 0 \\
    \frac{1}{2} & 0 & 0 \\
    0 & 0 & \frac{3}{4}
  \end{pmatrix}; \;
  \begin{pmatrix}
    Q' \\ Q''
  \end{pmatrix}
  =
  \begin{pmatrix}
    0 & 0 & \frac{1}{4} \\
    0 & 0 & \frac{7}{4}
  \end{pmatrix}.
\end{gather*}


\subsection{Family \textnumero3.2}\label{subsection:03-02}

The pencil \(\mathcal{S}\) is defined by the equation
\begin{gather*}
  X^{2} Y Z + X Y^{2} Z + X Y Z^{2} + X^{3} T + 3 X^{2} Y T + 3 X Y^{2} T + Y^{3} T + 3 X^{2} Z T + \\ 3 Y^{2} Z T + 3 X Z^{2} T + 3 Y Z^{2} T + Z^{3} T + X Z T^{2} + Y Z T^{2} + Z^{2} T^{2} = \lambda X Y Z T.
\end{gather*}
Members \(\mathcal{S}_{\lambda}\) of the pencil are irreducible for any \(\lambda \in \mathbb{P}^1\) except
\begin{gather*}
  \mathcal{S}_{\infty} = S_{(X)} + S_{(Y)} + S_{(Z)} + S_{(T)}, \;
  \mathcal{S}_{- 6} = S_{(X + Y + Z)} + S_{(T ((X + Y + Z)^2 + Z T) + X Y Z)}.
\end{gather*}
The base locus of the pencil \(\mathcal{S}\) consists of the following curves:
\begin{gather*}
  C_{1} = C_{(X, T)}, \;
  C_{2} = C_{(Y, T)}, \;
  C_{3} = C_{(Z, T)}, \;
  C_{4} = C_{(X, Y + Z)}, \;
  C_{5} = C_{(Y, X + Z)}, \\
  C_{6} = C_{(Z, X + Y)}, \;
  C_{7} = C_{(T, X + Y + Z)}, \;
  C_{8} = C_{(X, (Y + Z)^2 + Z T)}, \;
  C_{9} = C_{(Y, (X + Z)^2 + Z T)}.
\end{gather*}
Their linear equivalence classes on the generic member \(\mathcal{S}_{\Bbbk}\) of the pencil satisfy the following relations:
\begin{gather*}
  \begin{pmatrix}
    [C_{3}] \\ [C_{5}] \\ [C_{7}] \\ [C_{8}] \\ [C_{9}]
  \end{pmatrix} = 
  \begin{pmatrix}
    0 & 0 & 0 & -3 & 1 \\
    1 & 1 & -1 & -4 & 1 \\
    -1 & -1 & 0 & 3 & 0 \\
    -1 & 0 & -1 & 0 & 1 \\
    -1 & -2 & 1 & 4 & 0
  \end{pmatrix} \cdot
  \begin{pmatrix}
    [C_{1}] \\ [C_{2}] \\ [C_{4}] \\ [C_{6}] \\ [H_{\mathcal{S}}]
  \end{pmatrix}.
\end{gather*}

For a general choice of \(\lambda \in \mathbb{C}\) the surface \(\mathcal{S}_{\lambda}\) has the following singularities:
\begin{itemize}\setlength{\itemindent}{2cm}
\item[\(P_{1} = P_{(X, Y, Z)}\):] type \(\mathbb{A}_5\) with the quadratic term \(Z \cdot (X + Y + Z)\);
\item[\(P_{2} = P_{(X, T, Y + Z)}\):] type \(\mathbb{A}_2\) with the quadratic term \(X \cdot (X + Y + Z - (\lambda + 6) T)\);
\item[\(P_{3} = P_{(Y, T, X + Z)}\):] type \(\mathbb{A}_2\) with the quadratic term \(Y \cdot (X + Y + Z - (\lambda + 6) T)\);
\item[\(P_{4} = P_{(Z, T, X + Y)}\):] type \(\mathbb{A}_3\) with the quadratic term \(Z \cdot (X + Y + Z - (\lambda + 6) T)\).
\end{itemize}

Galois action on the lattice \(L_{\lambda}\) is trivial. The intersection matrix on \(L_{\lambda} = L_{\mathcal{S}}\) is represented by
\begin{table}[H]
  \begin{tabular}{|c||ccccc|cc|cc|ccc|ccccc|}
    \hline
    \(\bullet\) & \(E_1^1\) & \(E_1^2\) & \(E_1^3\) & \(E_1^4\) & \(E_1^5\) & \(E_2^1\) & \(E_2^2\) & \(E_3^1\) & \(E_3^2\) & \(E_4^1\) & \(E_4^2\) & \(E_4^3\) & \(\widetilde{C_{1}}\) & \(\widetilde{C_{2}}\) & \(\widetilde{C_{4}}\) & \(\widetilde{C_{6}}\) & \(\widetilde{H_{\mathcal{S}}}\) \\
    \hline
    \hline
    \(\widetilde{C_{1}}\) & \(0\) & \(0\) & \(0\) & \(0\) & \(0\) & \(1\) & \(0\) & \(0\) & \(0\) & \(0\) & \(0\) & \(0\) & \(-2\) & \(1\) & \(0\) & \(0\) & \(1\) \\
    \(\widetilde{C_{2}}\) & \(0\) & \(0\) & \(0\) & \(0\) & \(0\) & \(0\) & \(0\) & \(1\) & \(0\) & \(0\) & \(0\) & \(0\) & \(1\) & \(-2\) & \(0\) & \(0\) & \(1\) \\
    \(\widetilde{C_{4}}\) & \(1\) & \(0\) & \(0\) & \(0\) & \(0\) & \(1\) & \(0\) & \(0\) & \(0\) & \(0\) & \(0\) & \(0\) & \(0\) & \(0\) & \(-2\) & \(0\) & \(1\) \\
    \(\widetilde{C_{6}}\) & \(0\) & \(0\) & \(0\) & \(1\) & \(0\) & \(0\) & \(0\) & \(0\) & \(0\) & \(1\) & \(0\) & \(0\) & \(0\) & \(0\) & \(0\) & \(-2\) & \(1\) \\
    \(\widetilde{H_{\mathcal{S}}}\) & \(0\) & \(0\) & \(0\) & \(0\) & \(0\) & \(0\) & \(0\) & \(0\) & \(0\) & \(0\) & \(0\) & \(0\) & \(1\) & \(1\) & \(1\) & \(1\) & \(4\) \\
    \hline
  \end{tabular}.
\end{table}
Note that the intersection matrix is non-degenerate.

Discriminant groups and discriminant forms of the lattices \(L_{\mathcal{S}}\) and \(H \oplus \Pic(X)\) are given by
\begin{gather*}
  G' = 
  \begin{pmatrix}
    \frac{3}{16} & \frac{1}{4} & \frac{5}{16} & \frac{3}{8} & \frac{3}{16} & \frac{5}{8} & \frac{13}{16} & \frac{3}{8} & \frac{3}{16} & \frac{11}{16} & \frac{1}{8} & \frac{9}{16} & \frac{5}{16} & \frac{9}{16} & \frac{1}{8} & \frac{1}{4} & \frac{7}{16}
  \end{pmatrix}, \\
  G'' = 
  \begin{pmatrix}
    0 & 0 & -\frac{5}{8} & -\frac{1}{16} & \frac{1}{8}
  \end{pmatrix}; \;
  B' = 
  \begin{pmatrix}
    \frac{1}{16}
  \end{pmatrix}, \;
  B'' = 
  \begin{pmatrix}
    \frac{15}{16}
  \end{pmatrix}; \;
  Q' =
  \begin{pmatrix}
    \frac{1}{16}
  \end{pmatrix}, \;
  Q'' =
  \begin{pmatrix}
    \frac{31}{16}
  \end{pmatrix}.
\end{gather*}


\subsection{Family \textnumero3.3}\label{subsection:03-03}

The pencil \(\mathcal{S}\) is defined by the equation
\begin{gather*}
  X^{3} Y + 2 X^{2} Y Z + X Y^{2} Z + X Y Z^{2} + Y^{2} Z^{2} + 2 X^{2} Y T + X^{2} Z T + \\ Y^{2} Z T + Y Z^{2} T + X Y T^{2} + 2 X Z T^{2} + 2 Y Z T^{2} + Z T^{3} = \lambda X Y Z T.
\end{gather*}
Members \(\mathcal{S}_{\lambda}\) of the pencil are irreducible for any \(\lambda \in \mathbb{P}^1\) except
\(\mathcal{S}_{\infty} = S_{(X)} + S_{(Y)} + S_{(Z)} + S_{(T)}\).
The base locus of the pencil \(\mathcal{S}\) consists of the following curves:
\begin{gather*}
  C_{1} = C_{(X, Z)}, \;
  C_{2} = C_{(Y, Z)}, \;
  C_{3} = C_{(Y, T)}, \;
  C_{4} = C_{(X, Y + T)}, \;
  C_{5} = C_{(Y, X + T)}, \\
  C_{6} = C_{(Z, X + T)}, \;
  C_{7} = C_{(T, X + Z)}, \;
  C_{8} = C_{(X, Y (Z + T) + T^2)}, \;
  C_{9} = C_{(T, X^2 + Z (X + Y))}.
\end{gather*}
Their linear equivalence classes on the generic member \(\mathcal{S}_{\Bbbk}\) of the pencil satisfy the following relations:
\begin{gather*}
  \begin{pmatrix}
    [C_{2}] \\ [C_{3}] \\ [C_{8}] \\ [C_{9}]
  \end{pmatrix} = 
  \begin{pmatrix}
    -1 & 0 & 0 & -2 & 0 & 1 \\
    1 & 0 & -2 & 2 & 0 & 0 \\
    -1 & -1 & 0 & 0 & 0 & 1 \\
    -1 & 0 & 2 & -2 & -1 & 1
  \end{pmatrix} \cdot
  \begin{pmatrix}
    [C_{1}] & [C_{4}] & [C_{5}] & [C_{6}] & [C_{7}] & [H_{\mathcal{S}}]
  \end{pmatrix}^T.
\end{gather*}

For a general choice of \(\lambda \in \mathbb{C}\) the surface \(\mathcal{S}_{\lambda}\) has the following singularities:
\begin{itemize}\setlength{\itemindent}{2cm}
\item[\(P_{1} = P_{(X, Y, T)}\):] type \(\mathbb{A}_4\) with the quadratic term \(Y \cdot (X + Y + T)\);
\item[\(P_{2} = P_{(X, Z, T)}\):] type \(\mathbb{A}_4\) with the quadratic term \(Z \cdot (X + Z + T)\);
\item[\(P_{3} = P_{(Y, Z, X + T)}\):] type \(\mathbb{A}_3\) with the quadratic term \((\lambda + 4) Y \cdot Z\).
\end{itemize}

Galois action on the lattice \(L_{\lambda}\) is trivial. The intersection matrix on \(L_{\lambda} = L_{\mathcal{S}}\) is represented by
\begin{table}[H]
  \begin{tabular}{|c||cccc|cccc|ccc|cccccc|}
    \hline
    \(\bullet\) & \(E_1^1\) & \(E_1^2\) & \(E_1^3\) & \(E_1^4\) & \(E_2^1\) & \(E_2^2\) & \(E_2^3\) & \(E_2^4\) & \(E_3^1\) & \(E_3^2\) & \(E_3^3\) & \(\widetilde{C_{1}}\) & \(\widetilde{C_{4}}\) & \(\widetilde{C_{5}}\) & \(\widetilde{C_{6}}\) & \(\widetilde{C_{7}}\) & \(\widetilde{H_{\mathcal{S}}}\) \\
    \hline
    \hline
    \(\widetilde{C_{1}}\) & \(0\) & \(0\) & \(0\) & \(0\) & \(1\) & \(0\) & \(0\) & \(0\) & \(0\) & \(0\) & \(0\) & \(-2\) & \(1\) & \(0\) & \(0\) & \(0\) & \(1\) \\
    \(\widetilde{C_{4}}\) & \(1\) & \(0\) & \(0\) & \(0\) & \(0\) & \(0\) & \(0\) & \(0\) & \(0\) & \(0\) & \(0\) & \(1\) & \(-2\) & \(0\) & \(0\) & \(0\) & \(1\) \\
    \(\widetilde{C_{5}}\) & \(0\) & \(0\) & \(1\) & \(0\) & \(0\) & \(0\) & \(0\) & \(0\) & \(1\) & \(0\) & \(0\) & \(0\) & \(0\) & \(-2\) & \(0\) & \(0\) & \(1\) \\
    \(\widetilde{C_{6}}\) & \(0\) & \(0\) & \(0\) & \(0\) & \(0\) & \(1\) & \(0\) & \(0\) & \(0\) & \(0\) & \(1\) & \(0\) & \(0\) & \(0\) & \(-2\) & \(0\) & \(1\) \\
    \(\widetilde{C_{7}}\) & \(0\) & \(0\) & \(0\) & \(0\) & \(0\) & \(0\) & \(0\) & \(1\) & \(0\) & \(0\) & \(0\) & \(0\) & \(0\) & \(0\) & \(0\) & \(-2\) & \(1\) \\
    \(\widetilde{H_{\mathcal{S}}}\) & \(0\) & \(0\) & \(0\) & \(0\) & \(0\) & \(0\) & \(0\) & \(0\) & \(0\) & \(0\) & \(0\) & \(1\) & \(1\) & \(1\) & \(1\) & \(1\) & \(4\) \\
    \hline
  \end{tabular}.
\end{table}
Note that the intersection matrix is non-degenerate.

Discriminant groups and discriminant forms of the lattices \(L_{\mathcal{S}}\) and \(H \oplus \Pic(X)\) are given by
\begin{gather*}
  G' = 
  \begin{pmatrix}
    \frac{3}{7} & \frac{15}{28} & \frac{9}{14} & \frac{23}{28} & \frac{1}{28} & \frac{3}{14} & \frac{3}{4} & \frac{2}{7} & \frac{6}{7} & \frac{11}{14} & \frac{5}{7} & \frac{6}{7} & \frac{9}{28} & \frac{13}{14} & \frac{9}{14} & \frac{23}{28} & \frac{5}{14}
  \end{pmatrix}, \\
  G'' = 
  \begin{pmatrix}
    0 & 0 & -\frac{17}{28} & -\frac{3}{28} & \frac{1}{14}
  \end{pmatrix}; \;
  B' = 
  \begin{pmatrix}
    \frac{1}{28}
  \end{pmatrix}, \;
  B'' = 
  \begin{pmatrix}
    \frac{27}{28}
  \end{pmatrix}; \;
  Q' =
  \begin{pmatrix}
    \frac{1}{28}
  \end{pmatrix}, \;
  Q'' =
  \begin{pmatrix}
    \frac{55}{28}
  \end{pmatrix}.
\end{gather*}


\subsection{Family \textnumero3.4}\label{subsection:03-04}

The pencil \(\mathcal{S}\) is defined by the equation
\begin{gather*}
  X^{2} Y Z + X Y^{2} Z + X^{2} Z^{2} + 2 X Y Z^{2} + Y^{2} Z^{2} + 2 X^{2} Z T + 2 Y^{2} Z T + \\ Y Z^{2} T + X^{2} T^{2} + 2 X Y T^{2} + Y^{2} T^{2} + 2 Y Z T^{2} + Y T^{3} = \lambda X Y Z T.
\end{gather*}
Members \(\mathcal{S}_{\lambda}\) of the pencil are irreducible for any \(\lambda \in \mathbb{P}^1\) except
\(\mathcal{S}_{\infty} = S_{(X)} + S_{(Y)} + S_{(Z)} + S_{(T)}\).
The base locus of the pencil \(\mathcal{S}\) consists of the following curves:
\begin{gather*}
  C_{1} = C_{(X, Y)}, \;
  C_{2} = C_{(Z, T)}, \;
  C_{3} = C_{(X, Y + T)}, \;
  C_{4} = C_{(X, Z + T)}, \;
  C_{5} = C_{(Y, Z + T)}, \\
  C_{6} = C_{(T, X + Y)}, \;
  C_{7} = C_{(Z, Y T + (X + Y)^2)}, \;
  C_{8} = C_{(T, X Y + Z (X + Y))}.
\end{gather*}
Their linear equivalence classes on the generic member \(\mathcal{S}_{\Bbbk}\) of the pencil satisfy the following relations:
\begin{gather*}
  \begin{pmatrix}
    [C_{3}] \\ [C_{7}] \\ [C_{8}] \\ [H_{\mathcal{S}}]
  \end{pmatrix} = 
  \begin{pmatrix}
    1 & 0 & -2 & 2 & 0 \\
    2 & -2 & 0 & 2 & 0 \\
    2 & -1 & 0 & 2 & -1 \\
    2 & 0 & 0 & 2 & 0
  \end{pmatrix} \cdot
  \begin{pmatrix}
    [C_{1}] & [C_{2}] & [C_{4}] & [C_{5}] & [C_{6}]
  \end{pmatrix}^T.
\end{gather*}

For a general choice of \(\lambda \in \mathbb{C}\) the surface \(\mathcal{S}_{\lambda}\) has the following singularities:
\begin{itemize}\setlength{\itemindent}{2cm}
\item[\(P_{1} = P_{(X, Y, T)}\):] type \(\mathbb{A}_1\) with the quadratic term \(Y T + (X + Y)^2\);
\item[\(P_{2} = P_{(X, Z, T)}\):] type \(\mathbb{A}_1\) with the quadratic term \(X Z + (Z + T)^2\);
\item[\(P_{3} = P_{(Y, Z, T)}\):] type \(\mathbb{A}_1\) with the quadratic term \(Y Z + (Z + T)^2\);
\item[\(P_{4} = P_{(X, Y, Z + T)}\):] type \(\mathbb{A}_5\) with the quadratic term \((\lambda + 4) X \cdot Y\);
\item[\(P_{5} = P_{(Z, T, X + Y)}\):] type \(\mathbb{A}_2\) with the quadratic term \(Z \cdot (X + Y - (\lambda + 4) T)\);
\item[\(P_{6} = P_{(X, Z + T, Y - (\lambda + 4) T)}\):] type \(\mathbb{A}_1\) with the quadratic term
  \[
    (\lambda + 4) ((\lambda + 5) (X T + (Z + T)^2) - X (X + Y + T));
  \]
\item[\(P_{7} = P_{(Y, Z + T, X - (\lambda + 4) T)}\):] type \(\mathbb{A}_1\) with the quadratic term
  \[
    (\lambda + 4) ((\lambda + 4) (Y T + (Z + T)^2) - Y (X + Y)).
  \]
\end{itemize}

Galois action on the lattice \(L_{\lambda}\) is trivial. The intersection matrix on \(L_{\lambda} = L_{\mathcal{S}}\) is represented by
\begin{table}[H]
  \begin{tabular}{|c||c|c|c|ccccc|cc|c|c|ccccc|}
    \hline
    \(\bullet\) & \(E_1^1\) & \(E_2^1\) & \(E_3^1\) & \(E_4^1\) & \(E_4^2\) & \(E_4^3\) & \(E_4^4\) & \(E_4^5\) & \(E_5^1\) & \(E_5^2\) & \(E_6^1\) & \(E_7^1\) & \(\widetilde{C_{1}}\) & \(\widetilde{C_{2}}\) & \(\widetilde{C_{4}}\) & \(\widetilde{C_{5}}\) & \(\widetilde{C_{6}}\) \\
    \hline
    \hline
    \(\widetilde{C_{1}}\) & \(1\) & \(0\) & \(0\) & \(0\) & \(0\) & \(0\) & \(1\) & \(0\) & \(0\) & \(0\) & \(0\) & \(0\) & \(-2\) & \(0\) & \(0\) & \(0\) & \(0\) \\
    \(\widetilde{C_{2}}\) & \(0\) & \(1\) & \(1\) & \(0\) & \(0\) & \(0\) & \(0\) & \(0\) & \(1\) & \(0\) & \(0\) & \(0\) & \(0\) & \(-2\) & \(0\) & \(0\) & \(0\) \\
    \(\widetilde{C_{4}}\) & \(0\) & \(1\) & \(0\) & \(1\) & \(0\) & \(0\) & \(0\) & \(0\) & \(0\) & \(0\) & \(1\) & \(0\) & \(0\) & \(0\) & \(-2\) & \(0\) & \(0\) \\
    \(\widetilde{C_{5}}\) & \(0\) & \(0\) & \(1\) & \(0\) & \(0\) & \(0\) & \(0\) & \(1\) & \(0\) & \(0\) & \(0\) & \(1\) & \(0\) & \(0\) & \(0\) & \(-2\) & \(0\) \\
    \(\widetilde{C_{6}}\) & \(1\) & \(0\) & \(0\) & \(0\) & \(0\) & \(0\) & \(0\) & \(0\) & \(0\) & \(1\) & \(0\) & \(0\) & \(0\) & \(0\) & \(0\) & \(0\) & \(-2\) \\
    \hline
  \end{tabular}.
\end{table}
Note that the intersection matrix is non-degenerate.

Discriminant groups and discriminant forms of the lattices \(L_{\mathcal{S}}\) and \(H \oplus \Pic(X)\) are given by
\begin{gather*}
  G' = 
  \begin{pmatrix}
    0 & \frac{1}{2} & 0 & 0 & 0 & 0 & 0 & \frac{1}{2} & \frac{1}{2} & 0 & \frac{1}{2} & \frac{1}{2} & \frac{1}{2} & 0 & 0 & 0 & \frac{1}{2} \\
    0 & \frac{1}{2} & \frac{1}{2} & 0 & 0 & 0 & 0 & 0 & 0 & 0 & \frac{1}{2} & \frac{1}{2} & 0 & 0 & 0 & 0 & 0 \\
    0 & \frac{5}{6} & \frac{1}{2} & \frac{1}{6} & \frac{1}{3} & \frac{1}{2} & \frac{2}{3} & \frac{1}{2} & 0 & \frac{1}{3} & 0 & \frac{2}{3} & \frac{1}{3} & \frac{2}{3} & 0 & \frac{1}{3} & \frac{2}{3}
  \end{pmatrix}, \\
  G'' = 
  \begin{pmatrix}
    0 & 0 & \frac{1}{2} & 0 & 0 \\
    0 & 0 & 0 & \frac{1}{2} & 0 \\
    0 & 0 & -\frac{1}{3} & -\frac{1}{3} & \frac{1}{6}
  \end{pmatrix}; \;
  B' = 
  \begin{pmatrix}
    \frac{1}{2} & \frac{1}{2} & 0 \\
    \frac{1}{2} & 0 & 0 \\
    0 & 0 & \frac{1}{6}
  \end{pmatrix}, \;
  B'' = 
  \begin{pmatrix}
    \frac{1}{2} & \frac{1}{2} & 0 \\
    \frac{1}{2} & 0 & 0 \\
    0 & 0 & \frac{5}{6}
  \end{pmatrix}; \;
  \begin{pmatrix}
    Q' \\ Q''
  \end{pmatrix}
  =
  \begin{pmatrix}
    \frac{1}{2} & 0 & \frac{1}{6} \\
    \frac{3}{2} & 0 & \frac{11}{6}
  \end{pmatrix}.
\end{gather*}


\subsection{Family \textnumero3.5}\label{subsection:03-05}

The pencil \(\mathcal{S}\) is defined by the equation
\begin{gather*}
  X^{3} Z + 3 X^{2} Y Z + 3 X Y^{2} Z + Y^{3} Z + X Y Z^{2} + Y^{2} Z^{2} + X^{2} Y T + \\ X Y^{2} T + 2 X^{2} Z T + 2 Y^{2} Z T + X Y T^{2} + X Z T^{2} + Y Z T^{2} = \lambda X Y Z T.
\end{gather*}
Members \(\mathcal{S}_{\lambda}\) of the pencil are irreducible for any \(\lambda \in \mathbb{P}^1\) except
\(\mathcal{S}_{\infty} = S_{(X)} + S_{(Y)} + S_{(Z)} + S_{(T)}\).
The base locus of the pencil \(\mathcal{S}\) consists of the following curves:
\begin{gather*}
  C_{1} = C_{(X, Y)}, \;
  C_{2} = C_{(X, Z)}, \;
  C_{3} = C_{(Y, Z)}, \;
  C_{4} = C_{(Z, T)}, \;
  C_{5} = C_{(Y, X + T)}, \\
  C_{6} = C_{(T, X + Y)}, \;
  C_{7} = C_{(Z, X + Y + T)}, \;
  C_{8} = C_{(X, Y Z + (Y + T)^2)}, \;
  C_{9} = C_{(T, Y Z + (X + Y)^2)}.
\end{gather*}
Their linear equivalence classes on the generic member \(\mathcal{S}_{\Bbbk}\) of the pencil satisfy the following relations:
\begin{gather*}
  \begin{pmatrix}
    [C_{3}] \\ [C_{7}] \\ [C_{8}] \\ [C_{9}]
  \end{pmatrix} = 
  \begin{pmatrix}
    -1 & 0 & 0 & -2 & 0 & 1 \\
    1 & -1 & -1 & 2 & 0 & 0 \\
    -1 & -1 & 0 & 0 & 0 & 1 \\
    0 & 0 & -1 & 0 & -1 & 1
  \end{pmatrix} \cdot
  \begin{pmatrix}
    [C_{1}] & [C_{2}] & [C_{4}] & [C_{5}] & [C_{6}] & [H_{\mathcal{S}}]
  \end{pmatrix}^T.
\end{gather*}

For a general choice of \(\lambda \in \mathbb{C}\) the surface \(\mathcal{S}_{\lambda}\) has the following singularities:
\begin{itemize}\setlength{\itemindent}{2cm}
\item[\(P_{1} = P_{(X, Y, Z)}\):] type \(\mathbb{A}_1\) with the quadratic term \(X Y + Z (X + Y)\);
\item[\(P_{2} = P_{(X, Y, T)}\):] type \(\mathbb{A}_4\) with the quadratic term \(Y \cdot (X + Y)\);
\item[\(P_{3} = P_{(X, Z, Y + T)}\):] type \(\mathbb{A}_1\) with the quadratic term \(X (X + Y - (\lambda + 4) Z + T) - Z^2\);
\item[\(P_{4} = P_{(Y, Z, X + T)}\):] type \(\mathbb{A}_2\) with the quadratic term \(Y \cdot (X + Y - (\lambda + 4) Z + T)\);
\item[\(P_{5} = P_{(Z, T, X + Y)}\):] type \(\mathbb{A}_2\) with the quadratic term \(T \cdot (X + Y - (\lambda + 4) Z + T)\);
\item[\(P_{6} = P_{(Y, X + T, Z - (\lambda + 4) T)}\):] type \(\mathbb{A}_1\).
\end{itemize}

Galois action on the lattice \(L_{\lambda}\) is trivial. The intersection matrix on \(L_{\lambda} = L_{\mathcal{S}}\) is represented by
\begin{table}[H]
  \begin{tabular}{|c||c|cccc|c|cc|cc|c|cccccc|}
    \hline
    \(\bullet\) & \(E_1^1\) & \(E_2^1\) & \(E_2^2\) & \(E_2^3\) & \(E_2^4\) & \(E_3^1\) & \(E_4^1\) & \(E_4^2\) & \(E_5^1\) & \(E_5^2\) & \(E_6^1\) & \(\widetilde{C_{1}}\) & \(\widetilde{C_{2}}\) & \(\widetilde{C_{4}}\) & \(\widetilde{C_{5}}\) & \(\widetilde{C_{6}}\) & \(\widetilde{H_{\mathcal{S}}}\) \\
    \hline
    \hline
    \(\widetilde{C_{1}}\) & \(1\) & \(0\) & \(0\) & \(1\) & \(0\) & \(0\) & \(0\) & \(0\) & \(0\) & \(0\) & \(0\) & \(-2\) & \(0\) & \(0\) & \(0\) & \(0\) & \(1\) \\
    \(\widetilde{C_{2}}\) & \(1\) & \(0\) & \(0\) & \(0\) & \(0\) & \(1\) & \(0\) & \(0\) & \(0\) & \(0\) & \(0\) & \(0\) & \(-2\) & \(1\) & \(0\) & \(0\) & \(1\) \\
    \(\widetilde{C_{4}}\) & \(0\) & \(0\) & \(0\) & \(0\) & \(0\) & \(0\) & \(0\) & \(0\) & \(1\) & \(0\) & \(0\) & \(0\) & \(1\) & \(-2\) & \(0\) & \(0\) & \(1\) \\
    \(\widetilde{C_{5}}\) & \(0\) & \(1\) & \(0\) & \(0\) & \(0\) & \(0\) & \(1\) & \(0\) & \(0\) & \(0\) & \(1\) & \(0\) & \(0\) & \(0\) & \(-2\) & \(0\) & \(1\) \\
    \(\widetilde{C_{6}}\) & \(0\) & \(0\) & \(0\) & \(0\) & \(1\) & \(0\) & \(0\) & \(0\) & \(1\) & \(0\) & \(0\) & \(0\) & \(0\) & \(0\) & \(0\) & \(-2\) & \(1\) \\
    \(\widetilde{H_{\mathcal{S}}}\) & \(0\) & \(0\) & \(0\) & \(0\) & \(0\) & \(0\) & \(0\) & \(0\) & \(0\) & \(0\) & \(0\) & \(1\) & \(1\) & \(1\) & \(1\) & \(1\) & \(4\) \\
    \hline
  \end{tabular}.
\end{table}
Note that the intersection matrix is non-degenerate.

Discriminant groups and discriminant forms of the lattices \(L_{\mathcal{S}}\) and \(H \oplus \Pic(X)\) are given by
\begin{gather*}
  G' = 
  \begin{pmatrix}
    \frac{5}{28} & \frac{1}{4} & \frac{4}{7} & \frac{25}{28} & 0 & \frac{4}{7} & \frac{2}{7} & \frac{9}{14} & \frac{6}{7} & \frac{3}{7} & \frac{27}{28} & \frac{3}{14} & \frac{1}{7} & \frac{5}{28} & \frac{13}{14} & \frac{3}{28} & \frac{5}{14}
  \end{pmatrix}, \\
  G'' = 
  \begin{pmatrix}
    0 & 0 & -\frac{5}{28} & -\frac{1}{28} & \frac{1}{7}
  \end{pmatrix}; \;
  B' = 
  \begin{pmatrix}
    \frac{9}{28}
  \end{pmatrix}, \;
  B'' = 
  \begin{pmatrix}
    \frac{19}{28}
  \end{pmatrix}; \;
  Q' =
  \begin{pmatrix}
    \frac{9}{28}
  \end{pmatrix}, \;
  Q'' =
  \begin{pmatrix}
    \frac{47}{28}
  \end{pmatrix}.
\end{gather*}


\subsection{Family \textnumero3.6}\label{subsection:03-06}

The pencil \(\mathcal{S}\) is defined by the equation
\begin{gather*}
  X^{2} Y Z + 2 X Y^{2} Z + Y^{3} Z + X Y Z^{2} + Y^{2} Z^{2} + X^{2} Y T + X Y^{2} T + \\ 3 Y^{2} Z T + Y Z^{2} T + X Y T^{2} + X Z T^{2} + 3 Y Z T^{2} + Z T^{3} = \lambda X Y Z T.
\end{gather*}
Members \(\mathcal{S}_{\lambda}\) of the pencil are irreducible for any \(\lambda \in \mathbb{P}^1\) except
\begin{gather*}
  \mathcal{S}_{\infty} = S_{(X)} + S_{(Y)} + S_{(Z)} + S_{(T)}, \;
  \mathcal{S}_{- 3} = S_{(X + Y + T)} + S_{(X Y (Z + T) + Z (Y + T)^2 + Y Z^2)}.
\end{gather*}
The base locus of the pencil \(\mathcal{S}\) consists of the following curves:
\begin{gather*}
  C_1 = C_{(X, Z)}, \;
  C_2 = C_{(Y, Z)}, \;
  C_3 = C_{(Y, T)}, \;
  C_4 = C_{(Z, T)}, \;
  C_5 = C_{(X, Y + T)}, \\
  C_6 = C_{(Y, X + T)}, \;
  C_7 = C_{(T, X + Y)}, \;
  C_8 = C_{(Z, X + Y + T)}, \;
  C_9 = C_{(T, X + Y + Z)}, \;
  C_{10} = C_{(X, Y Z + (Y + T)^2)}.
\end{gather*}
Their linear equivalence classes on the generic member \(\mathcal{S}_{\Bbbk}\) of the pencil satisfy the following relations:
\begin{gather*}
  \begin{pmatrix}
    [C_{6}] \\ [C_{7}] \\ [C_{8}] \\ [C_{9}] \\ [C_{10}]
  \end{pmatrix} = 
  \begin{pmatrix}
    0 & -1 & -2 & 0 & 0 & 1 \\
    1 & 2 & 2 & 1 & -1 & -1 \\
    -1 & -1 & 0 & -1 & 0 & 1 \\
    -1 & -2 & -3 & -2 & 1 & 2 \\
    -1 & 0 & 0 & 0 & -1 & 1
  \end{pmatrix} \cdot
  \begin{pmatrix}
    [C_{1}] & [C_{2}] & [C_{3}] & [C_{4}] & [C_{5}] & [H_{\mathcal{S}}]
  \end{pmatrix}^T.
\end{gather*}

For a general choice of \(\lambda \in \mathbb{C}\) the surface \(\mathcal{S}_{\lambda}\) has the following singularities:
\begin{itemize}\setlength{\itemindent}{2cm}
\item[\(P_{1} = P_{(X, Y, T)}\):] type \(\mathbb{A}_3\) with the quadratic term \(Y \cdot (X + Y + T)\);
\item[\(P_{2} = P_{(Y, Z, T)}\):] type \(\mathbb{A}_2\) with the quadratic term \(Y \cdot (Z + T)\);
\item[\(P_{3} = P_{(X, Z, Y + T)}\):] type \(\mathbb{A}_2\) with the quadratic term \(X \cdot (X + Y - (\lambda + 3) Z + T)\);
\item[\(P_{4} = P_{(Y, Z, X + T)}\):] type \(\mathbb{A}_1\) with the quadratic term \((Y - Z) (X + Y + T) - (\lambda + 3) Y Z\);
\item[\(P_{5} = P_{(Y, T, X + Z)}\):] type \(\mathbb{A}_1\) with the quadratic term \(Y (X + Y + Z - (\lambda + 2) T) + T^2\);
\item[\(P_{6} = P_{(Z, T, X + Y)}\):] type \(\mathbb{A}_2\) with the quadratic term \(T \cdot (X + Y - (\lambda + 3) Z + T)\).
\end{itemize}

Galois action on the lattice \(L_{\lambda}\) is trivial. The intersection matrix on \(L_{\lambda} = L_{\mathcal{S}}\) is represented by
\begin{table}[H]
  \begin{tabular}{|c||ccc|cc|cc|cc|cc|cccccc|}
    \hline
    \(\bullet\) & \(E_1^1\) & \(E_1^2\) & \(E_1^3\) & \(E_2^1\) & \(E_2^2\) & \(E_3^1\) & \(E_3^2\) & \(E_4^1\) & \(E_5^1\) & \(E_6^1\) & \(E_6^2\) & \(\widetilde{C_{1}}\) & \(\widetilde{C_{2}}\) & \(\widetilde{C_{3}}\) & \(\widetilde{C_{4}}\) & \(\widetilde{C_{5}}\) & \(\widetilde{H_{\mathcal{S}}}\) \\
    \hline
    \hline
    \(\widetilde{C_{1}}\) & \(0\) & \(0\) & \(0\) & \(0\) & \(0\) & \(1\) & \(0\) & \(0\) & \(0\) & \(0\) & \(0\) & \(-2\) & \(1\) & \(0\) & \(1\) & \(0\) & \(1\) \\
    \(\widetilde{C_{2}}\) & \(0\) & \(0\) & \(0\) & \(1\) & \(0\) & \(0\) & \(0\) & \(1\) & \(0\) & \(0\) & \(0\) & \(1\) & \(-2\) & \(0\) & \(0\) & \(0\) & \(1\) \\
    \(\widetilde{C_{3}}\) & \(1\) & \(0\) & \(0\) & \(1\) & \(0\) & \(0\) & \(0\) & \(0\) & \(1\) & \(0\) & \(0\) & \(0\) & \(0\) & \(-2\) & \(0\) & \(0\) & \(1\) \\
    \(\widetilde{C_{4}}\) & \(0\) & \(0\) & \(0\) & \(0\) & \(1\) & \(0\) & \(0\) & \(0\) & \(0\) & \(1\) & \(0\) & \(1\) & \(0\) & \(0\) & \(-2\) & \(0\) & \(1\) \\
    \(\widetilde{C_{5}}\) & \(0\) & \(0\) & \(1\) & \(0\) & \(0\) & \(1\) & \(0\) & \(0\) & \(0\) & \(0\) & \(0\) & \(0\) & \(0\) & \(0\) & \(0\) & \(-2\) & \(1\) \\
    \(\widetilde{H_{\mathcal{S}}}\) & \(0\) & \(0\) & \(0\) & \(0\) & \(0\) & \(0\) & \(0\) & \(0\) & \(0\) & \(0\) & \(0\) & \(1\) & \(1\) & \(1\) & \(1\) & \(1\) & \(4\) \\
    \hline
  \end{tabular}.
\end{table}
Note that the intersection matrix is non-degenerate.

Discriminant groups and discriminant forms of the lattices \(L_{\mathcal{S}}\) and \(H \oplus \Pic(X)\) are given by
\begin{gather*}
  G' = 
  \begin{pmatrix}
    \frac{9}{32} & \frac{5}{8} & \frac{31}{32} & \frac{13}{32} & \frac{9}{16} & \frac{7}{16} & \frac{23}{32} & \frac{5}{32} & \frac{31}{32} & \frac{13}{16} & \frac{29}{32} & \frac{27}{32} & \frac{5}{16} & \frac{15}{16} & \frac{23}{32} & \frac{5}{16} & \frac{7}{32}
  \end{pmatrix}, \\
  G'' = 
  \begin{pmatrix}
    0 & 0 & \frac{1}{8} & -\frac{1}{32} & \frac{1}{4}
  \end{pmatrix}; \;
  B' = 
  \begin{pmatrix}
    \frac{25}{32}
  \end{pmatrix}, \;
  B'' = 
  \begin{pmatrix}
    \frac{7}{32}
  \end{pmatrix}; \;
  Q' =
  \begin{pmatrix}
    \frac{57}{32}
  \end{pmatrix}, \;
  Q'' =
  \begin{pmatrix}
    \frac{7}{32}
  \end{pmatrix}.
\end{gather*}


\subsection{Family \textnumero3.7}\label{subsection:03-07}

The pencil \(\mathcal{S}\) is defined by the equation
\begin{gather*}
  X^{2} Y Z + X Y^{2} Z + X Y Z^{2} + X Y^{2} T + Y^{2} Z T + X Z^{2} T + Y Z^{2} T + \\ X Y T^{2} + Y^{2} T^{2} + X Z T^{2} + 2 Y Z T^{2} + Z^{2} T^{2} + Y T^{3} + Z T^{3} = \lambda X Y Z T.
\end{gather*}
Members \(\mathcal{S}_{\lambda}\) of the pencil are irreducible for any \(\lambda \in \mathbb{P}^1\) except
\begin{gather*}
  \mathcal{S}_{\infty} = S_{(X)} + S_{(Y)} + S_{(Z)} + S_{(T)}, \;
  \mathcal{S}_{- 3} = S_{(X + T)} + S_{(X Y Z + (Y + Z) (Y + T) (Z + T))}.
\end{gather*}
The base locus of the pencil \(\mathcal{S}\) consists of the following curves:
\begin{gather*}
  C_{1} = C_{(X, T)}, \;
  C_{2} = C_{(Y, Z)}, \;
  C_{3} = C_{(Y, T)}, \;
  C_{4} = C_{(Z, T)}, \;
  C_{5} = C_{(X, Y + Z)}, \;
  C_{6} = C_{(X, Y + T)}, \\
  C_{7} = C_{(X, Z + T)}, \;
  C_{8} = C_{(Y, X + T)}, \;
  C_{9} = C_{(Y, Z + T)}, \;
  C_{10} = C_{(Z, X + T)}, \;
  C_{11} = C_{(Z, Y + T)}, \;
  C_{12} = C_{(T, X + Y + Z)}.
\end{gather*}
Their linear equivalence classes on the generic member \(\mathcal{S}_{\Bbbk}\) of the pencil satisfy the following relations:
\begin{gather*}
  \begin{pmatrix}
    [C_{7}] \\ [C_{9}] \\ [C_{10}] \\ [C_{11}] \\ [C_{12}]
  \end{pmatrix} = 
  \begin{pmatrix}
    -1 & 0 & 0 & 0 & -1 & -1 & 0 & 1 \\
    0 & -1 & -1 & 0 & 0 & 0 & -1 & 1 \\
    -2 & 0 & 0 & 0 & 0 & 0 & -1 & 1 \\
    2 & -1 & 0 & -1 & 0 & 0 & 1 & 0 \\
    -1 & 0 & -1 & -1 & 0 & 0 & 0 & 1
  \end{pmatrix} \cdot \\
  \begin{pmatrix}
    [C_{1}] & [C_{2}] & [C_{3}] & [C_{4}] & [C_{5}] & [C_{6}] & [C_{8}] & [H_{\mathcal{S}}]
  \end{pmatrix}^T.
\end{gather*}

For a general choice of \(\lambda \in \mathbb{C}\) the surface \(\mathcal{S}_{\lambda}\) has the following singularities:
\begin{itemize}\setlength{\itemindent}{2cm}
\item[\(P_{1} = P_{(X, Y, T)}\):] type \(\mathbb{A}_2\) with the quadratic term \((X + T) \cdot (Y + T)\);
\item[\(P_{2} = P_{(X, Z, T)}\):] type \(\mathbb{A}_2\) with the quadratic term \((X + T) \cdot (Z + T)\);
\item[\(P_{3} = P_{(Y, Z, T)}\):] type \(\mathbb{A}_3\) with the quadratic term \(Y \cdot Z\);
\item[\(P_{4} = P_{(X, T, Y + Z)}\):] type \(\mathbb{A}_1\) with the quadratic term \((X + T) (X + Y + Z) - (\lambda + 3) X T\);
\item[\(P_{5} = P_{(Y, Z, X + T)}\):] type \(\mathbb{A}_1\) with the quadratic term \((X + T) (Y + Z) + (\lambda + 3) Y Z\).
\end{itemize}

Galois action on the lattice \(L_{\lambda}\) is trivial. The intersection matrix on \(L_{\lambda} = L_{\mathcal{S}}\) is represented by
\begin{table}[H]
  \begin{tabular}{|c||cc|cc|ccc|c|c|cccccccc|}
    \hline
    \(\bullet\) & \(E_1^1\) & \(E_1^2\) & \(E_2^1\) & \(E_2^2\) & \(E_3^1\) & \(E_3^2\) & \(E_3^3\) & \(E_4^1\) & \(E_5^1\) & \(\widetilde{C_{1}}\) & \(\widetilde{C_{2}}\) & \(\widetilde{C_{3}}\) & \(\widetilde{C_{4}}\) & \(\widetilde{C_{5}}\) & \(\widetilde{C_{6}}\) & \(\widetilde{C_{8}}\) & \(\widetilde{H_{\mathcal{S}}}\) \\
    \hline
    \hline
    \(\widetilde{C_{1}}\) & \(1\) & \(0\) & \(1\) & \(0\) & \(0\) & \(0\) & \(0\) & \(1\) & \(0\) & \(-2\) & \(0\) & \(0\) & \(0\) & \(0\) & \(0\) & \(0\) & \(1\) \\
    \(\widetilde{C_{2}}\) & \(0\) & \(0\) & \(0\) & \(0\) & \(0\) & \(1\) & \(0\) & \(0\) & \(1\) & \(0\) & \(-2\) & \(0\) & \(0\) & \(1\) & \(0\) & \(0\) & \(1\) \\
    \(\widetilde{C_{3}}\) & \(0\) & \(1\) & \(0\) & \(0\) & \(1\) & \(0\) & \(0\) & \(0\) & \(0\) & \(0\) & \(0\) & \(-2\) & \(0\) & \(0\) & \(0\) & \(0\) & \(1\) \\
    \(\widetilde{C_{4}}\) & \(0\) & \(0\) & \(0\) & \(1\) & \(0\) & \(0\) & \(1\) & \(0\) & \(0\) & \(0\) & \(0\) & \(0\) & \(-2\) & \(0\) & \(0\) & \(0\) & \(1\) \\
    \(\widetilde{C_{5}}\) & \(0\) & \(0\) & \(0\) & \(0\) & \(0\) & \(0\) & \(0\) & \(1\) & \(0\) & \(0\) & \(1\) & \(0\) & \(0\) & \(-2\) & \(1\) & \(0\) & \(1\) \\
    \(\widetilde{C_{6}}\) & \(0\) & \(1\) & \(0\) & \(0\) & \(0\) & \(0\) & \(0\) & \(0\) & \(0\) & \(0\) & \(0\) & \(0\) & \(0\) & \(1\) & \(-2\) & \(0\) & \(1\) \\
    \(\widetilde{C_{8}}\) & \(1\) & \(0\) & \(0\) & \(0\) & \(0\) & \(0\) & \(0\) & \(0\) & \(1\) & \(0\) & \(0\) & \(0\) & \(0\) & \(0\) & \(0\) & \(-2\) & \(1\) \\
    \(\widetilde{H_{\mathcal{S}}}\) & \(0\) & \(0\) & \(0\) & \(0\) & \(0\) & \(0\) & \(0\) & \(0\) & \(0\) & \(1\) & \(1\) & \(1\) & \(1\) & \(1\) & \(1\) & \(1\) & \(4\) \\
    \hline
  \end{tabular}.
\end{table}
Note that the intersection matrix is non-degenerate.

Discriminant groups and discriminant forms of the lattices \(L_{\mathcal{S}}\) and \(H \oplus \Pic(X)\) are given by
\begin{gather*}
  G' = 
  \begin{pmatrix}
    \frac{1}{3} & \frac{1}{3} & \frac{2}{3} & \frac{1}{3} & 0 & 0 & 0 & \frac{1}{3} & \frac{2}{3} & 0 & 0 & 0 & 0 & \frac{2}{3} & \frac{1}{3} & \frac{1}{3} & \frac{2}{3} \\
    \frac{5}{12} & \frac{1}{6} & \frac{3}{4} & \frac{1}{6} & \frac{5}{6} & \frac{3}{4} & \frac{1}{6} & \frac{2}{3} & \frac{5}{12} & \frac{1}{3} & \frac{1}{2} & \frac{11}{12} & \frac{7}{12} & 0 & 0 & \frac{1}{3} & \frac{5}{6}
  \end{pmatrix}, \\
  G'' = 
  \begin{pmatrix}
    0 & 0 & -\frac{1}{3} & \frac{1}{3} & \frac{1}{3} \\
    0 & 0 & -\frac{1}{6} & -\frac{5}{12} & \frac{1}{12}
  \end{pmatrix}; \;
  B' = 
  \begin{pmatrix}
    0 & \frac{2}{3} \\
    \frac{2}{3} & \frac{7}{12}
  \end{pmatrix}, \;
  B'' = 
  \begin{pmatrix}
    0 & \frac{1}{3} \\
    \frac{1}{3} & \frac{5}{12}
  \end{pmatrix}; \;
  \begin{pmatrix}
    Q' \\ Q''
  \end{pmatrix}
  =
  \begin{pmatrix}
    0 & \frac{19}{12} \\
    0 & \frac{5}{12}    
  \end{pmatrix}.
\end{gather*}


\subsection{Family \textnumero3.8}\label{subsection:03-08}

The pencil \(\mathcal{S}\) is defined by the equation
\begin{gather*}
  X^{2} Y Z + X Y^{2} Z + X^{2} Z^{2} + X Y Z^{2} + X^{2} Z T + Y^{2} Z T + \\ X Z^{2} T + X Y T^{2} + 2 X Z T^{2} + 2 Y Z T^{2} + Y T^{3} + Z T^{3} = \lambda X Y Z T.
\end{gather*}
Members \(\mathcal{S}_{\lambda}\) of the pencil are irreducible for any \(\lambda \in \mathbb{P}^1\) except
\(\mathcal{S}_{\infty} = S_{(X)} + S_{(Y)} + S_{(Z)} + S_{(T)}\).
The base locus of the pencil \(\mathcal{S}\) consists of the following curves:
\begin{gather*}
  C_{1} = C_{(X, T)}, \;
  C_{2} = C_{(Y, Z)}, \;
  C_{3} = C_{(Z, T)}, \;
  C_{4} = C_{(Y, X + T)}, \;
  C_{5} = C_{(Z, X + T)}, \\
  C_{6} = C_{(T, X + Y)}, \;
  C_{7} = C_{(T, Y + Z)}, \;
  C_{8} = C_{(Y, T (X + T) + X Z)}, \;
  C_{9} = C_{(X, Z (Y + T)^2 + Y T^2)}.
\end{gather*}
Their linear equivalence classes on the generic member \(\mathcal{S}_{\Bbbk}\) of the pencil satisfy the following relations:
\begin{gather*}
  \begin{pmatrix}
    [C_{5}] \\ [C_{6}] \\ [C_{8}] \\ [C_{9}]
  \end{pmatrix} = 
  \begin{pmatrix}
    0 & -1 & -2 & 0 & 0 & 1 \\
    -1 & 0 & -1 & 0 & -1 & 1 \\
    0 & -1 & 0 & -1 & 0 & 1 \\
    -1 & 0 & 0 & 0 & 0 & 1
  \end{pmatrix} \cdot
  \begin{pmatrix}
    [C_{1}] & [C_{2}] & [C_{3}] & [C_{4}] & [C_{7}] & [H_{\mathcal{S}}]
  \end{pmatrix}^T.
\end{gather*}

For a general choice of \(\lambda \in \mathbb{C}\) the surface \(\mathcal{S}_{\lambda}\) has the following singularities:
\begin{itemize}\setlength{\itemindent}{2cm}
\item[\(P_{1} = P_{(X, Y, T)}\):] type \(\mathbb{A}_3\) with the quadratic term \(X \cdot (X + Y + T)\);
\item[\(P_{2} = P_{(X, Z, T)}\):] type \(\mathbb{A}_3\) with the quadratic term \(Z \cdot (X + T)\);
\item[\(P_{3} = P_{(Y, Z, T)}\):] type \(\mathbb{A}_2\) with the quadratic term \(Z \cdot (Y + Z + T)\);
\item[\(P_{4} = P_{(Y, Z, X + T)}\):] type \(\mathbb{A}_2\) with the quadratic term \(Y \cdot (X + (\lambda + 3) Z + T)\);
\item[\(P_{5} = P_{(Z, T, X + Y)}\):] type \(\mathbb{A}_1\) with the quadratic term \(Z (X + Y - (\lambda + 2) T) + T^2\).
\end{itemize}

Galois action on the lattice \(L_{\lambda}\) is trivial. The intersection matrix on \(L_{\lambda} = L_{\mathcal{S}}\) is represented by
\begin{table}[H]
  \begin{tabular}{|c||ccc|ccc|cc|cc|c|cccccc|}
    \hline
    \(\bullet\) & \(E_1^1\) & \(E_1^2\) & \(E_1^3\) & \(E_2^1\) & \(E_2^2\) & \(E_2^3\) & \(E_3^1\) & \(E_3^2\) & \(E_4^1\) & \(E_4^2\) & \(E_5^1\) & \(\widetilde{C_{1}}\) & \(\widetilde{C_{2}}\) & \(\widetilde{C_{3}}\) & \(\widetilde{C_{4}}\) & \(\widetilde{C_{7}}\) & \(\widetilde{H_{\mathcal{S}}}\) \\
    \hline
    \hline
    \(\widetilde{C_{1}}\) & \(1\) & \(0\) & \(0\) & \(1\) & \(0\) & \(0\) & \(0\) & \(0\) & \(0\) & \(0\) & \(0\) & \(-2\) & \(0\) & \(0\) & \(0\) & \(1\) & \(1\) \\
    \(\widetilde{C_{2}}\) & \(0\) & \(0\) & \(0\) & \(0\) & \(0\) & \(0\) & \(1\) & \(0\) & \(1\) & \(0\) & \(0\) & \(0\) & \(-2\) & \(0\) & \(0\) & \(0\) & \(1\) \\
    \(\widetilde{C_{3}}\) & \(0\) & \(0\) & \(0\) & \(0\) & \(0\) & \(1\) & \(1\) & \(0\) & \(0\) & \(0\) & \(1\) & \(0\) & \(0\) & \(-2\) & \(0\) & \(0\) & \(1\) \\
    \(\widetilde{C_{4}}\) & \(0\) & \(0\) & \(1\) & \(0\) & \(0\) & \(0\) & \(0\) & \(0\) & \(1\) & \(0\) & \(0\) & \(0\) & \(0\) & \(0\) & \(-2\) & \(0\) & \(1\) \\
    \(\widetilde{C_{7}}\) & \(0\) & \(0\) & \(0\) & \(0\) & \(0\) & \(0\) & \(0\) & \(1\) & \(0\) & \(0\) & \(0\) & \(1\) & \(0\) & \(0\) & \(0\) & \(-2\) & \(1\) \\
    \(\widetilde{H_{\mathcal{S}}}\) & \(0\) & \(0\) & \(0\) & \(0\) & \(0\) & \(0\) & \(0\) & \(0\) & \(0\) & \(0\) & \(0\) & \(1\) & \(1\) & \(1\) & \(1\) & \(1\) & \(4\) \\
    \hline
  \end{tabular}.
\end{table}
Note that the intersection matrix is non-degenerate.

Discriminant groups and discriminant forms of the lattices \(L_{\mathcal{S}}\) and \(H \oplus \Pic(X)\) are given by
\begin{gather*}
  G' = 
  \begin{pmatrix}
    \frac{27}{34} & \frac{16}{17} & \frac{3}{34} & \frac{13}{34} & \frac{2}{17} & \frac{29}{34} & \frac{5}{34} & \frac{15}{34} & 0 & \frac{1}{2} & \frac{27}{34} & \frac{11}{17} & \frac{9}{34} & \frac{10}{17} & \frac{4}{17} & \frac{25}{34} & \frac{13}{34}
  \end{pmatrix}, \\
  G'' = 
  \begin{pmatrix}
    0 & 0 & -\frac{11}{17} & -\frac{15}{34} & \frac{1}{17}
  \end{pmatrix}; \;
  B' = 
  \begin{pmatrix}
    \frac{1}{34}
  \end{pmatrix}, \;
  B'' = 
  \begin{pmatrix}
    \frac{33}{34}
  \end{pmatrix}; \;
  Q' =
  \begin{pmatrix}
    \frac{1}{34}
  \end{pmatrix}, \;
  Q'' =
  \begin{pmatrix}
    \frac{67}{34}
  \end{pmatrix}.
\end{gather*}


\subsection{Family \textnumero3.9}\label{subsection:03-09}

The pencil \(\mathcal{S}\) is defined by the equation
\begin{gather*}
  X^{2} Y Z + X Y^{2} Z + X Y Z^{2} + X^{3} T + Y^{2} Z T + Y Z^{2} T + 2 X^{2} T^{2} + Y Z T^{2} + X T^{3} = \lambda X Y Z T.
\end{gather*}
Members \(\mathcal{S}_{\lambda}\) of the pencil are irreducible for any \(\lambda \in \mathbb{P}^1\) except
\begin{gather*}
  \mathcal{S}_{\infty} = S_{(X)} + S_{(Y)} + S_{(Z)} + S_{(T)}, \;
  \mathcal{S}_{- 2} = S_{(X + T)} + S_{(X T (X + T) + Y Z (X + Y + Z + T))}.
\end{gather*}
The base locus of the pencil \(\mathcal{S}\) consists of the following curves:
\begin{gather*}
  C_{1} = C_{(X, Y)}, \;
  C_{2} = C_{(X, Z)}, \;
  C_{3} = C_{(X, T)}, \;
  C_{4} = C_{(Y, T)}, \;
  C_{5} = C_{(Z, T)}, \\
  C_{6} = C_{(Y, X + T)}, \;
  C_{7} = C_{(Z, X + T)}, \;
  C_{8} = C_{(X, Y + Z + T)}, \;
  C_{9} = C_{(T, X + Y + Z)}.
\end{gather*}
Their linear equivalence classes on the generic member \(\mathcal{S}_{\Bbbk}\) of the pencil satisfy the following relations:
\begin{gather*}
  \begin{pmatrix}
    [C_{4}] \\ [C_{5}] \\ [C_{7}] \\ [C_{8}] \\ [C_{9}]
  \end{pmatrix} = 
  \begin{pmatrix}
    -1 & 0 & 0 & -2 & 1 \\
    0 & -1 & 4 & 2 & -1 \\
    0 & 0 & -2 & -1 & 1 \\
    -1 & -1 & -1 & 0 & 1 \\
    1 & 1 & -5 & 0 & 1
  \end{pmatrix} \cdot
  \begin{pmatrix}
    [C_{1}] \\ [C_{2}] \\ [C_{3}] \\ [C_{6}] \\ [H_{\mathcal{S}}]
  \end{pmatrix}.
\end{gather*}

For a general choice of \(\lambda \in \mathbb{C}\) the surface \(\mathcal{S}_{\lambda}\) has the following singularities:
\begin{itemize}\setlength{\itemindent}{2cm}
\item[\(P_{1} = P_{(X, Y, T)}\):] type \(\mathbb{A}_5\) with the quadratic term \(Y \cdot (X + T)\);
\item[\(P_{2} = P_{(X, Z, T)}\):] type \(\mathbb{A}_5\) with the quadratic term \(Z \cdot (X + T)\);
\item[\(P_{3} = P_{(X, T, Y + Z)}\):] type \(\mathbb{A}_1\) with the quadratic term \((X + T) (X + Y + Z + T) - (\lambda + 2) X T\);
\item[\(P_{4} = P_{(Y, Z, X + T)}\):] type \(\mathbb{A}_1\) with the quadratic term \((X + T)^2 - (\lambda + 2) Y Z\).
\end{itemize}

Galois action on the lattice \(L_{\lambda}\) is trivial. The intersection matrix on \(L_{\lambda} = L_{\mathcal{S}}\) is represented by
\begin{table}[H]
  \begin{tabular}{|c||ccccc|ccccc|c|c|ccccc|}
    \hline
    \(\bullet\) & \(E_1^1\) & \(E_1^2\) & \(E_1^3\) & \(E_1^4\) & \(E_1^5\) & \(E_2^1\) & \(E_2^2\) & \(E_2^3\) & \(E_2^4\) & \(E_2^5\) & \(E_3^1\) & \(E_4^1\) & \(\widetilde{C_{1}}\) & \(\widetilde{C_{2}}\) & \(\widetilde{C_{3}}\) & \(\widetilde{C_{6}}\) & \(\widetilde{H_{\mathcal{S}}}\) \\
    \hline
    \hline
    \(\widetilde{C_{1}}\) & \(1\) & \(0\) & \(0\) & \(0\) & \(0\) & \(0\) & \(0\) & \(0\) & \(0\) & \(0\) & \(0\) & \(0\) & \(-2\) & \(1\) & \(0\) & \(0\) & \(1\) \\
    \(\widetilde{C_{2}}\) & \(0\) & \(0\) & \(0\) & \(0\) & \(0\) & \(1\) & \(0\) & \(0\) & \(0\) & \(0\) & \(0\) & \(0\) & \(1\) & \(-2\) & \(0\) & \(0\) & \(1\) \\
    \(\widetilde{C_{3}}\) & \(0\) & \(0\) & \(0\) & \(0\) & \(1\) & \(0\) & \(0\) & \(0\) & \(0\) & \(1\) & \(1\) & \(0\) & \(0\) & \(0\) & \(-2\) & \(0\) & \(1\) \\
    \(\widetilde{C_{6}}\) & \(0\) & \(1\) & \(0\) & \(0\) & \(0\) & \(0\) & \(0\) & \(0\) & \(0\) & \(0\) & \(0\) & \(1\) & \(0\) & \(0\) & \(0\) & \(-2\) & \(1\) \\
    \(\widetilde{H_{\mathcal{S}}}\) & \(0\) & \(0\) & \(0\) & \(0\) & \(0\) & \(0\) & \(0\) & \(0\) & \(0\) & \(0\) & \(0\) & \(0\) & \(1\) & \(1\) & \(1\) & \(1\) & \(4\) \\
    \hline
  \end{tabular}.
\end{table}
Note that the intersection matrix is non-degenerate.

Discriminant groups and discriminant forms of the lattices \(L_{\mathcal{S}}\) and \(H \oplus \Pic(X)\) are given by
\begin{gather*}
  G' = 
  \begin{pmatrix}
    \frac{11}{12} & \frac{1}{6} & \frac{7}{12} & 0 & \frac{5}{12} & \frac{11}{12} & \frac{1}{2} & \frac{1}{12} & \frac{2}{3} & \frac{1}{4} & \frac{11}{12} & \frac{5}{12} & \frac{2}{3} & \frac{1}{3} & \frac{5}{6} & \frac{5}{6} & \frac{1}{12}
  \end{pmatrix}, \\
  G'' = 
  \begin{pmatrix}
    0 & 0 & \frac{1}{12} & \frac{1}{3} & 0
  \end{pmatrix}; \;
  B' = 
  \begin{pmatrix}
    \frac{5}{12}
  \end{pmatrix}, \;
  B'' = 
  \begin{pmatrix}
    \frac{7}{12}
  \end{pmatrix}; \;
  Q' =
  \begin{pmatrix}
    \frac{5}{12}
  \end{pmatrix}, \;
  Q'' =
  \begin{pmatrix}
    \frac{19}{12}
  \end{pmatrix}.
\end{gather*}


\subsection{Family \textnumero3.10}\label{subsection:03-10}

The pencil \(\mathcal{S}\) is defined by the equation
\begin{gather*}
  X^{2} Y^{2} + X^{2} Y Z + X Y^{2} Z + X Y Z^{2} + X^{2} Y T + X Z^{2} T + \\
  Y Z^{2} T + X Y T^{2} + X Z T^{2} + Y Z T^{2} + Z^{2} T^{2} = \lambda X Y Z T.
\end{gather*}
Members \(\mathcal{S}_{\lambda}\) of the pencil are irreducible for any \(\lambda \in \mathbb{P}^1\) except
\(\mathcal{S}_{\infty} = S_{(X)} + S_{(Y)} + S_{(Z)} + S_{(T)}\).
The base locus of the pencil \(\mathcal{S}\) consists of the following curves:
\begin{gather*}
  C_{1} = C_{(X, Z)}, \;
  C_{2} = C_{(X, T)}, \;
  C_{3} = C_{(Y, Z)}, \;
  C_{4} = C_{(Y, T)}, \;
  C_{5} = C_{(T, X + Z)}, \\
  C_{6} = C_{(T, Y + Z)}, \;
  C_{7} = C_{(X, Z T + Y (Z + T))}, \;
  C_{8} = C_{(Y, Z T + X (Z + T))}, \;
  C_{9} = C_{(Z, T^2 + X (Y + T))}.
\end{gather*}
Their linear equivalence classes on the generic member \(\mathcal{S}_{\Bbbk}\) of the pencil satisfy the following relations:
\begin{gather*}
  \begin{pmatrix}
    [C_{6}] \\ [C_{7}] \\ [C_{8}] \\ [C_{9}]
  \end{pmatrix} = 
  \begin{pmatrix}
    0 & -1 & 0 & -1 & -1 & 1 \\
    -1 & -1 & 0 & 0 & 0 & 1 \\
    0 & 0 & -1 & -1 & 0 & 1 \\
    -1 & 0 & -1 & 0 & 0 & 1
  \end{pmatrix} \cdot
  \begin{pmatrix}
    [C_{1}] & [C_{2}] & [C_{3}] & [C_{4}] & [C_{5}] & [H_{\mathcal{S}}]
  \end{pmatrix}^T.
\end{gather*}

For a general choice of \(\lambda \in \mathbb{C}\) the surface \(\mathcal{S}_{\lambda}\) has the following singularities:
\begin{itemize}\setlength{\itemindent}{2cm}
\item[\(P_{1} = P_{(X, Y, Z)}\):] type \(\mathbb{A}_2\) with the quadratic term \((X + Z) \cdot (Y + Z)\);
\item[\(P_{2} = P_{(X, Y, T)}\):] type \(\mathbb{A}_2\) with the quadratic term \((X + T) \cdot (Y + T)\);
\item[\(P_{3} = P_{(X, Z, T)}\):] type \(\mathbb{A}_4\) with the quadratic term \(X \cdot (X + Z)\);
\item[\(P_{4} = P_{(Y, Z, T)}\):] type \(\mathbb{A}_3\) with the quadratic term \(Y \cdot (Y + Z + T)\).
\end{itemize}

Galois action on the lattice \(L_{\lambda}\) is trivial. The intersection matrix on \(L_{\lambda} = L_{\mathcal{S}}\) is represented by
\begin{table}[H]
  \begin{tabular}{|c||cc|cc|cccc|ccc|cccccc|}
    \hline
    \(\bullet\) & \(E_1^1\) & \(E_1^2\) & \(E_2^1\) & \(E_2^2\) & \(E_3^1\) & \(E_3^2\) & \(E_3^3\) & \(E_3^4\) & \(E_4^1\) & \(E_4^2\) & \(E_4^3\) & \(\widetilde{C_{1}}\) & \(\widetilde{C_{2}}\) & \(\widetilde{C_{3}}\) & \(\widetilde{C_{4}}\) & \(\widetilde{C_{5}}\) & \(\widetilde{H_{\mathcal{S}}}\) \\
    \hline
    \hline
    \(\widetilde{C_{1}}\) & \(1\) & \(0\) & \(0\) & \(0\) & \(0\) & \(1\) & \(0\) & \(0\) & \(0\) & \(0\) & \(0\) & \(-2\) & \(0\) & \(0\) & \(0\) & \(0\) & \(1\) \\
    \(\widetilde{C_{2}}\) & \(0\) & \(0\) & \(1\) & \(0\) & \(1\) & \(0\) & \(0\) & \(0\) & \(0\) & \(0\) & \(0\) & \(0\) & \(-2\) & \(0\) & \(0\) & \(0\) & \(1\) \\
    \(\widetilde{C_{3}}\) & \(0\) & \(1\) & \(0\) & \(0\) & \(0\) & \(0\) & \(0\) & \(0\) & \(1\) & \(0\) & \(0\) & \(0\) & \(0\) & \(-2\) & \(0\) & \(0\) & \(1\) \\
    \(\widetilde{C_{4}}\) & \(0\) & \(0\) & \(0\) & \(1\) & \(0\) & \(0\) & \(0\) & \(0\) & \(0\) & \(0\) & \(1\) & \(0\) & \(0\) & \(0\) & \(-2\) & \(1\) & \(1\) \\
    \(\widetilde{C_{5}}\) & \(0\) & \(0\) & \(0\) & \(0\) & \(0\) & \(0\) & \(0\) & \(1\) & \(0\) & \(0\) & \(0\) & \(0\) & \(0\) & \(0\) & \(1\) & \(-2\) & \(1\) \\
    \(\widetilde{H_{\mathcal{S}}}\) & \(0\) & \(0\) & \(0\) & \(0\) & \(0\) & \(0\) & \(0\) & \(0\) & \(0\) & \(0\) & \(0\) & \(1\) & \(1\) & \(1\) & \(1\) & \(1\) & \(4\) \\
    \hline
  \end{tabular}.
\end{table}
Note that the intersection matrix is non-degenerate.

Discriminant groups and discriminant forms of the lattices \(L_{\mathcal{S}}\) and \(H \oplus \Pic(X)\) are given by
\begin{gather*}
  G' = 
  \begin{pmatrix}
    \frac{1}{2} & \frac{1}{2} & \frac{1}{2} & \frac{1}{2} & \frac{1}{2} & \frac{1}{2} & 0 & \frac{1}{2} & \frac{1}{2} & \frac{1}{2} & \frac{1}{2} & \frac{1}{2} & \frac{1}{2} & \frac{1}{2} & \frac{1}{2} & 0 & 0 \\
    0 & \frac{1}{2} & 0 & 0 & 0 & 0 & \frac{1}{2} & 0 & \frac{1}{2} & 0 & \frac{1}{2} & \frac{1}{2} & 0 & 0 & 0 & \frac{1}{2} & 0 \\
0 & \frac{1}{2} & \frac{1}{5} & \frac{9}{10} & \frac{7}{10} & \frac{9}{10} & \frac{3}{5} & \frac{3}{10} & \frac{2}{5} & \frac{4}{5} & \frac{1}{5} & \frac{1}{2} & \frac{1}{2} & 0 & \frac{3}{5} & 0 & \frac{1}{10}
  \end{pmatrix}, \\
  G'' = 
  \begin{pmatrix}
    0 & 0 & \frac{1}{2} & \frac{1}{2} & \frac{1}{2} \\
    0 & 0 & \frac{1}{2} & \frac{1}{2} & 0 \\
    0 & 0 & \frac{3}{5} & \frac{1}{10} & \frac{1}{10}
  \end{pmatrix}; \;
  B' = 
  \begin{pmatrix}
    \frac{1}{2} & 0 & \frac{1}{2} \\
    0 & 0 & \frac{1}{2} \\
    \frac{1}{2} & \frac{1}{2} & \frac{2}{5}
  \end{pmatrix}, \;
  B'' = 
  \begin{pmatrix}
    \frac{1}{2} & 0 & \frac{1}{2} \\
    0 & 0 & \frac{1}{2} \\
    \frac{1}{2} & \frac{1}{2} & \frac{3}{5}
  \end{pmatrix}; \;
  \begin{pmatrix}
    Q' \\ Q''
  \end{pmatrix}
  =
  \begin{pmatrix}
    \frac{3}{2} & 1 & \frac{2}{5} \\
    \frac{1}{2} & 1 & \frac{8}{5}
  \end{pmatrix}.
\end{gather*}


\subsection{Family \textnumero3.11}\label{subsection:03-11}

The pencil \(\mathcal{S}\) is defined by the equation
\begin{gather*}
  X^{2} Y Z + X Y^{2} Z + X Y Z^{2} + X Y^{2} T + Y^{2} Z T + X Z^{2} T + \\ Y Z^{2} T + Y^{2} T^{2} + X Z T^{2} + Y Z T^{2} + Z^{2} T^{2} + Z T^{3} = \lambda X Y Z T.
\end{gather*}
Members \(\mathcal{S}_{\lambda}\) of the pencil are irreducible for any \(\lambda \in \mathbb{P}^1\) except
\begin{gather*}
  \mathcal{S}_{\infty} = S_{(X)} + S_{(Y)} + S_{(Z)} + S_{(T)}, \;
  \mathcal{S}_{- 2} = S_{(X + T)} + S_{(X Y Z + (Z + T) (Z (Y + T) + Y^2))}.
\end{gather*}
The base locus of the pencil \(\mathcal{S}\) consists of the following curves:
\begin{gather*}
  C_1 = C_{(X, T)}, \;
  C_2 = C_{(Y, Z)}, \;
  C_3 = C_{(Y, T)}, \;
  C_4 = C_{(Z, T)}, \;
  C_5 = C_{(X, Z + T)}, \\
  C_6 = C_{(Y, X + T)}, \;
  C_7 = C_{(Y, Z + T)}, \;
  C_8 = C_{(Z, X + T)}, \;
  C_9 = C_{(T, X + Y + Z)}, \;
  C_{10} = C_{(X, Y^2 + Z (Y + T))}.
\end{gather*}
Their linear equivalence classes on the generic member \(\mathcal{S}_{\Bbbk}\) of the pencil satisfy the following relations:
\begin{gather*}
  \begin{pmatrix}
    [C_{6}] \\ [C_{7}] \\ [C_{8}] \\ [C_{9}] \\ [C_{10}]
  \end{pmatrix} = 
  \begin{pmatrix}
    -2 & 2 & 0 & 1 & 0 & 0 \\
    2 & -3 & -1 & -1 & 0 & 1 \\
    0 & -2 & 0 & -1 & 0 & 1 \\
    -1 & 0 & -1 & -1 & 0 & 1 \\
    -1 & 0 & 0 & 0 & -1 & 1
  \end{pmatrix} \cdot
  \begin{pmatrix}
    [C_{1}] & [C_{2}] & [C_{3}] & [C_{4}] & [C_{5}] & [H_{\mathcal{S}}]
  \end{pmatrix}^T.
\end{gather*}

For a general choice of \(\lambda \in \mathbb{C}\) the surface \(\mathcal{S}_{\lambda}\) has the following singularities:
\begin{itemize}\setlength{\itemindent}{2cm}
\item[\(P_{1} = P_{(X, Y, T)}\):] type \(\mathbb{A}_2\) with the quadratic term \((X + T) \cdot (Y + T)\);
\item[\(P_{2} = P_{(X, Z, T)}\):] type \(\mathbb{A}_2\) with the quadratic term \((X + T) \cdot (Z + T)\);
\item[\(P_{3} = P_{(Y, Z, T)}\):] type \(\mathbb{A}_4\) with the quadratic term \(Y \cdot Z\);
\item[\(P_{4} = P_{(X, T, Y + Z)}\):] type \(\mathbb{A}_1\) with the quadratic term \((X + T) (X + Y + Z - T) - (\lambda + 2) X T\);
\item[\(P_{5} = P_{(Y, Z, X + T)}\):] type \(\mathbb{A}_2\) with the quadratic term \(Z \cdot (X + (\lambda + 2) Y + T)\).
\end{itemize}

Galois action on the lattice \(L_{\lambda}\) is trivial. The intersection matrix on \(L_{\lambda} = L_{\mathcal{S}}\) is represented by
\begin{table}[H]
  \begin{tabular}{|c||cc|cc|cccc|c|cc|cccccc|}
    \hline
    \(\bullet\) & \(E_1^1\) & \(E_1^2\) & \(E_2^1\) & \(E_2^2\) & \(E_3^1\) & \(E_3^2\) & \(E_3^3\) & \(E_3^4\) & \(E_4^1\) & \(E_5^1\) & \(E_5^2\) & \(\widetilde{C_{1}}\) & \(\widetilde{C_{2}}\) & \(\widetilde{C_{3}}\) & \(\widetilde{C_{4}}\) & \(\widetilde{C_{5}}\) & \(\widetilde{H_{\mathcal{S}}}\) \\
    \hline
    \hline
    \(\widetilde{C_{1}}\) & \(1\) & \(0\) & \(1\) & \(0\) & \(0\) & \(0\) & \(0\) & \(0\) & \(1\) & \(0\) & \(0\) & \(-2\) & \(0\) & \(0\) & \(0\) & \(0\) & \(1\) \\
    \(\widetilde{C_{2}}\) & \(0\) & \(0\) & \(0\) & \(0\) & \(0\) & \(0\) & \(1\) & \(0\) & \(0\) & \(1\) & \(0\) & \(0\) & \(-2\) & \(0\) & \(0\) & \(0\) & \(1\) \\
    \(\widetilde{C_{3}}\) & \(0\) & \(1\) & \(0\) & \(0\) & \(1\) & \(0\) & \(0\) & \(0\) & \(0\) & \(0\) & \(0\) & \(0\) & \(0\) & \(-2\) & \(0\) & \(0\) & \(1\) \\
    \(\widetilde{C_{4}}\) & \(0\) & \(0\) & \(0\) & \(1\) & \(0\) & \(0\) & \(0\) & \(1\) & \(0\) & \(0\) & \(0\) & \(0\) & \(0\) & \(0\) & \(-2\) & \(0\) & \(1\) \\
    \(\widetilde{C_{5}}\) & \(0\) & \(0\) & \(0\) & \(1\) & \(0\) & \(0\) & \(0\) & \(0\) & \(0\) & \(0\) & \(0\) & \(0\) & \(0\) & \(0\) & \(0\) & \(-2\) & \(1\) \\
    \(\widetilde{H_{\mathcal{S}}}\) & \(0\) & \(0\) & \(0\) & \(0\) & \(0\) & \(0\) & \(0\) & \(0\) & \(0\) & \(0\) & \(0\) & \(1\) & \(1\) & \(1\) & \(1\) & \(1\) & \(4\) \\
    \hline
  \end{tabular}.
\end{table}
Note that the intersection matrix is non-degenerate.

Discriminant groups and discriminant forms of the lattices \(L_{\mathcal{S}}\) and \(H \oplus \Pic(X)\) are given by
\begin{gather*}
  G' = 
  \begin{pmatrix}
    \frac{23}{28} & \frac{5}{7} & \frac{6}{7} & \frac{11}{14} & \frac{2}{7} & \frac{27}{28} & \frac{9}{14} & \frac{3}{7} & \frac{27}{28} & \frac{13}{14} & \frac{27}{28} & \frac{13}{14} & \frac{25}{28} & \frac{17}{28} & \frac{3}{14} & \frac{1}{2} & \frac{3}{14}
  \end{pmatrix}, \\
  G'' = 
  \begin{pmatrix}
    0 & 0 & -\frac{2}{7} & -\frac{1}{7} & \frac{1}{28}
  \end{pmatrix}; \;
  B' = 
  \begin{pmatrix}
    \frac{1}{28}
  \end{pmatrix}, \;
  B'' = 
  \begin{pmatrix}
    \frac{27}{28}
  \end{pmatrix}; \;
  Q' =
  \begin{pmatrix}
    \frac{1}{28}
  \end{pmatrix}, \;
  Q'' =
  \begin{pmatrix}
    \frac{55}{28}
  \end{pmatrix}.
\end{gather*}


\subsection{Family \textnumero3.12}\label{subsection:03-12}

The pencil \(\mathcal{S}\) is defined by the equation
\begin{gather*}
  X^{2} Y^{2} + X^{2} Y Z + X Y^{2} Z + X Y Z^{2} + X Y^{2} T + X Z^{2} T + Y Z^{2} T + X Y T^{2} + X Z T^{2} + Y Z T^{2} = \lambda X Y Z T.
\end{gather*}
Members \(\mathcal{S}_{\lambda}\) of the pencil are irreducible for any \(\lambda \in \mathbb{P}^1\) except
\(\mathcal{S}_{\infty} = S_{(X)} + S_{(Y)} + S_{(Z)} + S_{(T)}\).
The base locus of the pencil \(\mathcal{S}\) consists of the following curves:
\begin{gather*}
  C_1 = C_{(X, Y)}, \;
  C_2 = C_{(X, Z)}, \;
  C_3 = C_{(X, T)}, \;
  C_4 = C_{(Y, Z)}, \;
  C_5 = C_{(Y, T)}, \\
  C_6 = C_{(X, Z + T)}, \;
  C_7 = C_{(Y, Z + T)}, \;
  C_8 = C_{(T, X + Z)}, \;
  C_9 = C_{(T, Y + Z)}, \;
  C_{10} = C_{(Z, T^2 + Y (X + T))}.
\end{gather*}
Their linear equivalence classes on the generic member \(\mathcal{S}_{\Bbbk}\) of the pencil satisfy the following relations:
\begin{gather*}
  \begin{pmatrix}
    [C_{6}] \\ [C_{7}] \\ [C_{9}] \\ [C_{10}]
  \end{pmatrix} = 
  \begin{pmatrix}
    -1 & -1 & -1 & 0 & 0 & 0 & 1 \\
    -1 & 0 & 0 & -1 & -1 & 0 & 1 \\
    0 & 0 & -1 & 0 & -1 & -1 & 1 \\
    0 & -1 & 0 & -1 & 0 & 0 & 1
  \end{pmatrix} \cdot
  \begin{pmatrix}
    [C_{1}] & [C_{2}] & [C_{3}] & [C_{4}] & [C_{5}] & [C_{8}] & [H_{\mathcal{S}}]
  \end{pmatrix}^T.
\end{gather*}

For a general choice of \(\lambda \in \mathbb{C}\) the surface \(\mathcal{S}_{\lambda}\) has the following singularities:
\begin{itemize}\setlength{\itemindent}{2cm}
\item[\(P_{1} = P_{(X, Y, Z)}\):] type \(\mathbb{A}_1\) with the quadratic term \(X Y + Z (X + Y)\);
\item[\(P_{2} = P_{(X, Y, T)}\):] type \(\mathbb{A}_1\) with the quadratic term \(X Y + T (X + Y)\);
\item[\(P_{3} = P_{(X, Z, T)}\):] type \(\mathbb{A}_3\) with the quadratic term \(X \cdot (X + Z + T)\);
\item[\(P_{4} = P_{(Y, Z, T)}\):] type \(\mathbb{A}_4\) with the quadratic term \(Y \cdot (Y + Z)\);
\item[\(P_{5} = P_{(X, Y, Z + T)}\):] type \(\mathbb{A}_1\) with the quadratic term \((X + Y) (Z + T) - (\lambda + 2) X Y\).
\end{itemize}

Galois action on the lattice \(L_{\lambda}\) is trivial. The intersection matrix on \(L_{\lambda} = L_{\mathcal{S}}\) is represented by
\begin{table}[H]
  \begin{tabular}{|c||c|c|ccc|cccc|c|ccccccc|}
    \hline
    \(\bullet\) & \(E_1^1\) & \(E_2^1\) & \(E_3^1\) & \(E_3^2\) & \(E_3^3\) & \(E_4^1\) & \(E_4^2\) & \(E_4^3\) & \(E_4^4\) & \(E_5^1\) & \(\widetilde{C_{1}}\) & \(\widetilde{C_{2}}\) & \(\widetilde{C_{3}}\) & \(\widetilde{C_{4}}\) & \(\widetilde{C_{5}}\) & \(\widetilde{C_{8}}\) & \(\widetilde{H_{\mathcal{S}}}\) \\
    \hline
    \hline
    \(\widetilde{C_{1}}\) & \(1\) & \(1\) & \(0\) & \(0\) & \(0\) & \(0\) & \(0\) & \(0\) & \(0\) & \(1\) & \(-2\) & \(0\) & \(0\) & \(0\) & \(0\) & \(0\) & \(1\) \\
    \(\widetilde{C_{2}}\) & \(1\) & \(0\) & \(1\) & \(0\) & \(0\) & \(0\) & \(0\) & \(0\) & \(0\) & \(0\) & \(0\) & \(-2\) & \(0\) & \(0\) & \(0\) & \(0\) & \(1\) \\
    \(\widetilde{C_{3}}\) & \(0\) & \(1\) & \(1\) & \(0\) & \(0\) & \(0\) & \(0\) & \(0\) & \(0\) & \(0\) & \(0\) & \(0\) & \(-2\) & \(0\) & \(0\) & \(0\) & \(1\) \\
    \(\widetilde{C_{4}}\) & \(1\) & \(0\) & \(0\) & \(0\) & \(0\) & \(0\) & \(1\) & \(0\) & \(0\) & \(0\) & \(0\) & \(0\) & \(0\) & \(-2\) & \(0\) & \(0\) & \(1\) \\
    \(\widetilde{C_{5}}\) & \(0\) & \(1\) & \(0\) & \(0\) & \(0\) & \(0\) & \(0\) & \(0\) & \(1\) & \(0\) & \(0\) & \(0\) & \(0\) & \(0\) & \(-2\) & \(1\) & \(1\) \\
    \(\widetilde{C_{8}}\) & \(0\) & \(0\) & \(0\) & \(0\) & \(1\) & \(0\) & \(0\) & \(0\) & \(0\) & \(0\) & \(0\) & \(0\) & \(0\) & \(0\) & \(1\) & \(-2\) & \(1\) \\
    \(\widetilde{H_{\mathcal{S}}}\) & \(0\) & \(0\) & \(0\) & \(0\) & \(0\) & \(0\) & \(0\) & \(0\) & \(0\) & \(0\) & \(1\) & \(1\) & \(1\) & \(1\) & \(1\) & \(1\) & \(4\) \\
    \hline
  \end{tabular}.
\end{table}
Note that the intersection matrix is non-degenerate.

Discriminant groups and discriminant forms of the lattices \(L_{\mathcal{S}}\) and \(H \oplus \Pic(X)\) are given by
\begin{gather*}
  G' = 
  \begin{pmatrix}
    \frac{29}{36} & \frac{5}{36} & \frac{4}{9} & \frac{5}{6} & \frac{2}{9} & \frac{8}{9} & \frac{7}{9} & \frac{29}{36} & \frac{5}{6} & \frac{1}{36} & \frac{1}{18} & \frac{25}{36} & \frac{13}{36} & \frac{31}{36} & \frac{31}{36} & \frac{11}{18} & \frac{5}{36}
  \end{pmatrix}, \\
  G'' = 
  \begin{pmatrix}
    0 & 0 & \frac{1}{36} & -\frac{5}{12} & \frac{1}{18}
  \end{pmatrix}; \;
  B' = 
  \begin{pmatrix}
    \frac{17}{36}
  \end{pmatrix}, \;
  B'' = 
  \begin{pmatrix}
    \frac{19}{36}
  \end{pmatrix}; \;
  Q' =
  \begin{pmatrix}
    \frac{17}{36}
  \end{pmatrix}, \;
  Q'' =
  \begin{pmatrix}
    \frac{55}{36}
  \end{pmatrix}.
\end{gather*}


\subsection{Family \textnumero3.13}\label{subsection:03-13}

The pencil \(\mathcal{S}\) is defined by the equation
\begin{gather*}
  X^{2} Y Z + X Y^{2} Z + X Y Z^{2} + X Y^{2} T + X^{2} Z T + Y Z^{2} T + X Y T^{2} + X Z T^{2} + Y Z T^{2} = \lambda X Y Z T.
\end{gather*}
Members \(\mathcal{S}_{\lambda}\) of the pencil are irreducible for any \(\lambda \in \mathbb{P}^1\) except
\(\mathcal{S}_{\infty} = S_{(X)} + S_{(Y)} + S_{(Z)} + S_{(T)}\).
The base locus of the pencil \(\mathcal{S}\) consists of the following curves:
\begin{gather*}
  C_{1} = C_{(X, Y)}, \;
  C_{2} = C_{(X, Z)}, \;
  C_{3} = C_{(X, T)}, \;
  C_{4} = C_{(Y, Z)}, \;
  C_{5} = C_{(Y, T)}, \\
  C_{6} = C_{(Z, T)}, \;
  C_{7} = C_{(X, Z + T)}, \;
  C_{8} = C_{(Y, X + T)}, \;
  C_{9} = C_{(Z, Y + T)}, \;
  C_{10} = C_{(T, X + Y + Z)}.
\end{gather*}
Their linear equivalence classes on the generic member \(\mathcal{S}_{\Bbbk}\) of the pencil satisfy the following relations:
\begin{gather*}
  \begin{pmatrix}
    [C_{7}] \\ [C_{8}] \\ [C_{9}] \\ [C_{10}]
  \end{pmatrix} = 
  \begin{pmatrix}
    -1 & -1 & -1 & 0 & 0 & 0 & 1 \\
    -1 & 0 & 0 & -1 & -1 & 0 & 1 \\
    0 & -1 & 0 & -1 & 0 & -1 & 1 \\
    0 & 0 & -1 & 0 & -1 & -1 & 1
  \end{pmatrix} \cdot
  \begin{pmatrix}
    [C_{1}] & [C_{2}] & [C_{3}] & [C_{4}] & [C_{5}] & [C_{6}] & [H_{\mathcal{S}}]
  \end{pmatrix}^T.
\end{gather*}

For a general choice of \(\lambda \in \mathbb{C}\) the surface \(\mathcal{S}_{\lambda}\) has the following singularities:
\begin{itemize}\setlength{\itemindent}{2cm}
\item[\(P_{1} = P_{(X, Y, Z)}\):] type \(\mathbb{A}_1\) with the quadratic term \(X Y + X Z + Y Z\);
\item[\(P_{2} = P_{(X, Y, T)}\):] type \(\mathbb{A}_3\) with the quadratic term \(Y \cdot (X + T)\);
\item[\(P_{3} = P_{(X, Z, T)}\):] type \(\mathbb{A}_3\) with the quadratic term \(X \cdot (Z + T)\);
\item[\(P_{4} = P_{(Y, Z, T)}\):] type \(\mathbb{A}_3\) with the quadratic term \(Z \cdot (Y + T)\).
\end{itemize}

Galois action on the lattice \(L_{\lambda}\) is trivial. The intersection matrix on \(L_{\lambda} = L_{\mathcal{S}}\) is represented by
\begin{table}[H]
  \begin{tabular}{|c||c|ccc|ccc|ccc|ccccccc|}
    \hline
    \(\bullet\) & \(E_1^1\) & \(E_2^1\) & \(E_2^2\) & \(E_2^3\) & \(E_3^1\) & \(E_3^2\) & \(E_3^3\) & \(E_4^1\) & \(E_4^2\) & \(E_4^3\) & \(\widetilde{C_{1}}\) & \(\widetilde{C_{2}}\) & \(\widetilde{C_{3}}\) & \(\widetilde{C_{4}}\) & \(\widetilde{C_{5}}\) & \(\widetilde{C_{6}}\) & \(\widetilde{H_{\mathcal{S}}}\) \\
    \hline
    \hline
    \(\widetilde{C_{1}}\) & \(1\) & \(1\) & \(0\) & \(0\) & \(0\) & \(0\) & \(0\) & \(0\) & \(0\) & \(0\) & \(-2\) & \(0\) & \(0\) & \(0\) & \(0\) & \(0\) & \(1\) \\
    \(\widetilde{C_{2}}\) & \(1\) & \(0\) & \(0\) & \(0\) & \(1\) & \(0\) & \(0\) & \(0\) & \(0\) & \(0\) & \(0\) & \(-2\) & \(0\) & \(0\) & \(0\) & \(0\) & \(1\) \\
    \(\widetilde{C_{3}}\) & \(0\) & \(0\) & \(0\) & \(1\) & \(1\) & \(0\) & \(0\) & \(0\) & \(0\) & \(0\) & \(0\) & \(0\) & \(-2\) & \(0\) & \(0\) & \(0\) & \(1\) \\
    \(\widetilde{C_{4}}\) & \(1\) & \(0\) & \(0\) & \(0\) & \(0\) & \(0\) & \(0\) & \(1\) & \(0\) & \(0\) & \(0\) & \(0\) & \(0\) & \(-2\) & \(0\) & \(0\) & \(1\) \\
    \(\widetilde{C_{5}}\) & \(0\) & \(1\) & \(0\) & \(0\) & \(0\) & \(0\) & \(0\) & \(0\) & \(0\) & \(1\) & \(0\) & \(0\) & \(0\) & \(0\) & \(-2\) & \(0\) & \(1\) \\
    \(\widetilde{C_{6}}\) & \(0\) & \(0\) & \(0\) & \(0\) & \(0\) & \(0\) & \(1\) & \(1\) & \(0\) & \(0\) & \(0\) & \(0\) & \(0\) & \(0\) & \(0\) & \(-2\) & \(1\) \\
    \(\widetilde{H_{\mathcal{S}}}\) & \(0\) & \(0\) & \(0\) & \(0\) & \(0\) & \(0\) & \(0\) & \(0\) & \(0\) & \(0\) & \(1\) & \(1\) & \(1\) & \(1\) & \(1\) & \(1\) & \(4\) \\
    \hline
  \end{tabular}.
\end{table}
Note that the intersection matrix is non-degenerate.

Discriminant groups and discriminant forms of the lattices \(L_{\mathcal{S}}\) and \(H \oplus \Pic(X)\) are given by
\begin{gather*}
  G' = 
  \begin{pmatrix}
    \frac{1}{2} & \frac{1}{2} & 0 & \frac{1}{2} & \frac{1}{2} & 0 & \frac{1}{2} & \frac{1}{2} & 0 & \frac{1}{2} & 0 & 0 & 0 & 0 & 0 & 0 & 0 \\
    \frac{1}{2} & \frac{1}{2} & \frac{1}{2} & \frac{1}{2} & \frac{1}{2} & \frac{1}{2} & \frac{1}{2} & \frac{1}{2} & 0 & \frac{1}{2} & \frac{1}{2} & 0 & \frac{1}{2} & \frac{1}{2} & 0 & \frac{1}{2} & 0 \\
    \frac{3}{5} & \frac{3}{5} & 0 & \frac{2}{5} & \frac{3}{5} & \frac{1}{2} & \frac{2}{5} & \frac{3}{5} & \frac{1}{2} & \frac{2}{5} & \frac{9}{10} & \frac{9}{10} & \frac{4}{5} & \frac{2}{5} & \frac{3}{10} & \frac{3}{10} & \frac{3}{5}
  \end{pmatrix}, \\
  G'' = 
  \begin{pmatrix}
    0 & 0 & \frac{1}{2} & 0 & 0 \\
    0 & 0 & 0 & \frac{1}{2} & 0 \\
    0 & 0 & -\frac{3}{10} & -\frac{2}{5} & \frac{1}{10}
  \end{pmatrix}; \;
  B' = 
  \begin{pmatrix}
    \frac{1}{2} & \frac{1}{2} & 0 \\
    \frac{1}{2} & \frac{1}{2} & \frac{1}{2} \\
    0 & \frac{1}{2} & \frac{4}{5}
  \end{pmatrix}, \;
  B'' = 
  \begin{pmatrix}
    \frac{1}{2} & \frac{1}{2} & 0 \\
    \frac{1}{2} & \frac{1}{2} & \frac{1}{2} \\
    0 & \frac{1}{2} & \frac{1}{5}
  \end{pmatrix}; \;
  \begin{pmatrix}
    Q' \\ Q''
  \end{pmatrix}
  =
  \begin{pmatrix}
    \frac{1}{2} & \frac{3}{2} & \frac{9}{5} \\
  \frac{3}{2} & \frac{1}{2} & \frac{1}{5}
  \end{pmatrix}.
\end{gather*}


\subsection{Family \textnumero3.14}\label{subsection:03-14}

The pencil \(\mathcal{S}\) is defined by the equation
\begin{gather*}
  X^{2} Y Z + X Y^{2} Z + X Y Z^{2} + X^{3} T + Y^{2} Z T + Y Z^{2} T + X^{2} T^{2} + Y Z T^{2} = \lambda X Y Z T.
\end{gather*}
Members \(\mathcal{S}_{\lambda}\) of the pencil are irreducible for any \(\lambda \in \mathbb{P}^1\) except
\begin{gather*}
  \mathcal{S}_{\infty} = S_{(X)} + S_{(Y)} + S_{(Z)} + S_{(T)}, \;
  \mathcal{S}_{- 2} = S_{(X + T)} + S_{(Y Z (X + Y + Z + T) + X^2 T)}.
\end{gather*}
The base locus of the pencil \(\mathcal{S}\) consists of the following curves:
\begin{gather*}
  C_{1} = C_{(X, Y)}, \;
  C_{2} = C_{(X, Z)}, \;
  C_{3} = C_{(X, T)}, \;
  C_{4} = C_{(Y, T)}, \;
  C_{5} = C_{(Z, T)}, \\
  C_{6} = C_{(Y, X + T)}, \;
  C_{7} = C_{(Z, X + T)}, \;
  C_{8} = C_{(X, Y + Z + T)}, \;
  C_{9} = C_{(T, X + Y + Z)}.
\end{gather*}
Their linear equivalence classes on the generic member \(\mathcal{S}_{\Bbbk}\) of the pencil satisfy the following relations:
\begin{gather*}
  \begin{pmatrix}
    [C_{5}] \\ [C_{6}] \\ [C_{7}] \\ [C_{8}] \\ [C_{9}]
  \end{pmatrix} = 
  \begin{pmatrix}
    -2 & -2 & 2 & -1 & 1 \\
    -2 & 0 & 0 & -1 & 1 \\
    2 & 0 & -2 & 1 & 0 \\
    -1 & -1 & -1 & 0 & 1 \\
    2 & 2 & -3 & 0 & 0
  \end{pmatrix} \cdot
  \begin{pmatrix}
    [C_{1}] \\ [C_{2}] \\ [C_{3}] \\ [C_{4}] \\ [H_{\mathcal{S}}]
  \end{pmatrix}.
\end{gather*}

For a general choice of \(\lambda \in \mathbb{C}\) the surface \(\mathcal{S}_{\lambda}\) has the following singularities:
\begin{itemize}\setlength{\itemindent}{2cm}
\item[\(P_{1} = P_{(X, Y, Z)}\):] type \(\mathbb{A}_1\) with the quadratic term \(X^2 + Y Z\);
\item[\(P_{2} = P_{(X, Y, T)}\):] type \(\mathbb{A}_4\) with the quadratic term \(Y \cdot (X + T)\);
\item[\(P_{3} = P_{(X, Z, T)}\):] type \(\mathbb{A}_4\) with the quadratic term \(Z \cdot (X + T)\);
\item[\(P_{4} = P_{(X, Y, Z + T)}\):] type \(\mathbb{A}_1\) with the quadratic term \(X (X + (\lambda + 2) Y) - Y (X + Y + Z + T)\);
\item[\(P_{5} = P_{(X, Z, Y + T)}\):] type \(\mathbb{A}_1\) with the quadratic term \(X (X + (\lambda + 2) Z) - Z (X + Y + Z + T)\);
\item[\(P_{6} = P_{(X, T, Y + Z)}\):] type \(\mathbb{A}_1\) with the quadratic term \((X + T) (X + Y + Z + T) - (\lambda + 2) X T\).
\end{itemize}

Galois action on the lattice \(L_{\lambda}\) is trivial. The intersection matrix on \(L_{\lambda} = L_{\mathcal{S}}\) is represented by
\begin{table}[H]
  \begin{tabular}{|c||c|cccc|cccc|c|c|c|ccccc|}
    \hline
    \(\bullet\) & \(E_1^1\) & \(E_2^1\) & \(E_2^2\) & \(E_2^3\) & \(E_2^4\) & \(E_3^1\) & \(E_3^2\) & \(E_3^3\) & \(E_3^4\) & \(E_4^1\) & \(E_5^1\) & \(E_6^1\) & \(\widetilde{C_{1}}\) & \(\widetilde{C_{2}}\) & \(\widetilde{C_{3}}\) & \(\widetilde{C_{4}}\) & \(\widetilde{H_{\mathcal{S}}}\) \\
    \hline
    \hline
    \(\widetilde{C_{1}}\) & \(1\) & \(1\) & \(0\) & \(0\) & \(0\) & \(0\) & \(0\) & \(0\) & \(0\) & \(1\) & \(0\) & \(0\) & \(-2\) & \(0\) & \(0\) & \(0\) & \(1\) \\
    \(\widetilde{C_{2}}\) & \(1\) & \(0\) & \(0\) & \(0\) & \(0\) & \(1\) & \(0\) & \(0\) & \(0\) & \(0\) & \(1\) & \(0\) & \(0\) & \(-2\) & \(0\) & \(0\) & \(1\) \\
    \(\widetilde{C_{3}}\) & \(0\) & \(0\) & \(0\) & \(0\) & \(1\) & \(0\) & \(0\) & \(0\) & \(1\) & \(0\) & \(0\) & \(1\) & \(0\) & \(0\) & \(-2\) & \(0\) & \(1\) \\
    \(\widetilde{C_{4}}\) & \(0\) & \(1\) & \(0\) & \(0\) & \(0\) & \(0\) & \(0\) & \(0\) & \(0\) & \(0\) & \(0\) & \(0\) & \(0\) & \(0\) & \(0\) & \(-2\) & \(1\) \\
    \(\widetilde{H_{\mathcal{S}}}\) & \(0\) & \(0\) & \(0\) & \(0\) & \(0\) & \(0\) & \(0\) & \(0\) & \(0\) & \(0\) & \(0\) & \(0\) & \(1\) & \(1\) & \(1\) & \(1\) & \(4\) \\
    \hline
  \end{tabular}.
\end{table}
Note that the intersection matrix is non-degenerate.

Discriminant groups and discriminant forms of the lattices \(L_{\mathcal{S}}\) and \(H \oplus \Pic(X)\) are given by
\begin{gather*}
  G' = 
  \begin{pmatrix}
    \frac{7}{18} & \frac{2}{9} & \frac{4}{9} & \frac{2}{3} & \frac{8}{9} & \frac{5}{9} & \frac{4}{9} & \frac{1}{3} & \frac{2}{9} & \frac{1}{18} & \frac{5}{6} & \frac{5}{9} & \frac{1}{9} & \frac{2}{3} & \frac{1}{9} & \frac{8}{9} & \frac{5}{9}
  \end{pmatrix}, \\
  G'' = 
  \begin{pmatrix}
    0 & 0 & -\frac{1}{2} & \frac{1}{9} & -\frac{1}{3}
  \end{pmatrix}; \;
  B' = 
  \begin{pmatrix}
    \frac{5}{18}
  \end{pmatrix}, \;
  B'' = 
  \begin{pmatrix}
    \frac{13}{18}
  \end{pmatrix}; \;
  Q' =
  \begin{pmatrix}
    \frac{5}{18}
  \end{pmatrix}, \;
  Q'' =
  \begin{pmatrix}
    \frac{31}{18}
  \end{pmatrix}.
\end{gather*}


\subsection{Family \textnumero3.15}\label{subsection:03-15}

The pencil \(\mathcal{S}\) is defined by the equation
\begin{gather*}
  X^{2} Y Z + X Y^{2} Z + X Y Z^{2} + X^{2} Y T + X Z^{2} T + X Y T^{2} + X Z T^{2} + Y Z T^{2} + Z^{2} T^{2} = \lambda X Y Z T.
\end{gather*}
Members \(\mathcal{S}_{\lambda}\) of the pencil are irreducible for any \(\lambda \in \mathbb{P}^1\) except
\(\mathcal{S}_{\infty} = S_{(X)} + S_{(Y)} + S_{(Z)} + S_{(T)}\).
The base locus of the pencil \(\mathcal{S}\) consists of the following curves:
\begin{gather*}
  C_{1} = C_{(X, Z)}, \;
  C_{2} = C_{(X, T)}, \;
  C_{3} = C_{(Y, Z)}, \;
  C_{4} = C_{(Y, T)}, \;
  C_{5} = C_{(Z, T)}, \\
  C_{6} = C_{(X, Y + Z)}, \;
  C_{7} = C_{(Z, X + T)}, \;
  C_{8} = C_{(T, X + Y + Z)}, \;
  C_{9} = C_{(Y, X Z + T (X + Z))}.
\end{gather*}
Their linear equivalence classes on the generic member \(\mathcal{S}_{\Bbbk}\) of the pencil satisfy the following relations:
\begin{gather*}
  \begin{pmatrix}
    [C_{6}] \\ [C_{7}] \\ [C_{8}] \\ [C_{9}]
  \end{pmatrix} = 
  \begin{pmatrix}
    -1 & -2 & 0 & 0 & 0 & 1 \\
    -1 & 0 & -1 & 0 & -1 & 1 \\
    0 & -1 & 0 & -1 & -1 & 1 \\
    0 & 0 & -1 & -1 & 0 & 1
  \end{pmatrix} \cdot
  \begin{pmatrix}
    [C_{1}] & [C_{2}] & [C_{3}] & [C_{4}] & [C_{5}] & [H_{\mathcal{S}}]
  \end{pmatrix}^T.
\end{gather*}

For a general choice of \(\lambda \in \mathbb{C}\) the surface \(\mathcal{S}_{\lambda}\) has the following singularities:
\begin{itemize}\setlength{\itemindent}{2cm}
\item[\(P_{1} = P_{(X, Y, Z)}\):] type \(\mathbb{A}_2\) with the quadratic term \((X + Z) \cdot (Y + Z)\);
\item[\(P_{2} = P_{(X, Y, T)}\):] type \(\mathbb{A}_1\) with the quadratic term \(X (Y + T) + T^2\);
\item[\(P_{3} = P_{(X, Z, T)}\):] type \(\mathbb{A}_3\) with the quadratic term \(X \cdot Z\);
\item[\(P_{4} = P_{(Y, Z, T)}\):] type \(\mathbb{A}_3\) with the quadratic term \(Y \cdot (Z + T)\);
\item[\(P_{5} = P_{(X, T, Y + Z)}\):] type \(\mathbb{A}_2\) with the quadratic term \(X \cdot (X + Y + Z - (\lambda + 1) T)\).
\end{itemize}

Galois action on the lattice \(L_{\lambda}\) is trivial. The intersection matrix on \(L_{\lambda} = L_{\mathcal{S}}\) is represented by
\begin{table}[H]
  \begin{tabular}{|c||cc|c|ccc|ccc|cc|cccccc|}
    \hline
    \(\bullet\) & \(E_1^1\) & \(E_1^2\) & \(E_2^1\) & \(E_3^1\) & \(E_3^2\) & \(E_3^3\) & \(E_4^1\) & \(E_4^2\) & \(E_4^3\) & \(E_5^1\) & \(E_5^2\) & \(\widetilde{C_{1}}\) & \(\widetilde{C_{2}}\) & \(\widetilde{C_{3}}\) & \(\widetilde{C_{4}}\) & \(\widetilde{C_{5}}\) & \(\widetilde{H_{\mathcal{S}}}\) \\
    \hline
    \hline
    \(\widetilde{C_{1}}\) & \(1\) & \(0\) & \(0\) & \(0\) & \(1\) & \(0\) & \(0\) & \(0\) & \(0\) & \(0\) & \(0\) & \(-2\) & \(0\) & \(0\) & \(0\) & \(0\) & \(1\) \\
    \(\widetilde{C_{2}}\) & \(0\) & \(0\) & \(1\) & \(1\) & \(0\) & \(0\) & \(0\) & \(0\) & \(0\) & \(1\) & \(0\) & \(0\) & \(-2\) & \(0\) & \(0\) & \(0\) & \(1\) \\
    \(\widetilde{C_{3}}\) & \(0\) & \(1\) & \(0\) & \(0\) & \(0\) & \(0\) & \(1\) & \(0\) & \(0\) & \(0\) & \(0\) & \(0\) & \(0\) & \(-2\) & \(0\) & \(0\) & \(1\) \\
    \(\widetilde{C_{4}}\) & \(0\) & \(0\) & \(1\) & \(0\) & \(0\) & \(0\) & \(1\) & \(0\) & \(0\) & \(0\) & \(0\) & \(0\) & \(0\) & \(0\) & \(-2\) & \(0\) & \(1\) \\
    \(\widetilde{C_{5}}\) & \(0\) & \(0\) & \(0\) & \(0\) & \(0\) & \(1\) & \(0\) & \(0\) & \(1\) & \(0\) & \(0\) & \(0\) & \(0\) & \(0\) & \(0\) & \(-2\) & \(1\) \\
    \(\widetilde{H_{\mathcal{S}}}\) & \(0\) & \(0\) & \(0\) & \(0\) & \(0\) & \(0\) & \(0\) & \(0\) & \(0\) & \(0\) & \(0\) & \(1\) & \(1\) & \(1\) & \(1\) & \(1\) & \(4\) \\
    \hline
  \end{tabular}.
\end{table}
Note that the intersection matrix is non-degenerate.

Discriminant groups and discriminant forms of the lattices \(L_{\mathcal{S}}\) and \(H \oplus \Pic(X)\) are given by
\begin{gather*}
  G' = 
  \begin{pmatrix}
    \frac{16}{17} & \frac{15}{34} & \frac{23}{34} & \frac{6}{17} & \frac{31}{34} & \frac{1}{34} & \frac{7}{17} & \frac{11}{34} & \frac{4}{17} & \frac{9}{17} & \frac{9}{34} & \frac{15}{34} & \frac{27}{34} & \frac{16}{17} & \frac{19}{34} & \frac{5}{34} & \frac{1}{34}
  \end{pmatrix}, \\
  G'' = 
  \begin{pmatrix}
    0 & 0 & -\frac{8}{17} & -\frac{15}{34} & \frac{1}{17}
  \end{pmatrix}; \;
  B' = 
  \begin{pmatrix}
    \frac{33}{34}
  \end{pmatrix}, \;
  B'' = 
  \begin{pmatrix}
    \frac{1}{34}
  \end{pmatrix}; \;
  Q' =
  \begin{pmatrix}
    \frac{33}{34}
  \end{pmatrix}, \;
  Q'' =
  \begin{pmatrix}
    \frac{35}{34}
  \end{pmatrix}.
\end{gather*}


\subsection{Family \textnumero3.16}\label{subsection:03-16}

The pencil \(\mathcal{S}\) is defined by the equation
\begin{gather*}
  X^{2} Y Z + X Y^{2} Z + X Y Z^{2} + X Y^{2} T + X^{2} Z T + Y^{2} T^{2} + X Z T^{2} + Y Z T^{2} = \lambda X Y Z T.
\end{gather*}
Members \(\mathcal{S}_{\lambda}\) of the pencil are irreducible for any \(\lambda \in \mathbb{P}^1\) except
\(\mathcal{S}_{\infty} = S_{(X)} + S_{(Y)} + S_{(Z)} + S_{(T)}\).
The base locus of the pencil \(\mathcal{S}\) consists of the following curves:
\begin{gather*}
  C_{1} = C_{(X, Y)}, \;
  C_{2} = C_{(X, T)}, \;
  C_{3} = C_{(Y, T)}, \;
  C_{4} = C_{(Y, Z)}, \;
  C_{5} = C_{(Z, T)}, \\
  C_{6} = C_{(X, Y + Z)}, \;
  C_{7} = C_{(Y, X + T)}, \;
  C_{8} = C_{(Z, X + T)}, \;
  C_{9} = C_{(T, X + Y + Z)}.
\end{gather*}
Their linear equivalence classes on the generic member \(\mathcal{S}_{\Bbbk}\) of the pencil satisfy the following relations:
\begin{gather*}
  \begin{pmatrix}
    [C_{6}] \\ [C_{7}] \\ [C_{8}] \\ [C_{9}]
  \end{pmatrix} = 
  \begin{pmatrix}
    -1 & -2 & 0 & 0 & 0 & 1 \\
    -1 & 0 & -1 & -1 & 0 & 1 \\
    0 & 0 & 0 & -2 & -1 & 1 \\
    0 & -1 & -1 & 0 & -1 & 1
  \end{pmatrix} \cdot
  \begin{pmatrix}
    [C_{1}] & [C_{2}] & [C_{3}] & [C_{4}] & [C_{5}] & [H_{\mathcal{S}}]
  \end{pmatrix}^T.
\end{gather*}

For a general choice of \(\lambda \in \mathbb{C}\) the surface \(\mathcal{S}_{\lambda}\) has the following singularities:
\begin{itemize}\setlength{\itemindent}{2cm}
\item[\(P_{1} = P_{(X, Y, Z)}\):] type \(\mathbb{A}_1\) with the quadratic term \(Y^2 + Z (X + Y)\);
\item[\(P_{2} = P_{(X, Y, T)}\):] type \(\mathbb{A}_3\) with the quadratic term \(X \cdot Y\);
\item[\(P_{3} = P_{(X, Z, T)}\):] type \(\mathbb{A}_1\) with the quadratic term \(X (Z + T) + T^2\);
\item[\(P_{4} = P_{(Y, Z, T)}\):] type \(\mathbb{A}_2\) with the quadratic term \(Z \cdot (Y + T)\);
\item[\(P_{5} = P_{(X, T, Y + Z)}\):] type \(\mathbb{A}_2\) with the quadratic term \(X \cdot (X + Y + Z - (\lambda + 1) T)\);
\item[\(P_{6} = P_{(Y, Z, X + T)}\):] type \(\mathbb{A}_2\) with the quadratic term \(Z \cdot (X - (\lambda + 2) Y + T)\).
\end{itemize}

Galois action on the lattice \(L_{\lambda}\) is trivial. The intersection matrix on \(L_{\lambda} = L_{\mathcal{S}}\) is represented by
\begin{table}[H]
  \begin{tabular}{|c||c|ccc|c|cc|cc|cc|cccccc|}
    \hline
    \(\bullet\) & \(E_1^1\) & \(E_2^1\) & \(E_2^2\) & \(E_2^3\) & \(E_3^1\) & \(E_4^1\) & \(E_4^2\) & \(E_5^1\) & \(E_5^2\) & \(E_6^1\) & \(E_6^2\) & \(\widetilde{C_{1}}\) & \(\widetilde{C_{2}}\) & \(\widetilde{C_{3}}\) & \(\widetilde{C_{4}}\) & \(\widetilde{C_{5}}\) & \(\widetilde{H_{\mathcal{S}}}\) \\
    \hline
    \hline
    \(\widetilde{C_{1}}\) & \(1\) & \(0\) & \(1\) & \(0\) & \(0\) & \(0\) & \(0\) & \(0\) & \(0\) & \(0\) & \(0\) & \(-2\) & \(0\) & \(0\) & \(0\) & \(0\) & \(1\) \\
    \(\widetilde{C_{2}}\) & \(0\) & \(1\) & \(0\) & \(0\) & \(1\) & \(0\) & \(0\) & \(1\) & \(0\) & \(0\) & \(0\) & \(0\) & \(-2\) & \(0\) & \(0\) & \(0\) & \(1\) \\
    \(\widetilde{C_{3}}\) & \(0\) & \(0\) & \(0\) & \(1\) & \(0\) & \(1\) & \(0\) & \(0\) & \(0\) & \(0\) & \(0\) & \(0\) & \(0\) & \(-2\) & \(0\) & \(0\) & \(1\) \\
    \(\widetilde{C_{4}}\) & \(1\) & \(0\) & \(0\) & \(0\) & \(0\) & \(0\) & \(1\) & \(0\) & \(0\) & \(1\) & \(0\) & \(0\) & \(0\) & \(0\) & \(-2\) & \(0\) & \(1\) \\
    \(\widetilde{C_{5}}\) & \(0\) & \(0\) & \(0\) & \(0\) & \(1\) & \(0\) & \(1\) & \(0\) & \(0\) & \(0\) & \(0\) & \(0\) & \(0\) & \(0\) & \(0\) & \(-2\) & \(1\) \\
    \(\widetilde{H_{\mathcal{S}}}\) & \(0\) & \(0\) & \(0\) & \(0\) & \(0\) & \(0\) & \(0\) & \(0\) & \(0\) & \(0\) & \(0\) & \(1\) & \(1\) & \(1\) & \(1\) & \(1\) & \(4\) \\
    \hline
  \end{tabular}.
\end{table}
Note that the intersection matrix is non-degenerate.

Discriminant groups and discriminant forms of the lattices \(L_{\mathcal{S}}\) and \(H \oplus \Pic(X)\) are given by
\begin{gather*}
  G' = 
  \begin{pmatrix}
    \frac{1}{3} & \frac{8}{15} & \frac{1}{15} & \frac{7}{30} & \frac{4}{5} & \frac{7}{30} & \frac{1}{15} & \frac{1}{3} & \frac{2}{3} & \frac{13}{15} & \frac{13}{30} & \frac{11}{30} & 0 & \frac{2}{5} & \frac{3}{10} & \frac{3}{5} & \frac{1}{3}
  \end{pmatrix}, \\
  G'' = 
  \begin{pmatrix}
    0 & 0 & -\frac{7}{15} & -\frac{7}{30} & \frac{1}{10}
  \end{pmatrix}; \;
  B' = 
  \begin{pmatrix}
    \frac{17}{30}
  \end{pmatrix}, \;
  B'' = 
  \begin{pmatrix}
    \frac{13}{30}
  \end{pmatrix}; \;
    Q' =
  \begin{pmatrix}
    \frac{17}{30}
  \end{pmatrix}, \;
  Q'' =
  \begin{pmatrix}
    \frac{43}{30}
  \end{pmatrix}.
\end{gather*}


\subsection{Family \textnumero3.17}\label{subsection:03-17}

The pencil \(\mathcal{S}\) is defined by the equation
\begin{gather*}
  X^{2} Y Z + X Y^{2} Z + X Y Z^{2} + X Z^{2} T + Y^{2} T^{2} + X Z T^{2} + Y Z T^{2} + Y T^{3} = \lambda X Y Z T.
\end{gather*}
Members \(\mathcal{S}_{\lambda}\) of the pencil are irreducible for any \(\lambda \in \mathbb{P}^1\) except
\(\mathcal{S}_{\infty} = S_{(X)} + S_{(Y)} + S_{(Z)} + S_{(T)}\).
The base locus of the pencil \(\mathcal{S}\) consists of the following curves:
\begin{gather*}
  C_1 = C_{(X, Y)}, \;
  C_2 = C_{(X, T)}, \;
  C_3 = C_{(Y, Z)}, \;
  C_4 = C_{(Y, T)}, \;
  C_5 = C_{(Z, T)}, \\
  C_6 = C_{(Y, Z + T)}, \;
  C_7 = C_{(Z, Y + T)}, \;
  C_8 = C_{(X, Y + Z + T)}, \;
  C_9 = C_{(T, X + Y + Z)}.
\end{gather*}
Their linear equivalence classes on the generic member \(\mathcal{S}_{\Bbbk}\) of the pencil satisfy the following relations:
\begin{gather*}
  \begin{pmatrix}
    [C_{6}] \\ [C_{7}] \\ [C_{8}] \\ [C_{9}]
  \end{pmatrix} = 
  \begin{pmatrix}
    -1 & 0 & -1 & -1 & 0 & 1 \\
    0 & 0 & -1 & 0 & -2 & 1 \\
    -1 & -2 & 0 & 0 & 0 & 1 \\
    0 & -1 & 0 & -1 & -1 & 1
  \end{pmatrix} \cdot
  \begin{pmatrix}
    [C_{1}] & [C_{2}] & [C_{3}] & [C_{4}] & [C_{5}] & [H_{\mathcal{S}}]
  \end{pmatrix}^T.
\end{gather*}

For a general choice of \(\lambda \in \mathbb{C}\) the surface \(\mathcal{S}_{\lambda}\) has the following singularities:
\begin{itemize}\setlength{\itemindent}{2cm}
\item[\(P_{1} = P_{(X, Y, T)}\):] type \(\mathbb{A}_2\) with the quadratic term \(X \cdot (Y + T)\);
\item[\(P_{2} = P_{(X, Z, T)}\):] type \(\mathbb{A}_1\) with the quadratic term \(X Z + T^2\);
\item[\(P_{3} = P_{(Y, Z, T)}\):] type \(\mathbb{A}_4\) with the quadratic term \(Y \cdot Z\);
\item[\(P_{4} = P_{(X, Y, Z + T)}\):] type \(\mathbb{A}_1\) with the quadratic term \((X - Y) (Y + Z + T) - (\lambda + 2) X Y\);
\item[\(P_{5} = P_{(X, T, Y + Z)}\):] type \(\mathbb{A}_2\) with the quadratic term \(X \cdot (X + Y + Z - (\lambda + 1) T)\);
\item[\(P_{6} = P_{(Z, T, X + Y)}\):] type \(\mathbb{A}_1\) with the quadratic term \(Z (X + Y + Z + T) - T ((\lambda + 1) Z + T)\).
\end{itemize}

Galois action on the lattice \(L_{\lambda}\) is trivial. The intersection matrix on \(L_{\lambda} = L_{\mathcal{S}}\) is represented by
\begin{table}[H]
  \begin{tabular}{|c||cc|c|cccc|c|cc|c|cccccc|}
    \hline
    \(\bullet\) & \(E_1^1\) & \(E_1^2\) & \(E_2^1\) & \(E_3^1\) & \(E_3^2\) & \(E_3^3\) & \(E_3^4\) & \(E_4^1\) & \(E_5^1\) & \(E_5^2\) & \(E_6^1\) & \(\widetilde{C_{1}}\) & \(\widetilde{C_{2}}\) & \(\widetilde{C_{3}}\) & \(\widetilde{C_{4}}\) & \(\widetilde{C_{5}}\) & \(\widetilde{H_{\mathcal{S}}}\) \\
    \hline
    \hline
    \(\widetilde{C_{1}}\) & \(1\) & \(0\) & \(0\) & \(0\) & \(0\) & \(0\) & \(0\) & \(1\) & \(0\) & \(0\) & \(0\) & \(-2\) & \(0\) & \(1\) & \(0\) & \(0\) & \(1\) \\
    \(\widetilde{C_{2}}\) & \(1\) & \(0\) & \(1\) & \(0\) & \(0\) & \(0\) & \(0\) & \(0\) & \(1\) & \(0\) & \(0\) & \(0\) & \(-2\) & \(0\) & \(0\) & \(0\) & \(1\) \\
    \(\widetilde{C_{3}}\) & \(0\) & \(0\) & \(0\) & \(0\) & \(0\) & \(1\) & \(0\) & \(0\) & \(0\) & \(0\) & \(0\) & \(1\) & \(0\) & \(-2\) & \(0\) & \(0\) & \(1\) \\
    \(\widetilde{C_{4}}\) & \(0\) & \(1\) & \(0\) & \(1\) & \(0\) & \(0\) & \(0\) & \(0\) & \(0\) & \(0\) & \(0\) & \(0\) & \(0\) & \(0\) & \(-2\) & \(0\) & \(1\) \\
    \(\widetilde{C_{5}}\) & \(0\) & \(0\) & \(1\) & \(0\) & \(0\) & \(0\) & \(1\) & \(0\) & \(0\) & \(0\) & \(1\) & \(0\) & \(0\) & \(0\) & \(0\) & \(-2\) & \(1\) \\
    \(\widetilde{H_{\mathcal{S}}}\) & \(0\) & \(0\) & \(0\) & \(0\) & \(0\) & \(0\) & \(0\) & \(0\) & \(0\) & \(0\) & \(0\) & \(1\) & \(1\) & \(1\) & \(1\) & \(1\) & \(4\) \\
    \hline
  \end{tabular}.
\end{table}
Note that the intersection matrix is non-degenerate.

Discriminant groups and discriminant forms of the lattices \(L_{\mathcal{S}}\) and \(H \oplus \Pic(X)\) are given by
\begin{gather*}
  G' = 
  \begin{pmatrix}
    \frac{17}{28} & \frac{5}{14} & \frac{19}{28} & \frac{5}{7} & \frac{9}{28} & \frac{13}{14} & \frac{6}{7} & \frac{1}{7} & \frac{5}{7} & \frac{6}{7} & \frac{25}{28} & \frac{2}{7} & \frac{4}{7} & \frac{19}{28} & \frac{3}{28} & \frac{11}{14} & \frac{1}{7}
  \end{pmatrix}, \\
  G'' = 
  \begin{pmatrix}
    0 & 0 & -\frac{17}{28} & -\frac{3}{28} & \frac{1}{14}
  \end{pmatrix}; \;
  B' = 
  \begin{pmatrix}
    \frac{1}{28}
  \end{pmatrix}, \;
  B'' = 
  \begin{pmatrix}
    \frac{27}{28}
  \end{pmatrix}; \;
  Q' =
  \begin{pmatrix}
    \frac{1}{28}
  \end{pmatrix}, \;
  Q'' =
  \begin{pmatrix}
    \frac{55}{28}
  \end{pmatrix}.
\end{gather*}


\subsection{Family \textnumero3.18}\label{subsection:03-18}

The pencil \(\mathcal{S}\) is defined by the equation
\[
  X^{2} Y Z + X Y^{2} Z + X Y Z^{2} + X^{2} Y T + X Y^{2} T + X^{2} Z T + X Z T^{2} + Y Z T^{2} = \lambda X Y Z T.
\]
Members \(\mathcal{S}_{\lambda}\) of the pencil are irreducible for any \(\lambda \in \mathbb{P}^1\) except
\(\mathcal{S}_{\infty} = S_{(X)} + S_{(Y)} + S_{(Z)} + S_{(T)}\).
The base locus of the pencil \(\mathcal{S}\) consists of the following curves:
\begin{gather*}
  C_1 = C_{(X, Y)}, \;
  C_2 = C_{(X, Z)}, \;
  C_3 = C_{(X, T)}, \;
  C_4 = C_{(Y, Z)}, \;
  C_5 = C_{(Y, T)}, \\
  C_6 = C_{(Z, T)}, \;
  C_7 = C_{(Y, X + T)}, \;
  C_8 = C_{(Z, X + Y)}, \;
  C_9 = C_{(T, X + Y + Z)}.
\end{gather*}
Their linear equivalence classes on the generic member \(\mathcal{S}_{\Bbbk}\) of the pencil satisfy the following relations:
\begin{gather*}
  \begin{pmatrix}
    [C_{2}] \\ [C_{7}] \\ [C_{8}] \\ [C_{9}]
  \end{pmatrix} = 
  \begin{pmatrix}
    -1 & -2 & 0 & 0 & 0 & 1 \\
    -1 & 0 & -1 & -1 & 0 & 1 \\
    1 & 2 & -1 & 0 & -1 & 0 \\
    0 & -1 & 0 & -1 & -1 & 1
  \end{pmatrix} \cdot
  \begin{pmatrix}
    [C_{1}] & [C_{3}] & [C_{4}] & [C_{5}] & [C_{6}] & [H_{\mathcal{S}}]
  \end{pmatrix}^T.
\end{gather*}

For a general choice of \(\lambda \in \mathbb{C}\) the surface \(\mathcal{S}_{\lambda}\) has the following singularities:
\begin{itemize}\setlength{\itemindent}{2cm}
\item[\(P_{1} = P_{(X, Y, Z)}\):] type \(\mathbb{A}_3\) with the quadratic term \(Z \cdot (X + Y)\);
\item[\(P_{2} = P_{(X, Y, T)}\):] type \(\mathbb{A}_3\) with the quadratic term \(X \cdot Y\);
\item[\(P_{3} = P_{(X, Z, T)}\):] type \(\mathbb{A}_2\) with the quadratic term \(X \cdot (Z + T)\);
\item[\(P_{4} = P_{(Y, Z, T)}\):] type \(\mathbb{A}_1\) with the quadratic term \(Y Z + Y T + Z T\);
\item[\(P_{5} = P_{(X, T, Y + Z)}\):] type \(\mathbb{A}_1\) with the quadratic term \(X (X + Y + Z - (\lambda + 1) T) + T^2\);
\item[\(P_{6} = P_{(Z, T, X + Y)}\):] type \(\mathbb{A}_1\) with the quadratic term \((Z + T) (X + Y + Z) - (\lambda + 2) Z T\).
\end{itemize}

Galois action on the lattice \(L_{\lambda}\) is trivial. The intersection matrix on \(L_{\lambda} = L_{\mathcal{S}}\) is represented by
\begin{table}[H]
  \begin{tabular}{|c||ccc|ccc|cc|c|c|c|cccccc|}
    \hline
    \(\bullet\) & \(E_1^1\) & \(E_1^2\) & \(E_1^3\) & \(E_2^1\) & \(E_2^2\) & \(E_2^3\) & \(E_3^1\) & \(E_3^2\) & \(E_4^1\) & \(E_5^1\) & \(E_6^1\) & \(\widetilde{C_{1}}\) & \(\widetilde{C_{3}}\) & \(\widetilde{C_{4}}\) & \(\widetilde{C_{5}}\) & \(\widetilde{C_{6}}\) & \(\widetilde{H_{\mathcal{S}}}\) \\
    \hline
    \hline
    \(\widetilde{C_{1}}\) & \(1\) & \(0\) & \(0\) & \(0\) & \(1\) & \(0\) & \(0\) & \(0\) & \(0\) & \(0\) & \(0\) & \(-2\) & \(0\) & \(0\) & \(0\) & \(0\) & \(1\) \\
    \(\widetilde{C_{3}}\) & \(0\) & \(0\) & \(0\) & \(1\) & \(0\) & \(0\) & \(1\) & \(0\) & \(0\) & \(1\) & \(0\) & \(0\) & \(-2\) & \(0\) & \(0\) & \(0\) & \(1\) \\
    \(\widetilde{C_{4}}\) & \(0\) & \(0\) & \(1\) & \(0\) & \(0\) & \(0\) & \(0\) & \(0\) & \(1\) & \(0\) & \(0\) & \(0\) & \(0\) & \(-2\) & \(0\) & \(0\) & \(1\) \\
    \(\widetilde{C_{5}}\) & \(0\) & \(0\) & \(0\) & \(0\) & \(0\) & \(1\) & \(0\) & \(0\) & \(1\) & \(0\) & \(0\) & \(0\) & \(0\) & \(0\) & \(-2\) & \(0\) & \(1\) \\
    \(\widetilde{C_{6}}\) & \(0\) & \(0\) & \(0\) & \(0\) & \(0\) & \(0\) & \(0\) & \(1\) & \(1\) & \(0\) & \(1\) & \(0\) & \(0\) & \(0\) & \(0\) & \(-2\) & \(1\) \\
    \(\widetilde{H_{\mathcal{S}}}\) & \(0\) & \(0\) & \(0\) & \(0\) & \(0\) & \(0\) & \(0\) & \(0\) & \(0\) & \(0\) & \(0\) & \(1\) & \(1\) & \(1\) & \(1\) & \(1\) & \(4\) \\
    \hline
  \end{tabular}.
\end{table}
Note that the intersection matrix is non-degenerate.

Discriminant groups and discriminant forms of the lattices \(L_{\mathcal{S}}\) and \(H \oplus \Pic(X)\) are given by
\begin{gather*}
  G' = 
  \begin{pmatrix}
    \frac{15}{26} & \frac{1}{13} & \frac{15}{26} & \frac{3}{26} & \frac{7}{13} & \frac{23}{26} & \frac{5}{13} & \frac{1}{13} & \frac{7}{13} & \frac{11}{13} & \frac{23}{26} & \frac{1}{13} & \frac{9}{13} & \frac{1}{13} & \frac{3}{13} & \frac{10}{13} & \frac{1}{26}
  \end{pmatrix}, \\
  G'' = 
  \begin{pmatrix}
    0 & 0 & -\frac{6}{13} & -\frac{11}{26} & \frac{1}{13}
  \end{pmatrix}; \;
  B' = 
  \begin{pmatrix}
    \frac{25}{26}
  \end{pmatrix}, \;
  B'' = 
  \begin{pmatrix}
    \frac{1}{26}
  \end{pmatrix}; \;
  Q' =
  \begin{pmatrix}
    \frac{25}{26}
  \end{pmatrix}, \;
  Q'' =
  \begin{pmatrix}
    \frac{27}{26}
  \end{pmatrix}.
\end{gather*}


\subsection{Family \textnumero3.19}\label{subsection:03-19}

The pencil \(\mathcal{S}\) is defined by the equation
\begin{gather*}
  X^{2} Y Z + X Y^{2} Z + X Y Z^{2} + Y Z^{2} T + X^{2} T^{2} + Y Z T^{2} + X T^{3} = \lambda X Y Z T.
\end{gather*}
Members \(\mathcal{S}_{\lambda}\) of the pencil are irreducible for any \(\lambda \in \mathbb{P}^1\) except
\(\mathcal{S}_{\infty} = S_{(X)} + S_{(Y)} + S_{(Z)} + S_{(T)}\).
The base locus of the pencil \(\mathcal{S}\) consists of the following curves:
\begin{gather*}
  C_{1} = C_{(X, Y)}, \;
  C_{2} = C_{(X, Z)}, \;
  C_{3} = C_{(X, T)}, \;
  C_{4} = C_{(Y, T)}, \;
  C_{5} = C_{(Z, T)}, \\
  C_{6} = C_{(X, Z + T)}, \;
  C_{7} = C_{(Y, X + T)}, \;
  C_{8} = C_{(Z, X + T)}, \;
  C_{9} = C_{(T, X + Y + Z)}.
\end{gather*}
Their linear equivalence classes on the generic member \(\mathcal{S}_{\Bbbk}\) of the pencil satisfy the following relations:
\begin{gather*}
  \begin{pmatrix}
    [C_{6}] \\ [C_{7}] \\ [C_{8}] \\ [C_{9}]
  \end{pmatrix} = 
  \begin{pmatrix}
    -1 & -1 & -1 & 0 & 0 & 1 \\
    -1 & 0 & 0 & -2 & 0 & 1 \\
    0 & -1 & 0 & 0 & -2 & 1 \\
    0 & 0 & -1 & -1 & -1 & 1
  \end{pmatrix} \cdot
  \begin{pmatrix}
    [C_{1}] & [C_{2}] & [C_{3}] & [C_{4}] & [C_{5}] & [H_{\mathcal{S}}]
  \end{pmatrix}^T.
\end{gather*}

For a general choice of \(\lambda \in \mathbb{C}\) the surface \(\mathcal{S}_{\lambda}\) has the following singularities:
\begin{itemize}\setlength{\itemindent}{2cm}
\item[\(P_{1} = P_{(X, Y, T)}\):] type \(\mathbb{A}_4\) with the quadratic term \(Y \cdot (X + T)\);
\item[\(P_{2} = P_{(X, Z, T)}\):] type \(\mathbb{A}_4\) with the quadratic term \(X \cdot Z\);
\item[\(P_{3} = P_{(Y, Z, T)}\):] type \(\mathbb{A}_1\) with the quadratic term \(Y Z + T^2\);
\item[\(P_{4} = P_{(Y, T, X + Z)}\):] type \(\mathbb{A}_1\) with the quadratic term \(Y (X + Y + Z - (\lambda + 1) T) - T^2\);
\item[\(P_{5} = P_{(Z, T, X + Y)}\):] type \(\mathbb{A}_1\) with the quadratic term \(Z (X + Y + Z - \lambda T) - T^2\).
\end{itemize}

Galois action on the lattice \(L_{\lambda}\) is trivial. The intersection matrix on \(L_{\lambda} = L_{\mathcal{S}}\) is represented by
\begin{table}[H]
  \begin{tabular}{|c||cccc|cccc|c|c|c|cccccc|}
    \hline
    \(\bullet\) & \(E_1^1\) & \(E_1^2\) & \(E_1^3\) & \(E_1^4\) & \(E_2^1\) & \(E_2^2\) & \(E_2^3\) & \(E_2^4\) & \(E_3^1\) & \(E_4^1\) & \(E_5^1\) & \(\widetilde{C_{1}}\) & \(\widetilde{C_{2}}\) & \(\widetilde{C_{3}}\) & \(\widetilde{C_{4}}\) & \(\widetilde{C_{5}}\) & \(\widetilde{H_{\mathcal{S}}}\) \\
    \hline
    \hline
    \(\widetilde{C_{1}}\) & \(1\) & \(0\) & \(0\) & \(0\) & \(0\) & \(0\) & \(0\) & \(0\) & \(0\) & \(0\) & \(0\) & \(-2\) & \(1\) & \(0\) & \(0\) & \(0\) & \(1\) \\
    \(\widetilde{C_{2}}\) & \(0\) & \(0\) & \(0\) & \(0\) & \(0\) & \(0\) & \(1\) & \(0\) & \(0\) & \(0\) & \(0\) & \(1\) & \(-2\) & \(0\) & \(0\) & \(0\) & \(1\) \\
    \(\widetilde{C_{3}}\) & \(0\) & \(0\) & \(0\) & \(1\) & \(1\) & \(0\) & \(0\) & \(0\) & \(0\) & \(0\) & \(0\) & \(0\) & \(0\) & \(-2\) & \(0\) & \(0\) & \(1\) \\
    \(\widetilde{C_{4}}\) & \(1\) & \(0\) & \(0\) & \(0\) & \(0\) & \(0\) & \(0\) & \(0\) & \(1\) & \(1\) & \(0\) & \(0\) & \(0\) & \(0\) & \(-2\) & \(0\) & \(1\) \\
    \(\widetilde{C_{5}}\) & \(0\) & \(0\) & \(0\) & \(0\) & \(0\) & \(0\) & \(0\) & \(1\) & \(1\) & \(0\) & \(1\) & \(0\) & \(0\) & \(0\) & \(0\) & \(-2\) & \(1\) \\
    \(\widetilde{H_{\mathcal{S}}}\) & \(0\) & \(0\) & \(0\) & \(0\) & \(0\) & \(0\) & \(0\) & \(0\) & \(0\) & \(0\) & \(0\) & \(1\) & \(1\) & \(1\) & \(1\) & \(1\) & \(4\) \\
    \hline
  \end{tabular}.
\end{table}
Note that the intersection matrix is non-degenerate.

Discriminant groups and discriminant forms of the lattices \(L_{\mathcal{S}}\) and \(H \oplus \Pic(X)\) are given by
\begin{gather*}
  G' = 
  \begin{pmatrix}
    0 & 0 & 0 & 0 & 0 & 0 & 0 & 0 & \frac{1}{2} & \frac{1}{2} & \frac{1}{2} & 0 & 0 & 0 & 0 & 0 & 0 \\
    \frac{1}{2} & 0 & \frac{1}{2} & 0 & 0 & \frac{1}{2} & 0 & 0 & \frac{1}{2} & 0 & \frac{1}{2} & 0 & \frac{1}{2} & \frac{1}{2} & 0 & 0 & 0 \\
    \frac{5}{6} & \frac{5}{6} & \frac{5}{6} & \frac{5}{6} & \frac{2}{3} & \frac{1}{2} & \frac{1}{3} & \frac{5}{6} & 0 & \frac{1}{3} & \frac{2}{3} & \frac{1}{6} & \frac{1}{3} & \frac{5}{6} & \frac{2}{3} & \frac{1}{3} & \frac{1}{6}
  \end{pmatrix}, \\
  G'' = 
  \begin{pmatrix}
    0 & 0 & \frac{1}{2} & 0 & 0 \\
    0 & 0 & 0 & \frac{1}{2} & 0 \\
    0 & 0 & 0 & 0 & \frac{1}{6}
  \end{pmatrix}; \;
  B' = 
  \begin{pmatrix}
    \frac{1}{2} & 0 & 0 \\
    0 & \frac{1}{2} & 0 \\
    0 & 0 & \frac{5}{6}
  \end{pmatrix}, \;
  B'' = 
  \begin{pmatrix}
    \frac{1}{2} & 0 & 0 \\
    0 & \frac{1}{2} & 0 \\
    0 & 0 & \frac{1}{6}
  \end{pmatrix}; \;
  \begin{pmatrix}
    Q' \\ Q''
  \end{pmatrix}
  =
  \begin{pmatrix}
    \frac{1}{2} & \frac{1}{2} & \frac{11}{6} \\
    \frac{3}{2} & \frac{3}{2} & \frac{1}{6}
  \end{pmatrix}.
\end{gather*}


\subsection{Family \textnumero3.20}\label{subsection:03-20}

The pencil \(\mathcal{S}\) is defined by the equation
\begin{gather*}
  X^{2} Y Z + X Y^{2} Z + X Y Z^{2} + Y^{2} Z T + X^{2} T^{2} + X Z T^{2} + Y Z T^{2} = \lambda X Y Z T.
\end{gather*}
Members \(\mathcal{S}_{\lambda}\) of the pencil are irreducible for any \(\lambda \in \mathbb{P}^1\) except
\(\mathcal{S}_{\infty} = S_{(X)} + S_{(Y)} + S_{(Z)} + S_{(T)}\).
The base locus of the pencil \(\mathcal{S}\) consists of the following curves:
\begin{gather*}
  C_{1} = C_{(X, Y)}, \;
  C_{2} = C_{(X, Z)}, \;
  C_{3} = C_{(X, T)}, \;
  C_{4} = C_{(Y, T)}, \\
  C_{5} = C_{(Z, T)}, \;
  C_{6} = C_{(X, Y + T)}, \;
  C_{7} = C_{(Y, X + Z)}, \;
  C_{8} = C_{(T, X + Y + Z)}.
\end{gather*}
Their linear equivalence classes on the generic member \(\mathcal{S}_{\Bbbk}\) of the pencil satisfy the following relations:
\begin{gather*}
  \begin{pmatrix}
    [C_{6}] \\ [C_{7}] \\ [C_{8}] \\ [H_{\mathcal{S}}]
  \end{pmatrix} = 
  \begin{pmatrix}
    -1 & 1 & -1 & 0 & 2 \\
    -1 & 2 & 0 & -2 & 2 \\
    0 & 2 & -1 & -1 & 1 \\
    0 & 2 & 0 & 0 & 2
  \end{pmatrix} \cdot
  \begin{pmatrix}
    [C_{1}] & [C_{2}] & [C_{3}] & [C_{4}] & [C_{5}]
  \end{pmatrix}^T.
\end{gather*}

For a general choice of \(\lambda \in \mathbb{C}\) the surface \(\mathcal{S}_{\lambda}\) has the following singularities:
\begin{itemize}\setlength{\itemindent}{2cm}
\item[\(P_{1} = P_{(X, Y, Z)}\):] type \(\mathbb{A}_1\) with the quadratic term \(X^2 + Z (X + Y)\);
\item[\(P_{2} = P_{(X, Y, T)}\):] type \(\mathbb{A}_3\) with the quadratic term \(X \cdot Y\);
\item[\(P_{3} = P_{(X, Z, T)}\):] type \(\mathbb{A}_3\) with the quadratic term \(Z \cdot (X + T)\);
\item[\(P_{4} = P_{(Y, Z, T)}\):] type \(\mathbb{A}_1\) with the quadratic term \(Y Z + T^2\);
\item[\(P_{5} = P_{(X, Z, Y + T)}\):] type \(\mathbb{A}_1\) with the quadratic term \(Z ((\lambda + 2) X - Y - T) + X^2\);
\item[\(P_{6} = P_{(Y, T, X + Z)}\):] type \(\mathbb{A}_2\) with the quadratic term \(Y \cdot (X + Y + Z - \lambda T)\);
\item[\(P_{7} = P_{(Z, T, X + Y)}\):] type \(\mathbb{A}_1\) with the quadratic term \(Z (X + Y + Z + T) - T ((\lambda + 2) Z + T)\).
\end{itemize}

Galois action on the lattice \(L_{\lambda}\) is trivial. The intersection matrix on \(L_{\lambda} = L_{\mathcal{S}}\) is represented by
\begin{table}[H]
  \begin{tabular}{|c||c|ccc|ccc|c|c|cc|c|ccccc|}
    \hline
    \(\bullet\) & \(E_1^1\) & \(E_2^1\) & \(E_2^2\) & \(E_2^3\) & \(E_3^1\) & \(E_3^2\) & \(E_3^3\) & \(E_4^1\) & \(E_5^1\) & \(E_6^1\) & \(E_6^2\) & \(E_7^1\) & \(\widetilde{C_{1}}\) & \(\widetilde{C_{2}}\) & \(\widetilde{C_{3}}\) & \(\widetilde{C_{4}}\) & \(\widetilde{C_{5}}\) \\
    \hline
    \hline
    \(\widetilde{C_{1}}\) & \(1\) & \(0\) & \(1\) & \(0\) & \(0\) & \(0\) & \(0\) & \(0\) & \(0\) & \(0\) & \(0\) & \(0\) & \(-2\) & \(0\) & \(0\) & \(0\) & \(0\) \\
    \(\widetilde{C_{2}}\) & \(1\) & \(0\) & \(0\) & \(0\) & \(1\) & \(0\) & \(0\) & \(0\) & \(1\) & \(0\) & \(0\) & \(0\) & \(0\) & \(-2\) & \(0\) & \(0\) & \(0\) \\
    \(\widetilde{C_{3}}\) & \(0\) & \(1\) & \(0\) & \(0\) & \(0\) & \(0\) & \(1\) & \(0\) & \(0\) & \(0\) & \(0\) & \(0\) & \(0\) & \(0\) & \(-2\) & \(0\) & \(0\) \\
    \(\widetilde{C_{4}}\) & \(0\) & \(0\) & \(0\) & \(1\) & \(0\) & \(0\) & \(0\) & \(1\) & \(0\) & \(1\) & \(0\) & \(0\) & \(0\) & \(0\) & \(0\) & \(-2\) & \(0\) \\
    \(\widetilde{C_{5}}\) & \(0\) & \(0\) & \(0\) & \(0\) & \(1\) & \(0\) & \(0\) & \(1\) & \(0\) & \(0\) & \(0\) & \(1\) & \(0\) & \(0\) & \(0\) & \(0\) & \(-2\) \\
    \hline
  \end{tabular}.
\end{table}
Note that the intersection matrix is non-degenerate.

Discriminant groups and discriminant forms of the lattices \(L_{\mathcal{S}}\) and \(H \oplus \Pic(X)\) are given by
\begin{gather*}
  G' = 
  \begin{pmatrix}
    \frac{4}{7} & \frac{13}{28} & \frac{4}{7} & \frac{17}{28} & \frac{1}{28} & \frac{1}{7} & \frac{1}{4} & \frac{1}{4} & \frac{15}{28} & \frac{3}{7} & \frac{3}{14} & \frac{3}{7} & \frac{1}{14} & \frac{1}{14} & \frac{5}{14} & \frac{9}{14} & \frac{6}{7}
  \end{pmatrix}, \\
  G'' = 
  \begin{pmatrix}
    0 & 0 & \frac{1}{28} & -\frac{13}{28} & \frac{1}{14}
  \end{pmatrix}; \;
  B' = 
  \begin{pmatrix}
    \frac{13}{28}
  \end{pmatrix}, \;
  B'' = 
  \begin{pmatrix}
    \frac{15}{28}
  \end{pmatrix}; \;
  Q' =
  \begin{pmatrix}
    \frac{13}{28}
  \end{pmatrix}, \;
  Q'' =
  \begin{pmatrix}
    \frac{43}{28}
  \end{pmatrix}.
\end{gather*}


\subsection{Family \textnumero3.21}\label{subsection:03-21}

The pencil \(\mathcal{S}\) is defined by the equation
\begin{gather*}
  X^{2} Y Z + X Y^{2} Z + X Y Z^{2} + X Z^{2} T + X Y T^{2} + X Z T^{2} + Y Z T^{2} + Z^{2} T^{2} = \lambda X Y Z T.
\end{gather*}
Members \(\mathcal{S}_{\lambda}\) of the pencil are irreducible for any \(\lambda \in \mathbb{P}^1\) except
\(\mathcal{S}_{\infty} = S_{(X)} + S_{(Y)} + S_{(Z)} + S_{(T)}\).
The base locus of the pencil \(\mathcal{S}\) consists of the following curves:
\begin{gather*}
  C_{1} = C_{(X, Z)}, \;
  C_{2} = C_{(X, T)}, \;
  C_{3} = C_{(Y, Z)}, \;
  C_{4} = C_{(Y, T)}, \\
  C_{5} = C_{(Z, T)}, \;
  C_{6} = C_{(X, Y + Z)}, \;
  C_{7} = C_{(T, X + Y + Z)}, \;
  C_{8} = C_{(Y, X Z + T (X + Z))}.
\end{gather*}
Their linear equivalence classes on the generic member \(\mathcal{S}_{\Bbbk}\) of the pencil satisfy the following relations:
\begin{gather*}
  \begin{pmatrix}
    [C_{3}] \\ [C_{6}] \\ [C_{7}] \\ [C_{8}]
  \end{pmatrix} = 
  \begin{pmatrix}
    -1 & 0 & 0 & -2 & 1 \\
    -1 & -2 & 0 & 0 & 1 \\
    0 & -1 & -1 & -1 & 1 \\
    1 & 0 & -1 & 2 & 0
  \end{pmatrix} \cdot
  \begin{pmatrix}
    [C_{1}] & [C_{2}] & [C_{4}] & [C_{5}] & [H_{\mathcal{S}}]
  \end{pmatrix}^T.
\end{gather*}

For a general choice of \(\lambda \in \mathbb{C}\) the surface \(\mathcal{S}_{\lambda}\) has the following singularities:
\begin{itemize}\setlength{\itemindent}{2cm}
\item[\(P_{1} = P_{(X, Y, Z)}\):] type \(\mathbb{A}_2\) with the quadratic term \((X + Z) \cdot (Y + Z)\);
\item[\(P_{2} = P_{(X, Y, T)}\):] type \(\mathbb{A}_1\) with the quadratic term \(X (Y + T) + T^2\);
\item[\(P_{3} = P_{(X, Z, T)}\):] type \(\mathbb{A}_3\) with the quadratic term \(X \cdot Z\);
\item[\(P_{4} = P_{(Y, Z, T)}\):] type \(\mathbb{A}_3\) with the quadratic term \(Y \cdot Z\);
\item[\(P_{5} = P_{(X, T, Y + Z)}\):] type \(\mathbb{A}_2\) with the quadratic term \(X \cdot (X + Y + Z - (\lambda + 1) T)\);
\item[\(P_{6} = P_{(Z, T, X + Y)}\):] type \(\mathbb{A}_1\) with the quadratic term \(Z (X + Y + Z - \lambda T) + T^2\).
\end{itemize}

Galois action on the lattice \(L_{\lambda}\) is trivial. The intersection matrix on \(L_{\lambda} = L_{\mathcal{S}}\) is represented by
\begin{table}[H]
  \begin{tabular}{|c||cc|c|ccc|ccc|cc|c|ccccc|}
    \hline
    \(\bullet\) & \(E_1^1\) & \(E_1^2\) & \(E_2^1\) & \(E_3^1\) & \(E_3^2\) & \(E_3^3\) & \(E_4^1\) & \(E_4^2\) & \(E_4^3\) & \(E_5^1\) & \(E_5^2\) & \(E_6^1\) & \(\widetilde{C_{1}}\) & \(\widetilde{C_{2}}\) & \(\widetilde{C_{4}}\) & \(\widetilde{C_{5}}\) & \(\widetilde{H_{\mathcal{S}}}\) \\
    \hline
    \hline
    \(\widetilde{C_{1}}\) & \(1\) & \(0\) & \(0\) & \(0\) & \(1\) & \(0\) & \(0\) & \(0\) & \(0\) & \(0\) & \(0\) & \(0\) & \(-2\) & \(0\) & \(0\) & \(0\) & \(1\) \\
    \(\widetilde{C_{2}}\) & \(0\) & \(0\) & \(1\) & \(1\) & \(0\) & \(0\) & \(0\) & \(0\) & \(0\) & \(1\) & \(0\) & \(0\) & \(0\) & \(-2\) & \(0\) & \(0\) & \(1\) \\
    \(\widetilde{C_{4}}\) & \(0\) & \(0\) & \(1\) & \(0\) & \(0\) & \(0\) & \(1\) & \(0\) & \(0\) & \(0\) & \(0\) & \(0\) & \(0\) & \(0\) & \(-2\) & \(0\) & \(1\) \\
    \(\widetilde{C_{5}}\) & \(0\) & \(0\) & \(0\) & \(0\) & \(0\) & \(1\) & \(0\) & \(0\) & \(1\) & \(0\) & \(0\) & \(1\) & \(0\) & \(0\) & \(0\) & \(-2\) & \(1\) \\
    \(\widetilde{H_{\mathcal{S}}}\) & \(0\) & \(0\) & \(0\) & \(0\) & \(0\) & \(0\) & \(0\) & \(0\) & \(0\) & \(0\) & \(0\) & \(0\) & \(1\) & \(1\) & \(1\) & \(1\) & \(4\) \\
    \hline
  \end{tabular}.
\end{table}
Note that the intersection matrix is non-degenerate.

Discriminant groups and discriminant forms of the lattices \(L_{\mathcal{S}}\) and \(H \oplus \Pic(X)\) are given by
\begin{gather*}
  G' = 
  \begin{pmatrix}
    \frac{5}{11} & \frac{8}{11} & \frac{19}{22} & \frac{1}{2} & \frac{2}{11} & \frac{15}{22} & \frac{5}{22} & \frac{6}{11} & \frac{19}{22} & \frac{6}{11} & \frac{3}{11} & \frac{1}{11} & \frac{2}{11} & \frac{9}{11} & \frac{10}{11} & \frac{2}{11} & \frac{8}{11}
  \end{pmatrix}, \\
  G'' = 
  \begin{pmatrix}
    0 & 0 & \frac{81}{22} & \frac{9}{22} & -\frac{5}{11}
  \end{pmatrix}; \;
  B' = 
  \begin{pmatrix}
    \frac{3}{22}
  \end{pmatrix}, \;
  B'' = 
  \begin{pmatrix}
    \frac{19}{22}
  \end{pmatrix}; \;
  Q' =
  \begin{pmatrix}
    \frac{25}{22}
  \end{pmatrix}, \;
  Q'' =
  \begin{pmatrix}
    \frac{19}{22}
  \end{pmatrix}.
\end{gather*}


\subsection{Family \textnumero3.22}\label{subsection:03-22}

The pencil \(\mathcal{S}\) is defined by the equation
\begin{gather*}
  X^{2} Y Z + X Y^{2} Z + X Y Z^{2} + Y Z^{2} T + Y Z T^{2} + X T^{3} + T^{4} = \lambda X Y Z T.
\end{gather*}
Members \(\mathcal{S}_{\lambda}\) of the pencil are irreducible for any \(\lambda \in \mathbb{P}^1\) except
\(\mathcal{S}_{\infty} = S_{(X)} + S_{(Y)} + S_{(Z)} + S_{(T)}\).
The base locus of the pencil \(\mathcal{S}\) consists of the following curves:
\begin{gather*}
  C_{1} = C_{(X, T)}, \;
  C_{2} = C_{(Y, T)}, \;
  C_{3} = C_{(Z, T)}, \;
  C_{4} = C_{(Y, X + T)}, \\
  C_{5} = C_{(Z, X + T)}, \;
  C_{6} = C_{(T, X + Y + Z)}, \;
  C_{7} = C_{(X, T^3 + Y Z (Z + T))}.
\end{gather*}
Their linear equivalence classes on the generic member \(\mathcal{S}_{\Bbbk}\) of the pencil satisfy the following relations:
\begin{gather*}
  \begin{pmatrix}
    [C_{4}] \\ [C_{5}] \\ [C_{6}] \\ [C_{7}]
  \end{pmatrix} = 
  \begin{pmatrix}
    0 & -3 & 0 & 1 \\
    0 & 0 & -3 & 1 \\
    -1 & -1 & -1 & 1 \\
    -1 & 0 & 0 & 1
  \end{pmatrix} \cdot
  \begin{pmatrix}
    [C_{1}] \\ [C_{2}] \\ [C_{3}] \\ [H_{\mathcal{S}}]
  \end{pmatrix}.
\end{gather*}

For a general choice of \(\lambda \in \mathbb{C}\) the surface \(\mathcal{S}_{\lambda}\) has the following singularities:
\begin{itemize}\setlength{\itemindent}{2cm}
\item[\(P_{1} = P_{(X, Y, T)}\):] type \(\mathbb{A}_4\) with the quadratic term \(Y \cdot (X + T)\);
\item[\(P_{2} = P_{(X, Z, T)}\):] type \(\mathbb{A}_3\) with the quadratic term \(X \cdot Z\);
\item[\(P_{3} = P_{(Y, Z, T)}\):] type \(\mathbb{A}_2\) with the quadratic term \(Y \cdot Z\);
\item[\(P_{4} = P_{(Y, T, X + Z)}\):] type \(\mathbb{A}_2\) with the quadratic term \(Y \cdot (X + Y + Z - (\lambda + 1) T)\);
\item[\(P_{5} = P_{(Z, T, X + Y)}\):] type \(\mathbb{A}_2\) with the quadratic term \(Z \cdot (X + Y + Z - \lambda T)\).
\end{itemize}

Galois action on the lattice \(L_{\lambda}\) is trivial. The intersection matrix on \(L_{\lambda} = L_{\mathcal{S}}\) is represented by
\begin{table}[H]
  \begin{tabular}{|c||cccc|ccc|cc|cc|cc|cccc|}
    \hline
    \(\bullet\) & \(E_1^1\) & \(E_1^2\) & \(E_1^3\) & \(E_1^4\) & \(E_2^1\) & \(E_2^2\) & \(E_2^3\) & \(E_3^1\) & \(E_3^2\) & \(E_4^1\) & \(E_4^2\) & \(E_5^1\) & \(E_5^2\) & \(\widetilde{C_{1}}\) & \(\widetilde{C_{2}}\) & \(\widetilde{C_{3}}\) & \(\widetilde{H_{\mathcal{S}}}\) \\
    \hline
    \hline
    \(\widetilde{C_{1}}\) & \(1\) & \(0\) & \(0\) & \(0\) & \(1\) & \(0\) & \(0\) & \(0\) & \(0\) & \(0\) & \(0\) & \(0\) & \(0\) & \(-2\) & \(0\) & \(0\) & \(1\) \\
    \(\widetilde{C_{2}}\) & \(0\) & \(0\) & \(0\) & \(1\) & \(0\) & \(0\) & \(0\) & \(1\) & \(0\) & \(1\) & \(0\) & \(0\) & \(0\) & \(0\) & \(-2\) & \(0\) & \(1\) \\
    \(\widetilde{C_{3}}\) & \(0\) & \(0\) & \(0\) & \(0\) & \(0\) & \(0\) & \(1\) & \(0\) & \(1\) & \(0\) & \(0\) & \(1\) & \(0\) & \(0\) & \(0\) & \(-2\) & \(1\) \\
    \(\widetilde{H_{\mathcal{S}}}\) & \(0\) & \(0\) & \(0\) & \(0\) & \(0\) & \(0\) & \(0\) & \(0\) & \(0\) & \(0\) & \(0\) & \(0\) & \(0\) & \(1\) & \(1\) & \(1\) & \(4\) \\
    \hline
  \end{tabular}.
\end{table}
Note that the intersection matrix is non-degenerate.

Discriminant groups and discriminant forms of the lattices \(L_{\mathcal{S}}\) and \(H \oplus \Pic(X)\) are given by
\begin{gather*}
  G' = 
  \begin{pmatrix}
    \frac{1}{3} & \frac{5}{6} & \frac{1}{3} & \frac{5}{6} & \frac{2}{3} & \frac{1}{2} & \frac{1}{3} & \frac{11}{18} & \frac{8}{9} & \frac{5}{9} & \frac{7}{9} & \frac{4}{9} & \frac{13}{18} & \frac{5}{6} & \frac{1}{3} & \frac{1}{6} & \frac{2}{3}
  \end{pmatrix}, \\
  G'' = 
  \begin{pmatrix}
    0 & 0 & -\frac{1}{6} & -\frac{1}{9} & \frac{1}{3}
  \end{pmatrix}; \;
  B' = 
  \begin{pmatrix}
    \frac{5}{18}
  \end{pmatrix}, \;
  B'' = 
  \begin{pmatrix}
    \frac{13}{18}
  \end{pmatrix}; \;
  Q' =
  \begin{pmatrix}
    \frac{5}{18}
  \end{pmatrix}, \;
  Q'' =
  \begin{pmatrix}
    \frac{31}{18}
  \end{pmatrix}.
\end{gather*}


\subsection{Family \textnumero3.23}\label{subsection:03-23}

The pencil \(\mathcal{S}\) is defined by the equation
\begin{gather*}
  X^{2} Y Z + X Y^{2} Z + X Y Z^{2} + Y Z^{2} T + Y Z T^{2} + X T^{3} + Z T^{3} = \lambda X Y Z T.
\end{gather*}
Members \(\mathcal{S}_{\lambda}\) of the pencil are irreducible for any \(\lambda \in \mathbb{P}^1\) except
\(\mathcal{S}_{\infty} = S_{(X)} + S_{(Y)} + S_{(Z)} + S_{(T)}\).
The base locus of the pencil \(\mathcal{S}\) consists of the following curves:
\begin{gather*}
  C_{1} = C_{(X, Z)}, \;
  C_{2} = C_{(X, T)}, \;
  C_{3} = C_{(Y, T)}, \;
  C_{4} = C_{(Z, T)}, \\
  C_{5} = C_{(Y, X + Z)}, \;
  C_{6} = C_{(T, X + Y + Z)}, \;
  C_{7} = C_{(X, T^2 + Y (Z + T))}.
\end{gather*}
Their linear equivalence classes on the generic member \(\mathcal{S}_{\Bbbk}\) of the pencil satisfy the following relations:
\begin{gather*}
  \begin{pmatrix}
    [C_{1}] \\ [C_{5}] \\ [C_{6}] \\ [C_{7}]
  \end{pmatrix} = 
  \begin{pmatrix}
    0 & 0 & -3 & 1 \\
    0 & -3 & 0 & 1 \\
    -1 & -1 & -1 & 1 \\
    -1 & 0 & 3 & 0
  \end{pmatrix} \cdot
  \begin{pmatrix}
    [C_{2}] \\ [C_{3}] \\ [C_{4}] \\ [H_{\mathcal{S}}]
  \end{pmatrix}.
\end{gather*}

For a general choice of \(\lambda \in \mathbb{C}\) the surface \(\mathcal{S}_{\lambda}\) has the following singularities:
\begin{itemize}\setlength{\itemindent}{2cm}
\item[\(P_{1} = P_{(X, Y, T)}\):] type \(\mathbb{A}_2\) with the quadratic term \(Y \cdot (X + T)\);
\item[\(P_{2} = P_{(X, Z, T)}\):] type \(\mathbb{A}_4\) with the quadratic term \(X \cdot Z\);
\item[\(P_{3} = P_{(Y, Z, T)}\):] type \(\mathbb{A}_2\) with the quadratic term \(Y \cdot Z\);
\item[\(P_{4} = P_{(Y, T, X + Z)}\):] type \(\mathbb{A}_3\) with the quadratic term \(Y \cdot (X + Y + Z - (\lambda + 1) T)\);
\item[\(P_{5} = P_{(Z, T, X + Y)}\):] type \(\mathbb{A}_2\) with the quadratic term \(Z \cdot (X + Y + Z - \lambda T)\).
\end{itemize}

Galois action on the lattice \(L_{\lambda}\) is trivial. The intersection matrix on \(L_{\lambda} = L_{\mathcal{S}}\) is represented by
\begin{table}[H]
  \begin{tabular}{|c||cc|cccc|cc|ccc|cc|cccc|}
    \hline
    \(\bullet\) & \(E_1^1\) & \(E_1^2\) & \(E_2^1\) & \(E_2^2\) & \(E_2^3\) & \(E_2^4\) & \(E_3^1\) & \(E_3^2\) & \(E_4^1\) & \(E_4^2\) & \(E_4^3\) & \(E_5^1\) & \(E_5^2\) & \(\widetilde{C_{2}}\) & \(\widetilde{C_{3}}\) & \(\widetilde{C_{4}}\) & \(\widetilde{H_{\mathcal{S}}}\) \\
    \hline
    \hline
    \(\widetilde{C_{2}}\) & \(1\) & \(0\) & \(1\) & \(0\) & \(0\) & \(0\) & \(0\) & \(0\) & \(0\) & \(0\) & \(0\) & \(0\) & \(0\) & \(-2\) & \(0\) & \(0\) & \(1\) \\
    \(\widetilde{C_{3}}\) & \(0\) & \(1\) & \(0\) & \(0\) & \(0\) & \(0\) & \(1\) & \(0\) & \(1\) & \(0\) & \(0\) & \(0\) & \(0\) & \(0\) & \(-2\) & \(0\) & \(1\) \\
    \(\widetilde{C_{4}}\) & \(0\) & \(0\) & \(0\) & \(0\) & \(0\) & \(1\) & \(0\) & \(1\) & \(0\) & \(0\) & \(0\) & \(1\) & \(0\) & \(0\) & \(0\) & \(-2\) & \(1\) \\
    \(\widetilde{H_{\mathcal{S}}}\) & \(0\) & \(0\) & \(0\) & \(0\) & \(0\) & \(0\) & \(0\) & \(0\) & \(0\) & \(0\) & \(0\) & \(0\) & \(0\) & \(1\) & \(1\) & \(1\) & \(4\) \\
    \hline
  \end{tabular}.
\end{table}
Note that the intersection matrix is non-degenerate.

Discriminant groups and discriminant forms of the lattices \(L_{\mathcal{S}}\) and \(H \oplus \Pic(X)\) are given by
\begin{gather*}
  G' = 
  \begin{pmatrix}
    \frac{1}{10} & \frac{7}{20} & \frac{11}{20} & \frac{1}{4} & \frac{19}{20} & \frac{13}{20} & \frac{17}{20} & \frac{1}{10} & \frac{19}{20} & \frac{3}{10} & \frac{13}{20} & \frac{9}{10} & \frac{9}{20} & \frac{17}{20} & \frac{3}{5} & \frac{7}{20} & \frac{1}{20}
  \end{pmatrix}, \\
  G'' = 
  \begin{pmatrix}
    0 & 0 & -\frac{3}{5} & -\frac{3}{10} & \frac{1}{20}
  \end{pmatrix}; \;
  B' = 
  \begin{pmatrix}
    \frac{13}{20}
  \end{pmatrix}, \;
  B'' = 
  \begin{pmatrix}
    \frac{7}{20}
  \end{pmatrix}; \;
  Q' =
  \begin{pmatrix}
    \frac{13}{20}
  \end{pmatrix}, \;
  Q'' =
  \begin{pmatrix}
    \frac{27}{20}
  \end{pmatrix}.
\end{gather*}


\subsection{Family \textnumero3.24}\label{subsection:03-24}

The pencil \(\mathcal{S}\) is defined by the equation
\begin{gather*}
  X^{2} Y Z + X Y^{2} Z + X Y Z^{2} + Y^{2} Z T + X Z T^{2} + Y Z T^{2} + T^{4} = \lambda X Y Z T.
\end{gather*}
Members \(\mathcal{S}_{\lambda}\) of the pencil are irreducible for any \(\lambda \in \mathbb{P}^1\) except
\(\mathcal{S}_{\infty} = S_{(X)} + S_{(Y)} + S_{(Z)} + S_{(T)}\).
The base locus of the pencil \(\mathcal{S}\) consists of the following curves:
\begin{gather*}
  C_{1} = C_{(X, T)}, \;
  C_{2} = C_{(Y, T)}, \;
  C_{3} = C_{(Z, T)}, \;
  C_{4} = C_{(T, X + Y + Z)}, \;
  C_{5} = C_{(Y, X Z + T^2)}, \;
  C_{6} = C_{(X, T^3 + Y Z (Y + T))}.
\end{gather*}
Their linear equivalence classes on the generic member \(\mathcal{S}_{\Bbbk}\) of the pencil satisfy the following relations:
\begin{gather*}
  \begin{pmatrix}
    [C_{4}] \\ [C_{5}] \\ [C_{6}] \\ [H_{\mathcal{S}}]
  \end{pmatrix} = 
  \begin{pmatrix}
    -1 & -1 & 3 \\
    0 & -2 & 4 \\
    -1 & 0 & 4 \\
    0 & 0 & 4
  \end{pmatrix} \cdot
  \begin{pmatrix}
    [C_{1}] \\ [C_{2}] \\ [C_{3}]
  \end{pmatrix}.
\end{gather*}

For a general choice of \(\lambda \in \mathbb{C}\) the surface \(\mathcal{S}_{\lambda}\) has the following singularities:
\begin{itemize}\setlength{\itemindent}{2cm}
\item[\(P_{1} = P_{(X, Y, T)}\):] type \(\mathbb{A}_4\) with the quadratic term \(X \cdot Y\);
\item[\(P_{2} = P_{(X, Z, T)}\):] type \(\mathbb{A}_3\) with the quadratic term \(Z \cdot (X + T)\);
\item[\(P_{3} = P_{(Y, Z, T)}\):] type \(\mathbb{A}_3\) with the quadratic term \(Y \cdot Z\);
\item[\(P_{4} = P_{(Y, T, X + Z)}\):] type \(\mathbb{A}_1\) with the quadratic term \(Y (X + Y + Z - \lambda T) + T^2\);
\item[\(P_{5} = P_{(Z, T, X + Y)}\):] type \(\mathbb{A}_3\) with the quadratic term \(Z \cdot (X + Y + Z - (\lambda + 1) T)\).
\end{itemize}

Galois action on the lattice \(L_{\lambda}\) is trivial. The intersection matrix on \(L_{\lambda} = L_{\mathcal{S}}\) is represented by
\begin{table}[H]
  \begin{tabular}{|c||cccc|ccc|ccc|c|ccc|ccc|}
    \hline
    \(\bullet\) & \(E_1^1\) & \(E_1^2\) & \(E_1^3\) & \(E_1^4\) & \(E_2^1\) & \(E_2^2\) & \(E_2^3\) & \(E_3^1\) & \(E_3^2\) & \(E_3^3\) & \(E_4^1\) & \(E_5^1\) & \(E_5^2\) & \(E_5^3\) & \(\widetilde{C_{1}}\) & \(\widetilde{C_{2}}\) & \(\widetilde{C_{3}}\) \\
    \hline
    \hline
    \(\widetilde{C_{1}}\) & \(1\) & \(0\) & \(0\) & \(0\) & \(1\) & \(0\) & \(0\) & \(0\) & \(0\) & \(0\) & \(0\) & \(0\) & \(0\) & \(0\) & \(-2\) & \(0\) & \(0\) \\
    \(\widetilde{C_{2}}\) & \(0\) & \(0\) & \(0\) & \(1\) & \(0\) & \(0\) & \(0\) & \(1\) & \(0\) & \(0\) & \(1\) & \(0\) & \(0\) & \(0\) & \(0\) & \(-2\) & \(0\) \\
    \(\widetilde{C_{3}}\) & \(0\) & \(0\) & \(0\) & \(0\) & \(0\) & \(0\) & \(1\) & \(0\) & \(0\) & \(1\) & \(0\) & \(1\) & \(0\) & \(0\) & \(0\) & \(0\) & \(-2\) \\
    \hline
  \end{tabular}.
\end{table}
Note that the intersection matrix is non-degenerate.

Discriminant groups and discriminant forms of the lattices \(L_{\mathcal{S}}\) and \(H \oplus \Pic(X)\) are given by
\begin{gather*}
  G' = 
  \begin{pmatrix}
    \frac{6}{11} & \frac{4}{11} & \frac{2}{11} & 0 & \frac{10}{11} & \frac{1}{11} & \frac{3}{11} & \frac{5}{22} & \frac{7}{11} & \frac{1}{22} & \frac{9}{22} & \frac{13}{22} & \frac{8}{11} & \frac{19}{22} & \frac{8}{11} & \frac{9}{11} & \frac{5}{11}
  \end{pmatrix}, \\
  G'' = 
  \begin{pmatrix}
    0 & 0 & \frac{25}{11} & \frac{9}{22} & -\frac{5}{11}
  \end{pmatrix}; \;
  B' = 
  \begin{pmatrix}
    \frac{3}{22}
  \end{pmatrix}, \;
  B'' = 
  \begin{pmatrix}
    \frac{19}{22}
  \end{pmatrix}; \;
  Q' =
  \begin{pmatrix}
    \frac{25}{22}
  \end{pmatrix}, \;
  Q'' =
  \begin{pmatrix}
    \frac{19}{22}
  \end{pmatrix}.
\end{gather*}


\subsection{Family \textnumero3.25}\label{subsection:03-25}

The pencil \(\mathcal{S}\) is defined by the equation
\begin{gather*}
  X^{2} Y Z + X Y^{2} Z + X Y Z^{2} + X^{2} Y T + Y Z T^{2} + Z T^{3} = \lambda X Y Z T.
\end{gather*}
Members \(\mathcal{S}_{\lambda}\) of the pencil are irreducible for any \(\lambda \in \mathbb{P}^1\) except
\(\mathcal{S}_{\infty} = S_{(X)} + S_{(Y)} + S_{(Z)} + S_{(T)}\).
The base locus of the pencil \(\mathcal{S}\) consists of the following curves:
\begin{gather*}
  C_{1} = C_{(X, Z)}, \;
  C_{2} = C_{(X, T)}, \;
  C_{3} = C_{(Y, Z)}, \;
  C_{4} = C_{(Y, T)}, \;
  C_{5} = C_{(Z, T)}, \;
  C_{6} = C_{(X, Y + T)}, \;
  C_{7} = C_{(T, X + Y + Z)}.
\end{gather*}
Their linear equivalence classes on the generic member \(\mathcal{S}_{\Bbbk}\) of the pencil satisfy the following relations:
\begin{gather*}
  \begin{pmatrix}
    [C_{3}] \\ [C_{5}] \\ [C_{6}] \\ [C_{7}]
  \end{pmatrix} = 
  \begin{pmatrix}
    0 & 0 & -3 & 1 \\
    -2 & 0 & 3 & 0 \\
    -1 & -2 & 0 & 1 \\
    2 & -1 & -4 & 1
  \end{pmatrix} \cdot
  \begin{pmatrix}
    [C_{1}] \\ [C_{2}] \\ [C_{4}] \\ [H_{\mathcal{S}}]
  \end{pmatrix}.
\end{gather*}

For a general choice of \(\lambda \in \mathbb{C}\) the surface \(\mathcal{S}_{\lambda}\) has the following singularities:
\begin{itemize}\setlength{\itemindent}{2cm}
\item[\(P_{1} = P_{(X, Y, T)}\):] type \(\mathbb{A}_2\) with the quadratic term \(X \cdot Y\);
\item[\(P_{2} = P_{(X, Z, T)}\):] type \(\mathbb{A}_4\) with the quadratic term \(X \cdot Z\);
\item[\(P_{3} = P_{(Y, Z, T)}\):] type \(\mathbb{A}_3\) with the quadratic term \(Y \cdot (Z + T)\);
\item[\(P_{4} = P_{(X, Z, Y + T)}\):] type \(\mathbb{A}_1\) with the quadratic term \(Z ((\lambda + 1) X + Y + T) - X^2\);
\item[\(P_{5} = P_{(X, T, Y + Z)}\):] type \(\mathbb{A}_1\) with the quadratic term \(X (X + Y + Z - \lambda T) + T^2\);
\item[\(P_{6} = P_{(Y, T, X + Z)}\):] type \(\mathbb{A}_2\) with the quadratic term \(Y \cdot (X + Y + Z - (\lambda + 1) T)\).
\end{itemize}

Galois action on the lattice \(L_{\lambda}\) is trivial. The intersection matrix on \(L_{\lambda} = L_{\mathcal{S}}\) is represented by
\begin{table}[H]
  \begin{tabular}{|c||cc|cccc|ccc|c|c|cc|cccc|}
    \hline
    \(\bullet\) & \(E_1^1\) & \(E_1^2\) & \(E_2^1\) & \(E_2^2\) & \(E_2^3\) & \(E_2^4\) & \(E_3^1\) & \(E_3^2\) & \(E_3^3\) & \(E_4^1\) & \(E_5^1\) & \(E_6^1\) & \(E_6^2\) & \(\widetilde{C_{1}}\) & \(\widetilde{C_{2}}\) & \(\widetilde{C_{4}}\) & \(\widetilde{H_{\mathcal{S}}}\) \\
    \hline
    \hline
    \(\widetilde{C_{1}}\) & \(0\) & \(0\) & \(0\) & \(0\) & \(1\) & \(0\) & \(0\) & \(0\) & \(0\) & \(1\) & \(0\) & \(0\) & \(0\) & \(-2\) & \(0\) & \(0\) & \(1\) \\
    \(\widetilde{C_{2}}\) & \(1\) & \(0\) & \(1\) & \(0\) & \(0\) & \(0\) & \(0\) & \(0\) & \(0\) & \(0\) & \(1\) & \(0\) & \(0\) & \(0\) & \(-2\) & \(0\) & \(1\) \\
    \(\widetilde{C_{4}}\) & \(0\) & \(1\) & \(0\) & \(0\) & \(0\) & \(0\) & \(1\) & \(0\) & \(0\) & \(0\) & \(0\) & \(1\) & \(0\) & \(0\) & \(0\) & \(-2\) & \(1\) \\
    \(\widetilde{H_{\mathcal{S}}}\) & \(0\) & \(0\) & \(0\) & \(0\) & \(0\) & \(0\) & \(0\) & \(0\) & \(0\) & \(0\) & \(0\) & \(0\) & \(0\) & \(1\) & \(1\) & \(1\) & \(4\) \\
    \hline
  \end{tabular}.
\end{table}
Note that the intersection matrix is non-degenerate.

Discriminant groups and discriminant forms of the lattices \(L_{\mathcal{S}}\) and \(H \oplus \Pic(X)\) are given by
\begin{gather*}
  G' = 
  \begin{pmatrix}
    \frac{3}{5} & \frac{1}{10} & \frac{4}{5} & \frac{1}{2} & \frac{1}{5} & \frac{3}{5} & \frac{19}{20} & \frac{3}{10} & \frac{13}{20} & \frac{13}{20} & \frac{1}{20} & \frac{2}{5} & \frac{1}{5} & \frac{3}{10} & \frac{1}{10} & \frac{3}{5} & \frac{3}{4}
  \end{pmatrix}, \\
  G'' = 
  \begin{pmatrix}
    0 & 0 & \frac{1}{20} & -\frac{9}{20} & \frac{1}{10}
  \end{pmatrix}; \;
  B' = 
  \begin{pmatrix}
    \frac{9}{20}
  \end{pmatrix}, \;
  B'' = 
  \begin{pmatrix}
    \frac{11}{20}
  \end{pmatrix}; \;
  Q' =
  \begin{pmatrix}
    \frac{9}{20}
  \end{pmatrix}, \;
  Q'' =
  \begin{pmatrix}
    \frac{31}{20}
  \end{pmatrix}.
\end{gather*}


\subsection{Family \textnumero3.26}\label{subsection:03-26}

The pencil \(\mathcal{S}\) is defined by the equation
\[
  X^{2} Y Z + X Y^{2} Z + X Y Z^{2} + Y^{2} Z T + Y Z T^{2} + X T^{3} = \lambda X Y Z T.
\]
Members \(\mathcal{S}_{\lambda}\) of the pencil are irreducible for any \(\lambda \in \mathbb{P}^1\) except
\(\mathcal{S}_{\infty} = S_{(X)} + S_{(Y)} + S_{(Z)} + S_{(T)}\).
The base locus of the pencil \(\mathcal{S}\) consists of the following curves:
\[
  C_1 = C_{(X, Y)}, \;
  C_2 = C_{(X, Z)}, \;
  C_3 = C_{(X, T)}, \;
  C_4 = C_{(Y, T)}, \;
  C_5 = C_{(Z, T)}, \;
  C_6 = C_{(X, Y + T)}, \;
  C_7 = C_{(T, X + Y + Z)}.
\]
Their linear equivalence classes on the generic member \(\mathcal{S}_{\Bbbk}\) of the pencil satisfy the following relations:
\begin{gather*}
  \begin{pmatrix}
    [C_{1}] \\ [C_{2}] \\ [C_{6}] \\ [C_{7}]
  \end{pmatrix} = 
  \begin{pmatrix}
    0 & -3 & 0 & 1 \\
    0 & 0 & -3 & 1 \\
    -1 & 3 & 3 & -1 \\
    -1 & -1 & -1 & 1
  \end{pmatrix} \cdot
  \begin{pmatrix}
    [C_{3}] \\ [C_{4}] \\ [C_{5}] \\ [H_{\mathcal{S}}]
  \end{pmatrix}.
\end{gather*}

For a general choice of \(\lambda \in \mathbb{C}\) the surface \(\mathcal{S}_{\lambda}\) has the following singularities:
\begin{itemize}\setlength{\itemindent}{2cm}
\item[\(P_{1} = P_{(X, Y, T)}\):] type \(\mathbb{A}_4\) with the quadratic term \(X \cdot Y\);
\item[\(P_{2} = P_{(X, Z, T)}\):] type \(\mathbb{A}_3\) with the quadratic term \(Z \cdot (X + T)\);
\item[\(P_{3} = P_{(Y, Z, T)}\):] type \(\mathbb{A}_2\) with the quadratic term \(Y \cdot Z\);
\item[\(P_{4} = P_{(Y, T, X + Z)}\):] type \(\mathbb{A}_2\) with the quadratic term \(Y \cdot (X + Y + Z - \lambda T)\);
\item[\(P_{5} = P_{(Z, T, X + Y)}\):] type \(\mathbb{A}_2\) with the quadratic term \(Z \cdot (X + Y + Z - (\lambda + 1) T)\).
\end{itemize}

Galois action on the lattice \(L_{\lambda}\) is trivial. The intersection matrix on \(L_{\lambda} = L_{\mathcal{S}}\) is represented by
\begin{table}[H]
  \begin{tabular}{|c||cccc|ccc|cc|cc|cc|cccc|}
    \hline
    \(\bullet\) & \(E_1^1\) & \(E_1^2\) & \(E_1^3\) & \(E_1^4\) & \(E_2^1\) & \(E_2^2\) & \(E_2^3\) & \(E_3^1\) & \(E_3^2\) & \(E_4^1\) & \(E_4^2\) & \(E_5^1\) & \(E_5^2\) & \(\widetilde{C_{3}}\) & \(\widetilde{C_{4}}\) & \(\widetilde{C_{5}}\) & \(\widetilde{H_{\mathcal{S}}}\) \\
    \hline
    \hline
    \(\widetilde{C_{3}}\) & \(1\) & \(0\) & \(0\) & \(0\) & \(0\) & \(0\) & \(1\) & \(0\) & \(0\) & \(0\) & \(0\) & \(0\) & \(0\) & \(-2\) & \(0\) & \(0\) & \(1\) \\
    \(\widetilde{C_{4}}\) & \(0\) & \(0\) & \(0\) & \(1\) & \(0\) & \(0\) & \(0\) & \(1\) & \(0\) & \(1\) & \(0\) & \(0\) & \(0\) & \(0\) & \(-2\) & \(0\) & \(1\) \\
    \(\widetilde{C_{5}}\) & \(0\) & \(0\) & \(0\) & \(0\) & \(1\) & \(0\) & \(0\) & \(0\) & \(1\) & \(0\) & \(0\) & \(1\) & \(0\) & \(0\) & \(0\) & \(-2\) & \(1\) \\
    \(\widetilde{H_{\mathcal{S}}}\) & \(0\) & \(0\) & \(0\) & \(0\) & \(0\) & \(0\) & \(0\) & \(0\) & \(0\) & \(0\) & \(0\) & \(0\) & \(0\) & \(1\) & \(1\) & \(1\) & \(4\) \\
    \hline
  \end{tabular}.
\end{table}
Note that the intersection matrix is non-degenerate.

Discriminant groups and discriminant forms of the lattices \(L_{\mathcal{S}}\) and \(H \oplus \Pic(X)\) are given by
\begin{gather*}
  G' = 
  \begin{pmatrix}
    \frac{2}{3} & \frac{1}{6} & \frac{2}{3} & \frac{1}{6} & \frac{2}{3} & \frac{1}{2} & \frac{1}{3} & \frac{13}{18} & \frac{7}{9} & \frac{1}{9} & \frac{5}{9} & \frac{8}{9} & \frac{17}{18} & \frac{1}{6} & \frac{2}{3} & \frac{5}{6} & \frac{1}{3}
  \end{pmatrix}, \\
  G'' = 
  \begin{pmatrix}
    0 & 0 & -\frac{1}{2} & -\frac{2}{9} & -\frac{4}{9}
  \end{pmatrix}; \;
  B' = 
  \begin{pmatrix}
    \frac{11}{18}
  \end{pmatrix}, \;
  B'' = 
  \begin{pmatrix}
    \frac{7}{18}
  \end{pmatrix}; \;
  Q' =
  \begin{pmatrix}
    \frac{29}{18}
  \end{pmatrix}, \;
  Q'' =
  \begin{pmatrix}
    \frac{7}{18}
  \end{pmatrix}.
\end{gather*}


\subsection{Family \textnumero3.27}\label{subsection:03-27}

The pencil \(\mathcal{S}\) is defined by the equation
\begin{gather*}
  X^{2} Y Z + X Y^{2} Z + X Y Z^{2} + X Y T^{2} + X Z T^{2} + Y Z T^{2} = \lambda X Y Z T.
\end{gather*}
Members \(\mathcal{S}_{\lambda}\) of the pencil are irreducible for any \(\lambda \in \mathbb{P}^1\) except
\(\mathcal{S}_{\infty} = S_{(X)} + S_{(Y)} + S_{(Z)} + S_{(T)}\).
The base locus of the pencil \(\mathcal{S}\) consists of the following curves:
\begin{gather*}
  C_{1} = C_{(X, Y)}, \;
  C_{2} = C_{(X, Z)}, \;
  C_{3} = C_{(X, T)}, \;
  C_{4} = C_{(Y, Z)}, \;
  C_{5} = C_{(Y, T)}, \;
  C_{6} = C_{(Z, T)}, \;
  C_{7} = C_{(T, X + Y + Z)}.
\end{gather*}
Their linear equivalence classes on the generic member \(\mathcal{S}_{\Bbbk}\) of the pencil satisfy the following relations:
\begin{gather*}
  \begin{pmatrix}
    [C_{2}] \\ [C_{4}] \\ [C_{7}] \\ [H_{\mathcal{S}}]
  \end{pmatrix} = 
  \begin{pmatrix}
    1 & 0 & 2 & -2 \\
    1 & 2 & 0 & -2 \\
    2 & 1 & 1 & -3 \\
    2 & 2 & 2 & -2
  \end{pmatrix} \cdot
  \begin{pmatrix}
    [C_{1}] \\ [C_{3}] \\ [C_{5}] \\ [C_{6}]
  \end{pmatrix}.
\end{gather*}

For a general choice of \(\lambda \in \mathbb{C}\) the surface \(\mathcal{S}_{\lambda}\) has the following singularities:
\begin{itemize}\setlength{\itemindent}{2cm}
\item[\(P_{1} = P_{(X, Y, Z)}\):] type \(\mathbb{A}_1\) with the quadratic term \(X Y + X Z + Y Z\);
\item[\(P_{2} = P_{(X, Y, T)}\):] type \(\mathbb{A}_3\) with the quadratic term \(X \cdot Y\);
\item[\(P_{3} = P_{(X, Z, T)}\):] type \(\mathbb{A}_3\) with the quadratic term \(X \cdot Z\);
\item[\(P_{4} = P_{(Y, Z, T)}\):] type \(\mathbb{A}_3\) with the quadratic term \(Y \cdot Z\);
\item[\(P_{5} = P_{(X, T, Y + Z)}\):] type \(\mathbb{A}_1\) with the quadratic term \(X (X + Y + Z - \lambda T) + T^2\);
\item[\(P_{6} = P_{(Y, T, X + Z)}\):] type \(\mathbb{A}_1\) with the quadratic term \(Y (X + Y + Z - \lambda T) + T^2\);
\item[\(P_{7} = P_{(Z, T, X + Y)}\):] type \(\mathbb{A}_1\) with the quadratic term \(Z (X + Y + Z - \lambda T) + T^2\).
\end{itemize}

Galois action on the lattice \(L_{\lambda}\) is trivial. The intersection matrix on \(L_{\lambda} = L_{\mathcal{S}}\) is represented by
\begin{table}[H]
  \begin{tabular}{|c||c|ccc|ccc|ccc|c|c|c|cccc|}
    \hline
    \(\bullet\) & \(E_1^1\) & \(E_2^1\) & \(E_2^2\) & \(E_2^3\) & \(E_3^1\) & \(E_3^2\) & \(E_3^3\) & \(E_4^1\) & \(E_4^2\) & \(E_4^3\) & \(E_5^1\) & \(E_6^1\) & \(E_7^1\) & \(\widetilde{C_{1}}\) & \(\widetilde{C_{3}}\) & \(\widetilde{C_{5}}\) & \(\widetilde{C_{6}}\) \\
    \hline
    \hline
    \(\widetilde{C_{1}}\) & \(1\) & \(0\) & \(1\) & \(0\) & \(0\) & \(0\) & \(0\) & \(0\) & \(0\) & \(0\) & \(0\) & \(0\) & \(0\) & \(-2\) & \(0\) & \(0\) & \(0\) \\
    \(\widetilde{C_{3}}\) & \(0\) & \(1\) & \(0\) & \(0\) & \(1\) & \(0\) & \(0\) & \(0\) & \(0\) & \(0\) & \(1\) & \(0\) & \(0\) & \(0\) & \(-2\) & \(0\) & \(0\) \\
    \(\widetilde{C_{5}}\) & \(0\) & \(0\) & \(0\) & \(1\) & \(0\) & \(0\) & \(0\) & \(1\) & \(0\) & \(0\) & \(0\) & \(1\) & \(0\) & \(0\) & \(0\) & \(-2\) & \(0\) \\
    \(\widetilde{C_{6}}\) & \(0\) & \(0\) & \(0\) & \(0\) & \(0\) & \(0\) & \(1\) & \(0\) & \(0\) & \(1\) & \(0\) & \(0\) & \(1\) & \(0\) & \(0\) & \(0\) & \(-2\) \\
    \hline
  \end{tabular}.
\end{table}
Note that the intersection matrix is non-degenerate.

Discriminant groups and discriminant forms of the lattices \(L_{\mathcal{S}}\) and \(H \oplus \Pic(X)\) are given by
\begin{gather*}
  G' = 
  \begin{pmatrix}
    0 & 0 & 0 & 0 & \frac{1}{2} & 0 & \frac{1}{2} & 0 & 0 & 0 & \frac{1}{2} & 0 & \frac{1}{2} & 0 & 0 & 0 & 0 \\
    0 & \frac{1}{2} & 0 & \frac{1}{2} & 0 & 0 & 0 & 0 & 0 & 0 & \frac{1}{2} & \frac{1}{2} & 0 & 0 & 0 & 0 & 0 \\
    \frac{1}{2} & \frac{3}{4} & \frac{1}{2} & \frac{1}{4} & \frac{1}{4} & \frac{1}{2} & \frac{3}{4} & \frac{1}{4} & \frac{1}{2} & \frac{3}{4} & 0 & \frac{1}{2} & \frac{1}{2} & 0 & 0 & 0 & 0
  \end{pmatrix}, \\
  G'' = 
  \begin{pmatrix}
    0 & 0 & \frac{1}{2} & 0 & 0 \\
    0 & 0 & 0 & \frac{1}{2} & 0 \\
    0 & 0 & -\frac{1}{4} & -\frac{1}{4} & \frac{1}{4}
  \end{pmatrix}; \;
  B' = 
  \begin{pmatrix}
    0 & \frac{1}{2} & 0 \\
    \frac{1}{2} & 0 & 0 \\
    0 & 0 & \frac{1}{4}
  \end{pmatrix}, \;
  B'' = 
  \begin{pmatrix}
    0 & \frac{1}{2} & 0 \\
    \frac{1}{2} & 0 & 0 \\
    0 & 0 & \frac{3}{4}
  \end{pmatrix}; \;
  \begin{pmatrix}
    Q' \\ Q''
  \end{pmatrix}
  =
  \begin{pmatrix}
    0 & 0 & \frac{1}{4} \\
    0 & 0 & \frac{7}{4}
  \end{pmatrix}.
\end{gather*}


\subsection{Family \textnumero3.28}\label{subsection:03-28}

The pencil \(\mathcal{S}\) is defined by the equation
\begin{gather*}
  X^{2} Y Z + X Y^{2} Z + X Y Z^{2} + X^{2} Y T + X Z T^{2} + Y Z T^{2} = \lambda X Y Z T.
\end{gather*}
Members \(\mathcal{S}_{\lambda}\) of the pencil are irreducible for any \(\lambda \in \mathbb{P}^1\) except
\(\mathcal{S}_{\infty} = S_{(X)} + S_{(Y)} + S_{(Z)} + S_{(T)}\).
The base locus of the pencil \(\mathcal{S}\) consists of the following curves:
\begin{gather*}
  C_{1} = C_{(X, Y)}, \;
  C_{2} = C_{(X, Z)}, \;
  C_{3} = C_{(X, T)}, \;
  C_{4} = C_{(Y, Z)}, \;
  C_{5} = C_{(Y, T)}, \;
  C_{6} = C_{(Z, T)}, \;
  C_{7} = C_{(T, X + Y + Z)}.
\end{gather*}
Their linear equivalence classes on the generic member \(\mathcal{S}_{\Bbbk}\) of the pencil satisfy the following relations:
\begin{gather*}
  \begin{pmatrix}
    [C_{2}] \\ [C_{4}] \\ [C_{6}] \\ [C_{7}]
  \end{pmatrix} = 
  \begin{pmatrix}
    -1 & -2 & 0 & 1 \\
    -1 & 0 & -2 & 1 \\
    3 & 4 & 2 & -2 \\
    -3 & -5 & -3 & 3
  \end{pmatrix} \cdot
  \begin{pmatrix}
    [C_{1}] \\ [C_{3}] \\ [C_{5}] \\ [H_{\mathcal{S}}]
  \end{pmatrix}.
\end{gather*}

For a general choice of \(\lambda \in \mathbb{C}\) the surface \(\mathcal{S}_{\lambda}\) has the following singularities:
\begin{itemize}\setlength{\itemindent}{2cm}
\item[\(P_{1} = P_{(X, Y, Z)}\):] type \(\mathbb{A}_2\) with the quadratic term \(Z \cdot (X + Y)\);
\item[\(P_{2} = P_{(X, Y, T)}\):] type \(\mathbb{A}_3\) with the quadratic term \(X \cdot Y\);
\item[\(P_{3} = P_{(X, Z, T)}\):] type \(\mathbb{A}_4\) with the quadratic term \(X \cdot Z\);
\item[\(P_{4} = P_{(Y, Z, T)}\):] type \(\mathbb{A}_2\) with the quadratic term \(Y \cdot (Z + T)\);
\item[\(P_{5} = P_{(X, T, Y + Z)}\):] type \(\mathbb{A}_1\) with the quadratic term \(X (X + Y + Z - \lambda T) + T^2\);
\item[\(P_{6} = P_{(Y, T, X + Z)}\):] type \(\mathbb{A}_1\) with the quadratic term \(Y (X + Y + Z - (\lambda + 1) T) + T^2\).
\end{itemize}

Galois action on the lattice \(L_{\lambda}\) is trivial. The intersection matrix on \(L_{\lambda} = L_{\mathcal{S}}\) is represented by
\begin{table}[H]
  \begin{tabular}{|c||cc|ccc|cccc|cc|c|c|cccc|}
    \hline
    \(\bullet\) & \(E_1^1\) & \(E_1^2\) & \(E_2^1\) & \(E_2^2\) & \(E_2^3\) & \(E_3^1\) & \(E_3^2\) & \(E_3^3\) & \(E_3^4\) & \(E_4^1\) & \(E_4^2\) & \(E_5^1\) & \(E_6^1\) & \(\widetilde{C_{1}}\) & \(\widetilde{C_{3}}\) & \(\widetilde{C_{5}}\) & \(\widetilde{H_{\mathcal{S}}}\) \\
    \hline
    \hline
    \(\widetilde{C_{1}}\) & \(1\) & \(0\) & \(0\) & \(1\) & \(0\) & \(0\) & \(0\) & \(0\) & \(0\) & \(0\) & \(0\) & \(0\) & \(0\) & \(-2\) & \(0\) & \(0\) & \(1\) \\
    \(\widetilde{C_{3}}\) & \(0\) & \(0\) & \(1\) & \(0\) & \(0\) & \(1\) & \(0\) & \(0\) & \(0\) & \(0\) & \(0\) & \(1\) & \(0\) & \(0\) & \(-2\) & \(0\) & \(1\) \\
    \(\widetilde{C_{5}}\) & \(0\) & \(0\) & \(0\) & \(0\) & \(1\) & \(0\) & \(0\) & \(0\) & \(0\) & \(1\) & \(0\) & \(0\) & \(1\) & \(0\) & \(0\) & \(-2\) & \(1\) \\
    \(\widetilde{H_{\mathcal{S}}}\) & \(0\) & \(0\) & \(0\) & \(0\) & \(0\) & \(0\) & \(0\) & \(0\) & \(0\) & \(0\) & \(0\) & \(0\) & \(0\) & \(1\) & \(1\) & \(1\) & \(4\) \\
    \hline
  \end{tabular}.
\end{table}
Note that the intersection matrix is non-degenerate.

Discriminant groups and discriminant forms of the lattices \(L_{\mathcal{S}}\) and \(H \oplus \Pic(X)\) are given by
\begin{gather*}
  G' = 
  \begin{pmatrix}
    0 & \frac{1}{2} & \frac{13}{16} & \frac{1}{4} & \frac{3}{16} & \frac{1}{2} & \frac{5}{8} & \frac{3}{4} & \frac{7}{8} & \frac{3}{4} & \frac{3}{8} & \frac{11}{16} & \frac{9}{16} & \frac{1}{2} & \frac{3}{8} & \frac{1}{8} & \frac{3}{4}
  \end{pmatrix}, \\
  G'' = 
  \begin{pmatrix}
    0 & 0 & -\frac{3}{16} & -\frac{3}{8} & \frac{1}{16}
  \end{pmatrix}; \;
  B' = 
  \begin{pmatrix}
    \frac{1}{16}
  \end{pmatrix}, \;
  B'' = 
  \begin{pmatrix}
    \frac{15}{16}
  \end{pmatrix}; \;
  Q' =
  \begin{pmatrix}
    \frac{1}{16}
  \end{pmatrix}, \;
  Q'' =
  \begin{pmatrix}
    \frac{31}{16}
  \end{pmatrix}.
\end{gather*}


\subsection{Family \textnumero3.29}\label{subsection:03-29}

The pencil \(\mathcal{S}\) is defined by the equation
\begin{gather*}
  X^{2} Y Z + X Y^{2} Z + X Y Z^{2} + Y^{2} Z T + Y Z T^{2} + T^{4} = \lambda X Y Z T.
\end{gather*}
Members \(\mathcal{S}_{\lambda}\) of the pencil are irreducible for any \(\lambda \in \mathbb{P}^1\) except
\(\mathcal{S}_{\infty} = S_{(X)} + S_{(Y)} + S_{(Z)} + S_{(T)}\).
The base locus of the pencil \(\mathcal{S}\) consists of the following curves:
\begin{gather*}
  C_{1} = C_{(X, T)}, \;
  C_{2} = C_{(Y, T)}, \;
  C_{3} = C_{(Z, T)}, \;
  C_{4} = C_{(T, X + Y + Z)}, \;
  C_{5} = C_{(X, T^3 + Y Z (Y + T))}.
\end{gather*}
Their linear equivalence classes on the generic member \(\mathcal{S}_{\Bbbk}\) of the pencil satisfy the following relations:
\[
  \begin{pmatrix}
    [C_3] + [C_4] \\ 4 [C_4] \\ [C_5] \\ [H_{\mathcal{S}}]
  \end{pmatrix} =
  \begin{pmatrix}
    -1 & 3 \\
    -4 & 8 \\
    -1 & 4 \\
    0 & 4
  \end{pmatrix} \cdot
  \begin{pmatrix}
    [C_1] \\ [C_2]
  \end{pmatrix}.
\]

For a general choice of \(\lambda \in \mathbb{C}\) the surface \(\mathcal{S}_{\lambda}\) has the following singularities:
\begin{itemize}\setlength{\itemindent}{2cm}
\item[\(P_{1} = P_{(X, Y, T)}\):] type \(\mathbb{A}_3\) with the quadratic term \(X \cdot Y\);
\item[\(P_{2} = P_{(X, Z, T)}\):] type \(\mathbb{A}_3\) with the quadratic term \(Z \cdot (X + T)\);
\item[\(P_{3} = P_{(Y, Z, T)}\):] type \(\mathbb{A}_3\) with the quadratic term \(Y \cdot Z\);
\item[\(P_{4} = P_{(Y, T, X + Z)}\):] type \(\mathbb{A}_3\) with the quadratic term \(Y \cdot (X + Y + Z - \lambda T)\);
\item[\(P_{5} = P_{(Z, T, X + Y)}\):] type \(\mathbb{A}_3\) with the quadratic term \(Z \cdot (X + Y + Z - (\lambda + 1) T)\).
\end{itemize}

Galois action on the lattice \(L_{\lambda}\) is trivial. The intersection matrix on \(L_{\lambda} = L_{\mathcal{S}}\) is represented by
\begin{table}[H]
  \begin{tabular}{|c||ccc|ccc|ccc|ccc|ccc|cccc|}
    \hline
    \(\bullet\) & \(E_1^1\) & \(E_1^2\) & \(E_1^3\) & \(E_2^1\) & \(E_2^2\) & \(E_2^3\) & \(E_3^1\) & \(E_3^2\) & \(E_3^3\) & \(E_4^1\) & \(E_4^2\) & \(E_4^3\) & \(E_5^1\) & \(E_5^2\) & \(E_5^3\) & \(\widetilde{C_{1}}\) & \(\widetilde{C_{2}}\) & \(\widetilde{C_{3}}\) & \(\widetilde{C_{4}}\) \\
    \hline
    \hline
    \(\widetilde{C_{1}}\) & \(1\) & \(0\) & \(0\) & \(1\) & \(0\) & \(0\) & \(0\) & \(0\) & \(0\) & \(0\) & \(0\) & \(0\) & \(0\) & \(0\) & \(0\) & \(-2\) & \(0\) & \(0\) & \(1\) \\
    \(\widetilde{C_{2}}\) & \(0\) & \(0\) & \(1\) & \(0\) & \(0\) & \(0\) & \(1\) & \(0\) & \(0\) & \(1\) & \(0\) & \(0\) & \(0\) & \(0\) & \(0\) & \(0\) & \(-2\) & \(0\) & \(0\) \\
    \(\widetilde{C_{3}}\) & \(0\) & \(0\) & \(0\) & \(0\) & \(0\) & \(1\) & \(0\) & \(0\) & \(1\) & \(0\) & \(0\) & \(0\) & \(1\) & \(0\) & \(0\) & \(0\) & \(0\) & \(-2\) & \(0\) \\
    \(\widetilde{C_{4}}\) & \(0\) & \(0\) & \(0\) & \(0\) & \(0\) & \(0\) & \(0\) & \(0\) & \(0\) & \(0\) & \(0\) & \(1\) & \(0\) & \(0\) & \(1\) & \(1\) & \(0\) & \(0\) & \(-2\) \\
    \hline
  \end{tabular}.
\end{table}
Note that the intersection matrix is degenerate. We choose the following integral basis of the lattice \(L_{\lambda}\):
\begin{align*}
  \begin{pmatrix}
    [E_5^3] \\ [\widetilde{C_{4}}]
  \end{pmatrix} = 
  \begin{pmatrix}
    1 & 2 & 3 & -1 & -2 & -3 & 2 & 0 & -2 & 3 & 2 & 1 & -3 & -2 & 0 & 4 & -4 \\
    -1 & -1 & -1 & 0 & 1 & 2 & 0 & 1 & 2 & -1 & -1 & -1 & 2 & 1 & -1 & -1 & 3
  \end{pmatrix} \cdot \\
  \begin{pmatrix}
    [E_1^1] & [E_1^2] & [E_1^3] & [E_2^1] & [E_2^2] & [E_2^3] & [E_3^1] & [E_3^2] & [E_3^3] \\ [E_4^1] & [E_4^2] & [E_4^3] & [E_5^1] & [E_5^2] & [\widetilde{C_{1}}] & [\widetilde{C_{2}}] & [\widetilde{C_{3}}]
  \end{pmatrix}^T.
\end{align*}

Discriminant groups and discriminant forms of the lattices \(L_{\mathcal{S}}\) and \(H \oplus \Pic(X)\) are given by
\begin{gather*}
  G' = 
  \begin{pmatrix}
    \frac{7}{12} & \frac{1}{2} & \frac{5}{12} & \frac{3}{4} & \frac{5}{6} & \frac{11}{12} & \frac{1}{4} & \frac{1}{6} & \frac{1}{12} & 0 & \frac{2}{3} & \frac{1}{3} & 0 & 0 & \frac{2}{3} & \frac{1}{3} & 0
  \end{pmatrix}, \\
  G'' = 
  \begin{pmatrix}
    0 & 0 & \frac{2}{3} & -\frac{1}{3} & -\frac{1}{4}
  \end{pmatrix}; \;
  B' = 
  \begin{pmatrix}
    \frac{5}{12}
  \end{pmatrix}, \;
  B'' = 
  \begin{pmatrix}
    \frac{7}{12}
  \end{pmatrix}; \;
  Q' =
  \begin{pmatrix}
    \frac{5}{12}
  \end{pmatrix}, \;
  Q'' =
  \begin{pmatrix}
    \frac{19}{12}
  \end{pmatrix}.
\end{gather*}


\subsection{Family \textnumero3.30}\label{subsection:03-30}

The pencil \(\mathcal{S}\) is defined by the equation
\begin{gather*}
  X^{2} Y Z + X Y^{2} Z + X Y Z^{2} + X Y^{2} T + X^{2} Z T + Y Z T^{2} = \lambda X Y Z T.
\end{gather*}
Members \(\mathcal{S}_{\lambda}\) of the pencil are irreducible for any \(\lambda \in \mathbb{P}^1\) except
\(\mathcal{S}_{\infty} = S_{(X)} + S_{(Y)} + S_{(Z)} + S_{(T)}\).
The base locus of the pencil \(\mathcal{S}\) consists of the following curves:
\begin{gather*}
  C_{1} = C_{(X, Y)}, \;
  C_{2} = C_{(X, Z)}, \;
  C_{3} = C_{(X, T)}, \;
  C_{4} = C_{(Y, Z)}, \;
  C_{5} = C_{(Y, T)}, \;
  C_{6} = C_{(Z, T)}, \;
  C_{7} = C_{(T, X + Y + Z)}.
\end{gather*}
Their linear equivalence classes on the generic member \(\mathcal{S}_{\Bbbk}\) of the pencil satisfy the following relations:
\begin{gather*}
  \begin{pmatrix}
    [C_{2}] \\ [C_{5}] \\ [C_{6}] \\ [C_{7}]
  \end{pmatrix} = 
  \begin{pmatrix}
    -1 & -2 & 0 & 1 \\
    -2 & 0 & -1 & 1 \\
    1 & 2 & -2 & 0 \\
    1 & -3 & 3 & 0
  \end{pmatrix} \cdot
  \begin{pmatrix}
    [C_{1}] \\ [C_{3}] \\ [C_{4}] \\ [H_{\mathcal{S}}]
  \end{pmatrix}.
\end{gather*}

For a general choice of \(\lambda \in \mathbb{C}\) the surface \(\mathcal{S}_{\lambda}\) has the following singularities:
\begin{itemize}\setlength{\itemindent}{2cm}
\item[\(P_{1} = P_{(X, Y, Z)}\):] type \(\mathbb{A}_4\) with the quadratic term \(Y \cdot Z\);
\item[\(P_{2} = P_{(X, Y, T)}\):] type \(\mathbb{A}_4\) with the quadratic term \(X \cdot Y\);
\item[\(P_{3} = P_{(X, Z, T)}\):] type \(\mathbb{A}_2\) with the quadratic term \(X \cdot (Z + T)\);
\item[\(P_{4} = P_{(Y, Z, T)}\):] type \(\mathbb{A}_2\) with the quadratic term \(Z \cdot (Y + T)\);
\item[\(P_{5} = P_{(X, T, Y + Z)}\):] type \(\mathbb{A}_1\) with the quadratic term \(X (X + Y + Z - (\lambda + 1) T) + T^2\).
\end{itemize}

Galois action on the lattice \(L_{\lambda}\) is trivial. The intersection matrix on \(L_{\lambda} = L_{\mathcal{S}}\) is represented by
\begin{table}[H]
  \begin{tabular}{|c||cccc|cccc|cc|cc|c|cccc|}
    \hline
    \(\bullet\) & \(E_1^1\) & \(E_1^2\) & \(E_1^3\) & \(E_1^4\) & \(E_2^1\) & \(E_2^2\) & \(E_2^3\) & \(E_2^4\) & \(E_3^1\) & \(E_3^2\) & \(E_4^1\) & \(E_4^2\) & \(E_5^1\) & \(\widetilde{C_{1}}\) & \(\widetilde{C_{3}}\) & \(\widetilde{C_{4}}\) & \(\widetilde{H_{\mathcal{S}}}\) \\
    \hline
    \hline
    \(\widetilde{C_{1}}\) & \(1\) & \(0\) & \(0\) & \(0\) & \(0\) & \(0\) & \(1\) & \(0\) & \(0\) & \(0\) & \(0\) & \(0\) & \(0\) & \(-2\) & \(0\) & \(0\) & \(1\) \\
    \(\widetilde{C_{3}}\) & \(0\) & \(0\) & \(0\) & \(0\) & \(1\) & \(0\) & \(0\) & \(0\) & \(1\) & \(0\) & \(0\) & \(0\) & \(1\) & \(0\) & \(-2\) & \(0\) & \(1\) \\
    \(\widetilde{C_{4}}\) & \(0\) & \(0\) & \(1\) & \(0\) & \(0\) & \(0\) & \(0\) & \(0\) & \(0\) & \(0\) & \(1\) & \(0\) & \(0\) & \(0\) & \(0\) & \(-2\) & \(1\) \\
    \(\widetilde{H_{\mathcal{S}}}\) & \(0\) & \(0\) & \(0\) & \(0\) & \(0\) & \(0\) & \(0\) & \(0\) & \(0\) & \(0\) & \(0\) & \(0\) & \(0\) & \(1\) & \(1\) & \(1\) & \(4\) \\
    \hline
  \end{tabular}.
\end{table}
Note that the intersection matrix is non-degenerate.

Discriminant groups and discriminant forms of the lattices \(L_{\mathcal{S}}\) and \(H \oplus \Pic(X)\) are given by
\begin{gather*}
  G' = 
  \begin{pmatrix}
    0 & \frac{3}{14} & \frac{3}{7} & \frac{5}{7} & \frac{2}{7} & \frac{6}{7} & \frac{3}{7} & \frac{3}{14} & \frac{1}{7} & \frac{4}{7} & \frac{2}{7} & \frac{9}{14} & \frac{6}{7} & \frac{11}{14} & \frac{5}{7} & \frac{13}{14} & \frac{1}{7}
  \end{pmatrix}, \\
  G'' = 
  \begin{pmatrix}
    0 & 0 & -\frac{2}{7} & -\frac{9}{14} & \frac{1}{14}
  \end{pmatrix}; \;
  B' = 
  \begin{pmatrix}
    \frac{9}{14}
  \end{pmatrix}, \;
  B'' = 
  \begin{pmatrix}
    \frac{5}{14}
  \end{pmatrix}; \;
  Q' =
  \begin{pmatrix}
    \frac{9}{14}
  \end{pmatrix}, \;
  Q'' =
  \begin{pmatrix}
    \frac{19}{14}
  \end{pmatrix}.
\end{gather*}


\subsection{Family \textnumero3.31}\label{subsection:03-31}

The pencil \(\mathcal{S}\) is defined by the equation
\begin{gather*}
  X^{2} Y Z + X Y^{2} Z + X Y Z^{2} + X^{2} Y T + X^{2} Z T + Y Z T^{2} = \lambda X Y Z T.
\end{gather*}
Members \(\mathcal{S}_{\lambda}\) of the pencil are irreducible for any \(\lambda \in \mathbb{P}^1\) except
\(\mathcal{S}_{\infty} = S_{(X)} + S_{(Y)} + S_{(Z)} + S_{(T)}\).
The base locus of the pencil \(\mathcal{S}\) consists of the following curves:
\begin{gather*}
  C_{1} = C_{(X, Y)}, \;
  C_{2} = C_{(X, Z)}, \;
  C_{3} = C_{(X, T)}, \;
  C_{4} = C_{(Y, Z)}, \;
  C_{5} = C_{(Y, T)}, \;
  C_{6} = C_{(Z, T)}, \;
  C_{7} = C_{(T, X + Y + Z)}.
\end{gather*}
Their linear equivalence classes on the generic member \(\mathcal{S}_{\Bbbk}\) of the pencil satisfy the following relations:
\begin{gather*}
  \begin{pmatrix}
    [C_{2}] \\ [C_{5}] \\ [C_{6}] \\ [C_{7}]
  \end{pmatrix} = 
  \begin{pmatrix}
    -1 & -2 & 0 & 1 \\
    -2 & 0 & -1 & 1 \\
    2 & 4 & -1 & -1 \\
    0 & -5 & 2 & 1
  \end{pmatrix} \cdot
  \begin{pmatrix}
    [C_{1}] \\ [C_{3}] \\ [C_{4}] \\ [H_{\mathcal{S}}]
  \end{pmatrix}.
\end{gather*}

For a general choice of \(\lambda \in \mathbb{C}\) the surface \(\mathcal{S}_{\lambda}\) has the following singularities:
\begin{itemize}\setlength{\itemindent}{2cm}
\item[\(P_{1} = P_{(X, Y, Z)}\):] type \(\mathbb{A}_3\) with the quadratic term \(Y \cdot Z\);
\item[\(P_{2} = P_{(X, Y, T)}\):] type \(\mathbb{A}_4\) with the quadratic term \(X \cdot Y\);
\item[\(P_{3} = P_{(X, Z, T)}\):] type \(\mathbb{A}_4\) with the quadratic term \(X \cdot Z\);
\item[\(P_{4} = P_{(Y, Z, T)}\):] type \(\mathbb{A}_1\) with the quadratic term \(Y (Z + T) + Z T\);
\item[\(P_{5} = P_{(X, T, Y + Z)}\):] type \(\mathbb{A}_1\) with the quadratic term \(X (X + Y + Z - \lambda T) + T^2\).
\end{itemize}

Galois action on the lattice \(L_{\lambda}\) is trivial. The intersection matrix on \(L_{\lambda} = L_{\mathcal{S}}\) is represented by
\begin{table}[H]
  \begin{tabular}{|c||ccc|cccc|cccc|c|c|cccc|}
    \hline
    \(\bullet\) & \(E_1^1\) & \(E_1^2\) & \(E_1^3\) & \(E_2^1\) & \(E_2^2\) & \(E_2^3\) & \(E_2^4\) & \(E_3^1\) & \(E_3^2\) & \(E_3^3\) & \(E_3^4\) & \(E_4^1\) & \(E_5^1\) & \(\widetilde{C_{1}}\) & \(\widetilde{C_{3}}\) & \(\widetilde{C_{4}}\) & \(\widetilde{H_{\mathcal{S}}}\) \\
    \hline
    \hline
    \(\widetilde{C_{1}}\) & \(1\) & \(0\) & \(0\) & \(0\) & \(0\) & \(1\) & \(0\) & \(0\) & \(0\) & \(0\) & \(0\) & \(0\) & \(0\) & \(-2\) & \(0\) & \(0\) & \(1\) \\
    \(\widetilde{C_{3}}\) & \(0\) & \(0\) & \(0\) & \(1\) & \(0\) & \(0\) & \(0\) & \(1\) & \(0\) & \(0\) & \(0\) & \(0\) & \(1\) & \(0\) & \(-2\) & \(0\) & \(1\) \\
    \(\widetilde{C_{4}}\) & \(0\) & \(1\) & \(0\) & \(0\) & \(0\) & \(0\) & \(0\) & \(0\) & \(0\) & \(0\) & \(0\) & \(1\) & \(0\) & \(0\) & \(0\) & \(-2\) & \(1\) \\
    \(\widetilde{H_{\mathcal{S}}}\) & \(0\) & \(0\) & \(0\) & \(0\) & \(0\) & \(0\) & \(0\) & \(0\) & \(0\) & \(0\) & \(0\) & \(0\) & \(0\) & \(1\) & \(1\) & \(1\) & \(4\) \\
    \hline
  \end{tabular}.
\end{table}
Note that the intersection matrix is non-degenerate.

Discriminant groups and discriminant forms of the lattices \(L_{\mathcal{S}}\) and \(H \oplus \Pic(X)\) are given by
\begin{gather*}
  G' = 
  \begin{pmatrix}
    \frac{1}{4} & \frac{1}{6} & \frac{7}{12} & 0 & \frac{1}{6} & \frac{1}{3} & \frac{1}{6} & \frac{2}{3} & \frac{1}{2} & \frac{1}{3} & \frac{1}{6} & \frac{3}{4} & \frac{11}{12} & \frac{1}{3} & \frac{5}{6} & \frac{1}{2} & \frac{1}{12}
  \end{pmatrix}, \\
  G'' = 
  \begin{pmatrix}
    0 & 0 & \frac{1}{4} & -\frac{1}{4} & \frac{1}{6}
  \end{pmatrix}; \;
  B' = 
  \begin{pmatrix}
    \frac{1}{12}
  \end{pmatrix}, \;
  B'' = 
  \begin{pmatrix}
    \frac{11}{12}
  \end{pmatrix}; \;
  Q' =
  \begin{pmatrix}
    \frac{1}{12}
  \end{pmatrix}, \;
  Q'' =
  \begin{pmatrix}
    \frac{23}{12}
  \end{pmatrix}.
\end{gather*}


\section{Dolgachev--Nikulin duality for Fano threefolds: rank 4}\label{appendix:rank-04}
\subsection{Family \textnumero4.1}\label{subsection:04-01}

The pencil \(\mathcal{S}\) is defined by the equation
\begin{gather*}
  X^{2} Y Z + X Y^{2} Z + X Y Z^{2} + X Y^{2} T + Y^{2} Z T + X Z^{2} T + Y Z^{2} T + \\ X Y T^{2} + Y^{2} T^{2} + X Z T^{2} + 3 Y Z T^{2} + Z^{2} T^{2} + Y T^{3} + Z T^{3} = \lambda X Y Z T.
\end{gather*}
Members \(\mathcal{S}_{\lambda}\) of the pencil are irreducible for any \(\lambda \in \mathbb{P}^1\) except
\begin{gather*}
  \mathcal{S}_{\infty} = S_{(X)} + S_{(Y)} + S_{(Z)} + S_{(T)}, \;
  \mathcal{S}_{- 4} = S_{(X + T)} + S_{(X Y Z + (Y + Z + T) (Y Z + Y T + Z T))}.
\end{gather*}
The base locus of the pencil \(\mathcal{S}\) consists of the following curves:
\begin{gather*}
  C_{1} = C_{(X, T)}, \;
  C_{2} = C_{(Y, Z)}, \;
  C_{3} = C_{(Y, T)}, \;
  C_{4} = C_{(Z, T)}, \;
  C_{5} = C_{(Y, X + T)}, \;
  C_{6} = C_{(Y, Z + T)}, \\
  C_{7} = C_{(Z, X + T)}, \;
  C_{8} = C_{(Z, Y + T)}, \;
  C_{9} = C_{(X, Y + Z + T)}, \;
  C_{10} = C_{(T, X + Y + Z)}, \;
  C_{11} = C_{(X, Y Z + T (Y + Z))}.
\end{gather*}
Their linear equivalence classes on the generic member \(\mathcal{S}_{\Bbbk}\) of the pencil satisfy the following relations:
\begin{gather*}
  \begin{pmatrix}
    [C_{6}] \\ [C_{7}] \\ [C_{8}] \\ [C_{10}] \\ [C_{11}]
  \end{pmatrix} = 
  \begin{pmatrix}
    0 & -1 & -1 & 0 & -1 & 0 & 1 \\
    -2 & 0 & 0 & 0 & -1 & 0 & 1 \\
    2 & -1 & 0 & -1 & 1 & 0 & 0 \\
    -1 & 0 & -1 & -1 & 0 & 0 & 1 \\
    -1 & 0 & 0 & 0 & 0 & -1 & 1
  \end{pmatrix} \cdot
  \begin{pmatrix}
    [C_{1}] & [C_{2}] & [C_{3}] & [C_{4}] & [C_{5}] & [C_{9}] & [H_{\mathcal{S}}]
  \end{pmatrix}^T.
\end{gather*}

For a general choice of \(\lambda \in \mathbb{C}\) the surface \(\mathcal{S}_{\lambda}\) has the following singularities:
\begin{itemize}\setlength{\itemindent}{2cm}
\item[\(P_{1} = P_{(X, Y, T)}\):] type \(\mathbb{A}_2\) with the quadratic term \((X + T) \cdot (Y + T)\);
\item[\(P_{2} = P_{(X, Z, T)}\):] type \(\mathbb{A}_2\) with the quadratic term \((X + T) \cdot (Z + T)\);
\item[\(P_{3} = P_{(Y, Z, T)}\):] type \(\mathbb{A}_3\) with the quadratic term \(Y \cdot Z\);
\item[\(P_{4} = P_{(X, T, Y + Z)}\):] type \(\mathbb{A}_1\) with the quadratic term \((X + T) (X + Y + Z + T) - (\lambda + 4) X T\);
\item[\(P_{5} = P_{(Y, Z, X + T)}\):] type \(\mathbb{A}_1\) with the quadratic term \((X + T) (Y + Z) + (\lambda + 4) Y Z\).
\end{itemize}

Galois action on the lattice \(L_{\lambda}\) is trivial. The intersection matrix on \(L_{\lambda} = L_{\mathcal{S}}\) is represented by
\begin{table}[H]
  \begin{tabular}{|c||cc|cc|ccc|c|c|ccccccc|}
    \hline
    \(\bullet\) & \(E_1^1\) & \(E_1^2\) & \(E_2^1\) & \(E_2^2\) & \(E_3^1\) & \(E_3^2\) & \(E_3^3\) & \(E_4^1\) & \(E_5^1\) & \(\widetilde{C_{1}}\) & \(\widetilde{C_{2}}\) & \(\widetilde{C_{3}}\) & \(\widetilde{C_{4}}\) & \(\widetilde{C_{5}}\) & \(\widetilde{C_{9}}\) & \(\widetilde{H_{\mathcal{S}}}\) \\
    \hline
    \hline
    \(\widetilde{C_{1}}\) & \(1\) & \(0\) & \(1\) & \(0\) & \(0\) & \(0\) & \(0\) & \(1\) & \(0\) & \(-2\) & \(0\) & \(0\) & \(0\) & \(0\) & \(0\) & \(1\) \\
    \(\widetilde{C_{2}}\) & \(0\) & \(0\) & \(0\) & \(0\) & \(0\) & \(1\) & \(0\) & \(0\) & \(1\) & \(0\) & \(-2\) & \(0\) & \(0\) & \(0\) & \(0\) & \(1\) \\
    \(\widetilde{C_{3}}\) & \(0\) & \(1\) & \(0\) & \(0\) & \(1\) & \(0\) & \(0\) & \(0\) & \(0\) & \(0\) & \(0\) & \(-2\) & \(0\) & \(0\) & \(0\) & \(1\) \\
    \(\widetilde{C_{4}}\) & \(0\) & \(0\) & \(0\) & \(1\) & \(0\) & \(0\) & \(1\) & \(0\) & \(0\) & \(0\) & \(0\) & \(0\) & \(-2\) & \(0\) & \(0\) & \(1\) \\
    \(\widetilde{C_{5}}\) & \(1\) & \(0\) & \(0\) & \(0\) & \(0\) & \(0\) & \(0\) & \(0\) & \(1\) & \(0\) & \(0\) & \(0\) & \(0\) & \(-2\) & \(0\) & \(1\) \\
    \(\widetilde{C_{9}}\) & \(0\) & \(0\) & \(0\) & \(0\) & \(0\) & \(0\) & \(0\) & \(1\) & \(0\) & \(0\) & \(0\) & \(0\) & \(0\) & \(0\) & \(-2\) & \(1\) \\
    \(\widetilde{H_{\mathcal{S}}}\) & \(0\) & \(0\) & \(0\) & \(0\) & \(0\) & \(0\) & \(0\) & \(0\) & \(0\) & \(1\) & \(1\) & \(1\) & \(1\) & \(1\) & \(1\) & \(4\) \\
    \hline
  \end{tabular}.
\end{table}
Note that the intersection matrix is non-degenerate.

Discriminant groups and discriminant forms of the lattices \(L_{\mathcal{S}}\) and \(H \oplus \Pic(X)\) are given by
\begin{gather*}
  G' = 
  \begin{pmatrix}
    0 & \frac{1}{2} & 0 & 0 & \frac{1}{2} & 0 & 0 & 0 & 0 & 0 & \frac{1}{2} & 0 & 0 & \frac{1}{2} & 0 & 0 \\
    0 & \frac{1}{2} & 0 & 0 & 0 & 0 & \frac{1}{2} & \frac{1}{2} & \frac{1}{2} & 0 & \frac{1}{2} & 0 & 0 & \frac{1}{2} & 0 & \frac{1}{2} \\
    \frac{1}{2} & 0 & \frac{1}{2} & 0 & \frac{1}{2} & \frac{1}{2} & \frac{1}{2} & \frac{1}{2} & 0 & 0 & 0 & \frac{1}{2} & \frac{1}{2} & 0 & 0 & \frac{1}{2} \\
    \frac{1}{3} & \frac{5}{6} & 0 & \frac{5}{6} & \frac{2}{3} & 0 & \frac{1}{3} & \frac{5}{6} & \frac{5}{6} & \frac{1}{6} & 0 & \frac{1}{3} & \frac{2}{3} & \frac{2}{3} & \frac{1}{2} & \frac{1}{6}
  \end{pmatrix}, \;
  B' = 
  \begin{pmatrix}
    0 & \frac{1}{2} & \frac{1}{2} & \frac{1}{2} \\
    \frac{1}{2} & 0 & \frac{1}{2} & \frac{1}{2} \\
    \frac{1}{2} & \frac{1}{2} & 0 & \frac{1}{2} \\
    \frac{1}{2} & \frac{1}{2} & \frac{1}{2} & \frac{1}{3}
  \end{pmatrix}; \\
  G'' = 
  \begin{pmatrix}
    0 & 0 & \frac{1}{2} & 0 & 0 & 0 \\
    0 & 0 & 0 & \frac{1}{2} & 0 & 0 \\
    0 & 0 & 0 & 0 & \frac{1}{2} & 0 \\
    0 & 0 & -\frac{1}{3} & -\frac{1}{3} & -\frac{1}{3} & \frac{1}{6}
  \end{pmatrix}, \;
  B'' = 
  \begin{pmatrix}
    0 & \frac{1}{2} & \frac{1}{2} & \frac{1}{2} \\
    \frac{1}{2} & 0 & \frac{1}{2} & \frac{1}{2} \\
    \frac{1}{2} & \frac{1}{2} & 0 & \frac{1}{2} \\
    \frac{1}{2} & \frac{1}{2} & \frac{1}{2} & \frac{2}{3}
  \end{pmatrix}; \;
  \begin{pmatrix}
    Q' \\ Q''
  \end{pmatrix}
  =
  \begin{pmatrix}
    0 & 0 & 0 & \frac{4}{3} \\
    0 & 0 & 0 & \frac{2}{3}
  \end{pmatrix}.
\end{gather*}


\subsection{Family \textnumero4.2}\label{subsection:04-02}

The pencil \(\mathcal{S}\) is defined by the equation
\begin{gather*}
  X^{2} Y Z + X Y^{2} Z + X^{2} Z T + Y^{2} Z T + X Z^{2} T + Y Z^{2} T + 2 X Z T^{2} + 2 Y Z T^{2} + X T^{3} + Y T^{3} = \lambda X Y Z T.
\end{gather*}
Members \(\mathcal{S}_{\lambda}\) of the pencil are irreducible for any \(\lambda \in \mathbb{P}^1\) except
\begin{gather*}
  \mathcal{S}_{\infty} = S_{(X)} + S_{(Y)} + S_{(Z)} + S_{(T)}, \;
  \mathcal{S}_{- 2} = S_{(X + Y)} + S_{(Y Z T + X Z (Y + T) + (Z + T)^2 T)}.
\end{gather*}
The base locus of the pencil \(\mathcal{S}\) consists of the following curves:
\begin{gather*}
  C_{1} = C_{(X, Y)}, \;
  C_{2} = C_{(X, T)}, \;
  C_{3} = C_{(Y, T)}, \;
  C_{4} = C_{(Z, T)}, \;
  C_{5} = C_{(Z, X + Y)}, \\
  C_{6} = C_{(T, X + Y)}, \;
  C_{7} = C_{(X, Y Z + (Z + T)^2)}, \;
  C_{8} = C_{(Y, X Z + (Z + T)^2)}.
\end{gather*}
Their linear equivalence classes on the generic member \(\mathcal{S}_{\Bbbk}\) of the pencil satisfy the following relations:
\begin{gather*}
  \begin{pmatrix}
    [C_{3}] \\ [C_{5}] \\ [C_{6}] \\ [C_{7}] \\ [C_{8}]
  \end{pmatrix} = 
  \begin{pmatrix}
    2 & -1 & -4 & 1 \\
    0 & 0 & -3 & 1 \\
    -2 & 0 & 3 & 0 \\
    -1 & -1 & 0 & 1 \\
    -3 & 1 & 4 & 0
  \end{pmatrix} \cdot
  \begin{pmatrix}
    [C_{1}] \\ [C_{2}] \\ [C_{4}] \\ [H_{\mathcal{S}}]
  \end{pmatrix}.
\end{gather*}

Put \(\mu (\mu - 1) = (\lambda - 2)^{-1}\). For a general choice of \(\lambda \in \mathbb{C}\) the surface \(\mathcal{S}_{\lambda}\) has the following singularities:
\begin{itemize}\setlength{\itemindent}{2cm}
\item[\(P_{1} = P_{(X, Y, T)}\):] type \(\mathbb{A}_3\) with the quadratic term \(T \cdot (X + Y)\);
\item[\(P_{2} = P_{(X, Z, T)}\):] type \(\mathbb{A}_2\) with the quadratic term \(Z \cdot (X + T)\);
\item[\(P_{3} = P_{(Y, Z, T)}\):] type \(\mathbb{A}_2\) with the quadratic term \(Z \cdot (Y + T)\);
\item[\(P_{4} = P_{(X, Y, Z + T)}\):] type \(\mathbb{A}_3\) with the quadratic term \((\mu X - (\mu - 1) Y) \cdot ((\mu - 1) X - \mu Y)\);
\item[\(P_{5} = P_{(Z, T, X + Y)}\):] type \(\mathbb{A}_3\) with the quadratic term \(Z \cdot (X + Y - (\lambda + 2) T)\).
\end{itemize}

The only non-trivial Galois orbit on the lattice \(L_{\lambda}\) is \(E_4^1 + E_4^3\).

The intersection matrix on the lattice \(L_{\lambda}\) is represented by
\begin{table}[H]
  \begin{tabular}{|c||ccc|cc|cc|ccc|ccc|cccc|}
    \hline
    \(\bullet\) & \(E_1^1\) & \(E_1^2\) & \(E_1^3\) & \(E_2^1\) & \(E_2^2\) & \(E_3^1\) & \(E_3^2\) & \(E_4^1\) & \(E_4^2\) & \(E_4^3\) & \(E_5^1\) & \(E_5^2\) & \(E_5^3\) & \(\widetilde{C_{1}}\) & \(\widetilde{C_{2}}\) & \(\widetilde{C_{4}}\) & \(\widetilde{H_{\mathcal{S}}}\) \\
    \hline
    \hline
    \(\widetilde{C_{1}}\) & \(1\) & \(0\) & \(0\) & \(0\) & \(0\) & \(0\) & \(0\) & \(0\) & \(1\) & \(0\) & \(0\) & \(0\) & \(0\) & \(-2\) & \(0\) & \(0\) & \(1\) \\
    \(\widetilde{C_{2}}\) & \(0\) & \(0\) & \(1\) & \(1\) & \(0\) & \(0\) & \(0\) & \(0\) & \(0\) & \(0\) & \(0\) & \(0\) & \(0\) & \(0\) & \(-2\) & \(0\) & \(1\) \\
    \(\widetilde{C_{4}}\) & \(0\) & \(0\) & \(0\) & \(0\) & \(1\) & \(1\) & \(0\) & \(0\) & \(0\) & \(0\) & \(1\) & \(0\) & \(0\) & \(0\) & \(0\) & \(-2\) & \(1\) \\
    \(\widetilde{H_{\mathcal{S}}}\) & \(0\) & \(0\) & \(0\) & \(0\) & \(0\) & \(0\) & \(0\) & \(0\) & \(0\) & \(0\) & \(0\) & \(0\) & \(0\) & \(1\) & \(1\) & \(1\) & \(4\) \\
    \hline
  \end{tabular}.
\end{table}
Note that the intersection matrix is non-degenerate.

Discriminant groups and discriminant forms of the lattices \(L_{\mathcal{S}}\) and \(H \oplus \Pic(X)\) are given by
\begin{gather*}
  G' = 
  \begin{pmatrix}
    \frac{1}{4} & \frac{1}{2} & \frac{3}{4} & 0 & 0 & 0 & 0 & 0 & \frac{1}{2} & \frac{3}{4} & \frac{1}{2} & \frac{1}{4} & 0 & 0 & 0 & \frac{1}{4} \\
    \frac{3}{8} & \frac{1}{2} & \frac{5}{8} & 0 & \frac{1}{4} & 0 & \frac{1}{2} & \frac{5}{8} & \frac{1}{4} & \frac{7}{8} & \frac{1}{4} & \frac{5}{8} & \frac{1}{4} & \frac{3}{4} & \frac{1}{2} & \frac{7}{8}
  \end{pmatrix}, \\
  G'' = 
  \begin{pmatrix}
    0 & 0 & -\frac{1}{4} & -\frac{1}{4} & \frac{1}{4} & 0 \\
    0 & 0 & -\frac{3}{8} & 0 & 0 & \frac{1}{4}
  \end{pmatrix}; \;
  B' = 
  \begin{pmatrix}
    \frac{1}{4} & \frac{1}{4} \\
    \frac{1}{4} & \frac{3}{8}
  \end{pmatrix}, \;
  B'' = 
  \begin{pmatrix}
    \frac{3}{4} & \frac{3}{4} \\
    \frac{3}{4} & \frac{5}{8}
  \end{pmatrix}; \;
  \begin{pmatrix}
    Q' \\ Q''
  \end{pmatrix}
  =
  \begin{pmatrix}
    \frac{1}{4} & \frac{3}{8} \\
    \frac{7}{4} & \frac{13}{8}
  \end{pmatrix}.
\end{gather*}


\subsection{Family \textnumero4.3}\label{subsection:04-03}

The pencil \(\mathcal{S}\) is defined by the equation
\begin{gather*}
  X^{2} Y Z + X Y^{2} Z + X Y Z^{2} + X Y^{2} T + Y^{2} Z T + X Z^{2} T + Y Z^{2} T + X Y T^{2} + X Z T^{2} + Y Z T^{2} = \lambda X Y Z T.
\end{gather*}
Members \(\mathcal{S}_{\lambda}\) of the pencil are irreducible for any \(\lambda \in \mathbb{P}^1\) except
\(\mathcal{S}_{\infty} = S_{(X)} + S_{(Y)} + S_{(Z)} + S_{(T)}\).
The base locus of the pencil \(\mathcal{S}\) consists of the following curves:
\begin{gather*}
  C_{1} = C_{(X, Y)}, \;
  C_{2} = C_{(X, Z)}, \;
  C_{3} = C_{(X, T)}, \;
  C_{4} = C_{(Y, Z)}, \;
  C_{5} = C_{(Y, T)}, \;
  C_{6} = C_{(Z, T)}, \\
  C_{7} = C_{(Y, Z + T)}, \;
  C_{8} = C_{(Z, Y + T)}, \;
  C_{9} = C_{(X, Y + Z + T)}, \;
  C_{10} = C_{(T, X + Y + Z)}.
\end{gather*}
Their linear equivalence classes on the generic member \(\mathcal{S}_{\Bbbk}\) of the pencil satisfy the following relations:
\begin{gather*}
  \begin{pmatrix}
    [C_{7}] \\ [C_{8}] \\ [C_{9}] \\ [C_{10}]
  \end{pmatrix} = 
  \begin{pmatrix}
    -1 & 0 & 0 & -1 & -1 & 0 & 1 \\
    0 & -1 & 0 & -1 & 0 & -1 & 1 \\
    -1 & -1 & -1 & 0 & 0 & 0 & 1 \\
    0 & 0 & -1 & 0 & -1 & -1 & 1
  \end{pmatrix} \cdot
  \begin{pmatrix}
    [C_{1}] & [C_{2}] & [C_{3}] & [C_{4}] & [C_{5}] & [C_{6}] & [H_{\mathcal{S}}]
  \end{pmatrix}^T.
\end{gather*}

For a general choice of \(\lambda \in \mathbb{C}\) the surface \(\mathcal{S}_{\lambda}\) has the following singularities:
\begin{itemize}\setlength{\itemindent}{2cm}
\item[\(P_{1} = P_{(X, Y, Z)}\):] type \(\mathbb{A}_1\) with the quadratic term \(X (Y + Z) + Y Z\);
\item[\(P_{2} = P_{(X, Y, T)}\):] type \(\mathbb{A}_1\) with the quadratic term \(X (Y + T) + Y T\);
\item[\(P_{3} = P_{(X, Z, T)}\):] type \(\mathbb{A}_1\) with the quadratic term \(X (Z + T) + Z T\);
\item[\(P_{4} = P_{(Y, Z, T)}\):] type \(\mathbb{A}_3\) with the quadratic term \(Y \cdot Z\);
\item[\(P_{5} = P_{(X, Y, Z + T)}\):] type \(\mathbb{A}_1\) with the quadratic term \((X + Y) (Y + Z + T) - (\lambda + 3) X Y\);
\item[\(P_{6} = P_{(X, Z, Y + T)}\):] type \(\mathbb{A}_1\) with the quadratic term \((X + Z) (Y + Z + T) - (\lambda + 3) X Z\);
\item[\(P_{7} = P_{(X, T, Y + Z)}\):] type \(\mathbb{A}_1\) with the quadratic term \((X + T) (X + Y + Z + T) - (\lambda + 4) X T\).
\end{itemize}

Galois action on the lattice \(L_{\lambda}\) is trivial. The intersection matrix on \(L_{\lambda} = L_{\mathcal{S}}\) is represented by
\begin{table}[H]
  \begin{tabular}{|c||c|c|c|ccc|c|c|c|ccccccc|}
    \hline
    \(\bullet\) & \(E_1^1\) & \(E_2^1\) & \(E_3^1\) & \(E_4^1\) & \(E_4^2\) & \(E_4^3\) & \(E_5^1\) & \(E_6^1\) & \(E_7^1\) & \(\widetilde{C_{1}}\) & \(\widetilde{C_{2}}\) & \(\widetilde{C_{3}}\) & \(\widetilde{C_{4}}\) & \(\widetilde{C_{5}}\) & \(\widetilde{C_{6}}\) & \(\widetilde{H_{\mathcal{S}}}\) \\
    \hline
    \hline
    \(\widetilde{C_{1}}\) & \(1\) & \(1\) & \(0\) & \(0\) & \(0\) & \(0\) & \(1\) & \(0\) & \(0\) & \(-2\) & \(0\) & \(0\) & \(0\) & \(0\) & \(0\) & \(1\) \\
    \(\widetilde{C_{2}}\) & \(1\) & \(0\) & \(1\) & \(0\) & \(0\) & \(0\) & \(0\) & \(1\) & \(0\) & \(0\) & \(-2\) & \(0\) & \(0\) & \(0\) & \(0\) & \(1\) \\
    \(\widetilde{C_{3}}\) & \(0\) & \(1\) & \(1\) & \(0\) & \(0\) & \(0\) & \(0\) & \(0\) & \(1\) & \(0\) & \(0\) & \(-2\) & \(0\) & \(0\) & \(0\) & \(1\) \\
    \(\widetilde{C_{4}}\) & \(1\) & \(0\) & \(0\) & \(0\) & \(1\) & \(0\) & \(0\) & \(0\) & \(0\) & \(0\) & \(0\) & \(0\) & \(-2\) & \(0\) & \(0\) & \(1\) \\
    \(\widetilde{C_{5}}\) & \(0\) & \(1\) & \(0\) & \(1\) & \(0\) & \(0\) & \(0\) & \(0\) & \(0\) & \(0\) & \(0\) & \(0\) & \(0\) & \(-2\) & \(0\) & \(1\) \\
    \(\widetilde{C_{6}}\) & \(0\) & \(0\) & \(1\) & \(0\) & \(0\) & \(1\) & \(0\) & \(0\) & \(0\) & \(0\) & \(0\) & \(0\) & \(0\) & \(0\) & \(-2\) & \(1\) \\
    \(\widetilde{H_{\mathcal{S}}}\) & \(0\) & \(0\) & \(0\) & \(0\) & \(0\) & \(0\) & \(0\) & \(0\) & \(0\) & \(1\) & \(1\) & \(1\) & \(1\) & \(1\) & \(1\) & \(4\) \\
    \hline
  \end{tabular}.
\end{table}
Note that the intersection matrix is non-degenerate.

Discriminant groups and discriminant forms of the lattices \(L_{\mathcal{S}}\) and \(H \oplus \Pic(X)\) are given by
\begin{gather*}
  G' = 
  \begin{pmatrix}
    0 & \frac{1}{4} & \frac{3}{4} & \frac{1}{4} & \frac{1}{2} & \frac{3}{4} & \frac{1}{4} & \frac{3}{4} & \frac{1}{2} & \frac{1}{2} & \frac{1}{2} & 0 & 0 & 0 & 0 & \frac{1}{2} \\
    \frac{1}{4} & \frac{1}{2} & \frac{1}{2} & \frac{1}{12} & 0 & \frac{5}{12} & \frac{1}{6} & \frac{5}{6} & \frac{1}{4} & \frac{1}{3} & \frac{2}{3} & \frac{1}{2} & \frac{1}{2} & \frac{1}{6} & \frac{5}{6} & \frac{3}{4}
  \end{pmatrix}, \\
  G'' = 
  \begin{pmatrix}
    0 & 0 & -1 & -\frac{7}{4} & \frac{1}{4} & \frac{1}{4} \\
    0 & 0 & \frac{1}{6} & \frac{5}{12} & -\frac{1}{12} & 0
  \end{pmatrix}; \;
  B' = 
  \begin{pmatrix}
    \frac{1}{4} & \frac{1}{2} \\
    \frac{1}{2} & \frac{1}{12}
  \end{pmatrix}, \;
  B'' = 
  \begin{pmatrix}
    \frac{3}{4} & \frac{1}{2} \\
    \frac{1}{2} & \frac{11}{12}
  \end{pmatrix}; \;
  \begin{pmatrix}
    Q' \\ Q''
  \end{pmatrix}
  =
  \begin{pmatrix}
    \frac{5}{4} & \frac{1}{12} \\
    \frac{3}{4} & \frac{23}{12}
  \end{pmatrix}.
\end{gather*}


\subsection{Family \textnumero4.4}\label{subsection:04-04}

The pencil \(\mathcal{S}\) is defined by the equation
\begin{gather*}
  X^{2} Y Z + X Y^{2} Z + X Y Z^{2} + X^{2} Y T + X Y^{2} T + X^{2} Z T + Y^{2} Z T + X Z T^{2} + Y Z T^{2} = \lambda X Y Z T.
\end{gather*}
Members \(\mathcal{S}_{\lambda}\) of the pencil are irreducible for any \(\lambda \in \mathbb{P}^1\) except
\(\mathcal{S}_{\infty} = S_{(X)} + S_{(Y)} + S_{(Z)} + S_{(T)}\).
The base locus of the pencil \(\mathcal{S}\) consists of the following curves:
\begin{gather*}
  C_{1} = C_{(X, Y)}, \;
  C_{2} = C_{(X, Z)}, \;
  C_{3} = C_{(X, T)}, \;
  C_{4} = C_{(Y, Z)}, \;
  C_{5} = C_{(Y, T)}, \;
  C_{6} = C_{(Z, T)}, \\
  C_{7} = C_{(X, Y + T)}, \;
  C_{8} = C_{(Y, X + T)}, \;
  C_{9} = C_{(Z, X + Y)}, \;
  C_{10} = C_{(T, X + Y + Z)}.
\end{gather*}
Their linear equivalence classes on the generic member \(\mathcal{S}_{\Bbbk}\) of the pencil satisfy the following relations:
\begin{gather*}
  \begin{pmatrix}
    [C_{7}] \\ [C_{8}] \\ [C_{9}] \\ [C_{10}]
  \end{pmatrix} = 
  \begin{pmatrix}
    -1 & -1 & -1 & 0 & 0 & 0 & 1 \\
    -1 & 0 & 0 & -1 & -1 & 0 & 1 \\
    0 & -1 & 0 & -1 & 0 & -1 & 1 \\
    0 & 0 & -1 & 0 & -1 & -1 & 1
  \end{pmatrix} \cdot
  \begin{pmatrix}
    [C_{1}] & [C_{2}] & [C_{3}] & [C_{4}] & [C_{5}] & [C_{6}] & [H_{\mathcal{S}}]
  \end{pmatrix}^T.
\end{gather*}

For a general choice of \(\lambda \in \mathbb{C}\) the surface \(\mathcal{S}_{\lambda}\) has the following singularities:
\begin{itemize}\setlength{\itemindent}{2cm}
\item[\(P_{1} = P_{(X, Y, Z)}\):] type \(\mathbb{A}_3\) with the quadratic term \(Z \cdot (X + Y)\);
\item[\(P_{2} = P_{(X, Y, T)}\):] type \(\mathbb{A}_3\) with the quadratic term \(X \cdot Y\);
\item[\(P_{3} = P_{(X, Z, T)}\):] type \(\mathbb{A}_1\) with the quadratic term \(X (Z + T) + Z T\);
\item[\(P_{4} = P_{(Y, Z, T)}\):] type \(\mathbb{A}_1\) with the quadratic term \(Y (Z + T) + Z T\);
\item[\(P_{5} = P_{(Z, T, X + Y)}\):] type \(\mathbb{A}_1\) with the quadratic term \((Z + T) (X + Y + Z) - (\lambda + 3) Z T\).
\end{itemize}

Galois action on the lattice \(L_{\lambda}\) is trivial. The intersection matrix on \(L_{\lambda} = L_{\mathcal{S}}\) is represented by
\begin{table}[H]
  \begin{tabular}{|c||ccc|ccc|c|c|c|ccccccc|}
    \hline
    \(\bullet\) & \(E_1^1\) & \(E_1^2\) & \(E_1^3\) & \(E_2^1\) & \(E_2^2\) & \(E_2^3\) & \(E_3^1\) & \(E_4^1\) & \(E_5^1\) & \(\widetilde{C_{1}}\) & \(\widetilde{C_{2}}\) & \(\widetilde{C_{3}}\) & \(\widetilde{C_{4}}\) & \(\widetilde{C_{5}}\) & \(\widetilde{C_{6}}\) & \(\widetilde{H_{\mathcal{S}}}\) \\
    \hline
    \hline
    \(\widetilde{C_{1}}\) & \(1\) & \(0\) & \(0\) & \(0\) & \(1\) & \(0\) & \(0\) & \(0\) & \(0\) & \(-2\) & \(0\) & \(0\) & \(0\) & \(0\) & \(0\) & \(1\) \\
    \(\widetilde{C_{2}}\) & \(0\) & \(0\) & \(1\) & \(0\) & \(0\) & \(0\) & \(1\) & \(0\) & \(0\) & \(0\) & \(-2\) & \(0\) & \(0\) & \(0\) & \(0\) & \(1\) \\
    \(\widetilde{C_{3}}\) & \(0\) & \(0\) & \(0\) & \(1\) & \(0\) & \(0\) & \(1\) & \(0\) & \(0\) & \(0\) & \(0\) & \(-2\) & \(0\) & \(0\) & \(0\) & \(1\) \\
    \(\widetilde{C_{4}}\) & \(0\) & \(0\) & \(1\) & \(0\) & \(0\) & \(0\) & \(0\) & \(1\) & \(0\) & \(0\) & \(0\) & \(0\) & \(-2\) & \(0\) & \(0\) & \(1\) \\
    \(\widetilde{C_{5}}\) & \(0\) & \(0\) & \(0\) & \(0\) & \(0\) & \(1\) & \(0\) & \(1\) & \(0\) & \(0\) & \(0\) & \(0\) & \(0\) & \(-2\) & \(0\) & \(1\) \\
    \(\widetilde{C_{6}}\) & \(0\) & \(0\) & \(0\) & \(0\) & \(0\) & \(0\) & \(1\) & \(1\) & \(1\) & \(0\) & \(0\) & \(0\) & \(0\) & \(0\) & \(-2\) & \(1\) \\
    \(\widetilde{H_{\mathcal{S}}}\) & \(0\) & \(0\) & \(0\) & \(0\) & \(0\) & \(0\) & \(0\) & \(0\) & \(0\) & \(1\) & \(1\) & \(1\) & \(1\) & \(1\) & \(1\) & \(4\) \\
    \hline
  \end{tabular}.
\end{table}
Note that the intersection matrix is non-degenerate.

Discriminant groups and discriminant forms of the lattices \(L_{\mathcal{S}}\) and \(H \oplus \Pic(X)\) are given by
\begin{gather*}
  G' = 
  \begin{pmatrix}
    \frac{1}{2} & 0 & \frac{1}{2} & \frac{1}{2} & \frac{1}{2} & \frac{1}{2} & \frac{1}{2} & \frac{1}{2} & 0 & 0 & \frac{1}{2} & \frac{1}{2} & \frac{1}{2} & \frac{1}{2} & 0 & 0 \\
    \frac{1}{2} & \frac{1}{2} & \frac{1}{2} & \frac{3}{10} & \frac{1}{4} & \frac{7}{10} & \frac{3}{20} & \frac{7}{20} & \frac{1}{4} & \frac{1}{2} & \frac{9}{20} & \frac{7}{20} & \frac{1}{20} & \frac{3}{20} & \frac{1}{2} & \frac{1}{4}
  \end{pmatrix}, \\
  G'' = 
  \begin{pmatrix}
    0 & 0 & -2 & -\frac{7}{2} & -\frac{1}{2} & \frac{1}{2} \\
    0 & 0 & -\frac{3}{10} & -\frac{13}{20} & -\frac{3}{20} & \frac{1}{10}
  \end{pmatrix}; \;
  B' = 
  \begin{pmatrix}
    \frac{1}{2} & \frac{1}{2} \\
    \frac{1}{2} & \frac{13}{20}
  \end{pmatrix}, \;
  B'' = 
  \begin{pmatrix}
    \frac{1}{2} & \frac{1}{2} \\
    \frac{1}{2} & \frac{7}{20}
  \end{pmatrix}; \;
  \begin{pmatrix}
    Q' \\ Q''
  \end{pmatrix}
  =
  \begin{pmatrix}
    \frac{3}{2} & \frac{13}{20} \\
    \frac{1}{2} & \frac{27}{20}
  \end{pmatrix}.
\end{gather*}


\subsection{Family \textnumero4.5}\label{subsection:04-05}

The pencil \(\mathcal{S}\) is defined by the equation
\begin{gather*}
  X^{2} Y Z + X Y^{2} Z + X Y Z^{2} + X Y^{2} T + Y^{2} Z T + X Z^{2} T + X Y T^{2} + X Z T^{2} + Y Z T^{2} = \lambda X Y Z T.
\end{gather*}
Members \(\mathcal{S}_{\lambda}\) of the pencil are irreducible for any \(\lambda \in \mathbb{P}^1\) except
\(\mathcal{S}_{\infty} = S_{(X)} + S_{(Y)} + S_{(Z)} + S_{(T)}\).
The base locus of the pencil \(\mathcal{S}\) consists of the following curves:
\begin{gather*}
  C_{1} = C_{(X, Y)}, \;
  C_{2} = C_{(X, Z)}, \;
  C_{3} = C_{(X, T)}, \;
  C_{4} = C_{(Y, Z)}, \;
  C_{5} = C_{(Y, T)}, \;
  C_{6} = C_{(Z, T)}, \\
  C_{7} = C_{(X, Y + T)}, \;
  C_{8} = C_{(Y, Z + T)}, \;
  C_{9} = C_{(Z, Y + T)}, \;
  C_{10} = C_{(T, X + Y + Z)}.
\end{gather*}
Their linear equivalence classes on the generic member \(\mathcal{S}_{\Bbbk}\) of the pencil satisfy the following relations:
\begin{gather*}
  \begin{pmatrix}
    [C_{7}] \\ [C_{8}] \\ [C_{9}] \\ [C_{10}]
  \end{pmatrix} = 
  \begin{pmatrix}
    -1 & -1 & -1 & 0 & 0 & 0 & 1 \\
    -1 & 0 & 0 & -1 & -1 & 0 & 1 \\
    0 & -1 & 0 & -1 & 0 & -1 & 1 \\
    0 & 0 & -1 & 0 & -1 & -1 & 1
  \end{pmatrix} \cdot
  \begin{pmatrix}
    [C_{1}] & [C_{2}] & [C_{3}] & [C_{4}] & [C_{5}] & [C_{6}] & [H_{\mathcal{S}}]
  \end{pmatrix}^T.
\end{gather*}

For a general choice of \(\lambda \in \mathbb{C}\) the surface \(\mathcal{S}_{\lambda}\) has the following singularities:
\begin{itemize}\setlength{\itemindent}{2cm}
\item[\(P_{1} = P_{(X, Y, Z)}\):] type \(\mathbb{A}_1\) with the quadratic term \(Z (X + Y) + X Y\);
\item[\(P_{2} = P_{(X, Y, T)}\):] type \(\mathbb{A}_3\) with the quadratic term \(X \cdot (Y + T)\);
\item[\(P_{3} = P_{(X, Z, T)}\):] type \(\mathbb{A}_1\) with the quadratic term \(X (Z + T) + Z T\);
\item[\(P_{4} = P_{(Y, Z, T)}\):] type \(\mathbb{A}_3\) with the quadratic term \(Y \cdot Z\);
\item[\(P_{5} = P_{(X, Z, Y + T)}\):] type \(\mathbb{A}_1\) with the quadratic term \((X + Z) (Y + T) - (\lambda + 2) X Z\).
\end{itemize}

Galois action on the lattice \(L_{\lambda}\) is trivial. The intersection matrix on \(L_{\lambda} = L_{\mathcal{S}}\) is represented by
\begin{table}[H]
  \begin{tabular}{|c||c|ccc|c|ccc|c|ccccccc|}
    \hline
    \(\bullet\) & \(E_1^1\) & \(E_2^1\) & \(E_2^2\) & \(E_2^3\) & \(E_3^1\) & \(E_4^1\) & \(E_4^2\) & \(E_4^3\) & \(E_5^1\) & \(\widetilde{C_{1}}\) & \(\widetilde{C_{2}}\) & \(\widetilde{C_{3}}\) & \(\widetilde{C_{4}}\) & \(\widetilde{C_{5}}\) & \(\widetilde{C_{6}}\) & \(\widetilde{H_{\mathcal{S}}}\) \\
    \hline
    \hline
    \(\widetilde{C_{1}}\) & \(1\) & \(1\) & \(0\) & \(0\) & \(0\) & \(0\) & \(0\) & \(0\) & \(0\) & \(-2\) & \(0\) & \(0\) & \(0\) & \(0\) & \(0\) & \(1\) \\
    \(\widetilde{C_{2}}\) & \(1\) & \(0\) & \(0\) & \(0\) & \(1\) & \(0\) & \(0\) & \(0\) & \(1\) & \(0\) & \(-2\) & \(0\) & \(0\) & \(0\) & \(0\) & \(1\) \\
    \(\widetilde{C_{3}}\) & \(0\) & \(1\) & \(0\) & \(0\) & \(1\) & \(0\) & \(0\) & \(0\) & \(0\) & \(0\) & \(0\) & \(-2\) & \(0\) & \(0\) & \(0\) & \(1\) \\
    \(\widetilde{C_{4}}\) & \(1\) & \(0\) & \(0\) & \(0\) & \(0\) & \(0\) & \(1\) & \(0\) & \(0\) & \(0\) & \(0\) & \(0\) & \(-2\) & \(0\) & \(0\) & \(1\) \\
    \(\widetilde{C_{5}}\) & \(0\) & \(0\) & \(0\) & \(1\) & \(0\) & \(1\) & \(0\) & \(0\) & \(0\) & \(0\) & \(0\) & \(0\) & \(0\) & \(-2\) & \(0\) & \(1\) \\
    \(\widetilde{C_{6}}\) & \(0\) & \(0\) & \(0\) & \(0\) & \(1\) & \(0\) & \(0\) & \(1\) & \(0\) & \(0\) & \(0\) & \(0\) & \(0\) & \(0\) & \(-2\) & \(1\) \\
    \(\widetilde{H_{\mathcal{S}}}\) & \(0\) & \(0\) & \(0\) & \(0\) & \(0\) & \(0\) & \(0\) & \(0\) & \(0\) & \(1\) & \(1\) & \(1\) & \(1\) & \(1\) & \(1\) & \(4\) \\
    \hline
  \end{tabular}.
\end{table}
Note that the intersection matrix is non-degenerate.

Discriminant groups and discriminant forms of the lattices \(L_{\mathcal{S}}\) and \(H \oplus \Pic(X)\) are given by
\begin{gather*}
  G' = 
  \begin{pmatrix}
    \frac{11}{13} & \frac{7}{13} & \frac{28}{39} & \frac{35}{39} & \frac{23}{39} & \frac{2}{13} & \frac{3}{13} & \frac{28}{39} & \frac{7}{39} & \frac{29}{39} & \frac{14}{39} & \frac{8}{13} & \frac{23}{39} & \frac{1}{13} & \frac{8}{39} & \frac{4}{39}
  \end{pmatrix}, \\
  G'' = 
  \begin{pmatrix}
    0 & 0 & -\frac{31}{39} & -\frac{19}{39} & -\frac{25}{39} & \frac{1}{39}
  \end{pmatrix}; \;
  B' = 
  \begin{pmatrix}
    \frac{4}{39}
  \end{pmatrix}, \;
  B'' = 
  \begin{pmatrix}
    \frac{35}{39}
  \end{pmatrix}; \;
  Q' =
  \begin{pmatrix}
    \frac{4}{39}
  \end{pmatrix}. \;
  Q'' =
  \begin{pmatrix}
    \frac{74}{39}
  \end{pmatrix}.
\end{gather*}


\subsection{Family \textnumero4.6}\label{subsection:04-06}

The pencil \(\mathcal{S}\) is defined by the equation
\begin{gather*}
  X^{2} Y Z + X Y^{2} Z + X Y Z^{2} + X Z^{2} T + X Y T^{2} + X Z T^{2} + Y Z T^{2} + Y T^{3} + Z T^{3} = \lambda X Y Z T.
\end{gather*}
Members \(\mathcal{S}_{\lambda}\) of the pencil are irreducible for any \(\lambda \in \mathbb{P}^1\) except
\(\mathcal{S}_{\infty} = S_{(X)} + S_{(Y)} + S_{(Z)} + S_{(T)}\).
The base locus of the pencil \(\mathcal{S}\) consists of the following curves:
\begin{gather*}
  C_{1} = C_{(X, T)}, \;
  C_{2} = C_{(Y, Z)}, \;
  C_{3} = C_{(Y, T)}, \;
  C_{4} = C_{(Z, T)}, \;
  C_{5} = C_{(Z, X + T)}, \\
  C_{6} = C_{(T, X + Y + Z)}, \;
  C_{7} = C_{(X, Z T + Y (Z + T))}, \;
  C_{8} = C_{(Y, X (Z + T) + T^2)}.
\end{gather*}
Their linear equivalence classes on the generic member \(\mathcal{S}_{\Bbbk}\) of the pencil satisfy the following relations:
\begin{gather*}
  \begin{pmatrix}
    [C_{5}] \\ [C_{6}] \\ [C_{7}] \\ [C_{8}]
  \end{pmatrix} = 
  \begin{pmatrix}
    0 & -1 & 0 & -2 & 1 \\
    -1 & 0 & -1 & -1 & 1 \\
    -2 & 0 & 0 & 0 & 1 \\
    0 & -1 & -1 & 0 & 1
  \end{pmatrix} \cdot
  \begin{pmatrix}
    [C_{1}] & [C_{2}] & [C_{3}] & [C_{4}] & [H_{\mathcal{S}}]
  \end{pmatrix}^T.
\end{gather*}

For a general choice of \(\lambda \in \mathbb{C}\) the surface \(\mathcal{S}_{\lambda}\) has the following singularities:
\begin{itemize}\setlength{\itemindent}{2cm}
\item[\(P_{1} = P_{(X, Y, T)}\):] type \(\mathbb{A}_3\) with the quadratic term \(X \cdot (Y + T)\);
\item[\(P_{2} = P_{(X, Z, T)}\):] type \(\mathbb{A}_2\) with the quadratic term \(X \cdot Z\);
\item[\(P_{3} = P_{(Y, Z, T)}\):] type \(\mathbb{A}_3\) with the quadratic term \(Y \cdot Z\);
\item[\(P_{4} = P_{(X, T, Y + Z)}\):] type \(\mathbb{A}_1\) with the quadratic term \(X (X + Y + Z  - (\lambda + 1) T) + T^2\);
\item[\(P_{5} = P_{(Y, Z, X + T)}\):] type \(\mathbb{A}_1\) with the quadratic term \((Y + Z) (X - Z + T) + (\lambda + 3) Y Z\);
\item[\(P_{6} = P_{(Z, T, X + Y)}\):] type \(\mathbb{A}_1\) with the quadratic term \(Z (X + Y + Z - \lambda T) + T^2\).
\end{itemize}

Galois action on the lattice \(L_{\lambda}\) is trivial. The intersection matrix on \(L_{\lambda} = L_{\mathcal{S}}\) is represented by
\begin{table}[H]
  \begin{tabular}{|c||ccc|cc|ccc|c|c|c|ccccc|}
    \hline
    \(\bullet\) & \(E_1^1\) & \(E_1^2\) & \(E_1^3\) & \(E_2^1\) & \(E_2^2\) & \(E_3^1\) & \(E_3^2\) & \(E_3^3\) & \(E_4^1\) & \(E_5^1\) & \(E_6^1\) & \(\widetilde{C_{1}}\) & \(\widetilde{C_{2}}\) & \(\widetilde{C_{3}}\) & \(\widetilde{C_{4}}\) & \(\widetilde{H_{\mathcal{S}}}\) \\
    \hline
    \hline
    \(\widetilde{C_{1}}\) & \(1\) & \(0\) & \(0\) & \(1\) & \(0\) & \(0\) & \(0\) & \(0\) & \(1\) & \(0\) & \(0\) & \(-2\) & \(0\) & \(0\) & \(0\) & \(1\) \\
    \(\widetilde{C_{2}}\) & \(0\) & \(0\) & \(0\) & \(0\) & \(0\) & \(0\) & \(1\) & \(0\) & \(0\) & \(1\) & \(0\) & \(0\) & \(-2\) & \(0\) & \(0\) & \(1\) \\
    \(\widetilde{C_{3}}\) & \(0\) & \(0\) & \(1\) & \(0\) & \(0\) & \(1\) & \(0\) & \(0\) & \(0\) & \(0\) & \(0\) & \(0\) & \(0\) & \(-2\) & \(0\) & \(1\) \\
    \(\widetilde{C_{4}}\) & \(0\) & \(0\) & \(0\) & \(0\) & \(1\) & \(0\) & \(0\) & \(1\) & \(0\) & \(0\) & \(1\) & \(0\) & \(0\) & \(0\) & \(-2\) & \(1\) \\
    \(\widetilde{H_{\mathcal{S}}}\) & \(0\) & \(0\) & \(0\) & \(0\) & \(0\) & \(0\) & \(0\) & \(0\) & \(0\) & \(0\) & \(0\) & \(1\) & \(1\) & \(1\) & \(1\) & \(4\) \\
    \hline
  \end{tabular}.
\end{table}
Note that the intersection matrix is non-degenerate.

Discriminant groups and discriminant forms of the lattices \(L_{\mathcal{S}}\) and \(H \oplus \Pic(X)\) are given by
\begin{gather*}
  G' = 
  \begin{pmatrix}
    \frac{1}{2} & 0 & \frac{1}{2} & 0 & 0 & \frac{1}{2} & 0 & \frac{1}{2} & \frac{1}{2} & 0 & \frac{1}{2} & 0 & 0 & 0 & 0 & 0 \\
    0 & \frac{6}{11} & \frac{1}{11} & \frac{9}{11} & \frac{2}{11} & \frac{7}{22} & 0 & \frac{17}{22} & \frac{5}{22} & \frac{21}{22} & \frac{3}{11} & \frac{5}{11} & \frac{10}{11} & \frac{7}{11} & \frac{6}{11} & \frac{19}{22}
  \end{pmatrix}, \\
  G'' = 
  \begin{pmatrix}
    0 & 0 & -\frac{1}{2} & 0 & \frac{1}{2} & 0 \\
    0 & 0 & \frac{3}{11} & \frac{39}{22} & -\frac{5}{22} & -\frac{5}{11}
  \end{pmatrix}; \;
  B' = 
  \begin{pmatrix}
    0 & \frac{1}{2} \\
    \frac{1}{2} & \frac{4}{11}
  \end{pmatrix}, \;
  B'' = 
  \begin{pmatrix}
    0 & \frac{1}{2} \\
    \frac{1}{2} & \frac{7}{11}
  \end{pmatrix}; \;
  \begin{pmatrix}
    Q' \\ Q''
  \end{pmatrix}
  =
  \begin{pmatrix}
  1 & \frac{15}{11} \\
  1 & \frac{7}{11}
  \end{pmatrix}.
\end{gather*}


\subsection{Family \textnumero4.7}\label{subsection:04-07}

The pencil \(\mathcal{S}\) is defined by the equation
\begin{gather*}
  X^{2} Y Z + X Y^{2} Z + X Y Z^{2} + X Z^{2} T + X Y T^{2} + X Z T^{2} + Y Z T^{2} + Y T^{3} = \lambda X Y Z T.
\end{gather*}
Members \(\mathcal{S}_{\lambda}\) of the pencil are irreducible for any \(\lambda \in \mathbb{P}^1\) except
\(\mathcal{S}_{\infty} = S_{(X)} + S_{(Y)} + S_{(Z)} + S_{(T)}\).
The base locus of the pencil \(\mathcal{S}\) consists of the following curves:
\begin{gather*}
  C_{1} = C_{(X, Y)}, \;
  C_{2} = C_{(X, T)}, \;
  C_{3} = C_{(Y, Z)}, \;
  C_{4} = C_{(Y, T)}, \;
  C_{5} = C_{(Z, T)}, \\
  C_{6} = C_{(X, Z + T)}, \;
  C_{7} = C_{(Y, Z + T)}, \;
  C_{8} = C_{(Z, X + T)}, \;
  C_{9} = C_{(T, X + Y + Z)}.
\end{gather*}
Their linear equivalence classes on the generic member \(\mathcal{S}_{\Bbbk}\) of the pencil satisfy the following relations:
\begin{gather*}
  \begin{pmatrix}
    [C_{6}] \\ [C_{7}] \\ [C_{8}] \\ [C_{9}]
  \end{pmatrix} = 
  \begin{pmatrix}
    -1 & -2 & 0 & 0 & 0 & 1 \\
    -1 & 0 & -1 & -1 & 0 & 1 \\
    0 & 0 & -1 & 0 & -2 & 1 \\
    0 & -1 & 0 & -1 & -1 & 1
  \end{pmatrix} \cdot
  \begin{pmatrix}
    [C_{1}] & [C_{2}] & [C_{3}] & [C_{4}] & [C_{5}] & [H_{\mathcal{S}}]
  \end{pmatrix}^T.
\end{gather*}

For a general choice of \(\lambda \in \mathbb{C}\) the surface \(\mathcal{S}_{\lambda}\) has the following singularities:
\begin{itemize}\setlength{\itemindent}{2cm}
\item[\(P_{1} = P_{(X, Y, T)}\):] type \(\mathbb{A}_2\) with the quadratic term \(X \cdot (Y + T)\);
\item[\(P_{2} = P_{(X, Z, T)}\):] type \(\mathbb{A}_2\) with the quadratic term \(X \cdot Z\);
\item[\(P_{3} = P_{(Y, Z, T)}\):] type \(\mathbb{A}_3\) with the quadratic term \(Y \cdot Z\);
\item[\(P_{4} = P_{(X, Y, Z + T)}\):] type \(\mathbb{A}_1\) with the quadratic term \((X - Y) (Z + T) - (\lambda + 2) X Y\);
\item[\(P_{5} = P_{(X, T, Y + Z)}\):] type \(\mathbb{A}_1\) with the quadratic term \(X (X + Y + Z - (\lambda + 1) T) + T^2\);
\item[\(P_{6} = P_{(Z, T, X + Y)}\):] type \(\mathbb{A}_1\) with the quadratic term \(Z (X + Y + Z - \lambda T) + T^2\).
\end{itemize}

Galois action on the lattice \(L_{\lambda}\) is trivial. The intersection matrix on \(L_{\lambda} = L_{\mathcal{S}}\) is represented by
\begin{table}[H]
  \begin{tabular}{|c||cc|cc|ccc|c|c|c|cccccc|}
    \hline
    \(\bullet\) & \(E_1^1\) & \(E_1^2\) & \(E_2^1\) & \(E_2^2\) & \(E_3^1\) & \(E_3^2\) & \(E_3^3\) & \(E_4^1\) & \(E_5^1\) & \(E_6^1\) & \(\widetilde{C_{1}}\) & \(\widetilde{C_{2}}\) & \(\widetilde{C_{3}}\) & \(\widetilde{C_{4}}\) & \(\widetilde{C_{5}}\) & \(\widetilde{H_{\mathcal{S}}}\) \\
    \hline
    \hline
    \(\widetilde{C_{1}}\) & \(1\) & \(0\) & \(0\) & \(0\) & \(0\) & \(0\) & \(0\) & \(1\) & \(0\) & \(0\) & \(-2\) & \(0\) & \(1\) & \(0\) & \(0\) & \(1\) \\
    \(\widetilde{C_{2}}\) & \(1\) & \(0\) & \(1\) & \(0\) & \(0\) & \(0\) & \(0\) & \(0\) & \(1\) & \(0\) & \(0\) & \(-2\) & \(0\) & \(0\) & \(0\) & \(1\) \\
    \(\widetilde{C_{3}}\) & \(0\) & \(0\) & \(0\) & \(0\) & \(0\) & \(1\) & \(0\) & \(0\) & \(0\) & \(0\) & \(1\) & \(0\) & \(-2\) & \(0\) & \(0\) & \(1\) \\
    \(\widetilde{C_{4}}\) & \(0\) & \(1\) & \(0\) & \(0\) & \(1\) & \(0\) & \(0\) & \(0\) & \(0\) & \(0\) & \(0\) & \(0\) & \(0\) & \(-2\) & \(0\) & \(1\) \\
    \(\widetilde{C_{5}}\) & \(0\) & \(0\) & \(0\) & \(1\) & \(0\) & \(0\) & \(1\) & \(0\) & \(0\) & \(1\) & \(0\) & \(0\) & \(0\) & \(0\) & \(-2\) & \(1\) \\
    \(\widetilde{H_{\mathcal{S}}}\) & \(0\) & \(0\) & \(0\) & \(0\) & \(0\) & \(0\) & \(0\) & \(0\) & \(0\) & \(0\) & \(1\) & \(1\) & \(1\) & \(1\) & \(1\) & \(4\) \\
    \hline
  \end{tabular}.
\end{table}
Note that the intersection matrix is non-degenerate.

Discriminant groups and discriminant forms of the lattices \(L_{\mathcal{S}}\) and \(H \oplus \Pic(X)\) are given by
\begin{gather*}
  G' = 
  \begin{pmatrix}
    \frac{2}{13} & \frac{20}{39} & \frac{2}{39} & \frac{11}{13} & \frac{2}{39} & \frac{3}{13} & \frac{17}{39} & \frac{10}{13} & \frac{5}{39} & \frac{32}{39} & \frac{7}{13} & \frac{10}{39} & \frac{38}{39} & \frac{34}{39} & \frac{25}{39} & \frac{7}{39}
  \end{pmatrix}, \\
  G'' = 
  \begin{pmatrix}
    0 & 0 & -\frac{32}{39} & -\frac{19}{39} & -\frac{25}{39} & \frac{1}{39}
  \end{pmatrix}; \;
  B' = 
  \begin{pmatrix}
    \frac{4}{39}
  \end{pmatrix}, \;
  B'' = 
  \begin{pmatrix}
    \frac{35}{39}
  \end{pmatrix}; \;
  Q' =
  \begin{pmatrix}
    \frac{4}{39}
  \end{pmatrix}, \;
  Q'' =
  \begin{pmatrix}
    \frac{74}{39}
  \end{pmatrix}.
\end{gather*}


\subsection{Family \textnumero4.8}\label{subsection:04-08}

The pencil \(\mathcal{S}\) is defined by the equation
\begin{gather*}
  X^{2} Y Z + X Y^{2} Z + X Y Z^{2} + X Z^{2} T + Y Z^{2} T + X Y T^{2} + X Z T^{2} + Y Z T^{2} = \lambda X Y Z T.
\end{gather*}
Members \(\mathcal{S}_{\lambda}\) of the pencil are irreducible for any \(\lambda \in \mathbb{P}^1\) except
\(\mathcal{S}_{\infty} = S_{(X)} + S_{(Y)} + S_{(Z)} + S_{(T)}\).
The base locus of the pencil \(\mathcal{S}\) consists of the following curves:
\begin{gather*}
  C_{1} = C_{(X, Y)}, \;
  C_{2} = C_{(X, Z)}, \;
  C_{3} = C_{(X, T)}, \;
  C_{4} = C_{(Y, Z)}, \;
  C_{5} = C_{(Y, T)}, \\
  C_{6} = C_{(Z, T)}, \;
  C_{7} = C_{(X, Z + T)}, \;
  C_{8} = C_{(Y, Z + T)}, \;
  C_{9} = C_{(T, X + Y + Z)}.
\end{gather*}
Their linear equivalence classes on the generic member \(\mathcal{S}_{\Bbbk}\) of the pencil satisfy the following relations:
\begin{gather*}
  \begin{pmatrix}
    [C_{4}] \\ [C_{7}] \\ [C_{8}] \\ [C_{9}]
  \end{pmatrix} = 
  \begin{pmatrix}
    0 & -1 & 0 & 0 & -2 & 1 \\
    -1 & -1 & -1 & 0 & 0 & 1 \\
    -1 & 1 & 0 & -1 & 2 & 0 \\
    0 & 0 & -1 & -1 & -1 & 1
  \end{pmatrix} \cdot
  \begin{pmatrix}
    [C_{1}] & [C_{2}] & [C_{3}] & [C_{5}] & [C_{6}] & [H_{\mathcal{S}}]
  \end{pmatrix}^T.
\end{gather*}

For a general choice of \(\lambda \in \mathbb{C}\) the surface \(\mathcal{S}_{\lambda}\) has the following singularities:
\begin{itemize}\setlength{\itemindent}{2cm}
\item[\(P_{1} = P_{(X, Y, Z)}\):] type \(\mathbb{A}_1\) with the quadratic term \(X Y + Z (X + Y)\);
\item[\(P_{2} = P_{(X, Y, T)}\):] type \(\mathbb{A}_1\) with the quadratic term \(X Y + T (X + Y)\);
\item[\(P_{3} = P_{(X, Z, T)}\):] type \(\mathbb{A}_3\) with the quadratic term \(X \cdot Z\);
\item[\(P_{4} = P_{(Y, Z, T)}\):] type \(\mathbb{A}_3\) with the quadratic term \(Y \cdot Z\);
\item[\(P_{5} = P_{(X, Y, Z + T)}\):] type \(\mathbb{A}_1\) with the quadratic term \((X + Y) (Z + T) - (\lambda + 2) X Y\);
\item[\(P_{6} = P_{(Z, T, X + Y)}\):] type \(\mathbb{A}_1\) with the quadratic term \(Z (X + Y + Z - \lambda T) + T^2\).
\end{itemize}

Galois action on the lattice \(L_{\lambda}\) is trivial. The intersection matrix on \(L_{\lambda} = L_{\mathcal{S}}\) is represented by
\begin{table}[H]
  \begin{tabular}{|c||c|c|ccc|ccc|c|c|cccccc|}
    \hline
    \(\bullet\) & \(E_1^1\) & \(E_2^1\) & \(E_3^1\) & \(E_3^2\) & \(E_3^3\) & \(E_4^1\) & \(E_4^2\) & \(E_4^3\) & \(E_5^1\) & \(E_6^1\) & \(\widetilde{C_{1}}\) & \(\widetilde{C_{2}}\) & \(\widetilde{C_{3}}\) & \(\widetilde{C_{5}}\) & \(\widetilde{C_{6}}\) & \(\widetilde{H_{\mathcal{S}}}\) \\
    \hline
    \hline
    \(\widetilde{C_{1}}\) & \(1\) & \(1\) & \(0\) & \(0\) & \(0\) & \(0\) & \(0\) & \(0\) & \(1\) & \(0\) & \(-2\) & \(0\) & \(0\) & \(0\) & \(0\) & \(1\) \\
    \(\widetilde{C_{2}}\) & \(1\) & \(0\) & \(0\) & \(1\) & \(0\) & \(0\) & \(0\) & \(0\) & \(0\) & \(0\) & \(0\) & \(-2\) & \(0\) & \(0\) & \(0\) & \(1\) \\
    \(\widetilde{C_{3}}\) & \(0\) & \(1\) & \(1\) & \(0\) & \(0\) & \(0\) & \(0\) & \(0\) & \(0\) & \(0\) & \(0\) & \(0\) & \(-2\) & \(0\) & \(0\) & \(1\) \\
    \(\widetilde{C_{5}}\) & \(0\) & \(1\) & \(0\) & \(0\) & \(0\) & \(1\) & \(0\) & \(0\) & \(0\) & \(0\) & \(0\) & \(0\) & \(0\) & \(-2\) & \(0\) & \(1\) \\
    \(\widetilde{C_{6}}\) & \(0\) & \(0\) & \(0\) & \(0\) & \(1\) & \(0\) & \(0\) & \(1\) & \(0\) & \(1\) & \(0\) & \(0\) & \(0\) & \(0\) & \(-2\) & \(1\) \\
    \(\widetilde{H_{\mathcal{S}}}\) & \(0\) & \(0\) & \(0\) & \(0\) & \(0\) & \(0\) & \(0\) & \(0\) & \(0\) & \(0\) & \(1\) & \(1\) & \(1\) & \(1\) & \(1\) & \(4\) \\
    \hline
  \end{tabular}.
\end{table}
Note that the intersection matrix is non-degenerate.

Discriminant groups and discriminant forms of the lattices \(L_{\mathcal{S}}\) and \(H \oplus \Pic(X)\) are given by
\begin{gather*}
  G' = 
  \begin{pmatrix}
    \frac{1}{2} & \frac{3}{4} & \frac{1}{4} & \frac{1}{2} & \frac{1}{4} & \frac{1}{4} & \frac{1}{2} & \frac{3}{4} & \frac{3}{4} & 0 & \frac{1}{2} & \frac{1}{2} & 0 & 0 & 0 & 0 \\
    \frac{5}{8} & \frac{1}{4} & \frac{7}{8} & \frac{1}{2} & \frac{3}{8} & \frac{7}{8} & 0 & \frac{1}{8} & \frac{3}{4} & \frac{5}{8} & \frac{1}{2} & \frac{3}{4} & \frac{1}{4} & \frac{3}{4} & \frac{1}{4} & \frac{3}{8}
  \end{pmatrix}, \\
  G'' = 
  \begin{pmatrix}
    0 & 0 & 0 & -\frac{1}{4} & \frac{3}{4} & \frac{1}{4} \\
    0 & 0 & -\frac{1}{4} & \frac{1}{8} & -\frac{1}{2} & 0
  \end{pmatrix}; \;
  B' = 
  \begin{pmatrix}
    \frac{1}{4} & 0 \\
    0 & \frac{1}{8}
  \end{pmatrix}, \;
  B'' = 
  \begin{pmatrix}
    \frac{3}{4} & 0 \\
    0 & \frac{7}{8}
  \end{pmatrix}; \;
  \begin{pmatrix}
    Q' \\ Q''
  \end{pmatrix}
  =
  \begin{pmatrix}
    \frac{1}{4} & \frac{1}{8} \\
    \frac{7}{4} & \frac{15}{8}
  \end{pmatrix}.
\end{gather*}


\subsection{Family \textnumero4.9}\label{subsection:04-09}

The pencil \(\mathcal{S}\) is defined by the equation
\begin{gather*}
  X^{2} Y Z + X Y^{2} Z + X Y Z^{2} + Y^{2} Z T + X Z T^{2} + Y Z T^{2} + X T^{3} = \lambda X Y Z T.
\end{gather*}
Members \(\mathcal{S}_{\lambda}\) of the pencil are irreducible for any \(\lambda \in \mathbb{P}^1\) except
\(\mathcal{S}_{\infty} = S_{(X)} + S_{(Y)} + S_{(Z)} + S_{(T)}\).
The base locus of the pencil \(\mathcal{S}\) consists of the following curves:
\begin{gather*}
  C_{1} = C_{(X, Y)}, \;
  C_{2} = C_{(X, Z)}, \;
  C_{3} = C_{(X, T)}, \;
  C_{4} = C_{(Y, T)}, \\
  C_{5} = C_{(Z, T)}, \;
  C_{6} = C_{(X, Y + T)}, \;
  C_{7} = C_{(Y, Z + T)}, \;
  C_{8} = C_{(T, X + Y + Z)}.
\end{gather*}
Their linear equivalence classes on the generic member \(\mathcal{S}_{\Bbbk}\) of the pencil satisfy the following relations:
\begin{gather*}
  \begin{pmatrix}
    [C_{2}] \\ [C_{6}] \\ [C_{7}] \\ [C_{8}]
  \end{pmatrix} = 
  \begin{pmatrix}
    0 & 0 & 0 & -3 & 1 \\
    -1 & -1 & 0 & 3 & 0 \\
    -1 & 0 & -2 & 0 & 1 \\
    0 & -1 & -1 & -1 & 1
  \end{pmatrix} \cdot
  \begin{pmatrix}
    [C_{1}] & [C_{3}] & [C_{4}] & [C_{5}] & [H_{\mathcal{S}}]
  \end{pmatrix}^T.
\end{gather*}

For a general choice of \(\lambda \in \mathbb{C}\) the surface \(\mathcal{S}_{\lambda}\) has the following singularities:
\begin{itemize}\setlength{\itemindent}{2cm}
\item[\(P_{1} = P_{(X, Y, T)}\):] type \(\mathbb{A}_3\) with the quadratic term \(X \cdot Y\);
\item[\(P_{2} = P_{(X, Z, T)}\):] type \(\mathbb{A}_3\) with the quadratic term \(Z \cdot (X + T)\);
\item[\(P_{3} = P_{(Y, Z, T)}\):] type \(\mathbb{A}_2\) with the quadratic term \(Y \cdot Z\);
\item[\(P_{4} = P_{(Y, T, X + Z)}\):] type \(\mathbb{A}_1\) with the quadratic term \(Y (X + Y + Z - \lambda T) + T^2\);
\item[\(P_{5} = P_{(Z, T, X + Y)}\):] type \(\mathbb{A}_2\) with the quadratic term \(Z \cdot (X + Y + Z - (\lambda + 1) T)\).
\end{itemize}

Galois action on the lattice \(L_{\lambda}\) is trivial. The intersection matrix on \(L_{\lambda} = L_{\mathcal{S}}\) is represented by
\begin{table}[H]
  \begin{tabular}{|c||ccc|ccc|cc|c|cc|ccccc|}
    \hline
    \(\bullet\) & \(E_1^1\) & \(E_1^2\) & \(E_1^3\) & \(E_2^1\) & \(E_2^2\) & \(E_2^3\) & \(E_3^1\) & \(E_3^2\) & \(E_4^1\) & \(E_5^1\) & \(E_5^2\) & \(\widetilde{C_{1}}\) & \(\widetilde{C_{3}}\) & \(\widetilde{C_{4}}\) & \(\widetilde{C_{5}}\) & \(\widetilde{H_{\mathcal{S}}}\) \\
    \hline
    \hline
    \(\widetilde{C_{1}}\) & \(0\) & \(1\) & \(0\) & \(0\) & \(0\) & \(0\) & \(0\) & \(0\) & \(0\) & \(0\) & \(0\) & \(-2\) & \(0\) & \(0\) & \(0\) & \(1\) \\
    \(\widetilde{C_{3}}\) & \(1\) & \(0\) & \(0\) & \(1\) & \(0\) & \(0\) & \(0\) & \(0\) & \(0\) & \(0\) & \(0\) & \(0\) & \(-2\) & \(0\) & \(0\) & \(1\) \\
    \(\widetilde{C_{4}}\) & \(0\) & \(0\) & \(1\) & \(0\) & \(0\) & \(0\) & \(1\) & \(0\) & \(1\) & \(0\) & \(0\) & \(0\) & \(0\) & \(-2\) & \(0\) & \(1\) \\
    \(\widetilde{C_{5}}\) & \(0\) & \(0\) & \(0\) & \(0\) & \(0\) & \(1\) & \(0\) & \(1\) & \(0\) & \(1\) & \(0\) & \(0\) & \(0\) & \(0\) & \(-2\) & \(1\) \\
    \(\widetilde{H_{\mathcal{S}}}\) & \(0\) & \(0\) & \(0\) & \(0\) & \(0\) & \(0\) & \(0\) & \(0\) & \(0\) & \(0\) & \(0\) & \(1\) & \(1\) & \(1\) & \(1\) & \(4\) \\
    \hline
  \end{tabular}.
\end{table}
Note that the intersection matrix is non-degenerate.

Discriminant groups and discriminant forms of the lattices \(L_{\mathcal{S}}\) and \(H \oplus \Pic(X)\) are given by
\begin{gather*}
  G' = 
  \begin{pmatrix}
    \frac{18}{31} & \frac{29}{31} & \frac{19}{31} & \frac{14}{31} & \frac{21}{31} & \frac{28}{31} & \frac{28}{31} & \frac{16}{31} & \frac{20}{31} & \frac{13}{31} & \frac{22}{31} & \frac{21}{31} & \frac{7}{31} & \frac{9}{31} & \frac{4}{31} & \frac{13}{31}
  \end{pmatrix}, \\
  G'' = 
  \begin{pmatrix}
    0 & 0 & -\frac{21}{31} & -\frac{15}{31} & -\frac{20}{31} & \frac{1}{31}
  \end{pmatrix}; \;
  B' = 
  \begin{pmatrix}
    \frac{5}{31}
  \end{pmatrix}, \;
  B'' = 
  \begin{pmatrix}
    \frac{26}{31}
  \end{pmatrix}; \;
  Q' =
  \begin{pmatrix}
    \frac{36}{31}
  \end{pmatrix}, \;
  Q'' =
  \begin{pmatrix}
    \frac{26}{31}
  \end{pmatrix}.
\end{gather*}


\subsection{Family \textnumero4.10}\label{subsection:04-10}

The pencil \(\mathcal{S}\) is defined by the equation
\begin{gather*}
  X^{2} Y Z + X Y^{2} Z + X Y Z^{2} + Y^{2} Z T + X Y T^{2} + X Z T^{2} + Y Z T^{2} = \lambda X Y Z T.
\end{gather*}
Members \(\mathcal{S}_{\lambda}\) of the pencil are irreducible for any \(\lambda \in \mathbb{P}^1\) except
\(\mathcal{S}_{\infty} = S_{(X)} + S_{(Y)} + S_{(Z)} + S_{(T)}\).
The base locus of the pencil \(\mathcal{S}\) consists of the following curves:
\begin{gather*}
  C_{1} = C_{(X, Y)}, \;
  C_{2} = C_{(X, Z)}, \;
  C_{3} = C_{(X, T)}, \;
  C_{4} = C_{(Y, Z)}, \\
  C_{5} = C_{(Y, T)}, \;
  C_{6} = C_{(Z, T)}, \;
  C_{7} = C_{(X, Y + T)}, \;
  C_{8} = C_{(T, X + Y + Z)}.
\end{gather*}
Their linear equivalence classes on the generic member \(\mathcal{S}_{\Bbbk}\) of the pencil satisfy the following relations:
\begin{gather*}
  \begin{pmatrix}
    [C_{2}] \\ [C_{4}] \\ [C_{7}] \\ [C_{8}]
  \end{pmatrix} = 
  \begin{pmatrix}
    1 & 0 & 2 & -2 & 0 \\
    -1 & 0 & -2 & 0 & 1 \\
    -2 & -1 & -2 & 2 & 1 \\
    0 & -1 & -1 & -1 & 1
  \end{pmatrix} \cdot
  \begin{pmatrix}
    [C_{1}] & [C_{3}] & [C_{5}] & [C_{6}] & [H_{\mathcal{S}}]
  \end{pmatrix}^T.
\end{gather*}

For a general choice of \(\lambda \in \mathbb{C}\) the surface \(\mathcal{S}_{\lambda}\) has the following singularities:
\begin{itemize}\setlength{\itemindent}{2cm}
\item[\(P_{1} = P_{(X, Y, Z)}\):] type \(\mathbb{A}_1\) with the quadratic term \(X Y + X Z + Y Z\);
\item[\(P_{2} = P_{(X, Y, T)}\):] type \(\mathbb{A}_3\) with the quadratic term \(X \cdot Y\);
\item[\(P_{3} = P_{(X, Z, T)}\):] type \(\mathbb{A}_2\) with the quadratic term \(Z \cdot (X + T)\);
\item[\(P_{4} = P_{(Y, Z, T)}\):] type \(\mathbb{A}_3\) with the quadratic term \(Y \cdot Z\);
\item[\(P_{5} = P_{(Y, T, X + Z)}\):] type \(\mathbb{A}_1\) with the quadratic term \(Y (X + Y + Z - \lambda T) + T^2\);
\item[\(P_{6} = P_{(Z, T, X + Y)}\):] type \(\mathbb{A}_1\) with the quadratic term \(Z (X + Y + Z - (\lambda + 1) T) + T^2\).
\end{itemize}

Galois action on the lattice \(L_{\lambda}\) is trivial. The intersection matrix on \(L_{\lambda} = L_{\mathcal{S}}\) is represented by
\begin{table}[H]
  \begin{tabular}{|c||c|ccc|cc|ccc|c|c|ccccc|}
    \hline
    \(\bullet\) & \(E_1^1\) & \(E_2^1\) & \(E_2^2\) & \(E_2^3\) & \(E_3^1\) & \(E_3^2\) & \(E_4^1\) & \(E_4^2\) & \(E_4^3\) & \(E_5^1\) & \(E_6^1\) & \(\widetilde{C_{1}}\) & \(\widetilde{C_{3}}\) & \(\widetilde{C_{5}}\) & \(\widetilde{C_{6}}\) & \(\widetilde{H_{\mathcal{S}}}\) \\
    \hline
    \hline
    \(\widetilde{C_{1}}\) & \(1\) & \(0\) & \(1\) & \(0\) & \(0\) & \(0\) & \(0\) & \(0\) & \(0\) & \(0\) & \(0\) & \(-2\) & \(0\) & \(0\) & \(0\) & \(1\) \\
    \(\widetilde{C_{3}}\) & \(0\) & \(1\) & \(0\) & \(0\) & \(1\) & \(0\) & \(0\) & \(0\) & \(0\) & \(0\) & \(0\) & \(0\) & \(-2\) & \(0\) & \(0\) & \(1\) \\
    \(\widetilde{C_{5}}\) & \(0\) & \(0\) & \(0\) & \(1\) & \(0\) & \(0\) & \(1\) & \(0\) & \(0\) & \(1\) & \(0\) & \(0\) & \(0\) & \(-2\) & \(0\) & \(1\) \\
    \(\widetilde{C_{6}}\) & \(0\) & \(0\) & \(0\) & \(0\) & \(0\) & \(1\) & \(0\) & \(0\) & \(1\) & \(0\) & \(1\) & \(0\) & \(0\) & \(0\) & \(-2\) & \(1\) \\
    \(\widetilde{H_{\mathcal{S}}}\) & \(0\) & \(0\) & \(0\) & \(0\) & \(0\) & \(0\) & \(0\) & \(0\) & \(0\) & \(0\) & \(0\) & \(1\) & \(1\) & \(1\) & \(1\) & \(4\) \\
    \hline
  \end{tabular}.
\end{table}
Note that the intersection matrix is non-degenerate.

Discriminant groups and discriminant forms of the lattices \(L_{\mathcal{S}}\) and \(H \oplus \Pic(X)\) are given by
\begin{gather*}
  G' = 
  \begin{pmatrix}
    \frac{1}{2} & \frac{1}{2} & 0 & \frac{1}{2} & 0 & 0 & 0 & 0 & 0 & 0 & \frac{1}{2} & 0 & 0 & 0 & 0 & \frac{1}{2} \\
    \frac{9}{14} & \frac{1}{14} & 0 & \frac{9}{14} & \frac{2}{7} & \frac{3}{7} & \frac{5}{14} & \frac{3}{7} & \frac{1}{2} & \frac{9}{14} & \frac{2}{7} & \frac{2}{7} & \frac{1}{7} & \frac{2}{7} & \frac{4}{7} & \frac{13}{14}
  \end{pmatrix}, \\
  G'' = 
  \begin{pmatrix}
    0 & 0 & -\frac{1}{2} & \frac{1}{2} & 0 & 0 \\
    0 & 0 & \frac{9}{14} & \frac{1}{7} & \frac{5}{7} & -\frac{1}{14}
  \end{pmatrix}; \;
  B' = 
  \begin{pmatrix}
    0 & \frac{1}{2} \\
    \frac{1}{2} & \frac{2}{7}
  \end{pmatrix}, \;
  B'' = 
  \begin{pmatrix}
    0 & \frac{1}{2} \\
    \frac{1}{2} & \frac{5}{7}
  \end{pmatrix}; \;
  \begin{pmatrix}
    Q' \\ Q''
  \end{pmatrix}
  =
  \begin{pmatrix}
    1 & \frac{2}{7} \\
    1 & \frac{12}{7}    
  \end{pmatrix}.
\end{gather*}


\subsection{Family \textnumero4.11}\label{subsection:04-11}

The pencil \(\mathcal{S}\) is defined by the equation
\begin{gather*}
  X^{2} Y Z + X Y^{2} Z + X Y Z^{2} + Y^{2} Z T + X Z T^{2} + Y Z T^{2} + Y T^{3} = \lambda X Y Z T.
\end{gather*}
Members \(\mathcal{S}_{\lambda}\) of the pencil are irreducible for any \(\lambda \in \mathbb{P}^1\) except
\(\mathcal{S}_{\infty} = S_{(X)} + S_{(Y)} + S_{(Z)} + S_{(T)}\).
The base locus of the pencil \(\mathcal{S}\) consists of the following curves:
\begin{gather*}
  C_{1} = C_{(X, Y)}, \;
  C_{2} = C_{(X, T)}, \;
  C_{3} = C_{(Y, Z)}, \;
  C_{4} = C_{(Y, T)}, \\
  C_{5} = C_{(Z, T)}, \;
  C_{6} = C_{(T, X + Y + Z)}, \;
  C_{7} = C_{(X, Z (Y + T) + T^2)}.
\end{gather*}
Their linear equivalence classes on the generic member \(\mathcal{S}_{\Bbbk}\) of the pencil satisfy the following relations:
\begin{gather*}
  \begin{pmatrix}
    [C_{1}] \\ [C_{3}] \\ [C_{6}] \\ [C_{7}]
  \end{pmatrix} = 
  \begin{pmatrix}
    0 & -2 & 3 & 0 \\
    0 & 0 & -3 & 1 \\
    -1 & -1 & -1 & 1 \\
    -1 & 2 & -3 & 1
  \end{pmatrix} \cdot
  \begin{pmatrix}
    [C_{2}] \\ [C_{4}] \\ [C_{5}] \\ [H_{\mathcal{S}}]
  \end{pmatrix}.
\end{gather*}

For a general choice of \(\lambda \in \mathbb{C}\) the surface \(\mathcal{S}_{\lambda}\) has the following singularities:
\begin{itemize}\setlength{\itemindent}{2cm}
\item[\(P_{1} = P_{(X, Y, T)}\):] type \(\mathbb{A}_3\) with the quadratic term \(X \cdot Y\);
\item[\(P_{2} = P_{(X, Z, T)}\):] type \(\mathbb{A}_2\) with the quadratic term \(Z \cdot (X + T)\);
\item[\(P_{3} = P_{(Y, Z, T)}\):] type \(\mathbb{A}_4\) with the quadratic term \(Y \cdot Z\);
\item[\(P_{4} = P_{(Y, T, X + Z)}\):] type \(\mathbb{A}_1\) with the quadratic term \(Y (X + Y + Z - \lambda T) + T^2\);
\item[\(P_{5} = P_{(Z, T, X + Y)}\):] type \(\mathbb{A}_2\) with the quadratic term \(Z \cdot (X + Y + Z - (\lambda + 1) T)\).
\end{itemize}

Galois action on the lattice \(L_{\lambda}\) is trivial. The intersection matrix on \(L_{\lambda} = L_{\mathcal{S}}\) is represented by
\begin{table}[H]
  \begin{tabular}{|c||ccc|cc|cccc|c|cc|cccc|}
    \hline
    \(\bullet\) & \(E_1^1\) & \(E_1^2\) & \(E_1^3\) & \(E_2^1\) & \(E_2^2\) & \(E_3^1\) & \(E_3^2\) & \(E_3^3\) & \(E_3^4\) & \(E_4^1\) & \(E_5^1\) & \(E_5^2\) & \(\widetilde{C_{2}}\) & \(\widetilde{C_{4}}\) & \(\widetilde{C_{5}}\) & \(\widetilde{H_{\mathcal{S}}}\) \\
    \hline
    \hline
    \(\widetilde{C_{2}}\) & \(1\) & \(0\) & \(0\) & \(1\) & \(0\) & \(0\) & \(0\) & \(0\) & \(0\) & \(0\) & \(0\) & \(0\) & \(-2\) & \(0\) & \(0\) & \(1\) \\
    \(\widetilde{C_{4}}\) & \(0\) & \(0\) & \(1\) & \(0\) & \(0\) & \(1\) & \(0\) & \(0\) & \(0\) & \(1\) & \(0\) & \(0\) & \(0\) & \(-2\) & \(0\) & \(1\) \\
    \(\widetilde{C_{5}}\) & \(0\) & \(0\) & \(0\) & \(0\) & \(1\) & \(0\) & \(0\) & \(0\) & \(1\) & \(0\) & \(1\) & \(0\) & \(0\) & \(0\) & \(-2\) & \(1\) \\
    \(\widetilde{H_{\mathcal{S}}}\) & \(0\) & \(0\) & \(0\) & \(0\) & \(0\) & \(0\) & \(0\) & \(0\) & \(0\) & \(0\) & \(0\) & \(0\) & \(1\) & \(1\) & \(1\) & \(4\) \\
    \hline
  \end{tabular}.
\end{table}
Note that the intersection matrix is non-degenerate.

Discriminant groups and discriminant forms of the lattices \(L_{\mathcal{S}}\) and \(H \oplus \Pic(X)\) are given by
\begin{gather*}
  G' = 
  \begin{pmatrix}
    \frac{16}{23} & \frac{19}{23} & \frac{22}{23} & \frac{6}{23} & \frac{22}{23} & 0 & \frac{21}{23} & \frac{19}{23} & \frac{17}{23} & \frac{1}{23} & \frac{10}{23} & \frac{5}{23} & \frac{13}{23} & \frac{2}{23} & \frac{15}{23} & \frac{4}{23}
  \end{pmatrix}, \\
  G'' = 
  \begin{pmatrix}
    0 & 0 & -\frac{13}{23} & -\frac{20}{23} & -\frac{16}{23} & \frac{1}{23}
  \end{pmatrix}; \;
  B' = 
  \begin{pmatrix}
    \frac{12}{23}
  \end{pmatrix}, \;
  B'' = 
  \begin{pmatrix}
    \frac{11}{23}
  \end{pmatrix}; \;
  Q' =
  \begin{pmatrix}
    \frac{12}{23}
  \end{pmatrix}, \;
  Q'' =
  \begin{pmatrix}
    \frac{34}{23}
  \end{pmatrix}.
\end{gather*}


\subsection{Family \textnumero4.12}\label{subsection:04-12}

The pencil \(\mathcal{S}\) is defined by the equation
\begin{gather*}
  X^{2} Y Z + X Y^{2} Z + X Y Z^{2} + Y^{2} Z T + Y^{2} T^{2} + X Z T^{2} + Y Z T^{2} = \lambda X Y Z T.
\end{gather*}
Members \(\mathcal{S}_{\lambda}\) of the pencil are irreducible for any \(\lambda \in \mathbb{P}^1\) except
\(\mathcal{S}_{\infty} = S_{(X)} + S_{(Y)} + S_{(Z)} + S_{(T)}\).
The base locus of the pencil \(\mathcal{S}\) consists of the following curves:
\begin{gather*}
  C_{1} = C_{(X, Y)}, \;
  C_{2} = C_{(X, T)}, \;
  C_{3} = C_{(Y, Z)}, \;
  C_{4} = C_{(Y, T)}, \\
  C_{5} = C_{(Z, T)}, \;
  C_{6} = C_{(T, X + Y + Z)}, \;
  C_{7} = C_{(X, Z T + Y (Z + T))}.
\end{gather*}
Their linear equivalence classes on the generic member \(\mathcal{S}_{\Bbbk}\) of the pencil satisfy the following relations:
\begin{gather*}
  \begin{pmatrix}
    [C_{3}] \\ [C_{6}] \\ [C_{7}] \\ [H_{\mathcal{S}}]
  \end{pmatrix} = 
  \begin{pmatrix}
    1 & 0 & 2 & -2 \\
    2 & -1 & 3 & -3 \\
    1 & -1 & 4 & -2 \\
    2 & 0 & 4 & -2
  \end{pmatrix} \cdot
  \begin{pmatrix}
    [C_{1}] \\ [C_{2}] \\ [C_{4}] \\ [C_{5}]
  \end{pmatrix}.
\end{gather*}

For a general choice of \(\lambda \in \mathbb{C}\) the surface \(\mathcal{S}_{\lambda}\) has the following singularities:
\begin{itemize}\setlength{\itemindent}{2cm}
\item[\(P_{1} = P_{(X, Y, Z)}\):] type \(\mathbb{A}_1\) with the quadratic term \(Y^2 + Z (X + Y)\);
\item[\(P_{2} = P_{(X, Y, T)}\):] type \(\mathbb{A}_3\) with the quadratic term \(X \cdot Y\);
\item[\(P_{3} = P_{(X, Z, T)}\):] type \(\mathbb{A}_1\) with the quadratic term \(Z (X + T) + T^2\);
\item[\(P_{4} = P_{(Y, Z, T)}\):] type \(\mathbb{A}_5\) with the quadratic term \(Y \cdot Z\);
\item[\(P_{5} = P_{(Y, T, X + Z)}\):] type \(\mathbb{A}_1\) with the quadratic term \(Y (X + Y + Z - \lambda T) + T^2\);
\item[\(P_{6} = P_{(Z, T, X + Y)}\):] type \(\mathbb{A}_1\) with the quadratic term \(Z (X + Y + Z - (\lambda + 1) T) - T^2\).
\end{itemize}

Galois action on the lattice \(L_{\lambda}\) is trivial. The intersection matrix on \(L_{\lambda} = L_{\mathcal{S}}\) is represented by
\begin{table}[H]
  \begin{tabular}{|c||c|ccc|c|ccccc|c|c|cccc|}
    \hline
    \(\bullet\) & \(E_1^1\) & \(E_2^1\) & \(E_2^2\) & \(E_2^3\) & \(E_3^1\) & \(E_4^1\) & \(E_4^2\) & \(E_4^3\) & \(E_4^4\) & \(E_4^5\) & \(E_5^1\) & \(E_6^1\) & \(\widetilde{C_{1}}\) & \(\widetilde{C_{2}}\) & \(\widetilde{C_{4}}\) & \(\widetilde{C_{5}}\) \\
    \hline
    \hline
    \(\widetilde{C_{1}}\) & \(1\) & \(0\) & \(1\) & \(0\) & \(0\) & \(0\) & \(0\) & \(0\) & \(0\) & \(0\) & \(0\) & \(0\) & \(-2\) & \(0\) & \(0\) & \(0\) \\
    \(\widetilde{C_{2}}\) & \(0\) & \(1\) & \(0\) & \(0\) & \(1\) & \(0\) & \(0\) & \(0\) & \(0\) & \(0\) & \(0\) & \(0\) & \(0\) & \(-2\) & \(0\) & \(0\) \\
    \(\widetilde{C_{4}}\) & \(0\) & \(0\) & \(0\) & \(1\) & \(0\) & \(1\) & \(0\) & \(0\) & \(0\) & \(0\) & \(1\) & \(0\) & \(0\) & \(0\) & \(-2\) & \(0\) \\
    \(\widetilde{C_{5}}\) & \(0\) & \(0\) & \(0\) & \(0\) & \(1\) & \(0\) & \(0\) & \(0\) & \(0\) & \(1\) & \(0\) & \(1\) & \(0\) & \(0\) & \(0\) & \(-2\) \\
    \hline
  \end{tabular}.
\end{table}
Note that the intersection matrix is non-degenerate.

Discriminant groups and discriminant forms of the lattices \(L_{\mathcal{S}}\) and \(H \oplus \Pic(X)\) are given by
\begin{gather*}
  G' = 
  \begin{pmatrix}
    0 & \frac{1}{2} & 0 & \frac{1}{2} & \frac{1}{2} & \frac{1}{2} & 0 & \frac{1}{2} & 0 & \frac{1}{2} & 0 & 0 & 0 & 0 & 0 & 0 \\
    \frac{1}{5} & \frac{1}{2} & \frac{3}{5} & \frac{3}{10} & \frac{3}{10} & \frac{1}{5} & \frac{2}{5} & \frac{3}{5} & \frac{4}{5} & 0 & \frac{1}{2} & \frac{1}{10} & \frac{2}{5} & \frac{2}{5} & 0 & \frac{1}{5}
  \end{pmatrix}, \\
  G'' = 
  \begin{pmatrix}
    0 & 0 & -\frac{1}{2} & \frac{1}{2} & 0 & 0 \\
    0 & 0 & -\frac{1}{5} & \frac{3}{10} & \frac{2}{5} & -\frac{1}{10}
  \end{pmatrix}; \;
  B' = 
  \begin{pmatrix}
    0 & \frac{1}{2} \\
    \frac{1}{2} & \frac{7}{10}
  \end{pmatrix}, \;
  B'' = 
  \begin{pmatrix}
    0 & \frac{1}{2} \\
    \frac{1}{2} & \frac{3}{10}
  \end{pmatrix}; \;
  \begin{pmatrix}
    Q' \\ Q''
  \end{pmatrix}
  =
  \begin{pmatrix}
    1 & \frac{7}{10} \\
    1 & \frac{13}{10}
  \end{pmatrix}.
\end{gather*}


\subsection{Family \textnumero4.13}\label{subsection:04-13}

The pencil \(\mathcal{S}\) is defined by the equation
\begin{gather*}
  X^{2} Y Z + X Y^{2} Z + X Y Z^{2} + X^{2} Z T + Y^{2} Z T + X^{2} T^{2} + \\ X Y T^{2} + 2 X Z T^{2} + 2 Y Z T^{2} + X T^{3} + Z T^{3} = \lambda X Y Z T.
\end{gather*}
Members \(\mathcal{S}_{\lambda}\) of the pencil are irreducible for any \(\lambda \in \mathbb{P}^1\) except
\(\mathcal{S}_{\infty} = S_{(X)} + S_{(Y)} + S_{(Z)} + S_{(T)}\).
The base locus of the pencil \(\mathcal{S}\) consists of the following curves:
\begin{gather*}
  C_{1} = C_{(X, Z)}, \;
  C_{2} = C_{(X, T)}, \;
  C_{3} = C_{(Y, T)}, \;
  C_{4} = C_{(Z, T)}, \;
  C_{5} = C_{(X, Y + T)}, \\
  C_{6} = C_{(Y, X + T)}, \;
  C_{7} = C_{(Z, X + Y + T)}, \;
  C_{8} = C_{(T, X + Y + Z)}, \;
  C_{9} = C_{(Y, Z T + X (Z + T))}.
\end{gather*}
Their linear equivalence classes on the generic member \(\mathcal{S}_{\Bbbk}\) of the pencil satisfy the following relations:
\begin{gather*}
  \begin{pmatrix}
    [C_{2}] \\ [C_{7}] \\ [C_{8}] \\ [C_{9}]
  \end{pmatrix} = 
  \begin{pmatrix}
    -1 & 0 & 0 & -2 & 0 & 1 \\
    -1 & 0 & -2 & 0 & 0 & 1 \\
    1 & -1 & -1 & 2 & 0 & 0 \\
    0 & -1 & 0 & 0 & -1 & 1
  \end{pmatrix} \cdot
  \begin{pmatrix}
    [C_{1}] & [C_{3}] & [C_{4}] & [C_{5}] & [C_{6}] & [H_{\mathcal{S}}]
  \end{pmatrix}^T.
\end{gather*}

For a general choice of \(\lambda \in \mathbb{C}\) the surface \(\mathcal{S}_{\lambda}\) has the following singularities:
\begin{itemize}\setlength{\itemindent}{2cm}
\item[\(P_{1} = P_{(X, Y, T)}\):] type \(\mathbb{A}_2\) with the quadratic term \(X \cdot Y\);
\item[\(P_{2} = P_{(X, Z, T)}\):] type \(\mathbb{A}_2\) with the quadratic term \(Z \cdot (X + T)\);
\item[\(P_{3} = P_{(Y, Z, T)}\):] type \(\mathbb{A}_1\) with the quadratic term \(Z (Y + T) + T^2\);
\item[\(P_{4} = P_{(X, Z, Y + T)}\):] type \(\mathbb{A}_2\) with the quadratic term \(X \cdot (X + Y + (\lambda + 3) Z + T)\);
\item[\(P_{5} = P_{(Z, T, X + Y)}\):] type \(\mathbb{A}_2\) with the quadratic term \(Z \cdot (X + Y + Z - (\lambda + 2) T)\);
\item[\(P_{6} = P_{(X, Y + T, Z - (\lambda + 3) T)}\):] type \(\mathbb{A}_1\).
\end{itemize}

Galois action on the lattice \(L_{\lambda}\) is trivial. The intersection matrix on \(L_{\lambda} = L_{\mathcal{S}}\) is represented by
\begin{table}[H]
  \begin{tabular}{|c||cc|cc|c|cc|cc|c|cccccc|}
    \hline
    \(\bullet\) & \(E_1^1\) & \(E_1^2\) & \(E_2^1\) & \(E_2^2\) & \(E_3^1\) & \(E_4^1\) & \(E_4^2\) & \(E_5^1\) & \(E_5^2\) & \(E_6^1\) & \(\widetilde{C_{1}}\) & \(\widetilde{C_{3}}\) & \(\widetilde{C_{4}}\) & \(\widetilde{C_{5}}\) & \(\widetilde{C_{6}}\) & \(\widetilde{H_{\mathcal{S}}}\) \\
    \hline
    \hline
    \(\widetilde{C_{1}}\) & \(0\) & \(0\) & \(1\) & \(0\) & \(0\) & \(1\) & \(0\) & \(0\) & \(0\) & \(0\) & \(-2\) & \(0\) & \(0\) & \(0\) & \(0\) & \(1\) \\
    \(\widetilde{C_{3}}\) & \(1\) & \(0\) & \(0\) & \(0\) & \(1\) & \(0\) & \(0\) & \(0\) & \(0\) & \(0\) & \(0\) & \(-2\) & \(0\) & \(0\) & \(0\) & \(1\) \\
    \(\widetilde{C_{4}}\) & \(0\) & \(0\) & \(1\) & \(0\) & \(1\) & \(0\) & \(0\) & \(1\) & \(0\) & \(0\) & \(0\) & \(0\) & \(-2\) & \(0\) & \(0\) & \(1\) \\
    \(\widetilde{C_{5}}\) & \(0\) & \(1\) & \(0\) & \(0\) & \(0\) & \(1\) & \(0\) & \(0\) & \(0\) & \(1\) & \(0\) & \(0\) & \(0\) & \(-2\) & \(0\) & \(1\) \\
    \(\widetilde{C_{6}}\) & \(1\) & \(0\) & \(0\) & \(0\) & \(0\) & \(0\) & \(0\) & \(0\) & \(0\) & \(0\) & \(0\) & \(0\) & \(0\) & \(0\) & \(-2\) & \(1\) \\
    \(\widetilde{H_{\mathcal{S}}}\) & \(0\) & \(0\) & \(0\) & \(0\) & \(0\) & \(0\) & \(0\) & \(0\) & \(0\) & \(0\) & \(1\) & \(1\) & \(1\) & \(1\) & \(1\) & \(4\) \\
    \hline
  \end{tabular}.
\end{table}
Note that the intersection matrix is non-degenerate.

Discriminant groups and discriminant forms of the lattices \(L_{\mathcal{S}}\) and \(H \oplus \Pic(X)\) are given by
\begin{gather*}
  G' = 
  \begin{pmatrix}
    0 & \frac{1}{2} & 0 & \frac{1}{2} & 0 & 0 & \frac{1}{2} & 0 & 0 & \frac{1}{2} & \frac{1}{2} & 0 & 0 & 0 & \frac{1}{2} & 0 \\
    \frac{7}{11} & \frac{3}{11} & \frac{7}{11} & \frac{9}{11} & \frac{10}{11} & \frac{2}{11} & \frac{13}{22} & \frac{8}{11} & \frac{19}{22} & \frac{5}{11} & \frac{19}{22} & \frac{5}{22} & \frac{13}{22} & \frac{10}{11} & \frac{17}{22} & \frac{10}{11}
  \end{pmatrix}, \\
  B' = 
  \begin{pmatrix}
    0 & \frac{1}{2} \\
    \frac{1}{2} & \frac{1}{11}
  \end{pmatrix}; \;
  G'' = 
  \begin{pmatrix}
    0 & 0 & 0 & -\frac{1}{2} & \frac{1}{2} & 0 \\
    0 & 0 & -\frac{9}{11} & -\frac{7}{11} & -\frac{3}{22} & \frac{1}{22}
  \end{pmatrix}, \;
  B'' = 
  \begin{pmatrix}
    0 & \frac{1}{2} \\
    \frac{1}{2} & \frac{10}{11}
  \end{pmatrix}; \;
  \begin{pmatrix}
    Q' \\ Q''
  \end{pmatrix}
  =
  \begin{pmatrix}
    1 & \frac{1}{11} \\
    1 & \frac{21}{11}    
  \end{pmatrix}.
\end{gather*}


\section{Dolgachev--Nikulin duality for Fano threefolds: ranks 5--10}\label{appendix:higher-ranks}
\subsection{Family \textnumero5.1}\label{subsection:05-01}

The pencil \(\mathcal{S}\) is defined by the equation
\begin{gather*}
  X^{2} Y Z + X Y^{2} Z + X^{2} Z T + Y^{2} Z T + X Z^{2} T + Y Z^{2} T + \\ X Y T^{2} + 2 X Z T^{2} + 2 Y Z T^{2} + Z^{2} T^{2} + Z T^{3} = \lambda X Y Z T.
\end{gather*}
Members \(\mathcal{S}_{\lambda}\) of the pencil are irreducible for any \(\lambda \in \mathbb{P}^1\) except
\(\mathcal{S}_{\infty} = S_{(X)} + S_{(Y)} + S_{(Z)} + S_{(T)}\).
The base locus of the pencil \(\mathcal{S}\) consists of the following curves:
\begin{gather*}
  C_{1} = C_{(X, Z)}, \;
  C_{2} = C_{(X, T)}, \;
  C_{3} = C_{(Y, Z)}, \;
  C_{4} = C_{(Y, T)}, \;
  C_{5} = C_{(Z, T)}, \;
  C_{6} = C_{(X, Y + T)}, \\
  C_{7} = C_{(Y, X + T)}, \;
  C_{8} = C_{(T, X + Y)}, \;
  C_{9} = C_{(X, Y + Z + T)}, \;
  C_{10} = C_{(Y, X + Z + T)}.
\end{gather*}
Their linear equivalence classes on the generic member \(\mathcal{S}_{\Bbbk}\) of the pencil satisfy the following relations:
\begin{gather*}
  \begin{pmatrix}
    [C_{3}] \\ [C_{8}] \\ [C_{9}] \\ [C_{10}]
  \end{pmatrix} = 
  \begin{pmatrix}
    -1 & 0 & 0 & -2 & 0 & 0 & 1 \\
    0 & -1 & -1 & -1 & 0 & 0 & 1 \\
    -1 & -1 & 0 & 0 & -1 & 0 & 1 \\
    1 & 0 & -1 & 2 & 0 & -1 & 0
  \end{pmatrix} \cdot
  \begin{pmatrix}
    [C_{1}] & [C_{2}] & [C_{4}] & [C_{5}] & [C_{6}] & [C_{7}] & [H_{\mathcal{S}}]
  \end{pmatrix}^T.
\end{gather*}

For a general choice of \(\lambda \in \mathbb{C}\) the surface \(\mathcal{S}_{\lambda}\) has the following singularities:
\begin{itemize}\setlength{\itemindent}{2cm}
\item[\(P_{1} = P_{(X, Y, T)}\):] type \(\mathbb{A}_3\) with the quadratic term \(T \cdot (X + Y + T)\);
\item[\(P_{2} = P_{(X, Z, T)}\):] type \(\mathbb{A}_2\) with the quadratic term \(Z \cdot (X + T)\);
\item[\(P_{3} = P_{(Y, Z, T)}\):] type \(\mathbb{A}_2\) with the quadratic term \(Z \cdot (Y + T)\);
\item[\(P_{4} = P_{(Z, T, X + Y)}\):] type \(\mathbb{A}_1\) with the quadratic term \(Z (X + Y - (\lambda + 2) T) + T^2\).
\end{itemize}

Galois action on the lattice \(L_{\lambda}\) is trivial. The intersection matrix on \(L_{\lambda} = L_{\mathcal{S}}\) is represented by
\begin{table}[H]
  \begin{tabular}{|c||ccc|cc|cc|c|ccccccc|}
    \hline
    \(\bullet\) & \(E_1^1\) & \(E_1^2\) & \(E_1^3\) & \(E_2^1\) & \(E_2^2\) & \(E_3^1\) & \(E_3^2\) & \(E_4^1\) & \(\widetilde{C_{1}}\) & \(\widetilde{C_{2}}\) & \(\widetilde{C_{4}}\) & \(\widetilde{C_{5}}\) & \(\widetilde{C_{6}}\) & \(\widetilde{C_{7}}\) & \(\widetilde{H_{\mathcal{S}}}\) \\
    \hline
    \hline
    \(\widetilde{C_{1}}\) & \(0\) & \(0\) & \(0\) & \(1\) & \(0\) & \(0\) & \(0\) & \(0\) & \(-2\) & \(0\) & \(0\) & \(0\) & \(1\) & \(0\) & \(1\) \\
    \(\widetilde{C_{2}}\) & \(1\) & \(0\) & \(0\) & \(0\) & \(1\) & \(0\) & \(0\) & \(0\) & \(0\) & \(-2\) & \(0\) & \(0\) & \(0\) & \(0\) & \(1\) \\
    \(\widetilde{C_{4}}\) & \(1\) & \(0\) & \(0\) & \(0\) & \(0\) & \(1\) & \(0\) & \(0\) & \(0\) & \(0\) & \(-2\) & \(0\) & \(0\) & \(0\) & \(1\) \\
    \(\widetilde{C_{5}}\) & \(0\) & \(0\) & \(0\) & \(1\) & \(0\) & \(0\) & \(1\) & \(1\) & \(0\) & \(0\) & \(0\) & \(-2\) & \(0\) & \(0\) & \(1\) \\
    \(\widetilde{C_{6}}\) & \(0\) & \(0\) & \(1\) & \(0\) & \(0\) & \(0\) & \(0\) & \(0\) & \(1\) & \(0\) & \(0\) & \(0\) & \(-2\) & \(0\) & \(1\) \\
    \(\widetilde{C_{7}}\) & \(0\) & \(0\) & \(1\) & \(0\) & \(0\) & \(0\) & \(0\) & \(0\) & \(0\) & \(0\) & \(0\) & \(0\) & \(0\) & \(-2\) & \(1\) \\
    \(\widetilde{H_{\mathcal{S}}}\) & \(0\) & \(0\) & \(0\) & \(0\) & \(0\) & \(0\) & \(0\) & \(0\) & \(1\) & \(1\) & \(1\) & \(1\) & \(1\) & \(1\) & \(4\) \\
    \hline
  \end{tabular}.
\end{table}
Note that the intersection matrix is non-degenerate.

Discriminant groups and discriminant forms of the lattices \(L_{\mathcal{S}}\) and \(H \oplus \Pic(X)\) are given by
\begin{gather*}
  G' = 
  \begin{pmatrix}
    \frac{1}{2} & \frac{1}{2} & \frac{1}{2} & \frac{1}{2} & 0 & 0 & 0 & 0 & 0 & \frac{1}{2} & 0 & 0 & 0 & \frac{1}{2} & \frac{1}{2} \\
    \frac{1}{2} & \frac{1}{2} & \frac{1}{2} & 0 & 0 & 0 & \frac{1}{2} & 0 & 0 & 0 & \frac{1}{2} & 0 & \frac{1}{2} & 0 & \frac{1}{2} \\
    \frac{1}{2} & 0 & \frac{1}{2} & \frac{1}{7} & \frac{5}{7} & \frac{4}{7} & \frac{3}{7} & \frac{9}{14} & \frac{2}{7} & \frac{2}{7} & \frac{5}{7} & \frac{2}{7} & \frac{1}{14} & \frac{13}{14} & \frac{5}{14}
  \end{pmatrix}, \;
  B' = 
  \begin{pmatrix}
    0 & \frac{1}{2} & 0 \\
    \frac{1}{2} & 0 & 0 \\
    0 & 0 & \frac{13}{14}
  \end{pmatrix}; \\
  G'' = 
  \begin{pmatrix}
    0 & 0 & -\frac{1}{2} & 0 & \frac{1}{2} & 0 & 0 \\
    0 & 0 & -\frac{1}{2} & \frac{1}{2} & 0 & 0 & 0 \\
    0 & 0 & \frac{1}{7} & \frac{1}{7} & \frac{1}{7} & \frac{5}{7} & -\frac{1}{14}
  \end{pmatrix}, \;
  B'' = 
  \begin{pmatrix}
    0 & \frac{1}{2} & 0 \\
    \frac{1}{2} & 0 & 0 \\
    0 & 0 & \frac{1}{14}
  \end{pmatrix}; \;
  \begin{pmatrix}
    Q' \\ Q''
  \end{pmatrix}
  =
  \begin{pmatrix}
    1 & 1 & \frac{27}{14} \\
    1 & 1 & \frac{1}{14}
  \end{pmatrix}.
\end{gather*}


\subsection{Family \textnumero5.2}\label{subsection:05-02}

The pencil \(\mathcal{S}\) is defined by the equation
\begin{gather*}
  X^{2} Y Z + X Y^{2} Z + X Y Z^{2} + X^{2} Y T + X^{2} Z T + Y^{2} Z T + X Z T^{2} + Y Z T^{2} = \lambda X Y Z T.
\end{gather*}
Members \(\mathcal{S}_{\lambda}\) of the pencil are irreducible for any \(\lambda \in \mathbb{P}^1\) except
\(\mathcal{S}_{\infty} = S_{(X)} + S_{(Y)} + S_{(Z)} + S_{(T)}\).
The base locus of the pencil \(\mathcal{S}\) consists of the following curves:
\begin{gather*}
  C_{1} = C_{(X, Y)}, \;
  C_{2} = C_{(X, Z)}, \;
  C_{3} = C_{(X, T)}, \;
  C_{4} = C_{(Y, Z)}, \;
  C_{5} = C_{(Y, T)}, \\
  C_{6} = C_{(Z, T)}, \;
  C_{7} = C_{(X, Y + T)}, \;
  C_{8} = C_{(Y, X + T)}, \;
  C_{9} = C_{(T, X + Y + Z)}.
\end{gather*}
Their linear equivalence classes on the generic member \(\mathcal{S}_{\Bbbk}\) of the pencil satisfy the following relations:
\begin{gather*}
  \begin{pmatrix}
    [C_{6}] \\ [C_{7}] \\ [C_{8}] \\ [C_{9}]
  \end{pmatrix} = 
  \begin{pmatrix}
    0 & -2 & 0 & -1 & 0 & 1 \\
    -1 & -1 & -1 & 0 & 0 & 1 \\
    -1 & 0 & 0 & -1 & -1 & 1 \\
    0 & 2 & -1 & 1 & -1 & 0
  \end{pmatrix} \cdot
  \begin{pmatrix}
    [C_{1}] & [C_{2}] & [C_{3}] & [C_{4}] & [C_{5}] & [H_{\mathcal{S}}]
  \end{pmatrix}^T.
\end{gather*}

For a general choice of \(\lambda \in \mathbb{C}\) the surface \(\mathcal{S}_{\lambda}\) has the following singularities:
\begin{itemize}\setlength{\itemindent}{2cm}
\item[\(P_{1} = P_{(X, Y, Z)}\):] type \(\mathbb{A}_2\) with the quadratic term \(Z \cdot (X + Y)\);
\item[\(P_{2} = P_{(X, Y, T)}\):] type \(\mathbb{A}_3\) with the quadratic term \(X \cdot Y\);
\item[\(P_{3} = P_{(X, Z, T)}\):] type \(\mathbb{A}_2\) with the quadratic term \(Z \cdot (X + T)\);
\item[\(P_{4} = P_{(Y, Z, T)}\):] type \(\mathbb{A}_1\) with the quadratic term \(Y Z + Y T + Z T\);
\item[\(P_{5} = P_{(X, Z, Y + T)}\):] type \(\mathbb{A}_1\) with the quadratic term \(Z ((\lambda + 2) X - Y - T) - X^2\).
\end{itemize}

Galois action on the lattice \(L_{\lambda}\) is trivial. The intersection matrix on \(L_{\lambda} = L_{\mathcal{S}}\) is represented by
\begin{table}[H]
  \begin{tabular}{|c||cc|ccc|cc|c|c|cccccc|}
    \hline
    \(\bullet\) & \(E_1^1\) & \(E_1^2\) & \(E_2^1\) & \(E_2^2\) & \(E_2^3\) & \(E_3^1\) & \(E_3^2\) & \(E_4^1\) & \(E_5^1\) & \(\widetilde{C_{1}}\) & \(\widetilde{C_{2}}\) & \(\widetilde{C_{3}}\) & \(\widetilde{C_{4}}\) & \(\widetilde{C_{5}}\) & \(\widetilde{H_{\mathcal{S}}}\) \\
    \hline
    \hline
    \(\widetilde{C_{1}}\) & \(1\) & \(0\) & \(0\) & \(1\) & \(0\) & \(0\) & \(0\) & \(0\) & \(0\) & \(-2\) & \(0\) & \(0\) & \(0\) & \(0\) & \(1\) \\
    \(\widetilde{C_{2}}\) & \(0\) & \(1\) & \(0\) & \(0\) & \(0\) & \(1\) & \(0\) & \(0\) & \(1\) & \(0\) & \(-2\) & \(0\) & \(0\) & \(0\) & \(1\) \\
    \(\widetilde{C_{3}}\) & \(0\) & \(0\) & \(1\) & \(0\) & \(0\) & \(0\) & \(1\) & \(0\) & \(0\) & \(0\) & \(0\) & \(-2\) & \(0\) & \(0\) & \(1\) \\
    \(\widetilde{C_{4}}\) & \(0\) & \(1\) & \(0\) & \(0\) & \(0\) & \(0\) & \(0\) & \(1\) & \(0\) & \(0\) & \(0\) & \(0\) & \(-2\) & \(0\) & \(1\) \\
    \(\widetilde{C_{5}}\) & \(0\) & \(0\) & \(0\) & \(0\) & \(1\) & \(0\) & \(0\) & \(1\) & \(0\) & \(0\) & \(0\) & \(0\) & \(0\) & \(-2\) & \(1\) \\
    \(\widetilde{H_{\mathcal{S}}}\) & \(0\) & \(0\) & \(0\) & \(0\) & \(0\) & \(0\) & \(0\) & \(0\) & \(0\) & \(1\) & \(1\) & \(1\) & \(1\) & \(1\) & \(4\) \\
    \hline
  \end{tabular}.
\end{table}
Note that the intersection matrix is non-degenerate.

Discriminant groups and discriminant forms of the lattices \(L_{\mathcal{S}}\) and \(H \oplus \Pic(X)\) are given by
\begin{gather*}
  G' = 
  \begin{pmatrix}
    \frac{9}{44} & \frac{39}{44} & \frac{7}{11} & \frac{9}{44} & \frac{1}{4} & \frac{29}{44} & \frac{19}{22} & \frac{31}{44} & \frac{8}{11} & \frac{23}{44} & \frac{5}{11} & \frac{3}{44} & \frac{5}{44} & \frac{13}{44} & \frac{7}{11}
  \end{pmatrix}, \;
  B' = 
  \begin{pmatrix}
    \frac{31}{44}
  \end{pmatrix}; \\
  G'' =
  \begin{pmatrix}
    0 & 0 & -\frac{13}{44} & \frac{9}{44} & \frac{1}{22} & -\frac{9}{22} & \frac{1}{11}
  \end{pmatrix}, \;
  B'' =
  \begin{pmatrix}
    \frac{13}{44}
  \end{pmatrix}; \;
  Q' =
  \begin{pmatrix}
    \frac{31}{44}
  \end{pmatrix}, \;
  Q'' =
  \begin{pmatrix}
    \frac{57}{44}
  \end{pmatrix}.
\end{gather*}


\subsection{Family \textnumero5.3}\label{subsection:05-03}

The pencil \(\mathcal{S}\) is defined by the equation
\begin{gather*}
  X^{2} Y Z + X Y^{2} Z + X Y Z^{2} + X Y^{2} T + X Z^{2} T + X Y T^{2} + X Z T^{2} + Y Z T^{2} = \lambda X Y Z T.
\end{gather*}
Members \(\mathcal{S}_{\lambda}\) of the pencil are irreducible for any \(\lambda \in \mathbb{P}^1\) except
\(\mathcal{S}_{\infty} = S_{(X)} + S_{(Y)} + S_{(Z)} + S_{(T)}\).
The base locus of the pencil \(\mathcal{S}\) consists of the following curves:
\begin{gather*}
  C_{1} = C_{(X, Y)}, \;
  C_{2} = C_{(X, Z)}, \;
  C_{3} = C_{(X, T)}, \;
  C_{4} = C_{(Y, Z)}, \;
  C_{5} = C_{(Y, T)}, \\
  C_{6} = C_{(Z, T)}, \;
  C_{7} = C_{(Y, Z + T)}, \;
  C_{8} = C_{(Z, Y + T)}, \;
  C_{9} = C_{(T, X + Y + Z)}.
\end{gather*}
Their linear equivalence classes on the generic member \(\mathcal{S}_{\Bbbk}\) of the pencil satisfy the following relations:
\begin{gather*}
  \begin{pmatrix}
    [C_{2}] \\ [C_{7}] \\ [C_{8}] \\ [C_{9}]
  \end{pmatrix} = 
  \begin{pmatrix}
    -1 & -2 & 0 & 0 & 0 & 1 \\
    -1 & 0 & -1 & -1 & 0 & 1 \\
    1 & 2 & -1 & 0 & -1 & 0 \\
    0 & -1 & 0 & -1 & -1 & 1
  \end{pmatrix} \cdot
  \begin{pmatrix}
    [C_{1}] & [C_{3}] & [C_{4}] & [C_{5}] & [C_{6}] & [H_{\mathcal{S}}]
  \end{pmatrix}^T.
\end{gather*}

For a general choice of \(\lambda \in \mathbb{C}\) the surface \(\mathcal{S}_{\lambda}\) has the following singularities:
\begin{itemize}\setlength{\itemindent}{2cm}
\item[\(P_{1} = P_{(X, Y, Z)}\):] type \(\mathbb{A}_1\) with the quadratic term \(Y Z + X (Y + Z)\);
\item[\(P_{2} = P_{(X, Y, T)}\):] type \(\mathbb{A}_2\) with the quadratic term \(X \cdot (Y + T)\);
\item[\(P_{3} = P_{(X, Z, T)}\):] type \(\mathbb{A}_2\) with the quadratic term \(X \cdot (Z + T)\);
\item[\(P_{4} = P_{(Y, Z, T)}\):] type \(\mathbb{A}_3\) with the quadratic term \(Y \cdot Z\);
\item[\(P_{5} = P_{(X, T, Y + Z)}\):] type \(\mathbb{A}_1\) with the quadratic term \(X (X + Y + Z - (\lambda + 2) T) + T^2\).
\end{itemize}

Galois action on the lattice \(L_{\lambda}\) is trivial. The intersection matrix on \(L_{\lambda} = L_{\mathcal{S}}\) is represented by
\begin{table}[H]
  \begin{tabular}{|c||c|cc|cc|ccc|c|cccccc|}
    \hline
    \(\bullet\) & \(E_1^1\) & \(E_2^1\) & \(E_2^2\) & \(E_3^1\) & \(E_3^2\) & \(E_4^1\) & \(E_4^2\) & \(E_4^3\) & \(E_5^1\) & \(\widetilde{C_{1}}\) & \(\widetilde{C_{3}}\) & \(\widetilde{C_{4}}\) & \(\widetilde{C_{5}}\) & \(\widetilde{C_{6}}\) & \(\widetilde{H_{\mathcal{S}}}\) \\
    \hline
    \hline
    \(\widetilde{C_{1}}\) & \(1\) & \(1\) & \(0\) & \(0\) & \(0\) & \(0\) & \(0\) & \(0\) & \(0\) & \(-2\) & \(0\) & \(0\) & \(0\) & \(0\) & \(1\) \\
    \(\widetilde{C_{3}}\) & \(0\) & \(1\) & \(0\) & \(1\) & \(0\) & \(0\) & \(0\) & \(0\) & \(1\) & \(0\) & \(-2\) & \(0\) & \(0\) & \(0\) & \(1\) \\
    \(\widetilde{C_{4}}\) & \(1\) & \(0\) & \(0\) & \(0\) & \(0\) & \(0\) & \(1\) & \(0\) & \(0\) & \(0\) & \(0\) & \(-2\) & \(0\) & \(0\) & \(1\) \\
    \(\widetilde{C_{5}}\) & \(0\) & \(0\) & \(1\) & \(0\) & \(0\) & \(1\) & \(0\) & \(0\) & \(0\) & \(0\) & \(0\) & \(0\) & \(-2\) & \(0\) & \(1\) \\
    \(\widetilde{C_{6}}\) & \(0\) & \(0\) & \(0\) & \(0\) & \(1\) & \(0\) & \(0\) & \(1\) & \(0\) & \(0\) & \(0\) & \(0\) & \(0\) & \(-2\) & \(1\) \\
    \(\widetilde{H_{\mathcal{S}}}\) & \(0\) & \(0\) & \(0\) & \(0\) & \(0\) & \(0\) & \(0\) & \(0\) & \(0\) & \(1\) & \(1\) & \(1\) & \(1\) & \(1\) & \(4\) \\
    \hline
  \end{tabular}.
\end{table}
Note that the intersection matrix is non-degenerate.

Discriminant groups and discriminant forms of the lattices \(L_{\mathcal{S}}\) and \(H \oplus \Pic(X)\) are given by
\begin{gather*}
  G' = 
  \begin{pmatrix}
    \frac{1}{2} & 0 & 0 & 0 & 0 & \frac{1}{2} & 0 & \frac{1}{2} & \frac{1}{2} & 0 & 0 & 0 & 0 & 0 & \frac{1}{2} \\
    0 & 0 & \frac{1}{2} & 0 & 0 & \frac{1}{2} & 0 & 0 & 0 & \frac{1}{2} & 0 & \frac{1}{2} & 0 & 0 & 0 \\
    \frac{1}{12} & \frac{1}{12} & \frac{1}{3} & \frac{3}{4} & \frac{5}{6} & \frac{2}{3} & \frac{3}{4} & \frac{5}{6} & \frac{1}{3} & \frac{1}{6} & \frac{2}{3} & 0 & \frac{7}{12} & \frac{11}{12} & \frac{1}{6}
  \end{pmatrix}, \;
  B' = 
  \begin{pmatrix}
    0 & \frac{1}{2} & \frac{1}{2} \\
    \frac{1}{2} & 0 & \frac{1}{2} \\
    \frac{1}{2} & \frac{1}{2} & \frac{7}{12}
  \end{pmatrix}; \\  
  G'' = 
  \begin{pmatrix}
    0 & 0 & -\frac{1}{2} & 0 & \frac{1}{2} & 0 & 0 \\
    0 & 0 & -\frac{1}{2} & 0 & 0 & 0 & \frac{1}{2} \\
    0 & 0 & -\frac{5}{12} & \frac{1}{12} & \frac{1}{12} & -\frac{1}{6} & \frac{1}{4}
  \end{pmatrix}, \;
  B'' = 
  \begin{pmatrix}
    0 & \frac{1}{2} & \frac{1}{2} \\
    \frac{1}{2} & 0 & \frac{1}{2} \\
    \frac{1}{2} & \frac{1}{2} & \frac{5}{12}
  \end{pmatrix}; \;
  \begin{pmatrix}
    Q' \\ Q''
  \end{pmatrix}
  =
  \begin{pmatrix}
    1 & 0 & \frac{7}{12} \\
    1 & 0 & \frac{17}{12}
  \end{pmatrix}.
\end{gather*}


\subsection{Family \textnumero6.1}\label{subsection:06-01}

The pencil \(\mathcal{S}\) is defined by the equation
\begin{gather*}
  X^{2} Y Z + 2 X Y^{2} Z + Y^{3} Z + X Y Z^{2} + 3 Y^{2} Z T + X Y T^{2} + X Z T^{2} + 3 Y Z T^{2} + Z T^{3} = \lambda X Y Z T.
\end{gather*}
Members \(\mathcal{S}_{\lambda}\) of the pencil are irreducible for any \(\lambda \in \mathbb{P}^1\) except
\(\mathcal{S}_{\infty} = S_{(X)} + S_{(Y)} + S_{(Z)} + S_{(T)}\).
The base locus of the pencil \(\mathcal{S}\) consists of the following curves:
\begin{gather*}
  C_{1} = C_{(X, Z)}, \;
  C_{2} = C_{(Y, Z)}, \;
  C_{3} = C_{(Y, T)}, \;
  C_{4} = C_{(Z, T)}, \\
  C_{5} = C_{(X, Y + T)}, \;
  C_{6} = C_{(Y, X + T)}, \;
  C_{7} = C_{(T, X Z + (X + Y)^2)}.
\end{gather*}
Their linear equivalence classes on the generic member \(\mathcal{S}_{\Bbbk}\) of the pencil satisfy the following relations:
\begin{gather*}
  \begin{pmatrix}
    [C_{1}] \\ [C_{2}] \\ [C_{6}] \\ [C_{7}]
  \end{pmatrix} = 
  \begin{pmatrix}
    0 & 0 & -3 & 1 \\
    0 & -2 & 3 & 0 \\
    -2 & 2 & -3 & 1 \\
    -1 & -1 & 0 & 1
  \end{pmatrix} \cdot
  \begin{pmatrix}
    [C_{3}] \\ [C_{4}] \\ [C_{5}] \\ [H_{\mathcal{S}}]
  \end{pmatrix}.
\end{gather*}

Put \(\mu (\mu - 1) = (\lambda + 1)^{-1}\). For a general choice of \(\lambda \in \mathbb{C}\) the surface \(\mathcal{S}_{\lambda}\) has the following singularities:
\begin{itemize}\setlength{\itemindent}{2cm}
\item[\(P_{1} = P_{(X, Y, T)}\):] type \(\mathbb{A}_2\) with the quadratic term \(X \cdot Y\);
\item[\(P_{2} = P_{(Y, Z, T)}\):] type \(\mathbb{A}_3\) with the quadratic term \(Y \cdot Z\);
\item[\(P_{3} = P_{(Y, T, X + Z)}\):] type \(\mathbb{A}_1\) with the quadratic term \(Y (X + Y + Z - \lambda T) + (Y^2 + T^2)\);
\item[\(P_{4} = P_{(Z, T, X + Y)}\):] type \(\mathbb{A}_3\) with the quadratic term \((\mu Z - (\mu - 1) T) \cdot ((\mu - 1) Z - \mu T)\);
\item[\(P_{5} = P_{(X, \mu Y + (\mu- 1) Z, Y + T)}\):] type \(\mathbb{A}_2\) with the quadratic term
  \[
    X \cdot (\mu^2 (\mu - 1) (X + Y + T) + (2 \mu - 1) ((\mu - 1) Z -  \mu T));
  \]
\item[\(P_{6} = P_{(X, (\mu - 1) Y + \mu Z, Y + T)}\):] type \(\mathbb{A}_2\) with the quadratic term
  \[
    X \cdot (\mu (\mu - 1)^2 (X + Y + T) - (2 \mu - 1) (\mu Z - (\mu - 1) T)).
  \]
\end{itemize}

The only non-trivial Galois orbits on the lattice \(L_{\lambda}\) are
\(E_4^1 + E_4^3, E_5^1 + E_6^1, E_5^2 + E_6^2\).

The intersection matrix on the lattice \(L_{\lambda}\) is represented by
\begin{table}[H]
  \begin{tabular}{|c||cc|ccc|c|ccc|cc|cc|cccc|}
    \hline
    \(\bullet\) & \(E_1^1\) & \(E_1^2\) & \(E_2^1\) & \(E_2^2\) & \(E_2^3\) & \(E_3^1\) & \(E_4^1\) & \(E_4^2\) & \(E_4^3\) & \(E_5^1\) & \(E_5^2\) & \(E_6^1\) & \(E_6^2\) & \(\widetilde{C_{3}}\) & \(\widetilde{C_{4}}\) & \(\widetilde{C_{5}}\) & \(\widetilde{H_{\mathcal{S}}}\) \\
    \hline
    \hline
    \(\widetilde{C_{3}}\) & \(1\) & \(0\) & \(1\) & \(0\) & \(0\) & \(1\) & \(0\) & \(0\) & \(0\) & \(0\) & \(0\) & \(0\) & \(0\) & \(-2\) & \(0\) & \(0\) & \(1\) \\
    \(\widetilde{C_{4}}\) & \(0\) & \(0\) & \(0\) & \(0\) & \(1\) & \(0\) & \(0\) & \(1\) & \(0\) & \(0\) & \(0\) & \(0\) & \(0\) & \(0\) & \(-2\) & \(0\) & \(1\) \\
    \(\widetilde{C_{5}}\) & \(0\) & \(1\) & \(0\) & \(0\) & \(0\) & \(0\) & \(0\) & \(0\) & \(0\) & \(1\) & \(0\) & \(1\) & \(0\) & \(0\) & \(0\) & \(-2\) & \(1\) \\
    \(\widetilde{H_{\mathcal{S}}}\) & \(0\) & \(0\) & \(0\) & \(0\) & \(0\) & \(0\) & \(0\) & \(0\) & \(0\) & \(0\) & \(0\) & \(0\) & \(0\) & \(1\) & \(1\) & \(1\) & \(4\) \\
    \hline
  \end{tabular}.
\end{table}
Note that the intersection matrix is non-degenerate.

Discriminant groups and discriminant forms of the lattices \(L_{\mathcal{S}}\) and \(H \oplus \Pic(X)\) are given by
\begin{gather*}
  G' = 
  \begin{pmatrix}
    0 & 0 & 0 & 0 & 0 & 0 & 0 & 0 & \frac{1}{2} & 0 & 0 & 0 & 0 & 0 \\
    0 & 0 & 0 & 0 & 0 & 0 & 0 & 0 & \frac{1}{2} & \frac{1}{2} & 0 & 0 & 0 & 0 \\
    0 & 0 & \frac{1}{2} & 0 & \frac{1}{2} & \frac{1}{2} & 0 & \frac{1}{2} & 0 & \frac{1}{2} & 0 & 0 & 0 & 0 \\
    \frac{1}{5} & \frac{4}{5} & \frac{9}{10} & \frac{1}{5} & \frac{1}{2} & \frac{3}{10} & \frac{9}{10} & \frac{3}{10} & \frac{3}{5} & \frac{3}{10} & \frac{3}{5} & \frac{4}{5} & \frac{2}{5} & \frac{4}{5}
  \end{pmatrix}, \;
  B' = 
  \begin{pmatrix}
    0 & \frac{1}{2} & \frac{1}{2} & \frac{1}{2} \\
    \frac{1}{2} & 0 & \frac{1}{2} & \frac{1}{2} \\
    \frac{1}{2} & \frac{1}{2} & 0 & \frac{1}{2} \\
    \frac{1}{2} & \frac{1}{2} & \frac{1}{2} & \frac{2}{5}
  \end{pmatrix};
\end{gather*}
\begin{gather*}
  G'' = 
  \begin{pmatrix}
    0 & 0 & -\frac{1}{2} & 0 & \frac{1}{2} & 0 & 0 & 0 \\
    0 & 0 & -\frac{1}{2} & 0 & 0 & \frac{1}{2} & 0 & 0 \\
    0 & 0 & -\frac{1}{2} & \frac{1}{2} & 0 & 0 & 0 & 0 \\
    0 & 0 & -\frac{1}{5} & \frac{3}{10} & \frac{3}{10} & \frac{3}{10} & \frac{2}{5} & -\frac{1}{10}
  \end{pmatrix}, \;
  B'' = 
  \begin{pmatrix}
    0 & \frac{1}{2} & \frac{1}{2} & \frac{1}{2} \\
    \frac{1}{2} & 0 & \frac{1}{2} & \frac{1}{2} \\
    \frac{1}{2} & \frac{1}{2} & 0 & \frac{1}{2} \\
    \frac{1}{2} & \frac{1}{2} & \frac{1}{2} & \frac{3}{5}
  \end{pmatrix}; \;
  \begin{pmatrix}
    Q' \\ Q''
  \end{pmatrix}
  =
  \begin{pmatrix}
    1 & 1 & 1 & \frac{7}{5} \\
    1 & 1 & 1 & \frac{3}{5}
  \end{pmatrix}.
\end{gather*}


\subsection{Family \textnumero7.1}\label{subsection:07-01}

The pencil \(\mathcal{S}\) is defined by the equation
\begin{gather*}
  X^{3} Z + 3 X^{2} Y Z + 3 X Y^{2} Z + Y^{3} Z + X Y Z^{2} + 2 X^{2} Z T + 2 Y^{2} Z T + X Y T^{2} + X Z T^{2} + Y Z T^{2} = \lambda X Y Z T.
\end{gather*}
Members \(\mathcal{S}_{\lambda}\) of the pencil are irreducible for any \(\lambda \in \mathbb{P}^1\) except
\(\mathcal{S}_{\infty} = S_{(X)} + S_{(Y)} + S_{(Z)} + S_{(T)}\).
The base locus of the pencil \(\mathcal{S}\) consists of the following curves:
\begin{gather*}
  C_{1} = C_{(X, Y)}, \;
  C_{2} = C_{(X, Z)}, \;
  C_{3} = C_{(Y, Z)}, \;
  C_{4} = C_{(Z, T)}, \\
  C_{5} = C_{(X, Y + T)}, \;
  C_{6} = C_{(Y, X + T)}, \;
  C_{7} = C_{(T, (X + Y)^3 + X Y Z)}.
\end{gather*}
Their linear equivalence classes on the generic member \(\mathcal{S}_{\Bbbk}\) of the pencil satisfy the following relations:
\begin{gather*}
  \begin{pmatrix}
    [C_{2}] \\ [C_{3}] \\ [C_{7}] \\ [H_{\mathcal{S}}]
  \end{pmatrix} = 
  \begin{pmatrix}
    1 & -2 & 0 & 2 \\
    1 & -2 & 2 & 0 \\
    2 & -3 & 2 & 2 \\
    2 & -2 & 2 & 2
  \end{pmatrix} \cdot
  \begin{pmatrix}
    [C_{1}] \\ [C_{4}] \\ [C_{5}] \\ [C_{6}]
  \end{pmatrix}.
\end{gather*}

Put \(\mu (\mu - 1) = (\lambda + 2)^{-1}\). For a general choice of \(\lambda \in \mathbb{C}\) the surface \(\mathcal{S}_{\lambda}\) has the following singularities:
\begin{itemize}\setlength{\itemindent}{2cm}
\item[\(P_{1} = P_{(X, Y, Z)}\):] type \(\mathbb{A}_1\) with the quadratic term \(X Y + Z (X + Y)\);
\item[\(P_{2} = P_{(X, Y, T)}\):] type \(\mathbb{A}_3\) with the quadratic term \(X \cdot Y\);
\item[\(P_{3} = P_{(Z, T, X + Y)}\):] type \(\mathbb{A}_5\) with the quadratic term \((\mu Z - (\mu - 1) T) \cdot ((\mu - 1) Z - \mu T)\);
\item[\(P_{4} = P_{(X, Y + T, (\mu - 1) Z - \mu T)}\):] type \(\mathbb{A}_1\);
\item[\(P_{5} = P_{(X, Y + T, \mu Z - (\mu - 1) T)}\):] type \(\mathbb{A}_1\);
\item[\(P_{6} = P_{(Y, X + T, (\mu - 1) Z - \mu T)}\):] type \(\mathbb{A}_1\);
\item[\(P_{7} = P_{(Y, X + T, \mu Z - (\mu - 1) T)}\):] type \(\mathbb{A}_1\).
\end{itemize}

The only non-trivial Galois orbits on the lattice \(L_{\lambda}\) are
\(E_3^1 + E_3^5, E_3^2 + E_3^4, E_4^1 + E_5^1, E_6^1 + E_7^1\).

The intersection matrix on the lattice \(L_{\lambda}\) is represented by
\begin{table}[H]
  \begin{tabular}{|c||c|ccc|ccccc|c|c|c|c|cccc|}
    \hline
    \(\bullet\) & \(E_1^1\) & \(E_2^1\) & \(E_2^2\) & \(E_2^3\) & \(E_3^1\) & \(E_3^2\) & \(E_3^3\) & \(E_3^4\) & \(E_3^5\) & \(E_4^1\) & \(E_5^1\) & \(E_6^1\) & \(E_7^1\) & \(\widetilde{C_{1}}\) & \(\widetilde{C_{4}}\) & \(\widetilde{C_{5}}\) & \(\widetilde{C_{6}}\) \\
    \hline
    \hline
    \(\widetilde{C_{1}}\) & \(1\) & \(0\) & \(1\) & \(0\) & \(0\) & \(0\) & \(0\) & \(0\) & \(0\) & \(0\) & \(0\) & \(0\) & \(0\) & \(-2\) & \(0\) & \(0\) & \(0\) \\
    \(\widetilde{C_{4}}\) & \(0\) & \(0\) & \(0\) & \(0\) & \(0\) & \(0\) & \(1\) & \(0\) & \(0\) & \(0\) & \(0\) & \(0\) & \(0\) & \(0\) & \(-2\) & \(0\) & \(0\) \\
    \(\widetilde{C_{5}}\) & \(0\) & \(1\) & \(0\) & \(0\) & \(0\) & \(0\) & \(0\) & \(0\) & \(0\) & \(1\) & \(1\) & \(0\) & \(0\) & \(0\) & \(0\) & \(-2\) & \(0\) \\
    \(\widetilde{C_{6}}\) & \(0\) & \(0\) & \(0\) & \(1\) & \(0\) & \(0\) & \(0\) & \(0\) & \(0\) & \(0\) & \(0\) & \(1\) & \(1\) & \(0\) & \(0\) & \(0\) & \(-2\) \\
    \hline
  \end{tabular}.
\end{table}
Note that the intersection matrix is non-degenerate.

Discriminant groups and discriminant forms of the lattices \(L_{\mathcal{S}}\) and \(H \oplus \Pic(X)\) are given by
\begin{gather*}
  G' = 
  \begin{pmatrix}
    \frac{1}{2} & 0 & \frac{1}{2} & 0 & 0 & 0 & 0 & 0 & 0 & 0 & 0 & \frac{1}{2} & \frac{1}{2} \\
    0 & 0 & 0 & 0 & 0 & \frac{1}{2} & 0 & \frac{1}{2} & 0 & 0 & 0 & 0 & 0 \\
    0 & 0 & 0 & 0 & 0 & \frac{1}{2} & 0 & 0 & 0 & 0 & 0 & 0 & 0 \\
    \frac{1}{2} & 0 & \frac{1}{2} & 0 & \frac{1}{2} & 0 & 0 & 0 & 0 & 0 & 0 & \frac{1}{2} & \frac{1}{2} \\
    \frac{3}{4} & \frac{1}{4} & \frac{1}{4} & \frac{3}{4} & 0 & 0 & 0 & \frac{5}{8} & \frac{3}{8} & \frac{1}{2} & 0 & \frac{1}{4} & \frac{1}{4}
  \end{pmatrix}, \;
  G'' = 
  \begin{pmatrix}
    0 & 0 & \frac{1}{2} & \frac{1}{2} & \frac{1}{2} & 0 & \frac{1}{2} & 0 & 0 \\
    0 & 0 & \frac{1}{2} & \frac{1}{2} & 0 & \frac{1}{2} & \frac{1}{2} & 0 & 0 \\
    0 & 0 & \frac{1}{2} & 0 & 0 & 0 & \frac{1}{2} & 0 & 0 \\
    0 & 0 & \frac{1}{2} & \frac{1}{2} & 0 & 0 & 0 & 0 & 0 \\
    0 & 0 & \frac{1}{8} & \frac{1}{8} & \frac{1}{8} & \frac{1}{8} & \frac{1}{8} & \frac{3}{4} & \frac{7}{8}
  \end{pmatrix};
\end{gather*}
\begin{gather*}
  B' = 
  \begin{pmatrix}
    0 & \frac{1}{2} & 0 & 0 & 0 \\
    \frac{1}{2} & 0 & 0 & 0 & 0 \\
    0 & 0 & 0 & \frac{1}{2} & 0 \\
    0 & 0 & \frac{1}{2} & 0 & 0 \\
    0 & 0 & 0 & 0 & \frac{5}{8}
  \end{pmatrix}, \;
  B'' = 
  \begin{pmatrix}
    0 & \frac{1}{2} & 0 & 0 & 0 \\
    \frac{1}{2} & 0 & 0 & 0 & 0 \\
    0 & 0 & 0 & \frac{1}{2} & 0 \\
    0 & 0 & \frac{1}{2} & 0 & 0 \\
    0 & 0 & 0 & 0 & \frac{3}{8}
  \end{pmatrix}; \;
  \begin{pmatrix}
    Q' \\ Q''
  \end{pmatrix}
  =
  \begin{pmatrix}
    0 & 0 & 1 & 1 & \frac{13}{8} \\
    0 & 0 & 1 & 1 & \frac{3}{8}
  \end{pmatrix}.
\end{gather*}


\subsection{Family \textnumero8.1}\label{subsection:08-01}

The pencil \(\mathcal{S}\) is defined by the equation
\begin{gather*}
  X Y^{3} + X^{2} Y Z + 3 X Y^{2} Z + 3 X Y Z^{2} + X Z^{3} + 3 X Y^{2} T + \\ 3 X Z^{2} T + 3 X Y T^{2} + 3 X Z T^{2} + Y Z T^{2} + X T^{3} = \lambda X Y Z T.
\end{gather*}
Members \(\mathcal{S}_{\lambda}\) of the pencil are irreducible for any \(\lambda \in \mathbb{P}^1\) except
\(\mathcal{S}_{\infty} = S_{(X)} + S_{(Y)} + S_{(Z)} + S_{(T)}\).
The base locus of the pencil \(\mathcal{S}\) consists of the following curves:
\begin{gather*}
  C_{1} = C_{(X, Y)}, \;
  C_{2} = C_{(X, Z)}, \;
  C_{3} = C_{(X, T)}, \;
  C_{4} = C_{(Y, Z + T)}, \;
  C_{5} = C_{(Z, Y + T)}, \;
  C_{6} = C_{(T, (Y + Z)^3 + X Y Z)}.
\end{gather*}
Their linear equivalence classes on the generic member \(\mathcal{S}_{\Bbbk}\) of the pencil satisfy the following relations:
\begin{gather*}
  \begin{pmatrix}
    [C_{1}] \\ [C_{2}] \\ [C_{6}] \\ [H_{\mathcal{S}}]
  \end{pmatrix} = 
  \begin{pmatrix}
    -2 & 0 & 3 \\
    -2 & 3 & 0 \\
    -3 & 3 & 3 \\
    -2 & 3 & 3
  \end{pmatrix} \cdot
  \begin{pmatrix}
    [C_{3}] \\ [C_{4}] \\ [C_{5}]
  \end{pmatrix}.
\end{gather*}

Put \(\mu (\mu - 1) = (\lambda + 4)^{-1}\). For a general choice of \(\lambda \in \mathbb{C}\) the surface \(\mathcal{S}_{\lambda}\) has the following singularities:
\begin{itemize}\setlength{\itemindent}{2cm}
\item[\(P_{1} = P_{(Y, Z, T)}\):] type \(\mathbb{A}_2\) with the quadratic term \(Y \cdot Z\);
\item[\(P_{2} = P_{(X, T, Y + Z)}\):] type \(\mathbb{A}_5\) with the quadratic term \((\mu X - (\mu - 1) T) \cdot ((\mu - 1) X - \mu T)\);
\item[\(P_{3} = P_{(Y, (\mu - 1) X - \mu T, Z + T)}\):] type \(\mathbb{A}_2\) with the quadratic term \(Y \cdot ((\mu - 1) X - \mu T)\);
\item[\(P_{4} = P_{(Y, \mu X - (\mu - 1) T, Z + T)}\):] type \(\mathbb{A}_2\) with the quadratic term \(Y \cdot (\mu X - (\mu - 1) T)\);
\item[\(P_{5} = P_{(Z, (\mu - 1) X - \mu T, Y + T)}\):] type \(\mathbb{A}_2\) with the quadratic term \(Z \cdot (\mu X - (\mu - 1) T)\);
\item[\(P_{6} = P_{(Z, \mu X - (\mu - 1) T, Y + T)}\):] type \(\mathbb{A}_2\) with the quadratic term \(Z \cdot ((\mu - 1) X - \mu T)\).
\end{itemize}

The only non-trivial Galois orbits on the lattice \(L_{\lambda}\) are 
\[
  E_2^1 + E_2^5, E_2^2 + E_2^4, E_3^1 + E_4^1, E_3^2 + E_4^2, E_5^1 + E_6^1, E_5^2 + E_6^2.
\]

The intersection matrix on the lattice \(L_{\lambda}\) is represented by
\begin{table}[H]
  \begin{tabular}{|c||cc|ccccc|cc|cc|cc|cc|ccc|}
    \hline
    \(\bullet\) & \(E_1^1\) & \(E_1^2\) & \(E_2^1\) & \(E_2^2\) & \(E_2^3\) & \(E_2^4\) & \(E_2^5\) & \(E_3^1\) & \(E_3^2\) & \(E_4^1\) & \(E_4^2\) & \(E_5^1\) & \(E_5^2\) & \(E_6^1\) & \(E_6^2\) & \(\widetilde{C_{3}}\) & \(\widetilde{C_{4}}\) & \(\widetilde{C_{5}}\) \\
    \hline
    \hline
    \(\widetilde{C_{3}}\) & \(0\) & \(0\) & \(0\) & \(0\) & \(1\) & \(0\) & \(0\) & \(0\) & \(0\) & \(0\) & \(0\) & \(0\) & \(0\) & \(0\) & \(0\) & \(-2\) & \(0\) & \(0\) \\
    \(\widetilde{C_{4}}\) & \(1\) & \(0\) & \(0\) & \(0\) & \(0\) & \(0\) & \(0\) & \(1\) & \(0\) & \(1\) & \(0\) & \(0\) & \(0\) & \(0\) & \(0\) & \(0\) & \(-2\) & \(0\) \\
    \(\widetilde{C_{5}}\) & \(0\) & \(1\) & \(0\) & \(0\) & \(0\) & \(0\) & \(0\) & \(0\) & \(0\) & \(0\) & \(0\) & \(1\) & \(0\) & \(1\) & \(0\) & \(0\) & \(0\) & \(-2\) \\
    \hline
  \end{tabular}.
\end{table}
Note that the intersection matrix is non-degenerate.

Discriminant groups and discriminant forms of the lattices \(L_{\mathcal{S}}\) and \(H \oplus \Pic(X)\) are given by
\begin{gather*}
  G' = 
  \begin{pmatrix}
    0 & 0 & 0 & 0 & 0 & 0 & \frac{1}{2} & \frac{1}{2} & 0 & 0 & 0 & 0 \\
    0 & 0 & 0 & 0 & 0 & 0 & \frac{1}{2} & \frac{1}{2} & \frac{1}{2} & 0 & 0 & 0 \\
    0 & 0 & 0 & 0 & 0 & \frac{1}{2} & 0 & 0 & \frac{1}{2} & 0 & 0 & 0 \\
    0 & 0 & 0 & 0 & 0 & \frac{1}{2} & \frac{1}{2} & 0 & \frac{1}{2} & 0 & 0 & 0 \\
    0 & 0 & 0 & \frac{1}{2} & 0 & 0 & 0 & 0 & 0 & 0 & 0 & 0 \\
    0 & 0 & \frac{1}{2} & 0 & 0 & 0 & 0 & 0 & 0 & 0 & 0 & 0 \\
    \frac{2}{3} & \frac{1}{3} & 0 & 0 & 0 & \frac{2}{3} & \frac{1}{3} & \frac{1}{3} & \frac{2}{3} & 0 & 0 & 0
  \end{pmatrix}, \;
  G'' = 
  \begin{pmatrix}
    0 & 0 & \frac{1}{2} & \frac{1}{2} & \frac{1}{2} & 0 & \frac{1}{2} & \frac{1}{2} & 0 & \frac{1}{2} \\
    0 & 0 & \frac{1}{2} & \frac{1}{2} & 0 & \frac{1}{2} & \frac{1}{2} & \frac{1}{2} & 0 & \frac{1}{2} \\
    0 & 0 & 0 & \frac{1}{2} & 0 & 0 & 0 & 0 & 0 & \frac{1}{2} \\
    0 & 0 & \frac{1}{2} & 0 & 0 & 0 & 0 & 0 & 0 & \frac{1}{2} \\
    0 & 0 & \frac{1}{2} & \frac{1}{2} & 0 & 0 & \frac{1}{2} & 0 & 0 & \frac{1}{2} \\
    0 & 0 & \frac{1}{2} & \frac{1}{2} & 0 & 0 & 0 & \frac{1}{2} & 0 & \frac{1}{2} \\
    0 & 0 & \frac{2}{3} & \frac{2}{3} & \frac{2}{3} & \frac{2}{3} & \frac{2}{3} & \frac{2}{3} & \frac{2}{3} & 0
  \end{pmatrix};
\end{gather*}
\begin{gather*}
  B' = 
  \begin{pmatrix}
    0 & \frac{1}{2} & 0 & 0 & 0 & 0 & 0 \\
    \frac{1}{2} & 0 & 0 & 0 & 0 & 0 & 0 \\
    0 & 0 & 0 & \frac{1}{2} & 0 & 0 & 0 \\
    0 & 0 & \frac{1}{2} & 0 & 0 & 0 & 0 \\
    0 & 0 & 0 & 0 & 0 & \frac{1}{2} & 0 \\
    0 & 0 & 0 & 0 & \frac{1}{2} & 0 & 0 \\
    0 & 0 & 0 & 0 & 0 & 0 & \frac{2}{3}
  \end{pmatrix}, \;
  B'' = 
  \begin{pmatrix}
    0 & \frac{1}{2} & 0 & 0 & 0 & 0 & 0 \\
    \frac{1}{2} & 0 & 0 & 0 & 0 & 0 & 0 \\
    0 & 0 & 0 & \frac{1}{2} & 0 & 0 & 0 \\
    0 & 0 & \frac{1}{2} & 0 & 0 & 0 & 0 \\
    0 & 0 & 0 & 0 & 0 & \frac{1}{2} & 0 \\
    0 & 0 & 0 & 0 & \frac{1}{2} & 0 & 0 \\
    0 & 0 & 0 & 0 & 0 & 0 & \frac{1}{3}
  \end{pmatrix}; \;
  \begin{pmatrix}
    Q' \\ Q''
  \end{pmatrix}
  =
  \begin{pmatrix}
    0 & 0 & 0 & 0 & 1 & 1 & \frac{2}{3} \\
    0 & 0 & 0 & 0 & 1 & 1 & \frac{4}{3}    
  \end{pmatrix}.
\end{gather*}


\subsection{Family \textnumero9.1}\label{subsection:09-01}

The pencil \(\mathcal{S}\) is defined by the equation
\(X^3 Y = (Y^2 + Z^2 - \lambda Y Z) (X (Z - T) + T^2)\).
Members \(\mathcal{S}_{\lambda}\) of the pencil are irreducible for any \(\lambda \in \mathbb{P}^1\) except
\(\mathcal{S}_{\infty} = S_{(Y)} + S_{(Z)} + S_{(X (Z - T) + T^2)}\).
The base locus of the pencil \(\mathcal{S}\) consists of the following curves:
\begin{gather*}
  C_{1} = C_{(X, T)}, \;
  C_{2} = C_{(Y, Z)}, \;
  C_{3} = C_{(Y, X (Z - T) + T^2)}, \;
  C_{4} = C_{(Z, X^3 + Y T (X - T))}.
\end{gather*}
Their linear equivalence classes on the generic member \(\mathcal{S}_{\Bbbk}\) of the pencil satisfy the following relations:
\begin{gather*}
  \begin{pmatrix}
    [C_{3}] \\ [C_{4}]
  \end{pmatrix} = 
  \begin{pmatrix}
    0 & -2 & 1 \\
    0 & -1 & 1
  \end{pmatrix} \cdot
  \begin{pmatrix}
    [C_{1}] & [C_{2}] & [H_{\mathcal{S}}]
  \end{pmatrix}^T.
\end{gather*}

Put \(\mu (\mu - 1) = (\lambda - 2)^{-1}\). For a general choice of \(\lambda \in \mathbb{C}\) the surface \(\mathcal{S}_{\lambda}\) has the following singularities:
\begin{itemize}\setlength{\itemindent}{2cm}
\item[\(P_{1} = P_{(X, Y, Z)}\):] type \(\mathbb{A}_5\) with the quadratic term \((\mu Y - (\mu - 1) Z) \cdot ((\mu - 1) Y - \mu Z)\);
\item[\(P_{2} = P_{(X, Z, T)}\):] type \(\mathbb{A}_1\) with the quadratic term \(X (Z - T) + T^2\);
\item[\(P_{3} = P_{(X, T, (\mu - 1) Y - \mu Z)}\):] type \(\mathbb{A}_5\) with the quadratic term \(X \cdot ((\mu - 1) Y - \mu Z)\);
\item[\(P_{4} = P_{(X, T, \mu Y - (\mu - 1) Z)}\):] type \(\mathbb{A}_5\) with the quadratic term \(X \cdot (\mu Y - (\mu - 1) Z)\).
\end{itemize}

The only non-trivial Galois orbits on the lattice \(L_{\lambda}\) are
\[
  E_1^1 + E_1^5, E_1^2 + E_1^4, E_3^1 + E_4^1, E_3^2 + E_4^2, E_3^3 + E_4^3, E_3^4 + E_4^4, E_3^5 + E_4^5.
\]

The intersection matrix on the lattice \(L_{\lambda}\) is represented by
\begin{table}[H]
  \begin{tabular}{|c||ccccc|c|ccccc|ccccc|ccc|}
    \hline
    \(\bullet\) & \(E_1^1\) & \(E_1^2\) & \(E_1^3\) & \(E_1^4\) & \(E_1^5\) & \(E_2^1\) & \(E_3^1\) & \(E_3^2\) & \(E_3^3\) & \(E_3^4\) & \(E_3^5\) & \(E_4^1\) & \(E_4^2\) & \(E_4^3\) & \(E_4^4\) & \(E_4^5\) & \(\widetilde{C_{1}}\) & \(\widetilde{C_{2}}\) & \(\widetilde{H_{\mathcal{S}}}\) \\
    \hline
    \hline
    \(\widetilde{C_{1}}\) & \(0\) & \(0\) & \(0\) & \(0\) & \(0\) & \(1\) & \(1\) & \(0\) & \(0\) & \(0\) & \(0\) & \(1\) & \(0\) & \(0\) & \(0\) & \(0\) & \(-2\) & \(0\) & \(1\) \\
    \(\widetilde{C_{2}}\) & \(0\) & \(0\) & \(1\) & \(0\) & \(0\) & \(0\) & \(0\) & \(0\) & \(0\) & \(0\) & \(0\) & \(0\) & \(0\) & \(0\) & \(0\) & \(0\) & \(0\) & \(-2\) & \(1\) \\
    \(\widetilde{H_{\mathcal{S}}}\) & \(0\) & \(0\) & \(0\) & \(0\) & \(0\) & \(0\) & \(0\) & \(0\) & \(0\) & \(0\) & \(0\) & \(0\) & \(0\) & \(0\) & \(0\) & \(0\) & \(1\) & \(1\) & \(4\) \\
    \hline
  \end{tabular}.
\end{table}
Note that the intersection matrix is degenerate. We choose the following integral basis of the lattice \(L_{\mathcal{S}}\):
\begin{align*}
  \begin{pmatrix}
    [\widetilde{H_{\mathcal{S}}}]
  \end{pmatrix} =
  \begin{pmatrix}
    -1 & -2 & -3 & 3 & 5 & 4 & 3 & 2 & 1 & 6 & -2
  \end{pmatrix} \cdot \\
  \begin{pmatrix}
    [E_1^1 + E_1^5] & [E_1^2 + E_1^4] & [E_1^3] & [E_2^1] & [E_3^1 + E_4^1] & [E_3^2 + E_4^2] \\ [E_3^3 + E_4^3] & [E_3^4 + E_4^4] & [E_3^5 + E_4^5] & [\widetilde{C_{1}}] & [\widetilde{C_{2}}]
  \end{pmatrix}^T.
\end{align*}

Discriminant groups and discriminant forms of the lattices \(L_{\mathcal{S}}\) and \(H \oplus \Pic(X)\) are given by
\begin{gather*}
  G' = 
  \begin{pmatrix}
    0 & 0 & 0 & 0 & \frac{1}{2} & \frac{1}{2} & 0 & \frac{1}{2} & 0 & 0 & 0 \\
    0 & 0 & 0 & 0 & 0 & 0 & \frac{1}{2} & 0 & \frac{1}{2} & 0 & 0 \\
    0 & 0 & 0 & 0 & \frac{1}{2} & 0 & 0 & \frac{1}{2} & \frac{1}{2} & 0 & 0 \\
    0 & 0 & 0 & 0 & \frac{1}{2} & 0 & 0 & 0 & \frac{1}{2} & 0 & 0 \\
    0 & \frac{1}{2} & 0 & 0 & \frac{1}{2} & 0 & \frac{1}{2} & 0 & \frac{1}{2} & 0 & 0 \\
    \frac{1}{2} & \frac{1}{2} & 0 & 0 & \frac{1}{2} & 0 & \frac{1}{2} & 0 & \frac{1}{2} & 0 & 0 \\
    \frac{1}{2} & 0 & 0 & \frac{1}{2} & \frac{1}{4} & 0 & \frac{1}{4} & 0 & \frac{3}{4} & 0 & 0
  \end{pmatrix}, \;
  G'' = 
  \begin{pmatrix}
    0 & 0 & 0 & 0 & 0 & \frac{1}{2} & 0 & 0 & 0 & 0 & \frac{1}{2} \\
    0 & 0 & 0 & 0 & \frac{1}{2} & 0 & 0 & 0 & 0 & 0 & \frac{1}{2} \\
    0 & 0 & 0 & \frac{1}{2} & \frac{1}{2} & \frac{1}{2} & \frac{1}{2} & \frac{1}{2} & 0 & 0 & \frac{1}{2} \\
    0 & 0 & 0 & 0 & \frac{1}{2} & \frac{1}{2} & \frac{1}{2} & \frac{1}{2} & \frac{1}{2} & 0 & \frac{1}{2} \\
    0 & 0 & \frac{1}{2} & \frac{1}{2} & 0 & 0 & 0 & \frac{1}{2} & \frac{1}{2} & 0 & 0 \\
    0 & 0 & \frac{1}{2} & \frac{1}{2} & \frac{1}{2} & \frac{1}{2} & 0 & 0 & \frac{1}{2} & 0 & \frac{1}{2} \\
    0 & 0 & \frac{3}{4} & \frac{1}{4} & \frac{3}{4} & \frac{3}{4} & \frac{1}{4} & \frac{3}{4} & \frac{1}{4} & \frac{1}{2} & \frac{1}{4}
  \end{pmatrix};
\end{gather*}
\begin{gather*}
  B' = 
  \begin{pmatrix}
    0 & \frac{1}{2} & 0 & 0 & 0 & 0 & 0 \\
    \frac{1}{2} & 0 & 0 & 0 & 0 & 0 & 0 \\
    0 & 0 & 0 & \frac{1}{2} & 0 & 0 & 0 \\
    0 & 0 & \frac{1}{2} & 0 & 0 & 0 & 0 \\
    0 & 0 & 0 & 0 & 0 & \frac{1}{2} & 0 \\
    0 & 0 & 0 & 0 & \frac{1}{2} & 0 & 0 \\
    0 & 0 & 0 & 0 & 0 & 0 & \frac{3}{4}
  \end{pmatrix}, \;
  B'' = 
  \begin{pmatrix}
    0 & \frac{1}{2} & 0 & 0 & 0 & 0 & 0 \\
    \frac{1}{2} & 0 & 0 & 0 & 0 & 0 & 0 \\
    0 & 0 & 0 & \frac{1}{2} & 0 & 0 & 0 \\
    0 & 0 & \frac{1}{2} & 0 & 0 & 0 & 0 \\
    0 & 0 & 0 & 0 & 0 & \frac{1}{2} & 0 \\
    0 & 0 & 0 & 0 & \frac{1}{2} & 0 & 0 \\
    0 & 0 & 0 & 0 & 0 & 0 & \frac{1}{4}
  \end{pmatrix}; \;
  \begin{pmatrix}
    Q' \\ Q''
  \end{pmatrix}
  =
  \begin{pmatrix}
    0 & 0 & 0 & 0 & 0 & 0 & \frac{7}{4} \\ 
    0 & 0 & 0 & 0 & 0 & 0 & \frac{1}{4}
  \end{pmatrix}.
\end{gather*}


\subsection{Family \textnumero10.1}\label{subsection:10-01}

The pencil \(\mathcal{S}\) is defined by the equation
\(X Y C^3 = (A^3 - B C (A - B)) (X^2 + Y^2 - \lambda X Y)\).
Members \(\mathcal{S}_{\lambda}\) of the pencil are irreducible for any \(\lambda \in \mathbb{P}^1\) except
\(\mathcal{S}_{\infty} = S_{(X)} + S_{(Y)} + S_{(A^3 - B C (A - B))}\).
The base locus of the pencil \(\mathcal{S}\) consists of the following curves:
\[
  C_1 = C_{(A, C)}, \;
  C_2 = C_{(X, A^3 - B C (A - B))}, \;
  C_3 = C_{(Y, A^3 - B C (A - B))}.
\]
Their linear equivalence classes on the generic member \(\mathcal{S}_{\Bbbk}\) of the pencil satisfy the relations \([C_{3}] = [H_{\mathcal{S}}^{(1)}] = [C_{2}]\).

Put \(\mu (\mu - 1) = (\lambda + 4)^{-1}\). For a general choice of \(\lambda \in \mathbb{C}\) the surface \(\mathcal{S}_{\lambda}\) has the following singularities:
\begin{itemize}\setlength{\itemindent}{2cm}
\item[\(P_1 = P_{(A, C, \mu X - (\mu - 1) Y)}\):] type \(\mathbb{A}_8\);
\item[\(P_2 = P_{(A, C, (\mu - 1) X - \mu Y)}\):] type \(\mathbb{A}_8\).
\end{itemize}

The only non-trivial Galois orbits on the lattice \(L_{\lambda}\) are
\[
  (E_1^1, E_2^1), \; (E_1^2, E_2^2), \; (E_1^3, E_2^3), \; (E_1^4, E_2^4), \; (E_1^5, E_2^5), \; (E_1^6, E_2^6), \; (E_1^7, E_2^7), \; (E_1^8, E_2^8).
\]

The intersection matrix on the lattice \(L_{\lambda}\) is represented by
\begin{table}[H]
  \begin{tabular}{|c||cccccccc|cccccccc|ccc|}
    \hline
    \(\bullet\) & \(E_1^1\) & \(E_1^2\) & \(E_1^3\) & \(E_1^4\) & \(E_1^5\) & \(E_1^6\) & \(E_1^7\) & \(E_1^8\) & \(E_2^1\) & \(E_2^2\) & \(E_2^3\) & \(E_2^4\) & \(E_2^5\) & \(E_2^6\) & \(E_2^7\) & \(E_2^8\) & \(\widetilde{C_{1}}\) & \(\widetilde{C_{2}}\) & \(\widetilde{H_{\mathcal{S}}^{(2)}}\) \\
    \hline
    \hline
    \(\widetilde{C_{1}}\) & \(1\) & \(0\) & \(0\) & \(0\) & \(0\) & \(0\) & \(0\) & \(0\) & \(1\) & \(0\) & \(0\) & \(0\) & \(0\) & \(0\) & \(0\) & \(0\) & \(-2\) & \(1\) & \(0\) \\
    \(\widetilde{C_{2}}\) & \(0\) & \(0\) & \(0\) & \(0\) & \(0\) & \(0\) & \(0\) & \(0\) & \(0\) & \(0\) & \(0\) & \(0\) & \(0\) & \(0\) & \(0\) & \(0\) & \(1\) & \(0\) & \(3\) \\
    \(\widetilde{H_{\mathcal{S}}^{(2)}}\) & \(0\) & \(0\) & \(0\) & \(0\) & \(0\) & \(0\) & \(0\) & \(0\) & \(0\) & \(0\) & \(0\) & \(0\) & \(0\) & \(0\) & \(0\) & \(0\) & \(0\) & \(3\) & \(2\) \\
    \hline
  \end{tabular}.
\end{table}
Note that the intersection matrix is degenerate. We choose the following integral basis of the lattice \(L_{\mathcal{S}}\):
\begin{align*}                                                          
  \begin{pmatrix}
    [E_1^8 + E_2^8]
  \end{pmatrix} =                                                  
  \begin{pmatrix}
    -8 & -7 & -6 & -5 & -4 & -3 & -2 & -9 & -2 & 3
  \end{pmatrix} \cdot \\
  \begin{pmatrix}
    E_1^1 + E_2^1 & E_1^2 + E_2^2 & E_1^3 + E_2^3 & E_1^4 + E_2^4 & E_1^5 + E_2^5 & E_1^6 + E_2^6 & E_1^7 + E_2^7 & \widetilde{C_1} & \widetilde{C_2} & \widetilde{H_{\mathcal{S}}^{(2)}}
  \end{pmatrix}^T.
\end{align*}

Discriminant groups and discriminant forms of the lattices \(L_{\mathcal{S}}\) and \(H \oplus \Pic(X)\) are given by
\begin{gather*}
  G' = 
  \begin{pmatrix}
    0 & 0 & 0 & \frac{1}{2} & \frac{1}{2} & 0 & \frac{1}{2} & 0 & 0 & 0 \\
    0 & 0 & 0 & 0 & \frac{1}{2} & 0 & \frac{1}{2} & 0 & 0 & 0 \\
    0 & 0 & \frac{1}{2} & 0 & \frac{1}{2} & \frac{1}{2} & \frac{1}{2} & 0 & 0 & 0 \\
    0 & 0 & \frac{1}{2} & 0 & \frac{1}{2} & \frac{1}{2} & 0 & 0 & 0 & 0 \\
    0 & \frac{1}{2} & 0 & 0 & 0 & 0 & \frac{1}{2} & \frac{1}{2} & 0 & \frac{1}{2} \\
    0 & \frac{1}{2} & \frac{1}{2} & 0 & \frac{1}{2} & 0 & 0 & \frac{1}{2} & 0 & \frac{1}{2} \\
    \frac{1}{2} & 0 & 0 & 0 & 0 & 0 & 0 & \frac{1}{2} & 0 & \frac{1}{2} \\
    0 & 0 & 0 & 0 & 0 & 0 & 0 & \frac{1}{2} & 0 & \frac{1}{2}
  \end{pmatrix}, \;
  G'' = 
  \begin{pmatrix}
    0 & 0 & \frac{1}{2} & \frac{1}{2} & 0 & 0 & \frac{1}{2} & 0 & \frac{1}{2} & \frac{1}{2} & 0 & \frac{1}{2} \\
    0 & 0 & \frac{1}{2} & \frac{1}{2} & 0 & 0 & 0 & \frac{1}{2} & \frac{1}{2} & \frac{1}{2} & 0 & \frac{1}{2} \\
    0 & 0 & 0 & 0 & \frac{1}{2} & 0 & \frac{1}{2} & \frac{1}{2} & \frac{1}{2} & 0 & 0 & 0 \\
    0 & 0 & \frac{1}{2} & \frac{1}{2} & \frac{1}{2} & 0 & \frac{1}{2} & \frac{1}{2} & 0 & 0 & 0 & \frac{1}{2} \\
    0 & 0 & 0 & 0 & 0 & \frac{1}{2} & \frac{1}{2} & \frac{1}{2} & 0 & \frac{1}{2} & 0 & 0 \\
    0 & 0 & \frac{1}{2} & \frac{1}{2} & \frac{1}{2} & \frac{1}{2} & 0 & 0 & \frac{1}{2} & 0 & 0 & \frac{1}{2} \\
    0 & 0 & 0 & \frac{1}{2} & 0 & 0 & 0 & 0 & 0 & 0 & 0 & \frac{1}{2} \\
    0 & 0 & \frac{1}{2} & 0 & 0 & 0 & 0 & 0 & 0 & 0 & 0 & \frac{1}{2}
  \end{pmatrix};
\end{gather*}
\begin{gather*}
  B' = 
  \begin{pmatrix}
    0 & \frac{1}{2} & 0 & 0 & 0 & 0 & 0 & 0 \\
    \frac{1}{2} & 0 & 0 & 0 & 0 & 0 & 0 & 0 \\
    0 & 0 & 0 & \frac{1}{2} & 0 & 0 & 0 & 0 \\
    0 & 0 & \frac{1}{2} & 0 & 0 & 0 & 0 & 0 \\
    0 & 0 & 0 & 0 & 0 & \frac{1}{2} & 0 & 0 \\
    0 & 0 & 0 & 0 & \frac{1}{2} & 0 & 0 & 0 \\
    0 & 0 & 0 & 0 & 0 & 0 & 0 & \frac{1}{2} \\
    0 & 0 & 0 & 0 & 0 & 0 & \frac{1}{2} & 0
  \end{pmatrix}, \;
  B'' = 
  \begin{pmatrix}
    0 & \frac{1}{2} & 0 & 0 & 0 & 0 & 0 & 0 \\
    \frac{1}{2} & 0 & 0 & 0 & 0 & 0 & 0 & 0 \\
    0 & 0 & 0 & \frac{1}{2} & 0 & 0 & 0 & 0 \\
    0 & 0 & \frac{1}{2} & 0 & 0 & 0 & 0 & 0 \\
    0 & 0 & 0 & 0 & 0 & \frac{1}{2} & 0 & 0 \\
    0 & 0 & 0 & 0 & \frac{1}{2} & 0 & 0 & 0 \\
    0 & 0 & 0 & 0 & 0 & 0 & 0 & \frac{1}{2} \\
    0 & 0 & 0 & 0 & 0 & 0 & \frac{1}{2} & 0
  \end{pmatrix}; \;
  Q' = Q'' = 0.
\end{gather*}


\section{Dolgachev--Nikulin duality for parametrized toric Landau--Ginzburg models}
\label{appendix:parametrized}
\subsection{Toric Landau--Ginzburg models of smooth del Pezzo surfaces with very ample anticanonical class}

Let \(S\) be a smooth del Pezzo surface. If the anticanonical class \(-K_S\) is very ample, then the surface \(S\) admits a Gorenstein toric degeneration. More precisely, it is well-known that there exist 16 Gorenstein toric del Pezzo surfaces, which correspond to 16 reflexive polygons (for example, see~\cite{batyrev2010reflexive}). All of them can be realized as toric degenerations of smooth del Pezzo surfaces (see~\cite[Section~3]{przyjalkowski2017compactification} and Table~\ref{table:toric-dP}).

\begin{notation}
  We denote by \(R_i\) the \(i\)-th reflexive polygon with respect to the numeration in PALP package (see~\cite{kreuzer2004PALP}), and by \(T_i\) the corresponding Gorenstein toric del Pezzo surface. 
\end{notation}

\begin{table}[h]
  \centering
  \begin{minipage}[t]{0.5\textwidth}  
    \begin{tabular}{|c|c|c|}
      \hline
      \(S\) & \((-K_S)^2\) & Gorenstein toric dP surfaces \\
      \hline
      \hline
      \(\mathbb{P}^1 \times \mathbb{P}^1\) & \(8\) & \(T_2, T_4\) \\ 
      \hline
      \(\mathbb{P}^2\) & \(9\) & \(T_1\) \\ 
      \hline
      \(\Bl_1(\mathbb{P}^2)\) & \(8\) & \(T_3\) \\
      \hline  
      \(\Bl_2(\mathbb{P}^2)\) & \(7\) & \(T_5, T_6\) \\ 
      \hline
      \(\Bl_3(\mathbb{P}^2)\) & \(6\) & \(T_7, T_8, T_9, T_{10}\) \\ 
      \hline
    \end{tabular}
  \end{minipage}%
  \begin{minipage}[t]{0.5\textwidth}  
    \begin{tabular}{|c|c|c|}
      \hline
      \(S\) & \((-K_S)^2\) & Gorenstein toric dP surfaces \\
      \hline
      \hline
      \(\Bl_4(\mathbb{P}^2)\) & \(5\) & \(T_{11}, T_{12}\) \\ 
      \hline
      \(\Bl_5(\mathbb{P}^2)\) & \(4\) & \(T_{13}, T_{14}, T_{15}\) \\ 
      \hline
      \(\Bl_6(\mathbb{P}^2)\) & \(3\) & \(T_{16}\) \\ 
      \hline
      \(\Bl_7(\mathbb{P}^2)\) & \(2\) & -- \\ 
      \hline
      \(\Bl_8(\mathbb{P}^2)\) & \(1\) & -- \\ 
      \hline
    \end{tabular}
  \end{minipage}
  \caption{Gorenstein toric del Pezzo surfaces corresponding to smooth del Pezzo surface with the very ample anticanonical class \(-K_X\).}
  \label{table:toric-dP}
\end{table}

\begin{remark}
  Toric degenerations \(T_1, T_3, T_4, T_6, T_{10}\) correspond to smooth toric del Pezzo surfaces. 
\end{remark}

It follows from this table that a smooth del Pezzo surface \(S\) do not admit a Gorenstein toric degeneration precisely when the anticanonical degree is equal to \(1\) or \(2\). It is well-known that such a surface is either a sextic hypersurface in \(\mathbb{P}(1, 1, 2, 3)\) or a quartic hypersurface in \(\mathbb{P}(1, 1, 1, 2)\). Then \(S\) also has a toric degeneration, but its singularities are worse than Gorenstein (see~\cite[Remark~3.4]{przyjalkowski2017compactification}).

We construct a toric Landau--Ginzburg model for all smooth del Pezzo surfaces \(S\) with very ample anticanonical class by applying Przyjalkowski's algorithm (see~\cite[Section~3]{przyjalkowski2017compactification}).

\begin{notation}
  For any coefficients of the form \(\alpha, \beta, \alpha_i \in \mathbb{C}\) we put \(a = e^{-\alpha}\), \(b = e^{-\beta}\), and \(a_i = e^{-\alpha_i}\).
\end{notation}

\subsubsection{Quadric surface}\label{subsubsection:dP_quadric}

Let \(S \simeq \mathbb{P}^1 \times \mathbb{P}^1\) be a quadric surface, and \(D\) be an \((\alpha, \beta)\)-divisor on \(S\). The surface \(S\) admits two Gorenstein toric degenerations: \(T_4 \simeq S\) and a quadratic cone \(T_2\).

Firstly, let us choose \(T_4\) (see Figure~\ref{figure:polygons-T_2-4-1}) as a Gorenstein toric degeneration \(T\) of the surface \(S\), and let \(\widetilde{D}\) be a divisor on its crepant resolution \(\widetilde{T} \simeq S\). The toric Landau--Ginzburg model for the pair \((S, D) = (\widetilde{T}, \widetilde{D})\) equals to
\[
  f_{(S, D)} = f_{(\widetilde{T}, \widetilde{D})} = x + a x^{-1} + y + b y^{-1}.
\]

Secondly, let us choose \(T_2\) (see Figure~\ref{figure:polygons-T_2-4-1}) as a Gorenstein toric degeneration \(T\) of the surface \(S\), and let \(\widetilde{D}\) be a divisor on its crepant resolution \(\widetilde{T}\). In this case \(\widetilde{T}\) is the second Hirzebruch surface \(\mathbb{F}_2\). Consequently, we can present \(\widetilde{D}\) in the form \(\widetilde{D} = \alpha s + \beta f\), where \(s\) is a \((-2)\)-section of \(\widetilde{T}\), and \(f\) is a fibre of the map \(\widetilde{T} \rightarrow \mathbb{P}^1\). The toric Landau--Ginzburg model for the pair \((\widetilde{T}, \widetilde{D})\) has the form
\[
  f_{(\widetilde{T}, \widetilde{D})} = y + b x^{-1} y^{-1} + a y^{-1} + x y^{-1},
\]
and the toric Landau--Ginzburg model for the pair \((S, D)\) equals to
\[
  f_{(S, D)} = y + a x^{-1} y^{-1} + (a + b) y^{-1} + b x y^{-1}.
\]

\subsubsection{Smooth del Pezzo surface of degree 9}\label{subsubsection:dP_09}

Let \(S \simeq \mathbb{P}^2\), and \(D = \alpha_1 l \in \Pic(S) \otimes_{\mathbb{Z}} \mathbb{C}\) be a divisor on \(S\), where \(l\) is the linear equivalence class of a line. The surface \(S\) admits the unique Gorenstein toric degeneration \(T_1 \simeq S\). Let us choose \(T_1\) (see Figure~\ref{figure:polygons-T_2-4-1}) as a Gorenstein toric degeneration \(T\) of the surface \(S\), and let \(\widetilde{D} \in \Pic(\widetilde{T}) \otimes_{\mathbb{Z}} \mathbb{C}\) be the corresponding divisor on its crepant resolution \(\widetilde{T} \simeq S\) under the identification \(\Pic(S) \simeq \Pic(\widetilde{T})\).

The toric Landau--Ginzburg model for the pair \((S, D) = (\widetilde{T}, \widetilde{D})\) equals to
\[
  f_{(S, D)} = f_{(\widetilde{T}, \widetilde{D})} = x + y + a_1 x^{-1} y^{-1}.
\]

\begin{figure}[H]
  \centering
  \begin{minipage}{.33\textwidth}
    \centering
    \begin{tikzpicture}%
	[scale=1.500000,
	back/.style={loosely dotted, thin},
	edge/.style={color=blue!95!black, thick},
	facet/.style={fill=blue!95!black,fill opacity=0.250000},
	vertex/.style={inner sep=1pt,circle,draw=green!25!black,fill=green!75!black,thick}]


\tkzInit[xmax=1,ymax=1,xmin=-1,ymin=-1]
\tkzGrid


\coordinate (-1.00000, -1.00000) at (-1.00000, -1.00000);
\coordinate (0.00000, 1.00000) at (0.00000, 1.00000);
\coordinate (1.00000, -1.00000) at (1.00000, -1.00000);


\fill[facet] (1.00000, -1.00000) -- (-1.00000, -1.00000) -- (0.00000, 1.00000) -- cycle {};


\draw[edge] (-1.00000, -1.00000) -- (0.00000, 1.00000);
\draw[edge] (-1.00000, -1.00000) -- (1.00000, -1.00000);
\draw[edge] (0.00000, 1.00000) -- (1.00000, -1.00000);


\node[vertex] at (-1.00000, -1.00000)     {};
\node[vertex] at (0.00000, 1.00000)     {};
\node[vertex] at (1.00000, -1.00000)     {};

\end{tikzpicture} \\
    \(T_2\)
  \end{minipage}%
  \begin{minipage}{.33\textwidth}
    \centering
    \begin{tikzpicture}%
	[scale=1.500000,
	back/.style={loosely dotted, thin},
	edge/.style={color=blue!95!black, thick},
	facet/.style={fill=blue!95!black,fill opacity=0.250000},
	vertex/.style={inner sep=1pt,circle,draw=green!25!black,fill=green!75!black,thick}]


\tkzInit[xmax=1,ymax=1,xmin=-1,ymin=-1]
\tkzGrid


\coordinate (-1.00000, 0.00000) at (-1.00000, 0.00000);
\coordinate (0.00000, -1.00000) at (0.00000, -1.00000);
\coordinate (0.00000, 1.00000) at (0.00000, 1.00000);
\coordinate (1.00000, 0.00000) at (1.00000, 0.00000);


\fill[facet] (1.00000, 0.00000) -- (0.00000, -1.00000) -- (-1.00000, 0.00000) -- (0.00000, 1.00000) -- cycle {};


\draw[edge] (-1.00000, 0.00000) -- (0.00000, -1.00000);
\draw[edge] (-1.00000, 0.00000) -- (0.00000, 1.00000);
\draw[edge] (0.00000, -1.00000) -- (1.00000, 0.00000);
\draw[edge] (0.00000, 1.00000) -- (1.00000, 0.00000);


\node[vertex] at (-1.00000, 0.00000)     {};
\node[vertex] at (0.00000, -1.00000)     {};
\node[vertex] at (0.00000, 1.00000)     {};
\node[vertex] at (1.00000, 0.00000)     {};

\end{tikzpicture} \\
    \(T_4\)
  \end{minipage}%
  \begin{minipage}{.33\textwidth}
    \centering
    \begin{tikzpicture}%
	[scale=1.500000,
	back/.style={loosely dotted, thin},
	edge/.style={color=blue!95!black, thick},
	facet/.style={fill=blue!95!black,fill opacity=0.250000},
	vertex/.style={inner sep=1pt,circle,draw=green!25!black,fill=green!75!black,thick}]


\tkzInit[xmax=1,ymax=1,xmin=-1,ymin=-1]
\tkzGrid


\coordinate (-1.00000, -1.00000) at (-1.00000, -1.00000);
\coordinate (0.00000, 1.00000) at (0.00000, 1.00000);
\coordinate (1.00000, 0.00000) at (1.00000, 0.00000);


\fill[facet] (1.00000, 0.00000) -- (-1.00000, -1.00000) -- (0.00000, 1.00000) -- cycle {};


\draw[edge] (-1.00000, -1.00000) -- (0.00000, 1.00000);
\draw[edge] (-1.00000, -1.00000) -- (1.00000, 0.00000);
\draw[edge] (0.00000, 1.00000) -- (1.00000, 0.00000);


\node[vertex] at (-1.00000, -1.00000)     {};
\node[vertex] at (0.00000, 1.00000)     {};
\node[vertex] at (1.00000, 0.00000)     {};

\end{tikzpicture} \\
    \(T_1\)
  \end{minipage}
  \caption{Fan polygons of Gorenstein toric del Pezzo surfaces \(T_2\), \(T_4\), and \(T_1\).}
  \label{figure:polygons-T_2-4-1}
\end{figure}

\subsubsection{Smooth del Pezzo surface of degree 8}\label{subsubsection:dP_08}

Let \(S \xrightarrow{\varphi_1} S'\) be a blow-up of del Pezzo surface \(S'\) from the previous case, and \(D = \alpha_1 l + \alpha_2 e_1 \in \Pic(S) \otimes_{\mathbb{Z}} \mathbb{C}\) be a divisor on \(S\), where \(e_1\) is the \(\varphi_1\)-exceptional divisor. The surface \(S\) admits a unique Gorenstein toric degeneration \(T_3 \simeq S\). Let us choose \(T_3\) (see Figure~\ref{figure:polygons-T_3-6-10}) as a Gorenstein toric degeneration \(T\) of the surface \(S\), and let \(\widetilde{D} \in \Pic(\widetilde{T}) \otimes_{\mathbb{Z}} \mathbb{C}\) be the corresponding divisor on the crepant resolution \(\widetilde{T} \simeq S\) under the identification \(\Pic(S) \simeq \Pic(\widetilde{T})\).

Coefficients corresponding to the neighbour points \(L, R\) of the point \(K\) have the form \(c_L = a_1\) and \(c_R = 1\). Consequently, the toric Landau--Ginzburg model for the pair \((S, D) = (\widetilde{T}, \widetilde{D})\) equals to
\[
  f_{(S, D)} = f_{(\widetilde{T}, \widetilde{D})} =
  f_{(\widetilde{T'}, \widetilde{D'})} + c_L c_R a_2 x^{-1} = 
  x + y + a_1 x^{-1} y^{-1} + a_1 a_2 x^{-1}.
\]
where \((\widetilde{T'}, \widetilde{D'})\) is the crepant resolution of the Gorenstein toric degeneration from the previous case.

\subsubsection{Smooth del Pezzo surface of degree 7}\label{subsubsection:dP_07}

Let \(S \xrightarrow{\varphi_2} S'\) be a blow-up of del Pezzo surface \(S'\) from the previous case, and \(D = \alpha_1 l + \sum_{i = 1}^2 \alpha_{i + 1} e_i \in \Pic(S) \otimes_{\mathbb{Z}} \mathbb{C}\) be a divisor on \(S\), where \(e_2\) is the \(\varphi_2\)-exceptional divisor. The surface \(S\) admits two Gorenstein toric degenerations: \(T_6 \simeq S\) and \(T_5\). Let us choose \(T_6\) (see Figure~\ref{figure:polygons-T_3-6-10}) as a Gorenstein toric degeneration \(T\) of the surface \(S\), and let \(\widetilde{D} \in \Pic(\widetilde{T}) \otimes_{\mathbb{Z}} \mathbb{C}\) be a divisor on the crepant resolution \(\widetilde{T} \simeq S\) under the identification \(\Pic(S) \simeq \Pic(\widetilde{T})\).

Coefficients corresponding to the neighbour points \(L, R\) of the point \(K\) have the form \(c_L = 1\) and \(c_R = a_1\). Consequently, the toric Landau--Ginzburg model for the pair \((S, D) = (\widetilde{T}, \widetilde{D})\) equals to
\[
  f_{(S, D)} = f_{(\widetilde{T}, \widetilde{D})} =
  f_{(\widetilde{T'}, \widetilde{D'})} + c_L c_R a_3 y^{-1} =
  x + y + a_1 x^{-1} y^{-1} + a_1 a_2 x^{-1} + a_1 a_3 y^{-1},
\]
where \((\widetilde{T'}, \widetilde{D'})\) is the crepant resolution of the Gorenstein toric degeneration from the previous case.

\subsubsection{Smooth del Pezzo surface of degree 6}\label{subsubsection:dP_06}

Let \(S \xrightarrow{\varphi_3} S'\) be a blow-up of del Pezzo surface \(S'\) from the previous case, and \(D = \alpha_1 l + \sum_{i = 1}^3 \alpha_{i + 1} e_i \in \Pic(S) \otimes_{\mathbb{Z}} \mathbb{C}\) be a divisor on \(S\), where \(e_3\) is the \(\varphi_3\)-exceptional divisor. The surface \(S\) admits four Gorenstein toric degenerations: \(T_7\), \(T_8\), \(T_9\), and \(T_{10} \simeq S\). Let us choose \(T_{10}\) (see Figure~\ref{figure:polygons-T_3-6-10}) as a Gorenstein toric degeneration \(T\) of the surface \(S\), and let \(\widetilde{D} \in \Pic(\widetilde{T}) \otimes_{\mathbb{Z}} \mathbb{C}\) be a divisor on the crepant resolution \(\widetilde{T} \simeq S\) under the identification \(\Pic(S) \simeq \Pic(\widetilde{T})\).

Coefficients corresponding to the neighbour points \(L, R\) of the point \(K\) have the form \(c_L = 1\) and \(c_R = 1\). Consequently, the toric Landau--Ginzburg model for the pair \((S, D) = (\widetilde{T}, \widetilde{D})\) equals to
\[
  f_{(S, D)} = f_{(\widetilde{T}, \widetilde{D})} =
  f_{(\widetilde{T'}, \widetilde{D'})} + c_L c_R a_4 x y =
  x + y + a_1 x^{-1} y^{-1} + a_1 a_2 x^{-1} + a_1 a_3 y^{-1} + a_4 x y.
\]
where \((\widetilde{T'}, \widetilde{D'})\) is the crepant resolution of the Gorenstein toric degeneration from the previous case.

\begin{figure}[H]
  \centering
  \begin{minipage}{.33\textwidth}
    \centering
    \begin{tikzpicture}%
	[scale=1.500000,
	back/.style={loosely dotted, thin},
	edge/.style={color=blue!95!black, thick},
	facet/.style={fill=blue!95!black,fill opacity=0.250000},
	vertex/.style={inner sep=1pt,circle,draw=green!25!black,fill=green!75!black,thick}]


\tkzInit[xmax=1,ymax=1,xmin=-1,ymin=-1]
\tkzGrid


\coordinate (-1.00000, -1.00000) at (-1.00000, -1.00000);
\coordinate (-1.00000, 0.00000) at (-1.00000, 0.00000);
\coordinate (0.00000, 1.00000) at (0.00000, 1.00000);
\coordinate (1.00000, 0.00000) at (1.00000, 0.00000);


\fill[facet] (1.00000, 0.00000) -- (-1.00000, -1.00000) -- (-1.00000, 0.00000) -- (0.00000, 1.00000) -- cycle {};


\draw[edge] (-1.00000, -1.00000) -- (-1.00000, 0.00000);
\draw[edge] (-1.00000, -1.00000) -- (1.00000, 0.00000);
\draw[edge] (-1.00000, 0.00000) -- (0.00000, 1.00000);
\draw[edge] (0.00000, 1.00000) -- (1.00000, 0.00000);


\node[vertex] [draw=yellow!25!black,fill=yellow!75!black] at (-1.00000, -1.00000)     {L};
\node[vertex] [draw=orange!25!black,fill=orange!75!black] at (-1.00000, 0.00000)     {K};
\node[vertex] [draw=yellow!25!black,fill=yellow!75!black] at (0.00000, 1.00000)     {R};
\node[vertex] at (1.00000, 0.00000)     {};

\end{tikzpicture} \\
    \(T_3\)
  \end{minipage}%
  \begin{minipage}{.33\textwidth}
    \centering
    \begin{tikzpicture}%
	[scale=1.500000,
	back/.style={loosely dotted, thin},
	edge/.style={color=blue!95!black, thick},
	facet/.style={fill=blue!95!black,fill opacity=0.250000},
	vertex/.style={inner sep=1pt,circle,draw=green!25!black,fill=green!75!black,thick}]


\tkzInit[xmax=1,ymax=1,xmin=-1,ymin=-1]
\tkzGrid


\coordinate (-1.00000, -1.00000) at (-1.00000, -1.00000);
\coordinate (-1.00000, 0.00000) at (-1.00000, 0.00000);
\coordinate (0.00000, -1.00000) at (0.00000, -1.00000);
\coordinate (0.00000, 1.00000) at (0.00000, 1.00000);
\coordinate (1.00000, 0.00000) at (1.00000, 0.00000);


\fill[facet] (1.00000, 0.00000) -- (0.00000, -1.00000) -- (-1.00000, -1.00000) -- (-1.00000, 0.00000) -- (0.00000, 1.00000) -- cycle {};


\draw[edge] (-1.00000, -1.00000) -- (-1.00000, 0.00000);
\draw[edge] (-1.00000, -1.00000) -- (0.00000, -1.00000);
\draw[edge] (-1.00000, 0.00000) -- (0.00000, 1.00000);
\draw[edge] (0.00000, -1.00000) -- (1.00000, 0.00000);
\draw[edge] (0.00000, 1.00000) -- (1.00000, 0.00000);


\node[vertex] [draw=yellow!25!black,fill=yellow!75!black] at (-1.00000, -1.00000)     {R};
\node[vertex] at (-1.00000, 0.00000)     {};
\node[vertex] [draw=orange!25!black,fill=orange!75!black] at (0.00000, -1.00000)     {K};
\node[vertex] at (0.00000, 1.00000)     {};
\node[vertex] [draw=yellow!25!black,fill=yellow!75!black] at (1.00000, 0.00000)     {L};

\end{tikzpicture} \\
    \(T_6\)
  \end{minipage}%
  \begin{minipage}{.33\textwidth}
    \centering
   \begin{tikzpicture}%
	[scale=1.500000,
	back/.style={loosely dotted, thin},
	edge/.style={color=blue!95!black, thick},
	facet/.style={fill=blue!95!black,fill opacity=0.250000},
	vertex/.style={inner sep=1pt,circle,draw=green!25!black,fill=green!75!black,thick}]


\tkzInit[xmax=1,ymax=1,xmin=-1,ymin=-1]
\tkzGrid


\coordinate (-1.00000, -1.00000) at (-1.00000, -1.00000);
\coordinate (-1.00000, 0.00000) at (-1.00000, 0.00000);
\coordinate (0.00000, -1.00000) at (0.00000, -1.00000);
\coordinate (0.00000, 1.00000) at (0.00000, 1.00000);
\coordinate (1.00000, 0.00000) at (1.00000, 0.00000);
\coordinate (1.00000, 1.00000) at (1.00000, 1.00000);


\fill[facet] (1.00000, 1.00000) -- (0.00000, 1.00000) -- (-1.00000, 0.00000) -- (-1.00000, -1.00000) -- (0.00000, -1.00000) -- (1.00000, 0.00000) -- cycle {};


\draw[edge] (-1.00000, -1.00000) -- (-1.00000, 0.00000);
\draw[edge] (-1.00000, -1.00000) -- (0.00000, -1.00000);
\draw[edge] (-1.00000, 0.00000) -- (0.00000, 1.00000);
\draw[edge] (0.00000, -1.00000) -- (1.00000, 0.00000);
\draw[edge] (0.00000, 1.00000) -- (1.00000, 1.00000);
\draw[edge] (1.00000, 0.00000) -- (1.00000, 1.00000);


\node[vertex] at (-1.00000, -1.00000)     {};
\node[vertex] at (-1.00000, 0.00000)     {};
\node[vertex] at (0.00000, -1.00000)     {};
\node[vertex] [draw=yellow!25!black,fill=yellow!75!black] at (0.00000, 1.00000)     {L};
\node[vertex] [draw=yellow!25!black,fill=yellow!75!black] at (1.00000, 0.00000)     {R};
\node[vertex] [draw=orange!25!black,fill=orange!75!black] at (1.00000, 1.00000)     {K};

\end{tikzpicture} \\
   \(T_{10}\)
  \end{minipage}
  \caption{Fan polygons of Gorenstein toric del Pezzo surfaces \(T_3\), \(T_6\), and \(T_{10}\).}
  \label{figure:polygons-T_3-6-10}
\end{figure}

\subsubsection{Smooth del Pezzo surface of degree 5}\label{subsubsection:dP_05}

Let \(S \xrightarrow{\varphi_4} S'\) be a blow-up of del Pezzo surface \(S'\) from the previous case, and \(D = \alpha_1 l + \sum_{i = 1}^4 \alpha_{i + 1} e_i \in \Pic(S) \otimes_{\mathbb{Z}} \mathbb{C}\) be a divisor on \(S\), where \(e_4\) is the \(\varphi_4\)-exceptional divisor. The surface \(S\) admits two Gorenstein toric degenerations: \(T_{11}\) and \(T_{12}\). Let us choose \(T_{12}\) (see Figure~\ref{figure:polygons-T_12-14-16}) as a Gorenstein toric degeneration \(T\) of the surface \(S\), and let \(\widetilde{D} \in \Pic(\widetilde{T}) \otimes_{\mathbb{Z}} \mathbb{C}\) be a divisor on the crepant resolution \(\widetilde{T}\) under the identification \(\Pic(S) \simeq \Pic(\widetilde{T})\).

Coefficients corresponding to the neighbour points \(L, R\) of the point \(K\) have the form \(c_L = a_1 a_2\) and \(c_R = 1\). Consequently, the toric Landau--Ginzburg model for the pair \((\widetilde{T}, \widetilde{D})\) equals to
\[
  f_{(\widetilde{T}, \widetilde{D})} =
  f_{(\widetilde{T'}, \widetilde{D'})} + c_L c_R a_5 x^{-1} y =
  x + y + a_1 x^{-1} y^{-1} + a_1 a_2 x^{-1} + a_1 a_3 y^{-1} + a_4 x y + a_1 a_2 a_5 x^{-1} y.
\]
where \((\widetilde{T'}, \widetilde{D'})\) is the crepant resolution of the Gorenstein toric degeneration from the previous case.

To obtain the toric Landau--Ginzburg model for the pair \((S, D)\), we have to modify the coefficients of \(f_{(\widetilde{T}, \widetilde{D})}\) corresponding to non-vertex boundary points. Let us consider the facets of the fan polygon containing these points. Their marking polynomials and corresponding modified coefficients have the form
\begin{gather*}
  a_1 a_2 a_5 s^2 + a_1 (a_2 + a_5) s + a_1 =
  c_K s^2 + \widetilde{c}_L s + c_V, \\
  a_4 s^2 + (a_1 a_2 a_4 a_5 + 1) s + a_1 a_2 a_5 =
  c_W s^2 + \widetilde{c}_R s + c_K.
\end{gather*}
Consequently, the toric Landau--Ginzburg model for the pair \((S, D)\) equals to
\[
  f_{(S, D)} = x + (a_1 a_2 a_4 a_5 + 1) y + a_1 x^{-1} y^{-1} +
  a_1 (a_2 + a_5) x^{-1} + a_1 a_3 y^{-1} + a_4 x y + a_1 a_2 a_5 x^{-1} y.
\]

\subsubsection{Smooth del Pezzo surface of degree 4}\label{subsubsection:dP_04}

Let \(S \xrightarrow{\varphi_5} S'\) be a blow-up of del Pezzo surface \(S'\) from the previous case, and \(D = \alpha_1 l + \sum_{i = 1}^5 \alpha_{i + 1} e_i \in \Pic(S) \otimes_{\mathbb{Z}} \mathbb{C}\) be a divisor on \(S\), where \(e_5\) is the \(\varphi_5\)-exceptional divisor. The surface \(S\) admits three Gorenstein toric degenerations: \(T_{13}\), \(T_{14}\) and \(T_{15}\). Let us choose \(T_{14}\) (see Figure~\ref{figure:polygons-T_12-14-16}) as a Gorenstein toric degeneration \(T\) of the surface \(S\), and let \(\widetilde{D} \in \Pic(\widetilde{T}) \otimes_{\mathbb{Z}} \mathbb{C}\) be a divisor on the crepant resolution \(\widetilde{T}\) under the identification \(\Pic(S) \simeq \Pic(\widetilde{T})\).

Coefficients corresponding to the neighbour points \(L, R\) of the point \(K\) have the form \(c_L = a_1 a_3\) and \(c_R = a_1\). Consequently, the toric Landau--Ginzburg model for the pair \((\widetilde{T}, \widetilde{D})\) equals to
\begin{gather*}
  f_{(\widetilde{T}, \widetilde{D})} =
  f_{(\widetilde{T'}, \widetilde{D'})} + c_L c_R a_6 x^{-1} y^{-2} = \\
  x + y + a_1 x^{-1} y^{-1} + a_1 a_2 x^{-1} + 
  a_1 a_3 y^{-1} + a_4 x y + a_1 a_2 a_5 x^{-1} y +
  a_1^2 a_3 a_6 x^{-1} y^{-2},
\end{gather*}
where \((\widetilde{T'}, \widetilde{D'})\) is the crepant resolution of the Gorenstein toric degeneration from the previous case.

To obtain the toric Landau--Ginzburg model for the pair \((S, D)\), we have to modify the coefficients of \(f_{(\widetilde{T}, \widetilde{D})}\) corresponding to non-vertex boundary points. Let us consider the facets of the fan polygon containing these points. Their marking polynomials and corresponding modified coefficients have the form
\begin{gather*}
  a_1 a_2 a_5 s^3 + a_1 (a_1 a_2 a_3 a_5 a_6 + a_2 + a_5) s^2 +
  a_1 (a_1 a_3 a_6 (a_2 + a_5) + 1) s + a_1^2 a_3 a_6 =
  c_U s^3 + \widetilde{c}_A s^2 + \widetilde{c}_R s + c_K, \\
  a_4 s^2 + (a_1 a_2 a_4 a_5 + 1) s + a_1 a_2 a_5 =
  c_V s^2 + \widetilde{c}_B s + c_U, \\
  a_1^2 a_3 a_6 s^2 + a_1 (a_3 + a_6) s + 1 =
  c_K s^2 + \widetilde{c}_L s + c_W.
\end{gather*}
Consequently, the toric Landau--Ginzburg model for the pair \((S, D)\) equals to
\begin{gather*}
  f_{(S, D)} = x + (a_1 a_2 a_4 a_5 + 1) y + a_1 (a_1 a_3 a_6 (a_2 + a_5) + 1) x^{-1} y^{-1} + \\ a_1 (a_1 a_2 a_3 a_5 a_6 + a_2 + a_5) x^{-1} + a_1 (a_3 + a_6) y^{-1} + a_4 x y + a_1 a_2 a_5 x^{-1} y + a_1^2 a_3 a_6 x^{-1} y^{-2}.
\end{gather*}

\subsubsection{Smooth del Pezzo surface of degree 3}\label{subsubsection:dP_03}

Let \(S \xrightarrow{\varphi_6} S'\) be a blow-up of del Pezzo surface \(S'\) from the previous case, and \(D = \alpha_1 l + \sum_{i = 1}^6 \alpha_{i + 1} e_i \in \Pic(S) \otimes_{\mathbb{Z}} \mathbb{C}\) be a divisor on \(S\), where \(e_6\) is the \(\varphi_6\)-exceptional divisor. The surface \(S\) admits a unique Gorenstein toric degenerations \(T_{16}\). Let us choose \(T_{16}\) (see Figure~\ref{figure:polygons-T_12-14-16}) as a Gorenstein toric degeneration \(T\) of the surface \(S\), and let \(\widetilde{D}\) be a divisor on the crepant resolution \(\widetilde{T}\) under the identification \(\Pic(S) \simeq \Pic(\widetilde{T})\).

Coefficients corresponding to the neighbour points \(L, R\) of the point \(K\) have the form \(c_L = a_4\) and \(c_R = 1\). Consequently, the toric Landau--Ginzburg model for the pair \((\widetilde{T}, \widetilde{D})\) equals to
\begin{gather*}
  f_{(\widetilde{T}, \widetilde{D})} =
  f_{(\widetilde{T'}, \widetilde{D'})} + c_L c_R a_7 x^2 y = \\
  x + y + a_1 x^{-1} y^{-1} + a_1 a_2 x^{-1} + 
  a_1 a_3 y^{-1} + a_4 x y + a_1 a_2 a_5 x^{-1} y +
  a_1^2 a_3 a_6 x^{-1} y^{-2} + a_4 a_7 x^2 y.
\end{gather*}
where \((\widetilde{T'}, \widetilde{D'})\) is the crepant resolution of the Gorenstein toric degeneration from the previous case.

To obtain the toric Landau--Ginzburg model for the pair \((S, D)\), we have to modify the coefficients of \(f_{(\widetilde{T}, \widetilde{D})}\) corresponding to non-vertex boundary points. Let us consider the facets of the fan polygon containing these points. Their marking polynomials and corresponding modified coefficients have the form
\begin{gather*}
  a_1 a_2 a_5 s^3 + a_1 (a_1 a_2 a_3 a_5 a_6 + a_2 + a_5) s^2 + a_1 (a_1 a_3 a_6 (a_2 + a_5) + 1) s + a_1^2 a_3 a_6 = c_V s^3 + c_D s^2 + c_C s + c_W, \\
  a_4 a_7 s^3 + (a_1 a_2 a_4 a_5 a_7 + a_4 + a_7) s^2 + (a_1 a_2 a_5 (a_4 + a_7) + 1) s + a_1 a_2 a_5 = c_K s^3 + c_L s^2 + c_A s + c_V, \\
  a_1^2 a_3 a_6 s^3 + a_1 (a_1 a_3 a_4 a_6 a_7 + a_3 + a_6) s^2 + (a_1 a_4 a_7 (a_3 + a_6) + 1) s + a_4 a_7 = c_W s^3 + c_C s^2 + c_D s + c_W.
\end{gather*}

Consequently, the toric Landau--Ginzburg model for the pair \((S, D)\) equals to
\begin{gather*}
  f_{(S, D)} = 
  (a_1 a_4 a_7 (a_3 + a_6) + 1) x + (a_1 a_2 a_5 (a_4 + a_7) + 1) y + 
  a_1 (a_1 a_3 a_6 (a_2 + a_5) + 1) x^{-1} y^{-1} + \\
  a_1 (a_1 a_2 a_3 a_5 a_6 + a_2 + a_5) x^{-1} +
  a_1 (a_1 a_3 a_4 a_6 a_7 + a_3 + a_6) y^{-1} + 
  (a_1 a_2 a_4 a_5 a_7 + a_4 + a_7) x y + \\
  a_1 a_2 a_5 x^{-1} y +
  a_1^2 a_3 a_6 x^{-1} y^{-2} +
  a_4 a_7 x^2 y.
\end{gather*}

\begin{figure}[H]
  \centering
  \begin{minipage}{.33\textwidth}
    \centering
   \begin{tikzpicture}%
	[scale=1.500000,
	back/.style={loosely dotted, thin},
	edge/.style={color=blue!95!black, thick},
	facet/.style={fill=blue!95!black,fill opacity=0.250000},
	vertex/.style={inner sep=1pt,circle,draw=green!25!black,fill=green!75!black,thick}]


\tkzInit[xmax=1,ymax=1,xmin=-1,ymin=-1]
\tkzGrid


\coordinate (-1.00000, -1.00000) at (-1.00000, -1.00000);
\coordinate (-1.00000, 1.00000) at (-1.00000, 1.00000);
\coordinate (0.00000, -1.00000) at (0.00000, -1.00000);
\coordinate (1.00000, 0.00000) at (1.00000, 0.00000);
\coordinate (1.00000, 1.00000) at (1.00000, 1.00000);


\fill[facet] (1.00000, 1.00000) -- (-1.00000, 1.00000) -- (-1.00000, -1.00000) -- (0.00000, -1.00000) -- (1.00000, 0.00000) -- cycle {};


\draw[edge] (-1.00000, -1.00000) -- (-1.00000, 1.00000);
\draw[edge] (-1.00000, -1.00000) -- (0.00000, -1.00000);
\draw[edge] (-1.00000, 1.00000) -- (1.00000, 1.00000);
\draw[edge] (0.00000, -1.00000) -- (1.00000, 0.00000);
\draw[edge] (1.00000, 0.00000) -- (1.00000, 1.00000);


\node[vertex] at (-1.00000, -1.00000)     {};
\node[vertex] [draw=orange!25!black,fill=orange!75!black] at (-1.00000, 1.00000)     {K};
\node[vertex] at (0.00000, -1.00000)     {};
\node[vertex] at (1.00000, 0.00000)     {};
\node[vertex] at (1.00000, 1.00000)     {};
\node[vertex] [draw=yellow!25!black,fill=yellow!75!black] at (-1.00000, 0.00000)     {L};
\node[vertex] [draw=yellow!25!black,fill=yellow!75!black] at (0.00000, 1.00000)     {R};
\node[vertex] [draw=cyan!25!black,fill=cyan!75!black] at (-1.00000, -1.00000)     {V};
\node[vertex] [draw=cyan!25!black,fill=cyan!75!black] at (1.00000, 1.00000)     {W};

\end{tikzpicture} \\
   \(T_{12}\)
  \end{minipage}%
  \begin{minipage}{.33\textwidth}
    \centering
    \begin{tikzpicture}%
	[scale=1.250000,
	back/.style={loosely dotted, thin},
	edge/.style={color=blue!95!black, thick},
	facet/.style={fill=blue!95!black,fill opacity=0.250000},
	vertex/.style={inner sep=1pt,circle,draw=green!25!black,fill=green!75!black,thick}]


\tkzInit[xmax=1,ymax=1,xmin=-1,ymin=-2]
\tkzGrid


\coordinate (-1.00000, -2.00000) at (-1.00000, -2.00000);
\coordinate (-1.00000, 1.00000) at (-1.00000, 1.00000);
\coordinate (1.00000, 0.00000) at (1.00000, 0.00000);
\coordinate (1.00000, 1.00000) at (1.00000, 1.00000);


\fill[facet] (1.00000, 1.00000) -- (-1.00000, 1.00000) -- (-1.00000, -2.00000) -- (1.00000, 0.00000) -- cycle {};


\draw[edge] (-1.00000, -2.00000) -- (-1.00000, 1.00000);
\draw[edge] (-1.00000, -2.00000) -- (1.00000, 0.00000);
\draw[edge] (-1.00000, 1.00000) -- (1.00000, 1.00000);
\draw[edge] (1.00000, 0.00000) -- (1.00000, 1.00000);


\node[vertex] [draw=orange!25!black,fill=orange!75!black] at (-1.00000, -2.00000)     {K};
\node[vertex] at (-1.00000, 1.00000)     {};
\node[vertex] at (1.00000, 0.00000)     {};
\node[vertex] at (1.00000, 1.00000)     {};
\node[vertex] [draw=yellow!25!black,fill=yellow!75!black] at (0.00000, -1.00000)     {L};
\node[vertex] [draw=yellow!25!black,fill=yellow!75!black] at (-1.00000, -1.00000)     {R};
\node[vertex] [draw=green!25!black,fill=green!75!black] at (-1.00000, 0.00000)     {A};
\node[vertex] [draw=green!25!black,fill=green!75!black] at (0.00000, 1.00000)     {B};
\node[vertex] [draw=cyan!25!black,fill=cyan!75!black] at (-1.00000, 1.00000)     {U};
\node[vertex] [draw=cyan!25!black,fill=cyan!75!black] at (1.00000, 1.00000)     {V};
\node[vertex] [draw=cyan!25!black,fill=cyan!75!black] at (1.00000, 0.00000)     {W};

\end{tikzpicture} \\
    \(T_{14}\)
  \end{minipage}%
  \begin{minipage}{.33\textwidth}
    \centering
    \begin{tikzpicture}%
	[scale=1.250000,
	back/.style={loosely dotted, thin},
	edge/.style={color=blue!95!black, thick},
	facet/.style={fill=blue!95!black,fill opacity=0.250000},
	vertex/.style={inner sep=1pt,circle,draw=green!25!black,fill=green!75!black,thick}]


\tkzInit[xmax=2,ymax=1,xmin=-1,ymin=-2]
\tkzGrid


\coordinate (-1.00000, -2.00000) at (-1.00000, -2.00000);
\coordinate (-1.00000, 1.00000) at (-1.00000, 1.00000);
\coordinate (2.00000, 1.00000) at (2.00000, 1.00000);


\fill[facet] (2.00000, 1.00000) -- (-1.00000, -2.00000) -- (-1.00000, 1.00000) -- cycle {};


\draw[edge] (-1.00000, -2.00000) -- (-1.00000, 1.00000);
\draw[edge] (-1.00000, -2.00000) -- (2.00000, 1.00000);
\draw[edge] (-1.00000, 1.00000) -- (2.00000, 1.00000);


\node[vertex] at (-1.00000, -2.00000)     {};
\node[vertex] at (-1.00000, 1.00000)     {};
\node[vertex] [draw=orange!25!black,fill=orange!75!black] at (2.00000, 1.00000)     {K};
\node[vertex] [draw=yellow!25!black,fill=yellow!75!black] at (1.00000, 1.00000)     {L};
\node[vertex] [draw=yellow!25!black,fill=yellow!75!black] at (1.00000, 0.00000)     {R};
\node[vertex] [draw=green!25!black,fill=green!75!black] at (0.00000, 1.00000)     {A};
\node[vertex] [draw=green!25!black,fill=green!75!black] at (0.00000, -1.00000)     {B};
\node[vertex] [draw=green!25!black,fill=green!75!black] at (-1.00000, -1.00000)     {C};
\node[vertex] [draw=green!25!black,fill=green!75!black] at (-1.00000, 0.00000)     {D};
\node[vertex] [draw=cyan!25!black,fill=cyan!75!black] at (-1.00000, 1.00000)     {V};
\node[vertex] [draw=cyan!25!black,fill=cyan!75!black] at (-1.00000, -2.00000)     {W};

\end{tikzpicture} \\
    \(T_{16}\)
  \end{minipage}
  \caption{Fan polygons of Gorenstein toric del Pezzo surfaces \(T_{12}\), \(T_{14}\) and \(T_{16}\).}
  \label{figure:polygons-T_12-14-16}
\end{figure}


\subsection{Parametrized toric Landau--Ginzburg models for \texorpdfstring{\(S \times \mathbb{P}^1\)}{S x P\textasciicircum1} with \texorpdfstring{\((-K_S)^2 > 2\)}{(-K\_S)\textasciicircum 2 > 2}}\label{subsection:dPxP1_list}

In this subsection we mutate the obtained parametrized Landau--Ginzburg models to make them coincide with the standard Landau--Ginzburg models after the parameter specialization. Note that here we follow the numeration of Minkowski polynomials in~\cite{akhtar2012minkowski}, it differs from the similar numeration used in~\cite{cheltsov2018katzarkov} by adding one to the corresponding number of a Minkowski polynomial.

\subsubsection{Family \textnumero2.34}\label{subsubsection:dPxP1_02-34}

The toric Landau--Ginzburg model for this family is given by the Laurent polynomial \(x + y + a_1 x^{-1} y^{-1} + z + a_2 z^{-1}\). It is a Minkowski polynomial \textnumero5 (see~\cite[Appendix~B: bucket~10]{akhtar2012minkowski}). After the change of variables \((x, y, z) \mapsto (z, y, x)\) we obtain the Laurent polynomial \(x + y + z + a_2 x^{-1} + a_1 y^{-1} z^{-1}\).

\subsubsection{Family \textnumero3.27}\label{subsubsection:dPxP1_03-27}

The toric Landau--Ginzburg model for this family is given by the Laurent polynomial \(x + a_1 x^{-1} + y + a_2 y^{-1} + z + a_3 z^{-1}\). It is a Minkowski polynomial \textnumero31 (see~\cite[Appendix~B: bucket~45]{akhtar2012minkowski}).

\subsubsection{Family \textnumero3.28}\label{subsubsection:dPxP1_03-28}

The toric Landau--Ginzburg model for this family is given by the Laurent polynomial \(x + y + a_1 x^{-1} y^{-1} + a_1 a_2 x^{-1} + z + a_3 z^{-1}\). It is a Minkowski polynomial \textnumero30 (see~\cite[Appendix~B: bucket~28]{akhtar2012minkowski}). After the change of variables
\((x, y, z) \mapsto (x^{-1}, z, y)\) we obtain the Laurent polynomial \((a_1 a_2) x + y + z + a_1 x z^{-1} + a_3 y^{-1} + x^{-1}\).

\subsubsection{Family \textnumero4.10}\label{subsubsection:dPxP1_04-10}

The toric Landau--Ginzburg model for this family is given by the Laurent polynomial
\(x + y + a_1 x^{-1} y^{-1} + a_1 a_2 x^{-1} + a_1 a_3 y^{-1} + z + a_4 z^{-1}\).
It is a Minkowski polynomial \textnumero85 (see~\cite[Appendix~B: bucket~48]{akhtar2012minkowski}). After the change of variables
\((x, y, z) \mapsto (y^{-1}, x, z)\) we obtain the Laurent polynomial
\[
  x + \left(a_1 a_2\right) y + z + a_1 x^{-1} y + a_4 z^{-1} + y^{-1} + \left(a_1 a_3\right) x^{-1}.
\]

\subsubsection{Family \textnumero5.3}\label{subsubsection:dPxP1_05-03}

The toric Landau--Ginzburg model for this family is given by the Laurent polynomial
\(x + y + a_1 x^{-1} y^{-1} + a_1 a_2 x^{-1} + a_1 a_3 y^{-1} + a_4 x y + z + a_5^{-1}\).
It is a Minkowski polynomial \textnumero219 (see~\cite[Appendix~B: bucket~76]{akhtar2012minkowski}). After the change of variables
\((x, y, z) \mapsto (y^{-1}, z, x)\) we obtain the Laurent polynomial
\(x + (a_1 a_2) y + z + a_1 y z^{-1} + a_4 y^{-1} z + (a_1 a_3) z^{-1} + y^{-1} + a_5 x^{-1}\).

\subsubsection{Family \textnumero6.1}\label{subsubsection:dPxP1_06-01}

The toric Landau--Ginzburg model for this family is given by the Laurent polynomial
\[
  x + (a_1 a_2 a_4 a_5 + 1) y + a_1 x^{-1} y^{-1} + a_1 (a_2 + a_5) x^{-1} +
  a_1 a_3 y^{-1} + a_4 x y + a_1 a_2 a_5 x^{-1} y + z + a_6 z^{-1}.
\]
It is a Minkowski polynomial \textnumero357 (see~\cite[Appendix~B: bucket~107]{akhtar2012minkowski}). After the change of variables
\((x, y, z) \mapsto (y, x^{-1}, z)\)
we obtain the Laurent polynomial
\[
  \left(a_1 a_3\right) x + y + z + a_1 x y^{-1} + a_4 x^{-1} y + a_6 z^{-1} + \left(a_1 a_2 + a_1 a_5\right) y^{-1} + \left(a_1 a_2 a_4 a_5 + 1\right) x^{-1} + \left(a_1 a_2 a_5\right) x^{-1} y^{-1}.
\]

We apply the mutation
\[
  (x, y, z) \mapsto (x, y, z)^{(M, f, N)}, \;
  M =
  \begin{pmatrix}
    -1 & 0 & 1 \\
    0 & 1 & 0 \\
    0 & 1 & -1
  \end{pmatrix}, \;
  f = x (a_4 y + 1) (a_1 a_2 a_5 + y) + y^2, \;
  N =
  \begin{pmatrix}
    -1 & 1 & -1 \\
    0 & 1 & 0 \\
    -1 & -1 & 0    
  \end{pmatrix},
\]
to obtain a Minkowski polynomial \textnumero1353 from~\cite[Appendix~B: bucket~107]{akhtar2012minkowski}:
\begin{gather*}
  x + y + \left(a_1 a_3\right) z + \left(a_1 a_3 a_4\right) x^{-1} y + a_1 y^{-1} z + \left(a_1 a_2 + a_1 a_5\right) y^{-1} + \left(a_1^{2} a_2 a_3 a_4 a_5 + a_1 a_3 + a_1 a_4 + a_6\right) x^{-1} + \\ \left(a_4 a_6\right) x^{-2} y z^{-1} + \left(a_1^{2} a_2 a_3 a_5 + a_1^{2} a_2 a_4 a_5 + a_1\right) x^{-1} y^{-1} + \left(a_1 a_2 a_4 a_5 a_6 + a_6\right) x^{-2} z^{-1} + \left(a_1^{2} a_2 a_5\right) x^{-1} y^{-2} + \\ \left(a_1 a_2 a_5 a_6\right) x^{-2} y^{-1} z^{-1}.
\end{gather*}

We apply the mutation
\[
  (x, y, z) \mapsto (x, y, z)^{(M, f, N)}, \;
  M =
  \begin{pmatrix}
    0 & -1 & 0 \\
    1 & 1 & 0 \\
    -1 & 0 & 1
  \end{pmatrix}, \;
  f = (a_4 x y + 1) (a_1 a_2 a_5 + x y), \;
  N =
  \begin{pmatrix}
    1 & 1 & 0 \\
    -1 & 0 & 0 \\
    -1 & 0 & -1
  \end{pmatrix},
\]
to obtain a Minkowski polynomial \textnumero1231 from~\cite[Appendix~B: bucket~107]{akhtar2012minkowski}:
\begin{gather*}
  x + y + a_6 z + \left(a_1 a_3 a_4\right) x^{-1} y + \left(a_1 a_2 + a_1 a_5\right) y^{-1} + \left(a_1^{2} a_2 a_3 a_4 a_5 + a_1 a_3 + a_1 a_4 + a_6\right) x^{-1} + \\ \left(a_1 a_3 a_4\right) x^{-2} y z^{-1} + \left(a_1^{2} a_2 a_3 a_5 + a_1^{2} a_2 a_4 a_5 + a_1\right) x^{-1} y^{-1} + \left(a_1^{2} a_2 a_3 a_4 a_5 + a_1 a_3 + a_1 a_4\right) x^{-2} z^{-1} + \\ \left(a_1^{2} a_2 a_5\right) x^{-1} y^{-2} + \left(a_1^{2} a_2 a_3 a_5 + a_1^{2} a_2 a_4 a_5 + a_1\right) x^{-2} y^{-1} z^{-1} + \left(a_1^{2} a_2 a_5\right) x^{-2} y^{-2} z^{-1}.
\end{gather*}

We apply the mutation
\[
  (x, y, z) \mapsto (x, y, z)^{(M, f, N)}, \;
  M =
  \begin{pmatrix}
    0 & -1 & 1 \\
    1 & 0 & 0 \\
    -1 & 0 & -1
  \end{pmatrix}, \;
  f = x y + 1, \;
  N =
  \begin{pmatrix}
    0 & 1 & 0 \\
    -1 & -1 & -1 \\
    0 & -1 & -1    
  \end{pmatrix},
\]
to obtain a Minkowski polynomial \textnumero284 from~\cite[Appendix~B: bucket~107]{akhtar2012minkowski}:
\begin{gather*}
  x + y + a_6 z + \left(a_1 a_3 a_4\right) x^{-1} y + z^{-1} + \left(a_1 a_2 + a_1 a_5\right) y^{-1} + \left(a_1^{2} a_2 a_3 a_4 a_5 + a_1 a_3 + a_1 a_4\right) x^{-1} + \\ \left(a_1^{2} a_2 a_3 a_5 + a_1^{2} a_2 a_4 a_5 + a_1\right) x^{-1} y^{-1} + \left(a_1^{2} a_2 a_5\right) x^{-1} y^{-2}.
\end{gather*}

Finally, after the change of variables
\((x, y, z) \mapsto (x, y^-1, z)\)
we obtain the Laurent polynomial
\begin{gather*}
  x + (a_1 a_2 + a_1 a_5) y + (a_1^{2} a_2 a_5) x^{-1} y^{2} + a_6 z + (a_1^{2} a_2 a_3 a_5 + a_1^{2} a_2 a_4 a_5 + a_1) x^{-1} y + \\ z^{-1} + y^{-1} + (a_1^{2} a_2 a_3 a_4 a_5 + a_1 a_3 + a_1 a_4) x^{-1} + (a_1 a_3 a_4) x^{-1} y^{-1}.
\end{gather*}

\subsubsection{Family \textnumero7.1}\label{subsubsection:dPxP1_07-01}

The toric Landau--Ginzburg model for this family is given by the Laurent polynomial
\begin{gather*}
  x + (a_1 a_2 a_4 a_5 + 1) y + a_1 (a_1 a_3 a_6 (a_2 + a_5) + 1) x^{-1} y^{-1} + \\ a_1 (a_1 a_2 a_3 a_5 a_6 + a_2 + a_5) x^{-1} + a_1 (a_3 + a_6) y^{-1} + a_4 x y + a_1 a_2 a_5 x^{-1} y + a_1^2 a_3 a_6 x^{-1} y^{-2} + z + a_7 z^{-1}.
\end{gather*}
It is a Minkowski polynomial \textnumero506 (see~\cite[Appendix~B: bucket~136]{akhtar2012minkowski}). After the change of variables
\((x, y, z) \mapsto (y^{-1}, x^{-1} y, z)\)
we obtain the Laurent polynomial
\begin{gather*}
  \left(a_1^{2} a_3 a_6\right) x^{2} y^{-1} + \left(a_1^{2} a_2 a_3 a_6 + a_1^{2} a_3 a_5 a_6 + a_1\right) x + \left(a_1^{2} a_2 a_3 a_5 a_6 + a_1 a_2 + a_1 a_5\right) y + \\ \left(a_1 a_2 a_5\right) x^{-1} y^{2} + z + \left(a_1 a_3 + a_1 a_6\right) x y^{-1} + \left(a_1 a_2 a_4 a_5 + 1\right) x^{-1} y + a_7 z^{-1} + y^{-1} + a_4 x^{-1}
\end{gather*}

\subsubsection{Family \textnumero8.1}\label{subsubsection:dPxP1_08-01}

The toric Landau--Ginzburg model for this family is given by the Laurent polynomial
\begin{gather*}
  (a_1 a_4 a_7 (a_3 + a_6) + 1) x + (a_1 a_2 a_5 (a_4 + a_7) + 1) y + 
  a_1 (a_1 a_3 a_6 (a_2 + a_5) + 1) x^{-1} y^{-1} + \\
  a_1 (a_1 a_2 a_3 a_5 a_6 + a_2 + a_5) x^{-1} +
  a_1 (a_1 a_3 a_4 a_6 a_7 + a_3 + a_6) y^{-1} + 
  (a_1 a_2 a_4 a_5 a_7 + a_4 + a_7) x y + \\
  a_1 a_2 a_5 x^{-1} y +
  a_1^2 a_3 a_6 x^{-1} y^{-2} +
  a_4 a_7 x^2 y + z + a_8 z^{-1}.
\end{gather*}
It is a Minkowski polynomial \textnumero769 (see~\cite[Appendix~B: bucket~155]{akhtar2012minkowski}). After the change of variables
\((x, y, z) \mapsto (y z^{-1}, z, x)\)
we obtain the Laurent polynomial
\begin{gather*}
  (a_4 a_7) y^{2} z^{-1} + x + (a_1 a_2 a_4 a_5 a_7 + a_4 + a_7) y + (a_1 a_2 a_4 a_5 + a_1 a_2 a_5 a_7 + 1) z + (a_1 a_2 a_5) y^{-1} z^{2} + \\ (a_1 a_3 a_4 a_7 + a_1 a_4 a_6 a_7 + 1) y z^{-1} + (a_1^{2} a_2 a_3 a_5 a_6 + a_1 a_2 + a_1 a_5) y^{-1} z + (a_1^{2} a_3 a_4 a_6 a_7 + a_1 a_3 + a_1 a_6) z^{-1} + \\ (a_1^{2} a_2 a_3 a_6 + a_1^{2} a_3 a_5 a_6 + a_1) y^{-1} + a_8 x^{-1} + (a_1^{2} a_3 a_6) y^{-1} z^{-1}
\end{gather*}


\subsection{Family \textnumero2.34}\label{subsection:02-34_parametrised}

The pencil \(\mathcal{S}(\alpha)\), where \(\alpha = (a_1, a_2) \in (\mathbb{C}^*)^2\), is defined by the equation
\begin{gather*}
  X^{2} Y Z + X Y^{2} Z + X Y Z^{2} + a_2 Y Z T^{2} + a_1 X T^{3} = \lambda X Y Z T.
\end{gather*}
Note that \(\mathcal{S}(\alpha)_{\infty} = S_{(X)} + S_{(Y)} + S_{(Z)} + S_{(T)}\). The base locus of the pencil \(\mathcal{S}(\alpha)\) consists of the curves
\begin{gather*}
  C_{1} = C_{(X, Y)}, \;
  C_{2} = C_{(X, Z)}, \;
  C_{3} = C_{(X, T)}, \;
  C_{4} = C_{(Y, T)}, \;
  C_{5} = C_{(Z, T)}, \;
  C_{6} = C_{(T, X + Y + Z)}.
\end{gather*}
Their linear equivalence classes on the generic member \(\mathcal{S}(\alpha)_{\Bbbk}\) of the pencil satisfy the following relations:
\begin{gather*}
  \begin{pmatrix}
    [C_{1}] \\ [C_{2}] \\ [C_{6}] \\ [H_{\mathcal{S}(\alpha)}]
  \end{pmatrix} = 
  \begin{pmatrix}
    -2 & 0 & 3 \\
    -2 & 3 & 0 \\
    -3 & 2 & 2 \\
    -2 & 3 & 3
  \end{pmatrix} \cdot
  \begin{pmatrix}
    [C_{3}] \\ [C_{4}] \\ [C_{5}]
  \end{pmatrix}.
\end{gather*}

For a general choice of \(\lambda \in \mathbb{C}\) and \(\alpha \in (\mathbb{C}^*)^2\) the surface \(\mathcal{S}(\alpha)_{\lambda}\) has the following singularities:
\begin{itemize}\setlength{\itemindent}{2cm}
\item[\(P_{1} = P_{(X, Y, T)}\):] type \(\mathbb{A}_4\) with the quadratic term \(X \cdot Y\);
\item[\(P_{2} = P_{(X, Z, T)}\):] type \(\mathbb{A}_4\) with the quadratic term \(X \cdot Z\);
\item[\(P_{3} = P_{(Y, Z, T)}\):] type \(\mathbb{A}_2\) with the quadratic term \(Y \cdot Z\);
\item[\(P_{4} = P_{(X, T, Y + Z)}\):] type \(\mathbb{A}_1\) with the quadratic term \(X (X + Y + Z - \lambda T) + a_2 T^2\);
\item[\(P_{5} = P_{(Y, T, X + Z)}\):] type \(\mathbb{A}_2\) with the quadratic term \(Y \cdot (X + Y + Z - \lambda T)\);
\item[\(P_{6} = P_{(Z, T, X + Y)}\):] type \(\mathbb{A}_2\) with the quadratic term \(Z \cdot (X + Y + Z - \lambda T)\).
\end{itemize}

Galois action on the lattice \(L(\alpha)_{\lambda}\) is trivial. The intersection matrix on \(L(\alpha)_{\lambda} = L(\alpha)_{\mathcal{S}}\) is represented by
\begin{table}[H]
  \begin{tabular}{|c||cccc|cccc|cc|c|cc|cc|ccc|}
    \hline
    \(\bullet\) & \(E_1^1\) & \(E_1^2\) & \(E_1^3\) & \(E_1^4\) & \(E_2^1\) & \(E_2^2\) & \(E_2^3\) & \(E_2^4\) & \(E_3^1\) & \(E_3^2\) & \(E_4^1\) & \(E_5^1\) & \(E_5^2\) & \(E_6^1\) & \(E_6^2\) & \(\widetilde{C_{3}}\) & \(\widetilde{C_{4}}\) & \(\widetilde{C_{5}}\) \\
    \hline
    \hline
    \(\widetilde{C_{3}}\) & \(1\) & \(0\) & \(0\) & \(0\) & \(1\) & \(0\) & \(0\) & \(0\) & \(0\) & \(0\) & \(1\) & \(0\) & \(0\) & \(0\) & \(0\) & \(-2\) & \(0\) & \(0\) \\
    \(\widetilde{C_{4}}\) & \(0\) & \(0\) & \(0\) & \(1\) & \(0\) & \(0\) & \(0\) & \(0\) & \(1\) & \(0\) & \(0\) & \(1\) & \(0\) & \(0\) & \(0\) & \(0\) & \(-2\) & \(0\) \\
    \(\widetilde{C_{5}}\) & \(0\) & \(0\) & \(0\) & \(0\) & \(0\) & \(0\) & \(0\) & \(1\) & \(0\) & \(1\) & \(0\) & \(0\) & \(0\) & \(0\) & \(1\) & \(0\) & \(0\) & \(-2\) \\
    \hline
  \end{tabular}.
\end{table}
Note that the intersection matrix is non-degenerate.

Discriminant groups and discriminant forms of the lattices \(L(\alpha)_{\mathcal{S}}\) and \(H \oplus \Pic(X)\) are given by
\begin{gather*}
  G' = 
  \begin{pmatrix}
    \frac{1}{3} & \frac{2}{3} & 0 & \frac{1}{3} & \frac{2}{3} & \frac{1}{3} & 0 & \frac{2}{3} & \frac{8}{9} & \frac{1}{9} & 0 & \frac{1}{9} & \frac{5}{9} & \frac{4}{9} & \frac{8}{9} & 0 & \frac{2}{3} & \frac{1}{3}
  \end{pmatrix}, \\
  G'' = 
  \begin{pmatrix}
    0 & 0 & \frac{7}{9} & \frac{1}{3}
  \end{pmatrix}; \;
  B' = 
  \begin{pmatrix}
    \frac{2}{9}
  \end{pmatrix}, \;
  B'' = 
  \begin{pmatrix}
    \frac{7}{9}
  \end{pmatrix}; \;
  Q' =
  \begin{pmatrix}
    \frac{2}{9}
  \end{pmatrix}; \;
  Q'' =
  \begin{pmatrix}
    \frac{16}{9}
  \end{pmatrix}.
\end{gather*}


\subsection{Family \textnumero3.27}\label{subsection:03-27_parametrised}

The pencil \(\mathcal{S}(\alpha)\), where \(\alpha = (a_1, a_2, a_3) \in (\mathbb{C}^*)^3\), is defined by the equation
\begin{gather*}
  X^{2} Y Z + X Y^{2} Z + X Y Z^{2} + a_3 X Y T^{2} + a_2 X Z T^{2} + a_1 Y Z T^{2} = \lambda X Y Z T.
\end{gather*}
Note that \(\mathcal{S}(\alpha)_{\infty} = S_{(X)} + S_{(Y)} + S_{(Z)} + S_{(T)}\). The base locus of the pencil \(\mathcal{S}(\alpha)\) consists of the curves
\begin{gather*}
  C_{1} = C_{(X, Y)}, \;
  C_{2} = C_{(X, Z)}, \;
  C_{3} = C_{(X, T)}, \;
  C_{4} = C_{(Y, Z)}, \;
  C_{5} = C_{(Y, T)}, \;
  C_{6} = C_{(Z, T)}, \;
  C_{7} = C_{(T, X + Y + Z)}.
\end{gather*}
Their linear equivalence classes on the generic member \(\mathcal{S}(\alpha)_{\Bbbk}\) of the pencil satisfy the following relations:
\begin{gather*}
  \begin{pmatrix}
    [C_{2}] \\ [C_{4}] \\ [C_{7}] \\ [H_{\mathcal{S}(\alpha)}]
  \end{pmatrix} = 
  \begin{pmatrix}
    1 & 0 & 2 & -2 \\
    1 & 2 & 0 & -2 \\
    2 & 1 & 1 & -3 \\
    2 & 2 & 2 & -2
  \end{pmatrix} \cdot
  \begin{pmatrix}
    [C_{1}] \\ [C_{3}] \\ [C_{5}] \\ [C_{6}]
  \end{pmatrix}.
\end{gather*}

For a general choice of \(\lambda \in \mathbb{C}\) and \(\alpha \in (\mathbb{C}^*)^3\) the surface \(\mathcal{S}(\alpha)_{\lambda}\) has the following singularities:
\begin{itemize}\setlength{\itemindent}{2cm}
\item[\(P_{1} = P_{(X, Y, Z)}\):] type \(\mathbb{A}_1\) with the quadratic term \(a_3 X Y + a_2 X Z + a_1 Y Z\);
\item[\(P_{2} = P_{(X, Y, T)}\):] type \(\mathbb{A}_3\) with the quadratic term \(X \cdot Y\);
\item[\(P_{3} = P_{(X, Z, T)}\):] type \(\mathbb{A}_3\) with the quadratic term \(X \cdot Z\);
\item[\(P_{4} = P_{(Y, Z, T)}\):] type \(\mathbb{A}_3\) with the quadratic term \(Y \cdot Z\);
\item[\(P_{5} = P_{(X, T, Y + Z)}\):] type \(\mathbb{A}_1\) with the quadratic term \(X (X + Y + Z - \lambda T) + a_1 T^2\);
\item[\(P_{6} = P_{(Y, T, X + Z)}\):] type \(\mathbb{A}_1\) with the quadratic term \(Y (X + Y + Z - \lambda T) + a_2 T^2\);
\item[\(P_{7} = P_{(Z, T, X + Y)}\):] type \(\mathbb{A}_1\) with the quadratic term \(Z (X + Y + Z - \lambda T) + a_3 T^2\).
\end{itemize}

Galois action on the lattice \(L(\alpha)_{\lambda}\) is trivial. The intersection matrix on \(L(\alpha)_{\lambda} = L(\alpha)_{\mathcal{S}}\) is represented by
\begin{table}[H]
  \begin{tabular}{|c||c|ccc|ccc|ccc|c|c|c|cccc|}
    \hline
    \(\bullet\) & \(E_1^1\) & \(E_2^1\) & \(E_2^2\) & \(E_2^3\) & \(E_3^1\) & \(E_3^2\) & \(E_3^3\) & \(E_4^1\) & \(E_4^2\) & \(E_4^3\) & \(E_5^1\) & \(E_6^1\) & \(E_7^1\) & \(\widetilde{C_{1}}\) & \(\widetilde{C_{3}}\) & \(\widetilde{C_{5}}\) & \(\widetilde{C_{6}}\) \\
    \hline
    \hline
    \(\widetilde{C_{1}}\) & \(1\) & \(0\) & \(1\) & \(0\) & \(0\) & \(0\) & \(0\) & \(0\) & \(0\) & \(0\) & \(0\) & \(0\) & \(0\) & \(-2\) & \(0\) & \(0\) & \(0\) \\
    \(\widetilde{C_{3}}\) & \(0\) & \(1\) & \(0\) & \(0\) & \(1\) & \(0\) & \(0\) & \(0\) & \(0\) & \(0\) & \(1\) & \(0\) & \(0\) & \(0\) & \(-2\) & \(0\) & \(0\) \\
    \(\widetilde{C_{5}}\) & \(0\) & \(0\) & \(0\) & \(1\) & \(0\) & \(0\) & \(0\) & \(1\) & \(0\) & \(0\) & \(0\) & \(1\) & \(0\) & \(0\) & \(0\) & \(-2\) & \(0\) \\
    \(\widetilde{C_{6}}\) & \(0\) & \(0\) & \(0\) & \(0\) & \(0\) & \(0\) & \(1\) & \(0\) & \(0\) & \(1\) & \(0\) & \(0\) & \(1\) & \(0\) & \(0\) & \(0\) & \(-2\) \\
    \hline
  \end{tabular}.
\end{table}
Note that the intersection matrix is non-degenerate.

Discriminant groups and discriminant forms of the lattices \(L(\alpha)_{\mathcal{S}}\) and \(H \oplus \Pic(X)\) are given by
\begin{gather*}
  G' =
  \begin{pmatrix}
    0 & 0 & 0 & 0 & \frac{1}{2} & 0 & \frac{1}{2} & 0 & 0 & 0 & \frac{1}{2} & 0 & \frac{1}{2} & 0 & 0 & 0 & 0 \\
    0 & \frac{1}{2} & 0 & \frac{1}{2} & 0 & 0 & 0 & 0 & 0 & 0 & \frac{1}{2} & \frac{1}{2} & 0 & 0 & 0 & 0 & 0 \\
    \frac{1}{2} & \frac{3}{4} & \frac{1}{2} & \frac{1}{4} & \frac{1}{4} & \frac{1}{2} & \frac{3}{4} & \frac{1}{4} & \frac{1}{2} & \frac{3}{4} & 0 & \frac{1}{2} & \frac{1}{2} & 0 & 0 & 0 & 0
  \end{pmatrix}, \\
  G'' =
  \begin{pmatrix}
    0 & 0 & \frac{1}{2} & 0 & 0 \\
    0 & 0 & 0 & \frac{1}{2} & 0 \\
    0 & 0 & -\frac{1}{4} & -\frac{1}{4} & \frac{1}{4}
  \end{pmatrix}; \;
  B' = 
  \begin{pmatrix}
    0 & \frac{1}{2} & 0 \\
    \frac{1}{2} & 0 & 0 \\
    0 & 0 & \frac{1}{4}
  \end{pmatrix}, \;
  B'' = 
  \begin{pmatrix}
    0 & \frac{1}{2} & 0 \\
    \frac{1}{2} & 0 & 0 \\
    0 & 0 & \frac{3}{4}
  \end{pmatrix}; \;
  \begin{pmatrix}
    Q' \\ Q''
  \end{pmatrix}
  =
  \begin{pmatrix}
    0 & 0 & \frac{1}{4} \\
    0 & 0 & \frac{7}{4}    
  \end{pmatrix}.
\end{gather*}


\subsection{Family \textnumero3.28}\label{subsection:03-28_parametrised}

The pencil \(\mathcal{S}(\alpha)\), where \(\alpha = (a_1, a_2, a_3) \in (\mathbb{C}^*)^3\), is defined by the equation
\begin{gather*}
  \left(a_1 a_2\right) X^{2} Y Z + X Y^{2} Z + X Y Z^{2} + a_1 X^{2} Y T + a_3 X Z T^{2} + Y Z T^{2} = \lambda X Y Z T.
\end{gather*}
Note that \(\mathcal{S}(\alpha)_{\infty} = S_{(X)} + S_{(Y)} + S_{(Z)} + S_{(T)}\).
The base locus of the pencil \(\mathcal{S}(\alpha)\) consists of the curves
\begin{gather*}
  C_{1} = C_{(X, Y)}, \;
  C_{2} = C_{(X, Z)}, \;
  C_{3} = C_{(X, T)}, \;
  C_{4} = C_{(Y, Z)}, \;
  C_{5} = C_{(Y, T)}, \;
  C_{6} = C_{(Z, T)}, \;
  C_{7} = C_{(T, a_1 a_2 X + Y + Z)}.
\end{gather*}
Their linear equivalence classes on the generic member \(\mathcal{S}(\alpha)_{\Bbbk}\) of the pencil satisfy the following relations:
\begin{gather*}
  \begin{pmatrix}
    [C_{2}] \\ [C_{4}] \\ [C_{6}] \\ [C_{7}]
  \end{pmatrix} = 
  \begin{pmatrix}
    -1 & -2 & 0 & 1 \\
    -1 & 0 & -2 & 1 \\
    3 & 4 & 2 & -2 \\
    -3 & -5 & -3 & 3
  \end{pmatrix} \cdot
  \begin{pmatrix}
    [C_{1}] \\ [C_{3}] \\ [C_{5}] \\ [H_{\mathcal{S}(\alpha)}]
  \end{pmatrix}.
\end{gather*}

For a general choice of \(\lambda \in \mathbb{C}\) and \(\alpha \in (\mathbb{C}^*)^3\) the surface \(\mathcal{S}(\alpha)_{\lambda}\) has the following singularities:
\begin{itemize}\setlength{\itemindent}{2cm}
\item[\(P_{1} = P_{(X, Y, Z)}\):] type \(\mathbb{A}_2\) with the quadratic term \(Z \cdot (a_3 X + Y)\);
\item[\(P_{2} = P_{(X, Y, T)}\):] type \(\mathbb{A}_3\) with the quadratic term \(X \cdot Y\);
\item[\(P_{3} = P_{(X, Z, T)}\):] type \(\mathbb{A}_4\) with the quadratic term \(X \cdot Z\);
\item[\(P_{4} = P_{(Y, Z, T)}\):] type \(\mathbb{A}_2\) with the quadratic term \(Y \cdot (a_2 Z + T)\);
\item[\(P_{5} = P_{(X, T, Y + Z)}\):] type \(\mathbb{A}_1\) with the quadratic term \(X (a_1 a_2 X + Y + Z - \lambda T) + T^2\);
\item[\(P_{6} = P_{(Y, T, a_1 a_2 X + Z)}\):] type \(\mathbb{A}_1\) with the quadratic term \(Y (a_2 (a_1 a_2 X + Y + Z - \lambda T) - T) + a_2 a_3 T^2\).
\end{itemize}

Galois action on the lattice \(L(\alpha)_{\lambda}\) is trivial. The intersection matrix on \(L(\alpha)_{\lambda} = L(\alpha)_{\mathcal{S}}\) is represented by
\begin{table}[H]
  \begin{tabular}{|c||cc|ccc|cccc|cc|c|c|cccc|}
    \hline
    \(\bullet\) & \(E_1^1\) & \(E_1^2\) & \(E_2^1\) & \(E_2^2\) & \(E_2^3\) & \(E_3^1\) & \(E_3^2\) & \(E_3^3\) & \(E_3^4\) & \(E_4^1\) & \(E_4^2\) & \(E_5^1\) & \(E_6^1\) & \(\widetilde{C_{1}}\) & \(\widetilde{C_{3}}\) & \(\widetilde{C_{5}}\) & \(\widetilde{H_{\mathcal{S}}}\) \\
    \hline
    \hline
    \(\widetilde{C_{1}}\) & \(1\) & \(0\) & \(0\) & \(1\) & \(0\) & \(0\) & \(0\) & \(0\) & \(0\) & \(0\) & \(0\) & \(0\) & \(0\) & \(-2\) & \(0\) & \(0\) & \(1\) \\
    \(\widetilde{C_{3}}\) & \(0\) & \(0\) & \(1\) & \(0\) & \(0\) & \(1\) & \(0\) & \(0\) & \(0\) & \(0\) & \(0\) & \(1\) & \(0\) & \(0\) & \(-2\) & \(0\) & \(1\) \\
    \(\widetilde{C_{5}}\) & \(0\) & \(0\) & \(0\) & \(0\) & \(1\) & \(0\) & \(0\) & \(0\) & \(0\) & \(1\) & \(0\) & \(0\) & \(1\) & \(0\) & \(0\) & \(-2\) & \(1\) \\
    \(\widetilde{H_{\mathcal{S}}}\) & \(0\) & \(0\) & \(0\) & \(0\) & \(0\) & \(0\) & \(0\) & \(0\) & \(0\) & \(0\) & \(0\) & \(0\) & \(0\) & \(1\) & \(1\) & \(1\) & \(4\) \\
    \hline
  \end{tabular}.
\end{table}
Note that the intersection matrix is non-degenerate.

Discriminant groups and discriminant forms of the lattices \(L(\alpha)_{\mathcal{S}}\) and \(H \oplus \Pic(X)\) are given by
\begin{gather*}
  G' =
  \begin{pmatrix}
    0 & \frac{1}{2} & \frac{13}{16} & \frac{1}{4} & \frac{3}{16} & \frac{1}{2} & \frac{5}{8} & \frac{3}{4} & \frac{7}{8} & \frac{3}{4} & \frac{3}{8} & \frac{11}{16} & \frac{9}{16} & \frac{1}{2} & \frac{3}{8} & \frac{1}{8} & \frac{3}{4}
  \end{pmatrix}, \\
  G'' =
  \begin{pmatrix}
    0 & 0 & -\frac{3}{16} & -\frac{3}{8} & \frac{1}{16}
  \end{pmatrix}; \;
  B' = 
  \begin{pmatrix}
    \frac{1}{16}
  \end{pmatrix}, \;
  B'' = 
  \begin{pmatrix}
    \frac{15}{16}
  \end{pmatrix}; \;
  Q' =
  \begin{pmatrix}
    \frac{1}{16}
  \end{pmatrix}, \;
  Q'' =
  \begin{pmatrix}
    \frac{31}{16}
  \end{pmatrix}.
\end{gather*}


\subsection{Family \textnumero4.10}\label{subsection:04-10_parametrised}

The pencil \(\mathcal{S}(\alpha)\), where \(\alpha = (a_1, \ldots, a_4) \in (\mathbb{C}^*)^4\), is defined by the equation
\begin{gather*}
  X^{2} Y Z + \left(a_1 a_2\right) X Y^{2} Z + X Y Z^{2} + a_1 Y^{2} Z T + a_4 X Y T^{2} + X Z T^{2} + \left(a_1 a_3\right) Y Z T^{2} = \lambda X Y Z T.
\end{gather*}
Note that \(\mathcal{S}(\alpha)_{\infty} = S_{(X)} + S_{(Y)} + S_{(Z)} + S_{(T)}\). The base locus of the pencil \(\mathcal{S}(\alpha)\) consists of the curves
\begin{gather*}
  C_{1} = C_{(X, Y)}, \;
  C_{2} = C_{(X, Z)}, \;
  C_{3} = C_{(X, T)}, \;
  C_{4} = C_{(Y, Z)}, \\
  C_{5} = C_{(Y, T)}, \;
  C_{6} = C_{(Z, T)}, \;
  C_{7} = C_{(X, Y + a_3 T)}, \;
  C_{8} = C_{(T, X + a_1 a_2 Y + Z)}.
\end{gather*}
Their linear equivalence classes on the generic member \(\mathcal{S}(\alpha)_{\Bbbk}\) of the pencil satisfy the following relations:
\begin{gather*}
  \begin{pmatrix}
    [C_{2}] \\ [C_{4}] \\ [C_{7}] \\ [C_{8}]
  \end{pmatrix} = 
  \begin{pmatrix}
    1 & 0 & 2 & -2 & 0 \\
    -1 & 0 & -2 & 0 & 1 \\
    -2 & -1 & -2 & 2 & 1 \\
    0 & -1 & -1 & -1 & 1
  \end{pmatrix} \cdot
  \begin{pmatrix}
    [C_{1}] & [C_{3}] & [C_{5}] & [C_{6}] & [H_{\mathcal{S}(\alpha)}]
  \end{pmatrix}^T.
\end{gather*}

For a general choice of \(\lambda \in \mathbb{C}\) and \(\alpha \in (\mathbb{C}^*)^4\) the surface \(\mathcal{S}(\alpha)_{\lambda}\) has the following singularities:
\begin{itemize}\setlength{\itemindent}{2cm}
\item[\(P_{1} = P_{(X, Y, Z)}\):] type \(\mathbb{A}_1\) with the quadratic term \(a_1 a_3 Y Z + a_4 X Y + X Z\);
\item[\(P_{2} = P_{(X, Y, T)}\):] type \(\mathbb{A}_3\) with the quadratic term \(X \cdot Y\);
\item[\(P_{3} = P_{(X, Z, T)}\):] type \(\mathbb{A}_2\) with the quadratic term \(Z \cdot (a_2 X + T)\);
\item[\(P_{4} = P_{(Y, Z, T)}\):] type \(\mathbb{A}_3\) with the quadratic term \(Y \cdot Z\);
\item[\(P_{5} = P_{(Y, T, X + Z)}\):] type \(\mathbb{A}_1\) with the quadratic term \(Y (X + a_1 a_2 Y + Z - \lambda T) + T^2\);
\item[\(P_{6} = P_{(Z, T, a_1 a_2 Y + X)}\):] type \(\mathbb{A}_1\) with the quadratic term \(Z (a_2 (X + a_1 a_2 Y + Z - \lambda T) - T) + a_2 a_4 T^2\).
\end{itemize}

Galois action on the lattice \(L(\alpha)_{\lambda}\) is trivial. The intersection matrix on \(L(\alpha)_{\lambda} = L(\alpha)_{\mathcal{S}}\) is represented by
\begin{table}[H]
  \begin{tabular}{|c||c|ccc|cc|ccc|c|c|ccccc|}
    \hline
    \(\bullet\) & \(E_1^1\) & \(E_2^1\) & \(E_2^2\) & \(E_2^3\) & \(E_3^1\) & \(E_3^2\) & \(E_4^1\) & \(E_4^2\) & \(E_4^3\) & \(E_5^1\) & \(E_6^1\) & \(\widetilde{C_{1}}\) & \(\widetilde{C_{3}}\) & \(\widetilde{C_{5}}\) & \(\widetilde{C_{6}}\) & \(\widetilde{H_{\mathcal{S}}}\) \\
    \hline
    \hline
    \(\widetilde{C_{1}}\) & \(1\) & \(0\) & \(1\) & \(0\) & \(0\) & \(0\) & \(0\) & \(0\) & \(0\) & \(0\) & \(0\) & \(-2\) & \(0\) & \(0\) & \(0\) & \(1\) \\
    \(\widetilde{C_{3}}\) & \(0\) & \(1\) & \(0\) & \(0\) & \(1\) & \(0\) & \(0\) & \(0\) & \(0\) & \(0\) & \(0\) & \(0\) & \(-2\) & \(0\) & \(0\) & \(1\) \\
    \(\widetilde{C_{5}}\) & \(0\) & \(0\) & \(0\) & \(1\) & \(0\) & \(0\) & \(1\) & \(0\) & \(0\) & \(1\) & \(0\) & \(0\) & \(0\) & \(-2\) & \(0\) & \(1\) \\
    \(\widetilde{C_{6}}\) & \(0\) & \(0\) & \(0\) & \(0\) & \(0\) & \(1\) & \(0\) & \(0\) & \(1\) & \(0\) & \(1\) & \(0\) & \(0\) & \(0\) & \(-2\) & \(1\) \\
    \(\widetilde{H_{\mathcal{S}}}\) & \(0\) & \(0\) & \(0\) & \(0\) & \(0\) & \(0\) & \(0\) & \(0\) & \(0\) & \(0\) & \(0\) & \(1\) & \(1\) & \(1\) & \(1\) & \(4\) \\
    \hline
  \end{tabular}.
\end{table}
Note that the intersection matrix is non-degenerate.

Discriminant groups and discriminant forms of the lattices \(L(\alpha)_{\mathcal{S}}\) and \(H \oplus \Pic(X)\) are given by
\begin{gather*}
  G' =
  \begin{pmatrix}
    \frac{1}{2} & \frac{1}{2} & 0 & \frac{1}{2} & 0 & 0 & 0 & 0 & 0 & 0 & \frac{1}{2} & 0 & 0 & 0 & 0 & \frac{1}{2} \\
    \frac{9}{14} & \frac{1}{14} & 0 & \frac{9}{14} & \frac{2}{7} & \frac{3}{7} & \frac{5}{14} & \frac{3}{7} & \frac{1}{2} & \frac{9}{14} & \frac{2}{7} & \frac{2}{7} & \frac{1}{7} & \frac{2}{7} & \frac{4}{7} & \frac{13}{14}
  \end{pmatrix}, \\
  G'' =
  \begin{pmatrix}
    0 & 0 & -\frac{1}{2} & \frac{1}{2} & 0 & 0 \\
    0 & 0 & \frac{9}{14} & \frac{1}{7} & \frac{5}{7} & -\frac{1}{14}
  \end{pmatrix}; \;
  B' = 
  \begin{pmatrix}
    0 & \frac{1}{2} \\
    \frac{1}{2} & \frac{2}{7}
  \end{pmatrix}, \;
  B'' = 
  \begin{pmatrix}
    0 & \frac{1}{2} \\
    \frac{1}{2} & \frac{5}{7}
  \end{pmatrix}; \;
  \begin{pmatrix}
    Q' \\ Q''
  \end{pmatrix}
  =
  \begin{pmatrix}
    1 & \frac{2}{7} \\
    1 & \frac{12}{7}
  \end{pmatrix}.
\end{gather*}


\subsection{Family \textnumero5.3}\label{subsection:05-03_parametrised}

The pencil \(\mathcal{S}(\alpha)\), where \(\alpha = (a_1, \ldots, a_5) \in (\mathbb{C}^*)^5\), is defined by the equation
\begin{gather*}
  X^{2} Y Z + (a_1 a_2) X Y^{2} Z + X Y Z^{2} + a_1 X Y^{2} T + a_4 X Z^{2} T + (a_1 a_3) X Y T^{2} + X Z T^{2} + a_5 Y Z T^{2} = \lambda X Y Z T.
\end{gather*}
Note that \(\mathcal{S}(\alpha)_{\infty} = S_{(X)} + S_{(Y)} + S_{(Z)} + S_{(T)}\). The base locus of the pencil \(\mathcal{S}(\alpha)\) consists of the curves
\begin{gather*}
  C_{1} = C_{(X, Y)}, \;
  C_{2} = C_{(X, Z)}, \;
  C_{3} = C_{(X, T)}, \;
  C_{4} = C_{(Y, Z)}, \;
  C_{5} = C_{(Y, T)}, \\
  C_{6} = C_{(Z, T)}, \;
  C_{7} = C_{(Y, a_4 Z + T)}, \;
  C_{8} = C_{(Z, Y + a_3 T)}, \;
  C_{9} = C_{(T, X + a_1 a_2 Y + Z)}.
\end{gather*}
Their linear equivalence classes on the generic member \(\mathcal{S}(\alpha)_{\Bbbk}\) of the pencil satisfy the following relations:
\begin{gather*}
  \begin{pmatrix}
    [C_{2}] \\ [C_{7}] \\ [C_{8}] \\ [C_{9}]
  \end{pmatrix} = 
  \begin{pmatrix}
    -1 & -2 & 0 & 0 & 0 & 1 \\
    -1 & 0 & -1 & -1 & 0 & 1 \\
    1 & 2 & -1 & 0 & -1 & 0 \\
    0 & -1 & 0 & -1 & -1 & 1
  \end{pmatrix} \cdot
  \begin{pmatrix}
    [C_{1}] & [C_{3}] & [C_{4}] & [C_{5}] & [C_{6}] & [H_{\mathcal{S}(\alpha)}]
  \end{pmatrix}^T.
\end{gather*}

For a general choice of \(\lambda \in \mathbb{C}\) and \(\alpha \in (\mathbb{C}^*)^5\) the surface \(\mathcal{S}(\alpha)_{\lambda}\) has the following singularities:
\begin{itemize}\setlength{\itemindent}{2cm}
\item[\(P_{1} = P_{(X, Y, Z)}\):] type \(\mathbb{A}_1\) with the quadratic term \(a_1 a_3 X Y + a_5 Y Z + X Z\);
\item[\(P_{2} = P_{(X, Y, T)}\):] type \(\mathbb{A}_2\) with the quadratic term \(X \cdot (Y + a_4 T)\);
\item[\(P_{3} = P_{(X, Z, T)}\):] type \(\mathbb{A}_2\) with the quadratic term \(X \cdot (a_2 Z + T)\);
\item[\(P_{4} = P_{(Y, Z, T)}\):] type \(\mathbb{A}_3\) with the quadratic term \(Y \cdot Z\);
\item[\(P_{5} = P_{(X, T, a_1 a_2 Y + Z)}\):] type \(\mathbb{A}_1\) with the quadratic term
  \[
    X (a_2 (X + a_1 a_2 Y + Z - (a_1 a_2 a_4 + \lambda) T) - T) + a_2 a_5 T^2.
  \]
\end{itemize}

Galois action on the lattice \(L(\alpha)_{\lambda}\) is trivial. The intersection matrix on \(L(\alpha)_{\lambda} = L(\alpha)_{\mathcal{S}}\) is represented by
\begin{table}[H]
  \begin{tabular}{|c||c|cc|cc|ccc|c|cccccc|}
    \hline
    \(\bullet\) & \(E_1^1\) & \(E_2^1\) & \(E_2^2\) & \(E_3^1\) & \(E_3^2\) & \(E_4^1\) & \(E_4^2\) & \(E_4^3\) & \(E_5^1\) & \(\widetilde{C_{1}}\) & \(\widetilde{C_{3}}\) & \(\widetilde{C_{4}}\) & \(\widetilde{C_{5}}\) & \(\widetilde{C_{6}}\) & \(\widetilde{H_{\mathcal{S}}}\) \\
    \hline
    \hline
    \(\widetilde{C_{1}}\) & \(1\) & \(1\) & \(0\) & \(0\) & \(0\) & \(0\) & \(0\) & \(0\) & \(0\) & \(-2\) & \(0\) & \(0\) & \(0\) & \(0\) & \(1\) \\
    \(\widetilde{C_{3}}\) & \(0\) & \(1\) & \(0\) & \(1\) & \(0\) & \(0\) & \(0\) & \(0\) & \(1\) & \(0\) & \(-2\) & \(0\) & \(0\) & \(0\) & \(1\) \\
    \(\widetilde{C_{4}}\) & \(1\) & \(0\) & \(0\) & \(0\) & \(0\) & \(0\) & \(1\) & \(0\) & \(0\) & \(0\) & \(0\) & \(-2\) & \(0\) & \(0\) & \(1\) \\
    \(\widetilde{C_{5}}\) & \(0\) & \(0\) & \(1\) & \(0\) & \(0\) & \(1\) & \(0\) & \(0\) & \(0\) & \(0\) & \(0\) & \(0\) & \(-2\) & \(0\) & \(1\) \\
    \(\widetilde{C_{6}}\) & \(0\) & \(0\) & \(0\) & \(0\) & \(1\) & \(0\) & \(0\) & \(1\) & \(0\) & \(0\) & \(0\) & \(0\) & \(0\) & \(-2\) & \(1\) \\
    \(\widetilde{H_{\mathcal{S}}}\) & \(0\) & \(0\) & \(0\) & \(0\) & \(0\) & \(0\) & \(0\) & \(0\) & \(0\) & \(1\) & \(1\) & \(1\) & \(1\) & \(1\) & \(4\) \\
    \hline
  \end{tabular}.
\end{table}
The intersection matrix is non-degenerate.

Discriminant groups and discriminant forms of the lattices \(L(\alpha)_{\mathcal{S}}\) and \(H \oplus \Pic(X)\) are given by
\begin{gather*}
  G' =
  \begin{pmatrix}
    \frac{1}{2} & 0 & 0 & 0 & 0 & \frac{1}{2} & 0 & \frac{1}{2} & \frac{1}{2} & 0 & 0 & 0 & 0 & 0 & \frac{1}{2} \\
    0 & 0 & \frac{1}{2} & 0 & 0 & \frac{1}{2} & 0 & 0 & 0 & \frac{1}{2} & 0 & \frac{1}{2} & 0 & 0 & 0 \\
    \frac{1}{12} & \frac{1}{12} & \frac{1}{3} & \frac{3}{4} & \frac{5}{6} & \frac{2}{3} & \frac{3}{4} & \frac{5}{6} & \frac{1}{3} & \frac{1}{6} & \frac{2}{3} & 0 & \frac{7}{12} & \frac{11}{12} & \frac{1}{6}
  \end{pmatrix}, \;
  B' = 
  \begin{pmatrix}
    0 & \frac{1}{2} & \frac{1}{2} \\
    \frac{1}{2} & 0 & \frac{1}{2} \\
    \frac{1}{2} & \frac{1}{2} & \frac{7}{12}
  \end{pmatrix}; \\
  G'' =
  \begin{pmatrix}
    0 & 0 & -\frac{1}{2} & 0 & \frac{1}{2} & 0 & 0 \\
    0 & 0 & -\frac{1}{2} & 0 & 0 & 0 & \frac{1}{2} \\
    0 & 0 & -\frac{5}{12} & \frac{1}{12} & \frac{1}{12} & -\frac{1}{6} & \frac{1}{4}
  \end{pmatrix}, \;
  B'' = 
  \begin{pmatrix}
    0 & \frac{1}{2} & \frac{1}{2} \\
    \frac{1}{2} & 0 & \frac{1}{2} \\
    \frac{1}{2} & \frac{1}{2} & \frac{5}{12}
  \end{pmatrix}; \;
  \begin{pmatrix}
    Q' \\ Q''
  \end{pmatrix}
  =
  \begin{pmatrix}
    1 & 0 & \frac{7}{12} \\
    1 & 0 & \frac{17}{12}.
  \end{pmatrix}.
\end{gather*}


\subsection{Family \textnumero6.1}\label{subsection:06-01_parametrised}

The pencil \(\mathcal{S}(\alpha)\), where \(\alpha = (a_1, \ldots, a_6) \in (\mathbb{C}^*)^6\), is defined by the equation
\begin{gather*}
  X^{2} Y Z + (a_1 a_2 + a_1 a_5) X Y^{2} Z + (a_1^{2} a_2 a_5) Y^{3} Z + a_6 X Y Z^{2} + (a_1^{2} a_2 a_3 a_5 + a_1^{2} a_2 a_4 a_5 + a_1) Y^{2} Z T + \\ X Y T^{2} + X Z T^{2} + (a_1^{2} a_2 a_3 a_4 a_5 + a_1 a_3 + a_1 a_4) Y Z T^{2} + (a_1 a_3 a_4) Z T^{3} = \lambda X Y Z T.
\end{gather*}
Note that \(\mathcal{S}(\alpha)_{\infty} = S_{(X)} + S_{(Y)} + S_{(Z)} + S_{(T)}\). The base locus of the pencil \(\mathcal{S}(\alpha)\) consists of the curves
\begin{gather*}
  C_{1} = C_{(X, Z)}, \;
  C_{2} = C_{(Y, Z)}, \;
  C_{3} = C_{(Y, T)}, \;
  C_{4} = C_{(Z, T)}, \;
  C_{5} = C_{(X, Y + a_3 T)}, \;
  C_{6} = C_{(X, Y + a_4 T)}, \\
  C_{7} = C_{(X, a_1 a_2 a_5 Y + T)}, \;
  C_{8} = C_{(Y, X + a_1 a_3 a_4 T)}, \;
  C_{9} = C_{(T, (X + a_1 a_2 Y) (X + a_1 a_5 Y) + a_6 X Z)}.
\end{gather*}
Their linear equivalence classes on the generic member \(\mathcal{S}(\alpha)_{\Bbbk}\) of the pencil satisfy the following relations:
\begin{gather*}
  \begin{pmatrix}
    [C_{2}] \\ [C_{7}] \\ [C_{8}] \\ [C_{9}]
  \end{pmatrix} = 
  \begin{pmatrix}
    -1 & 0 & -2 & 0 & 0 & 1 \\
    -1 & 0 & 0 & -1 & -1 & 1 \\
    1 & -2 & 2 & 0 & 0 & 0 \\
    0 & -1 & -1 & 0 & 0 & 1
  \end{pmatrix} \cdot
  \begin{pmatrix}
    [C_{1}] & [C_{3}] & [C_{4}] & [C_{5}] & [C_{6}] & [H_{\mathcal{S}(\alpha)}]
  \end{pmatrix}^T.
\end{gather*}

For a general choice of \(\lambda \in \mathbb{C}\) and \(\alpha \in (\mathbb{C}^*)^6\) the surface \(\mathcal{S}(\alpha)_{\lambda}\) has the following singularities:
\begin{itemize}\setlength{\itemindent}{2cm}
\item[\(P_{1} = P_{(X, Y, T)}\):] type \(\mathbb{A}_2\) with the quadratic term \(X \cdot Y\);
\item[\(P_{2} = P_{(Y, Z, T)}\):] type \(\mathbb{A}_3\) with the quadratic term \(Y \cdot Z\);
\item[\(P_{3} = P_{(Y, T, X + a_6 Z)}\):] type \(\mathbb{A}_1\) with the quadratic term \(Y (X + a_1 a_2 Y + a_6 Z - \lambda T) + (a_1 a_5 Y^2 + T^2)\);
\item[\(P_{4} = P_{(Z, T, X + a_1 a_2 Y)}\):] type \(\mathbb{A}_1\) with the quadratic term
  \[
    a_2 a_6 Z^2 - (a_1 a_2 a_5 (a_3 + a_4) + a_2 \lambda + 1) Z T + a_2 T^2 + (a_2 - a_5) Z (X + a_1 a_2 Y);
  \]
\item[\(P_{5} = P_{(Z, T, X + a_1 a_5 Y)}\):] type \(\mathbb{A}_1\) with the quadratic term
  \[
    a_5 a_6 Z^2 - (a_1 a_2 a_5 (a_3 + a_4) + a_5 \lambda + 1) Z T + a_5 T^2 - (a_2 - a_5) Z (X + a_1 a_5 Y).
  \]
\end{itemize}

Galois action on the lattice \(L(\alpha)_{\lambda}\) is trivial. The intersection matrix on \(L(\alpha)_{\lambda} = L(\alpha)_{\mathcal{S}}\) is represented by
\begin{table}[H]
  \begin{tabular}{|c||cc|ccc|c|c|c|cccccc|}
    \hline
    \(\bullet\) & \(E_1^1\) & \(E_1^2\) & \(E_2^1\) & \(E_2^2\) & \(E_2^3\) & \(E_3^1\) & \(E_4^1\) & \(E_5^1\) & \(\widetilde{C_{1}}\) & \(\widetilde{C_{3}}\) & \(\widetilde{C_{4}}\) & \(\widetilde{C_{5}}\) & \(\widetilde{C_{6}}\) & \(\widetilde{H_{\mathcal{S}}}\) \\
    \hline
    \hline
    \(\widetilde{C_{1}}\) & \(0\) & \(0\) & \(0\) & \(0\) & \(0\) & \(0\) & \(0\) & \(0\) & \(-2\) & \(0\) & \(1\) & \(1\) & \(1\) & \(1\) \\
    \(\widetilde{C_{3}}\) & \(1\) & \(0\) & \(1\) & \(0\) & \(0\) & \(1\) & \(0\) & \(0\) & \(0\) & \(-2\) & \(0\) & \(0\) & \(0\) & \(1\) \\
    \(\widetilde{C_{4}}\) & \(0\) & \(0\) & \(0\) & \(0\) & \(1\) & \(0\) & \(1\) & \(1\) & \(1\) & \(0\) & \(-2\) & \(0\) & \(0\) & \(1\) \\
    \(\widetilde{C_{5}}\) & \(0\) & \(1\) & \(0\) & \(0\) & \(0\) & \(0\) & \(0\) & \(0\) & \(1\) & \(0\) & \(0\) & \(-2\) & \(0\) & \(1\) \\
    \(\widetilde{C_{6}}\) & \(0\) & \(1\) & \(0\) & \(0\) & \(0\) & \(0\) & \(0\) & \(0\) & \(1\) & \(0\) & \(0\) & \(0\) & \(-2\) & \(1\) \\
    \(\widetilde{H_{\mathcal{S}}}\) & \(0\) & \(0\) & \(0\) & \(0\) & \(0\) & \(0\) & \(0\) & \(0\) & \(1\) & \(1\) & \(1\) & \(1\) & \(1\) & \(4\) \\
    \hline
  \end{tabular}
\end{table}
Note that the intersection matrix is non-degenerate.

Discriminant groups and discriminant forms of the lattices \(L(\alpha)_{\mathcal{S}}\) and \(H \oplus \Pic(X)\) are given by
\begin{gather*}
  G' =
  \begin{pmatrix}
    0 & 0 & \frac{1}{2} & 0 & \frac{1}{2} & \frac{1}{2} & \frac{1}{2} & 0 & 0 & 0 & 0 & 0 & 0 & 0 \\
    0 & 0 & \frac{1}{2} & 0 & \frac{1}{2} & \frac{1}{2} & 0 & \frac{1}{2} & 0 & 0 & 0 & 0 & 0 & 0 \\
    0 & 0 & 0 & 0 & 0 & 0 & 0 & 0 & 0 & 0 & 0 & \frac{1}{2} & \frac{1}{2} & 0 \\
    \frac{1}{2} & 0 & 0 & 0 & 0 & 0 & 0 & 0 & \frac{1}{2} & 0 & 0 & \frac{1}{2} & 0 & \frac{1}{2} \\
    \frac{4}{5} & \frac{2}{5} & \frac{4}{5} & \frac{2}{5} & 0 & \frac{3}{5} & \frac{4}{5} & \frac{4}{5} & \frac{2}{5} & \frac{1}{5} & \frac{3}{5} & 0 & 0 & \frac{1}{5}
  \end{pmatrix}, \;
  G'' =
  \begin{pmatrix}
    0 & 0 & 0 & \frac{1}{2} & 0 & 0 & 0 & \frac{1}{2} \\
    0 & 0 & \frac{1}{2} & 0 & 0 & 0 & 0 & \frac{1}{2} \\
    0 & 0 & \frac{1}{2} & \frac{1}{2} & \frac{1}{2} & 0 & 0 & \frac{1}{2} \\
    0 & 0 & \frac{1}{2} & \frac{1}{2} & 0 & \frac{1}{2} & 0 & \frac{1}{2} \\
    0 & 0 & \frac{2}{5} & \frac{2}{5} & \frac{2}{5} & \frac{2}{5} & \frac{1}{5} & \frac{1}{5}
  \end{pmatrix};
\end{gather*}
\begin{gather*}
  B' = 
  \begin{pmatrix}
    0 & \frac{1}{2} & 0 & 0 & 0 \\
    \frac{1}{2} & 0 & 0 & 0 & 0 \\
    0 & 0 & 0 & \frac{1}{2} & 0 \\
    0 & 0 & \frac{1}{2} & 0 & 0 \\
    0 & 0 & 0 & 0 & \frac{3}{5}
  \end{pmatrix}, \;
  B'' = 
  \begin{pmatrix}
    0 & \frac{1}{2} & 0 & 0 & 0 \\
    \frac{1}{2} & 0 & 0 & 0 & 0 \\
    0 & 0 & 0 & \frac{1}{2} & 0 \\
    0 & 0 & \frac{1}{2} & 0 & 0 \\
    0 & 0 & 0 & 0 & \frac{2}{5}
  \end{pmatrix}; \;
  \begin{pmatrix}
    Q' \\ Q''
  \end{pmatrix}
  =
  \begin{pmatrix}
    0 & 0 & 1 & 1 & \frac{8}{5} \\
    0 & 0 & 1 & 1 & \frac{2}{5}
  \end{pmatrix}.
\end{gather*}


\subsection{Family \textnumero7.1}\label{subsection:07-01_parametrised}

The pencil \(\mathcal{S}(\alpha)\), where \(\alpha = (a_1, \ldots, a_7) \in (\mathbb{C}^*)^7\), is defined by the equation
\begin{gather*}
  \left(a_1^{2} a_3 a_6\right) X^{3} Z + \left(a_1^{2} a_2 a_3 a_6 + a_1^{2} a_3 a_5 a_6 + a_1\right) X^{2} Y Z + \left(a_1^{2} a_2 a_3 a_5 a_6 + a_1 a_2 + a_1 a_5\right) X Y^{2} Z + \left(a_1 a_2 a_5\right) Y^{3} Z + \\ X Y Z^{2} + \left(a_1 a_3 + a_1 a_6\right) X^{2} Z T + \left(a_1 a_2 a_4 a_5 + 1\right) Y^{2} Z T + a_7 X Y T^{2} + X Z T^{2} + a_4 Y Z T^{2} = \lambda X Y Z T.
\end{gather*}
Note that \(\mathcal{S}(\alpha)_{\infty} = S_{(X)} + S_{(Y)} + S_{(Z)} + S_{(T)}\). The base locus of the pencil \(\mathcal{S}(\alpha)\) consists of the curves
\begin{gather*}
  C_{1} = C_{(X, Y)}, \;
  C_{2} = C_{(X, Z)}, \;
  C_{3} = C_{(Y, Z)}, \;
  C_{4} = C_{(Z, T)}, \;
  C_{5} = C_{(X, Y + a_4 T)}, \;
  C_{6} = C_{(X, a_1 a_2 a_5 Y + T)}, \\
  C_{7} = C_{(Y, a_1 a_3 X + T)}, \;
  C_{8} = C_{(Y, a_1 a_6 X + T)}, \;
  C_{9} = C_{(T, a_1 (X + a_2 Y) (X + a_5 Y) (a_1 a_3 a_6 X + Y) + X Y Z)}.
\end{gather*}
Their linear equivalence classes on the generic member \(\mathcal{S}(\alpha)_{\Bbbk}\) of the pencil satisfy the following relations:
\begin{gather*}
  \begin{pmatrix}
    [C_{2}] \\ [C_{6}] \\ [C_{8}] \\ [C_{9}]
  \end{pmatrix} = 
  \begin{pmatrix}
    0 & -1 & -2 & 0 & 0 & 1 \\
    -1 & 1 & 2 & -1 & 0 & 0 \\
    -1 & -1 & 0 & 0 & -1 & 1 \\
    0 & 0 & -1 & 0 & 0 & 1
  \end{pmatrix} \cdot
  \begin{pmatrix}
    [C_{1}] & [C_{3}] & [C_{4}] & [C_{5}] & [C_{7}] & [H_{\mathcal{S}}]
  \end{pmatrix}^T.
\end{gather*}

For a general choice of \(\lambda \in \mathbb{C}\) and \(\alpha \in (\mathbb{C}^*)^7\) the surface \(\mathcal{S}(\alpha)_{\lambda}\) has the following singularities:
\begin{itemize}\setlength{\itemindent}{2cm}
\item[\(P_{1} = P_{(X, Y, Z)}\):] type \(\mathbb{A}_1\) with the quadratic term \(a_7 X Y + a_4 Y Z + X Z\);
\item[\(P_{2} = P_{(X, Y, T)}\):] type \(\mathbb{A}_3\) with the quadratic term \(X \cdot Y\);
\item[\(P_{3} = P_{(Z, T, X + a_2 Y)}\):] type \(\mathbb{A}_1\) with the quadratic term
  \[
    a_2 (Z^2 + a_7 T^2) - (a_1 a_2 (a_2 (a_3 + a_6) + a_4 a_5) + a_2 \lambda + 1) Z T - a_1 (a_2 - a_5) (a_1 a_2 a_3 a_6 - 1) Z (X + a_2 Y);
  \]
\item[\(P_{4} = P_{(Z, T, X + a_5 Y)}\):] type \(\mathbb{A}_1\) with the quadratic term
  \[
    a_5 (Z^2 + a_7 T^2) - (a_1 a_5 (a_5 (a_3 + a_6) + a_2 a_4) + a_5 \lambda + 1) Z T + a_1 (a_2 - a_5) (a_1 a_3 a_5 a_6 - 1) Z (X + a_5 Y); 
  \]
\item[\(P_{5} = P_{(Z, T, a_1 a_3 a_6 X + Y)}\):] type \(\mathbb{A}_1\) with the quadratic term
  \[
    a_3 a_6 (Z^2 + a_7 T^2) - (a_1 (a_3 a_6)^2 (a_1 a_2 a_4 a_5 + 1) + a_3 a_6 \lambda + a_3 + a_6) Z T - (a_1 a_2 a_3 a_6 - 1) (a_1 a_3 a_5 a_6 - 1) Z (a_1 a_3 a_6 X + Y).
  \]
\end{itemize}

Galois action on the lattice \(L(\alpha)_{\lambda}\) is trivial. The intersection matrix on \(L(\alpha)_{\lambda} = L(\alpha)_{\mathcal{S}}\) is represented by
\begin{table}[H]
  \begin{tabular}{|c||c|ccc|c|c|c|cccccc|}
    \hline
    \(\bullet\) & \(E_1^1\) & \(E_2^1\) & \(E_2^2\) & \(E_2^3\) & \(E_3^1\) & \(E_4^1\) & \(E_5^1\) & \(\widetilde{C_{1}}\) & \(\widetilde{C_{3}}\) & \(\widetilde{C_{4}}\) & \(\widetilde{C_{5}}\) & \(\widetilde{C_{7}}\) & \(\widetilde{H_{\mathcal{S}}}\) \\
    \hline
    \hline
    \(\widetilde{C_{1}}\) & \(1\) & \(0\) & \(1\) & \(0\) & \(0\) & \(0\) & \(0\) & \(-2\) & \(0\) & \(0\) & \(0\) & \(0\) & \(1\) \\
    \(\widetilde{C_{3}}\) & \(1\) & \(0\) & \(0\) & \(0\) & \(0\) & \(0\) & \(0\) & \(0\) & \(-2\) & \(1\) & \(0\) & \(1\) & \(1\) \\
    \(\widetilde{C_{4}}\) & \(0\) & \(0\) & \(0\) & \(0\) & \(1\) & \(1\) & \(1\) & \(0\) & \(1\) & \(-2\) & \(0\) & \(0\) & \(1\) \\
    \(\widetilde{C_{5}}\) & \(0\) & \(1\) & \(0\) & \(0\) & \(0\) & \(0\) & \(0\) & \(0\) & \(0\) & \(0\) & \(-2\) & \(0\) & \(1\) \\
    \(\widetilde{C_{7}}\) & \(0\) & \(0\) & \(0\) & \(1\) & \(0\) & \(0\) & \(0\) & \(0\) & \(1\) & \(0\) & \(0\) & \(-2\) & \(1\) \\
    \(\widetilde{H_{\mathcal{S}}}\) & \(0\) & \(0\) & \(0\) & \(0\) & \(0\) & \(0\) & \(0\) & \(1\) & \(1\) & \(1\) & \(1\) & \(1\) & \(4\) \\
    \hline
  \end{tabular}.
\end{table}
Note that the intersection matrix is non-degenerate.

Discriminant groups and discriminant forms of the lattices \(L(\alpha)_{\mathcal{S}}\) and \(H \oplus \Pic(X)\) are given by
\begin{gather*}
  G' =
  \begin{pmatrix}
    \frac{1}{2} & \frac{1}{2} & 0 & 0 & \frac{1}{2} & 0 & \frac{1}{2} & \frac{1}{2} & \frac{1}{2} & 0 & 0 & 0 & \frac{1}{2} \\
    \frac{1}{2} & 0 & \frac{1}{2} & 0 & 0 & 0 & 0 & 0 & 0 & 0 & \frac{1}{2} & \frac{1}{2} & 0 \\
    \frac{1}{2} & 0 & \frac{1}{2} & 0 & \frac{1}{2} & \frac{1}{2} & 0 & 0 & 0 & 0 & \frac{1}{2} & \frac{1}{2} & 0 \\
    0 & 0 & 0 & 0 & \frac{1}{2} & 0 & \frac{1}{2} & 0 & 0 & 0 & 0 & 0 & 0 \\
    \frac{1}{8} & \frac{3}{8} & \frac{1}{4} & \frac{5}{8} & \frac{3}{8} & \frac{3}{8} & \frac{3}{8} & \frac{1}{2} & \frac{3}{4} & \frac{3}{4} & \frac{1}{2} & 0 & \frac{5}{8}
  \end{pmatrix}, \;
  G'' =
  \begin{pmatrix}
    0 & 0 & \frac{1}{2} & \frac{1}{2} & \frac{1}{2} & 0 & \frac{1}{2} & 0 & 0 \\
    0 & 0 & \frac{1}{2} & \frac{1}{2} & 0 & \frac{1}{2} & \frac{1}{2} & 0 & 0 \\
    0 & 0 & \frac{1}{2} & 0 & 0 & 0 & \frac{1}{2} & 0 & 0 \\
    0 & 0 & \frac{1}{2} & \frac{1}{2} & 0 & 0 & 0 & 0 & 0 \\
    0 & 0 & \frac{1}{8} & \frac{1}{8} & \frac{1}{8} & \frac{1}{8} & \frac{1}{8} & \frac{3}{4} & \frac{7}{8}
  \end{pmatrix};
\end{gather*}
\begin{gather*}
  B' = 
  \begin{pmatrix}
    0 & \frac{1}{2} & 0 & 0 & 0 \\
    \frac{1}{2} & 0 & 0 & 0 & 0 \\
    0 & 0 & 0 & \frac{1}{2} & 0 \\
    0 & 0 & \frac{1}{2} & 0 & 0 \\
    0 & 0 & 0 & 0 & \frac{5}{8}
  \end{pmatrix}, \;
  B'' = 
  \begin{pmatrix}
    0 & \frac{1}{2} & 0 & 0 & 0 \\
    \frac{1}{2} & 0 & 0 & 0 & 0 \\
    0 & 0 & 0 & \frac{1}{2} & 0 \\
    0 & 0 & \frac{1}{2} & 0 & 0 \\
    0 & 0 & 0 & 0 & \frac{3}{8}
  \end{pmatrix}; \;
  \begin{pmatrix}
    Q' \\ Q''
  \end{pmatrix}
  =
  \begin{pmatrix}
    0 & 0 & 1 & 1 & \frac{13}{8} \\
    0 & 0 & 1 & 1 & \frac{3}{8}
  \end{pmatrix}.
\end{gather*}


\subsection{Family \textnumero8.1}\label{subsection:08-01_parametrised}

The pencil \(\mathcal{S}(\alpha)\), where \(\alpha = (a_1, \ldots, a_8) \in (\mathbb{C}^*)^8\), is defined by the equation
\begin{gather*}
  \lambda X Y Z T = (a_4 a_7) X Y^{3} + X^{2} Y Z + (a_1 a_2 a_4 a_5 a_7 + a_4 + a_7) X Y^{2} Z + (a_1 a_2 a_4 a_5 + a_1 a_2 a_5 a_7 + 1) X Y Z^{2} + \\ (a_1 a_2 a_5) X Z^{3} + (a_1 a_3 a_4 a_7 + a_1 a_4 a_6 a_7 + 1) X Y^{2} T + (a_1^{2} a_2 a_3 a_5 a_6 + a_1 a_2 + a_1 a_5) X Z^{2} T + \\ (a_1^{2} a_3 a_4 a_6 a_7 + a_1 a_3 + a_1 a_6) X Y T^{2} + (a_1^{2} a_2 a_3 a_6 + a_1^{2} a_3 a_5 a_6 + a_1) X Z T^{2} + a_8 Y Z T^{2} + (a_1^{2} a_3 a_6) X T^{3}.
\end{gather*}
Note that \(\mathcal{S}(\alpha)_{\infty} = S_{(X)} + S_{(Y)} + S_{(Z)} + S_{(T)}\). The base locus of the pencil \(\mathcal{S}(\alpha)\) consists of the curves
\begin{gather*}
  C_1 = C_{(X, Y)}, \;
  C_2 = C_{(X, Z)}, \;
  C_3 = C_{(X, T)}, \;
  C_4 = C_{(Y, a_2 Z + T)}, \;
  C_5 = C_{(Y, a_5 Z + T)}, \;
  C_6 = C_{(Y, Z + a_1 a_3 a_6 T)}, \\
  C_7 = C_{(Z, Y + a_1 a_3 T)}, \;
  C_8 = C_{(Z, Y + a_1 a_6 T)}, \;
  C_9 = C_{(Z, a_4 a_7 Y + T)}, \;
  C_{10} = C_{(T, (a_4 Y + Z) (a_7 Y + Z) (Y + a_1 a_2 a_5 Z) + X Y Z)}.
\end{gather*}
Their linear equivalence classes on the generic member \(\mathcal{S}(\alpha)_{\Bbbk}\) of the pencil satisfy the following relations:
\begin{gather*}
  \begin{pmatrix}
    [C_{2}] \\ [C_{6}] \\ [C_{9}] \\ [C_{10}]
  \end{pmatrix} = 
  \begin{pmatrix}
    -1 & -2 & 0 & 0 & 0 & 0 & 1 \\
    -1 & 0 & -1 & -1 & 0 & 0 & 1 \\
    1 & 2 & 0 & 0 & -1 & -1 & 0 \\
    0 & -1 & 0 & 0 & 0 & 0 & 1
  \end{pmatrix} \cdot
  \begin{pmatrix}
    [C_{1}] & [C_{3}] & [C_{4}] & [C_{5}] & [C_{7}] & [C_{8}] & [H_{\mathcal{S}(\alpha)}]
  \end{pmatrix}^T.
\end{gather*}

For a general choice of \(\lambda \in \mathbb{C}\) and \(\alpha \in (\mathbb{C}^*)^8\) the surface \(\mathcal{S}(\alpha)_{\lambda}\) has the following singularities:
\begin{itemize}\setlength{\itemindent}{2cm}
\item[\(P_{1} = P_{(Y, Z, T)}\):] type \(\mathbb{A}_2\) with the quadratic term \(Y \cdot Z\);
\item[\(P_{2} = P_{(X, T, a_4 Y + Z)}\):] type \(\mathbb{A}_1\) with the quadratic term
  \begin{gather*}
    (a_4 - a_7) (a_1 a_2 a_4 a_5 - 1) X (a_4 Y + Z) - a_4 (X^2 + a_8 T^2) + \\
    (a_1 a_4 (a_1 a_2 a_3 a_4 a_5 a_6 + a_4 (a_2 + a_5) + a_7 (a_3 + a_6)) + a_4 \lambda + 1) X T;
  \end{gather*}
\item[\(P_{3} = P_{(X, T, a_7 Y + Z)}\):] type \(\mathbb{A}_1\) with the quadratic term
  \begin{gather*}
    (a_4 - a_7) (a_1 a_2 a_5 a_7 - 1) X (a_7 Y + Z) + a_7 (X^2 + a_8 T^2) - \\
    (a_1 a_7 (a_1 a_2 a_3 a_5 a_6 a_7 + a_4 (a_3 + a_6) + a_7 (a_2 + a_5)) + a_7 \lambda + 1) X T;
  \end{gather*}
\item[\(P_{4} = P_{(X, T, Y + a_1 a_2 a_5 Z)}\):] type \(\mathbb{A}_1\) with the quadratic term
  \begin{gather*}
    (a_1 a_2 a_4 a_5 - 1) (a_1 a_2 a_5 a_7 - 1) X (Y + a_1 a_2 a_5 Z) - a_1 a_2 a_5 (X^2 + a_8 T^2) + \\
    a_1 (a_1 a_2 a_5 (a_1 a_2 a_4 a_5 a_7 (a_3 + a_6) + a_2 a_5 + a_3 a_6) + a_2 a_5 \lambda + a_2 + a_5) X T.
  \end{gather*}
\end{itemize}

Galois action on the lattice \(L(\alpha)_{\lambda}\) is trivial. The intersection matrix on \(L(\alpha)_{\lambda} = L(\alpha)_{\mathcal{S}}\) is represented by
\begin{table}[H]
  \begin{tabular}{|c||cc|c|cc|ccccccc|}
    \hline
    \(\bullet\) & \(E_1^1\) & \(E_1^2\) & \(E_2^1\) & \(E_3^1\) & \(E_4^1\) & \(\widetilde{C_{1}}\) & \(\widetilde{C_{3}}\) & \(\widetilde{C_{4}}\) & \(\widetilde{C_{5}}\) & \(\widetilde{C_{7}}\) & \(\widetilde{C_{8}}\) & \(\widetilde{H_{\mathcal{S}}}\) \\
    \hline
    \hline
    \(\widetilde{C_{1}}\) & \(0\) & \(0\) & \(0\) & \(0\) & \(0\) & \(-2\) & \(1\) & \(1\) & \(1\) & \(0\) & \(0\) & \(1\) \\
    \(\widetilde{C_{3}}\) & \(0\) & \(0\) & \(1\) & \(1\) & \(1\) & \(1\) & \(-2\) & \(0\) & \(0\) & \(0\) & \(0\) & \(1\) \\
    \(\widetilde{C_{4}}\) & \(1\) & \(0\) & \(0\) & \(0\) & \(0\) & \(1\) & \(0\) & \(-2\) & \(0\) & \(0\) & \(0\) & \(1\) \\
    \(\widetilde{C_{5}}\) & \(1\) & \(0\) & \(0\) & \(0\) & \(0\) & \(1\) & \(0\) & \(0\) & \(-2\) & \(0\) & \(0\) & \(1\) \\
    \(\widetilde{C_{7}}\) & \(0\) & \(1\) & \(0\) & \(0\) & \(0\) & \(0\) & \(0\) & \(0\) & \(0\) & \(-2\) & \(0\) & \(1\) \\
    \(\widetilde{C_{8}}\) & \(0\) & \(1\) & \(0\) & \(0\) & \(0\) & \(0\) & \(0\) & \(0\) & \(0\) & \(0\) & \(-2\) & \(1\) \\
    \(\widetilde{H_{\mathcal{S}}}\) & \(0\) & \(0\) & \(0\) & \(0\) & \(0\) & \(1\) & \(1\) & \(1\) & \(1\) & \(1\) & \(1\) & \(4\) \\
    \hline
  \end{tabular}
\end{table}
Note that the intersection matrix is non-degenerate.

Discriminant groups and discriminant forms of the lattices \(L(\alpha)_{\mathcal{S}}\) and \(H \oplus \Pic(X)\) are given by
\begin{gather*}
  G' =
  \begin{pmatrix}
    0 & 0 & \frac{1}{2} & \frac{1}{2} & 0 & 0 & 0 & 0 & 0 & \frac{1}{2} & \frac{1}{2} & 0 \\
    0 & 0 & \frac{1}{2} & 0 & \frac{1}{2} & 0 & 0 & 0 & 0 & \frac{1}{2} & \frac{1}{2} & 0 \\
    \frac{1}{2} & 0 & \frac{1}{2} & 0 & 0 & \frac{1}{2} & 0 & 0 & 0 & 0 & \frac{1}{2} & 0 \\
    \frac{1}{2} & 0 & \frac{1}{2} & 0 & 0 & \frac{1}{2} & 0 & 0 & 0 & \frac{1}{2} & 0 & 0 \\
    0 & 0 & 0 & 0 & 0 & 0 & 0 & \frac{1}{2} & \frac{1}{2} & 0 & 0 & 0 \\
    0 & \frac{1}{2} & 0 & 0 & 0 & \frac{1}{2} & 0 & \frac{1}{2} & 0 & 0 & 0 & \frac{1}{2} \\
    \frac{1}{3} & \frac{2}{3} & \frac{2}{3} & \frac{2}{3} & \frac{2}{3} & \frac{1}{3} & \frac{1}{3} & 0 & 0 & 0 & 0 & \frac{1}{3}
  \end{pmatrix}, \;
  G'' =
  \begin{pmatrix}
    0 & 0 & \frac{1}{2} & \frac{1}{2} & \frac{1}{2} & 0 & \frac{1}{2} & \frac{1}{2} & 0 & \frac{1}{2} \\
    0 & 0 & \frac{1}{2} & \frac{1}{2} & 0 & \frac{1}{2} & \frac{1}{2} & \frac{1}{2} & 0 & \frac{1}{2} \\
    0 & 0 & 0 & \frac{1}{2} & 0 & 0 & 0 & 0 & 0 & \frac{1}{2} \\
    0 & 0 & \frac{1}{2} & 0 & 0 & 0 & 0 & 0 & 0 & \frac{1}{2} \\
    0 & 0 & \frac{1}{2} & \frac{1}{2} & 0 & 0 & \frac{1}{2} & 0 & 0 & \frac{1}{2} \\
    0 & 0 & \frac{1}{2} & \frac{1}{2} & 0 & 0 & 0 & \frac{1}{2} & 0 & \frac{1}{2} \\
    0 & 0 & \frac{2}{3} & \frac{2}{3} & \frac{2}{3} & \frac{2}{3} & \frac{2}{3} & \frac{2}{3} & \frac{2}{3} & 0
  \end{pmatrix};
\end{gather*}
\begin{gather*}
  B' = 
  \begin{pmatrix}
    0 & \frac{1}{2} & 0 & 0 & 0 & 0 & 0 \\
    \frac{1}{2} & 0 & 0 & 0 & 0 & 0 & 0 \\
    0 & 0 & 0 & \frac{1}{2} & 0 & 0 & 0 \\
    0 & 0 & \frac{1}{2} & 0 & 0 & 0 & 0 \\
    0 & 0 & 0 & 0 & 0 & \frac{1}{2} & 0 \\
    0 & 0 & 0 & 0 & \frac{1}{2} & 0 & 0 \\
    0 & 0 & 0 & 0 & 0 & 0 & \frac{2}{3}
  \end{pmatrix}, \;
  B'' =
  \begin{pmatrix}
    0 & \frac{1}{2} & 0 & 0 & 0 & 0 & 0 \\
    \frac{1}{2} & 0 & 0 & 0 & 0 & 0 & 0 \\
    0 & 0 & 0 & \frac{1}{2} & 0 & 0 & 0 \\
    0 & 0 & \frac{1}{2} & 0 & 0 & 0 & 0 \\
    0 & 0 & 0 & 0 & 0 & \frac{1}{2} & 0 \\
    0 & 0 & 0 & 0 & \frac{1}{2} & 0 & 0 \\
    0 & 0 & 0 & 0 & 0 & 0 & \frac{1}{3}
  \end{pmatrix}; \;
  \begin{pmatrix}
    Q' \\ Q''
  \end{pmatrix}
  =
  \begin{pmatrix}
    0 & 0 & 0 & 0 & 1 & 1 & \frac{2}{3} \\
    0 & 0 & 0 & 0 & 1 & 1 & \frac{4}{3}
  \end{pmatrix}.
\end{gather*}


\printbibliography

@article{akhtar2012minkowski,
  title={Minkowski polynomials and mutations},
  author={Mohammad Akhtar and Tom Coates and Sergey Galkin and Alexander Kasprzyk},
  journal={SIGMA. Symmetry, Integrability and Geometry: Methods and Applications},
  volume={8},
  number={094},
  pages={1--707},
  year={2012},
  publisher={SIGMA. Symmetry, Integrability and Geometry: Methods and Applications}}

@article{akhtar2016mirror,
  title={Mirror symmetry and the classification of orbifold del {Pezzo} surfaces},
  author={Mohammad Akhtar and Tom Coates and Alessio Corti and Liana Heuberger and Alexander Kasprzyk and Alessandro Oneto and Andrea Petracci and Thomas Prince and Ketil Tveiten},
  journal={Proceedings of the American Mathematical Society},
  volume={144},
  number={2},
  pages={513--527},
  year={2016}}

@article{alexeev/k3,
  author={Valery Alexeev and Philip Engel},
  title={Compact moduli of K3 surfaces},
  journal={Annals of Mathematics},
  volume={198},
  number={2},
  year={2023},
  pages={727--789}}

@article{auroux2007mirror,
  title={Mirror symmetry and {T}-duality in the complement of an anticanonical divisor},
  author={Denis Auroux},
  journal={Journal of G{\"o}kova Geometry Topology},
  volume={1},
  pages={51--91},
  year={2007}}

@article{batyrev2000mirror,
  title={Mirror symmetry and toric degenerations of partial flag manifolds},
  author={Victor Batyrev and Bumsig Kim and Ionu{\c{t}} Ciocan-Fontanine and Duco van Straten},
  journal={Acta Mathematica},
  volume={184},
  number={1},
  pages={1--39},
  year={2000},
  publisher={Springer Netherlands}}

@InProceedings{beauvillefano,
  title={Fano threefolds and {K3} surfaces},
  author={Arnaud Beauville},
  booktitle={The Fano Conference proceedings},
  year={2004},
  publisher={Dipartimento di Matematica dell'Universit{\`a} di Torino}}

@article{bruzzo2012picard,
  title={Picard group of hypersurfaces in toric 3-folds},
  author={Ugo Bruzzo and Antonella Grassi},
  journal={International Journal of Mathematics},
  volume={23},
  number={2},
  pages={1--14},
  year={2012},
  publisher={World Scientific}}

@article{batyrev2010reflexive,
  author={Victor Batyrev and Dorothee Juny},
  title={Classification of {Gorenstein} toric del {Pezzo} varieties in arbitrary dimension},
  journal={Moscow Mathematical Journal},
  volume={10},
  number={2},
  pages={285--316},
  year={2010}}

@article{clingher2007modular,
  title={Modular invariants for lattice polarized {K3} surfaces},
  author={Adrian Clingher and Charles Doran},
  journal={Michigan Mathematical Journal},
  volume={55},
  number={2},
  pages={355--393},
  year={2007},
  publisher={University of Michigan, Department of Mathematics}}

@article{clingher2012lattice,
  title={Lattice polarized {K3} surfaces and {Siegel} modular forms},
  author={Adrian Clingher and Charles Doran},
  journal={Advances in Mathematics},
  volume={231},
  number={1},
  pages={172--212},
  year={2012},
  publisher={Elsevier}}

@article{coates2019inversion,
  title={Laurent inversion},
  author={Tom Coates and Alexander Kasprzyk and Thomas Prince},
  journal={Pure and Applied Mathematics Quarterly},
  volume={15},
  number={4},
  pages={1135--1179},
  year={2019}}

@book{cox2011toric,
  title={Toric varieties},
  author={David Cox and John Little and Hal Schenck},
  series={Graduate Studies in Mathematics},
  number={124},
  year={2011},
  publisher={American Mathematical Society}}

@inproceedings{coates2014mirror,
  title={Mirror symmetry and {Fano} manifolds},
  author={Tom Coates and Alessio Corti and Sergey Galkin and Vasily Golyshev and Alexander Kasprzyk},
  booktitle={European Congress of Mathematics Krak{\'o}w, 2--7 July, 2012},
  pages={285--300},
  year={2014}}

@article{coates2016quantum,
  title={Quantum periods for 3-dimensional {Fano} manifolds},
  author={Tom Coates and Alessio Corti and Sergey Galkin and Alexander Kasprzyk},
  journal={Geometry \& Topology},
  volume={20},
  number={1},
  pages={103--256},
  year={2016},
  publisher={Mathematical Sciences Publishers}}

@article{coates2021maximally,
  title={Maximally mutable {Laurent} polynomials},
  author={Tom Coates and Alexander Kasprzyk and Giuseppe Pitton and Ketil Tveiten},
  journal={Proceedings of the Royal Society A},
  volume={477},
  number={2254},
  eid={20210584},
  pages={1--21},
  year={2021},
  publisher={The Royal Society}}

@article{corti2013asymptotically,
  title={Asymptotically cylindrical {Calabi--Yau} 3-folds from weak {Fano} 3-folds},
  author={Alessio Corti and Mark Haskins and Johannes Nordstr{\"o}m and Tommaso Pacini},
  journal={Geometry \& Topology},
  volume={17},
  number={4},
  pages={1955--2059},
  year={2013},
  publisher={Mathematical Sciences Publishers}}

@article{corti2015g_2,
  title={\({G}_2\)-manifolds and associative submanifolds via semi-{F}ano 3-folds},
  author={Alessio Corti and Mark Haskins and Johannes Nordstrom and Tommaso Pacini},
  journal={Duke Mathematical Journal},
  volume={164},
  number={10},
  pages={1971--2092},
  year={2015}}

@article{cheltsov2018katzarkov,
  title={{Katzarkov--Kontsevich--Pantev} conjecture for {Fano} threefolds},
  author={Ivan Cheltsov and Victor Przyjalkowski},
  journal={Proceedings of the Steklov Institute of Mathematics},
  volume={328},
  year={2025}}

@book{dolgachev2012classical,
  author={Igor Dolgachev},
  title={Classical Algebraic Geometry: A Modern View},
  publisher={Cambridge University Press},
  year={2012}}

@inproceedings{doran2017mirror,
  title={Mirror symmetry, {Tyurin} degenerations and fibrations on {Calabi--Yau} manifolds},
  author={Charles Doran and Andrew Harder and Alan Thompson},
  booktitle={String-Math 2015},
  pages={93--131},
  series={Proceedings of Symposia in Pure Mathematics},
  number={96},
  year={2017}}

@article{danilov1987newton,
  title={Newton polyhedra and an algorithm for computing {Hodge--Deligne} numbers},
  author={Vladimir Danilov and Askold Khovanskii},
  journal={Mathematics of the USSR-Izvestiya},
  volume={29},
  number={2},
  pages={279},
  year={1987},
  publisher={IOP Publishing}}

@article{duistermaat1998constant,
  title={Constant terms in powers of a {Laurent} polynomial},
  author={Johannes Jisse Duistermaat and Wilberd Van Der Kallen},
  journal={Indagationes mathematicae},
  volume={9},
  number={2},
  pages={221--231},
  year={1998},
  publisher={Elsevier}}

@article{diemer2013compactifications,
  title={Compactifications of spaces of {Landau--Ginzburg} models},
  author={Colin Diemer and Ludmil Katzarkov and Gabriel Kerr},
  journal={Izvestiya: Mathematics},
  volume={77},
  number={3},
  pages={487},
  year={2013},
  publisher={IOP Publishing}}

@article{diemer2016symplectomorphism,
  title={Symplectomorphism group relations and degenerations of {Landau--Ginzburg} models},
  author={Colin Diemer and Ludmil Katzarkov and Gabriel Kerr},
  journal={Journal of the European Mathematical Society},
  volume={18},
  number={10},
  year={2016},
  pages={2167--2271}}

@article{dolgachev1982lattices,
  author={Igor Dolgachev},
  title={Integral quadratic forms: applications to algebraic geometry},
  journal={S\'eminaire N. Bourbaki},
  volume={25},
  number={611},
  year={1982},
  pages={251--278}}

@article{dolgachev1996mirror,
  author={Igor Dolgachev},
  title={Mirror symmetry for lattice polarized {K3} surfaces},
  journal={Journal of Mathematical Sciences},
  volume={81},
  number={3},
  year={1996},
  pages={2599--2630}}

@article{doran2016calabi,
  title={Calabi--{Y}au threefolds fibred by mirror quartic {K3} surfaces},
  author={Charles Doran and Andrew Harder and Andrey Novoseltsev and Alan Thompson},
  journal={Advances in Mathematics},
  volume={298},
  pages={369--392},
  year={2016},
  publisher={Elsevier}}

@article{doran2020calabi,
  title={Calabi--{Y}au threefolds fibred by high rank lattice Polarized {K3} surfaces},
  author={Charles Doran and Andrew Harder and Andrey Novoseltsev and Alan Thompson},
  journal={Mathematische Zeitschrift},
  volume={294},
  pages={783--815},
  year={2020},
  publisher={Springer}}

@book{ebeling2013lattices,
  author={Wolfgang Ebeling},
  title={Lattices and codes},
  series={Advanced Lectures in Mathematics},
  year={2013},
  publisher={Springer},
  address={Berlin}}

@article{elkies2014k3,
  title={K3 surfaces and equations for {Hilbert} modular surfaces},
  author={Noam Elkies and Abhinav Kumar},
  journal={Algebra \& Number Theory},
  volume={8},
  number={10},
  pages={2297--2411},
  year={2014},
  publisher={Mathematical Sciences Publishers}}

@book{fulton1993introduction,
  title={Introduction to toric varieties},
  author={William Fulton},
  series={Annals of Mathematics Studies},
  number={131},
  year={1993},
  publisher={Princeton University Press}}

@article{fomin2002cluster,
  title={Cluster algebras {I}: Foundations},
  author={Sergey Fomin and Andrei Zelevinsky},
  journal={Journal of the American mathematical society},
  volume={15},
  number={2},
  pages={497--529},
  year={2002}}

@book{griffiths2014principles,
  title={Principles of algebraic geometry},
  author={Phillip Griffiths and Joseph Harris},
  year={2014},
  publisher={John Wiley \& Sons}}

@article{gritsenko2014uniruledness,
  title={Uniruledness of orthogonal modular varieties},
  author={Valeri Gritsenko and Klaus Hulek},
  journal={Journal of Algebraic Geometry},
  volume={23},
  number={4},
  pages={711--725},
  year={2014}}

@article{griffiths1985noether,
  title={On the {Noether--Lefschetz} theorem and some remarks on codimension-two cycles},
  author={Phillip Griffiths and Joseph Harris},
  journal={Mathematische annalen},
  volume={271},
  number={1},
  pages={31--51},
  year={1985},
  publisher={Springer-Verlag Berlin/Heidelberg}}

@article{gritsenko1998minimal,
  title={Minimal {Siegel} modular threefolds},
  author={Valeri Gritsenko and Klaus Hulek},
  journal={Mathematical Proceedings of the Cambridge Philosophical Society},
  volume={123},
  number={3},
  pages={461--485},
  year={1998},
  organization={Cambridge University Press}}

@incollection{givental1998mirror,
  title={A mirror theorem for toric complete intersections},
  author={Alexander Givental},
  booktitle={Topological field theory, primitive forms and related topics},
  pages={141--175},
  year={1998},
  publisher={Springer}}

@article{golyshev2004modularity,
  title={Modularity of equations {D3} and the {Iskovskikh} classification},
  author={Vasily Golyshev},
  journal={Doklady Mathematics},
  volume={69},
  number={3},
  pages={443--449},
  year={2004}}

@inbook{golyshev2007classification,
  title={Classification problems and mirror duality},
  author={Vasily Golyshev},
  booktitle={Surveys in geometry and number theory: reports on contemporary Russian mathematics},
  series={London Mathematical Society Lecture Note Series},
  number={338},
  pages={88},
  year={2007},
  publisher={Cambridge University Press}}

@article{gross2001calabi,
  title={Calabi--{Y}au threefolds and moduli of abelian surfaces {I}},
  author={Mark Gross and Sorin Popescu},
  journal={Compositio Mathematica},
  volume={127},
  number={2},
  pages={169--228},
  year={2001},
  publisher={London Mathematical Society}}

@article{gross2011calabi,
  title={Calabi--{Y}au threefolds and moduli of abelian surfaces {II}},
  author={Mark Gross and Sorin Popescu},
  journal={Transactions of the American Mathematical Society},
  volume={363},
  number={7},
  pages={3573--3599},
  year={2011}}

@article{gritsenko1994irrationality,
  title={Irrationality of the moduli spaces of polarized abelian surfaces},
  author={Valeri Gritsenko},
  journal={International Mathematics Research Notices},
  volume={1994},
  number={6},
  pages={235--243},
  year={1994}}

@article{galkin2010mutations,
  title={Mutations of potentials},
  author={Sergey Galkin and Alexandr Usnich},
  journal={preprint IPMU},
  volume={10},
  number={0100},
  year={2010}}

@mastersthesis{harder2011moduli,
  title={Moduli spaces of {K3} surfaces with large {Picard} number},
  author={Andrew Harder},
  school={Queen's University},
  address={Kingston, Ontario, Canada},
  year={2011}}

@book{hartshorne2013algebraic,
  title={Algebraic geometry},
  author={Robin Hartshorne},
  series={Graduate Texts in Mathematics},
  number={52},
  year={2013},
  publisher={Springer Science \& Business Media}}

@Phdthesis{harder2016geometry,
  author={Andrew Harder},
  title={The geometry of {Landau--Ginzburg} models},
  school={University of Alberta},
  year={2016}}

@article{harder2021hodge,
  title={Hodge numbers of {Landau--Ginzburg} models},
  author={Andrew Harder},
  journal={Advances in mathematics},
  volume={378},
  eid={107436},
  year={2021},
  publisher={Elsevier}}

@article{horikawa1976deformations,
  title={On deformations of holomorphic maps {III}},
  author={Eiji Horikawa},
  journal={Mathematische Annalen},
  volume={222},
  pages={275--282},
  year={1976},
  publisher={Springer}}

@book{huybrechts2016K3,
  author={Daniel Huybrechts},
  title={Lectures on {K3} surfaces},
  publisher={Cambridge University Press},
  series={Cambridge Studies in Advanced Mathematics},
  number={158},
  year={2016}}

@article{iacono2015deformations,
  title={Deformations and obstructions of pairs {\((X, D)\)}},
  author={Donatella Iacono},
  journal={International Mathematics Research Notices},
  volume={2015},
  number={19},
  pages={9660--9695},
  year={2015},
  publisher={Oxford University Press}}

@article{igusa1960arithmetic,
  title={Arithmetic variety of moduli for genus two},
  author={Jun-ichi Igusa},
  journal={Annals of Mathematics},
  volume={72},
  number={3},
  pages={612--649},
  year={1960}}

@article{ilten2013toric,
  title={Toric degenerations of {Fano} threefolds giving weak {Landau--Ginzburg} models},
  author={Nathan Ilten and Jacob Lewis and Victor Przyjalkowski},
  journal={Journal of Algebra},
  volume={374},
  pages={104--121},
  year={2013},
  publisher={Elsevier}}

@article{iritani2016quantum,
  title={Quantum {Serre} theorem as a duality between quantum {D}-modules},
  author={Hiroshi Iritani and Etienne Mann and Thierry Mignon},
  journal={International Mathematics Research Notices},
  volume={2016},
  number={9},
  pages={2828--2888},
  year={2016},
  publisher={Oxford University Press}}

@article{iritani2009integral,
  title={An integral structure in quantum cohomology and mirror symmetry for toric orbifolds},
  author={Hiroshi Iritani},
  journal={Advances in Mathematics},
  volume={222},
  number={3},
  pages={1016--1079},
  year={2009},
  publisher={Elsevier}}

@article{iritani2011quantum,
  title={Quantum cohomology and periods},
  author={Hiroshi Iritani},
  journal={Annales de l'Institut Fourier},
  volume={61},
  number={7},
  pages={2909--2958},
  year={2011}}

@article{iskovskikh1977fano,
  title={Fano 3-folds {I}},
  author={Vasilii Iskovskikh},
  journal={Mathematics of the USSR-Izvestiya},
  volume={11},
  number={3},
  pages={485--527},
  year={1977},
  publisher={IOP Publishing}}

@article{kawamata1977deformations,
  title={On deformations of compactifiable complex manifolds},
  author={Yujiro Kawamata},
  journal={Proceedings of the Japan Academy},
  volume={53},
  number={3},
  pages={106--109},
  year={1977},
  publisher={The Japan Academy}}

@inproceedings{katzarkov2008hodge,
  title={Hodge theoretic aspects of mirror symmetry},
  author={Ludmil Katzarkov and Maxim Kontsevich and Tony Pantev},
  booktitle={From Hodge theory to integrability and TQFT: tt*-geometry},
  series={Proceedings of Symposia in Pure Mathematics},
  number={78},
  pages={87--174},
  year={2008}}

@article{katzarkov2017bogomolov,
  title={{Bogomolov--Tian--Todorov} theorems for {Landau--Ginzburg} models},
  author={Ludmil Katzarkov and Maxim Kontsevich and Tony Pantev},
  journal={Journal of Differential Geometry},
  volume={105},
  number={1},
  pages={55--117},
  year={2017},
  publisher={Lehigh University}}

@inproceedings{kontsevich1994homological,
  title={Homological algebra of mirror symmetry},
  author={Maxim Kontsevich},
  booktitle={Proceedings of the International Congress of Mathematicians},
  volume={1},
  pages={120--139},
  year={1994},
  address={Zurich}}

@article{kreuzer2004PALP,
  author={Maximilian Kreuzer and Harald Skarke},
  title={{PALP}: a package for analyzing lattice polytopes with applications to toric geometry},
  journal={Computer Physics Communications},
  volume={157},
  number={1},
  year={2004},
  pages={87--106}}

@article{kulikov1977degenerations,
  title={Degenerations of {K3} surfaces and {Enriques} surfaces},
  author={Viktor Kulikov},
  journal={Mathematics of the USSR-Izvestiya},
  volume={11},
  number={5},
  pages={957--989},
  year={1977},
  publisher={IOP Publishing}}

@article{kumar2008k3,
  title={K3 surfaces associated with curves of genus two},
  author={Abhinav Kumar},
  journal={International Mathematics Research Notices},
  volume={2008},
  number={9},
  pages={rnm165},
  year={2008},
  publisher={OUP}}

@article{kumar2015hilbert,
  title={Hilbert modular surfaces for square discriminants and elliptic subfields of genus 2 function fields},
  author={Abhinav Kumar},
  journal={Research in the Mathematical Sciences},
  volume={2},
  pages={1--46},
  year={2015},
  publisher={Springer}}

@book{lefschetz1924analysis,
  title={L'analysis situs et la g{\'e}om{\'e}trie alg{\'e}brique},
  author={Solomon Lefschetz},
  year={1924},
  publisher={Gauthier-Villars et cie}}

@article{mase2014families,
  title={Families of {K3} Surfaces in Smooth {Fano} 3-Folds with {Picard} Number 2},
  author={Makiko Mase},
  journal={Vietnam Journal of Mathematics},
  volume={42},
  pages={295--304},
  year={2014},
  publisher={Springer}}

@article{mori1981classification,
  title={Classification of {Fano} 3-folds with \(b_2 \geqslant 2\)},
  author={Shigefumi Mori and Shigeru Mukai},
  journal={Manuscripta mathematica},
  volume={36},
  number={2},
  pages={147--162},
  year={1981}}

@article{morrison1984k3,
  title={On K3 surfaces with large {Picard} number},
  author={David Morrison},
  journal={Inventiones mathematicae},
  volume={75},
  pages={105--121},
  year={1984}}

@article{orlov1997equivalences,
  title={Equivalences of derived categories and {K3} surfaces},
  author={Dmitri Orlov},
  journal={Journal of Mathematical Sciences},
  volume={84},
  pages={1361--1381},
  year={1997},
  publisher={Springer}}

@article{przyjalkowski2008LG,
  author={Victor Przyjalkowski},
  title={On {Landau--Ginzburg} models for {Fano} varieties},
  journal={Communications in Number Theory and Physics},
  volume={1},
  year={2008},
  pages={713--728}}

@article{przyjalkowski2013weak,
  title={Weak {Landau--Ginzburg} models for smooth {Fano} threefolds},
  author={Victor Przyjalkowski},
  journal={Izvestiya: Mathematics},
  volume={77},
  number={4},
  pages={772--794},
  year={2013},
  publisher={IOP Publishing}}

@article{przyjalkowski2017compactification,
  author={Victor Przyjalkowski},
  title={{Calabi--Yau} compactifications of toric {Landau--Ginzburg} models for smooth {Fano} threefolds},
  journal={Sbornik: Mathematics},
  volume={208},
  number={7},
  year={2017},
  pages={992--1013}}

@article{przyjalkowski2018review,
  author={Victor Przyjalkowski},
  title={Toric {Landau--Ginzburg} models},
  journal={Russian Mathematical Surveys},
  volume={73},
  number={6},
  year={2018},
  pages={1033--1118}}

@article{przyjalkowski2022singular,
  author={Victor Przyjalkowski},
  title={On singular log {Calabi-Yau} compactifications of {Landau-Ginzburg} models},
  journal={Sbornik: Mathematics},
  volume={213},
  number={1},
  pages={88--108},
  year={2022},
  publisher={IOP Publishing}}

@inproceedings{reid46young,
  title={Young person's guide to canonical singularities},
  author={Miles Reid},
  booktitle={Algebraic Geometry, Bowdoin 1985},
  series={Proceedings of Symposia in Pure Mathematics},
  number={46.1},
  pages={345--414},
  editor={Spencer Janney Bloch},
  year={1987}}

@online{rohsiepe2004lattice,
  title={Lattice polarized toric K3 surfaces},
  author={Falk Rohsiepe},
  archivePrefix={arXiv},
  eprint={hep-th/0409290},
  year={2004}}

@book{scattone1987compactification,
  title={On the compactification of moduli spaces for algebraic {K3} surfaces},
  author={Francesco Scattone},
  series={Memoirs of the American Mathematical Society},
  number={374},
  year={1987},
  publisher={American Mathematical Society}}

@inproceedings{seidel2001vanishing,
  title={Vanishing cycles and mutation},
  author={Paul Seidel},
  booktitle={European Congress of Mathematics: Barcelona, July 10--14, 2000},
  volume={2},
  pages={65--85},
  year={2001},
  organization={Springer}}

@article{shamoto2018hodge,
  title={{Hodge--Tate} conditions for {Landau--Ginzburg} models},
  author={Yota Shamoto},
  journal={Publications of the Research institute for mathematical sciences},
  volume={54},
  number={3},
  pages={469--515},
  year={2018}}

@article{steenbrink1976limits,
  title={Limits of {Hodge} structures},
  author={Joseph Steenbrink},
  journal={Inventiones mathematicae},
  volume={31},
  number={3},
  pages={229--257},
  year={1976},
  publisher={Springer-Verlag Berlin/Heidelberg}}

@inbook{ueda2020mirror,
  title={Mirror symmetry and {K3} surfaces},
  author={Kazushi Ueda},
  booktitle={Handbook for mirror symmetry of {Calabi--Yau} and {Fano} manifolds},
  series={Advanced Lectures in Mathematics},
  number={47},
  pages={483--521},
  year={2020},
  publisher={International Press of Boston}}

@mvbook{voisin2003hodge,
  title={Hodge Theory and Complex Algebraic Geometry},
  volume={2},
  author={Claire Voisin},
  series={Cambridge Studies in Advanced Mathematics},
  number={77},
  year={2003},
  publisher={Cambridge University Press}}

@inbook{whitcher2015reflexive,
  title={Reflexive polytopes and lattice-polarized {K3} surfaces},
  author={Ursula Whitcher},
  booktitle={Calabi-Yau Varieties: Arithmetic, Geometry and Physics},
  series={Fields Institute Monograph},
  number={34},
  pages={65--79},
  year={2015},
  publisher={Springer}}

\end{document}